\documentclass[11pt]{article}
\usepackage[margin=1in]{geometry}

\usepackage{listings}
\usepackage{amsmath,amsthm,amsfonts,amssymb,amscd}
\usepackage{enumitem}
\usepackage[colorlinks=true, pdfstartview=FitV, linkcolor=blue,citecolor=blue, urlcolor=blue]{hyperref}

\usepackage{times}

\usepackage{esint,stackrel}
\usepackage{enumitem}
\usepackage{bbm}

\usepackage{mathabx}
\usepackage{float}
\usepackage{makeidx}
\usepackage{emptypage}
\usepackage{amsrefs}
\usepackage{xfrac}

\usepackage{stmaryrd}
\usepackage{cases}
\usepackage[mathscr]{euscript}

\usepackage{tikz}
\usepackage{pgfplots}

\usepgfplotslibrary{fillbetween}
\pgfplotsset{compat=1.12}
\usetikzlibrary{intersections}

\allowdisplaybreaks

\def\p{\partial}
\def\eps{\varepsilon}
\def\les{\lesssim}
\renewcommand*{\div}{\ensuremath{\mathrm{div\,}}}

\newcommand{\norm}[1]{\left \| #1 \right\|} 
\newcommand{\snorm}[1]{\bigl\| #1 \bigr\|} 
\newcommand{\abs}[1]{\left|#1\right|}
\newcommand{\sabs}[1]{\bigl|#1\bigr|}

\newcommand*{\sgn}{\ensuremath{\mathrm{sgn\,}}}
\newcommand{\RR}{\mathbb R}

\newcommand{\TT}{\mathbb T}
\newcommand{\OO}{\mathcal O}

\renewcommand*{\tilde}{\widetilde}
\renewcommand*{\hat}{\widehat}
\renewcommand*{\bar}{\overline}

\newcommand{\XXX}{{\mathcal X}}
\newcommand{\FFF}{{\mathcal F}}
\newcommand{\LLL}{{\mathcal L}}
\renewcommand{\epsilon}{\varepsilon}

\newcommand{\kcal}{S}

\newtheoremstyle{thmboldstyle}
   {}{}{\itshape}{}{\bfseries}{.}{.5em}{{\thmname{#1 }}{\thmnumber{#2}}{\thmnote{ (#3)}}}
\theoremstyle{thmboldstyle}

\newtheoremstyle{remboldstyle}
   {}{}{}{}{\bfseries}{.}{.5em}{{\thmname{#1 }}{\thmnumber{#2}}{\thmnote{ (#3)}}}
\theoremstyle{remboldstyle}

\makeatletter
\renewenvironment{proof}[1][\proofname]{%
   \par\pushQED{\qed}\normalfont%
   \topsep6\p@\@plus6\p@\relax
   \trivlist\item[\hskip\labelsep\bfseries#1\@addpunct{.}]%
   \ignorespaces
}{%
   \popQED\endtrivlist\@endpefalse
}
\makeatother

\newcommand{\be}{\begin{equation}}
\newcommand{\ee}{\end{equation}}
\newcommand{\rmd}{{\rm d}}
\newcommand{\DD}{\mathcal D}

\newcommand{\Root}{\mathtt Z}

\newcommand{\wb}{{w_{\mathsf B}}}
\newcommand{\etab}{{\eta_{\mathsf B}}}

\newcommand{\xbl}{x_{{\mathsf B},-}}
\newcommand{\xbr}{x_{{\mathsf B},+}}
\newcommand{\xbpm}{x_{{\mathsf B},\pm}}

\theoremstyle{thmboldstyle}
\newtheorem{theorem}{Theorem}[section]
\newtheorem{lemma}[theorem]{Lemma}
\newtheorem{proposition}[theorem]{Proposition}
\newtheorem{corollary}[theorem]{Corollary}
\newtheorem{definition}[theorem]{Definition}

\theoremstyle{remboldstyle}
\newtheorem{remark}[theorem]{Remark}

\numberwithin{equation}{section}

\newcommand{\bb}{\mathsf b} 
\newcommand{\cc}{\mathsf c} 
 
\newcommand{\mm}{\mathsf m} 
\renewcommand{\ss}{{\scriptstyle \mathcal{C}}} 
\renewcommand{\ss}{c}

\newcommand{\vr}{w_{\mathsf  +}} 
\newcommand{\vl}{w_{\mathsf  -}}

\def\jump#1{{[\hspace{-1.5pt}[#1]\hspace{-1.5pt}]}}
\def\mean#1{{\langle\hspace{-2.5pt}\langle#1\rangle\hspace{-2.5pt}\rangle}}

\def\sc{\mathfrak{s}}
\def\sn{\mathfrak{n}}

\def\cc{ \mathsf{c} }

\def\zl{ {z}_{-}}
\def\kl{ {k}_{-}}
\def\cl{ {c}_{-}}
\def\el{ {\mathfrak e}_{-}}

\def\pt{\phi_t}
\def\pst{\psi_t}
\def\kk{ {\scriptstyle\mathcal{K}} }

\def\aa{ {\mathsf a} }

\def\uu{ {\mathcal{U}_\mathsf{L} }}

\newcommand{\Rsf}{\mathsf R} 
\newcommand{\Dsf}{\mathsf D} 
\newcommand{\Isf}{\mathsf I} 

\def\st{ {\scriptstyle {\mathscr{T}}}} 
\def\stt{ {\scriptstyle {\mathscr{J}}} }

\def\ie{ \eta_{ \operatorname{inv} } }

\newcommand{\yy}{{\mathsf y}}
\newcommand{\wsf}{{\mathsf w}}
\newcommand{\zsf}{{\mathsf z}}
\newcommand{\ksf}{{\mathsf k}}
\newcommand{\asf}{{\mathsf a}}
\newcommand{\csf}{{\mathsf c}}

\newcommand\ringring[1]{%
  {
   \mathop{\kern0pt #1}\limits^{
     \vbox to-1.85ex{
       \kern-2ex 
       \hbox to 0pt{\hss\normalfont\kern.1em \r{}\kern-.45em \r{}\hss}%
       \vss 
     }
   }
  }
}
\newcommand\ringringring[1]{%
  {
   \mathop{\kern0pt #1}\limits^{
     \vbox to-1.85ex{
       \kern-2ex 
       \hbox to 0pt{\hss\normalfont\kern.1em \r{}\kern-.45em \r{}\kern-.45em \r{}\hss}%
       \vss 
     }
   }
  }
}

\usepackage{fancyhdr}
\title{{\bf Simultaneous development of shocks and cusps for 2D Euler with azimuthal symmetry from smooth data \\ \,
}}

\author{
{ \small {\bf Tristan Buckmaster}}\thanks{\footnotesize Department of Mathematics, 
Princeton University, Princeton, NJ 08544,
 \href{tjb4@math.princeton.edu}{tjb4@math.princeton.edu}.}
 \and
{ \small {\bf Theodore D. Drivas}}\thanks{\footnotesize Department of Mathematics, Stony Brook University,
Stony Brook, NY, 11794,
 \href{tdrivas@math.stonybrook.edu}{tdrivas@math.stonybrook.edu}.}
\and  
{\small {\bf Steve  Shkoller}}\thanks{\footnotesize Department of Mathematics, UC Davis, Davis, CA 95616, \href{shkoller@math.ucdavis.edu}{shkoller@math.ucdavis.edu}.}
\and 
{\small {\bf Vlad Vicol}}\thanks{\footnotesize Courant Institute of Mathematical Sciences, New York University, New York, NY 10012, \href{vicol@cims.nyu.edu}{vicol@cims.nyu.edu}.}
}

\date{}

\begin{document}

\maketitle

\begin{abstract}
A fundamental question in fluid dynamics concerns the formation of discontinuous shock waves from smooth initial data.   We prove that from smooth initial data, smooth solutions to the 2d Euler equations in azimuthal symmetry form a first singularity, the so-called $C^{\frac{1}{3}} $ {\it pre-shock}. The solution in the vicinity of this pre-shock is shown to have a
 fractional series expansion with coefficients computed from the data.
Using this precise description of the pre-shock, we prove that a {\em discontinuous shock} instantaneously develops after the pre-shock. This {\em regular shock solution} is shown to be unique in a  class of  entropy solutions with azimuthal symmetry and regularity determined by the pre-shock expansion.  Simultaneous to the development of the shock front, two other characteristic surfaces of cusp-type singularities emerge from the  pre-shock.  These surfaces have been termed  {\it weak  discontinuities} by Landau \& Lifschitz \cite[Chapter IX, \S 96]{landaulifshitz}, who conjectured   some
type of singular behavior of derivatives  along such surfaces.  
We prove that along
the slowest surface,  all fluid variables except the entropy have $C^{1, {\frac{1}{2}} }$ one-sided cusps from the shock side, and that the normal velocity is decreasing in the direction of its motion; we thus term this surface a  {\it weak rarefaction wave}.  Along the surface moving with the fluid velocity,
density and entropy form $C^{1, {\frac{1}{2}} }$ one-sided cusps while the pressure and normal velocity remain $C^2$; as such, we term this surface a
{\it weak contact discontinuity}.  
\end{abstract}

\renewcommand{\baselinestretch}{0.8}\normalsize
\setcounter{tocdepth}{2}

\tableofcontents
\renewcommand{\baselinestretch}{1.0}\normalsize


\section{Introduction}
 We consider the simultaneous development of {\em shock waves} and {\em weak singularities} (contact and rarefaction cusps) from smooth initial data, for the
 two-dimensional compressible Euler equations in azimuthal symmetry.   This problem consists of:
 \begin{itemize}
\item the {\it shock formation} process, in which we start from smooth initial data and construct the first singularity,
 the so-called {\it pre-shock};
 \item the {\it shock development} process, in which the pre-shock  instantaneously evolves into a discontinuous entropy producing shock wave, and two other families of weak characteristic singularities (cusps).
 \end{itemize}
 
\subsection{The compressible Euler equations}
 For shock development, it is essential to write the Euler equations
 in conservation form, so as to ensure the physical jump conditions (conserving total mass, momentum and energy) are satisfied. The system reads
\begin{subequations} 
\label{euler-weak}
\begin{align}
\partial_t (\rho u)  + \operatorname{div} (  \rho u\otimes u + p I ) &=0\,,  \label{ee2}\\
\partial_t \rho + \operatorname{div}  (  \rho u ) &=0\,, \label{ee1}\\
\partial_t E + \operatorname{div}  (  (p+ E) u ) &=0 \label{ee3}\,,
\end{align}
\end{subequations} 
where 
 $u :\mathbb{R}^2  \times \mathbb{R}  \to \mathbb{R}^2  $ denotes the velocity vector field, $\rho: \mathbb{R}^2  \times \mathbb{R}  \to \mathbb{R}  _+$ denotes the
strictly positive density,    $E: \mathbb{R}^2  \times \mathbb{R}  \to \mathbb{R}$ denotes the total energy, and 
$p: \mathbb{R}^2  \times \mathbb{R}  \to \mathbb{R}$ denotes the pressure function which is related to $(u,\rho, E)$ by the identity
$$
p = (\gamma-1) \left( E- \tfrac{1}{2} \rho \abs{u}^2\right) \,,
$$
where $\gamma>1$ denotes the adiabatic exponent.
For smooth solutions, the conservation of energy equation \eqref{ee3} can be replaced by the transport of (specific) entropy $\p_t S + u \cdot \nabla S =0$, where 
$S: \mathbb{R}^2  \times \mathbb{R}  \to \mathbb{R}$ denotes the entropy function, and the pressure has the equivalent form
\begin{align}
 p(\rho,\kcal) = \tfrac{1}{\gamma} \rho^\gamma e^{\kcal}\,.
 \label{peos}
\end{align}

We   consider solutions to the Euler equations \eqref{euler-weak}  which start from smooth {\it non-degenerate} initial data at time $T_0$, form a first singularity or 
{\it pre-shock} at time $T_1$,  and  simultaneously develop a  discontinuous shock wave  and surfaces of weak characteristic discontinuities on the  time interval
$(T_1,T_2]$.
Solutions on the time interval $[T_0,T_1)$ are  {\it classical solutions} to \eqref{euler-weak}, and only the continuation of these solutions past $T_1$ requires the introduction of the Rankine-Hugoniot jump conditions.

Suppose that for  $t\in (T_1,T_2]$, the shock front $\mathcal{S}\subset \mathbb{R}^d\times (T_1,T_2]$ is an orientable space-time hypersurface across which the velocity $u^\pm$, density $\rho^\pm$, and energy $E^\pm$ jump.  We consider the case where this surface is given by $\mathcal{S}:=\{ \sc(t,x_1, x_2, \dots x_d) = 0 \}$ with spacetime normal $- (\dot{\sc}, \nabla_x \sc )|_{\mathcal S}:=(-\dot{\sc}, n) $.
We assume that $(u^\pm,\rho^\pm, E^\pm)$ are defined in the sets $\Omega^\pm(t)\subset \mathbb{R}^2$ separated by the shock front at time $t$. 
Let $n( \cdot , t)$ point from $\Omega^-(t)$ to $\Omega^+(t)$, which is in the direction of propagation of the shock front.   In two-dimensions, we let $\tau(\cdot,t) = n(\cdot,t)^\perp$ denote the tangent vector.
We denote $\jump{f}=f^-  -f^+$ where $f^\pm$ (sometimes denoted $f_\pm$) are the traces of $f$ along ${\mathcal S}$ in the regions $\Omega^\pm$ respectively,  and   $u_n = u\cdot n |n|^{-1}$, $u_\tau = u \cdot  \tau |\tau|^{-1}$. 
The shock speed is denoted by $\dot \sc$.
The Rankine-Hugoniot jump conditions state that the shock speed $\dot\sc$ along with the jumps of the fields across ${\mathcal S}$ must simultaneously satisfy
\begin{subequations}\label{RH_condition}
\begin{align}
\dot\sc |n|^{-1}  \jump{\rho u_n} &=\jump{\rho u_n^2+  p I}\,, \label{RH1}\\
\dot\sc |n|^{-1} \jump{\rho} &= \jump{\rho u_n}  \,, \label{RH2}\\
\dot\sc |n|^{-1} \jump{E} &= \jump{ (p+ E) u_n} \,, \label{RH3}
\end{align}
\end{subequations}
where we have used $\jump{ u_\tau} =0$ for a shock discontinuity.

\begin{definition}[\bf Regular shock solution]
\label{shockdef}
We say that $(u, \rho, E, \sc)$ is a {\em regular shock solution} on $\mathbb{R}^d \times [T_1,T_2]$ if
the following conditions hold:
\begin{enumerate}
\item $(u, \rho, E)$ is a weak solution of \eqref{euler-weak}  and $\rho\geq  \rho_{\operatorname{min}} >0$;
\item the shock front $\mathcal{S}\subset \mathbb{R}^d\times \mathbb{R}_+$ is an orientable hypersurface;
\item  $(u, \rho, E)$ are Lipschitz continuous in space and time on the complement of the shock surface  $ (\mathbb{R}  ^d \times [T_1,T_2]) \setminus \mathcal{S}$;
\item  $(u, \rho, E)$  have discontinuities across the shock which  satisfy the Rankine-Hugoniot conditions \eqref{RH_condition}.  
\end{enumerate}
Furthermore, the solution has a {\em weak shock} if 
$$\sup_{t\in[T_1,T_2]} \left(  \abs{ \jump{u(t)}}  + \abs{\jump{\rho(t)} } + \abs{ \jump{E(t)}  } \right)    \ll  1\,.$$
 \end{definition}

\subsection{Prior results in shock development problem for Euler}
%
%
 For hyperbolic systems in one space dimension, existence (and in some cases uniqueness) of global weak solutions is well understood using either the Glimm scheme or compensated compactness techniques (see e.g.~\cite{dafermos2005hyperbolic}).  Unfortunately, these methods cannot provide a description of the surfaces across which weak and strong singularities propagate.
In multiple space dimensions, Majda~\cite{majda1983existence,majda1983stability} establishes the short-time evolution (and stability) of a shock front.
This is a free-boundary problem in which the parameterized shock surface moves with the shock speed given by the Rankine-Hugoniot conditions.   In
this problem, the  initial data consists of a shock surface and discontinuous $(u,\rho, E)$ which are smooth on either side of the shock.   As such, this framework
does not include the shock development problem, in which the surface of discontinuity must evolve from a H\"{o}lder  pre-shock. 
 
There are very few results on the formation and development of shocks.
For the one-dimensional $p$-system (which models
1d isentropic Euler),  Lebaud~\cite{Lebaud1994} was the first  to prove shock formation and development.  Following \cite{Lebaud1994},
Chen \& Dong~\cite{ChDo2001} and Kong~\cite{Kong2002} also  proved formation and development of shocks for the 1d $p$-system
with slightly more general initial data.   However, because entropy is created at the shock, the use of the isentropic 2$\times$2 $p$-system cannot produce weak solutions
to the 1d Euler equations.\footnote{We emphasize that the Rankine-Hugoniot jump conditions are not satisfied under the isentropic assumption, see Lemma~\ref{lemmaent}.}  Yin~\cite{Yin2004} was the first
to consider the formation and development problem for the non-isentropic 3$\times$3 Euler equations  in spherical symmetry.  Independently, 
shock development for the barotropic Euler equations under spherical symmetry was established by Christodoulou \& Lisbach~\cite{ChLi2016}.  
The use of the isentropic model  or the assumption of an irrotational flow in higher dimensions cannot produce  weak solutions to the
Euler equations, and as such has been termed the {\it restricted shock development}.   Christodoulou~\cite{Ch2019}
has established  restricted shock development for the irrotational and isentropic Euler equations in three spatial dimensions and completely outside of symmetry.   Yin \& Zhu~\cite{YinZhu2021b} have recently established shock development in
two dimensions for a scalar conservation law.  
 
As previously noted by Landau \& Lifschitz in~\cite[Chapter IX, \S 96]{landaulifshitz}, at the same time that the discontinuous shock wave develops,  other surfaces of singularities are expected to simultaneously form. Landau \& Lifschitz termed these surfaces
{\it weak discontinuities}.   In the restricted shock development problem,  Christodoulou~\cite[Page 3]{Ch2019} constructs  $C^{1,\frac 12}$ cusp singularities  along the characteristic of the  fluid velocity minus the sound speed, emanating from the first singularity (akin to the $\sc_1$ curve in Theorem~\ref{thm:main:azi:soft}).  For the full Euler system (with or without symmetry, even in one dimension)  the analysis of these surfaces of weak discontinuity has been heretofore nonexistent. In this paper we prove that  two such surfaces of weak singularities emerge from the pre-shock and move with the slower sound-speed characteristic  and the fluid velocity respectively. 
We shall refer to these two surfaces as a {\it weak contact} ($\sc_2$), respectively a {\it weak rarefaction} ($\sc_1$). 
We call the curve $\sc_2$ a weak contact because it moves with the fluid velocity, and both the normal velocity  and the pressure are one degree smoother than the density and entropy. The curve $\sc_1$ is called a weak rarefaction because the normal velocity to this curve is decreasing in the direction of its motion -- see Section
\ref{sec:2D:Euler}.

\subsection{Statement of the main results}
The goal of this paper is to prove the following (we refer to Theorems~\ref{thm:main:formationeuler} and~\ref{thm:main:developmenteuler} for a precise statement):
\begin{theorem}[\bf Main result for 2D Euler  -- abbreviated version]
\label{thm:main:soft}
From smooth isentropic initial data with azimuthal symmetry, at time $T_0$, there exist smooth solutions to the
2d Euler equations~\eqref{euler-weak} that form a pre-shock singularity at a time $T_1 > T_0$.
The first singularity occurs along a half-infinite ray and the blowup  is  asymptotically self-similar, exhibiting a $C^{\frac{1}{3}}$ cusp in the angular velocity and mass density, and a $C^{1,\frac{1}{3}}$ cusp in the radial velocity.  Moreover, the blowup is  given by a series expansion whose coefficients are computed 
as a function of the initial data.

Past the pre-shock, the solution is continued on $(T_1,T_2]$, as an entropy--producing regular shock solution of the full 2d non-isentropic Euler equations~\eqref{euler-weak}. The solution is unique in the class of entropy producing weak solutions with azimuthal symmetry, with a certain weak shock structure and suitable regularity off the shock (see Definition~\ref{def:XT} below). The following properties are established:
\begin{itemize}
\item Across the shock curve, all the state variables jump:
\begin{align*}
\jump{u_\theta} \sim (t-T_1)^{\frac{1}{2}}, \qquad \jump{\rho} \sim (t-T_1)^{\frac{1}{2}},\qquad \jump{\partial_\theta u_r} \sim (t-T_1)^{\frac{1}{2}},\qquad \jump{S} \sim (t-T_1)^{\frac{3}{2}}.
\end{align*}
\item Across the characteristic emanating from the pre-shock and moving with the fluid velocity,  the entropy, density and radial velocity all have a
$C^{1,\frac{1}{2}}$ one-sided cusp from the right, while from the left, they are all $C^2$ smooth.  The second derivative of the angular velocity and of the pressure is bounded across this curve for $t \in (T_1,T_2]$.
\item Across the characteristic emanating from the pre-shock and moving with sound speed minus the fluid velocity, the entropy is zero while the angular velocity and density have $C^{1, \frac{1}{2}}$ one-sided cusps from the right, while from the left, they are all $C^2$ smooth.  The second derivative of the radial velocity is bounded  across this curve for $t \in (T_1,T_2]$.
\end{itemize}
We thereby obtain a full propagation of singularities result for regular shock solutions, capturing both the jump discontinuity and the weak singularities emanating from the initial cusp in the pre-shock.
\end{theorem}

%

\begin{figure}[htp]
\centering
 \begin{tikzpicture}[]
   \pgftext{\includegraphics[width=\textwidth]{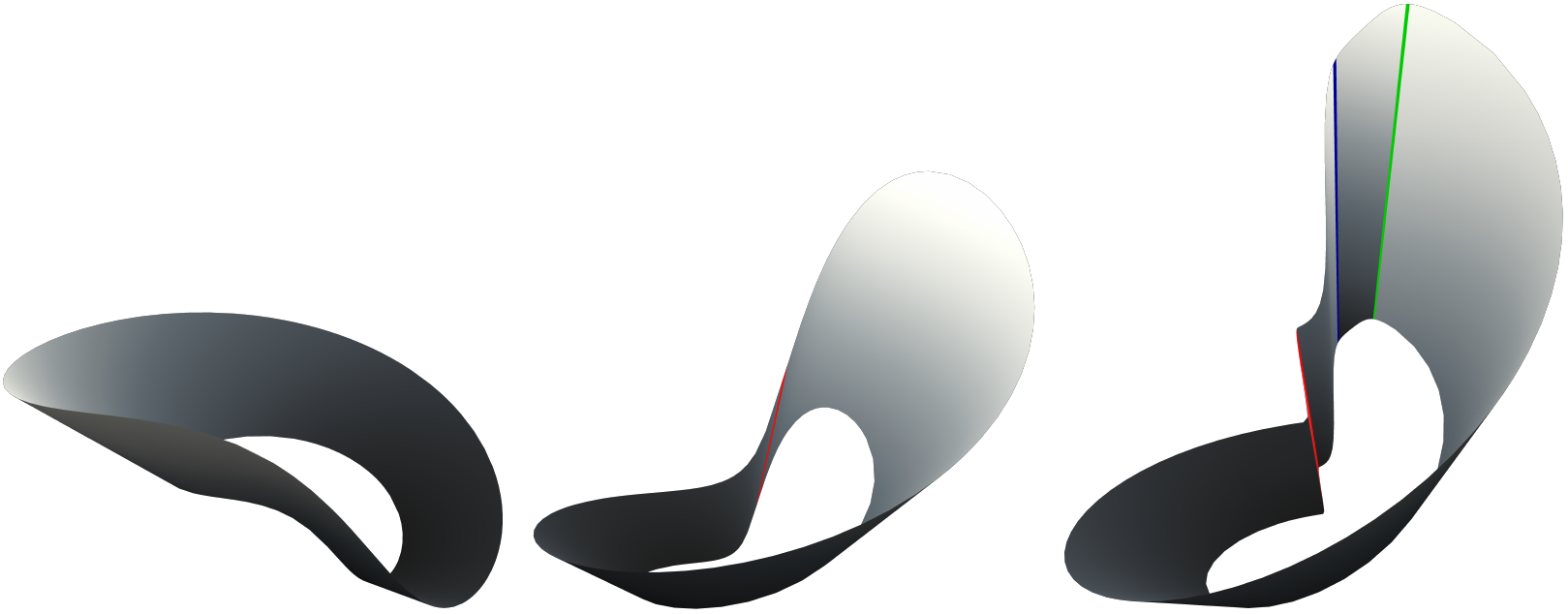}};
\draw[green!40!gray, ultra thick] (6.2,-0.14)  -- (6.57,3.2) ;
 \draw[blue, ultra thick] (5.84,-0.37)  -- (5.8,2.6) ;
 \draw[red, ultra thick] (5.63,-1.69)  -- (5.4,-0.22) ;
  \draw[red, ultra thick] (-.26,-2)  -- (0,-.75) ;

 \end{tikzpicture}
\caption{\footnotesize The images represent values of the density written in polar coordinates $\rho(r,\theta,t)$, and plotted for $r\in[1,2]$. The image on the left represents the smooth data at time $T_0$. The center image shows the pre-shock formed at time $T_1$, at one specific value of the angular coordinate; we marked the corresponding line in red. The image on the right represents the density at time $T_2$, where we have represented in red the line along which the shock discontinuity occurs, in blue the line containing the weak contact, and in green the line corresponding to the weak rarefaction.}
\label{fig:headache}
\end{figure}
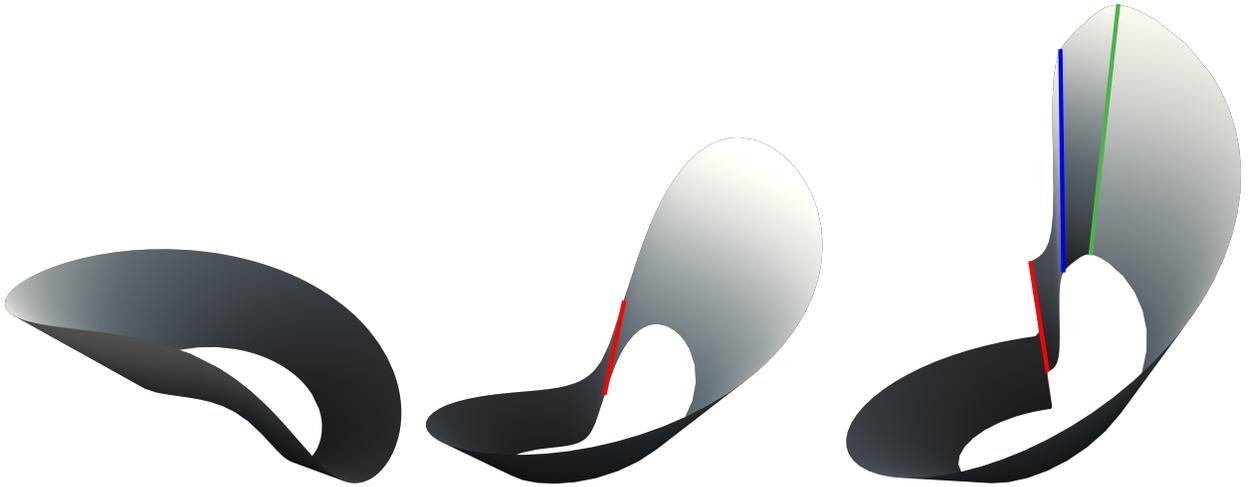

\begin{remark}[Anomalous entropy production]
In analogy with Onsager's conjecture on anomalous  dissipation of kinetic energy by weak solutions of incompressible Euler, entropy can be anomalously
 produced by singular inviscid solutions of the compressible Euler equations.  Theorem 3 of \cite{drivas2018onsager} establishes the following $L^3$-based 
 Onsager-criterion:  if $u,\rho,E\in L^\infty(0, T;(B_{3,\infty}^{1/3+}\cap L^\infty)_{\rm loc} (\mathbb{R}^d))$ then there is no entropy production. Our Theorem \ref{thm:main:soft} provides an example of an entropy producing weak solution resulting from continuing past a finite time singularity.  In fact, the solution we construct lies in $u,\rho,E\in (BV \cap L^\infty)_{\rm loc}\subset (B_{p,\infty}^{1/p})_{\rm loc}$, for  every $p\geq 1$,  illustrating the sharpness of the Onsager criterion in this context.  
\end{remark}
\begin{remark}[Uniqueness and entropy] 
With regards to the question of uniqueness,   the recent work~\cite{klingenberg2020shocks} established that infinitely many entropy-producing weak solutions emanating from 1d Riemann data exist (see also the references therein for the rich history of such convex-integration constructions going back to
\cite{ChDeKr2015}).  The solutions in~\cite{klingenberg2020shocks} break the 1d symmetry and are in general just bounded, and show that the usual entropy condition cannot ensure uniqueness in the class of bounded weak Euler solutions.  By contrast, we establish uniqueness in a class of weak solutions with azimuthal symmetry, exhibiting  {\em weak shock  structure}, and which have regularity consistent with the fact that they emanate from a $C^{\frac 13}$ pre-shock (see Definition~\ref{def:XT}).
\end{remark}

\section{Jump Conditions and Entropy Conditions}
\subsection{The Rankine-Hugoniot jump conditions for the Euler equations}
\label{sec:RH:Euler}

We now return to the Rankine-Hugoniot conditions \eqref{RH_condition}. 
 The weak shock regime is relevant to the development of a discontinuous shock wave from H\"{o}lder continuous data (the pre-shock).  A  key feature of a regular shock solution 
 to the Euler equations is the production of entropy along the shock surface.
 
In order to best exemplify this entropy production,  we shall set
\begin{align} 
\kcal_+ =0 \,.  \label{k-plus-zero}
\end{align} 
 We then define 
\be
v = u_n-\dot\sc \,.
\ee
Then noting that $ u_n^2= (v+\dot\sc)^2= v^2 + 2 \dot\sc v+ \dot\sc^2$ and $\dot\sc u_n = \dot\sc v+\dot \sc^2$, the jump conditions \eqref{RH_condition} become
\begin{subequations} 
\label{rh-all}
\begin{align}
0&= \jump{\rho v^2 +  p}\,, \label{RH1t}\\
0&= \jump{\rho v}  \,, \label{RH2t}\\
0 &=\jump{ v^2 + \tfrac{2\gamma}{\gamma-1}\tfrac{p}{\rho}  }\,, \label{RH3t}
\end{align}
\end{subequations} 
From \eqref{RH2t}, we know that the mass flux is continuous $\rho_-v_-=\rho_+v_+=:j$.   For a {shock discontinuity}  $j\neq 0$ implying the tangential velocity is continuous across the shock $\jump{u_\tau}=0$. In our setup, mass is crossing the shock from the `$+$' phase to the `$-$' phase, so the shock is traveling from `$-$' to `$+$'.  With our choice of orientation for the normal, this fixes $j<0$, which implies that
\be\label{lax2}
u^- \cdot n < \dot\sc, \qquad u^+ \cdot n < \dot\sc.
\ee
Thus, the shock speed is greater than the normal velocity of the fluid on both sides of the shock, consistent with that mass flux being negative $j<0$. We will refer to `$-$' state as behind the shock and the `$+$' state  as the front.

\subsection{Second Law of Thermodynamics and the physical entropy condition}

We now explain the meaning and consequences of the physical entropy condition. 
The motion of a viscous compressible fluid  in $d$-spatial dimensions is, to good approximation, governed by the  Navier-Stokes system.  In that system, any non-trivial state has the property that net entropy is increasing
\be\label{entprod}
\frac{\rmd}{\rmd t} \int_\Omega \rho S\  \rmd x > 0,
\ee
provided $u$ is tangent to $\Omega$ (and the boundaries $\partial \Omega$ are insulating if the thermal diffusivity is non-vanishing). Namely, the second law of thermodynamics holds.
For the Euler equations, the entropy satisfies 
\be\label{eulerent}
\partial_t (\rho \kcal)  + \nabla \cdot (   \rho u \kcal ) =0  \,
\ee
and is thus has conserved average for smooth solutions.  We recall here the following classical result

\begin{lemma}\label{lemmaent}
Let $(u, \rho, E)$ be a weak shock solution.
Then, entropy is produced \eqref{entprod} if and only if $\jump{\kcal}>0$.  Moreover, provided that the specific volume $V:= 1/\rho$ and enthalpy $h = \frac{p}{\rho} +  e$  when viewed as a functions of pressure and entropy are $C^4$,  then the following leading order description of the entropy jump holds 
\be\label{spcube}
\jump{S} =   \frac{1}{12} \frac{1}{T_+}  \left. \left(\frac{\partial^2V}{\partial p_+^2}\right)\right|_S \jump{p}^3 + \mathcal{O}( \jump{p}^4).
\ee
\end{lemma}
The notation ``$f(x) = \mathcal{O}(x)$" means, as usual,  $|f(x)| \leq {\rm (const.)} |x|$ for all sufficiently small $x$. An immediate implication of equation \eqref{spcube} is that entropy variation is produced once a shock is formed, even if the flow was initially isentropic.

\begin{remark}[Equations of State]
Although we only require finite regularity in Lemma \ref{lemmaent}, away from phase transitions, all thermodynamic functions
 are smooth in their arguments. Thus, the specific volume $V:=V(\rho, S)$ and the enthalpy $h:= h(\rho, S)$ which are used in the subsequent proof are smooth functions of $p$ and $S$. As such, our assumption physically is that our medium is far from criticality. 
Moreover, strict convexity ${\partial^2V}/{\partial p^2}>0$ is a material property.
For example, for a ideal gas (the family we consider) we have explicitly
\be\label{spvder}
\left. \left(\frac{\partial^2V}{\partial p^2}\right)\right|_S = (1+\gamma^{-1}) \frac{V} p^2>0
\ee
which can be obtained by differentiating the relationship $p V^\gamma = {\rm (const.)}$ (equation \eqref{peos}).  The thermodynamic temperature appearing in \eqref{spcube} can also be explicitly related to $\rho$ and $p$ in this setting.  Specifically, for the ideal gas law
$T= e/c_v$ where the internal energy $e= \frac{p}{(\gamma-1) \rho}$, we have the following explicit formula
$
\frac{1}{T} = \frac{c_v (\gamma-1) \rho}{p} .
$
\end{remark}

\begin{remark}[Correlations in jumps]
One consequence of Lemma \ref{lemmaent} is that, if $V$ is a strictly convex function of the pressure (as it is for the ideal gas), then positive entropy production implies positivity of the jumps $\jump{p}>0$, $\jump{\rho}>0 $ and $\jump{u_n}>0$. This conclusion is simply the well known fact  that pressure and mass density trailing the shock exceed their values at the front, due to compression.  See Landau and Lifshitz \cite{landaulifshitz}, Chapter IX for an extended discussion.
\end{remark}

\begin{proof}[Proof of Lemma~\ref{lemmaent}]
Integrating the entropy balance \eqref{eulerent} over the domain one finds
\begin{align} 
\dot\sc \jump{\rho \kcal} -\jump{\rho u \kcal}  = -j \jump{\kcal}  \,. \label{jumpk-sign}
\end{align} 
Thus we must have $-j \jump{\kcal}>0$, to be consistent with  the second law of thermodynamics  \eqref{entprod} imposed, for example, by the effects of infinitesimal viscosity. Recalling that, with our conventions the mass flux $j = \rho v$ is negative (mass is passing the shock from $+$ to $-$),
we see that the \emph{physical entropy condition}  \eqref{jumpk-sign} is equivalent to the condition
\begin{align} \label{entcond}
\jump{\kcal}>0 \,.
\end{align} 
We now derive the consequences of \eqref{entcond} for \emph{weak shocks}.
In what follows, we will show that $\jump{S}=\mathcal{O}( \jump{p}^3)$.  In the calculations below, we anticipate this result in our expansions. It is convenient to work with the enthalpy $h = \frac{p}{\rho} +  e$.  We  regard $h= h(p,S)$ and Taylor expand to obtain 
\begin{align*}
\jump{h}  &=\left. \left(\frac{\partial h}{\partial S_+}\right)\right|_p \jump{S} + \left. \left(\frac{\partial h}{\partial p_+}\right)\right|_S \jump{p} \notag\\
&\qquad + \frac{1}{2}  \left. \left(\frac{\partial^2h}{\partial p_+^2}\right)\right|_S \jump{p}^2 
+ \frac{1}{6}  \left. \left(\frac{\partial^3h}{\partial p_+^3}\right)\right|_S \jump{p}^3 + \mathcal{O} (\jump{S}   \jump{p}, \jump{p}^4 , \jump{S}^2).
\end{align*}
Recalling the  first law of thermodynamics in the form
\be
 \rmd h = T  \rmd S +   V \rmd p,
\ee
where $V:=1/\rho$ is the specific volume, we find that
\be
\left. \left(\frac{\partial h}{\partial S}\right)\right|_p = T, \qquad \left. \left(\frac{\partial h}{\partial p}\right)\right|_S= V.
\ee
Thus, the Taylor expansion of the enthalpy becomes
\begin{align}\label{heqn1}
\jump{h} &=T_+ \jump{S} + V_-\jump{p} + \frac{1}{2}  \left. \left(\frac{\partial V}{\partial p_+}\right)\right|_S \jump{p}^2  + \frac{1}{6}  \left. \left(\frac{\partial^2V}{\partial p_+^2}\right)\right|_S \jump{p}^3 +   \mathcal{O} (\jump{S}   \jump{p}, \jump{p}^4 , \jump{S}^2).
\end{align}
Recalling that  the mass flux $j$ is continuous across the shock, we note that by \eqref{RH1t} that
\be\label{PVrel}
\jump{p} = - \jump{\rho v^2} = -j \jump{v} =  -j^2 \jump{V },
\ee 
which implies $ j^2 =- {\jump{p} }/{\jump{V} }$.
Moreover, from \eqref{RH3t}, we have
\be
\jump{h}= - \jump{ \tfrac{1}{2}v^2}  = -\frac{j^2}{2} \jump{V^2}= \frac{1}{2}\jump{p} \frac{  \jump{V^2} }{\jump{V} }=  \jump{p} V_{ \rm ave},
\ee
where $V_{ \rm ave}= \frac{1}{2} (V_-+ V_+)$.  Combining with \eqref{heqn1}, after some manipulation we find
\be\label{seqn1}
T_+ \jump{S} = \frac{1}{2} \jump{V} \jump{p} - \frac{1}{2}  \left. \left(\frac{\partial V}{\partial p_+}\right)\right|_S \jump{p}^2  - \frac{1}{6}  \left. \left(\frac{\partial^2V}{\partial p_+^2}\right)\right|_S \jump{p}^3 +  \mathcal{O} (\jump{S}   \jump{p}, \jump{p}^4 , \jump{S}^2).
\ee
Finally,  Taylor expanding the specific volume yields
\be
\jump{V} =\left. \left(\frac{\partial V}{\partial p_+}\right)\right|_S \jump{p}  + \frac{1}{2}  \left. \left(\frac{\partial^2V}{\partial p_+^2}\right)\right|_S \jump{p}^2+\mathcal{O}( \jump{p}^3,\jump{S}).
\ee
Upon substitution into \eqref{seqn1}, we obtain the relation  \eqref{spcube}.
Note that provided ${\partial^2V}/{\partial p^2}>0$, for weak shocks $\jump{p}\ll 1$, equation \eqref{spcube} shows that  $\jump{p}>0$. Hence, by \eqref{PVrel}, we have  $\jump{\rho}>0$  
and $\jump{u_n}>0$.
\end{proof}

\subsection{Lax geometric entropy conditions and determinism of shock development }

In this section, we show that the entropy condition implies that the shock discontinuity is supersonic relative
to the state ahead (`$+$' phase) and subsonic relative to the state behind  (`$-$' phase)
\begin{align} \label{lax1}
u^+ \cdot n + c^+  &  < \dot \sc < u^- \cdot n + c^- \,,
\end{align} 
where $c^-$ and $c^+$ are the sound speeds behind and at the front of the shock.
In this way, the $\{t=0\}$ hypersurface is the Cauchy surface for the state ahead $(+)$ whereas $\{t=0\}$ together with the shock front serve as the Cauchy surface for the state behind $(-)$.  The region behind the shock is thus determined by the initial conditions together with data along the shock front which are determined by enforcing Rankine-Hugoniot conditions.    

  Equations \eqref{lax1} (together with \eqref{lax2})  are called  \emph{Lax's geometric entropy conditions}. 
We now show that the Lax geometric entropy conditions are \emph{equivalent} to the physical entropy condition  \eqref{entcond}, at least for weak shocks.

\begin{lemma}
\label{lem:Lax}
In the setting of Lemma \ref{lemmaent}, the physical entropy condition \eqref{entprod} holds if and only if the geometric Lax entropy conditions \eqref{lax1} and \eqref{lax2} hold.  
\end{lemma}

\begin{proof}[Proof of Lemma~\ref{lem:Lax}]
Conditions \eqref{lax2}  hold since the mass flux $j<0$. Using $u_\pm\cdot n -\dot \sc  = j V^\pm$,  \eqref{lax1} becomes
\be\label{causality}
\tfrac{c^+}{V^+} < -j < \tfrac{c^-}{V^-}.
\ee 
Thus, when the jump conditions, the Lax geometric conditions hold provided
 \be\label{detjump}
 \jump{c/V}>0.
 \ee
We now show how this is implied by  $\jump{S}>0$ in the weak shock regime. Letting $w:=V^2/c^2$, we have
\be
\jump{w}
= 
  \left(\tfrac{1}{(c/V)^-}+ \tfrac{1}{(c/V)^+}\right) \jump{V/c} =  - \left(\tfrac{1}{(c/V)^-}+ \tfrac{1}{(c/V)^+}\right)  \tfrac{\jump{c/V}}{(c/V)^-(c/V)^+}.
\ee
 Thus, verifying condition \eqref{detjump} and thus \eqref{causality} is equivalent to showing $\jump{w} <0$.
To verify this note first, that viewing $\rho:=\rho(p, S)$,  as an application of the chain rule we have
\be
\frac{1}{c^2} =\left.\left( \frac{\partial\rho }{\partial p }\right)\right|_S  = -\frac{1}{V^2}\left.\left( \frac{\partial V}{\partial p}\right)\right|_S  ,
\ee
which yields $w = -\left.\left( \frac{\partial V}{\partial p}\right)\right|_S $.
Appealing to the leading order entropy jump \eqref{spcube} of Lemma  \ref{lemmaent}, we obtain
\begin{align}\label{wcond}
\jump{w} &= - \left.\left( \frac{\partial^2 V}{\partial p_+^2}\right)\right|_S  \jump{p} + \mathcal{O}(\jump{p}^2)= - \frac{12 T_+}{\jump{p}^2} \jump{S}  + \mathcal{O}(\jump{p}^2).
\end{align}
Thus, we see that  $\jump{S}>0$ if and only if  $\jump{w} <0$ which in turn implies the Lax conditions \eqref{lax1},  \eqref{lax2}.
\end{proof}

\begin{remark}[Determinism of shock development and entropy conditions]
We now discuss an interpretation of the Lax geometric inequalities as they pertain to the issue of determinism of the shock development problem.  To simplify ideas, we specialize to 1D setting in which the spacetime shock curve is given by $\{x = \sc(t)\}$. The spacetime normal to the shock curve  is
$
\sn =(-\dot \sc ,1).
$
With the notation $ \nabla_{t,x}= (\partial_t, \partial_x)$,  the transport operators for the Riemann invariants are
\be
(1,u-c)\cdot \nabla_{t,x}, \qquad (1,u)\cdot \nabla_{t,x}, \qquad (1,u+c)\cdot \nabla_{t,x}.
\ee
See equations  \eqref{xland-z}, \eqref{xland-k} and \eqref{xland-w} respectively.
In the front of the shock ($+$ phase), the Lax inequalities \eqref{lax1} read
\begin{align} 
u^+ - c^+    < \dot \sc , \qquad u^+    < \dot \sc , \qquad u^+  + c^+    < \dot \sc 
\end{align} 
all of which follow directly using the fact that the sounds speed is positive. Geometrically, these translate to
\be
\sn\cdot (1, u^+ -c^+)<0, \qquad \sn\cdot(1, u^+)<0, \qquad \sn\cdot (1,u^+ +c^+)<0,
\ee
showing that all the associated characteristics in front of the shock ($+$ phase) impinge on the shock front, carrying with they Cauchy data from the $\{t=0\}$ hypersurface.  This ensures that the front of the shock is causally isolated from shock and determined solely from initial conditions.  On the other hand, behind the shock ($-$ phase) we have from  \eqref{lax1} and  \eqref{lax2} that
\begin{align} 
u^-- c^-    < \dot \sc , \qquad u^-    < \dot \sc , \qquad u^- + c^-    > \dot \sc 
\end{align} 
which has the geometric meaning of
\be
\sn\cdot (1,u^- -c^-)<0, \qquad \sn\cdot (1, u^-)<0, \qquad \sn\cdot (1, u^- +c^-)>0.
\ee
Unlike the situation in the $+$ phase, we see that two of the characteristics corresponding to wave speeds $u^- -c^-$ and $u^-$ are ``exiting the shock", carrying with them data from along shock hypersurface. Only one of the characteristics corresponding to $u^- +c^-$ is impinging on the surface, carrying Cauchy data from $\{t=0\}$.  
The significance of this is the following: the data \emph{along the shock front} for the Riemann invariants carried by characteristics leaving the shock are free and will be chosen to enforce two out of the three jump conditions for mass, momentum and energy.  The third invariant whose characteristics impinge on the shock enjoys no such freedom -- rather the speed of the shock will be designed to arrange for the last jump condition to be satisfied.  Simultaneously  ensuring these constraints hold define a free boundary problem for the shock development.  If additional characteristics were to lack this freedom, the problem would become overdetermined and no solution could be found in general.  As such, the entropy condition is precisely what is required for the shock development problem to be ``deterministic".
\end{remark}

\begin{remark}[Shock speed near formation]\label{shockcurvest}
From the Rankine-Hugoniot conditions, it follows that the rate of propagation of weak shock waves (relative to the fluid) is the sound speed, $\dot{\sc} \approx u_n + c$.  This follows from the fact that, at the pre-shock, $v_-=v_+$  so 
\be
v_-= v_+ = v= j V = -\sqrt{-V^2 (\partial p/\partial V)|_S} =  -\sqrt{ (\partial p/\partial \rho)|_S} = -c,
\ee
which follows from the identity $j^2 = -\jump{p}/\jump{V}$.
Since $\dot{\sc}= u_n - v$, the claim follows.
\end{remark}

\subsection{The Euler system in terms of entropy, velocity,  and sound speed}

In preparation for reducing the equations to a symmetry class and deriving equations of motion for the Riemann variables, we reformulate
the two-dimensional non-isentropic compressible Euler equations.  First, for classical solutions the energy equation can be replaced by the transport of entropy 
\begin{subequations}
\label{eq:Euler}
\begin{align}
\partial_t (\rho u) + \div (\rho\, u \otimes u) + \nabla p(\rho) &= 0 \,,  \label{eq:momentum} \\
\partial_t \rho  + \div (\rho u)&=0 \,,  \label{eq:mass} \\
\partial_t \kcal  + u \cdot \nabla \kcal &=0 \,,  \label{eq:entropy2}
\end{align}
\end{subequations} 
where 
$\kcal: \mathbb{R}^2  \to \mathbb{R}$ is the (specific) entropy.
If the initial entropy is chosen to be a constant $\kcal_0\in \mathbb{R}$, then the entropy function  satisfies $\kcal(\cdot , t)= \kcal_0$ as  long as the solution remains smooth. 
The formulation of Euler given in \eqref{eq:Euler} is equivalent to the usual conservation law form (see \eqref{euler-weak}) up to the pre-shock, and will be used for the shock
formation process.

We introduce the adiabatic exponent
$$
\alpha = \tfrac{ \gamma -1}{2}  \,
$$
so that
the (rescaled) sound speed reads
\begin{align} 
\sigma = \tfrac{1}{\alpha }  \sqrt{ \sfrac{\p p}{\p \rho} } =  \tfrac{1}{\alpha }  e^{{\frac{\kcal}{2}} } \rho^ \alpha \,.
\label{sigma1}
\end{align} 
With this notation, the ideal gas equation of state \eqref{peos} becomes 
\begin{align} 
p= \tfrac{\alpha ^2}{\gamma} \rho \sigma^2 \,. \label{p1}
\end{align} 
The Euler equations \eqref{eq:Euler}  as a system for $(u, \sigma , \kcal)$ are then given by
\begin{subequations}
\label{eq:Euler2}
\begin{align}
\p_t u + (u \cdot \nabla) u + \alpha \sigma  \nabla \sigma  &=  \tfrac{\alpha }{2\gamma} \sigma^2  \nabla \kcal \,,  \label{eq:momentum2} \\
\partial_t \sigma + (u \cdot \nabla) \sigma  + \alpha \sigma \operatorname{div}u&=0 \,,  \label{eq:mass2} \\
\partial_t \kcal  +  (u \cdot\nabla) \kcal &=0 \,.  \label{eq:entropy22}
\end{align}
\end{subequations}

We let $\omega=\nabla^\perp\cdot  u$ denote the scalar vorticity, and 
define the {\it specific vorticity}  by $ \zeta = \tfrac{\omega}{\rho} $.  A straightforward computation shows that $\zeta$ is a solution to
\begin{align} 
\p_t \zeta  + (u \cdot \nabla) \zeta
=  \tfrac{\alpha }{\gamma} \tfrac{\sigma}{\rho} \nabla^\perp \sigma \cdot   \nabla \kcal   \,.   \label{specific-vorticity}
\end{align} 
The term term $\tfrac{\alpha }{\gamma} \tfrac{\sigma}{\rho} \nabla^\perp \sigma \cdot   \nabla \kcal  $ on the right side of \eqref{specific-vorticity}
can also be written as $\rho^{-3}  \nabla^\perp \rho \cdot  \nabla p$ and is referred to as {\it baroclinic torque}.

\subsection{Jump formulas for ideal gas equation of state}

In this section, we perform some manipulations of the  Rankine-Hugoniot conditions  \eqref{RH1}--\eqref{RH3} which will be used later in the paper.
Combining \eqref{rh-all} together with  \eqref{k-plus-zero}, we find that
\begin{align} 
\jump{p} =  - \tfrac{2 \rho_+^\gamma}{ (\gamma-1) \rho_- - (\gamma+1) \rho_+} \jump{\rho} \,. \label{pjump1}
\end{align} 
We can also compute the jump in pressure as
\begin{align} 
\jump{p}= \tfrac{1}{\gamma}  (e^{\kcal_-}-1) \rho_-^\gamma +  \tfrac{1}{\gamma}  \jump{\rho^\gamma} \,. \label{pjump2}
\end{align} 
Equating \eqref{pjump1} and \eqref{pjump2}, we see that
\begin{align} 
  \rho_-^\gamma ( e^{\kcal_-} -1) =   - \tfrac{2\gamma \rho_+^\gamma}{ (\gamma-1) \rho_- - (\gamma+1) \rho_+} \jump{\rho} -   \jump{\rho^\gamma} 
  \,,
  \label{pjump3}
\end{align} 
where we recall that $\kcal_- = \jump{\kcal}$. In order to simplify \eqref{pjump3}, we introduce 
\begin{align}
 Q = \frac{\rho_+}{\rho_-}
\end{align}
which we expect to be close to $1$ on the shock curve, for a short time after the pre-shock. Then, \eqref{pjump3} reads
\begin{align} 
  e^{\kcal_-} -1  =    \tfrac{  (Q-1)^3 }{ (\gamma-1)   - (\gamma+1) Q}  \left(\tfrac{\gamma (\gamma-1)(1+\gamma)}{6} - (Q-1) B_\gamma(Q) \right)\,,
  \label{pjump4}
\end{align} 
where $B_\gamma(Q)$ is a smooth function in the neighborhood of $Q=1$, with $B_\gamma(1) =  \frac{1}{12} (\gamma-2)(\gamma-1) \gamma (\gamma+1)$ and $B_\gamma'(1) = \frac{-1}{40} (\gamma-3)(\gamma-2)(\gamma-1) \gamma (\gamma+1)$. 

When $\gamma = 2$ and $\alpha = \frac 12$, the above formulae simplify. First we note that $B_2 (Q) = 0$ for all $Q$, and in that case, \eqref{pjump4} becomes
\begin{align} 
e^{\kcal_-} -1  =    \frac{  (Q-1)^3 }{ 1   - 3 Q}  =  \frac{ \jump{\rho}^3 }{\rho_-^2(3 \rho_+  -  \rho_-  )}   \,.
  \label{pjump5}
\end{align} 
From \eqref{sigma1} and the fact that $S_+ = 0$, we have that 
$$
\rho_-  = \tfrac{1}{4} \sigma_- ^2  e^{-\kcal_-} \,, \qquad  \rho_+ = \tfrac{1}{4} \sigma_+^2   \,, 
$$
from which it follows that
$$
\jump{\rho} = \tfrac 14 e^{-\kcal_-} \left(  \sigma_-^2 -  e^{\kcal_-} \sigma_+^2 \right) \,.
$$
This allows \eqref{pjump5} to be rewritten as
\begin{align} 
( e^{\kcal_-} -1)  \sigma_-^4(3 \sigma_+^2  e^{\kcal_-}   -  \sigma_-^2  )  =  \left(  \sigma_-^2 -  e^{\kcal_-}  \sigma_+^2\right)^3   \,.
  \label{pjump7old}
\end{align}

\section{Azimuthal symmetry}

\subsection{The Euler equations in polar coordinates and azimuthal symmetry}\label{azisec}

The 2D Euler equations
 \eqref{eq:Euler2}  take the following form in polar coordinates for the variables $ ( u_\theta, u_r, \rho, \kcal)$:
\begin{subequations}
\label{eq:Euler:polar}
\begin{align}
\left(\partial_t  + u_r\partial_r + \tfrac{1}{r} u_\theta \partial_{\theta}\right) u_r -\tfrac{1}{r}u_{\theta}^2+ \alpha  \sigma \partial_r  \sigma&= \tfrac{\alpha }{2 \gamma }\sigma^2 \p_r \kcal  \,, \\
\left(\partial_t  + u_r\partial_r +\tfrac{1}{r} u_\theta \partial_{\theta}\right)u_\theta+\tfrac{1}{r}u_r u_\theta + \alpha \tfrac{ \sigma}{r}\partial_\theta\sigma&=
\tfrac{\alpha }{2 \gamma }\tfrac{\sigma^2}{r} \p_\theta \kcal  \,,  \\
\left(\partial_t  + u_r\partial_r + \tfrac{1}{r} u_\theta \partial_{\theta}\right) \sigma + \alpha\sigma\left( \tfrac{1}{r} u_r + \partial_r u_r + \tfrac{1}{r} \p_\theta u_\theta \right)  &=0 \,,\\
\left(\partial_t  + u_r\partial_r + \tfrac{1}{r} u_\theta \partial_{\theta}\right) \kcal &=0.
\end{align}
\end{subequations}
We introduce the new variables
\begin{equation}\label{scale0}
u_\theta(r,\theta,t) = r b( \theta, t) \,, \ \ u_r(r,\theta,t) = r a( \theta, t)\, , \  \ \sigma(r,\theta,t) =r \ss(\theta,t),  \ \  \kcal(r,\theta,t) =k(\theta,t) \,.
\end{equation} 
The system \eqref{eq:Euler:polar} then takes the form
\begin{subequations}
\label{eq:Euler:polar3}
\begin{align}
\left(\partial_t  + b\partial_{\theta}\right) a  + a^2-b^2+ \alpha   \ss^2 &=0 \label{g3_a_evo}\\
\left(\partial_t  + b\partial_{\theta}\right)b + \alpha  \ss \partial_\theta\ss +2a b&=\tfrac{\alpha }{2 \gamma }\ss^2 \p_\theta k \\
\left(\partial_t  + b\partial_{\theta}\right) \ss +   \alpha \ss  \p_\theta b + \gamma a \ss &=0 \,  \label{sigma-eqn} \\
\left(\partial_t  + b\partial_{\theta}\right) k&=0 \, .
\end{align}
\end{subequations}
For simplicity of presentation we shall henceforth focus\footnote{The pre-shock formation for general $\gamma>1$ in \eqref{eq:Euler:polar3} was already done in~\cite{BuShVi2019a} for an open set of smooth isentropic initial data. Using the arguments in \cite{BuShVi2020}, the same result may be obtained also for the non-isentropic problem. The more detailed information required for shock-development can be obtained in analogy with the analysis in Section~\ref{sec:formation}.
The shock development problem for general $\gamma>1$ is conceptually the same; see the outline of the proof in Section~\ref{sec:outline}. One of the main differences is that the slightly more complicated Rankine-Hugoniot condition \eqref{pjump4} must be used in place of \eqref{pjump5}. Another difference is that for general $\gamma>1$, in the formation part the  subdominant Riemann variable is not  transported and thus cannot be taken to equal  a constant up to the pre-shock; this issue was already addressed in~\cite{BuShVi2019a,BuShVi2019b,BuShVi2020}.} on the case  
$$
\gamma=2 \ \ \text{ and } \ \ \alpha = \tfrac{1}{2} \,.
$$
The Riemann functions  $w$ and $z$  are defined by
\begin{subequations} 
\label{eq:riemann}
\begin{alignat}{2}
w&= b+  \ss  \,, \qquad  &&z= b- \ss  \,, \\
b&= \tfrac{1}{2} (w+z) \,, \qquad && \ss= \tfrac{1}{2} (w-z) \,.
\end{alignat}   
\end{subequations} 
It is convenient to rescale time, letting $\p_t \mapsto \tfrac{3}{4} \p_{\tilde t}$, and for notational simplicity, we continue to write $t$ for $\tilde t$. 
With this temporal rescaling employed,  the system  \eqref{sigma-eqn} can be equivalently
written as 
\begin{subequations} 
\label{eq:w:z:k:a} 
\begin{align}
\p_t w + \lambda_3 \p_\theta w & = - \tfrac{8}{3}  a w + \tfrac{1}{24}(w-z)^2 \p_\theta k   \,,  \label{xland-w} \\
\p_t z + \lambda_1 \p_\theta z & = - \tfrac{8}{3}  a z + \tfrac{1}{24}(w-z)^2 \p_\theta k   \,,  \label{xland-z} \\
\partial_t k   + \lambda_2 \partial_{\theta} k & = 0  \,, \label{xland-k} \\
\partial_t a   + \lambda_2  \partial_{\theta} a & = - \tfrac43 a^2 + \tfrac{1}{3} (w+ z)^2 - \tfrac{1 }{6} (w- z )^2  \,. \label{xland-a} 
\end{align}\end{subequations} 
where the three wave speeds are given by 
\begin{align}
\lambda_1 &=  \tfrac{1}{3} w + z   \,, \qquad
\lambda_2 = \tfrac{2}{3} w+  \tfrac 23 z \,, \qquad 
\lambda_3 =  w+  \tfrac{1}{3}  z\,.   \label{eq:wave-speeds}
\end{align} 
We note that \eqref{sigma-eqn} takes the form
\begin{align} 
\partial_t \ss  + \lambda_2  \partial_{\theta} \ss +\tfrac{1}{2}\ss \p_\theta \lambda _2  & = -\tfrac{8}{3}a \ss \,. \label{xland-sigma}
\end{align} 
Finally, we denote the specific vorticity in azimuthal symmetry by
\begin{align}
\varpi = 4 (w + z - \p_\theta a) c ^{-2} e^k\,,
\label{xland-svort:def}
\end{align}
which satisfies the evolution equation
\begin{align} 
\p_t \varpi + \lambda_2\p_\theta \varpi =   \tfrac{8}{3} a\varpi  +  \tfrac{4}{3}  e^k \p_\theta k  \,. \label{xland-svort}
\end{align}

We supplement \eqref{eq:w:z:k:a}
with initial conditions
\begin{align*} 
w_0(\theta) = w(\theta,T_0) \,, \ \ \  z_0(\theta) = z(\theta,T_0) \,, \ \ \ 
a_0(\theta) = a(\theta,T_0)\,, \ \ \  k_0(\theta)=k(\theta,T_0) \,, \ \ \ \varpi_0(\theta)= \varpi(\theta,T_0)\,.
\end{align*}

We shall study the shock formation process for solutions to \eqref{eq:w:z:k:a}  on
the time interval $T_0 < t \le T_1$, where $T_1$ denotes the time of the first singularity, also known as
the {\it pre-shock}.   One of our main objectives is to provide a detailed description of the pre-shock $w( \cdot , T_1)$.   We shall provide the fractional
series expansion of $w(\theta , T_1)$ for $\theta$ in a neighborhood of the blowup location $\theta_*$.

For the shock formation process, we choose initial data\footnote{This choice is made for the following reason: irregardless of the choice of initial entropy
function $k_0$, the Rankine-Hugoniot conditions guarantee that a jump in entropy {\it must occur} at the shock.
 As such the choice of $k_0=0$ emphasizes the production of entropy in the clearest possible terms.   Similarly, the choice
of $\gamma=2$ and that $k_0=0$ allows the equation \eqref{xland-z} to reduce to a transport-type equation.  Just as we did for entropy, we can (in this case) choose
$z_0=0$ and up to the pre-shock, the sub-dominant Riemann variable $z$ will remain zero.   Once again the Rankine-Hugoniot conditions ensure that
$z$ must experience a jump discontinuity along the shock, and thus the choice of $z_0=0$ allows us to most easily demonstrate this fact.
}
$$
k_0(\theta)=0\,, \ \ \ z_0(\theta) =0 \,,
$$
which is preserved by the dynamics so that $k(\theta, t)=0$ and $z(\theta,t)=0$ for all time $t$ up to the time time of the pre-shock.   Thus \eqref{eq:w:z:k:a} is reduced to 
a coupled system of equations for $a$ and $w$, satisfying
\begin{subequations} 
\label{eq:w:z:k:a2}
\begin{align}
\p_t w + w \p_\theta w & = - \tfrac{8}{3} a w  \,,  \label{xland2-w} \\
\p_t a + \tfrac{2}{3} w \p_\theta a &= - \tfrac{4}{3} a^2 + \tfrac{1}{6} w^2 \,. \label{xland2-a}
\end{align}
\end{subequations} 

\subsection{The Rankine-Hugoniot jump conditions under azimuthal symmetry}
\label{sec:jump!}
Under the azimuthal symmetry assumptions and using our temporal rescaling $\tilde t \mapsto \tfrac{3}{4}t$,  from \eqref{scale0} (fixing  $\gamma=2$), we have that
 the shock hypersurface is given as the graph $\{(r,\theta,t) \ : \  \theta = \sc(t)\}$. The spacetime normal to this curve is $\sn=(-\dot{\sc}, \frac{1}{r})$. Thus, $\sc$ satisfies
the Rankine-Hugoniot  conditions \eqref{RH1} and \eqref{RH2}
\begin{subequations}\label{RH_condition2}
\begin{align}
\dot\sc  &= \tfrac{4}{3}   \frac{ \jump{e^{-k} \ss^2 b^2 + \tfrac{1}{8}  e^{-k} \ss^4} }{\jump{e^{-k} \ss^2 b}}\,, \\
\dot\sc  &=\tfrac{4}{3}   \frac{ \jump{e^{-k} \ss^2 b} }{ \jump{e^{-k} \ss^2}}  \,.
\end{align}
\end{subequations}
We note that the third Rankine-Hugoniot condition \eqref{RH3} has already been employed to deduce the relation \eqref{pjump7old}.

Let us now convert \eqref{RH_condition2} and \eqref{pjump7old} into our azimuthal  variables as follows.  We denote by $w_\pm( \cdot , t)$, $z_\pm(\cdot , t) $, $k_\pm( \cdot , t)$ 
the  limiting values, from the left ($-$) and right ($+$), of the shock curve $\sc(t)$.   We also note the fact that $k_+=0$ and $z_+=0$.
Now, from  \eqref{eq:riemann}, the system \eqref{RH_condition2} becomes
\begin{subequations}
\begin{align} 
\dot \sc(t) & =  \frac{2}{3}  
 \frac{ e^{-k_-} (w_- -z_-)^2 (w_- + z_-)^2 + \tfrac 18 e^{-k_-} (w_- - z_-)^4 -  \tfrac{9}{8}   w_+^4 
}{ e^{-k_-} (w_- -z_-)^2(w_- + z_-) -   w_+^3 } \,, \label{sdot2}
\\
\dot \sc(t) & =   
 \frac{2}{3}    \frac{ e^{-k_-} (w_- -z_-)^2 (w_-+z_-) -  w_+^3}{ e^{-k_-} (w_- -z_-)^2 -  w_+^2 } \,. \label{sdot1} 
\end{align} 
\end{subequations}
We note that the jump conditions \eqref{sdot2} and \eqref{sdot1} for the mass and the momentum equations are a priori two different equations for the shock speed. To remedy this, we set the right sides of these equations equal to each other, and instead work with one evolution equation for $\dot{\sc}$, namely \eqref{sdot1}, and one constraint
\begin{subequations}
\begin{align}
&\left(  (w_- -z_-)^2 (w_- + z_-)^2 + \tfrac 18  (w_- - z_-)^4 -  \tfrac{9}{8} e^{k_-}  w_+^4 \right)
\left( (w_- -z_-)^2 - e^{k_-} w_+^2 \right) \notag\\
&=
\left( (w_- -z_-)^2(w_- + z_-) - e^{k_-}  w_+^3 \right)^2
  \label{pjump77}
\end{align}
Also, we have that \eqref{pjump7old} 
takes the form
 \begin{align} 
(e^{k_-} - 1)  (w_- - z_-)^4 \Bigl(3 w_+^2  e^{k_-}  -  (w_- - z_-)^2  \Bigr)  =  \left(  (w_- - z_-)^2 -  e^{k_-}  w_+^2\right)^3   \,.
  \label{pjump7}
\end{align} 
\end{subequations}
To summarize, we shall first use the system formed by the equations \eqref{pjump77} and \eqref{pjump7} in order to solve for $z_-$ and $k_-$ in terms of $w_-$ and $w_+$, and then insert these solutions into \eqref{sdot1} and determine an evolution equation for $\sc$, solely in terms of $w_-$ and $w_+$. This is discussed in Section~\ref{sec:z:k:shock}.

\subsection{Main result in azimuthal symmetry}
As mentioned in Theorem~\ref{thm:main:soft}, in the {\em formation part} of our result, i.e. for $t \in [T_0,T_1)$, we have that the solution $(w,z,k,a)$ of the  Euler equations in azimuthal symmetry is smooth, so that the notion of solution is the classical one: the system \eqref{eq:w:z:k:a} is satisfied in the sense of $C^1$ functions of space and time. On the time interval $[T_1,T_2]$, which covers the {\em development part} of our result, the notion of {\em regular shock solution} is used, as defined by Definition~\ref{shockdef} above. In azimuthal symmetry, this definition becomes:
\begin{definition}[\bf Regular azimuthal shock solution]
\label{def:sol:azimuthal}
We say that  $(w,z,k,a,\sc)$  is a {\em regular azimuthal shock solution} on $\mathbb{T} \times [T_1,T_2]$ if
\begin{enumerate}
\item $(w,z,k,a)$ are $C^1_{\theta,t} $ smooth, and $\varpi$ is $C^0_{\theta,t}$ smooth, on the complement of the shock curve $\{ \theta = \sc(t)\}$;
\item on the complement of the shock curve $(w,z,k,a)$ solve the equations \eqref{eq:w:z:k:a} pointwise, and $\varpi$ solves \eqref{xland-svort} pointwise;
\item $(w,z,k)$ have jump discontinuities across the shock curve which satisfy the algebraic equations  \eqref{pjump77}, \eqref{pjump7} arising from the Rankine-Hugoniot conditions;
\item the shock location $\sc:[T_1,T_2]\to \mathbb{T}$ is $C^1_{t} $ smooth and solves  \eqref{sdot1}.
\end{enumerate}
\end{definition}

Our main result for the azimuthal 2D Euler equations~\eqref{eq:w:z:k:a} is stated in detail in Theorems~\ref{thm:main:development} and~\ref{sec:C2}; here we only give a condensed statement: 
 
\begin{theorem}[\bf Main result in azimuthal symmetry -- abbreviated version]
\label{thm:main:azi:soft}
From smooth isentropic initial data with vanishing subdominant Riemann variable at time $T_0$, there exist smooth solutions to the azimuthal Euler system \eqref{eq:w:z:k:a} that form a pre-shock singularity, at a time $T_1 > T_0$.
The first singularity occurs at a single point in space, $\theta_*$, and this first singularity is shown to have an asymptotically self-similar shock profile exhibiting a $C^{\frac{1}{3}}$ cusp in the dominant Riemann variable velocity and a $C^{1,\frac{1}{3}}$ in the radial velocity. After the pre-shock, the solution  to \eqref{eq:w:z:k:a} is continued  for a short time $(T_1,T_2]$ as a regular azimuthal shock solution (cf.~Definition~\ref{def:sol:azimuthal}) with the following properties:
\begin{itemize}
\item Across the shock curve $\sc$, all the state variables jump
\begin{align*}
\jump{w} \sim (t-T_1)^{\frac{1}{2}}, \qquad \jump{\partial_\theta a} \sim (t-T_1)^{\frac{1}{2}},\qquad 
\jump{z} \sim (t-T_1)^{\frac{3}{2}}, \qquad \jump{k} \sim (t-T_1)^{\frac{3}{2}}
\end{align*}
for $t\in (T_1,T_2]$.
\item Across the characteristic $\sc_2$ emanating from the pre-shock and moving with the fluid velocity,  the Riemann variables and the entropy make $C^{1,\frac{1}{2}}$ cusps  approaching from the right side.  Approaching from the left side, are these variables are $C^2$ smooth.   
\item Across the characteristic $\sc_1$ emanating from the pre-shock and moving with the sound speed minus the fluid velocity, the entropy is zero while the subdominant Riemann variable $z$ makes a $C^{1,\frac{1}{2}}$ cusp approaching from the right.  Approaching from left, they all variables are $C^2$ smooth on $(T_1,T_2]$. 
\end{itemize}
\end{theorem}

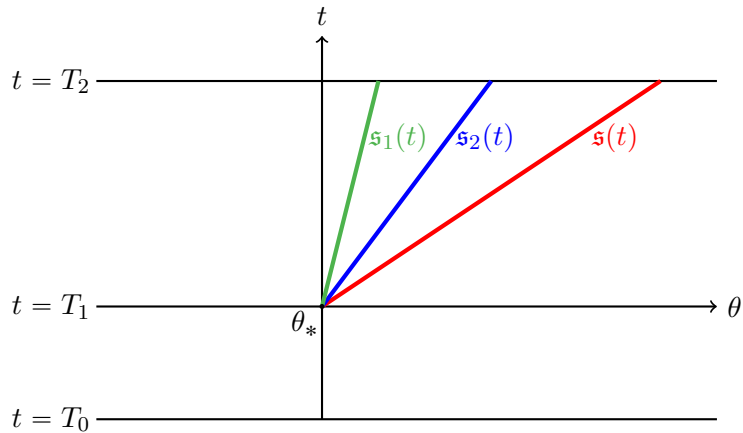
\begin{figure}[htb!]
\centering
\begin{tikzpicture}[scale=1.5]
    \draw [<->,thick] (0,2.4) node (yaxis) [above] {$t$}
        |- (3.5,0) node (xaxis) [right] {$\theta$};
    \draw[black,thick] (0,0)  -- (-2,0) ;
    \draw[black,thick] (0,0)  -- (0,-1) ;
     \draw[black,thick] (-2,-1)  -- (3.5,-1) ;
     \draw[black,thick] (-2,2)  -- (3.5,2) ;
    \draw[name path=A,red,ultra thick] (0,0) -- (3,2) ;
        \draw[red] (2.6,1.5) node { $\sc(t)$}; 
    \draw[name path = B,blue,ultra thick] (0,0)  -- (1.5,2) ;
    \draw[blue] (1.45,1.5) node { $\sc_2(t)$}; 
     \draw[green!40!gray, ultra thick] (0,0)  -- (.5,2) ;
      \draw[green!40!gray] (.68,1.5) node { $\sc_1(t)$}; 
     \draw[black] (-2.4,0) node { $t=T_1$}; 
      \draw[black] (-2.4,2) node { $t=T_2$}; 
        \draw[black] (-2.4,-1) node { $t=T_0$}; 
        \filldraw[black] (0,0) circle (0.5pt);
         \draw[black] (-0.15,-0.15) node { $\theta_*$}; 
\end{tikzpicture}

\vspace{-0.2cm}
\caption{\footnotesize At time $T_0$ a smooth datum is given, which forms a first singularity at time $T_1$, at a single angle $\theta_*$; this is the pre-shock. For $t\in (T_1,T_2]$, we have three curves of singularities emerging from the point $(\theta_*,T_1)$: $\sc$ is a classical shock curve across which $(w,z,k,\p_\theta a)$ jump, and the Rankine-Hugoniot conditions are satisfied; along the characteristic curve $\sc_2$ the quantities $(w,z,k)$ have regularity $C^{1,1/2}$ and no better, while along the characteristic curve $\sc_1$, the function $z$ has regularity $C^{1,1/2}$ and no better.}
\label{fig:basic}

\end{figure}

\subsection{Outline of the proof} 
\label{sec:outline}
The proof of Theorem~\ref{thm:main:azi:soft} consists of five main steps, which we outline next. For simplicity, in this outline we focus only on the intuition behind the result, and skip over the technical difficulties which emerge when we turn this intuition into a complete proof. 

\vspace{.05in}
\noindent{\bf Step 1: detailed formation of first singularity, the pre-schock.}
The formation of the first gradient singularity for the Euler equations, from an  open set of smooth initial datum, was previously established in~\cite{BuShVi2019a,BuShVi2019b,BuShVi2020}. In azimuthal symmetry,~\cite{BuShVi2019a} shows that that the first singularity is characterized as an asymptotically self-similar $C_\theta^{1/3}$ cusp for the dominant Riemann variable $w$ defined in~\eqref{eq:riemann}; this is the so-called {\em pre-shock}.

In order to best illustrate a symmetry breaking phenomenon which occurs after the formation of the pre-shock, in this paper we consider smooth initial conditions for \eqref{eq:w:z:k:a} which are both isentropic ($k|_{t=T_0} \equiv 0$) and have vanishing subdominant Riemann variable ($z|_{t=T_0} \equiv 0$). Both of these conditions are propagated for smooth solutions (the interval $[T_0,T_1]$ in Figure~\ref{fig:basic}), but we shall prove that this symmetry is broken as soon as the shock forms (the interval $(T_1,T_2]$ in Figure~\ref{fig:basic}). From such smooth initial data, satisfying in addition a genericity condition on the initial gradient of the dominant Riemann variable, we construct a first singularity occurring at a point $(\theta_*,T_1)$. For simplicity of notation, this space-time location of the pre-shock is relabelled as $(0,0)$, and the solution $(w,z,k,a)|_{t=T_1}$ is denoted as $(w_0,z_0,k_0,a_0)$. From~\cite{BuShVi2019a} we have that at the  pre-shock, the solution takes the form
\begin{subequations}
\label{pre-shock}
\begin{align}
w_0(\theta) &= \kappa  - \bb  \theta^{\frac{1}{3}} + \dots,\label{pre-shocka}\\
a_0(\theta) &=  {\rm a}_0  + {\rm a}_1 \theta + {\rm a}_2 \theta^{\frac{4}{3}}+ \dots,\label{pre-shockb}\\
z_0(\theta) &=  0, \label{pre-shockc} \\
k_0(\theta) &=  0 \,, \label{pre-shocke}
\end{align}
\end{subequations}
asymptotically for $|\theta| \ll 1$. 
We note also that  specific vorticity $\varpi$ (see \eqref{xland-svort:def}) at the pre-shock is Lipschitz continuous; we denote it as $\varpi_0$. 

While for the schematic understanding of shock development the asymptotic expansions in \eqref{pre-shock} are sufficient, in order to rigorously capture the formation of higher order characteristic singularities emerging along the curves $\sc_1$ and $\sc_2$ in Figure~\ref{fig:basic}, a much finer understanding of the pre-shock is required. In particular, we need to show that the equality~\eqref{pre-shocka} holds in a $C^3$ sense; by this we mean that $w_0'(\theta) = - \frac 13 \bb \theta^{-\frac 23} + \ldots$, that $w_0''(\theta) =  \frac 29 \bb \theta^{-\frac 53} + \ldots$, and that $w_0'''(\theta) = - \frac{10}{27} \bb \theta^{-\frac 83} + \ldots$,  for $|\theta| \ll 1$. This information is not provided by our previous work~\cite{BuShVi2019a} and is established in Section~\ref{sec:formation} of this paper; here we combine the information provided by the self-similar analysis in~\cite{BuShVi2019a} with a Lagrangian perspective in unscaled variables for \eqref{eq:w:z:k:a2}, and the characterization of the pre-shock as the point in space time where  the   characteristic associated with the speed $\lambda_1$ has a vanishing first and second gradient (with respect to the Lagrangian label).

\vspace{.05in}
\noindent{\bf Step 2: emergence of shock front.}
By  Remark \ref{shockcurvest}, for short time $\dot{\sc} \approx u_n + c$.  Accounting for the temporal rescaling done in Section~\ref{azisec} (see paragraph above \eqref{eq:w:z:k:a}), this says $\dot{\sc} \approx b+ c = w$ close to the pre-shock, so that from \eqref{pre-shocka} we have
\begin{align*}
{\sc}(t)\approx \kappa  t .
\end{align*}
Entropy is produced as soon as the shock has developed, cf.~Lemma \ref{lemmaent}.  However, this contribution is small at small times, and thus the dynamics of $w$ (cf.~\eqref{xland-w}) near  the pre-shock can be roughly thought of as
\begin{subequations}
\label{approxburg1}
\begin{align}
\p_t w+ w \partial_\theta w&= \mbox{(small amplitude  error involving  entropy  gradients)},\label{approxburg1a} \\
w|_{t=0} &=  \kappa  - \bb  \theta^{\frac{1}{3}} + \mbox{(small error near pre-shock)}.
\end{align}
\end{subequations}
Note that the characteristics of this equation, the flow of $\partial_t + w \partial_\theta$, are to leading order  in time tangent  to the shock, if initiated at the pre-shock location. Otherwise, these characteristics impinge upon the shock  from either the left or right sides, since the pre-shock data ensures that the Lax entropy conditions \eqref{lax1} are satisfied.
As such, we can view the dominant Riemann variable $w$ as being a perturbation of an inviscid Burgers solution:
\begin{align}
\label{Bsol}
\wb (\etab( \theta,t),t) =  w_0 (\theta), \qquad \etab(\theta,t) = \theta+ t w_0(\theta).
\end{align}
A large part of the proof of Theorem~\ref{thm:main:development} is indeed dedicated to proving that the errors made in approximating equation~\eqref{approxburg1a} with the Burgers equation can indeed be controlled, in a $C^1$ topology of a suitable space-time. This part of the analysis uses in a crucial way the specific transport structure of the entropy gradient present on the right side of~\eqref{approxburg1a} or~\eqref{xland-w}, and the evolution equations for the good unknowns $q^w$ and $q^z$ defined in~\eqref{eq:lazy:cat} below, which relate the gradients of entropy to those of the Riemann variables and the sound speed. 
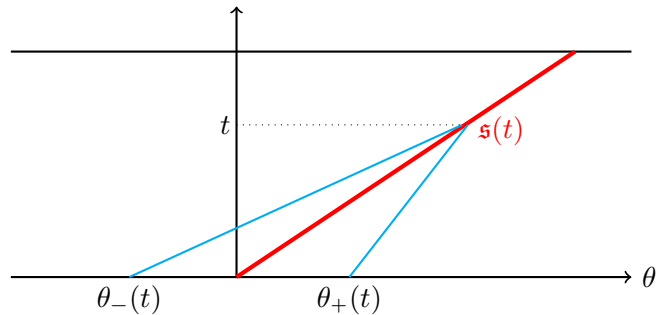
\begin{figure}[htb!]
\centering
\begin{tikzpicture}[scale=1.5]
    \draw [<->,thick] (0,2.4) node (yaxis) [above] {}
        |- (3.5,0) node (xaxis) [right] {$\theta$};
    \draw[black,thick] (0,0)  -- (-2,0) ;
     \draw[black,thick] (-2,2)  -- (3.5,2) ;
    \draw[name path=A,red,ultra thick] (0,0) -- (3,2) ;
        \draw[red] (2.35,1.3) node { $\sc(t) $};

\draw[cyan, thick] (1,0) -- (2.05,1.35) ;
\draw[black] (1,-.2) node { $\theta_+(t)$};

\draw[cyan, thick] (-0.95,0) -- (2,1.35) ;
\draw[black] (-0.95,-.2) node {$\theta_-(t)$};

\draw[black,dotted] (0,1.35)  -- (2.05,1.35);
\draw[black] (-0.1,1.35) node { $t$}; 
\end{tikzpicture}

\vspace{-0.2cm}
\caption{\footnotesize The shock curve is represented in bold red, while the paths $\{\etab(\theta_\pm(t),s)\}_{s\in[0,t]}$ are the cyan  paths.}
\end{figure}

The outcome of this analysis is that indeed we may approximate   $\jump{w}\approx \jump{\wb}$ where
\be\label{jumpw}
\jump{\wb}(t)= w_0(\theta_-(t))- w_0(\theta_+(t)),
\ee
where $\theta_\pm(t) = \etab^{-1}(\sc(t)^{\pm},t)$ are the locations of the labels of the particles which fell into the shock at time $t$. To find how these labels depend on the elapsed time, we use the expression for the Burgers flowmap \eqref{Bsol} near the pre-shock
\be\label{shockfronthit}
\etab(\theta,t)-  \kappa  t  \approx  \theta- \bb  t \theta^{\frac{1}{3}}
\ee
when $\etab(\theta,t)= \sc(t)$.  This yields $\theta_\pm(t) \approx \pm\left( \bb  t \right)^{\frac{2}{3}}$ and returning to \eqref{jumpw} we find
\be
\jump{w}(t)  \sim  {t}^{ \frac{1}{2}} .
\ee

\vspace{.05in}
\noindent{\bf Step 3: jumps of entropy and the subdominant Riemann variable on the shock front.}
In analogy to Lemma~\ref{lemmaent}, by choosing the smallest root of the system \eqref{pjump77}--\eqref{pjump7} it can be shown that in the weak shock regime $|\jump{w}| \ll 1$ which corresponds to short times after the pre-shock, the Rankine-Hugoniot conditions imply 
\begin{align}
- \jump{z}(t) \sim \jump{w}^3(t) \sim t^{ \frac{3}{2}} ,
\end{align}
for the subdominant Riemann variable, and similarly 
\begin{align}
\jump{k}(t) \sim \jump{w}^3 (t)\sim t^{\frac{3}{2}} \,,
\end{align}
for the jump in entropy along the shock front. 
As such, entropy and the subdominant Riemann variable are {\em produced instantaneously} along the shock in order to enforce that mass, momentum and total energy are not lost.  This is a manifestation of \emph{symmetry breaking}  associated to physical shocks, and emphasizing this point is the reason for the choice \eqref{pre-shockc}--\eqref{pre-shocke}.

At this point we note that since $a$ is being forced in \eqref{xland-a} by both $z$ and $w$, which themselves jump across $\sc(t)$, the function $a$ too exhibits a singularity on $\sc(t)$.
Ordinarily, this singularity might be expected to appear in $a$ itself, but since the characteristics of $a$ are transversal to the shock, together with the special structure of the specific vorticity evolution~\eqref{xland-svort}, we prove that $a$ is continuous across the shock, and that its derivative exhibits a jump discontinuity:
\begin{align}
\jump{\partial_\theta a}(t) \sim \jump{w} (t)\sim t^{\frac{1}{2}} \,.
\end{align}
An extended discussion of this point will appear in the next step.

\vspace{.05in}
\noindent{\bf Step 4: development of weak singularities.}
We use equations \eqref{eq:w:z:k:a} to determine the solution away from the shock curve.  In front of the shock (to the right in our case), the solution is determined by its initial data on the Cauchy surface $\{t=0\}$.  This is because all of the characteristic curves moving with velocities $\lambda_i$, $i=1,2,3$, as defined in \eqref{eq:wave-speeds}, impinge upon the shock front in that region, since the shock  is supersonic there. As such, in that region $z$ and $k$ are identically zero since they are zero initially and \eqref{xland-z}--\eqref{xland-k} have no forcing when $z=k=0$.
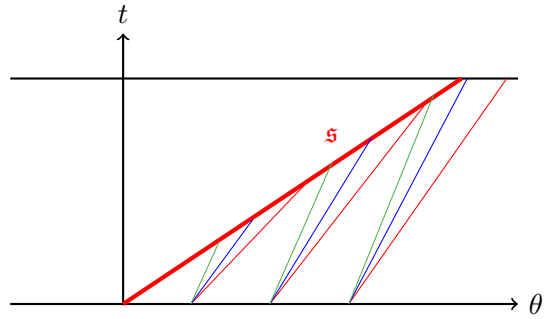
\begin{figure}[htb!]
\centering
\begin{tikzpicture}[scale=1.5]
\draw [<->,thick] (0,2.4) node (yaxis) [above] {$t$}
        |- (3.5,0) node (xaxis) [right] {$\theta$};
\draw[black,thick] (0,0)  -- (-1,0) ;
\draw[black,thick] (-1,2)  -- (3.5,2) ;
\draw[name path=A,red,ultra thick] (0,0) -- (3,2) ;
\draw[red] (1.85,1.5) node { $\sc$}; 
\draw[red] (0.6,0) -- (1.65,1.1) ;
\draw[red] (1.3,0) -- (2.7,1.8) ;
\draw[red] (2,0) -- (3.4,2) ;
\draw[blue] (0.6,0) -- (1.15,0.75) ;
\draw[blue] (1.3,0) -- (2.2,1.47) ;
\draw[blue] (2,0) -- (3.05,2) ;
\draw[green!40!gray] (0.6,0) -- (.85,0.55) ;
\draw[green!40!gray] (1.3,0) -- (1.85,1.25) ;
\draw[green!40!gray] (2,0) -- (2.75,1.85) ;
\end{tikzpicture}

\vspace{-0.2cm}
\caption{\footnotesize The characteristic curves of $\lambda_1 = w$ in front of the shock curve are represented in red, those of $\lambda_2 = \frac 23 w$ are in blue, and those of $\lambda_3 = \frac 13 w$ are plotted in green.}
\end{figure}

On the other hand, behind the shock (to the left side in our case), this is not the case. As discussed in \textbf{Step 3}, along the shock front $z$ and $k$ \emph{must} be produced in order to enforce the three Rankine-Hugoniot  jump conditions.  These values $z_-$ and $k_-$  are propagated off the shock along their characteristics with speeds $\lambda_1$ and $\lambda_2$ which are both \emph{slower} than the speed of the shock $\dot\sc(t)$.  As such, the surface $\{ \theta< 0 , t=0\} \cup \{\theta= \sc(t), t>0 \}$ serves as a new Cauchy surface for the $z,k,a$ equations   \eqref{eq:w:z:k:a}  once the shock has formed.  Schematically, the initial data on this new Cauchy surface is 
\begin{subequations}
\begin{align}\label{newdata1}
\tilde{z}_0(\theta) &\approx \begin{cases} 0 & \text{ on } \{\theta<0, t=0\} \\    \tilde{\rm z}_0   \theta^{ \frac{3}{2}} + \dots \phantom{{\rm ffffffas} } & \text{ on } \{\theta =  \kappa   t, \ t\geq 0 \}
\end{cases},\\ \label{newdata2}
\tilde{k}_0(\theta) &\approx \begin{cases} 0 & \text{ on } \{\theta <0, t=0\} \\    \tilde{\rm k}_0   \theta^{ \frac{3}{2}} + \dots \phantom{{\rm ffffffas} }& \text{ on } \{\theta =  \kappa   t, \ t\geq 0 \}
\end{cases},\\ \label{newdata3}
 \tilde{a}_0(\theta) &\approx \begin{cases}  \tilde{\rm a}_0 \theta +  \tilde{\rm a}_1   \theta^{ \frac{4}{3}} + \dots & \text{ on } \{\theta <0, t=0\} \\  {\rm smooth}  & \text{ on } \{\theta=  \kappa   t, \ t\geq 0 \}
\end{cases},
\end{align}
\end{subequations}
for some constants $ \tilde{\rm z}_0,  \tilde{\rm a}_0,  \tilde{\rm k}_0$, and for $|\theta| , t \ll 1$.
As discussed above, this data is carried away from the shock surface along characteristics which are slower than the shock.  The entropy is simply transported cf.~\eqref{xland-k}, whereas the subdominant Riemann variable is transported, self-amplified and forced by the entropy cf.~\eqref{xland-z}, and the radial velocity is forced by $a$, $w$, and $z$ cf.~\eqref{xland-a}.  

We begin by discussing what happens to the entropy.  Since its data \eqref{newdata2} is smooth away from the point $\theta=0$, the solution in the domain of influence of this region is likewise smooth.   Only across one single curve can the entropy be non-smooth: the $\lambda_2$-characteristic curve  $\sc_2(t)$ emanating from the pre-shock location $(0,0)$; see Figure~\ref{fig:s2!}. Along this curve, one may expect that the  $ \frac{3}{2}$--H\"{o}lder regularity of the Cauchy data $\tilde k_0$ is transported. Since at the initial time we have   $\lambda_2(0)\approx  \tfrac{2}{3}w_0$, due to \eqref{pre-shocka} at short times  we expect 
\begin{align*}
\sc_2(t) \approx   \tfrac{2}{3} \kappa t.
\end{align*}
The entropy exhibits a $C^{1,\frac{1}{2}}$ cusp singularity across $\{\theta= \sc_2(t) \}$, taking the approximate form
\begin{align}\label{entsol}
k(\theta,t) &\approx \begin{cases} 0, &\phantom{\sc_2(t) <}\  \theta< \sc_2(t) 
\\   3^{\frac 32}  \tilde{\rm k}_0    \left(\theta- \sc_2(t)\right)^{ \frac{3}{2}}, &   \sc_2(t) < \theta< \sc(t)
\\    0,& \phantom{\sc_2(t) <} \ \theta> \sc(t)
\end{cases}.
\end{align}
Note that along the shock curve $\sc(t)$ (for $t>0$) the entropy $k$ smoothly matches its generated values along shock given by \eqref{newdata2}; this is because $\sc(t) - \sc_2(t) \approx \frac{1}{3} \kappa t$.  We emphasize that equation \eqref{entsol} gives quite an accurate picture of the entropy for short times, even in the fully nonlinear problem; this fact is established in Sections~\ref{sec:development} and~\ref{sec:C2}, and the proof uses a precise understanding of the second derivative of the $\lambda_2$ wavespeed in the region between $\sc_2$ and $\sc$.

\begin{figure}[htb!]  
\centering
\begin{tikzpicture}[scale=1.5]
\draw [<->,thick] (0,2.4) node (yaxis) [above] {$t$}
        |- (3.5,0) node (xaxis) [right] {$\theta$};
\draw[black,thick] (0,0)  -- (-3,0) ;
\draw[black,thick] (-3,2)  -- (3.5,2) ;
\draw[name path=A,red,ultra thick] (0,0) -- (3,2) ;
\draw[red] (2.5,1.5) node { $\sc$}; 
\draw[name path = B,blue,ultra thick] (0,0)  -- (1.5,2) ;
\draw[blue] (0.9,1.5) node { $\sc_2$}; 
\draw[blue] (.85,0.58) -- (2,2);
\draw[blue] (1.75,1.17)  -- (2.5,2) ;
\draw[blue] (2.55,1.7) -- (2.8,2);
\draw[green!40!gray] (.85,0.58) -- (1.2,2);
\draw[green!40!gray] (1.75,1.17)  -- (2.1,2) ;
\draw[green!40!gray] (2.55,1.7) -- (2.7,2);
\draw[red]   (-0.5,0)-- (.85,0.58);
\draw[red]  (-1,0) -- (1.75,1.17) ;
\draw[red]  (-2,0)--  (2.55,1.7);
\end{tikzpicture}

\vspace{-0.2cm}
\caption{\footnotesize The entropy $k$ is propagated off the shock curve along the $\lambda_2$ characteristics represented by blue curves. The subdominant Riemann variable $z$ is also propagated off the shock curve $\sc$, but along the $\lambda_1$ characteristics represented in green. The $\lambda_1$ characteristics initiated at $\{t=0\}$, represented in red, impinge on the shock curve from the left side, determining $w$ in terms of $w_0$.}
\label{fig:s2!}

\end{figure}
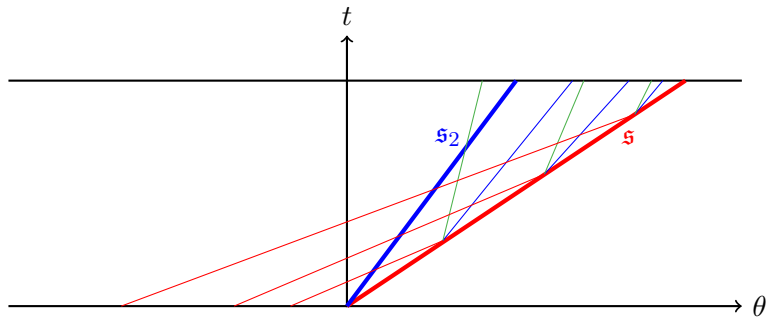

With the structure of the entropy understood, we can study the behavior of $w$, $z$ and $a$ which evolve according to   \eqref{eq:w:z:k:a}.
First note that, since the shock is subsonic relative to the state behind it, the $\lambda_3$ characteristics impinge upon the shock front, and therefore the  initial data for $w$ is determined entirely by the values on the surface $\{t=0\}$, i.e. by $w_0$ as given in \eqref{pre-shocka} (see Figure~\ref{fig:s2!}). As such, $w$ is smooth away from the pre-shock and we are able to precisely quantify how the the bounds degenerate as $(\theta,t) \to (0,0)$.  On the other hand, the characteristics of the subdominant Riemann variable and radial velocity are slower than the shock and thus the solutions in the region  $\sc_2(t) \leq   \theta < \sc(t)$ are determined entirely by their data along the shock curve.  Near the shock curve  $\sc(t)$, approaching from the left, the solution fields $z$ and $a$ smoothly match their values along the shock (see Figure~\ref{fig:s2!}).

Since away from $\sc_2(t)$ the entropy given by \eqref{entsol} is smooth, in spite of both $w$ and $z$ being forced by an entropy gradient, it can be shown that $w$ and $z$ are smooth away from $\sc_2(t)$; this also uses the fact that both $w_0$ (see~\eqref{pre-shocka}) and $\tilde z_0$ (see~\eqref{newdata1}) are smooth away from $(0,0)$.  

The most interesting behavior happens along $\sc_2(t)$, from the right side.  Here, we have determined that the entropy exhibits a cusp-type H\"{o}lder singularity in its derivative; by \eqref{entsol} we have that $\p_\theta k \sim (\theta - \sc_2(t))^{\frac 12}$.  This singularity is seen by the Riemann variables $w$ and $z$ and radial velocity $a$ through their forcing terms $(w-z)^2 \p_\theta k$ which, naively, are just $C^{ \frac{1}{2}}$ across $\sc_2$.  However, the fact that the entropy has a specific cusp structure \eqref{entsol} near the curve $\sc_2$, together with the fact that the wavespeeds of $w$ and $z$ are strictly different from the wavespeed of $k$, actually provides a \emph{regularization effect} for $w$ and $z$. The situation with the radial velocity $a$ is more challenging because it shares the same wavespeed as the entropy; here, the evolution for the specific vorticity is used crucially in our analysis.

In order to explain this regularization effect in greater detail, let us denote the $\lambda_i$ characteristics by
\begin{align*}
 \tfrac{\rmd}{\rmd t} \eta_i(\theta,t) = \lambda_i (\eta_i(\theta,t),t), \qquad \eta_i(\theta,0) = \theta\,,
\end{align*}
for every $i \in \{1,2,3\}$ (in the proof, we in fact denote by $\eta$ the characterstic of $\lambda_3$, but for the $\lambda_1$ and $\lambda_2$ we need to use backwards in time flows, denoted by $\pst$ and $\pt$, see Section~\ref{sec:transport} for details).
Since for $|\theta| \ll 1$ the wave speeds at the pre-shock are given by $ \lambda_1\approx  \tfrac{1}{3} \kappa$, $\lambda_2\approx \tfrac{2}{3} \kappa$ and $ \lambda_3\approx \kappa$, to leading order in time and for small values of $|\theta|$, we have that 
\begin{align*}
\eta_1(\theta,t) \approx \theta + \lambda_1 t \approx \theta + \tfrac{1}{3} \kappa t \,,
\qquad 
\eta_2(\theta,t) \approx \theta + \lambda_2 t \approx \theta + \tfrac{2}{3} \kappa t \,,
\qquad 
\eta_3(\theta,t) \approx \theta + \lambda_3 t  \approx \theta + \kappa t \,.
\end{align*}
We are interested in the behavior near the curve $\sc_2$.  Thus, we seek labels $\theta_i(t)$ such that $\eta_i(\theta_i(t),t) = \sc_2(t) +y$, where $0 < y \ll 1$.  Since $\sc_2(t) \approx   \lambda_2 t$, we have
$
 \theta_i(t) \approx y+ (\lambda_2-\lambda_i)  t  . 
$
The flowmaps are 
\begin{align}
\eta_i(\theta_i(t),s) &\approx y+  (\lambda_2-\lambda_i)  t +\lambda_1 s, \qquad s\in [0,t], \qquad i=1,3.
\end{align}
Ignoring the integrating factors $e^{\frac{8}{3} \int_0^t a(\eta_3(\theta,\tau),\tau) \rmd \tau} \approx 1$ at short times,  the solutions of  \eqref{xland-w} and \eqref{xland-z} take the form
\begin{subequations}
\label{eq:twerking:1}
\begin{align}
w (\sc_2(t) + y,t) &\approx w_0(y+  (\lambda_2-\lambda_3)  t )+  \tfrac{1}{24}\int_0^t ((w-z)^2 \p_\theta k)  ( y+  (\lambda_2-\lambda_3)  t  +\lambda_3 s,s) \rmd s,\\
z (\sc_2(t) + y,t)  &\approx z_0(y+  (\lambda_2-\lambda_1)  t )+ \tfrac{1}{24}\int_0^t  ( (w-z)^2 \p_\theta k) (y+  (\lambda_2-\lambda_1)  t +\lambda_1 s,s) \rmd s.
\end{align}
\end{subequations}
As discussed above, since $\lambda_3 \approx \lambda_2 + \frac 13 \kappa >\lambda_2$, the characteristic curves of $w$ impinge on $\sc_2(t)$ from the left, carrying up initial data $w_0$ from $\{t=0\}$.   On the other hand, the characteristics of $z$ impinge from the right of $\sc_2$ since $\lambda_1 \approx  \lambda_2 - \frac 13 \kappa < \lambda_2$.  Therefore, the data for $z$ is carried from the shock surface $\{\sc(t)=\theta\}$.  Although this data is singular at $(0,0)$, this point is not sampled by the characteristics above since $t>0$ is fixed, and thus $(\lambda_2-\lambda_1)  t  >0$. 
Regarding the forcing terms appearing on the right sides of \eqref{eq:twerking:1}, from the asymptotic description of $k$ in \eqref{entsol}, the approximation $\sc_2(t) \approx \frac 23 \kappa t \approx \lambda_2 t$, and the fact that by \eqref{eq:riemann} $w-z$ equals twice the azimuthal sound speed $c$, which we expect to remain bounded from above and below in terms of $\kappa$, we obtain that
\begin{align}
\label{eq:twerking:2}
 \int_0^t ( (w-z)^2 \p_\theta k) (y+  (\lambda_2-\lambda_i)  t  +\lambda_i s,s) \rmd s 
&\sim \int_{t+ \frac{y}{\lambda_2-\lambda_i}} ^t  (y -\lambda_i (t-s))^{ \frac{1}{2}} \rmd s \sim  y^{ \frac{3}{2}}
 \,,
\end{align}
for $0 < y ,t \ll1$.
Thus, the forcing \emph{gains one derivative}, expressed above by an extra power of $y$, due to the fact that it is integrated along curves which are transversal (since $\lambda_i\neq \lambda_2$) to the characteristics of the entropy (namely, the flow of $\partial_t + \lambda_2 \partial_\theta$).
Thus, from \eqref{eq:twerking:1} and \eqref{eq:twerking:2}, we expect that $w$ and $z$ are both $C^{1, \frac{1}{2}}$ across the curve $\sc_2(t)$, rather than just  $C^{ \frac{1}{2}}$ which is the naive expectation. 

Turning this intuition into a proof requires a $C^2$-type analysis of the characteristics of $\{\lambda_i\}_{i=1}^3$, including an understanding of the times at which the $\lambda_1$ and $\lambda_2$ characteristics intersect the shock curve $\sc$; see for instance Lemmas~\ref{lem:12flows},~\ref{lem:dxeta-dff}, and~\ref{lem:dyy12}. Additionally, in this stage of the proof we need to analyze the time integrals of $\p_y w$ and $\p_{yy} w$ (objects which do blow up rather severely as one approaches the pre-shock) when composed with the flows of $\lambda_1$ and $\lambda_2$; here the transversality of these flows with respect to $\sc$ plays a crucial role, along with a precise understanding of the function $\wb$ in the vicinity of the pre-shock; see Lemmas~\ref{lem:Steve:needs:this},~\ref{lem:characteristic-bounds}, and~\ref{lem:wyy}. This is one of the principal reasons why the pre-shock obtained in {\bf Step 1} needs to be analyzed in a $C^3$ sense.

The intuition behind the gain of regularity for the radial velocity $a$ is less direct. The data for $a$ along the new Cauchy surface (including the shock curve) is $C^{1, \frac{1}{3}}$ due to the formula \eqref{newdata3}.  Thus, such a singularity would be expected to propagate along its characteristic emanating from the pre-shock location. 
To see this, we recall that the specific vorticity at the pre-shock is Lipschitz.  Since by \eqref{xland-svort} it is transported by the  velocity $\lambda_2$, it is forced by $\partial_\theta k$, and  because the wavespeed for $\varpi$ is  the same as that of $k$,  we conclude from \eqref{xland-svort} only that $\varpi$ is $C^{ \frac{1}{2}}$  across the curve $\sc_2$.    Since $k$, $z$ and $w$ are all $C^{1, \frac{1}{2}}$  across this curve, by \eqref{xland-svort:def} we deduce that $\p_x a\in  C^{ \frac{1}{2}}$, and consequently that $a\in C^{1, \frac{1}{2}}$  across $\sc_2$.  Thus, for positive times $t>0$, the radial velocity   becomes smoother than its initial condition ($C^{1,\frac 12}$ vs $C^{1,\frac 13}$). This regularization effect is in essence  a consequence of Lemmas~\ref{lem:varpi:y} and~\ref{lem:ayy}.

Finally, we discuss the region to the right of $\sc_2(t)$.  In this region the entropy is trivial ($k\equiv 0$)  since it is determined solely by its data on the surface $\{\theta <0, \ t=0\}$, see~\eqref{entsol}.  The equations reduce to
\begin{align*}
\p_t w + \lambda_3 \p_\theta w & = - \tfrac{8}{3}  a w  \,,  \\
\p_t z + \lambda_1 \p_\theta z & = - \tfrac{8}{3}  a z    \,,  \\
\partial_t a   + \lambda_2  \partial_{\theta} a & = - \tfrac43 a^2 + \tfrac{1}{3} (w+ z)^2 - \tfrac{1 }{6} (w- z )^2  \,.
\end{align*}
The object $z$ has singular data as in \eqref{newdata1}, which will be propagated along the $\lambda_1$-characteristic curve.  Specifically,
 we have that at the pre-shock $\lambda_1 \approx \tfrac{1}{3} w_0 \approx \frac 13 \kappa$, so that the curve $\sc_1$ along which $z$ is transported from the pre-shock is given by 
\begin{align*}
\sc_1(t) \approx  \tfrac{1}{3}\kappa t.
\end{align*}
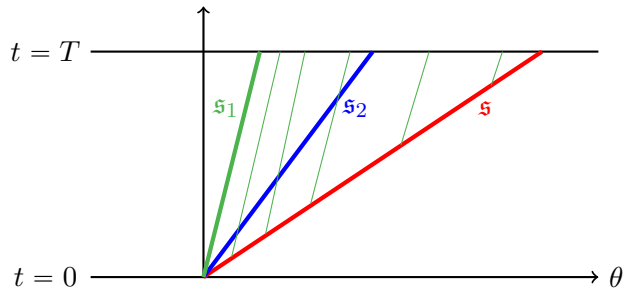
\begin{figure}[htb!]
\centering
\begin{tikzpicture}[scale=1.5]
    \draw [<->,thick] (0,2.4) node (yaxis) [above]{}
        |- (3.5,0) node (xaxis) [right] {$\theta$};
    \draw[black,thick] (0,0)  -- (-1,0) ;
     \draw[black,thick] (-1,2)  -- (3.5,2) ;
    \draw[name path=A,red,ultra thick] (0,0) -- (3,2) ;
        \draw[red] (2.5,1.5) node { $\sc$}; 
    \draw[name path = B,blue,ultra thick] (0,0)  -- (1.5,2) ;
    \draw[blue] (1.35,1.5) node { $\sc_2$}; 
     \draw[green!40!gray, ultra thick] (0,0)  -- (.5,2) ;
      \draw[green!40!gray] (.2,1.5) node { $\sc_1$}; 
     \draw[black] (-1.4,0) node { $t=0$}; 
      \draw[black] (-1.4,2) node { $t=T$}; 
\draw[green!40!gray] (.95,0.65) -- (1.3,2);
\draw[green!40!gray] (1.75,1.17)  -- (2,2) ;
\draw[green!40!gray] (2.55,1.7) -- (2.65,2);
\draw[green!40!gray] (.55,.38)  -- (.9,2) ;
\draw[green!40!gray] (.25,.18) -- (.68,2);
\end{tikzpicture}

\vspace{-0.2cm}
\caption{\footnotesize The $\lambda_1$ characteristics, represented here by the green curves, propagate information about $z$ from the shock curve $\sc$ into the region between $\sc_1$ and $\sc$}
\end{figure}
The $\frac 32$-H\"older singularity in the Cauchy data for $z$ \eqref{newdata2} is morally speaking transported along these $\lambda_1$-characteristics for short times $t\ll 1$, resulting in
\begin{align}\label{zsol}
z(\theta,t) &\approx \begin{cases} 0, &\phantom{\sc_1(t) <}\  \theta< \sc_1(t) 
\\    {\rm z}_0   \left(\theta- \sc_1(t)\right)^{\frac{3 }{2}}, &   \sc_1(t) < \theta\ll \sc_2(t)
\end{cases}.
\end{align}
The difficulty in showing that the intuitive behavior \eqref{zsol} is indeed true lies in the fact that the $\lambda_1$-characteristics emanating from the shock curve do spend some time in the region between $\sc_2$ and $\sc$, and in this region the entropy gradient present in \eqref{xland-z} causes the first and second derivatives of $z$ to behave badly. By using the transversality of the $\lambda_1$ and $\lambda_2$ characteristics, we are nonetheless able to show in Section~\ref{sec:zyy} that \eqref{zsol} is morally correct. 

Note that in this region, the relevant initial data for $w$ and $a$ is far away from the pre-shock, and so the fields $w$ and $a$ are as regular as their forcing for short times.
This forcing involves the field $z$, which makes a $C^{1,\frac 12}$ cusp along $\sc_1(t)$.  However, again the wave speeds for $w$ and $a$ are different than that of $z$, and as such  their characteristics are transversal to $\sc_1(t)$.  This means that the solution fields gain a derivative relative to the forcing, similar to \eqref{eq:twerking:2}.  It thus seems reasonable to conjecture that  $w,a\in C^{2,\frac 12}$ on the right side of $\sc_1$. Establishing this fact would in turn require us to show that \eqref{pre-shocka} holds in a $C^{3,\frac 12}$ sense, a regularity level which we did not pursue in {\bf Step 1}. As such, in this paper we only prove that $w,a\in C^{2}$ on $\sc_1$, which is nonetheless a better regularity exponent that the naively expected $C^{1,\frac 12}$.

\vspace{.05in}
\noindent{\bf Step 5: returning to basic fluid variables.}
There is a certain regularization effect along the curve $\sc_2$, when returning to the original fluid variables, as we now explain. A straightforward calculation shows that the good unknowns
\begin{align}
q^w:=\partial_\theta w-  \tfrac{1}{4}  c \partial_\theta k, \qquad q^z:=\partial_\theta z+  \tfrac{1}{4}  c \partial_\theta k,
\label{eq:lazy:cat}
\end{align}
satisfy the evolution equations
\begin{subequations}
\label{eq:twerking:3}
\begin{align}
(\p_t + \lambda_3\partial_\theta) q^w + (\partial_\theta \lambda_3 + \tfrac{8}{3}   a) q^w &=    - \tfrac{8}{3}  \partial_\theta a w  + \Big( \tfrac{4}{3}a  c  +   \tfrac{1}{6}  c  \partial_\theta\lambda_2 \Big)\partial_\theta k,\\
(\p_t + \lambda_1\partial_\theta)  q^z +  (\partial_\theta \lambda_1 + \tfrac{8}{3}   a) q^z &=  - \tfrac{8}{3}  \partial_\theta a z  - \Big( \tfrac{4}{3}a  c  +   \tfrac{1}{6}  c  \partial_\theta\lambda_2 \Big)\partial_\theta k.
\end{align}
\end{subequations}
The remarkable feature of the system \eqref{eq:twerking:3} is that the second derivatives of $k$ do not appear in the equations; indeed, if one naively considers the evolution equation for $\p_\theta w$ or $\p_\theta z$ alone, then from \eqref{xland-w} and respectively \eqref{xland-z} we note the emergence of the forcing term $\frac{1}{24} (w-z)^2 \p_{\theta\theta} k$. The unknowns $q^w$ and $q^z$, and the system \eqref{eq:twerking:3}, is useful because it involves only $\partial_\theta k$, and this forcing makes a $C^{\frac{1}{2}}$ cusp along the curve $\sc_2$. However, since the characteristic speed of $k$ is  $\lambda_2$, and the characteristics of  $q^w$ and $q^z$  are $\lambda_3$ and respectively $\lambda_1$, and are thus transversal, we again have a regularization effect akin to \eqref{eq:twerking:2}, and we find that the (Lagrangian) force is actually $C^{ 1, \frac{1}{2}}$ across $\sc_2$.  Now, the initial data relevant to the behavior of $q^w$ and $q^z$ comes from different places.  For $q^w$, it originates along the $\{t=0\}$ surface and so it is easy to see that it is smooth for positive time (away from the pre-shock).  On the other hand, the data for $q^z$ originates on the shock curve itself and once again, away from the pre-shock it is smooth.  It follows that, for $t>0$ the regularity is set by the forcing, resulting in bounds consistent with $q^w,q^z\in C^{1, \frac{1}{2}}$. Again, in the proof we only establish the $C^1$ regularity of $q^w$ and $q^z$, due to the $C^3$ expansion of the pre-shock; this argument is made rigorous in Sections~\ref{sec:improve:wyy} and~\ref{sec:zyy}. The outcome is that 
$q^w+q^z = \partial_\theta z+\partial_\theta w = \frac{2}{3} \partial_\theta u_\theta$ is smoother than the naive expectation $C^{\frac 12}$: we prove that it lies in $C^1$ across $\sc_2$ (which translates into $C^2$ regularity for the angular velocity $u_\theta$), and conjecture that the sharp regularity is  $C^{ 1, \frac{1}{2}}$.  Similarly, the improved regularity for $q^w$ and $q^z$ shows that the second derivative of the pressure is bounded on $\sc_2$, see~\eqref{eq:done}.

\vspace{.05in}
\noindent{\bf Summary.}
In terms of the Riemann variables in azimuthal symmetry, we find
\begin{itemize}
\item Across the shock curve $\sc(t)$,  we have
\begin{align*}
\jump{w} \sim t^{\frac{1}{2}}, \qquad \jump{\partial_\theta a} \sim t^{\frac{1}{2}},\qquad 
\jump{z} \sim t^{\frac{3}{2}}, \qquad \jump{k} \sim t^{\frac{3}{2}}.
\end{align*}
\item Across the curve  $\sc_2(t)$, the functions $\partial_\theta w,\partial_\theta a,\partial_\theta k,\partial_\theta z$ all behave as $C^{\frac{1}{2}}$ cusps approaching $\sc_2$ from the right.  Approaching from the left, they are all smooth, in positive time.
\item Across the curve  $\sc_1(t)$, the entropy is zero, $\partial_\theta w$ and $\partial_\theta a$ are $C^1$ (expected to be $C^{1,\frac{1}{2}}$) and $\partial_\theta z$ behaves as a $C^{ \frac{1}{2}}$ cusp  approaching $\sc_1$ from the right.   Approaching from the left, $\p_\theta z$ is $C^1$ in positive time.
\end{itemize}
In terms of the  physical variables, we find
\begin{itemize}
\item Across the shock curve $\sc(t)$, all state variables jump
\be
\jump{u_\theta} \sim r t^{\frac{1}{2}}, \qquad \jump{\rho} \sim r^2 t^{\frac{1}{2}},\qquad \jump{\partial_\theta u_r} \sim r t^{\frac{1}{2}},\qquad \jump{S} \sim t^{\frac{3}{2}}.
\ee
\item Across the curve  $\sc_2(t)$, the entropy, density and radial velocity derivatives all make $C^{\frac{1}{2}}$ cusps  approaching $\sc_2$ from the right.  Approaching from the left, they are all smooth.  The second derivative of the angular velocity and the pressure are {\em bounded}  for $t>0$, and are expected to be $C^{\frac{1}{2}}$ smooth.

\item Across the curve  $\sc_1(t)$, the entropy is zero while the angular velocity and density derivatives make $C^{\frac{1}{2}}$ cusps  approaching $\sc_1$ from the right.  Approaching from the left, they are all smooth for $t>0$. The second derivative of the radial velocity is bounded and is expected to make a $C^{\frac{1}{2}}$ cusp.
\end{itemize}


\section{Detailed shock formation}
\label{sec:formation}

In \cite{BuShVi2019a}, it was established that for an open set of $C^4$ initial data,  solutions to   \eqref{eq:w:z:k:a} form a generic, stable, asymptotically 
self-similar pre-shock at time $t=T_*$, and that the dominant Riemann variable $w(\cdot , T_*) \in C^{\frac{1}{3}} $.  
The primary objective of this section is to provide a precise description of $w( \cdot ,T_*)$ in the vicinity of the pre-shock.   We shall prove the following 
\begin{theorem}[Detailed shock formation]\label{thm:blowup-profile} For $\kappa_0 >1 $ taken sufficiently large and $\eps >0$ sufficiently small,
and for initial data $(w,z,k,a)|_{t=-\eps}=(w_0,0,0,a_0)$ satisfying \eqref{w0-to-W0}--\eqref{eq:A_bootstrap:IC} below,   there exists a  blowup time $t=T_*$, a unique blowup location $\xi_*$,
 and unique solutions $(w,a)$ to \eqref{eq:w:z:k:a}  in 
$C^0([-\eps,T_*), C^4( \mathbb{T}  ))\cap C^4([-\eps,T_*), C^0( \mathbb{T}  )) $ such that
\begin{align} 
  w( \cdot ,T_*) \in C^ {\frac{1}{3}}(\TT  )  \,, \qquad a(\cdot , T_*) \in C^{1, {\frac{1}{3}} }( \TT ) \,, \qquad \varpi(\cdot , T_*) \in C^{0, 1 }( \TT  ) \,.\label{wa-at-blowup}
\end{align} 
Furthermore,  there exists a unique blowup label $x_*$ satisfying
$$ \abs{x_*} \le 20\kappa_0 \eps^4 \ \ \ \text{ such that } \ \ \ \lim_{t \to T_*} \eta(x_*, t) = \xi_* \,,$$
where $\eta$ is  the $3$-characteristic 
defined by \eqref{eta-eqn}.  
The pre-shock 
 $w( \cdot ,T_*)$ has the fractional series expansion  
\begin{align} 
\sabs{ w(\theta,T_*) -  \kappa_* - \aa_1 (\theta-\xi_*)^\frac{1}{3} - \aa_2 (\theta-\xi_*)^\frac{2}{3} - \aa_3 (\theta-\xi_*)}
\les \sabs{\theta-\xi_*}^ {\frac{4}{3}} 
\label{w-blowup}
\end{align} 
for all $\theta \in \eta( B_{x_*} (\eps^{3}))$, 
where
$$
\kappa_* = e^{- \tfrac{8}{3}\int_{-\eps}^{T_*}  \p_x ( a(\eta(x_*,r),r))dr}w_0(x_*)\,,
$$
and 
\begin{align} 
\abs{\kappa_*  - \kappa_0} \le 2\eps \kappa_0\,, \ \ \ 
- \tfrac{6}{5} \le \aa_1 \le  -  \tfrac{4}{5}  \,, \ \ \ \sabs{ \aa_2}   \le  \eps^ {\frac{1}{10}}  \,, \ \ \ \sabs{ \aa_3}   \le   \tfrac{7}{6\eps}  \,.   \label{Taylor-coefficients}
\end{align} 
In fact, the expansion \eqref{w-blowup} is valid in a $C^3$-sense, by which we mean that the bounds 
\begin{subequations} 
\label{preshock-derivatives}
\begin{align} 
\abs{\p_\theta w(\theta, T_*)  -    \tfrac{1}{3}  \aa_1 (\theta-\xi_*)^{-\frac{2}{3}}  -    \tfrac{2}{3} \aa_2 (\theta-\xi_*)^{-\frac{1}{3}} } & \les \tfrac{1}{\eps} \,,    \label{need-like-a-hole-in-the-head1} \\
\abs{\p_\theta^2w(\theta, T_*)  -    \tfrac{2}{9}  \aa_1 (\theta-\xi_*)^{-\frac{5}{3}}   } &  \les   \eps^ {-\frac{63}{8}}  \sabs{\theta-\xi_*}^{-\frac{4}{3}}\,,   \label{need-like-a-hole-in-the-head2} \\
\abs{\p_\theta^3w(\theta, T_*)  } &  \les   \eps^ {-\frac{151}{8}}   \sabs{\theta-\xi_*}^{-\frac{8}{3}}\,,   \label{need-like-a-hole-in-the-head3}
\end{align} 
\end{subequations} 
hold  for all $\theta \in \eta( B_{x_*} (\eps^{3}))$.  Moreover, the $C^4$ regularity away from the pre-shock is characterized by
\begin{align} 
&\sup_{t \in [-\eps,T_*)}\max_{\gamma\le 4}\left( \abs{ \p_\theta^\gamma a(\eta(x,t),t)}
+\abs{ \p_\theta^\gamma w(\eta(x,t),t) } \right) \notag \\
& \qquad\qquad
 \le 
  \begin{cases}
C_\eps  \left((T_* -t) +   3\eps^{-3}(\eps+t)(x-x_*)^2\right)^{-4}   & \ \ \abs{x-x_*} \le \eps^{2}  \\
C_\eps& \ \ \abs{x-x_*} \ge \eps^{2}
 \end{cases}
 \,,   \label{thm-aw-good-bound}
\end{align}  
where $C_\eps > 0 $ is a sufficiently large constant depending on inverse powers of $\eps$.
Lastly, the specific vorticity satisfies the bounds
\begin{align} 
\tfrac{10}{\kappa_0} \le \varpi(x,t) \le \tfrac{28}{\kappa_0 } \,, \qquad  \sabs{\p_x\varpi(x,t )} \le \tfrac{70}{\kappa^2_0 \eps} \,,
\label{svort-thm-bounds}
\end{align} 
for all $x\in \TT$ and $t\in[-\eps,T_*)$.
\end{theorem}

The proof of this theorem makes use of detailed estimates for the characteristic families and their derivatives.
As we will detail below, we let  $\eta(x,t)$ denote the flow $w$.  Here $x$ denotes a particle label, and $\eta(x,t)$ provides the location of
$x$ at time $t$; specifically we have the formula   $ \eta(x,t) = x + \int_{-\eps}^t w(\eta(x,s),s) ds$.  Moreover,  we see that
$w(\eta(x,t),t) = e^{- \frac{8}{3} \int_{-\eps}^t a(\eta(x,s),s)ds}w_0(x) $
and hence that 
$$
w(\theta,t) = e^{- \frac{8}{3} \int_{-\eps}^t a(\eta(\eta ^{-1} (\theta,t),s),s)ds}w_0(\eta ^{-1} (\theta,t))  \,.
$$
It follows that  a power series expansion of $w(\theta,T_*)$ about  the blowup location $\theta= \xi_*$ requires a series expansion for the inverse flow
map $\eta ^{-1}(\theta,t)$ about $\theta= \xi_*$.   
The formula for $\eta ^{-1} (\theta,T_*)$ requires us to first compute $\eta(x, T*)$, and then invert the polynomial equation $\eta(x,T_*) =\theta$ for $\theta$ in a neighborhood
of $\xi_*$.

We shall write  $\eta(x,T_*)$ as a Taylor series about the blowup label $x_*$.  To do so, we prove the existence of a unique blowup trajectory
$\eta(x_*,t)$ which converges to $\xi_*$, and study the behavior of $\p_x^\gamma\eta(x,t)$,   $\gamma \le 4$.   Our analysis makes use of 
self-similar coordinates only for the purpose of isolating the unique blowup  trajectory $\eta(x_*,t)$, whereas all of our estimates  for 
$\p_x^\gamma\eta(x,t)$,  $\p_\theta^\gamma w(\eta(x,t),t)$, and $\p_\theta^\gamma a(\eta(x,t),t)$ are obtained in physical coordinates.
With these bounds in hand, we establish the Taylor expansion for $\eta(x,t)$ about the blowup label $x_*$, proceed
to invert this relation, and then obtain a detailed description of the pre-shock.

\subsection{Changing variables to modulated self-similar variables}
We shall make use of self-similar coordinates $(y,s)$  that rely upon time dependent  modulation functions $\kappa(t)$, $\xi(t)$ and $\tau(t)$, which are
introduced to enforce three pointwise constraints.   Specifically, 
we map the physical coordinates $(\theta,t)$ to  self-similar coordinates $(y,s)$ by the following transformations:
\begin{align*} 
s(t):=-\log(\tau(t)-t)\,,
\qquad 
y(\theta,t):=\tfrac{\theta-\xi(t)}{ (\tau(t)-t)^{\frac32}}  = e^{\frac{3}{2}s} (\theta-\xi(t)) \,.
\end{align*} 
It follows that
\begin{align}
\tau-t =e^{-s}\,, \ \tfrac{ds}{dt}= (1-\dot \tau)e^s\,,  \label{dsdt}
\end{align} 
and thus
\begin{align} 
\partial_{  \theta}   y = e^{\frac32s}\,, \  \partial_t y = \tfrac{- \dot \xi}{(\tau -t)^ {\frac{3}{2}} }   -\tfrac{3(\dot\tau -1)(\theta-\xi)}{2(\tau-t)^ {\frac{5}{2}} }
=- e^{\frac{3}{2}s} \dot\xi + \tfrac{3}{2}(1-\dot\tau) y e^s\,. 
\label{xt-to-ys}
\end{align}
We then transform the physical variables $(a,w)$ to self-similar  variables $(A,W)$ by
\begin{align} 
 w(\theta,t)=e^{-\frac s2}W(y,s)+\kappa (t)\,, \quad
a(\theta,t)=A(y,s) \,.   \label{w-to-W}
\end{align} 
Introducing the parameter
\begin{align} 
\beta_\tau = \beta_\tau(t) = \tfrac{1}{1-\dot \tau(t)}\,, \label{beta-tau}
\end{align} 
a simple computation shows that $(W,A)$ solve
\begin{subequations} 
\label{shallow-water-SS}
\begin{align}
\p_s W - \tfrac{1}{2} W + ( \tfrac{3}{2}y+\beta_\tau W + e^{\frac{s}{2}} \beta_\tau (\kappa -\dot \xi)) \p_y W &=  -e^{-\frac s2} \beta_\tau \dot\kappa
- \tfrac{8}{3}e^{-\frac s2} \beta_\tau A (e^{-\frac s2}W + \kappa )  \,, \label{eq:W-SS} \\
\p_s A + ( \tfrac{3}{2}y+ \tfrac{2}{3} \beta_\tau W + e^{\frac{s}{2}} \beta_\tau ( \tfrac{2}{3} \kappa -\dot \xi) ) \p_y A &=  - \tfrac{4}{3}\beta_\tau  e^{-s}A^2   + \tfrac{1}{6}\beta_\tau e^{-s} (e^{-\frac s2}W + \kappa ) ^2\,,  \label{eq:A-SS}
\end{align}
\end{subequations} 
with initial conditions given at self-similar time $s=-\log\eps$ by
\begin{align} 
W(y,-\eps)= \eps^ {-\frac{1}{2}} (w_0(\theta) - \kappa_0) \,, \qquad  A(y,-\eps)=a_0(\theta) \,,
\label{eq:data:dump}
\end{align} 
and
\begin{align} 
\kappa(-\eps) = \kappa_0\,, \qquad  \tau(-\eps)=0\,, \qquad \xi(-\eps)=0 \,. \label{IC-mods}
\end{align} 
For notational brevity, we introduce the transport velocities and forcing functions 
\begin{subequations}
\begin{alignat}{2}
\mathcal{V} _W &:=  \tfrac{3}{2}y+\beta_\tau W + e^{\frac{s}{2}} \beta_\tau (\kappa -\dot \xi) \,, 
\qquad  && F_W :=- \tfrac{8}{3}e^{-\frac s2} \beta_\tau A (e^{-\frac s2}W + \kappa )
\label{eq:VW:def}\\ 
\mathcal{V} _A &:= \tfrac{3}{2}y+ \tfrac{2}{3} \beta_\tau W + e^{\frac{s}{2}} \beta_\tau (\tfrac{2}{3}\kappa -\dot \xi) \,, 
\qquad && F_A := - \tfrac{4}{3}\beta_\tau  e^{-s}A^2   + \tfrac{1}{6}\beta_\tau e^{-s} (e^{-\frac s2}W + \kappa ) ^2 \,,
\label{eq:VA:def}
\end{alignat}
\end{subequations}
so that \eqref{shallow-water-SS} takes the form
\begin{subequations}
\label{eq:ssWZA}
\begin{align}
\partial_s W-\tfrac12 W+ \mathcal{V} _W \p_yW  
&=- \beta_\tau  e^{-\frac s2} \dot \kappa+  F_W \,, \label{e:W_eq}\\
 \partial_s A+\mathcal{V} _A \p_yA &= F_A \,. \label{e:A_eq}
\end{align}
\end{subequations} 
We shall also consider the perturbation of the stable self-similar stationary solution $\bar W(y)$ of the Burgers equation\footnote{Recall that $\bar W(y)$ is the solution of   $ -\frac 12 \bar W + \left( \frac{3y}{2} + \bar W \right) \partial_y \bar W = 0$ and has an explicit formula which is obtained by inverting the cubic 
polynomial $\bar W^3 + \bar W = -y$.}; the function 
 $\tilde W = W-\bar W$ solves
\begin{align} 
&\p_s \tilde W + \left(- \tfrac{1}{2} +  \beta_\tau  \p_y \bar W +  \tfrac{8}{3} e^{-s} \beta_\tau A\right) \tilde W 
+\mathcal{V} _W \p_{y} \tilde W  
=(1- \beta_\tau ) \bar W \p_y \bar W - \tfrac{8}{3} e^{-s} \beta_\tau A \bar W  -e^{-\frac s2} \beta_\tau \dot\kappa
 \,.
\label{tildeW-gamma}
\end{align}

\subsection{Bounds on the solution}
In order to obtain the necessary quantitative bounds on characteristics and their derivatives, we shall make use of the bounds on $W$ provided by
Theorem 4.4 of \cite{BuShVi2019a} for the shock formation process.   As such, we give a precise description of the  initial data used for
the asymptotically self-similar shock formation.

\subsubsection{Initial data in self-similar variables}
\label{subsec:support}
It is convenient to describe the initial data in terms of the self-similar variables
$(W(\cdot ,-\log\eps),A(\cdot ,-\log\eps))$ defined in \eqref{eq:data:dump}, which may be equivalently written as
\begin{align} 
w_0(\theta) = \eps^ {\frac{1}{2}} W(y, -\log\eps) + \kappa_0 \,, 
\qquad 
a_0(\theta) = A(y,-\log \eps) \,.  \label{w0-to-W0}
\end{align} 
We choose $w_0\in C^4( \mathbb{T}  )$ so that for all $\theta \in \TT$:
\begin{align} 
\tfrac{7}{8} \kappa_0 \le w_0(\theta) \le  \tfrac{9}{8} \kappa_0 \,,
\qquad
\mbox{where}
\qquad  
\kappa_0 \ge 3 \,.
\label{w0-lowerupper}
\end{align} 
We assume that the initial data $(W(\cdot ,-\log\eps),A(\cdot ,-\log\eps))$ has compact  support in the set
\begin{align}
\XXX_0:=\left\{\abs{y}\leq  2\eps^{-1}   \right\} \,.
\notag
\end{align}
In order to obtain  stable shock formation, we 
require that\footnote{As shown in Corollary 4.7 in \cite{BuShVi2019a}, the conditions \eqref{eq:fat:cat} on the initial data are satisfied by any data in an open set (within azimuthal symmetry) in the $C^4$ topology, as long as a global non-degenerate minimal slope is attained at a point.}
\begin{align}
W(0,-\log \eps) = 0\, , 
\quad 
\p_y W(0,-\log \eps) = -1\,,
\quad
\p_y^2 W(0,-\log \eps) = 0\,.
\label{eq:fat:cat}
\end{align}
As in~\cite{BuShVi2019b}, there exists a sufficiently large parameter $M = M(\kappa_0)\geq 1$ (which is in particular independent of $\eps$), a small length scale $\ell$, and a large length scale $\LLL$  by 
\begin{align}
\ell = (\log M)^{-5}\,, \qquad
\LLL =\eps^{-\frac{1}{10}}  \, .
\label{l-and-L}
\end{align}
The initial datum of $\tilde W = W - \bar W$ is given by
$$
\tilde W (y,-\log\eps)= W(y,-\log \eps) - \bar W(y) = \eps^{-\frac 12}\left(w_0(\theta) - \bar w_\eps(\theta) \right) 
=: \eps^{-\frac 12} \tilde w_0(\theta) \,,
$$
where we have defined $\bar w_\eps(\theta) = \eps^ {\frac{1}{2}} \bar W( \eps^ {-\frac{3}{2}} \theta) + \kappa_0$.
We consider data such  that for $\abs{y} \leq \LLL$,
\begin{subequations}
\label{ss-ic-1}
\begin{alignat}{2}
(1+y^2)^{-\frac 16} \sabs{\tilde W(y,-\log \eps)} &\leq \eps^{\frac{1}{10}} \,,
\label{eq:tilde:W:zero:derivative}
\\
(1+y^2)^{\frac 13} \sabs{\p_y \tilde W(y,-\log \eps)} &\leq \eps^{\frac{1}{11}}
\label{eq:tilde:W:p1=1}
\,,
\end{alignat}
\end{subequations}
for $\abs{y} \leq \ell$  (equivalently $\abs{\theta} \le \eps^ {\frac{3}{2}} \ell$), we assume that
\begin{align}
 \sabs{\p_y^4\tilde W(y,-\log \eps)} \leq \eps^{\frac 58} 
 \qquad \Leftrightarrow \qquad 
 \abs{\p_\theta^4\tilde w_0(\theta)} &\leq \eps^ {-\frac{39}{8}}  
\label{eq:tilde:W:5:derivative}
\, ,
\end{align}
and at $y=0$, we have that
\begin{align}
\sabs{\p_y^3 \tilde W(0,-\log \eps)} \leq   \eps^{\frac{3}{8}}   
\qquad \Leftrightarrow \qquad 
\abs{\p_\theta^3\tilde w_0(0)} &\leq \eps^ {-\frac{29}{8}}  
\,.
\label{eq:tilde:W:3:derivative:0} 
\end{align}
For $y$ in the region  $\{ \abs{y} \geq \LLL \} \cap \XXX(-\log\eps)$, we suppose that
\begin{subequations}
\begin{align}
  (1+y^2)^{-\frac{1}{6}}\abs{W(y,-\log \eps)} &\leq  1+ \eps^ {\frac{1}{11}} \,,
 \label{eq:rio:de:caca:1}
\\
  (1+y^2)^{\frac 13} \abs{\p_y W(y,-\log \eps)} &\leq  1 +  \eps^{\frac{1}{12}} \,,
   \label{eq:rio:de:caca:2}
\end{align}
\end{subequations}
while for  $W_y$, globally for all $y \in \XXX(-\log\eps)$ we shall assume that 
\begin{subequations}
\begin{align} 
\sabs{\p_yW(y,-\log \eps)} &\le  (1+y^2)^{-\frac 13}  \,,  
\label{eq:W:gamma=2:p1} \\
\sabs{\partial_y^2 W(y,-\log\eps)} & \le  7 (1+y^2)^{-\frac 13} \,, \label{d2W} \\
\sabs{\partial_y^\gamma W(y,-\log\eps)} & \les (1+y^2)^{-\frac 13} \ \ \text{ for } \ \ \gamma=3,4 \,. \label{d3W}
\end{align} 
\end{subequations}

For the initial conditions of $A(y,-\log \eps) =  a_0 (\theta)$, we  require that $ a_0 \in C^4(\mathbb{T}  )$, and that
\begin{align}
\norm{a_0}_{C^0} \le  \epsilon \,, \qquad  \norm{\p_x a_0}_{C^0} \le \tfrac{\kappa_0}{14}    \,, \qquad \norm{a_0}_{C^4} \les 1\,.    \label{eq:A_bootstrap:IC} 
\end{align}

\subsubsection{Bounds on $W$ and $A$}
The following facts are established in \cite{BuShVi2019b}.
The  spatial support of  $(W,A)$ is the $s$-dependent ball 
\begin{align}
\XXX(s):=\left\{\abs{y}\leq  2\eps^{\frac12}  e^{\frac 32s} \right\} \text{ for all } s \ge -\log\eps \,.
\label{eq:support}
\end{align}
It follows that
\begin{equation}
1+y^2 \leq 40 \eps e^{3s} \qquad \Leftrightarrow \qquad (1+y^2)^{\frac{1}{3}} \leq 4 \eps^{\frac 13}  e^{s} \,.
\label{support-bound}
 \end{equation}
We have the following bounds for $W(y,s)$ for all $y \in \mathbb{R}  $ and $s \ge -\log \eps$:
\begin{align}
 \sabs{\partial^{\gamma} W(y,s)}&
\leq   \begin{cases}
(1+ 2\eps^{\frac{1}{20}}) (1+y^2)^{\frac 16}, & \mbox{if } {\gamma}  = 0\,,  \\
2 (1+y^2)^{-\frac 13}, & \mbox{if }  \gamma = 1 \,, \\
M^{\frac{1}{3}}  (1+y^2)^{-\frac 13}   , & \mbox{if }  \gamma = 2 \,.
\end{cases}\label{Wbootstrap}
\end{align}
For the perturbation function $\tilde W(y,s) = W(y,s) - \bar W(y)$ and for $\abs{y}\leq \LLL= \eps^{-\frac{1}{10}}    $,
\begin{subequations}
\begin{align}
\sabs{ \tilde W(y,s)} &\leq 2\eps^{\frac{1}{11}} (1+y^2)^{\frac 16} \,, \label{loco1} \\
\sabs{\p_y \tilde W(y,s)} &\leq 2\eps^{\frac{1}{12}} (1+y^2)^{-\frac 13} \,,  \label{loco2} 
\end{align}
\end{subequations}
while for $\abs{y} \leq \ell=(\log M) ^{-5}$, 
\begin{subequations}
\begin{align}
\sabs{\p^\gamma \tilde W(y,s)} & 
\le (\log M)^4 \eps^ {\frac{1}{10}} \abs{y}^{4-\gamma} + M \eps^ {\frac{1}{4}} \abs{y}^{3-\gamma} 
\,, &&  \gamma \leq 3\,, 
 \label{tildeWbootstrap0}\\
\sabs{\p^4 \tilde W(y,s)} &\leq \eps^{\frac{1}{10}} \,, && 
\label{d4tildeWbootstrap}
\end{align}
\end{subequations}
and at $y=0$,  
\begin{align}
\sabs{\p^3\tilde W(0,s)} &\leq \eps^{\frac{1}{4}}  \,,\label{bootstrap:Wtilde3:at:0}
\end{align} 
for all $s\geq -\log \eps$. 
With $w_0$ satisfying \eqref{w0-lowerupper}, 
as shown in \cite{BuShVi2019a} via the maximum principle, we have that
\begin{align} 
\tfrac{4\kappa_0}{5} \le w(\theta,t) \le \tfrac{5\kappa_0}{4} \,, \ \ \ t\in [-\eps,T_*)   \,.  \label{w-lowerupper}
\end{align}

\subsubsection{Bootstrap assumptions on $\p_\theta^\gamma a$, $\gamma \le 2$}
Bounds for $a$ and $\p_\theta a$ were previously established in \cite{BuShVi2019b}. In this paper, we revisit these estimates
and establish the following sharp bootstrap bounds 
\begin{subequations} 
\label{a-boot}
\begin{align}
\abs{a(\theta,t)} &\le  2 \kappa_0^2  \eps\,, \label{dx0a-boot} \\  
\abs{\partial_\theta a(\theta,t)}  &\le  2\kappa_0 \,,   \label{dx1a-boot}  \\
\sabs{\p_\theta^2 a(\theta,t)}  & \le 12 e^{s} \,, \label{dx2a-boot} 
\end{align}
\end{subequations} 
for all $ \theta\in \mathbb{T}  $ and $t\in[-\eps, T_*)$.
The bootstrap bounds~\eqref{a-boot} are closed in Section~\ref{sec:abounds} below.

\subsection{Evolution equations and bounds for the modulation variables}
The modulation variables  $\tau(t)$, $\xi(t)$, and $\kappa (t)$ are used to impose the following constraints at $y=0$
\begin{align}
W(0,s)=0,\qquad \p_yW(0,s)=-1,\qquad \p_y^2W(0,s)=0 \,.
\label{eq:constraints}
\end{align}
Imposing $\p_yW(0,s)=-1$ in the first derivative of \eqref{eq:W-SS}  shows that
\begin{subequations}
\label{mod-evo}
\begin{align}
\label{eq:tau:dot}
\dot \tau (t) =    e^{-\frac s2}   \tfrac{8}{3} \left(\kappa(t) A_y(0,s) + e^{-\frac s2}A(0,s) \right)    \,.
\end{align}
Next, requiring that $\p_y^2W(0,s)=0$ holds, by taking the second derivative of \eqref{eq:W-SS}  we obtain
\begin{align}
\dot \xi (t) - \kappa(t)
&=- \tfrac{8}{3}\tfrac{ e^{- s}}{W_{yyy}(0,s) } \left(2e^{-\frac{s}{2}}  A_y(0,s) -  \kappa A_{yy}(0,s) \right) \,,
\label{eq:dot:xi}
\end{align} 
and finally with $W(0,s)=0$ used in \eqref{eq:W-SS}, we find that
\begin{align}
\dot \kappa (t)
&=  -\tfrac{8}{3} \left(  \kappa(t) A(0,s)+  \kappa \tfrac{A_{yy}(0,s)  }{W_{yyy}(0,s)}  - 2e^{-\frac{s}{2}}\tfrac{A_y(0,s)   }{W_{yyy}(0,s)} \right)  
=  -\tfrac{8}{3} \kappa(t) A(0,s) - e^s( \dot \xi -\kappa)(t)
\,.
\label{eq:dot:kappa}
\end{align}
\end{subequations}

The equations \eqref{mod-evo} are ODEs for the modulation functions.  From
\eqref{a-boot}, it follows that for $\eps$ taken sufficiently small, for all $t\in [-\eps,T_*)$ we have
\begin{align} 
\sabs{\dot \tau(t)} &  \le 9 \kappa_0^2 \eps e^{-s}  \,, \qquad
\sabs{\dot \kappa(t)}   \le  6 \kappa_0^3 \eps \,, \qquad
\sabs{\dot \xi(t)}   \le \kappa_0 + 8 \kappa_0^2\eps^2 \,.
\label{tau-kappa-xi-bounds}
\end{align} 
For the last bound, we have used that 
since  $\int_{-\eps}^{T_*} (1 -\dot \tau(t')) dt'= \eps$, then
\begin{align}
\abs{T_*}  \le 7 \kappa_0^2\eps^3 \,. \label{T*-bound}
\end{align} 
It follows that
\begin{align} 
\abs{\kappa - \kappa_0} \le 7 \kappa_0^3 \eps^2  \,, \qquad 
\abs{\tau} \le 7 \kappa_0^2 \eps^3 \,,\qquad 
\abs{\xi} \le  2 \eps \kappa_0 \,, \qquad 
\abs{1 - \beta_\tau } \le 7\kappa_0^2 \eps e^{-s} \,.  \label{mod2-bound}
\end{align}

\subsection{Characteristics in physical variables $(x,t)$}

\subsubsection{$3$-characteristics $\eta$ associated $\lambda_3$.}
We let $\eta(x,t)$ denote the characteristics of  $\lambda_3=w$ so that
\begin{subequations} 
\label{eta-eqn}
\begin{alignat}{2}
\p_t \eta(x,t) &= w(\eta(x,t),t) \  && \text{ for }  \ -\eps< t <T_* \,, \label{pt-eta} \\
\eta(x,-\eps)& = x \,,
\end{alignat} 
\end{subequations} 
for all labels $x$.

\subsubsection{$2$-characteristics $\phi$ associated to $\lambda_2$.}
We let $\phi(x,t)$ denote the characteristics of $\lambda_2= {\tfrac{2}{3}} w$ so that 
\begin{subequations} 
\label{phi-eqn}
\begin{alignat}{2}
\p_t \phi(x,t) &= {\tfrac{2}{3}} w(\phi(x,t),t) \  && \text{ for }  \ -\eps< t <T_* \,, \label{pt-phi} \\
\phi(x,-\eps)& = x \,,
\end{alignat} 
\end{subequations} 
for all labels $x$.

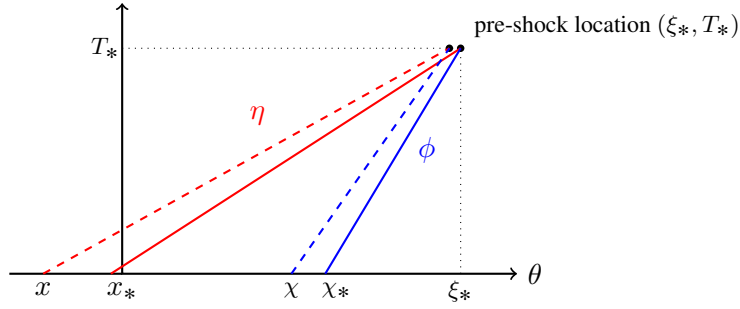
\begin{figure}[htb!]
\centering
\begin{tikzpicture}[scale=1.5]
    \draw [<->,thick] (0,2.4) node (yaxis) [above] {}
        |- (3.5,0) node (xaxis) [right] {$\theta$};
    \draw[black,thick] (0,0)  -- (-1,0) ;

\node at (3,2)[circle,fill,inner sep=1pt]{};
\node at (2.9,2)[circle,fill,inner sep=1pt]{};

\draw[black,dotted] (0,2)  -- (3,2) ;
\draw[black] (-0.15,2) node {{\footnotesize $T_*$}}; 

\draw[black,dotted] (3,0)  -- (3,2) ;
\draw[black] (3,-.15) node {{\footnotesize $\xi_*$}}; 

\draw[black] (4.3,2.2) node { {\footnotesize pre-shock location $(\xi_*,T_*)$}}; 
     
\draw[name path=A,red, thick] (-.1,0) -- (3,2) ;
\draw[name path=A,red, thick,dashed] (-.7,0) -- (2.9,2) ;
      
\draw[red] (1.2,1.4) node { $\eta$};  

\draw[blue] (2.7,1.1) node { $\phi$}; 

\draw[blue, thick,dashed] (1.5,0) -- (2.9,2) ;
\draw[black] (1.9,-.15) node {{\footnotesize $\chi_*$}};

\draw[blue, thick] (1.8,0) -- (3,2) ;
\draw[black] (1.5,-.15) node {{\footnotesize $\chi$ }};

\draw[black] (-0.7,-.13) node {$x$};
\draw[black] (-0.0,-.15) node {$x_*$};
\end{tikzpicture}

\vspace{-0.2cm}
\caption{\footnotesize Characteristic evolution during the pre-shock formation. The blowup point is $\xi_*$, the blowup time is $T_*$, and the blowup label $x_*$ satisfies $\eta(x_*,T_*) = \xi_*$. In red, we display the $3$-characteristics $\eta(\cdot,t)$ originating from the blowup label $x_*$ and a nearby label $x$, while in blue we display the $2$-characteristics $\phi(\cdot,t)$ originating from the label $\chi_* = \phi^{-1}(\xi_*,T_*)$ and a nearby label $\chi$.}
\end{figure}

\subsubsection{Identities involving the $3$-characteristics $\eta$}

From \eqref{eta-eqn} it follows that 
\begin{align} 
\eta(x,t) = x + \int_{-\eps}^t w(\eta(x,t'),t') dt' \label{eq00}
\end{align} 
and from \eqref{xland-w} that 
$$\p_t w(\eta(x,t),t) = - \tfrac 83 a(\eta( x, t),t) w (\eta(x,t),t)\,.$$
We define the integrating factor
\begin{align} 
I_t(x) & =  e^{-\frac{8}{3}  \int_{-\eps}^t a(\eta(x,r),r)dr} \,, \label{Is} 
\end{align} 
Integration yields
\begin{align} 
w(\eta(x,t),t) = I_t(x) w_0(x) \,. \label{eq0}
\end{align} 
We make use of the following identities:
\begin{subequations} 
\label{dx-I-identities}
 \begin{align} 
I_\tau'  I_\tau^{-1} & = -\tfrac{8}{3}    \int_{-\eps}^\tau a' \circ \eta \  \eta_x dr   \,, \label{Is1} \\
 I_\tau''   I_\tau^{-1} & =    \Bigl( \tfrac{8}{3}  \int_{-\eps}^\tau a' \circ \eta \ \eta_x dr 
\Bigr)^2 - \tfrac{8}{3}  \int_{-\eps}^\tau  \bigl( a'' \circ \eta \ \eta_x^2 + a' \circ \eta \ \eta_{xx} \bigr) dr 
 \label{Is2}\\
 I_\tau'''  I_\tau^{-1}  & =  -\tfrac{512}{27} \bigl(\int_{-\epsilon }^\tau a' \circ \eta \ \eta_x\, dr\bigr)^3
+\tfrac{64}{3}   \bigl(\int_{-\epsilon }^\tau a' \circ \eta \ \eta_x dr\bigr)
 \int_{-\epsilon }^\tau \bigl(a'' \circ \eta \ \eta_x^2 + a' \circ \eta \ \eta_{xx}\bigr) dr \notag \\
 & \qquad -\tfrac{8}{3}  \int_{-\epsilon }^\tau \bigl( a''' \circ \eta \ \eta_x^3 + 3 a'' \circ \eta \ \eta_x \eta_{xx} + a' \circ \eta \ \eta_{xxx}\bigr) dr  
  \,,  \label{Is3} \\
 I_\tau''''   I_\tau^{-1} 
 & 
 =  
\tfrac{4096}{81} \bigl(\int_{-\epsilon }^\tau a' \circ \eta \ \eta_x\, dr\bigr)^4
-\tfrac{1024}{9}   \bigl(\int_{-\epsilon }^\tau a' \circ \eta \ \eta_x dr\bigr)^2
 \int_{-\epsilon }^\tau \bigl(a'' \circ \eta \ \eta_x^2 + a' \circ \eta \ \eta_{xx}\bigr) dr  \notag \\
 & \qquad
  +\tfrac{64}{3}   \bigl(\int_{-\epsilon }^\tau  ( a'' \circ \eta \eta_x^2 +  a' \circ \eta \ \eta_{xx} )dr\bigr)^2 \notag \\
 & \qquad
  +\tfrac{256}{9}   \bigl(\int_{-\epsilon }^\tau   a' \circ \eta \ \eta_{x}  dr\bigr)
   \int_{-\epsilon }^\tau \bigl(  a''' \circ \eta \ \eta_x^3 + 3 a'' \circ \eta \ \eta_x \eta_{xx} + a' \circ \eta \ \eta_{xxx} \bigr) dr
 \notag \\
 & \qquad 
 -\tfrac{8}{3}  \int_{-\epsilon }^\tau \bigl( a'''' \circ \eta \ \eta_x^4 
 + 6 a''' \circ \eta \ \eta_x^2 \eta_{xx} 
 + 3 a'' \circ \eta \ \eta_{xx}^2
 + 4 a'' \circ \eta \ \eta_x \eta_{xxx} 
 + a' \circ \eta \ \eta_{xxxx}\bigr) dr   \, \label{Is4}
 \end{align} 
\end{subequations} 
and from \eqref{eq00},
\begin{subequations} 
\label{derivative-eta-identities}
\begin{align} 
 \eta& = x + w_0  \int_{-\eps}^t I_\tau d\tau \,, \label{eq1} \\
\p_x \eta & = 1 + w_0'  \int_{-\eps}^t I_\tau d\tau  +w_0  \int_{-\eps}^t I'_\tau d\tau    \,, \label{eq2} \\
\p_x^2 \eta & = w_0'' \int_{-\eps}^t I_\tau d\tau +2 w'_0\int_{-\eps}^t I'_\tau d\tau   + w_0 \int_{-\eps}^t I''_\tau d\tau    \,, \label{eq3} \\
\p_x^3 \eta & =  w_0''' \int_{-\eps}^t I_\tau d\tau  +3 w''_0 \int_{-\eps}^t I'_\tau d\tau   +3 w'_0 \int_{-\eps}^t I''_\tau  d\tau 
 + w_0 \int_{-\eps}^t I'''_\tau d\tau    \,, \label{eq4} \\
\p_x^4\eta & =  w_0'''' \int_{-\eps}^t I_\tau d\tau +  4w_0''' \int_{-\eps}^t I_\tau'd\tau  +  6w_0'' \int_{-\eps}^t I_\tau''d\tau   +  4w_0'  \int_{-\eps}^t I_\tau''' d\tau 
+w_0 \int_{-\eps}^t I_s''''ds   \,. \label{eq5}
\end{align} 
\end{subequations}

\subsubsection{Identities involving the $2$-characteristics $\phi$}
We write \eqref{xland-w} as
\begin{align} 
\p_t w + {\tfrac{2}{3}} w \p_x w + {\tfrac{1}{3}} w \p_x w = - \tfrac{8}{3}  a w \,, \label{w-xland2}
\end{align} 
and define the Lagrangian variables
$$
\mathcal{W} = w \circ \phi \,, \ \ \ 
\mathcal{V} = {\tfrac{2}{3}} w \circ \phi = \p_t \phi \,.
$$
Then it follows from the chain-rule that \eqref{w-xland2} can be written as
\begin{align} 
\p_t \mathcal{W}  + \tfrac{1}{2} (\p_x\phi)^{-1} \p_x \mathcal{V} \mathcal{W} = - \tfrac{8}{3} \mathcal{W}  \, a \circ \phi \,.   \label{f-eqn}
\end{align} 
We multiply \eqref{f-eqn} by $(\p_x \phi)^ {\frac{1}{2}} $ to find that
$$
\p_t \big(  (\p_x \phi )^ {\frac{1}{2}} \mathcal{W} \big)  = - \tfrac{8}{3} \big( (\p_x \phi ) ^ {\frac{1}{2}} \mathcal{W} \big) \, a \circ \phi \,,
$$
and hence that
\begin{align} 
\p_x \phi(x,t) 
=  \tfrac{w^2_0(x)}{w^2(\phi(x,t),t)}e^{-{\frac{16}{3}} \int_{-\eps}^t a(\phi(x,s),s)ds}  \,.  
\label{px-phi-great}
\end{align} 
It follows from \eqref{w0-lowerupper}, \eqref{w-lowerupper},   \eqref{a-boot}, and since is $\eps$ small enough, that 
\begin{align} 
\tfrac{12}{25 }  \le \p_x \phi(x,t) \le 2\,, \quad \tfrac{1}{2}  \le \p_x \phi^{-1} (x,t) \le \tfrac{25}{12}
  \,, \qquad t \in [-\eps,T_*) \,.   \label{phix-crude}
\end{align} 
Differentiating \eqref{pt-phi}, we see that $\p_t \p_x \phi =  \tfrac{2}{3}   \p_x w \circ \phi \ \p_x \phi$ and that $\p_x \phi(x,-\eps)=1$.   Hence we have that
\begin{align} 
\tfrac{12}{25}  \le e^{{\frac{2}{3}} \int_{-\eps}^t \p_\theta w(\phi(x,s),s) ds} \le 2\,, \quad  \tfrac{1}{2}  \le e^{-{\frac{2}{3}} \int_{-\eps}^t \p_\theta w(\phi(x,s),s) ds} \le \tfrac{25}{12}
\,, \qquad t \in [-\eps,T_*) \,.  \label{wx-phi-crude}
\end{align}

Differentiating  \eqref{px-phi-great}, we have that
\begin{align} 
\p_x^2 \phi(x,t) 
&=
e^{-{\frac{16}{3}}  \int_{-\eps}^t a(\phi(x,s),s)ds} \tfrac{w^2_0(x)}{w^2(\phi(x,t),t)} \Bigl( -{\tfrac{16}{3}}\int_{-\eps}^t \p_\theta a \circ \phi \, \phi_x ds + 2\tfrac{w_0'}{w_0 } 
-2 \tfrac{\p_\theta w \circ \phi \ \phi_x }{w \circ \phi} \Bigr) \notag\\
&=
\p_x \phi(x,t) \Bigl( -{\tfrac{16}{3}}\int_{-\eps}^t \p_\theta a \circ \phi \, \phi_x ds + 2\tfrac{w_0'}{w_0 } 
-2 \tfrac{\p_\theta w \circ \phi \ \phi_x }{w \circ \phi} \Bigr) 
\,.
 \label{pxx-phi-great}
\end{align} 
Using that $|w_0'(x)|\le \eps^{-1}$, and the bounds \eqref{w0-lowerupper}, \eqref{a-boot}, and \eqref{phix-crude}, we see that
\begin{align} 
\sabs{\p_x^2 \phi(x,t) } \les  \tfrac{1}{\eps} + \sabs{\p_\theta w(\phi(x,t),t)}  \,. \label{dxx-phi-bound0}
\end{align} 

Finally, differentiating \eqref{pxx-phi-great}, 
\begin{align} 
\p_x^3 \phi(x,t)
&= \phi_{xx} \Bigl( -{\tfrac{16}{3}}\int_{-\eps}^t \p_\theta a \circ \phi \, \phi_x ds + 2\tfrac{w_0'}{w_0 } 
-2 \tfrac{w_\theta \circ \phi \ \phi_x }{w \circ \phi} \Bigr) \notag\\
&\qquad 
+ \phi_x
\Bigl( -{\tfrac{16}{3}}\int_{-\eps}^t \bigl( \p_\theta^2 a \circ \phi \, \phi_x^2 + \p_\theta a \circ \phi \, \phi_{xx}\bigr)ds   + 2\tfrac{w_0 w_0''- (w_0')^2}{w_0^2} 
+ 2 \left(\tfrac{\p_\theta  w\circ \phi \ \phi_x}{w \circ \phi}\right)^2 \Bigr) \notag\\
&\qquad\qquad -2 \phi_x \tfrac{\p_\theta^2 w\circ \phi \ \phi_x^2 + \p_\theta w \circ \phi \ \phi_{xx}}{w \circ \phi} 
\,.
 \label{pxxx-phi-great}
\end{align} 
We will make use of the fact that by \eqref{dsdt}, the change of variables formula, and \eqref{best-ever3}, 
\begin{align*} 
\int_{-\eps}^t \abs{ \p_\theta w(\phi(x,t,t')} dt' = \int_{-\log\eps}^s \abs{\p_y W(\Phi_A(y,s'),s')} \beta_\tau ds' \,.
\end{align*} 
As we will show in \eqref{abs-Wy-PhiA}, $\int_{-\log\eps}^s \abs{\p_y W(\Phi_A(y,s'),s')} \beta_\tau ds'  \les 1$. 
Together with \eqref{eq:fat:cat}, \eqref{eq:tilde:W:3:derivative:0}, and  \eqref{a-boot},  we see that
\begin{align} 
\sabs{\p_x^3 \phi(x,t) } 
&\les  \tfrac{1}{\eps^2}  + \sabs{w_0''(x)} + \tfrac{1}{\eps}  \sabs{\p_\theta w (\phi(x,t),t)} + \sabs{\p_\theta w (\phi(x,t),t)}^2
+ \sabs{\p_\theta^2 w (\phi(x,t),t) }  \notag \\
& \qquad\qquad + \int_{-\eps}^t \sabs{\p_\theta^2 a(\phi(x,t'),t')} dt' \,. \label{dxxx-phi-bound0}
\end{align}

\subsection{Characteristics in  self-similar coordinates} 
\subsubsection{$3$-characteristics in self-similar coordinates} 
Having defined the $3$-characteristics $\eta(x,t)$ in \eqref{eta-eqn}, 
we now let $\Phi_W(y,s)$ denote the $3$-characteristic of the transport velocity for $\mathcal{V} _W$ which emanates from the label $y$ so that
\begin{subequations} 
\label{Phi-eqn}
\begin{alignat}{2}
\p_s \Phi_W(y,s) &=
\mathcal{V}_W( \Phi_W(y,s),s) \qquad  && \text{ for }  \ -\log\eps<s< \infty 
\label{Phi-eqn-a} \,, \\
\Phi_W(y,-\log\eps)& = y \,,
\end{alignat} 
\end{subequations} 
where the velocity $\mathcal{V}_W$ is defined in \eqref{eq:VW:def}. 
Before stating the next lemma, we recall from \eqref{IC-mods} that $\xi(-\eps) =0$ and that particle labels are assigned at $t=-\eps \Leftrightarrow s = - \log \eps$.
\begin{lemma}[$3$-characteristics in physical and self-similar coordinates]\label{lem:flow-relation} 
With particle labels related by
\begin{align} 
x = \eps^ {\frac{3}{2}} y\,,  \label{label-relation}
\end{align} 
we have that
\begin{align}
\eta(x,t) =  e^{-\frac{3}{2}s} \Phi_W(y,s) + \xi(t) \,,
\label{best-ever}
\end{align} 
or equivalently 
\begin{align} 
\Phi_W(y,s) = e^{\frac{3}{2}s}\left( \eta(x,t)  - \xi(t)\right) \,.   \label{best-ever2}
\end{align} 
\end{lemma} 
\begin{proof}[Proof of Lemma \ref{lem:flow-relation}]
From \eqref{Phi-eqn-a}, we have that
$$
\p_s \left( e^{-\frac{3}{2}s} \Phi_W(y,s)\right) = \left( e^{-\frac{s}{2}} W(\Phi_W(y,s),s) + \kappa - \dot \xi\right)\beta_\tau e^{-s} \,.
$$
Using \eqref{dsdt} and \eqref{beta-tau}, we see that
\begin{align*} 
\p_t \left( e^{-\frac{3}{2}s} \Phi_W(y,s) + \xi\right) = e^{-\frac{s}{2}} W(\Phi_W(y,s),s) + \kappa \,.
\end{align*} 
Then, from \eqref{w-to-W}, we have that $e^{-\frac{s}{2}} W(y,s) + \kappa(t) = w (e^{-\frac{3}{2}s} y + \xi(t),t )$, and hence 
\begin{align*} 
\p_t \left( e^{-\frac{3}{2}s} \Phi_W(y,s) + \xi\right) = w\left(e^{-\frac{3}{2}s}  \Phi_W(y,s) + \xi(t),t\right) \,.
\end{align*} 
On the other hand, from \eqref{pt-eta} we have $\p_t  \eta(x,t) = w(\eta(x,t),t)$, which then proves the identity \eqref{best-ever}.
\end{proof} 

\subsubsection{$2$-characteristics $\Phi_A$ in self-similar coordinates}
Having defined the $2$-characteristics $\phi$ in $(x,t)$ coordinates, we now define their self-similar counterparts in $(y,s)$ coordinates.
We define the $2$-characteristics $\Phi_A$ by
\begin{subequations} 
\label{PhiA-eqn}
\begin{alignat}{2}
\p_s \Phi_A(y,s) &=\mathcal{V}_A(\Phi_A(y,s),s) \qquad  && \text{ for }  \ -\log\eps<s< \infty 
\label{PhiA-eqn-a} \,, \\
\Phi_A(y,-\log\eps)& = y \,.
\end{alignat} 
\end{subequations} 
where the transport velocity $\mathcal{V} _A$ is given in  \eqref{eq:VA:def}.     In the same way that we established \eqref{best-ever2}, we have that
\begin{align} 
\Phi_A(y,s) = e^{\frac{3}{2}s}\left( \phi(x,t)  - \xi(t)\right) \,,   \label{best-ever3}
\end{align} 
where $x = \eps^ {\frac{3}{2}} y$.
The following integral bound was proven in Corollary 8.4 in \cite{BuShVi2019b}:
\begin{align} 
\sup_{y \in \XXX(-\log\eps) }\int_{-\log\eps}^s \abs{W_y( \Phi_A(y,s'),s')} ds' &  \les 1  \label{abs-Wy-PhiA} \,.
\end{align} 

\subsubsection{The unique blowup trajectory associated to $3$-characteristics}
A basic advantage of the use of self-similar coordinates is that the blowup trajectory can be isolated. In particular, 
all but one of the trajectories $\Phi_W(y,s)$ ``eventually escape'' exponentially fast towards infinity.  

\begin{lemma}[The unique blowup trajectory]\label{lem:Phi-rate}
 There exists a unique blowup label $y_*$ such that 
$$\Phi_W(y_*,s) =e^{\frac{3}{2}s}\left( \eta(x_*,t)  - \xi(t)\right)$$
 is the unique trajectory which converges to  $y=0$ as $s \to \infty $. 
 Moreover,
\begin{align} 
\abs{\Phi_W(y_*,s)}\le 20\kappa_0  e^{- {\frac{5}{2}} s}  \qquad  \operatorname{ for } \operatorname{ all } s \ge -\log\eps \,,
\label{Phi-rate}
\end{align} 
and 
\begin{align} 
|y_*| \le 20\kappa_0 \eps^ {\frac{5}{2}}  \qquad \Leftrightarrow \qquad |x_*| \le 20\kappa_0 \eps^4  \,.  \label{y*-bound}
\end{align}  
\end{lemma} 
\begin{proof}[Proof of Lemma \ref{lem:Phi-rate}]
Using \eqref{Phi-eqn-a}, we can write the evolution equation for $\Phi_W$ as
\begin{align} 
\p_s \Phi_W(y,s) =  \mathcal{V} _W \circ \Phi_W  = \tfrac{1}{2} \Phi_W(y,s) + G_\Phi(y,s) + h(s) \,, \label{Phi-ODE}
\end{align} 
where 
\begin{subequations} 
\label{Gh-need}
\begin{align} 
G_\Phi &= G \circ \Phi_W \,, \\
G &=  \left(\bar W + y\right) +  (1- \beta_\tau ) \bar W + \beta_\tau \tilde W \,, \label{G} \\
h & = e^{\frac{s}{2}} \beta_\tau (\kappa -\dot \xi) \,. \label{h}
\end{align} 
\end{subequations} 
The particular form of  $G_\Phi$ in \eqref{G} is chosen to make use  of the fact that  for all $y$,
\begin{align} 
\sabs{y + \bar W(y)} \le \abs{y}^3 \,, \label{barWfact}
\end{align} 
which follows from the identity $|y + \bar W(y)| = |\bar W (y)|^3$ and the bound $|\bar W(y)| \leq |y|$.

Hence, we integrate \eqref{Phi-ODE} to obtain
\begin{align} 
\Phi_W(y_*,s) & = e^{\frac{s}{2}} \eps^ {\frac{1}{2}} y_* + e^{\frac{s}{2}} \int_{-\log\eps}^s e^{-\frac{s'}{2}} \left( G_\Phi(y_*, s') + h(s')\right) ds'  \,. \label{shiner}
\end{align} 
If $e^{-\frac{s'}{2}} \left( G_\Phi(y_*,s') + h(s')\right)$ is integrable on $[-\log\eps, \infty )$ then, we can rewrite \eqref{shiner} as
\begin{align} 
\Phi_W(y_*,s)
&= e^{\frac{s}{2}} \left( \eps^ {\frac{1}{2}} y_* + \int_{-\log\eps}^\infty \!\! e^{-\frac{s'}{2}} \left( G_\Phi(y_*,s') + h(s')\right) ds' \right) 
- e^{\frac{s}{2}} \int_{s}^\infty\!\!  e^{-\frac{s'}{2}} \left( G_\Phi(y_*,s') + h(s')\right) ds'  \,. \label{Phi-temp0}
\end{align}

Together with \eqref{bootstrap:Wtilde3:at:0}, \eqref{a-boot}, and \eqref{Ayy-bound},  the identity \eqref{eq:dot:xi} shows that
\begin{align}
\sabs{\dot \xi - \kappa} \le 38\kappa_0 e^{-3s} \,,
\label{savestheday}
\end{align} 
so that using \eqref{h} and \eqref{savestheday}, we have the bound
\begin{align} 
\abs{h(s)} \le 39 \kappa_0 e^ {-\frac{5}{2}s}  \,, \label{hbound1}
\end{align}
so the integrability of $ e^ {- \frac{s'}{2}} G_\Phi(y_*, s' )$ will be of paramount importance.

We additionally note  that since the first term on the right side of \eqref{Phi-temp0} is a constant multiplying $e^{\frac{s}{2}} $, 
in order for $ \Phi_W(y_*,s)  \to 0   $ as $s \to \infty $, this constant  must vanish, and thus, we must insist that
\begin{subequations} 
\label{fat-mouse}
\begin{align} 
y_*  & = - \eps^ {-\frac{1}{2}} \int_{-\log\eps}^\infty  e^{-\frac{s'}{2}}  \left( G_\Phi(y_*,s') + h(s')\right) ds'  \,,  \label{y*}
\end{align}
which then implies
\begin{align}
\Phi_W(y_*,s)& = - e^{\frac{s}{2}} \int_{s}^\infty  e^{-\frac{s'}{2}} \left( G_\Phi(y_*,s') + h(s')\right) ds'  \,. \label{Phi-y*}
\end{align} 
\end{subequations} 
Notice that \eqref{fat-mouse} implies that as long as $ e^ {-\frac{s'}{2}} G_\Phi(y_*, s' )$ is integrable,
\begin{align*} 
\Phi_W(y_*, -\log \eps) = y_* \,, \ \ \text{ and } \ \ \lim_{s \to \infty } \Phi_W(y_*,s)=0 \,.
\end{align*} 

We shall now establish the existence of a unique trajectory $\Phi_W(y_*,s)$ solving \eqref{Phi-y*}.   We define the
set
\begin{align*} 
\mathcal{T} = \{ \varphi \in C^0([-\log\eps, \infty )) \ : \ \abs{\varphi(s)} \le 20  \kappa _0   e^{-{\frac{5}{2}} s} \} \,,
\end{align*} 
with norm given by $\snorm{\varphi}_ \mathcal{T} := \sup_{s \in [-\log\eps, \infty ) } e^{{\frac{5}{2}} s}\abs{  \varphi(s)}$, 
and consider the map $\Psi$, which maps $\bar\varphi  \in \mathcal{T} $ to $\varphi$,  given by
\begin{align*} 
\varphi(s) = \Psi(\bar \varphi(s) ) :=  - e^{\frac{s}{2}} \int_{s}^\infty  e^{-\frac{s'}{2}} \left( G_{\bar \varphi}(s')  + h(s')\right) ds' \,.
\end{align*} 
We note that for $\bar \varphi \in \mathcal{T} $, $\abs{ \bar\varphi} \le  \alpha \kappa_0 \eps^{ {\frac{5}{2}} } \le \ell $ for
$\eps$ small enough,  so that we may apply the bounds
 \eqref{tildeWbootstrap0} to the function $ G_{\bar \varphi}(s') $.   Doing so, we see that 
the bounds \eqref{barWfact},  \eqref{tildeWbootstrap0} and \eqref{mod2-bound} show that
 for $\eps$ taken small enough,
\begin{align*} 
\abs{G_{\bar \varphi}(s)} 
&\le \abs{{\bar \varphi}(s)}^3 + (1+ 6 \eps^2)  \left((\log M)^4 \eps^ {\frac{1}{10}} \abs{{\bar \varphi}(s)}^{4} 
+ M \eps^ {\frac{1}{3}} \abs{{\bar \varphi}(s)}^3 \right) + 6 \eps e^{-s} \bar \varphi(s) \\
&
\le ( 20 \kappa_0)^3 e^{-{\frac{15}{2}}  s} 
+ \eps^{ \frac{1}{12}} \eps^{3 \alpha } (20 \kappa_0)^4 e^{-10 s}
+ 120 \kappa_0 \eps e^{- {\frac{7}{2}} s}  
 \le 122 \kappa_0 \eps e^{- {\frac{7}{2}} s}  \,.
\end{align*} 
 
Together with \eqref{hbound1}, we have that
\begin{align*} 
e^{-\frac{s'}{2}} \left(\sabs{G_{\bar \varphi}(s')}  + \sabs{h(s') }\right) \le 40  \kappa_0  e^{-3s'}  \,.
\end{align*} 
By the fundamental theorem of calculus,   $ s\mapsto \varphi(s)$ is continuous, and satisfies the bound
\begin{align*} 
\abs{ \varphi(s) } \le 18 \kappa_0  e^{- {\frac{5}{2}} s} \ \ \text{ for all } s \ge -\log\eps \,.
\end{align*} 
Therefore, $\Psi : \mathcal{T} \to \mathcal{T} $.   

Let us now prove that $\Psi$ is a contraction.  Suppose that $\varphi_1 = \Psi(\bar \varphi_1)$ and
$\varphi_2 = \Psi(\bar \varphi_2)$.
We then  have
\begin{align} 
\abs{ \varphi_{1}(s) - \varphi_2(s) } \le e^{\frac{s}{2}} \int_{s}^\infty  e^{-\frac{s'}{2}} \abs{ G_{\bar \varphi_1}(s')  -G_{\bar \varphi_2}(s')} ds'  \,. \label{scheme}
\end{align} 
From \eqref{barWfact}, we have that 
$$
\abs{ \bar W(y_1) +y_1 - \bar W(y_2)  - y_2} \le \abs{y_1^3-y_2^3}\,,
$$
so that
\begin{align} 
\abs{ \left(\bar W(\bar\varphi_1(s)) + \bar \varphi_1(s)\right) -  \left(\bar W(\bar\varphi_{2}(s)) + \varphi_{2}(s)\right)} &\le \abs{ \bar\varphi_1^3(s)- \bar\varphi_ {2}^3(s)} \notag \\
 &\le \abs{ \bar\varphi_1(s)^2+\bar\varphi_1(s)\bar \varphi_2(s)+ \bar\varphi_ {2}(s)^2}\abs{ \bar\varphi_1(s)- \bar\varphi_ {2}(s)}  \notag \\
& \le \eps^2 e^{-s} \abs{ \bar\varphi_1 (s)- \bar\varphi_ {2} (s)}  \,, \label{G1}
\end{align}
where we have used that both $ \bar\varphi_1$ and $ \bar\varphi_2$ are in $ \mathcal{T} $.
Next, since $\abs{ \bar W(y_1) - \bar W(y_2) } \le \abs{x-y}$,  by \eqref{mod2-bound},
\begin{align} 
\abs{1- \beta_\tau } \abs{ \bar W(\bar\varphi_1(s)) - \bar W(\bar\varphi_{2}(s)) } 
\le 6 \eps e^{-s} \abs{ \bar\varphi_1 (s)- \bar\varphi_ {2} (s)} \,, \label{G2}
\end{align} 
and finally, employing the mean value theorem together with the bound \eqref{tildeWbootstrap0},  for some a function $s\mapsto \alpha(s) \in (0,1)$ and
\begin{align} 
\abs{ \beta_\tau}\abs{  \tilde W (\bar\varphi_1(s),s) - \tilde W (\bar\varphi_{2}(s),s)}
& \le 2 \abs{\p_y \tilde W((1- \alpha (s))\bar\varphi_1(s)+ \alpha (s)  \bar\varphi_{2}(s)   ,s) } \abs{ \bar\varphi_1 (s)- \bar\varphi_ {2} (s)}  \notag \\
& \le 2 \left( (\log M)^4 \eps^ {\frac{1}{10}}  (20\kappa_0)^3  e^{- {\frac{15}{2}} s}  +M \eps^ {\frac{1}{3}} (20\kappa_0)^2  e^{- 5 s}  \right)
\abs{ \bar\varphi_1 (s)- \bar\varphi_{2} (s)} \notag  \\
& \le \eps^4 e^{-s} \abs{ \bar\varphi_1 (s)- \bar\varphi_{2} (s)} \,. \label{G3}
\end{align} 
 Combining the bounds \eqref{G1}, \eqref{G2}, and 
\eqref{G3}, and taking $\eps$ sufficiently small,  we have that
\begin{align*} 
\abs{ G_{\bar \varphi_1}(s')  -G_{\bar\varphi_2}(s')}  \le 7 \eps e^{-s} \abs{ \bar\varphi_1 (s)- \bar\varphi_2(s)} \,,
\end{align*} 
and thus from \eqref{scheme}, we see that
\begin{align*} 
e^{{\frac{5}{2}} s}\abs{ \varphi_{1}(s) - \varphi_2(s) } &\le e^{3 s} \int_{s}^\infty  e^{-\frac{1}{2}s'}  \abs{ G_{\bar \varphi_1}(s')  -G_{\bar \varphi_2}(s')} ds'  \\
&\le 7 \eps e^{3 s} \int_{s}^\infty  e^{-\frac{1}{2}s'} \abs{ \bar \varphi_1(s')  -\bar \varphi_2(s')} ds'   \\
&\le 14 \eps \sup_{s \in [-\log\eps, \infty ) } e^{{\frac{5}{2}}  s}\abs{ \bar \varphi_1(s)  -\bar \varphi_2(s)}
 \,, 
\end{align*}
so that 
\begin{align*} 
\norm{ \varphi_{1} - \varphi_2 }_\mathcal{T}  \le 14 \eps \norm{ \bar\varphi_1- \bar\varphi_ {2}}_\mathcal{T}  \,,
\end{align*} 
which shows that $\Psi $ is a contraction.   By the contraction mapping theorem, there exists a unique trajectory $\varphi \in \mathcal{T} $ such that for all
$s \ge -\log \eps$,
\begin{align*} 
\varphi(s) & =  - e^{\frac{s}{2}} \int_{s}^\infty  e^{-\frac{s'}{2}} \left(\left(\bar W(\varphi(s)) + \varphi(s)\right) +  (1- \beta_\tau ) \bar W( \varphi(s)) 
+ \beta_\tau \tilde W( \varphi(s),s)  + h(s')\right) ds' \,,
\end{align*}
or equivalently
\begin{align*} 
e^{-\frac{s}{2}}\varphi(s) 
& =   -\int_{s}^\infty  e^{-\frac{s'}{2}} \left( \varphi(s)  + \beta_\tau W( \varphi(s),s)  + h(s')\right) ds' \,.
\end{align*} 
Differentiating this identity in self-similar time shows that
\begin{align*}
\p_s \varphi = \mathcal{V}_W \circ \varphi  \,.
\end{align*} 
Setting 
$$
y_* = 
- \eps^ {-\frac{1}{2}}  \int_{-\log\eps}^\infty  e^{-\frac{s'}{2}} \left(\left(\bar W(\varphi(s)) + \varphi(s)\right) +  (1- \beta_\tau ) \bar W( \varphi(s)) 
+ \beta_\tau \tilde W( \varphi(s),s)  + h(s')\right) ds'  \,,
$$
we see that $\varphi(-\log\eps) = y_*$ from which it follows that
$$
\Phi_W(y_*,s) = \varphi(s)  \ \ \text{ for all } s \ge -\log\eps\,,
$$
and $\Phi_W(y_*,s)$ is a solution to \eqref{fat-mouse}.
Clearly $\sabs{y_*} \le 20\kappa_0 \eps^ {\frac{5}{2}} $ and by \eqref{label-relation}, it follows that $\sabs{x_*} \le 20\kappa_0 \eps^4$.

We next show that $y_*$ is the only blowup label.
From \eqref{eq:VA:def} and  \eqref{Phi-eqn}, we have that
\begin{align*} 
\p_s (\Phi_W(y_*,s) - \Phi_W(y,s)) = \tfrac{3}{2}  (\Phi_W(y_*,s) - \Phi_W(y,s)) +  \beta_\tau W(\Phi_W(y_*,s),s) - \beta_\tau W(\Phi_W(y,s),s)  \,.
\end{align*} 
Suppose that $y_* \ge y$.
By the mean value theorem and the bound \eqref{mod2-bound}, we have that
\begin{align*} 
\abs{ \beta_\tau W(\Phi_W(y_*,s),s) - \beta_\tau W(\Phi_W(y,s) ,s)}  \le (1+ 6\eps) (\Phi_W(y_*,s) - \Phi_W(y,s)) \,.
\end{align*} 
Here we have used the global bound $\abs{\p_yW(y,s)} \le 1$ and the fact that characteristics cannot cross so that $\Phi_W(y_*,s) - \Phi_W(y,s)\ge 0.$
Therefore,
\begin{align*} 
\p_s (\Phi_W(y_*,s) - \Phi_W(y,s)) \ge ( \tfrac{1}{2} - \eps^ \frac{3}{4} )  (\Phi_W(y_*,s) - \Phi_W(y,s))  \,,
\end{align*} 
and then
\begin{align*} 
\Phi_W(y_*,s) - \Phi_W(y,s) \ge  \eps^ {\frac{1}{2}} e^{(\frac{1}{2}-\eps^ \frac{3}{4} )s} (y_*-y) \,.
\end{align*} 
If $y \ge y_*$, in the same way we, we obtain $\Phi_W(y,s) - \Phi_W(y_*,s) \ge  \eps^ {\frac{1}{2}} e^{(\frac{1}{2}-\eps^ \frac{3}{4} )s} (y-y_*)$. 
\end{proof}

\subsection{Bounds for $\p_x^\gamma a$, $\gamma \le 4$}
\label{sec:abounds}
\subsubsection{Improving the bootstrap  bound for $a$}
We note here that from \eqref{xland-svort:def} and \eqref{xland-svort},  the specific vorticity $\varpi = \tfrac{16}{w^2} (w-a_x)$ solves
$$
\p_t\varpi +  \tfrac{2}{3}  w \p_x \varpi= \tfrac{8}{3} a \varpi \,, \ \ \varpi(x,-\eps)= \varpi_0(x) \,,
$$
and hence 
\begin{align} 
\varpi( \phi(x,t),t) = e^{ \frac{8}{3}\int_{-\eps}^t a( \phi(x,t'),t') dt' }\varpi_0(x) \,. \label{svort-phi}
\end{align}   
We also have from \eqref{xland2-a}, that 
\begin{align} 
a (\phi(x,t),t)  = a_0(x) + \int_{-\eps}^t  (- \tfrac{4}{3}  a^2 + \tfrac{1}{6}   w^2) \circ \phi ds \,  \label{a-pos}
\end{align} 
so that assuming the bootstrap bound $\sabs{a(\theta,t)} \le 2 \kappa_0^2 \eps$ and using \eqref{eq:A_bootstrap:IC} and \eqref{w-lowerupper}, we find that
for $\eps$ taken sufficiently small,
\begin{align} 
\snorm{a( \cdot , t)}_{L^ \infty } \le  \tfrac{3}{2} \kappa_0^2 \eps\,, \   \ \ t \in [-\eps,T_*) \,,  \label{a-bound1}
\end{align} 
which improves the bootstrap bound \eqref{dx0a-boot}.

\subsubsection{Improving the bootstrap bound for $\p_\theta a$}
From \eqref{phix-crude}, we see that $\phi( \cdot, t)$ is a diffeomorphism with a well-defined inverse map, so that for
each $t \in [-\eps,T_*)$ and for $\eps$ small enough, the identity \eqref{svort-phi}  and the bound \eqref{a-bound1} show that
\begin{align} 
(1- \eps) \varpi_0(\theta) \le \varpi(\phi(\theta,t),t) \le   (1+ \eps ) \varpi_0(\theta) \,, \ \ t \in [-\eps,T_*) \,,  \label{sp-vort-bound}
\end{align} 

 From \eqref{w0-lowerupper}, 
 $ \tfrac{7}{8} \kappa_0 \le w_0(\theta) \le  \tfrac{9}{8} \kappa_0$.  
 Since $\varpi_0 = \tfrac{16}{w_0^2} (w_0- \p_\theta a_0)$,  by  \eqref{eq:A_bootstrap:IC}, we then have that
 for $\eps$ sufficiently small, 
$$
\tfrac{101}{10\kappa_0 } \le \varpi_0(\theta) \le \tfrac{27}{\kappa_0 } \,,
$$
 and by \eqref{sp-vort-bound},  for $\eps$ small enough,  
\begin{align} 
\tfrac{10}{\kappa_0 } \le \varpi(\theta,t) \le \tfrac{28}{\kappa_0 } \,, \ \ \theta \in \mathbb{T} , t\in[-\eps,T_*)   \,.  \label{spvort-bound1}
\end{align} 
Again using that 
\begin{align} 
\p_\theta a = w - \tfrac{w^2}{16} \varpi \,,  \label{ax-varphi-w}
\end{align} 
we then have that
\begin{align} 
\sabs{\p_\theta a(\theta,t)}=  \sabs{w-\tfrac{w^2}{16}\varpi } \le  \tfrac{3}{2}  \kappa_0 \,,  \ \ \theta \in \mathbb{T} , t\in[-\eps,T_*)   \,,  \label{ax-bound1}
\end{align} 
which  improves the bootstrap bound  \eqref{dx1a-boot}.

\subsubsection{Improving the bootstrap bound for $\p_\theta^2a$}
Differentiating \eqref{svort-phi}, we have that
\begin{align} 
\p_\theta \varpi ( \phi(x,t),t) 
& = ( \p_x \phi(x,t)) ^{-1} e^{ -\frac{8}{3}\int_0^t a( \phi(x,t'),t') dt' } \left( \p_\theta  \varpi_0(x)  +\tfrac{8}{3}\varpi_0\int_{-\eps}^t 
 \p_\theta a( \phi(x,t'),t')  \p_x \phi(x,t')dt' 
\right) \notag\\
&= ( \p_x \phi(x,t)) ^{-1} \varpi ( \phi(x,t),t)
\left( \tfrac{\p_\theta  \varpi_0(x)}{\varpi_0(x)}  +\tfrac{8}{3} \int_{-\eps}^t 
 \p_\theta a( \phi(x,t'),t')  \p_x \phi(x,t')dt' 
\right) \,. \label{dxspvort-phi}
\end{align} 
It follows from \eqref{phix-crude},  \eqref{sp-vort-bound}, \eqref{ax-bound1},  and  \eqref{dxspvort-phi} that for $\eps$ small enough,
\begin{align} 
\sabs{ \p_\theta\varpi(\phi(x,t),t) }\le  \tfrac{51}{24} \sabs{\p_\theta \varpi_0(x)}  + 500 \eps\,.  \label{varpi-pos}
\end{align} 
Using the formula
$$
\p_\theta \varpi_0 = \tfrac{16}{w_0^2} (\p_\theta w_0-\p_x^2 a_0) -  \tfrac{32}{w_0^3} (w_0-\p_\theta a_0) \p_\theta w_0 
$$
and the bounds  \eqref{w0-lowerupper}, $ -\tfrac{1}{\eps}  \le \p_\theta w_0(x)$, and \eqref{eq:A_bootstrap:IC}, we estimate that
\begin{align} 
\sabs{\p_\theta \varpi_0(x)} \le \tfrac{34}{\kappa_0^2\eps } \,,  \label{dx-varpi-0}
\end{align} 
and hence from \eqref{varpi-pos}, 
\begin{align} 
\sabs{\p_\theta \varpi(x,t )} \le \tfrac{70}{\kappa_0^2\eps } \,,  \ \ x \in \mathbb{T} , t\in[-\eps,T_*)  \,.  \label{dx-spvort-bound1}
\end{align} 
We shall use the fact that 
\begin{align} 
\p_\theta^2 a = \p_\theta w (1- \tfrac{1}{8} w \varpi ) - \tfrac{w^2}{16}  \p_\theta \varpi \,,  \label{d2a-dw}
\end{align} 
so that combined with the above estimates,
\begin{align} 
\sabs{\p_\theta^2 a(x,t)} \le \tfrac{7}{2}   \sabs{\p_\theta w(x,t)}+  \tfrac{7}{\eps}   \,,  \label{goodmusickeepsmegoing}
\end{align}
and hence by \eqref{w-to-W}, we have that
\begin{align} 
\sabs{\p_y^2 A(y,s)} \le \tfrac{7}{2}  e^{-2s} \sabs{\p_y W(y,s)}  + e^{-\frac{3s}{2}}  \tfrac{7}{\eps}   \le \tfrac{23}{2}  e^{-2s}  \,, \label{Ayy-bound}
\end{align}
where we have used that $\abs{\p_y W(y,s) } \le 1$ as proven in \cite{BuShVi2019b}.    
This then implies that
\begin{align} 
\sabs{\p_\theta^2 a(x,t)} = e^{3s} \sabs{\p_y^2 A(y,s)} \le  \tfrac{23}{2}  e^{s}   \,,   \ \ x \in \mathbb{T} , t\in[-\eps,T_*)  \label{axx-bound}
\end{align}
which  improves the bootstrap  bound \eqref{dx2a-boot}.

\subsubsection{A bound for $\p_\theta^3a$}
We next differentiate \eqref{dxspvort-phi} to obtain
\begin{align} 
\p_\theta^2\varpi ( \phi(x,t),t)   
&=\left(\phi_x^{-1}  \p_\theta \varpi \circ \phi -\phi_x^{-3} \phi_{xx} \varpi \circ \phi  \right)
\left( \tfrac{\p_\theta  \varpi_0 }{\varpi_0 }  +\tfrac{8}{3} \int_{-\eps}^t 
 \p_\theta a \circ\phi \   \phi_x dt' 
\right)
\notag \\
& + \phi_x^{-2} \varpi \circ \phi   \Bigl( \tfrac{\varpi_0  \p_\theta^2  \varpi_0 - (\p_\theta \varpi_0 )^2}{\varpi_0^2 }  
+\tfrac{8}{3} 
\underbrace{\int_{-\eps}^t  \bigl(\p^2_\theta a \circ \phi \ \phi_x^2 + \p_\theta a \circ \phi \ \phi_{xx} \bigr) dt'}_{ \mathcal{R}(x,t) } \Bigr)
 \,. \label{dxxspvort-phi}
\end{align} 
We first bound the integral $\mathcal{R} $. By  \eqref{phix-crude}, \eqref{dxx-phi-bound0}, and \eqref{goodmusickeepsmegoing}, we have that
\begin{align} 
\sabs{ \mathcal{R} (x,t) } \les  \int_{-\eps}^t \sabs{ \p_\theta a \circ \phi} ( \tfrac{1}{\eps}  + \sabs{ \p_\theta w \circ \phi}) dt' \,.
\label{much-easier-1}
\end{align} 
We note that by \eqref{best-ever3}, 
\begin{align*} 
\p_\theta w(\phi(x,t),t) = e^sW_y(\Phi_A(y,s),s) \,.
\end{align*} 
The identity \eqref{dsdt} then shows that $dt = \beta_\tau e^{-s} ds$ so that by the change of variables formula, we have that
\begin{align} 
\int_{-\eps}^t \sabs{\p_\theta w(\phi(x,t'),t')} dt' 
=\int_{-\log \eps}^s \sabs{W_y(\Phi_A(y,s'),s') }  \beta_\tau  ds' \les 1 \,,  \label{integral-wx-phi-bound}
\end{align} 
where we have used \eqref{abs-Wy-PhiA} for the last inequality.  Hence, with 
\eqref{a-boot} and \eqref{much-easier-1}, we have that
\begin{align} 
\sabs{\mathcal{R}(x,t)} \les 1 \,. \label{Rboundisuseful}
\end{align} 
With \eqref{Rboundisuseful},  the formula \eqref{dxxspvort-phi} and the bounds \eqref{a-boot} and \eqref{dx-varpi-0} allow us to
estimate $\p_\theta^2 \varpi \circ \phi$ in the following way:
\begin{align} 
\sabs{\p_\theta^2 \varpi ( \phi(x,t),t)   } & \les 1+ \tfrac{1}{\eps} \sabs{\phi_{xx}(x,t)} + \sabs{\p_\theta^2 \varpi_0 (x)   }  
 \les  \tfrac{1}{\eps^2} + \tfrac{1}{\eps} \sabs{\p_\theta w(\phi(x,t),t)} + \sabs{\p_\theta^2 \varpi_0 (x)   }   \label{dxx-varphi-phi}
\end{align} 
where we have used \eqref{dxx-phi-bound0} for the last inequality.

Differentiating  \eqref{d2a-dw} yields the identity
\begin{align} 
\p_\theta^3 a = \p_\theta^2 w (1- \tfrac{1}{8} w \varpi ) - \tfrac{w^2}{16}  \p_\theta^2 \varpi  - \tfrac{1}{4} w \p_\theta w  \p_\theta \varpi - \tfrac{1}{8}   \varpi  (\p_\theta w)^2
\label{d3a-d2w}
\end{align}
so that
\begin{align} 
\sabs{\p_\theta^3 a(x,t)} &  \les \sabs{\p_\theta^2 w (x,t)}  +\sabs{\p_\theta^2 \varpi (x,t)} + \sabs{\p_\theta w (x,t)} ^2 + \tfrac{1}{\eps} \sabs{\p_\theta w (x,t)} \notag  \\
&  \les  \tfrac{1}{\eps^2} +  \sabs{\p_\theta^2 w (x,t)}  +\sabs{\p_\theta^2 \varpi_0 (\phi ^{-1} (x,t),t)} + \sabs{\p_\theta w (x,t)} ^2 + \tfrac{1}{\eps} \sabs{\p_\theta w (x,t)} 
\label{dxxx-a}
\end{align} 
where we have used \eqref{dxx-varphi-phi} for the last inequality.

Restricting the identity \eqref{d3a-d2w} to $t=-\eps$, we see that
\begin{align} 
 \tfrac{w_0^2}{16}  \p_\theta^2 \varpi_0 =
-\p_\theta^3 a_0 + \p_\theta^2 w_0 ( \tfrac{1}{8} w_0 \varpi _0-1) 
- \tfrac{1}{4} w_0 \p_\theta w_0 \p_\theta\varpi_0  - \tfrac{1}{8}  (\p_\theta w_0)^2\varpi_0  \,, \label{d2vort0}
\end{align} 
and so
\begin{align} 
\sabs{\p_\theta^2 \varpi_0} \les  \tfrac{1}{\eps^2} + \sabs{\p_\theta^3 a_0 } + \sabs{\p_\theta^2 w_0 }  \les  \tfrac{1}{\eps^2} +  \sabs{\p_\theta^2 w_0 } 
\label{just-write-it}
\end{align} 
since we assumed that $\abs{\p_\theta^3a_0(x)} \les 1$ in \eqref{eq:A_bootstrap:IC}.   Using the bound \eqref{just-write-it} in \eqref{dxxx-a} shows that
\begin{align} 
&\sabs{\p_\theta^3 a(\eta(x,t),t)} \notag \\
& \ \
  \les  \tfrac{1}{\eps^2} +  \sabs{\p_\theta^2 w (\eta(x,t),t)}  +\sabs{\p_\theta^2 w_0 (\phi ^{-1} (\eta(x,t),t),t)} + \sabs{\p_\theta w (\eta(x,t),t)} ^2 + \tfrac{1}{\eps} \sabs{\p_\theta w (\eta(x,t),t)} 
  \,.
\label{dxxx-a-eta}
\end{align} 
By   \eqref{d2W}, we have that for $x \in \mathbb{T}  $,
$\abs{\p_\theta^2 w_0(x) } \les  \eps^{-{\frac{5}{2}} }$
and therefore
\begin{align} 
\sabs{\p_\theta^2 w_0 (\phi ^{-1} (\eta(x,t),t),t)} \les  \eps^{-{\frac{5}{2}} }   \,.  \label{d2w0-eta}
\end{align} 
Using this bound in \eqref{dxxx-a-eta}, for all $t\in [-\eps, T_*)$, 
\begin{align} 
\sabs{\p_\theta^3 a(\eta(x,t),t)}
&  \les  \eps^{-{\frac{5}{2}} }   +  \sabs{\p_\theta^2 w (\eta(x,t),t)}   + \sabs{\p_\theta w (\eta(x,t),t)} ^2 + \tfrac{1}{\eps} \sabs{\p_\theta w (\eta(x,t),t)}\label{dxxx-a-final0} \,.
\end{align}  

\subsubsection{A  bound for $\p_\theta^4 a$}
As we will now explain, the bound for $\p_\theta^4 a(x,t)$ does not depend on $\p_x^4 \eta$, $\p_x^4\phi$, or $\p_\theta^4 w$, and as such is merely a consequence of
the bounds that have already been established.

To obtain this bound, we make one final differentiation of  \eqref{dxxspvort-phi} and obtain that

\begin{align}
&\p_\theta^3\varpi ( \phi(x,t),t) \notag\\
&= \left(\phi_x ^{-1}  \p_\theta^2 \varpi \circ \phi   + 3 \phi_x^{-5} \phi_{xx}^2 \varpi \circ \phi -\phi_x^{-4} \phi_{xxx} \varpi \circ \phi  
 - 2 \phi_x^{-3} \phi_{xx} \p_\theta \varpi \circ \phi   \right)
\left( \tfrac{\p_\theta  \varpi_0 }{\varpi_0 }  
+\tfrac{8}{3} \int_{-\eps}^t \!\!\! \p_\theta a \circ\phi \   \phi_x dt' 
\right) \notag\\
&
+ \left(2\phi_x^{-2}  \p_\theta \varpi \circ \phi - 3\phi_x^{-4} \phi_{xx} \varpi \circ \phi  \right)
\Bigl( \tfrac{\varpi_0  \p_\theta^2  \varpi_0 - (\p_\theta \varpi_0 )^2}{\varpi_0^2}  
+\tfrac{8}{3} 
\int_{-\eps}^t  \!\! \bigl(\p^2_\theta a \circ \phi \ \phi_x^2 + \p_\theta a \circ \phi \ \phi_{xx} \bigr) dt'  \Bigr)
\notag \\
&+  \phi_x^{-3} \varpi \circ \phi  
\Bigl( \tfrac{\varpi_0^2  \p_\theta^3  \varpi_0 - 3 \varpi_0 \p_\theta \varpi_0 \p_\theta^2 \varpi_0 +2 (\p_\theta \varpi_0 )^3}{\varpi_0^3}  
+\tfrac{8}{3} \underbrace{\int_{-\eps}^t \!\!  \bigl(\p_\theta^3 a \circ \phi \ \phi_x^3 
+ 3 \p_\theta^2 a \circ \phi \ \phi_x \phi_{xx}
+ \p_\theta a \circ \phi \ \phi_{xxx} \bigr) dt'}_{ \mathcal{S} (x,t)} \Bigr)
.
\label{dxxxspvort-phi}
\end{align}
Our goal is to bound $\sabs{\p_\theta^3\varpi ( \phi(x,t),t) }$ using the identity \eqref{dxxxspvort-phi}.   The time integral in the  first  line is $\OO(\eps)$ due to \eqref{a-boot} and \eqref{phix-crude}.
The time  integral  in the second  line is the term $ \mathcal{R} (x,t)$ in \eqref{dxxspvort-phi}, which was estimated in \eqref{Rboundisuseful}.  It thus remains to establish the
bound for the integral term $ \mathcal{S} (x,t)$ on the third line.    We write
$ \mathcal{S} = \mathcal{S} _1 + \mathcal{S} _2 + \mathcal{S} _3$, where
\begin{subequations} 
\label{S123}
\begin{align} 
\mathcal{S} _1(x,t) &= \int_{-\eps}^t  \p_\theta^3 a \circ \phi \ \phi_x^3  dt'\,, \\
\mathcal{S} _2(x,t) &= \int_{-\eps}^t  3 \p_\theta^2 a \circ \phi \ \phi_x \phi_{xx}  dt'\,, \\
\mathcal{S} _3(x,t) &= \int_{-\eps}^t   \p_\theta a \circ \phi \ \phi_{xxx}  dt' \,,
\end{align} 
\end{subequations} 
and we shall first estimate the integral $\mathcal{S} _3$.   
The key idea in estimating $\mathcal{S} _3$ is to use the identity  \eqref{pxxx-phi-great} for $\phi_{xxx}$ and isolate the
term 
$$
\p_\theta^2 w\circ \phi \ \phi_x^2 + \p_\theta w \circ \phi \ \phi_{xx} =: \p_x (\p_\theta w \circ \phi  \ \phi_x) \,,
$$
and estimate its integral in a very careful manner.

The identity for $\phi_{xxx}$ in \eqref{pxxx-phi-great} and the bound \eqref{dxxx-phi-bound0}, together with the estimates 
\eqref{goodmusickeepsmegoing} and \eqref{integral-wx-phi-bound}, and the integral bound  \eqref{abs-Wy-PhiA},
we conclude that
 \begin{align}
 \sabs{\mathcal{S} _3(x,t)} &\les \tfrac{1}{\eps} + \eps \sabs{ w_0''(x)}  
 + \int_{-\eps}^t \sabs{ \p_\theta w(\phi(x,t'),t')}^2dt' +  \sabs{\mathcal{S} _4(x,t)}   \,,
 \label{S3-prelim}
\end{align} 
where the term $ \mathcal{S} _4$ 
contains the important term on the last line of \eqref{pxxx-phi-great}, and is given by
\begin{align} 
\mathcal{S} _4(x,t)=
 \int_{-\eps}^t  \p_x(  a \circ \phi ) (w \circ \phi)^{-1}  \p_x ( \p_\theta w \circ \phi \ \phi_x) dt' 
\,.  \label{S4}
\end{align} 

We now rewrite the evolution equation 
\eqref{w-xland2} as $\p_t (w \circ \phi) + \tfrac{8}{3} (aw) \circ \phi = -  \tfrac{1}{3} ( w \p_\theta w ) \circ \phi$ which yields
\begin{align*} 
  \p_\theta w \circ \phi \, \phi_x 
  = -3 \phi_x \,  (w \circ \phi) ^{-1} \p_t (w \circ \phi) -8 a \circ \phi \, \phi_x \,.
\end{align*}  
Differentiating this equation, we have that
\begin{align*} 
\p_x\bigl(  \p_\theta w \circ \phi \, \phi_x\bigr) 
& =-3 \phi_x \,  (w \circ \phi) ^{-1} \p_t (\p_\theta w \circ \phi \, \phi_x)   \\
&\qquad  - 3   \phi_x^2 \tfrac{ \p_\theta w \circ \phi}{w \circ \phi}    \left( \tfrac{8}{3} a\circ \phi +  \tfrac{1}{3}  \p_\theta w  \circ \phi \right) +   \phi_{xx}    \p_\theta w  \circ \phi   -8 \p_\theta a \circ \phi \, \phi_x^2 
\,.
\end{align*} 
We can then write the term $ \mathcal{S} _4$ in \eqref{S4} as $\mathcal{S} _4= \mathcal{S}_{4a}+ \mathcal{S}_{4b}$,
where 
\begin{align*} 
\mathcal{S}_{4a}(x,t) & = -3 \int_{-\eps}^t  \p_x(  a \circ \phi ) \phi_x \,  (w \circ \phi) ^{-2} \p_t (\p_\theta w \circ \phi \, \phi_x) dt'  \,,  \\
\mathcal{S}_{4b}(x,t)  & =  - \int_{-\eps}^t  \p_x(  a \circ \phi )  \left( 3   \phi_x^2 \tfrac{ \p_\theta w \circ \phi}{(w \circ \phi)^2}    \left( \tfrac{8}{3} a\circ \phi -  \tfrac{1}{3}  \p_\theta w  \circ \phi \right) -   \phi_{xx}    \tfrac{ \p_\theta w \circ \phi}{w \circ \phi}  + 8 \tfrac{\p_\theta a \circ \phi}{w \circ \phi} \phi_x^2 \right)  dt' \,. 
\end{align*} 
The term $\mathcal{S}_{4a}(x,t) $ requires a careful analysis; meanwhile, 
the bounds
\eqref{w0-lowerupper}, \eqref{w-lowerupper}, \eqref{a-boot}, \eqref{T*-bound},  \eqref{phix-crude}, \eqref{integral-wx-phi-bound} together with 
\eqref{dxx-phi-bound0} show that
\begin{align*} 
\sabs{\mathcal{S}_{4b}(x,t)} \les \tfrac{1}{\eps} + \int_{-\eps}^t \sabs{\p_\theta w \circ \phi}^2 dt' \,.
\end{align*} 
To estimate $\mathcal{S}_{4a}(x,t)$ we integrate by parts and obtain that
\begin{align*}
\mathcal{S}_{4a}(x,t) 
&=3 a_0'  w_0^{-2}  w_0'  -3 \p_x(  a \circ \phi ) \phi_x^2 \,  (w \circ \phi) ^{-2} \p_\theta w \circ \phi   \notag\\
&\quad + 4 \int_{-\eps}^t \p_x(  a \circ \phi )    (w \circ \phi) ^{-2}  (\p_\theta w \circ \phi)^2 \, \phi_x^2  dt'  
\notag\\
&\quad + 3 \int_{-\eps}^t    \p_x(  - \tfrac{4}{3} a^2 \circ \phi + \tfrac{1}{6} w^2 \circ\phi ) \phi_x^2 \,  (w \circ \phi) ^{-2}  \p_\theta w \circ \phi    dt' 
\notag\\
&\quad + 16 \int_{-\eps}^t  \p_x(  a \circ \phi ) \phi_x^2 \,  (w \circ \phi) ^{-2}   a \circ \phi  \p_\theta w \circ \phi   dt' \,.
\end{align*}
From the above identity and the bounds \eqref{w0-lowerupper}, \eqref{w-lowerupper}, \eqref{a-boot},  \eqref{phix-crude},  \eqref{integral-wx-phi-bound}, we obtain that
\begin{align*} 
\sabs{\mathcal{S}_{4a}(x,t)}  
&\les  \tfrac{1}{\eps} + \sabs{\p_\theta w(\phi(x,t),t)} +  \int_{-\eps}^t \sabs{\p_\theta w(\phi(x,t'),t')}^2 dt' \,.
\end{align*} 
Using the above bound in  \eqref{S3-prelim} shows that
\begin{align} 
\sabs{\mathcal{S}_{3}(x,t)} &   \les \tfrac{1}{\eps}  + \eps \sabs{w_0''(x)} + \sabs{\p_\theta w(\phi(x,t),t)} + \int_{-\eps}^t \sabs{\p_\theta w(\phi(x,t'),t')}^2 dt' \,.
\label{eq:S3:vomey}
\end{align} 
Having estimated $ \mathcal{S} _3$ in \eqref{S123}, it remains to bound $ \mathcal{S} _1$ and $ \mathcal{S} _2$.   

For $\mathcal{S} _1$, we return to the identity \eqref{d2a-dw} and write 
\begin{align} 
\p_\theta^2 a \circ \phi   = \bigl( \p_\theta w \circ \phi  \phi_x\bigr)  \  \phi_x ^{-1}  (1- \tfrac{1}{8} w \varpi )\circ \phi - \tfrac{w^2\circ \phi}{16}  \p_\theta \varpi \circ \phi \,,  \notag
\end{align} 
so that after differentiation in $x$ 
\begin{align} 
\p_\theta^3 a \circ \phi \, \phi_x   &=  \p_x \bigl( \p_\theta w \circ \phi  \phi_x\bigr)  \  \phi_x ^{-1}  (1- \tfrac{1}{8} w \varpi )\circ \phi   
-   \bigl( \p_\theta w \circ \phi  \phi_x\bigr)  \  \phi_x ^{-2}  \phi_{xx} (1- \tfrac{1}{8} w \varpi )\circ \phi  \notag  \\
& \qquad 
- \tfrac{1}{8}    \bigl( \p_\theta w \circ \phi  \phi_x\bigr)  \    \p_\theta( w \varpi ) \circ \phi  
- \tfrac{w^2\circ \phi}{16}  \p_\theta^2 \varpi \circ \phi \, \phi_x  - \tfrac{1}{8}   w \circ \phi \ \p_\theta w \circ \phi \, \phi_x \p_\theta\varpi \circ \phi \,.
\label{eq:vomey}
\end{align} 
Due to \eqref{eq:vomey}, the integrand $\p_\theta^3 a \circ \phi \, \phi_x^3$ in $\mathcal{S} _1$ has the same structure to the integrand in $ \mathcal{S} _3$, with one additional
type of term in the form of $ - \tfrac{w^2\circ \phi}{16}  \p_\theta^2 \varpi \circ \phi \, \phi_x$, which requires us to use the already established bounds 
\eqref{dxx-varphi-phi} and \eqref{just-write-it}.   We therefore can show that $ \mathcal{S} _1$   is bounded as
\begin{align} 
\sabs{\mathcal{S}_{1}(x,t)} &   \les  \tfrac{1}{\eps} +  \eps   \sabs{w_0''(x)}  +  \sabs{\p_\theta w(\phi(x,t),t)} + \int_{-\eps}^t \sabs{\p_\theta w(\phi(x,t'),t')}^2 dt' \,.
\label{eq:S1:vomey}
\end{align} 
The integral $ \mathcal{S} _2$ in \eqref{S123} is relatively straightforward to bound.  We use the inequalities \eqref{dxx-phi-bound0} and
\eqref{goodmusickeepsmegoing} together with \eqref{abs-Wy-PhiA}, and find that
\begin{align} 
\sabs{\mathcal{S}_{2}(x,t)} &   \les  \tfrac{1}{\eps} +  \int_{-\eps}^t \sabs{\p_\theta w(\phi(x,t'),t')}^2 dt'    \,.
\label{eq:S2:vomey}
\end{align} 
Combining the  bounds \eqref{eq:S3:vomey}, \eqref{eq:S1:vomey}, and \eqref{eq:S2:vomey},   we have shown that the $ \mathcal{S} (x,t)$ integral in \eqref{dxxxspvort-phi} satisfies
\begin{align*} 
\sabs{\mathcal{S}(x,t)} &   \les \tfrac{1}{\eps} +   \eps   \sabs{\p_\theta^2 w_0(x)}+  \sabs{\p_\theta w(\phi(x,t),t)} + \int_{-\eps}^t \sabs{\p_\theta w(\phi(x,t'),t')}^2 dt'   \,.
\end{align*} 
It thus follows from  \eqref{dxx-phi-bound0}, \eqref{dxxx-phi-bound0}, \eqref{goodmusickeepsmegoing}, and  \eqref{dxxxspvort-phi} that
\begin{align} 
\sabs{\p_\theta^3 \varpi(\phi(x,t),t) } 
& \les 
 \tfrac{1}{\eps^3}  + 
\sabs{\p_\theta^3 \varpi_0(x)}+  
 \tfrac{1}{\eps}\sabs{\p_\theta^2w_0(x)}
+ \tfrac{1}{\eps^2}  \sabs{\p_\theta w(\phi(x,t),t)} + \tfrac{1}{\eps} \sabs{\p_\theta w(\phi(x,t),t)}^2 \notag \\
 &
+ \tfrac{1}{\eps} \sabs{\p_\theta^2 w(\phi(x,t),t)} +  \int_{-\eps}^t \sabs{\p_\theta w(\phi(x,t'),t')}^2 dt'  \,. \notag
\end{align} 
Therefore, we have that
\begin{align} 
\sabs{\p_\theta^3 \varpi(\eta(x,t),t) } 
&  \les
 \tfrac{1}{\eps^3}
 +\sabs{\p_\theta^3 \varpi_0(\phi ^{-1} (\eta(x,t),t))}+
\tfrac{1}{\eps} \sabs{\p_\theta^2 w_0(\phi ^{-1} (\eta(x,t),t))}
 + \tfrac{1}{\eps^2}  \sabs{\p_\theta w(\eta(x,t),t)} \notag \\
 &\ \ + \tfrac{1}{\eps}  \sabs{\p_\theta w(\eta(x,t),t)}^2
+ \tfrac{1}{\eps}  \sabs{\p_\theta^2 w(\eta(x,t),t)} +   \int_{-\eps}^t \sabs{\p_\theta w(\phi(\phi ^{-1} (\eta(x,t),t),t'),t')}^2 dt'   \,. \label{dxxx-varpi-bound}
\end{align} 
In order to bound the first term in the above inequality, we differentiate \eqref{d2vort0} to obtain
\begin{align*} 
\tfrac{w_0^2}{16}  \p_\theta^3 \varpi_0 & = - \tfrac{1}{8}w_0 \p_\theta w_0  \p_\theta^2 \varpi_0
-\p_\theta^4 a_0 + \p_\theta^3 w_0 ( \tfrac{1}{8} w_0 \varpi _0-1) +\tfrac{1}{8} \p_\theta^2 w_0 \p_\theta( w_0 \varpi _0) \\
& \qquad\qquad
- \p_\theta( \tfrac{1}{4} w_0 \p_\theta w_0 \p_\theta\varpi_0  + \tfrac{1}{8}  (\p_\theta w_0)^2\varpi_0 ) \,.
\end{align*} 
With \eqref{just-write-it},  we see that
\begin{align} 
\sabs{ \p_\theta^3 \varpi_0(\theta)} \les \tfrac{1}{\eps^3} +  \tfrac{1}{\eps}  \sabs{\p_\theta^2 w_0 (\theta)} + \sabs{\p_\theta^3 w_0 (\theta)} 
+ \sabs{\p_\theta^4 a_0 (\theta)} 
\les   \tfrac{1}{\eps^3} +  \tfrac{1}{\eps}  \sabs{\p_\theta^2 w_0 (\theta)} + \sabs{\p_\theta^3 w_0 (\theta)} \,, \label{d3varpi0-prelim}
\end{align} 
where we have used that $\sabs{\p_\theta^4 a_0 (x)} \les 1$ by \eqref{eq:A_bootstrap:IC}.   From  \eqref{d3W}, for all $x \in \mathbb{T}  $, $\sabs{\p_\theta^3 w_0(x) } \les \eps^{-4}$,
so that
\begin{align*} 
\sabs{\p_\theta^3 w_0 (\phi ^{-1} (\eta(x,t),t),t)} \les \eps^ {-4}    \,,
\end{align*} 
and hence by \eqref{d3varpi0-prelim}, 
\begin{align*} 
\sabs{\p_\theta^3 \varpi_0 (\phi ^{-1} (\eta(x,t),t),t)} \les  \eps^ {-4}   \,.
\end{align*}
With this bound and using \eqref{d2w0-eta}, estimate \eqref{dxxx-varpi-bound} becomes
\begin{align} 
\sabs{\p_\theta^3 \varpi(\eta(x,t),t) } 
&  \les \eps^{-4}    
+ \tfrac{1}{\eps^2}  \sabs{\p_\theta w(\eta(x,t),t)} + \tfrac{1}{\eps} \sabs{\p_\theta w(\eta(x,t),t)}^2  \notag \\
 &\qquad 
+ \tfrac{1}{\eps} \sabs{\p_\theta^2 w(\eta(x,t),t)} +   \int_{-\eps}^t \sabs{\p_\theta w(\phi(\phi ^{-1} (\eta(x,t),t),t'),t')}^2 dt'   \,. \label{dxxx-varpi-bound:vomey}
\end{align} 

Having established a bound for the third derivative of $\varpi$, we are now ready to estimate the fourth derivative of $a$. We differentiate the identity \eqref{d3a-d2w} and obtain
\begin{align} 
\p_\theta^4 a & = \p_\theta^3 w ( \tfrac{1}{8} w \varpi -1)+  \tfrac{1}{8} \p_\theta^2w \p_\theta( w^2 \varpi)- \tfrac{w^2}{16}  \p_\theta^3 \varpi 
-\tfrac{1}{8}   w \p_\theta w \p_\theta^2 \varpi 
 - \tfrac{1}{8}   \p_\theta\bigl(2 w\p_\theta w \p_\theta\varpi + (\p_\theta w)^2\varpi \bigr) \,,
\label{d4a-d3w}
\end{align}
so that
\begin{align*} 
\sabs{\p_\theta^4 a(\theta,t)} \les \sabs{\p_\theta^3 \varpi(\theta,t)}
+  \sabs{\p_\theta^3 w(\theta ,t)}+ \bigl( \tfrac{1}{\eps} + \sabs{ \p_\theta w(\theta ,t)} \bigr)\bigl( \tfrac{1}{\eps} \sabs{ \p_\theta w(\theta ,t)}  + \sabs{ \p_\theta^2 w(\theta,t)}+\sabs{ \p_\theta^2\varpi (\theta,t)}\bigr) \,,
\end{align*} 
and with \eqref{dxxx-varpi-bound:vomey}, we have that
\begin{align*} 
\sabs{\p_\theta^4 a(\eta(x,t),t)}& \les   \eps^ {- 4} + \tfrac{1}{\eps^2}  \sabs{\p_\theta w(\eta(x,t),t)} +  \tfrac{1}{\eps}  \sabs{\p_\theta w(\eta(x,t),t)}^2
+  \sabs{\p_\theta^3 w(\eta(x,t),t)} \\
&\qquad
+ \bigl( \tfrac{1}{\eps} + \sabs{ \p_\theta w(\eta(x,t),t)} \bigr)\bigl( \sabs{ \p_\theta^2 w(\eta(x,t),t)}+\sabs{\p_\theta^2 \varpi  (\eta(x,t),t)}\bigr)  \notag \\
&\qquad
 + \int_{-\eps}^t \sabs{\p_\theta w(\phi(\phi ^{-1} (\eta(x,t),t),t'),t')}^2 dt' \,.
\end{align*} 
We observe that by \eqref{dxx-varphi-phi}, \eqref{just-write-it}, and \eqref{d2w0-eta}, 
\begin{align*} 
\sabs{ \p_\theta^2 \varpi  (\eta(x,t),t)} & \les \eps^ {-{\frac{5}{2}} }  + \eps^{-1} \sabs{ \p_\theta w(\eta(x,t),t)} \,,
\end{align*} 
and thus
\begin{align} 
\sabs{\p_\theta^4 a(\eta(x,t),t)}& \les   \eps^ {-4}   +  \eps^ {-{\frac{5}{2}} }  \sabs{\p_\theta w(\eta(x,t),t)} +  \eps^{-1}  \sabs{\p_\theta w(\eta(x,t),t)}^2
+  \eps^{-1}  \sabs{ \p_\theta^2 w(\eta(x,t),t)} \notag \\
&\qquad
 +  \sabs{ \p_\theta w(\eta(x,t),t)}  \sabs{ \p_\theta^2 w(\eta(x,t),t)}  +  \sabs{\p_\theta^3 w(\eta(x,t),t)} \notag \\
 &\qquad
 +   \int_{-\eps}^t \sabs{\p_\theta w(\phi(\phi ^{-1} (\eta(x,t),t),t'),t')}^2 dt'  \,.  \label{d4xa-great}
\end{align}

\subsection{Bounds on derivatives of $3$-characteristics}
\subsubsection{Identities for $\p_\theta^\gamma w \circ  \eta$}
With the integrating factor  $I_t(x)$ defined in \eqref{Is}, the equation  \eqref{eq0} is written as $ w \circ \eta = I_t w_0$, and differentiation yields
\begin{subequations} 
\label{w-derivative-idents}
\begin{align} 
\p_\theta w \circ \eta \  \eta_x & =  I_t w_0' + I_t' w_0 \,,  \\
\p_\theta^2w\circ \eta \ \eta_x^2  &= I_t w_0'' + 2I_t' w_0'  +  I_t'' w_0 -   \p_\theta w  \circ \eta \eta_{xx}\,, \\
\p_\theta^3w\circ \eta \ \eta_x^3  &=  I_t w_0''' + 3 I_t' w_0'' + 3 I_t '' w_0' +I_t ''' w_0 -3 \p_\theta^2 w \circ \eta \eta_x \eta_{xx} - \p_\theta w \circ \eta \eta_{xxx} \,, \\
\p_\theta^4w\circ \eta \ \eta_x^4  &=  I_t w_0'''' + 4 I_t ' w_0'''  + 6 I_t'' w_0'' + 4 I_t ''' w_0' +I_t '''' w_0 
  \notag \\
& \qquad
-6 \p_\theta^3 w \circ \eta \eta_x^2 \eta_{xx} -4 \p_\theta^2 w \circ \eta \eta_x \eta_{xxx} -3 \p_\theta^2 w \circ \eta  \eta_{xx}^2
 - \p_\theta w \circ \eta \eta_{xxxx}  \,.
\end{align} 
\end{subequations}

\subsubsection{Bounds for $\p_x \eta$}
We shall now obtain the precise rate at which $\p_x \eta(x_*,t) \to 0$  as $t \to T_*$, as well as a global bound for $\p_x\eta(x,t)$.

\begin{lemma}\label{lem:d1eta}  For  $-\eps\le t \le T_*$,  at the blowup label $x_*= \eps^ {\frac{3}{2}} y_*$, 
\begin{align} 
 \tfrac{1-\eps}{\eps} e^{-s} \le 
\p_x \eta(x_*,t)  \le  \tfrac{1+ \eps}{\eps} e^{-s}\,,
 \label{dx-eta-bound}
 \end{align} 
and for all labels $x$, we have that
 \begin{align} 
\sup_{t\in[-\log\eps,T_*) } \p_x \eta(x,t) & \le 
\begin{cases}
  11 \eps  & \ \ \abs{x-x_*} \le \eps^{2}  \\
   3& \ \ \abs{x-x_*} \ge \eps^{2}
 \end{cases} 
   \,.
 \label{dx-eta-bound0} 
 \end{align} 
 and
\begin{align} 
\p_x \eta(x,t) \ge \tfrac{1}{4} \eps   \ \ \operatorname{for}  \ \   \abs{x-x_*} \ge \eps^{2}  \,. \label{etax-lower-bad}
\end{align} 
 \end{lemma} 
\begin{proof}[Proof of Lemma \ref{lem:d1eta}]
{\it Step 1. Bounds at the blowup label $y_*$.}
From \eqref{label-relation} and \eqref{best-ever}, we have that
\begin{align} 
\p_x \eta(x,t) = \eps^ {-\frac{3}{2}}  e^{-\frac{3}{2}s} \p_y\Phi_W(y,s) \,, \qquad  y=\eps^ {-\frac{3}{2}}x\,.
\label{blacksheep1}
\end{align}
We will use the identity
\begin{align} 
\p_y \Phi_W(y,s) = e^{\frac{3}{2}s} \eps^{\frac{3}{2}} e^{\int_{-\log\eps}^s \beta_\tau \p_y W (\Phi_W(y,r),r)dr}  \,,
 \label{dyPhi1} 
 \end{align} 
We consider the blowup trajectory $\Phi_W(y_*,s)$.  For this, 
we decompose $ \beta_\tau \p_y W$ as
\begin{align} 
\beta_\tau \p_y W = \p_y \bar W - (1- \beta_\tau )\p_y \bar W + \beta_\tau \p_y \tilde W \,. \label{expand1}
\end{align} 
By \eqref{y*-bound},  $\abs{y_*} \le 20\kappa_0 \eps^ {\frac{5}{2}} $ and by \eqref{Phi-rate}, $\abs{\Phi_W(y_*,s)} \le 20\kappa_0  e^{- {\frac{5}{2}} s} $
and as such, this unique trajectory stays in the Taylor region $\abs{y}\le \ell$ for $\eps$ sufficiently small.
Using the Taylor remainder theorem, we have that $\p_y \bar W(y) = -1 + b_2 y^2$, where $b_2 = \tfrac{1}{2}  \p_y^3\bar W(\bar y)$ for
some $\bar y$ between $0$ and $y$, so that $\sabs{b_2 - 3} \le \eps^2$.   Substitution of this expansion into \eqref{expand1} gives
\begin{align} 
\beta_\tau \p_y W =  -1 + b_2 y^2 - (1- \beta_\tau )\p_y \bar W + \beta_\tau \p_y \tilde W \,. \label{expand2}
\end{align} 
Hence,
\begin{align} 
&e^{\int_{-\log\eps}^s \beta_\tau \p_y W (\Phi_W(y_*,r),r)dr}  \notag \\
& \qquad =  \tfrac{1}{\eps} e^{-s} e^{b_2 \int_{-\log\eps}^s  \Phi_W(y_*,r)^2 dr}
e^{\int_{-\log\eps}^s (\beta_\tau-1) \p_y \bar W (\Phi_W(y_*,r),r)dr }e^{\int_{-\log\eps}^s \beta_\tau \p_y \tilde W (\Phi_W(y_*,r),r)dr } \,.
\label{raptor}
\end{align} 
From \eqref{tildeWbootstrap0}, \eqref{mod2-bound}, the fact that  $\abs{\p_y \bar W} \le 1$, and  \eqref{Phi-rate}   we have that for
$\eps$ small enough,
\begin{align*} 
1- \eps \le 
e^{b_2 \int_{-\log\eps}^s  \Phi_W(y_*,r)^2 dr}
e^{\int_{-\log\eps}^s (\beta_\tau -1 ) \p_y \bar W (\Phi_W(y_*,r),r)dr }e^{\int_{-\log\eps}^s \beta_\tau \p_y \tilde W (\Phi_W(y_*,r),r)dr }  \le 1+\eps   \,,
\end{align*} 
and therefore 
\begin{align} 
 \tfrac{1-\eps}{\eps} e^{-s} \le 
e^{\int_{-\log\eps}^s \beta_\tau \p_y W (\Phi_W(y_*,r),r)dr}  \le  \tfrac{1+\eps}{\eps} e^{-s} \,.
\label{rasputen3}
\end{align} 
The bound \eqref{rasputen3} and the identity \eqref{dyPhi1}  then shows that for
 $\eps$ sufficiently small, 
\begin{align} 
(1-\eps) \eps^ {\frac{1}{2}}  e^{\frac{s}{2}}  \le 
\p_y \Phi_W(y_*,s)  \le (1+\eps)   \eps^ {\frac{1}{2}}  e^{\frac{s}{2}}
\label{d1Phi-bound}
\end{align}  
It follows from \eqref{blacksheep1} that \eqref{dx-eta-bound} holds.

\vspace{.05in}
\noindent
{\it Step 2. A bound for $\p_x \eta$  with $\abs{x-x_*}\le \eps^2$.}
The  identity \eqref{eq2} together with \eqref{Is1} show that
\begin{align} 
\eta_x & =   1 + \int_{-\eps}^t I_\tau d\tau  w_0' -\tfrac{8}{3} w_0 \int_{-\eps}^t I_\tau  \int_{-\eps}^\tau a' \circ \eta \  \eta_x dr d\tau     \,. \label{nice1}
 \end{align} 
From \eqref{a-boot},
 \begin{align} 
 \abs{a(\theta,t)} \le 2 \kappa_0^2 \eps \ \ \text{ and } \ \  \abs{\p_\theta a(\theta,t)} \le 2 \kappa_0 \,. \label{nice2}
 \end{align} 
Therefore,  for $\eps$ taken sufficiently small, we have that
\begin{align} 
 1-  \eps\le I_\tau(x) \le 1+  \eps \,.  \label{nice3} 
\end{align} 
By \eqref{tildeWbootstrap0},  for $\eps $ taken sufficiently small, 
\begin{subequations} 
 \label{nice4}
\begin{alignat}{2}
- \tfrac{1}{\eps}  \le w_0'(x)  & \le  - \tfrac{1-4\eps}{\eps} \ \  &&\text{ for } \ \ \abs{x-x_*} \le \eps^{2} \,,  \\
\abs{w_0'(x)}  & \le   \tfrac{1}{\eps} \ \  &&\text{ for } \ \ \abs{x-x_*} \ge \eps^{2} \,, 
\end{alignat} 
\end{subequations} 
From  \eqref{T*-bound}, \eqref{nice1}--\eqref{nice3}, we have that for $\eps$ taken sufficiently small, 
\begin{align*}
\sup_{t\in[-\eps,T_*)}\eta_x(x,t)& \le 
\begin{cases}
 1  - \tfrac{1-4\eps}{\eps} \cdot(1-\eps) (\eps+ 6 \eps^3 )  + 7 \eps^2 \kappa_0^2 \sup_{t\in[-\eps,T_*)}\eta_x(x,t)  & \ \ \abs{x-x_*} \ge \eps^{2} \\
 {\tfrac{5}{2}}  + 6 \eps^2 \kappa_0^2 \sup_{t\in[-\eps,T_*)}\eta_x(x,t)  & \ \ \abs{x-x_*} \le \eps^{2}
\end{cases}
\end{align*} 
and hence 
\begin{align} 
 \sup_{t\in[-\eps,T_*)}\eta_x(x,t) \le
 \begin{cases}
  11 \eps  & \ \ \abs{x-x_*} \le \eps^{2}  \\
   3& \ \ \abs{x-x_*} \ge \eps^{2}
   \label{nice5}
 \end{cases} \,,
\end{align} 
which proves \eqref{dx-eta-bound0}.

Notice also from \eqref{nice1} that with the bound \eqref{eq:W:gamma=2:p1}, for all $\abs{x-x_*} \ge \eps^ {\frac{1}{2}} $ and for $\eps$ taken small enough,
$\sabs{ w_0'(x)} \le (1- \tfrac{\eps}{2} ) \eps^{-1} $, and hence  for all $t \in [-\log\eps, T_*)$, 
we have the lower bound
\begin{align*} 
\p_x \eta(x,t) \ge \tfrac{\eps}{4}    \,,
\end{align*} 
which gives the bound \eqref{etax-lower-bad}.
\end{proof}

\subsubsection{Bounds for $\p_x^2 \eta$}
We establish the rate at which $\p_x^2 \eta(x_*,t) \to 0$  as $t \to T_*$, and obtain bounds for $\p_x^2\eta(x,t)$   for all labels $x$.
\begin{lemma} \label{lem:d2eta} 
 For all $-\eps\le t \le T_*$,  we have the decay estimate
\begin{align} 
\abs{ \p_x^2 \eta(x_*, t)} \le 62 \kappa_0 e^{-s} \,
 \label{dx2-eta-bound}
 \end{align} 
and  for any label  $x$, we have the bound
\begin{align} 
\sabs{\p_x^2 \eta(x,t)} \le 
 \begin{cases}
8\eps^{-1}  & \ \ \abs{x-x_*} \le \eps^{2}  \\
 8\eps^{- {\frac{3}{2}} } & \ \ \abs{x-x_*} \ge \eps^{2}
 \end{cases}
   \,. \label{etaxx-good-bound}
\end{align} 
\end{lemma} 
\begin{proof}[Proof of Lemma \ref{lem:d2eta}]
{\it Step 1. A bound for $\p_x^2 \eta$ along the blowup label $x_*$.}
Since $\eta_x = e^{\int_{-\eps}^t  \p_\theta w \circ \eta dr}$, we have that
\begin{align} 
\eta_{xx}(x,t)  & = \eta_x(x,t) \int_{-\eps}^t \p_\theta^2w (\eta(x,t'), t') \eta_x(x,t') dt' \notag \\
&  = \eta_x(x,t) \int_{-\log\eps}^s e^{\frac{3}{2}s'} \beta_\tau W_{yy}(\Phi_W(y,s'), s') \eta_x(x,t')  ds' \label{etaxx-identity} \,,
\end{align} 
where we have used the change of variables formula together with the  identity \eqref{dsdt} which shows that $dt' = \beta_\tau e^{-s'}ds'$.
 By  Lemma \ref{lem:Phi-rate},   $\sabs{\Phi_W(y_*,s)}\le 20 \kappa_0e^{-{\frac{5}{2}s} }$ and  $\abs{y_*}\le 20\kappa_0 \eps^{\frac{5}{2}} $, so that together with \eqref{tildeWbootstrap0},  
we have that for $\eps$ taken small enough and for all $ -\log\eps \le s'\le s $, 
\begin{align} 
\abs{ \beta_\tau  W_{yy} (\Phi_W(y_*,s'),s')}  \le 122 \kappa_0 e^{-\frac{5s}{2}}  \,.  \label{whynot9monthsago1}
\end{align} 
Hence, with \eqref{dx-eta-bound} and the identity \label{etaxx-identity} evaluated at the label $x_*$, we have that
\begin{align} 
\abs{\eta_{xx}(x_*,t)} \le 62 \kappa_0 e^{-s} \,,
\notag
\end{align} 
which proves \eqref{dx2-eta-bound}.

\vspace{.05in}
\noindent
{\it Step 2. A bound for $\p_x^2 \eta$ for all labels $x$.}
Using the identity in \eqref{eq3} and \eqref{Is2}, we have that
\begin{align} 
\p_x^2 \eta & =  \int_{-\eps}^t I_\tau d\tau w_0''-\tfrac{16}{3} w_0' \int_{-\eps}^t I_\tau  \int_{-\eps}^\tau a' \circ \eta \  \eta_x dr d\tau  \notag \\
& \qquad   
+ w_0 \int_{-\eps}^t I_\tau  \left( \Bigl( \tfrac{8}{3}  \int_{-\eps}^\tau a' \circ \eta \ \eta_x dr 
\Bigr)^2 - \tfrac{8}{3}  \int_{-\eps}^\tau  \bigl( a'' \circ \eta \ \eta_x^2 + a' \circ \eta \ \eta_{xx} \bigr) dr 
\right)d\tau     \,. \label{nice6}
 \end{align} 
 From \eqref{goodmusickeepsmegoing} and \eqref{w-derivative-idents},  
 \begin{align} 
\sabs{a''(\eta(x,t),t)} \le \tfrac{7}{2}   \sabs{\p_\theta w(\eta(x,t),t)}+  \tfrac{7}{\eps}  
\le \tfrac{7}{2}  ( I_t w_0' + I_t' w_0 )\eta_x^{-1} +  \tfrac{7}{\eps}   \,. \label{nice7}
\end{align}
It follows from \eqref{nice5} that
 \begin{align} 
\sabs{a''(\eta(x,t),t) \eta_x^2} 
\le \tfrac{7}{2}  \sabs{ I_t w_0' + I_t' w_0 }\eta_x+  \tfrac{7}{\eps} \eta_x^2
\le
 \begin{cases}
 42  & \ \ \abs{x-x_*} \le \eps^{2}  \\
   \tfrac{74}{\eps} & \ \ \abs{x-x_*} \ge \eps^{2}
 \end{cases} 
  \,. \label{nice7b}
\end{align}

By \eqref{d2W} and \eqref{tildeWbootstrap0}, for $\eps$ small enough,
\begin{align}
\sabs{w_0''(x)} \le
 \begin{cases}
 7 \eps^{-2} & \ \ \abs{x-x_*} \le \eps^{2}  \\
   7\eps^{-{\frac{5}{2}} }& \ \ \abs{x-x_*} \ge \eps^{2}
 \end{cases} 
\,. \label{nice8}
\end{align} 
It follows from \eqref{w0-lowerupper}, \eqref{nice2}--\eqref{nice5}, \eqref{nice6}--\eqref{nice8} that
\begin{align} 
(1- 7 \kappa_0^2 \eps^2) \sup_{t \in [-\eps,T_*)}\sabs{\p_x^2 \eta(x,t)} \le
 \begin{cases}
  \tfrac{15}{2}\eps^{-1}  + \OO( \eps^2)   & \ \ \abs{x-x_*} \le \eps^{2}  \\
   \tfrac{15}{2} \eps^{- {\frac{3}{2}} } + \OO(\eps) & \ \ \abs{x-x_*} \ge \eps^{2}
 \end{cases}
 \,, \notag
\end{align} 
and thus taking $\eps$ sufficiently small, 
\begin{align} 
\sup_{t \in [-\eps,T_*)}\sabs{\p_x^2 \eta(x,t)} \le 
 \begin{cases}
8\eps^{-1}  & \ \ \abs{x-x_*} \le \eps^{2}  \\
 8\eps^{- {\frac{3}{2}} } & \ \ \abs{x-x_*} \ge \eps^{2}
 \end{cases}
\,, \notag
\end{align} 
which proves \eqref{etaxx-good-bound}.
\end{proof} 
\begin{remark} 
We have shown in the proof of Lemmas  \ref{lem:d1eta} and  \ref{lem:d2eta} that  for $\eps $ taken sufficiently small, 
\begin{subequations} 
\label{Is012}
\begin{align} 
 \sabs{I_t(x)}   & \le 1+  \eps \,, \\
\sabs{ I'_t}  & \le 
 \begin{cases}
50\kappa_0 \eps^2 & \ \ \abs{x-x_*} \le \eps^{2}  \\
14 \kappa_0 \eps& \ \ \abs{x-x_*} \ge \eps^{2}
 \end{cases}
 \,, \\
 \sabs{ I''_t}  & \le 
 \begin{cases}
40 \kappa_0 & \ \ \abs{x-x_*} \le \eps^{2}  \\
40 \kappa_0 \eps^ {- {\frac{1}{2}} } & \ \ \abs{x-x_*} \ge \eps^{2}
 \end{cases} \,.
 \end{align} 
 \end{subequations} 
 \end{remark}

\subsubsection{Bounds for $\p_x^3 \eta$}
\begin{lemma}\label{lem:d3eta}  For all $-\eps\le t \le T_*$, we have that 
\begin{align} 
\sup_{t \in [-\eps,T_*)} \abs{\p_x^3 \eta(x,t)} \le
 \begin{cases}
 \tfrac{ (6+ \eps^ \frac{1}{6} )}{\eps^3}  & \ \ \abs{x-x_*} \le \eps^{2}  \\
 \tfrac{ C}{\eps^4} & \ \ \abs{x-x_*} \ge \eps^{2}
 \end{cases} 
 \,, \label{dx3-eta-bound-all-t} 
\end{align} 
and for $\abs{x-x_*} \le \eps^2$, 
\begin{align} 
\tfrac{(\eps+t)(6- \eps^{\frac{1}{6}} )}{\eps^4}  \le  \p_x^3 \eta(x, t) \le   \tfrac{6+\eps^ {\frac{1}{6}}}{ \eps^3}  \,.
 \label{dx3-eta-bound}
 \end{align} 
 \end{lemma} 
\begin{proof}[Proof of Lemma \ref{lem:d3eta}]
We first note that the bounds \eqref{dx-eta-bound0} and \eqref{etaxx-good-bound} show that for labels $x$ satisfying $\abs{x-x_*} \le \eps^{2}$, we have
that
\begin{align} 
\abs{\eta_x(x,t)} \le 
\begin{cases}
  11 \eps  & \ \ \abs{x-x_*} \le \eps^{2}  \\
   3& \ \ \abs{x-x_*} \ge \eps^{2}
 \end{cases} 
\ \ \text{ and } \ \  \abs{\eta_{xx}(x,t)} \le  
 \begin{cases}
8\eps^{-1} & \ \ \abs{x-x_*} \le \eps^{2}  \\
   8\eps^{- {\frac{3}{2}} } & \ \ \abs{x-x_*} \ge \eps^{2}
 \end{cases}
 \,.    \label{etax-etaxx-small}
\end{align} 
The identities \eqref{Is3} and \eqref{eq4} give
\begin{align} 
\p_x^3 \eta & = w_0'''  \int_{-\eps}^t I_\tau d\tau +3w''_0   \int_{-\eps}^t I'_\tau d\tau +3 w'_0\int_{-\eps}^t I''_\tau d\tau    \notag \\
& \qquad
+ w_0  \int_{-\epsilon }^t I_\tau \Bigl(-\tfrac{512}{27} \bigl(\int_{-\epsilon }^\tau a' \circ \eta \ \eta_x\, dr\bigr)^3
+\tfrac{64}{3}   \bigl(\int_{-\epsilon }^\tau a' \circ \eta \ \eta_x dr\bigr)
 \int_{-\epsilon }^\tau \bigl(a'' \circ \eta \ \eta_x^2 + a' \circ \eta \ \eta_{xx}\bigr) dr \notag \\
 & \qquad \qquad \qquad  -\tfrac{8}{3}  \int_{-\epsilon }^\tau \bigl( a''' \circ \eta \ \eta_x^3 + 3 a'' \circ \eta \ \eta_x \eta_{xx} + a' \circ \eta \ \eta_{xxx}\bigr) dr\Bigr) d\tau
  \,. \label{ironmaiden-runningfree}
 \end{align} 
From \eqref{w-derivative-idents}, we have that
\begin{align} 
\p_\theta^2 w( \eta(x,t),t)   =  \eta_x ^{-2} (I_t w_0'' + 2 I_t' w_0' + I''_t w_0 - \eta_x ^{-1} (w_0'I_t + I_t' w_0) \eta_{xx} ) \,. \label{funny-raccoon4} 
\end{align} 
From \eqref{dxxx-a-final0},  \eqref{w-derivative-idents},  and \eqref{funny-raccoon4}, 
\begin{align} 
\sabs{\p_\theta^3 a(\eta(x,t),t)}
&  \les   \eps^ {- {\frac{5}{2}} }   +  \sabs{\p_\theta^2 w (\eta(x,t),t)}   + \sabs{\p_\theta w (\eta(x,t),t)} ^2 + \tfrac{1}{\eps} \sabs{\p_\theta w (\eta(x,t),t)}  \notag \\
& \les  \eps^ {-{\frac{5}{2}} }   +   \eta_x ^{-2} (I_t w_0'' + 2 I_t' w_0' + I''_t w_0 - \eta_x ^{-1} (w_0'I_t + I_t' w_0) \eta_{xx} )  \notag \\
& \qquad + \abs{ \eta_x ^{-1} (I_t w_0' + I_t' w_0 )}^2 + \tfrac{1}{\eps} \abs{ \eta_x ^{-1} (I_t w_0' + I_t' w_0 )} \,.   \label{funny-raccoon5} 
\end{align} 
We will use \eqref{nice4},   \eqref{nice8}, and the fact that  by \eqref{eq:tilde:W:3:derivative:0} and \eqref{d3W}, 
\begin{subequations} 
 \label{nice9}
\begin{alignat}{2}
\tfrac{6- \eps^{\frac{1}{4}}}{\eps^4}  \le w_0'''(\theta) & \le  \tfrac{ 6 + \eps^{\frac{1}{4}} }{\eps^4}  \ \  &&\text{ for } \ \ \abs{x-x_*} \le \eps^{2} \,,  \\
\abs{w_0'''(x)}  & \les  \eps^{-4} \ \  &&\text{ for } \ \ \abs{x-x_*} \ge \eps^{2} \,.
\end{alignat} 
\end{subequations} 
Then, with \eqref{nice8},  \eqref{Is012} and \eqref{funny-raccoon5}, we have that
 \begin{align} 
\sabs{a'''(\eta(x,t),t) \eta_x^3} 
\les
 \begin{cases}
 \eps^{-2}  & \ \ \abs{x-x_*} \le \eps^{2}  \\
 \eps^{-{\frac{5}{2}} } & \ \ \abs{x-x_*} \ge \eps^{2}
 \end{cases} 
  \,. \label{mice7b}
\end{align}

With these bounds, and with \eqref{Is012}, \eqref{etax-etaxx-small}--\eqref{funny-raccoon5} applied to \eqref{ironmaiden-runningfree}, we have that
\begin{align} 
\sup_{t \in [-\eps,T_*)} \abs{\p_x^3 \eta(x,t)} \le
 \begin{cases}
 (\eps+ \eps^ 2 ) \tfrac{ (6+ \eps^ \frac{1}{4} )}{\eps^4} +  \tfrac{C}{\eps} 
+ 7 \eps^2 \kappa_0^2\sup_{t \in [-\eps,T_*)} \abs{\p_x^3 \eta(x,t)}  & \ \ \abs{x-x_*} \le \eps^{2}  \\
C \eps^{-4} 
+ 7 \eps^2 \kappa_0^2\sup_{t \in [-\eps,T_*)} \abs{\p_x^3 \eta(x,t)}   & \ \ \abs{x-x_*} \ge \eps^{2}
 \end{cases} 
 \,. \label{nice10}
\end{align} 
It immediately follows that for $\eps $ small enough,
\begin{align} 
\sup_{t \in [-\eps,T_*)} \abs{\eta_{xxx}(x,t)} \le
 \begin{cases}
 \tfrac{ (6+ \eps^ \frac{1}{6} )}{\eps^3}  & \ \ \abs{x-x_*} \le \eps^{2}  \\
 \tfrac{ C}{\eps^4} & \ \ \abs{x-x_*} \ge \eps^{2}
 \end{cases} 
  \,, \label{funny-raccoon7plus} 
\end{align} 
which establishes  \eqref{dx3-eta-bound-all-t}.  

For labels $\sabs{x-x_*} \le \eps^2$, we can easily see that $\p_x^3 \eta(x,t)$ is positive.
With \eqref{nice9}, we have that the first term on the right side of
\eqref{ironmaiden-runningfree} has the lower bound
\begin{align*} 
\tfrac{(\eps+t)(6- \eps^{\frac{1}{5}} )}{\eps^4} \le  \int_{-\eps}^t I_\tau d\tau w_0'''  \,.
 \end{align*} 
Thus, with  \eqref{nice4},   \eqref{nice8}, \eqref{Is012}, \eqref{etax-etaxx-small}--\eqref{funny-raccoon5}, in the same way that we obtained \eqref{nice10}, we find that
\begin{align*} 
\tfrac{(\eps+t)(6- \eps^{\frac{1}{6}} )}{\eps^4} \le  \p_x^3\eta(x,t) \le  \tfrac{ (6+ \eps^ \frac{1}{6} )}{\eps^3} \,,
 \end{align*} 
 which establishes \eqref{dx3-eta-bound}.
 \end{proof} 
 
\subsubsection{A sharp bound for $\p_x \eta$ and $\p_x^2\eta$}
\begin{proposition}\label{prop:etax-taylor}
For $\abs{x-x_*} \le \eps^2$, we have that
\begin{align} 
 \tfrac{1-\eps^ \frac{1}{2} }{\eps}(T_* -t) +  \tfrac{(\eps+t) (3- \eps^ \frac{1}{8} )}{\eps^4} (x-x_*)^2\le 
 \p_x\eta(x,t) \le   \tfrac{1+\eps^ \frac{1}{2} }{\eps}(T_* -t) +  \tfrac{ (3+ \eps^ \frac{1}{8} )}{\eps^3} (x-x_*)^2 \,, \label{etax-best}
\end{align} 
and
\begin{subequations} 
\label{etaxx-best}
\begin{align}
 -7\eps^{-2}(T_* -t)  + \tfrac{(\eps+t) (6- 2\eps^ \frac{1}{8} )}{\eps^4} (x-x_*) \le 
 \p_x^2\eta(x,t) \le   7\eps^{-2}(T_* -t) +  \tfrac{ 6+ 2\eps^ \frac{1}{8} }{\eps^3} (x-x_*)  \ \ \text{ for } \ \ x\ge x_* \,, \\
  -7\eps^{-2}(T_* -t)  +  \tfrac{ 6+ 2\eps^ \frac{1}{8} }{\eps^3} (x-x_*) \le 
 \p_x^2\eta(x,t) \le   7\eps^{-2}(T_* -t) +  \tfrac{(\eps+t) (6- 2\eps^ \frac{1}{8} )}{\eps^4}  (x-x_*)  \ \ \text{ for } \ \ x\le x_* \,.
\end{align} 
\end{subequations} 
\end{proposition} 
\begin{proof}[Proof of Proposition \ref{prop:etax-taylor}]
 By Lemma 2.1 in \cite{BuShVi2019a}, there exists a short time $\bar T \ge -\eps$, such that $(w,a)$ is a unique solution to \eqref{eq:w:z:k:a2} with initial
 data $(w_0,a_0)$ and 
\begin{align} 
(a,w) \in C^0([-\eps,\bar T];C^4( \mathbb{T}   )) \cap C^1([-\eps,\bar T];C^3( \mathbb{T}   ))  \,. \label{old-paper}
\end{align} 
By the local existence and uniqueness theorem for ODE,
$\eta \in C^1([-\eps,\bar T];C^3( \mathbb{T}   ))\cap  C^2([-\eps,\bar T];C^2( \mathbb{T}   )) $.   Given the uniform bounds \eqref{dx3-eta-bound-all-t} and \eqref{etax-etaxx-small},
the standard continuation argument shows that 
$$
\eta \in C^1([-\eps,T_*], C^3( \mathbb{T}   )) \cap C^2([-\eps,T_*], C^2( \mathbb{T}   )) \,.
$$
By the Taylor remainder theorem, there exist a point $x_1$ between $x$ and $x_*$ and a point $t_1$ between $t$ and $T_*$ such
that
\begin{align} 
\p_x \eta(x,t) &= \p_t\p_x \eta(x_*,T_*) (t-T_*) + \tfrac{1}{2} \p_x^3 \eta(x_1, t_1) (x-x_*)^2 
+ \tfrac{1}{2}  \p_t^2 \p_x \eta(x_1, t_1) (t-T_*) ^2  \notag \\
& \qquad 
+   \p_t \p_x^2 \eta(x_1,t_1) (t-T_*) (x-x_*)
 \,.\label{eta-taylor}
\end{align} 
Note that we have used \eqref{dx-eta-bound} and \eqref{dx2-eta-bound} which give
\begin{align} 
\p_x\eta(x_*,T_*) =0 \,, \ \ \ \ \p_x^2\eta(x_*,T_*) =0 \,. \label{mice1}
\end{align} 

From \eqref{eq2}, we have that 
\begin{align} 
\p_t \p_x \eta (x,t) = I_t(x) w_0' (x) + I_t' (x) w_0 (x) \,. \label{mice0}
\end{align} 
We use the bounds
\eqref{T*-bound}, \eqref{nice2}--\eqref{nice5} to find that for $\eps$ small enough,
\begin{align}
- \tfrac{1+  \eps^\frac{3}{4} }{\eps}  \le \p_t \p_x \eta(x_*,T_*) \le  - \tfrac{1-\eps^ \frac{3}{4} }{\eps} \,. \label{mice2}
\end{align} 
Differentiation of \eqref{mice0} with respect to $\p_x$ yields
\begin{align*} 
\p_t \p_x^2 \eta = I_t w_0''+ 2I_t' w_0'+ I_t'' w_0 \,,
\end{align*} 
while differentiation of \eqref{mice0} with respect to $\p_t$ gives
\begin{align*} 
\p_t^2 \p_x \eta =  I_t' \, w_0'+ I_t'' w_0  \,.
\end{align*} 
We again use the  bounds
\eqref{T*-bound}, \eqref{nice2}--\eqref{nice4},  \eqref{nice8}, and  \eqref{Is012}  to obtain that
\begin{align} 
\abs{\p_t ^2\p_x \eta(x_1,t_1)} & \le 50 \kappa_0^2 \,,  \label{mice3}  \\
\abs{\p_t \p_x^2 \eta(x_1,t_1)} & \le 8 \eps^{-2} \,. \label{mice4} 
\end{align} 
From \eqref{dx3-eta-bound}, we have that 
\begin{align} 
\tfrac{(\eps+t)(3- \eps^{\frac{1}{7}} )}{\eps^4}  
 \le  \tfrac{1}{2} \p_x^3\eta(x_1,t_2) \le  \tfrac{ (3+ \eps^ \frac{1}{7} )}{\eps^3} \,. \label{mice5}
 \end{align} 

Since $t \le T_*$,  $\sabs{x-x_*} \le \eps^2$,  and $(T_*-t)^2 \le 2 \eps^2$, 
the bounds \eqref{mice2}--\eqref{mice5} used in the identity \eqref{eta-taylor} show that for $\eps$ taken sufficiently small, 
\begin{align*} 
 \tfrac{1-\eps^ \frac{1}{2} }{\eps}(T_* -t) +  \tfrac{(\eps+t) (3- \eps^ \frac{1}{8} )}{\eps^4} (x-x_*)^2\le 
 \p_x\eta(x,t) \le   \tfrac{1+\eps^ \frac{1}{2} }{\eps}(T_* -t) +  \tfrac{ 3+ \eps^ \frac{1}{8} }{\eps^3} (x-x_*)^2\,,
\end{align*} 
which establishes \eqref{etax-best}.

We can again apply the Taylor remainder theorem to find that for  a point $\mathring x_1$ between $x$ and $x_*$ and a
point $\mathring t_1$ between $t$ and $T_*$, 
\begin{align*} 
\p_x^2\eta(x,t) = \p_x^3\eta(\mathring x_1, \mathring t_1) (x-x_*)  +  \p_t\p_x^2\eta(\mathring x_1, \mathring t_1) (t-T_*)  \,.
\end{align*} 
It then follows from \eqref{mice4} and \eqref{mice5} that 
\eqref{etaxx-best} holds.
\end{proof} 

 \subsubsection{Bounds for $\p_\theta w$}
\begin{lemma}[Bound for $\p_\theta w$]\label{lem:wthetabound}
 For $t \in [-\log\eps, T_*)$, 
\begin{align} 
\abs{\p_\theta w ( \eta(x,t),t)  } \le
 \begin{cases}
\tfrac{2}{ (T_* -t) +   3\eps^{-3}(\eps+t)(x-x_*)^2 }   & \ \ \abs{x-x_*} \le \eps^{2}  \\
5 \eps^{-2} & \ \ \abs{x-x_*} \ge \eps^{2}
 \end{cases} 
  \,.
\label{wtheta-eta-bound}
\end{align} 
\end{lemma} 
\begin{proof}[Proof of Lemma \ref{lem:wthetabound}]
From \eqref{w-derivative-idents}, we have that
\begin{align} 
\p_\theta w ( \eta(x,t),t)  & =  ( I_t(x) w_0' (x) + I_t' (x)w_0(x)) \eta_x ^{-1} (x,t) \,.   \label{dthetaw-eta}
\end{align} 
Using the bounds \eqref{w0-lowerupper},  \eqref{etax-lower-bad},  \eqref{nice4},  \eqref{Is012}, and \eqref{etax-best}, obtain
the bound \eqref{wtheta-eta-bound}.
\end{proof}

 \subsubsection{Bounds for $\p_x^4 \eta$}
In order to obtain a bound for the fourth derivative of $\eta$, we shall appeal to the identity \eqref{eq5}. Before estimating the terms on the right side of \eqref{eq5}, we first record a useful estimate:
 \begin{lemma}\label{lem:hardintegral}
For $\abs{x-x_*} \leq \eps^2$ it holds that
 \begin{align} 
 \label{eq:harakiri:0}
 \eta_x^4(x,t) \int_{-\eps}^{t} \sabs{\p_\theta w ( \phi( \phi ^{-1}(\eta(x,t),t) ,t'),t')}^2 dt'  \les \eps^{-1} \eta_x^2 (x,t)  
 \,.
 \end{align} 
 \end{lemma} 
 \begin{proof}[Proof of Lemma \ref{lem:hardintegral}]
 Fix a label $x$ which is within $\eps^2$ of $x_*$, and a time $t \in [-\eps,T_*)$, throughout the proof.
In order to estimate  the integral in \eqref{eq:harakiri:0} we use the bound on $\p_\theta w$ obtained in \eqref{wtheta-eta-bound}. Note however that this estimate is obtained when we compose with the flow $\eta$; as such we first define the label
\begin{align} 
\chi(x,t)  =\phi ^{-1} (\eta(x,t),t) \,, \label{chi-label}
\end{align} 
and then for each $t'\in [-\eps, t]$, we also define the label
\begin{align} 
q(x,t') = \eta ^{-1} (  \phi(\chi(x,t),t'), t')  \,.   \label{q-eta}
\end{align} 
\begin{figure}[htb!]
\centering
\begin{tikzpicture}[scale=1.5]
    \draw [<->,thick] (0,2.4) node (yaxis) [above] {}
        |- (3.5,0) node (xaxis) [right] {$\theta$};
    \draw[black,thick] (0,0)  -- (-1,0) ;

\node at (3,2)[circle,fill,inner sep=1pt]{};
\node at (2.68,1.88)[circle,fill,inner sep=1pt]{};
\node at (2.12,1.)[circle,fill,inner sep=1pt]{};

\draw[black,dotted] (0,2)  -- (3,2) ;
\draw[black] (-0.15,2.1) node {{\footnotesize $T_*$}}; 

\draw[black,dotted] (3,0)  -- (3,2) ;
\draw[black] (3,-.15) node {{\footnotesize $\xi_*$}}; 

\draw[black] (4.3,2.2) node { {\footnotesize pre-shock location $(\xi_*,T_*)$}}; 
 
\draw[name path=A,red, thick,dashed] (-.7,0) -- (2.9,2) ;    
\draw[name path=A,red, thick] (-.1,0) -- (3,2) ;
\draw[name path=A,red, thick,dotted] (.7,0) -- (2.1,1) ;
\draw[black] (.6,-.15) node {{\footnotesize $q(x,t')$ }};


\draw[blue, thick,dashed] (1.5,0) -- (2.7,1.91) ;
\draw[black] (2.2,-.15) node {{\footnotesize $\chi(x_*,T_*)$}};

\draw[blue, thick] (1.8,0) -- (3,2) ;
\draw[black] (1.4,-.15) node {{\footnotesize $\chi(x,t)$ }};

\draw[black] (-0.7,-.13) node {$x$};
\draw[black] (-0.0,-.15) node {$x_*$};

\draw[black,dotted] (0,1.)  -- (2.1,1.);
\draw[black] (-0.1,1.) node { $t'$}; 
\draw[black,dotted] (0,1.88)  -- (2.7,1.88);
\draw[black] (-0.1,1.85) node { $t$}; 
\end{tikzpicture}

\vspace{-0.2cm}
\caption{\footnotesize  The identity \eqref{brilliant} is explained. The $3$-characteristics $\eta$ are shown in red and $2$-characteristics $\phi$ are shown in blue.  The worst case scenario is depicted: the label $x$
is to the left of the blowup label $x_*$.   For each such label $x$ and each $t \in [-\eps, T_*)$,  $\chi(x,t)$ denotes the label which satisfies $\phi(\chi(x,t),t)=\eta(x,t)$.   For each
$t' \in [-\eps, t]$, we define the label $q(x,t')$ such that $ \eta(q(x,t'),t') = \phi(\chi(x,t),t')$.   As $t' \to t$, $q(x,t') \to x$.  A particle moving
up the dashed blue curve is equivalent to that particle moving by the $3$-characteristic but emanating from  the moving label $q(x,t')$.}
\end{figure}
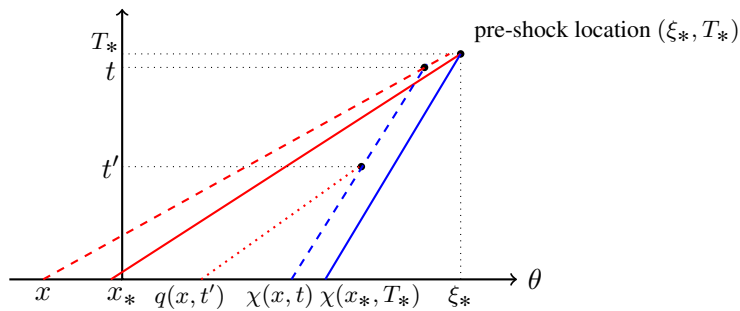

The definitions \eqref{chi-label} and \eqref{q-eta} show that
\begin{align} 
\p_\theta w (\phi(\phi ^{-1} (\eta(x,t),t),t'),t')
=
\p_\theta w (\phi(\chi(x,t),t'),t')
= \p_\theta w (\eta(q(x,t')    ,t'),t') \,. \label{brilliant}
\end{align}

Therefore, $q(x,-\eps) = \chi(x,t)$, $q(x,t') \to x$ from the right as $t' \to t$, while from \eqref{wtheta-eta-bound} we have that
\begin{align} 
\abs{\p_\theta w ( \eta(q(x,t'),t'),t')  } \le
 \begin{cases}
 \tfrac{2}{(T_* -t') +   3\eps^{-3}(\eps+t')(q(x,t')-x_*)^2 }  & \ \ \abs{q(x,t')-x_*} \le \eps^{2}  \\
5 \eps^{-2} & \ \ \abs{q(x,t')-x_*} \ge \eps^{2}
 \end{cases} 
  \,.
\label{wtheta-eta-bound:b}
\end{align} 
We will assume first that $x_* \in [x,\chi(x_*,t)]$. 
The proof is based on decomposing the interval $[-\eps,t)$ into three different  sets 
\begin{subequations}
\label{eq:calculate:intervals}
\begin{align}
I_{\rm start} &= \{ t' \in [-\eps,t) \colon  |q(x,t')-x_*|\geq \eps^2 \mbox{  or  } t' \leq - \tfrac 12 \eps \} \\
I_{\rm middle} &= \{ t' \in [-\tfrac \eps 2 ,t) \colon  |q(x,t')-x_*| <  \eps^2 \mbox{  and  } x_* - \tfrac 12 (x_* - x) < q(x,t') < x_* + \eps^2    \} \\
I_{\rm end} &= \{ t' \in [-\tfrac \eps 2  ,t) \colon   |q(x,t')-x_*| < \eps^2 \mbox{ and } x \leq q(x,t') \leq x_* - \tfrac 12 (x_* - x)  \}
\,.
\end{align}
\end{subequations}
From \eqref{wtheta-eta-bound:b} we immediately have that
\begin{align}
\int_{I_{\rm start}} \sabs{\p_\theta w (\eta(q(x,t')    ,t'),t')}^2 dt' 
&\leq
\int_{-\eps}^{-\frac 12 \eps} \frac{4}{(T_* - t')^2} dt' + 
\int_{- \eps}^{T_*} \frac{25}{\eps^4} dt'  
\notag\\
&\leq \frac{4}{T_* + \frac{\eps}{2}} + \frac{25 (T_* + \eps)}{\eps^4}
\leq 50 \eps^{-3} 
\label{eq:calculate:I}
\end{align}
since $T_* = \OO(\eps^3)$ and $\eps$ is sufficiently small.
 
For the remaining two time intervals, since $\frac 12 \eps \leq t' +\eps \leq 2 \eps$, and $|q(x,t') - x_*| < \eps^2$, we will use that 
\begin{align}
 \abs{\p_\theta w ( \eta(q(x,t'),t'),t')  } 
 \le \frac{1}{G(x,t')}\,, \quad \mbox{where} \quad 
 G(x,t'): = (T_* -t') + \tfrac 32 \eps^{-2} (q(x,t')-x_*)^2   \,.
 \label{eq:calculate:G}
\end{align}
The second important fact that we will use frequently is that \eqref{etax-best} implies
\begin{align}
\tfrac{1}{3\eps} G(x,t') \leq (\p_x \eta)(q(x,t'),t') \leq \tfrac{3}{\eps} G(x,t')\,.
 \label{eq:calculate:GG}
\end{align}
The third important ingredient is an estimate for the time derivative of the label $q(x,t')$. Using that $\eta^{-1}$ solves the transport equation $(\p_{t'} + w \p_\theta)\eta^{-1} = 0$, upon differentiating \eqref{q-eta} with respect to $t'$ we obtain
\begin{align}
\p_{t'} q(x,t') 
& = \p_{t'} \eta ^{-1} (  \phi(\chi(x,t),t'), t') + \p_{\theta} \eta ^{-1} (  \phi(\chi(x,t),t'), t')   \p_{t'}  \phi(\chi(x,t),t') \notag \\  
& = \p_{t'} \eta ^{-1} (  \eta(q(x,t'),t'), t') +  \tfrac{2}{3}   \p_{\theta} \eta ^{-1} (  \eta(q(x,t'),t'), t')   w  (\eta(q(x,t'),t')) \notag\\
&=  -  \tfrac{1}{3}   \p_{\theta} \eta ^{-1} (  \eta(q(x,t'),t'))   w  (\eta(q(x,t'),t'),t')  \notag\\
&=  -  \frac{ w  (\eta(q(x,t'),t'),t') }{3 \eta_x(q(x,t'),t')} \,.
\end{align}
From the above identity, using the bounds \eqref{w-lowerupper} and \eqref{eq:calculate:GG} we conclude that 
\begin{align}
\frac{\eps \kappa_0}{20 G(x,t')} \leq - \p_{t'} q(x,t') \leq \frac{2 \eps \kappa_0}{G(x,t')}
 \,.
 \label{eq:calculate:GGG}
\end{align}

With \eqref{eq:calculate:G}, \eqref{eq:calculate:GG}, and \eqref{eq:calculate:GGG} in hand, we return to the two remaining  cases described in \eqref{eq:calculate:intervals}.
First, we note that \eqref{eq:calculate:GGG} shows that the function $q(x,t')$ is strictly decreasing, as a function of $t'$, and thus when  $\eps$ is sufficiently small  there exists a unique time $t^\sharp \in [- \frac 12  \eps,t)$ such that 
\begin{align*}
q(x,t^\sharp) = x_* - \tfrac 12 (x_* -x ) \,. 
\end{align*}
As such, $I_{\rm end}= [t^\sharp,t]$, and $I_{\rm middle} \subset [-\frac 12 \eps, t^\sharp]$.
Since $q(x,t) = x$, the fundamental theorem of calculus, \eqref{eq:calculate:GGG}, and the definition of $I_{\rm end}$ show that 
\begin{align*}
\tfrac 12 (x_* - x) = q(x,t^\sharp) - q(x,t)  
&= \int_{t^{\sharp}}^t ( - \p_{t'} q(x,t')) dt' \notag\\
&\leq \int_{t^{\sharp}}^t \frac{2 \eps \kappa_0}{G(x,t')}dt'
\leq   \frac{2 \eps \kappa_0   (t-t^\sharp)}{(T_* - t) + \frac 38 \eps^{-2} (x_*-x)^2}
\leq   \frac{8 \eps \kappa_0   (t-t^\sharp)}{G(x,t)}
\,.
\end{align*}
The purpose of the above estimate is to provide the lower bound
\begin{align}
t - t^\sharp \geq \frac{1}{16 \eps \kappa_0} (x_* - x) G(x,t) \,.
\label{eq:calculate:II} 
\end{align}
With \eqref{eq:calculate:G}, \eqref{eq:calculate:GGG} and \eqref{eq:calculate:II}, since $I_{\rm middle} \subset [-\frac 12 \eps, t^\sharp]$ we may then estimate 
\begin{align}
\int_{I_{\rm middle}} \sabs{\p_\theta w (\eta(q(x,t')    ,t'),t')}^2 dt' 
&\leq
\int_{I_{\rm middle}} \frac{1}{G(x,t')^2} dt' 
\notag\\
&\leq  
\frac{1}{T_* - t^\sharp} 
\int_{I_{\rm middle}} \frac{- 20 \p_{t'} q(x,t')}{\eps \kappa_0  } dt'
\notag\\
&\leq \frac{20}{\eps \kappa_0} \cdot \frac{(x_* + \eps^2) - (x_* - \frac 12 (x_*-x)) }{(T_* - t) + \frac{1}{16 \eps \kappa_0} (x_* -x) G(x,t)}
\notag\\
&\leq \frac{30 \eps}{ \kappa_0} \cdot \frac{1}{(T_* - t) + \frac{1}{16 \eps \kappa_0} (x_* -x) G(x,t)}
\notag\\
&\leq
\frac{30 \eps}{ \kappa_0}
\begin{cases}
 \frac{1}{G(x,t)} \,, &\mbox{if } G(x,t) \geq 24 \kappa_0 \eps^{-1} (x_*-x) \\
  \frac{400 \kappa_0^2 }{ G(x,t)^2} \,, &\mbox{if } G(x,t) < 24 \kappa_0 \eps^{-1} (x_*-x)
\end{cases} \notag \\
&\leq \frac{12000 \eps \kappa_0}{G(x,t)^2}  \leq \frac{108000  \kappa_0}{\eps \eta_x(x,t)^2}
\,, 
\label{eq:calculate:III}
\end{align}
where in the second-to-last inequality we have used that $0 < G(x,t) \leq  \eps$, and in the last inequality we have appealed to \eqref{eq:calculate:GG}. 
Lastly, since $I_{\rm end}= [t^\sharp,t]$, a similar argument and the bound \eqref{eq:calculate:GG} shows that
\begin{align}
\int_{I_{\rm end}} \sabs{\p_\theta w (\eta(q(x,t')    ,t'),t')}^2 dt' 
&\leq
\int_{I_{\rm end}} \frac{1}{G(x,t')^2}  dt' \notag\\
&\leq 
\frac{1}{(T_* - t) + \frac 38 \eps^{-2} (x_* - x)^2} 
\int_{t^\sharp}^t \frac{- 20 \p_{t'} q(x,t')}{\eps \kappa_0  } dt' 
\notag\\
&\leq \frac{40 }{ \eps \kappa_0   G(x,t)} \left((x_* - \frac 12 (x_*-x)) - x\right)
\notag\\
&\leq \frac{20 (x_*- x)}{\eps \kappa_0   G(x,t)} 
\leq  \frac{60 (x_*- x)}{\eps^2 \kappa_0  \eta_x(x,t)} 
\leq  \frac{60}{\kappa_0  \eta_x(x,t)} \,.
\label{eq:calculate:IV}
\end{align}
Combining \eqref{eq:calculate:I}, \eqref{eq:calculate:III}, and \eqref{eq:calculate:IV}, 
we arrive at 
\begin{align*}
 \eta_x^4(x,t) \int_{-\eps}^{t} \sabs{\p_\theta w ( \phi( \phi ^{-1}(\eta(x,t),t) ,t'),t')}^2 dt'  \les \eps^{-3} \eta_x^4(x,t) + \eps^{-1} \eta_x^2 (x,t) + \eta_x^3(x,t)   \,, 
\end{align*}
and then by appealing to the first case in \eqref{dx-eta-bound0}, concludes the proof of the lemma in the case that $x_* > x$.

For the other case, $x_* < x$, we have that $(q(x,t') - x_*)^2 \geq (x-x_*)^2$, and then we simply have 
\begin{align}
\int_{[-\eps,t) \setminus I_{\rm start}} \sabs{\p_\theta w (\eta(q(x,t')    ,t'),t')}^2 dt' 
&\leq \int_{-\eps}^t \frac{1}{( (T_* - t') + \frac 32 \eps^{-2}(x_* - x)^2)^2} dt' \notag\\ 
&\leq  \frac{1}{ (T_* - t) + \frac 32 \eps^{-2}(x_* - x)^2 }  \notag\\
&\leq \frac{3}{\eps \p_x \eta(x,t)}
\end{align}
in light of the definition of $G$ and of \eqref{eq:calculate:GG}. The estimate \eqref{eq:harakiri:0} follows as before (it is in fact better in this case).
\end{proof} 
 
\begin{lemma}\label{lem:d4eta}  For labels $x$, we have that 
 \begin{align} 
\sup_{t\in[-\eps, T_*)} \abs{\p_x^4\eta(x,t)}  \le
 \begin{cases}
3 \eps^ {-\frac{31}{8}}   & \ \ \abs{x-x_*} \le \eps^{3}  \\
363\eps^ {-4}   & \ \ \abs{x-x_*} \le \eps^{2}  \\
C \eps^{-{\frac{9}{2}} } 
  & \ \ \abs{x-x_*} \ge \eps^{2}
 \end{cases} 
  \,,
 \label{dx4-eta-bound}
 \end{align} 
 where $C_\eps$ denotes a positive constant that depends on inverse powers of $\eps$.
 \end{lemma} 
\begin{proof}[Proof of Lemma \ref{lem:d4eta}] We shall first consider the case that the label $x$ satisfies $\abs{x-x_*} \le \eps^2$.
The identity \eqref{w-derivative-idents} shows that
\begin{align} 
\p_\theta^3 w( \eta(x,t),t)  & =   \eta_x ^{-3} (I_t w_0''' + 3 I_s' w_0'' + 3 I_t '' w_0' +I_t ''' w_0 ) \notag \\
& \qquad 
-3 \eta_x ^{-4} \eta_{xx} (I_t w_0'' + 2 I_t' w_0' + I''_t w_0 - \eta_x ^{-1} (w_0'I_t + I_t' w_0) \eta_{xx} )  \notag \\ 
&\qquad
- \eta_x ^{-4} \eta_{xxx} (I_t w_0' + I_t' w_0 ) \,.  \label{funny-raccoon8}
\end{align} 
We next use the inequality \eqref{d4xa-great} together with the identities \eqref{dthetaw-eta},  \eqref{funny-raccoon8}, and \eqref{funny-raccoon4}, 
\begin{align} 
\sabs{\p_\theta^4 a(\eta(x,t),t)}& \les   \eps^ {-4}  +  \eps^ {-2}  \sabs{\eta_x ^{-1} (I_t w_0' + I_t' w_0 )}
+\tfrac{1}{\eps}  \sabs{\eta_x ^{-1} (I_t w_0' + I_t' w_0 )}^2
 \notag \\
&\qquad
+ \tfrac{1}{\eps}  \sabs{  \eta_x ^{-2} (I_t w_0'' + 2 I_t' w_0' + I''_t w_0 - \eta_x ^{-1} (w_0'I_t + I_t' w_0) \eta_{xx} )} \notag \\
&\qquad
 + \sabs{\eta_x ^{-1} (I_t w_0' + I_t' w_0 )} \sabs{  \eta_x ^{-2} (I_t w_0'' + 2 I_t' w_0' + I''_t w_0 - \eta_x ^{-1} (w_0'I_t + I_t' w_0) \eta_{xx} )} \notag \\
 &\qquad
   +  \sabs{\ \eta_x ^{-3} (I_t w_0''' + 3 I_t' w_0'' + 3 I_t '' w_0' +I_t ''' w_0 ) } \notag \\
    &\qquad
   + 3 \sabs{ \eta_x ^{-4} \eta_{xx} (I_t w_0'' + 2 I_t' w_0' + I''_t w_0 - \eta_x ^{-1} (w_0'I_t + I_t' w_0) \eta_{xx} )  } \notag \\
       &\qquad
   +  \sabs{ \eta_x ^{-4} \eta_{xxx} (I_t w_0' + I_t' w_0 ) }  \notag \\
   &\qquad
   +    \int_{-\eps}^t \sabs{\p_\theta w (\phi(\phi ^{-1} (\eta(x,t),t),t'),t')}^2 dt'  \,. \label{funny-raccoon9}
\end{align}

By \eqref{eq:tilde:W:5:derivative} and \eqref{d3W}, for $\eps$ sufficiently small, we have that
\begin{subequations} 
 \label{rice1}
\begin{alignat}{2}
 \abs{w_0''''(\theta)} & \le  2 \eps^{-\frac{39}{8}}   \ \  &&\text{ for } \ \ \abs{x-x_*} \le \eps^{3} \,,  \\
  \abs{w_0''''(\theta)} & \le  361 \eps^{-5}   \ \  &&\text{ for } \ \ \abs{x-x_*} \le \eps^{2} \,,  \\
\abs{w_0''''(x)}  & \les  \eps^{-{\frac{11}{2}} } \ \  &&\text{ for } \ \ \abs{x-x_*} \ge \eps^{2} \,.
\end{alignat} 
\end{subequations} 
Using the identity \eqref{Is3} together with
 \eqref{nice7b},  \eqref{dx3-eta-bound-all-t},  \eqref{etax-etaxx-small},   \eqref{mice7b}
\begin{align} 
 \sabs{ I'''_t}  & \le 
 \begin{cases}
7\eps^{-2}  & \ \ \abs{x-x_*} \le \eps^{2}  \\
C \eps^{-3}& \ \ \abs{x-x_*} \ge \eps^{2}
 \end{cases} \,.
 \label{It'''}
\end{align} 

Then, with  \eqref{Is012} and  \eqref{nice4}, \eqref{nice8}, \eqref{nice9}, \eqref{eq:harakiri:0},  and \eqref{funny-raccoon9}, we have that for $\eps$ taken sufficiently small, 
 \begin{align} 
\sabs{\p_\theta^4 a(\eta(x,t),t) \eta_x^4} 
\les
 \begin{cases}
  \eps^{-4}  & \ \ \abs{x-x_*} \le \eps^{2}  \\
 \eps^{-7} & \ \ \abs{x-x_*} \ge \eps^{2}
 \end{cases} 
  \,. \label{d4a-etax4}
\end{align}
Using the identities \eqref{Is4} and \eqref{eq5}, we have that
\begin{align*} 
\p_x^4\eta & = w_0''''  \int_{-\eps}^t I_\tau d\tau +  4w_0'''  \int_{-\eps}^t I_\tau'd\tau +  6w_0'' \int_{-\eps}^t I_\tau''d\tau   +  4w_0'  \int_{-\eps}^t I_\tau''' d\tau 
\notag \\
&  
+ w_0 \int_{-\epsilon }^t I_\tau 
 \Biggl(\tfrac{4096}{81} \bigl(\int_{-\epsilon }^\tau a' \circ \eta \ \eta_x\, dr\bigr)^4
-\tfrac{1024}{9}   \bigl(\int_{-\epsilon }^\tau a' \circ \eta \ \eta_x dr\bigr)^2
 \int_{-\epsilon }^\tau \bigl(a'' \circ \eta \ \eta_x^2 + a' \circ \eta \ \eta_{xx}\bigr) dr  \notag \\
 & \qquad
  +\tfrac{64}{3}   \bigl(\int_{-\epsilon }^\tau  ( a'' \circ \eta \ \eta_x^2 +  a' \circ \eta \ \eta_{xx} )dr\bigr)^2 \notag \\
 & \qquad
  +\tfrac{256}{9}   \bigl(\int_{-\epsilon }^\tau   a' \circ \eta \ \eta_{x}  dr\bigr)
   \bigl(\int_{-\epsilon }^\tau  ( a''' \circ \eta \ \eta_x^3 + 3 a'' \circ \eta \ \eta_x \eta_{xx} + a' \circ \eta \ \eta_{xxx} )dr\bigr)
 \notag \\
 & \qquad 
 -\tfrac{8}{3}  \int_{-\epsilon }^\tau \bigl( a'''' \circ \eta \ \eta_x^4 
 + 6 a''' \circ \eta \ \eta_x^2 \eta_{xx} 
 + 3 a'' \circ \eta \ \eta_{xx}^2
 + 4 a'' \circ \eta \ \eta_x \eta_{xxx} 
 + a' \circ \eta \ \eta_{xxxx}\bigr) dr\Biggr) d\tau  \,.
\end{align*} 
Notice that from \eqref{etax-best} and  \eqref{etaxx-best}, for $\abs{x-x_*} \le \eps^2$,  we have that 
$$
\eta_x ^{-1} \eta_{xx}^2 \le 100 \eps^{-3} \,.
$$
Then, together with
 the bounds \eqref{nice7}, \eqref{nice7b}, 
  \eqref{Is012}, \eqref{etax-etaxx-small}--\eqref{mice7b}, \eqref{etax-best},  \eqref{etaxx-best}, 
 \eqref{rice1}--\eqref{d4a-etax4}, and with \eqref{Is012}, 
\eqref{etax-etaxx-small}--\eqref{funny-raccoon5}, 
we find that for $\eps$ sufficiently small, 
\begin{align} 
\sup_{t \in [-\eps,T_*)} \abs{\p_x^4 \eta(x,t)} \le
 \begin{cases}
{\tfrac{5}{2}} \eps^ {-\frac{31}{8}} + 7 \eps^2 \kappa_0^2\sup_{t \in [-\eps,T_*)} \abs{\p_x^4 \eta(x,t)}  & \ \ \abs{x-x_*} \le \eps^{3}  \\
362\eps^ {-4} + 7 \eps^2 \kappa_0^2\sup_{t \in [-\eps,T_*)} \abs{\p_x^4 \eta(x,t)}  & \ \ \abs{x-x_*} \le \eps^{2}  \\
C \eps^{-{\frac{9}{2}} } 
+ 7 \eps^2 \kappa_0^2\sup_{t \in [-\eps,T_*)} \abs{\p_x^4 \eta(x,t)}   & \ \ \abs{x-x_*} \ge \eps^{2}
 \end{cases} 
 \,, \label{mice10}
\end{align} 
and  hence 
\begin{align} 
\sup_{t \in [-\eps,T_*)} \abs{\p_x^4 \eta(x,t)} \le
 \begin{cases}
3 \eps^ {-\frac{31}{8}}   & \ \ \abs{x-x_*} \le \eps^{3}  \\
363\eps^ {-4}   & \ \ \abs{x-x_*} \le \eps^{2}  \\
C \eps^{-{\frac{9}{2}} } 
  & \ \ \abs{x-x_*} \ge \eps^{2}
 \end{cases} 
 \,, \label{mice11}
\end{align} 
which proves  \eqref{dx4-eta-bound}.
\end{proof}

\subsection{$C^4$ regularity away from the blowup}
\begin{lemma}\label{lem:awbound}
For labels $x$, we have that 
\begin{align} 
&\sup_{t \in [-\eps,T_*)}\max_{\gamma\le 4}\left( \abs{ \p_\theta^\gamma a(\eta(x,t),t)}
+\abs{ \p_\theta^\gamma w(\eta(x,t),t) } \right) \notag \\
& \qquad\qquad
 \le 
  \begin{cases}
C_\eps  \left((T_* -t) +   3\eps^{-3}(\eps+t)(x-x_*)^2\right)^{-4}   & \ \ \abs{x-x_*} \le \eps^{2}  \\
C_\eps& \ \ \abs{x-x_*} \ge \eps^{2}
 \end{cases}
 \,,   \label{aw-good-bound}
\end{align}  
where $C_\eps$ denotes a generic positive constant depending on inverse powers of $\eps$.
\end{lemma} 
\begin{proof}[Proof of Lemma \ref{lem:awbound}]
We use the identities \eqref{w-derivative-idents} for $\p_\theta ^\gamma w \circ \eta$.
The bounds on the initial data \eqref{nice4}, \eqref{nice8}, \eqref{nice9}, \eqref{rice1}, 
the bounds on derivatives of $\eta$ given in \eqref{dx-eta-bound0},  \eqref{etaxx-good-bound}, \eqref{dx3-eta-bound-all-t},
 \eqref{etax-best}, \eqref{etaxx-best}, and \eqref{dx4-eta-bound},
the bounds on $I_t$ and its derivatives given in \eqref{Is012} and \eqref{It'''}
prove the stated bound for $\p_\theta ^\gamma w \circ \eta$ in  \eqref{aw-good-bound}.

The additional  inequalities \eqref{a-boot},  \eqref{nice7}, \eqref{funny-raccoon5}, and \eqref{funny-raccoon9} then proved
 the stated bound for $\p_\theta ^\gamma a \circ \eta$ in  \eqref{aw-good-bound}.
\end{proof}

\begin{proposition}[Taylor expansion for $\eta(x,t)$]\label{prop:lag-smooth}
The $3$-characteristics $\eta$ satisfy
$$
\eta \in C^1([-\eps,T_*], C^4( \mathbb{T}  )) \,,
$$
and  at the blowup time, $\eta(x, T_*)$  has the Taylor expansion about $x_*$ given by
\begin{align} 
\eta(x,T_*) = \eta(x_*,T_*) + \tfrac{1}{6} \p_x^3 \eta(x_*,T_*) (x-x_*)^3 + \tfrac{1}{6} \p_x^4 \eta(\bar x,T_*) (x-x_*)^4 \,, \label{eta-taylor:b}
\end{align} 
for some $\bar x$ between $x_*$ and $x$.
\end{proposition} 
\begin{proof}[Proof of Proposition \ref{prop:lag-smooth}]
 By Lemma 2.1 in \cite{BuShVi2019a}, there exists a short time $\bar T \ge -\eps$, such that $(w,a)$ is a unique solution to \eqref{eq:w:z:k:a2} with initial
 data $(w_0,a_0)$ and 
\begin{align} 
(a,w) \in C^0([-\eps,\bar T];C^4( \mathbb{T}   ))   \,. 
\end{align} 
for any open set $U$ which does not intersect $\xi_*$.   By the local existence and uniqueness theorem for ODE,
$\eta \in C^1([-\eps,\bar T];C^4( \mathbb{T}   ))$.   Given the uniform bounds 
\eqref{dx-eta-bound0},  \eqref{etaxx-good-bound}, \eqref{dx3-eta-bound-all-t},  and \eqref{dx4-eta-bound},
the standard continuation argument shows that 
$$
\eta \in C^1([-\eps,T_*], C^4( \mathbb{T}  )) \,.
$$
The Taylor remainder theorem provides the expansion \eqref{eta-taylor:b}.
\end{proof}

\subsection{Newton iteration to solve quartic equations in a fractional series}
We wish to invert the polynomial equation $\eta(x,T_*)=z$.  As given by \eqref{eta-taylor}, this requires inversion of a quartic polynomial.  We shall  derive the
root that yields a H\"{o}lder-$ \tfrac{1}{3} $ solution for $\eta ^{-1} ( \cdot, T_*)$ and satisfies $\eta ^{-1} ( \xi_*, T_*) = x_*$.
\begin{lemma}[Quartic inversion]\label{lem1}If
$$f(x,y)= -x + a_3  y^3 + a_4 y^4 \,,$$
and $a_3>0$, then the
solution $y(x)$ to $f(x,y)=0$ such that $y(0)=0$ is given by the fractional power-series
\begin{align} 
y(x) = a_3^ {-\frac{1}{3}}  x^ {\frac{1}{3}} -  \tfrac{1}{3} a_4 a_3^ {-\frac{5}{3}}  x^ {\frac{2}{3}} +  \tfrac{1}{3}  a_3^{-3} a_4^2 x +\OO(|x|^ {\frac{4}{3}} ) \,.\label{expansion1}
\end{align} 
\end{lemma} 
\begin{proof}[Proof of Lemma \ref{lem1}]
We will first obtain an approximate solution using the Newton polygon method.
Each term of the polynomial $f(x,y)$ is written as $c x^a y^b$, and the Newton polygon for $f(x,y)$ is constructed as the smallest convex polygonal set that contains
the points $b e_1 + a e_2$.  This polygon consists of a finite set of segments, and we consider the segment $\Gamma_1$, such that each of the points 
$(b,a) = b e_1 + a e_2$
is either above or to the right of this segment.

We will construct a fractional-series solution to $f(x,y)=0$ as
\begin{align} \label{y-expand0}
y(x)= c_1 x^{\gamma_1} + c_2 x^{\gamma_1+\gamma_2}+  c_3 x^{\gamma_1+\gamma_2+\gamma_3}+  \cdot \cdot \cdot  \,.
\end{align} 
The first fractional power $\gamma_1$ is chosen as minus the slope of $\Gamma_1$.  For $-x + a_3  y^3 + a_4 y^4=0$, the points $(b,a)$ are given by $(0,1)$, $(3,0)$, and
$(4,0)$, and thus it is easy to see that the two lower segments of the Newton polygon have slopes $- \tfrac{1}{3} $ and $0$, but that the segment with 
slope $0$ exists only if $c_4 \neq 0$.   We first consider the  segment $\Gamma_1$ with slope $- \tfrac{1}{3} $, in which case
$\gamma_1= {\tfrac{1}{3}} $.   We thus factor $x^ {\frac{1}{3}} $ from \eqref{y-expand0}, and write
\begin{align*}
y(x) = x^ {\frac{1}{3}} (c_1 + y_1(x))\,, \ \ y_1(x) = c_2 x^{ \gamma_2}+ c_3 x^{  \gamma_2+ \gamma_3} + \cdot \cdot \cdot  \,.
\end{align*} 
We compute
\begin{align*} 
f(x,  x^ {\frac{1}{3}} (c_1 + y_1) ) = -x +a_3 x(c_1 +y_1)^3 + a_4 x^ {\frac{4}{3}} (c_1 +y_1)^4 \,.
\end{align*} 
The coefficient of the monomial $x$ must equal to zero, so  we can determine $c_1$:
\begin{align*} 
x(-1+ a_3 c_1^3) =0 \ \  \Longrightarrow \ \ c_1 =a_3^ {-\sfrac{1}{3}}  \,.
\end{align*} 
We next define $f_1(x,y_1) = x^{- \alpha_1 } f(x,  x^ {\frac{1}{3}} (c_1 + y_1) )$ where $ \alpha_1 $ is the intersection of the segment $\Gamma_1$ and the vertical $a$-axis, so that $ \alpha_1 =1$.   We have that
\begin{align*} 
f_1(x,y_1) &= x^{-1 } f(x,  x^ {\frac{1}{3}} (a_3^ {-\sfrac{1}{3}} + y_1) ) \\
&= a_4 a_3^{-\sfrac{4}{3}}  x^\frac{1}{3} + 3 a_3^{\sfrac{1}{3}}y_1 
+ 4 a_4 a_3 ^{-1} x^ {\frac{1}{3}} y_1 
\notag\\
&\qquad + 3 a_3^{\sfrac{2}{3}} y_1^2 
+ 6 a_4 a_3^{-\sfrac{2}{3}} x^ {\frac{1}{3}} y_1^2
+ a_3 y_1^3 + 4 a_4 a_3^ {-\frac{1}{3}}  x^ {\frac{1}{3}} y_1^3
+ a_4 x^ {\frac{1}{3}} y_1^4 
 \,.
\end{align*} 
The Newton polygon for $f_1(x, y_1) =0$ shows that the segment $\Gamma_2$, whose slope is equal to minus the exponent $\gamma_2$, 
connects the points $(0, {\tfrac{1}{3}} )$
and $(0,1)$, so that $\gamma_2 = {\tfrac{1}{3}} $.   We next write
\begin{align*} 
y_1(x) = x^ {\frac{1}{3}} (c_2 +y_2(x))\,, \ \ y_2(x) = c_3 x^{ \gamma_3}+ c_4 x^{  \gamma_3+ \gamma_4} + \cdot \cdot \cdot  \,.
\end{align*} 
We compute $f_1(x, x^ {\frac{1}{3}} (c_2+ y_2))$ and cancel the coefficients in the lowest-order term to find that $c_2 = - \tfrac{1}{3} a_4 a_3^ {-\frac{5}{3}}  $.  
We then define
\begin{align*} 
f_2(x,y_2) &=x^{- \alpha _2  } f_1(x,  x^ {\frac{1}{3}} (c_2  + y_2) ) = x^{-{\frac{1}{3}}  } f_1(x,  x^ {\frac{1}{3}} (- {\tfrac{1}{3}} a_4 a_3^ {-\frac{5}{3}}  + y_2) ) \,,
\end{align*} 
where $ \alpha _2= {\tfrac{1}{3}} $ is the $a$-intercept for  the segment $\Gamma_2$.    A   computation reveals that
\begin{align*} 
f_2(x,y_2)= -a_4^2 a_3 ^ {-\frac{8}{3}}   x^ {\frac{1}{3}}  + 3 a_3^{\frac{1}{3}}  y_2 + o(|x|^ {\frac{1}{3}} )\,,
\end{align*} 
and the 
 Newton polygon for $f_2(x,y)$ shows that the exponent $\gamma_3 = \tfrac{1}{3} $, which in turn shows that $y_2(x) = C x^ {\frac{1}{3}} +  \cdot \cdot \cdot $. 
 Continuing one more step in the iteration to $f_3(x,y_3)$ (whose details we omit), we find that
  $C=\tfrac{1}{3}  a_3^{-3} a_4^2$.
We thus determined the first two non-trivial terms of this fractional series expansion  \eqref{expansion1}.  The result follows by an application of the implicit function theorem to the approximate solution that we have just determined.

We now return to the case in which the first fractional power uses the 
segment of the Newton polygon with slope $0$.  In this case,  we begin the iteration with $\gamma_1=0$, we find that
$y(x) = -\tfrac{a_3}{a_4}  -\tfrac{a_4^2}{a_3^3} x + \OO(x^2)$. Note however that $y(0)\neq 0$ in this case.
\end{proof}

\subsection{Proof of Theorem \ref{thm:blowup-profile}}
Having established the expansion for $\eta(x, T_*)$ we can now prove the main result of this section.
 \begin{proof}[Proof of Theorem \ref{thm:blowup-profile}] 
We consider labels $x$ satisfying $|x - x_*| \le \eps^{3}$.
By Proposition \ref{prop:lag-smooth}, we have that $\eta(x,T_*)$ has the Taylor series expansion \eqref{eta-taylor}, which we write again as
\be
\label{etaTstar}
\eta(x,T_*) = \xi_*+ \tfrac{1}{6} \p_x^3 \eta(x_*,T_*) (x-x_*)^3 + \tfrac{1}{24} \p_x^4 \eta(\bar x,T_*) (x-x_*)^4  \,,
\ee
where $\xi_*= \eta(x_*,T_*)$, and $\bar x$ is a point between $x_*$ and $x$.
   By \eqref{dx3-eta-bound}, the  coefficient for the cubic monomial cannot vanish: 
\begin{align} 
 \tfrac{1}{6} \p_x^3 \eta(x_*,T_*) \ge \tfrac{6 - \eps^ {\frac{1}{6}} }{\eps^3}  >0 \,.
 \label{nozero1}
\end{align} 

Setting $\eta(x,T_*)  = \theta $,  we find that 
$$
\tfrac{1}{6} (x-x_*)^3  \p_x^3\eta(x_*,T_*) +    \tfrac{1}{24} (x-x_*)^4 \p_x^4 \eta(\bar x,T_*)   = \theta  - \xi_* \,.
$$
We define the constants\footnote{Note that, as defined by \eqref{alpheqn},  $\alpha _2$ and $\alpha _3$ actually depend on $x$ through the intermediate point $\bar{x}$, and thus are not truly constants.  Nevertheless, in our proof we need only upper and lower bounds on $\alpha _2$ and $\alpha _3$ which are independent of $x$;  no information on the regularity of these functions in $x$ is needed. The same comment applies to $b_3$ defined in \eqref{pos01}. It is however crucial that $\alpha_1, b_1$ and $b_2$ are independent of $x$.  We emphasize that if the initial data $(w_0,a_0)$ is taken to be $C^5$ smooth instead of just $C^4$, then the expansion \eqref{etaTstar} can be developed to fifth order, making $\alpha_1, \alpha_2,b_3$ constants in $x$.  We omit these computations which are straightforward but involved.}
\begin{subequations} 
\label{alpheqn}
\begin{align} 
\alpha _1 & = \bigl(\tfrac{6}{\p_x^3\eta(x_*,T_*)}\bigr)^ {\frac{1}{3}} >0 \,, \label{nozero2} \\
\alpha _2 & = -\tfrac{1}{24}  \p_x^4\eta(\bar x,T_*)  \bigl(\tfrac{6}{\p_x^3\eta(x_*,T_*)}\bigr)^ {\frac{5}{3}} \,, \\
\alpha _3 & = \bigl(\tfrac{6}{\p_x^3\eta(x_*,T_*)}\bigr)^ 3 \bigl( \tfrac{\p_x^4\eta(\bar x,T_*)}{24}  \bigr)^2 \,,
\end{align} 
\end{subequations} 
where clearly the positivity condition \eqref{nozero2} is merely a restatement of \eqref{nozero1}.
Using Lemma \ref{lem1}, we have that
\begin{align} 
x- x_* &
=  \alpha _1  (\theta - \xi_*)^ {\frac{1}{3}}  +
\alpha _2  (\theta - \xi_*)^ {\frac{2}{3}} +  \alpha _3 (\theta -\xi_*) + \OO(|\theta - \xi_*|^ {\frac{4}{3}}) \,
\label{eta-inv}
\end{align}

We define the function
\begin{align*} 
\mathcal{I}(x) = - \tfrac{8}{3}\int_{-\eps}^{T_*}a(\eta(x,r),r)dr\,.
\end{align*} 
Taylor expanding $w_0(x)$ about $ x_*$ in the identity \eqref{eq0}, we have that
\begin{align} 
w(\eta(x,t),T_*) & = e^{\mathcal{I} (x)}w_0(x) \notag \\ 
 &= e^{\mathcal{I} (x)}\Bigl( w_0(x_*)  + \p_x w_0(x_*)(x-x_*)  + \tfrac{1}{2} \p_x^2 w_0(x_*)  (x-x_*)^2  \notag \\
 & \qquad\qquad\qquad\qquad\qquad
 + \tfrac{1}{6} \p_x^3 w_0(x_*)  (x-x_*)^3 + \tfrac{1}{24} \p_x^4 w_0(\bar x)  (x-x_*)^4   \Bigr) \,,
 \label{headache1}
\end{align} 
for some $\bar x$ between $x_*$ and $x$.  

By Proposition \ref{prop:lag-smooth}, $a \circ \eta \in C^4$, so we can  apply the Taylor remainder theorem to the function $e^{ \mathcal{I} (x)}$, expanding
about  about $x_*$, and obtain
\begin{align} 
e^{ \mathcal{I} (x)} &= e^{ \mathcal{I} (x_*)}\Bigl(1+ \mathcal{I} '(x_*) (x-x_*) + \tfrac{1}{2} ( \mathcal{I} '(x_*)^2 + \mathcal{I} ''(x_*)) (x-x_*)^2 \notag  \\
& \qquad \qquad\qquad \qquad
+ \tfrac{1}{6} ( \mathcal{I} '(\hat x)^3 + 3 \mathcal{I} '(\hat x) \mathcal{I} ''( \hat x) + \mathcal{I} '''(\hat x)  ) (x-x_*)^3 \Bigr) \,,  \label{pos00}
\end{align} 
where $\hat x$ is a point between $x_*$ and $x$.
To simplify notation, we define the constants
\begin{align} 
b_1 = \mathcal{I} '(x_*) \,,
 \ \ b_2 =  \tfrac{1}{2} ( \mathcal{I} '(x_*)^2 + \mathcal{I} ''(x_*)) \,, 
\ \ b_3 = \tfrac{1}{6} ( \mathcal{I} '(\hat x)^3 + 3 \mathcal{I} '(\hat x) \mathcal{I} ''( \hat x) + \mathcal{I} '''(\hat x)  ) \,, \label{pos01}
\end{align} 
and write \eqref{pos00} as 
\begin{align} 
e^{ \mathcal{I} (x)} &= e^{ \mathcal{I} (x_*)}\Bigl(1+b_1 (x-x_*) + b_2 (x-x_*)^2  + b_3(x-x_*)^3 \Bigr) \,. \label{pos02}
\end{align} 

From \eqref{headache1} and  \eqref{pos02}, we have that
\begin{align} 
w(\eta(x,t),T_*) & = \!
e^{ \mathcal{I} (x_*)}\Bigl(1+b_1 (x-x_*) + b_2 (x-x_*)^2  + b_3(x-x_*)^3 \Bigr)
\Bigl( w_0(x_*) \! +\! \p_x w_0(x_*)(x-x_*)\!  \notag \\
& \qquad \qquad + \! \tfrac{1}{2} \p_x^2w_0( x_*)  (x-x_*)^2 + \! \tfrac{1}{6} \p_x^3w_0( x_*)  (x-x_*)^3 + \! \tfrac{1}{24} \p_x^4w_0(\bar x)  (x-x_*)^4  \Bigr) \notag  \\
& = e^{ \mathcal{I} (x_*)}\Bigl( w_0(x_*) + \bigl( b_1 w_0(x_*) + \p_x w_0(x_*)\bigr) (x-x_*) \notag \\
& \qquad\qquad
+ \bigl( b_2 w_0(x_*) +  b_1 \p_x w_0(x_*) + \tfrac{1}{2} \p_x^2 w_0(x_*)\bigr) (x-x_*)^2 \notag \\
& \qquad\qquad \qquad\qquad + \bigl( b_3 w_0(x_*) +  b_2 \p_x w_0(x_*) + \tfrac{1}{2}b_1 \p_x^2 w_0(x_*) \notag \\
& \qquad\qquad \qquad\qquad \qquad\qquad
+ \tfrac{1}{6} \p_x^3w_0(x_*)  \bigr) (x-x_*)^3\Bigl) + \OO(|x-x_*|^4) \,.
\label{headache2}
\end{align} 
We define the constants
\begin{subequations} 
\label{Bs}
\begin{align} 
B_1 & =  b_1 w_0(x_*) + \p_x w_0(x_*) \,, \label{B1} \\
B_2 & =  b_2 w_0(x_*) +  b_1 \p_x w_0(x_*) + \tfrac{1}{2} \p_x^2 w_0(x_*) \,, \\
B_3 & =  b_3 w_0(x_*) +  b_2 \p_x w_0(x_*) + \tfrac{1}{2}b_1 \p_x^2 w_0(x_*) + \tfrac{1}{6} \p_x^3w_0(x_*) \,,
\end{align}
\end{subequations} 
and
\begin{align*} 
\kappa_* =  e^{ \mathcal{I} (x_*)} w_0(x_*) \,,
\end{align*} 
and thus
\begin{align} 
w(\eta(x,t),T_*) = \kappa _*  + e^{ \mathcal{I} (x_*)}\Bigl( B_1 (x-x_*) + B_2 (x-x_*)^2 + B_3 (x-x_*)^3\Bigl) + \OO(|x-x_*|^4)  \,.
\label{headache3}
\end{align} 
With $\theta= \eta(x,T_*)$ as before, it follows from \eqref{eta-inv} that
\begin{align} 
w(\theta,T_*) &= 
\kappa_* + e^{ \mathcal{I} (x_*)}\Bigl( \alpha _1 B_1 (\theta - \xi_*)^{\frac{1}{3}} + \bigl(\alpha _2 B_1 + \alpha _1^2 B_2 \bigr)(\theta - \xi_*)^{\frac{2}{3}} 
\notag \\
& \qquad\qquad 
+ \bigl(\alpha _3 B_1 + 2 \alpha _1 \alpha _2 B_2 + \alpha _1^3 B_3 \bigr)(\theta - \xi_*)\Bigr)
 + \OO( |\theta-\xi_*|^ {\frac{4}{3}} ) \,,
 \label{w-profile}
\end{align} 
%

We can now define the constants $ \aa_1$, $\aa_2$,  and $\aa_3$ in \eqref{w-blowup} as follows:
\begin{subequations} 
\label{a1a2C}
\begin{align} 
\aa_1 & =  e^{ \mathcal{I} (x_*)}\alpha _1 B_1 \,, \label{constant:a1} \\
\aa_2 & = e^{ \mathcal{I} (x_*)}( \alpha _2 B_1 + \alpha _1^2 B_2)  \,,  \label{constant:a2}\\
\aa_3   & = e^{ \mathcal{I} (x_*)}( \alpha _3 B_1 + 2 \alpha _1 \alpha _2 B_2 + \alpha _1^3 B_3) \,.  \label{constant:C}
\end{align} 
\end{subequations} 
We note that 
by Lemma \ref{lem:d3eta}, 
\begin{align} 
 \tfrac{9}{10}  \eps \le  \alpha _1 \le \tfrac{11}{10}  \eps\,, \ \ \ \abs{ \alpha _2} \les \eps^{\frac{9}{8}} \,, \ \ \ \abs{ \alpha _3} \le \eps^ {\frac{5}{4}} 
 \,. \label{alphacrap}
\end{align} 
Furthermore, since by \eqref{eq:fat:cat}, $w_0(0) -\kappa_0 =0$, and we assume the  inequality \eqref{eq:tilde:W:3:derivative:0}, we see that
since $\sabs{x_*} \le 2 \kappa_0 \eps^4$, we have that
\begin{align} 
\kappa_0 - 2 \eps^ {\frac{5}{2}}  \ \le w_0(x_*) \le \kappa_0 + 2  \eps^ {\frac{5}{2}}  \,,  \label{crapcycle1}
\end{align} 
and from \eqref{whynot9monthsago1}
\begin{align} 
 - \tfrac{1+\eps}{\eps} \le  \p_x w_0 (x_*) \le   - \tfrac{1-\eps}{\eps}\,, \ \ \ |\p_x^2 w_0(x_*)| \le 7  \eps^ {\frac{1}{2}} \,.  \label{crapcycle2}
 \end{align} 
 From \eqref{pos01} and \eqref{eq:A_bootstrap:IC},  we see that $b_1$, $b_2$, and $b_3$ are $\OO(\eps)$.   Using \eqref{Bs} together with 
 \eqref{crapcycle1} and  \eqref{crapcycle2}, we find that
 \begin{align*} 
 - \tfrac{1+\eps}{\eps}-\eps^{\frac{9}{10}}  \le  B_1 \le   - \tfrac{1-\eps}{\eps}+\eps^{\frac{9}{10}}   \,, \ \ \ \abs{B_2} \le 4 \eps^ {\frac{1}{2}} \,.
 \end{align*} 
 Together with \eqref{a1a2C} and  \eqref{alphacrap}, we have  that for $\eps$ taken small enough,
 \begin{align*} 
 -\tfrac{6}{5} \le \aa_1 \le - \tfrac{4}{5}  \,, \ \ \ \sabs{ \aa_2} \les \eps^ {\frac{1}{8}} \le  \eps^ {\frac{1}{10}} \,, \ \ \ \sabs{ \aa_3} \le \tfrac{7}{6 \eps}    \,.
 \end{align*}

Let us now follow the same argument that we used above to produce an expansion for $w_x(\eta(x,T_*),T_*)$.
We see that
\begin{align} 
w_x(\eta(x,t),T_*) \eta_x(x,T_*) & =e^{\mathcal{I} (x)} \bigl( w_0'(x)  + \mathcal{I} '(x) w_0(x) \bigr)  \notag \\ 
 &= e^{\mathcal{I} (x)}\Bigl(  \p_x w_0(x_*)(x-x_*)  +  \p_x^2 w_0(x_*)  (x-x_*)   + \tfrac{1}{2} \p_x^3 w_0(x_*)  (x-x_*)^2 \notag \\
 & \ \ \ 
+ \tfrac{1}{6} \p_x^4 w_0(\bar x)  (x-x_*)^3   \Bigr) + e^{\mathcal{I} (x)} \mathcal{I} '(x)\Bigl( w_0(x_*)  + \p_x w_0(x_*)(x-x_*)  \notag \\
 & \qquad + \tfrac{1}{2} \p_x^2 w_0(x_*)  (x-x_*)^2 
 + \tfrac{1}{6} \p_x^3 w_0(x_*)  (x-x_*)^3 + \tfrac{1}{24} \p_x^4 w_0(\bar x)  (x-x_*)^4   \Bigr) \,,
 \label{headache1-dx}
\end{align} 
where $\bar x$ lies between $x$ and $x_*$.
In addition to \eqref{pos02},  we shall need the expansion of $e^{ \mathcal{I} (x)}  \mathcal{I}'(x)$ and we continue to use $b_1$, $b_2$, $b_3$ 
defined in \eqref{pos01} and write
\begin{align} 
e^{ \mathcal{I} (x)}  \mathcal{I}'(x)
&= e^{ \mathcal{I} (x_*)}\Bigl(  b_1 +  2 b_2 (x-x_*) 
+ 3 b_3 (x-x_*)^2 \Bigr) \,,  \label{pos000}
\end{align} 
We can then write
\begin{align} 
& w_x(\eta(x,t),T_*) \eta_x(x,T_*) \notag \\
&\qquad = e^{\mathcal{I} (x_*)} \Bigl( b_1w_0(x_*)  +\p_x w_0(x_*) + \bigl(2b_2 w_0(x_*) + 2 b_1 \p_x w_0(x_*) + \p_x^2 w_0(x_*)\bigr) (x-x_*) \notag \\
&\qquad
+ \tfrac{1}{2} \bigl(6 b_3 w_0(x_*) + 6b_2 \p_x w_0(x_*) + 3 b_1 \p_x^2(x_*) + \p_x^3 w_0(x_*) \bigr) (x-x_*)^2
+ \OO(|x-x_*|^3)
\Bigr) \,. \label{so-complicated-0}
\end{align} 
With the expansion $\eta_x(x,T_*) $ is written as
\begin{align} 
\eta_x(x,T_*)  = \tfrac{1}{2} \p_x^3\eta(x_*,T_*) (x-x_*)^2 + \tfrac{1}{6} \p_x^4\eta(\mathring x ,T_*) (x-x_*)^3 \label{dxeta-taylor}
\end{align} 
for some $\mathring x \in (x, x_*)$.  Therefore, with \eqref{so-complicated-0},  we have that
\begin{align} 
& w_x(\eta(x,t),T_*)    = e^{\mathcal{I} (x_*)} \Bigl( b_1w_0(x_*) +\p_x w_0(x_*) + \bigl(2b_2 w_0(x_*) + 2 b_1 \p_x w_0(x_*) + \p_x^2 w_0(x_*)\bigr) (x-x_*) 
\notag \\
&\qquad
+ \tfrac{1}{2} \bigl(6 b_3 w_0(x_*) + 6b_2 \p_x w_0(x_*) + 3 b_1 \p_x^2(x_*) + \p_x^3 w_0(x_*) \bigr) (x-x_*)^2
+ \OO(|x-x_*|^3) \Bigr) \notag \\
&\qquad \qquad \times \Bigl(  \tfrac{1}{2} \p_x^3\eta(x_*,T_*) (x-x_*)^2 + \tfrac{1}{6} \p_x^4\eta(\mathring x ,T_*) (x-x_*)^2     \Bigr)^{-1}  \,.
\label{so-complicated-1}
\end{align} 
Another expansion of the right side of \eqref{so-complicated-1} gives 
\begin{align} 
w_x(\eta(x,t),T_*)   & = e^{\mathcal{I} (x_*)} \Bigl( d_{-2} (x-x_*)^{-2}  + d_{-1}(x-x_*)^{-1} 
+ d_0  \Bigr) + \OO(|x-x_*|) \,,  \label{so-complicated-2}
\end{align} 
where
\begin{align*} 
d_{-2} & =  \tfrac{2(b_1 w_0(x_*) + \p_x w_0(x_*))}{\p_x^3\eta(x_*,T_*)}  \,, \\
d_{-1} & = \tfrac{2\bigl( 2 b_2 w_0(x_*) + 2 b_1 \p_x w_0(x_*) + \p_x^2w_0(x_*)\bigr)}{\p_x^3\eta(x_*,T_*)} 
- \tfrac{2\bigl(  b_1  w_0(x_*) + \p_x w_0(x_*) \bigr)\p_x^4\eta(\mathring x,T_*)}{\p_x^3\eta(x_*,T_*)^2}  \,, \\
d_0 & =
 \tfrac{ 6 b_3 w_0(x_*) + 6 b_2 \p_x w_0(x_*) +  3b_1 \p_x^2w_0(x_*)+\p_x^3w_0(x_*)}{\p_x^3\eta(x_*,T_*)} 
- \tfrac{2\bigl(  2b_2  w_0(x_*) + 2b_1  \p_x w_0(x_*) + \p_x^2 w_0(x_*)\bigr) \p_x^4\eta(\mathring x,T_*)}{3\p_x^3\eta(x_*,T_*)^2}  \notag \\
& \qquad 
+  \tfrac{ 2 \bigl(b_1 w_0(x_*) +  \p_xw_0(x_*) \bigr)\p_x^4\eta(\mathring x,T_*) }{9\p_x^3\eta(x_*,T_*)^3}  \,.
\end{align*} 
By substituting \eqref{eta-inv} into \eqref{so-complicated-2}, we obtain that
\begin{align} 
&\abs{w_x(\theta, T_*)  -  e^{\mathcal{I} (x_*)} \alpha _1 ^{-2} d_{-2} z^{-\frac{2}{3}}  -  e^{\mathcal{I} (x_*)} \bigl( \alpha_1^{-1} d_{-1} - 2 \alpha_1^{-3} \alpha_2 d_{-2}\bigr) z^{-\frac{1}{3}} }  
\notag \\
& \qquad \qquad\qquad \qquad\qquad \qquad
\le 2 e^{\mathcal{I} (x_*)} \Bigl( d_0 -  \alpha _1^{-2} \alpha _2 d_{-1} + (3 \alpha _2^2 - 2 \alpha_1 \alpha _3) \alpha _1^{-4} d_{-2}\Bigr) \,.
\end{align} 
Notice from \eqref{nozero2}, \eqref{B1},  \eqref{constant:a1}  that since
$$
\aa_1 =  e^{\mathcal{I} (x_*)}  \bigl(\tfrac{6}{\p_x^3\eta(x_*,T_*)}\bigr)^ {\frac{1}{3}} \bigl(b_1 w_0(x_*) + \p_x w_0(x_*)\bigr) \,,
$$
and since
\begin{align*} 
e^{\mathcal{I} (x_*)} \alpha _1 ^{-2} d_{-2}  & = 2e^{\mathcal{I} (x_*)}  \bigl(\tfrac{6}{\p_x^3\eta(x_*,T_*)}\bigr)^ {-\frac{2}{3}}
\Bigl(\tfrac{b_1 w_0(x_*) + \p_x w_0(x_*)}{\p_x^3\eta(x_*,T_*)} \Bigr) \\
&=   \tfrac{1}{3}  e^{\mathcal{I} (x_*)}  \bigl(\tfrac{6}{\p_x^3\eta(x_*,T_*)}\bigr)^ {\frac{1}{3}} \bigl(b_1 w_0(x_*) + \p_x w_0(x_*)\bigr)    =  \tfrac{1}{3}  \aa_1 \,,
\end{align*} 
A similar computation shows that
\begin{align*} 
 e^{\mathcal{I} (x_*)} \bigl( \alpha_1^{-1} d_{-1} - 2 \alpha_1^{-3} \alpha_2 d_{-2}\bigr)  =   \tfrac{2}{3}  \aa_2 \,.
\end{align*} 
As such, we have established the inequality
\begin{align} 
&\abs{\p_\theta w (\theta, T_*)  -    \tfrac{1}{3} \aa_1 (\theta - \xi_*)^{-\frac{2}{3}}  -    \tfrac{2}{3}  \aa_2 (\theta - \xi_*)^{-\frac{1}{3}} }  \le C_\mm\,,  \label{wx-taylor-fk}
\end{align} 
where
\begin{align*} 
C_\mm =  2 e^{\mathcal{I} (x_*)} \Bigl( d_0 -  \alpha _1^{-2} \alpha _2 d_{-1} + (3 \alpha _2^2 - 2 \alpha_1 \alpha _3) \alpha _1^{-4} d_{-2}\Bigr) \,,
\end{align*} 
satisfies $\sabs{C_\mm} \les \tfrac{1}{\eps} $. The inequality \eqref{wx-taylor-fk} and the bound for $C_\mm$ establishes \eqref{need-like-a-hole-in-the-head1}.

From \eqref{headache1-dx}, we see that
\begin{align} 
\p_\theta^2 w (\eta(x,T_*),T_*)  \eta_x^2(x,T_*)& = -\p_\theta (\eta(x,T_*),T_*) \eta_{xx}(x,T_*)  \notag \\
& \qquad \qquad + 
e^{\mathcal{I} (x)} \bigl( w_0''(x) + 2I'(x)w'(x) + \mathcal{I} ''(x) w_0(x) \bigr)  \,. \label{headache1-dxx}
\end{align} 
In addition to the expansion \eqref{dxeta-taylor}, we shall also need the fact that
\begin{align*} 
\eta_{xx}(x,T_*) = \p_x^3\eta(x_*,T_*) (x-x_*) + \tfrac{1}{2} \p_x^4\eta(\ringring x,T_*) (x-x_*)^2
\end{align*} 
for some $\ringring x \in (x, x_*)$.    
After a lengthy computation, we find that
\begin{align} 
&\abs{\p_\theta^2 w(\theta , T_*)  -    \tfrac{2}{9}  \aa_1 (\theta - \xi_*)^{-\frac{5}{3}}   }  \le   \bar C_\mm  (\theta - \xi_*)^{-\frac{4}{3}}\,,  \label{wxx-taylor-fk}
\end{align} 
where
$$
\abs{\bar C_\mm} \les  \eps^ {-\frac{63}{8}}  \,,
$$
which establishes \eqref{need-like-a-hole-in-the-head2}.

Finally, from \eqref{headache1-dxx}, we see that
\begin{align} 
\p_\theta^3w (\eta(x,T_*),T_*)  \eta_x^3(x,T_*)& =  -3 \p_\theta^2 w(\eta(x,T_*),T_*)  \eta_{x}(x,T_*)  \eta_{xx}(x,T_*)    - \p_\theta w (\eta(x,T_*),T_*) \eta_{xxx}(x,T_*)  \notag \\
& \quad  + 
e^{\mathcal{I} (x)} \bigl( w_0'''(x) + 3I'(x)w''(x) + 3I''(x)w'(x)+ \mathcal{I} '''(x) w_0(x) \bigr)  \,. \label{headache1-dxxx}
\end{align}
We make use of one further expansion given by
\begin{align*} 
 \p_x^3\eta(x,T_*) = \p_x^3\eta(x_*,T_*)  +  \p_x^4\eta(\ringringring x,T_*) (x-x_*) 
\end{align*} 
for some $\ringringring x \in (x, x_*)$.   
A final lengthy computation shows that 
\begin{align} 
&\abs{\p_\theta^3 w (\theta, T_*)   }  \les   \eps^ {-\frac{151}{8}}  \sabs{\theta - \xi_*}^{-\frac{8}{3}}\,,  \label{wxxx-taylor-fk}
\end{align} 
which establishes \eqref{need-like-a-hole-in-the-head3}.

The estimates \eqref{thm-aw-good-bound} are established by \eqref{aw-good-bound}.
The bounds \eqref{svort-thm-bounds} for the specific vorticity are established in  \eqref{spvort-bound1} and \eqref{dx-spvort-bound1}. From
\eqref{a-bound1}  we have that
$\sup_{[0,T_*)}\norm{ a( \cdot , t)}_{L^ \infty } \le \tfrac{3}{2}  \eps$. From \eqref{w-blowup}, we have that $w( \cdot , T_*) \in C^{{\frac{1}{3}} }( \mathbb{T}  )$; therefore,
since $\p_x a = \tfrac{w^2}{16} \varpi - w$, by \eqref{w-lowerupper} and \eqref{spvort-bound1}, 
we have that
  $a( \cdot , T_*) \in C^{1,{\frac{1}{3}} }( \mathbb{T}  )$ which gives the regularity statement in \eqref{wa-at-blowup}. 
The bounds for $\varpi$ are given in \eqref{spvort-bound1}, and for $\p_x\varpi$ in \eqref{dx-spvort-bound1}.
\end{proof}


\section{Shock development}
\label{sec:development}

In this section we consider the system \eqref{eq:w:z:k:a}--\eqref{eq:wave-speeds}, with pre-shock initial datum as obtained in Section~\ref{sec:formation}, and consider the associated {\em development problem}. The main result is Theorem~\ref{thm:main:development} below.

\subsection{Initial data for shock development comes from the {\em pre-shock}}
\label{sec:shock:IC}
Theorem~\ref{thm:blowup-profile} guarantees the finite time formation of a first singularity for the $(w,z,a,k)$ system~\eqref{eq:w:z:k:a} at $(\theta,t) = (\xi_*,T_*)$; more 
precisely, the first Riemann variable $w$ forms a $C^{\frac 13}$ {\em pre-shock} as described in \eqref{w-blowup}, $z$ and $k$ remain equal to 
$0$ (their initial datum), while the function $a$  retains $C^{1,\frac 13}$ regularity at the time that the pre-shock forms.  

The initial data for the development problem is provided by Theorem~\ref{thm:blowup-profile}.   
For the remainder of paper, it is convenient to change coordinates so that the pre-shock occurs at $\theta = 0$ (instead of $\xi_*$), at time $t=0$ (instead of $T_*$). 
The initial condition for the first Riemann variable thus is 
$w_0(\theta)= w(\theta-\xi_*,T_*)$, with the latter function being given by \eqref{w-blowup}. In particular, we have that $w_0$ satisfies the quantitative estimates 
\begin{subequations}
 \label{eq:u0:ass:quant}
\begin{align}
  w_0(\theta) 
&\leq \mm
\label{eq:u0:ass:2}
\\
 w_0(\theta) 
&\geq \tfrac 12 \kappa
\label{eq:u0:ass:2a}
\\
 \sabs{ w_0(\theta) - \kappa + \bb \theta^{\frac 13} - \cc \theta^{\frac 23} } &\leq \mm  \abs{\theta}\,,
 \label{eq:u0:ass}
 \\
 \sabs{w_0'(\theta) + \tfrac{1}{3} \bb \theta^{-\frac 23} - \tfrac{2 }{3} \cc \theta^{-\frac 13} }  &\leq \mm  \,,
\label{eq:u0:ass:3}
\\
\sabs{{w_0''(\theta) - \tfrac{2 }{9} \bb \theta^{-\frac 53}  } }
&\leq \mm  |\theta|^{-\frac 43} \,,
\label{eq:u0:ass:4}
\\
\sabs{w_0'''(\theta)   } 
&\leq {\mm  \abs{\theta}^{-\frac 83}} \,,
\label{eq:u0:ass:5}
\end{align}
\end{subequations}
for all $\theta\in \TT$, 
where $\kappa , \mm \geq 1, \bb > 0$, and $\cc \in \RR$ are  suitable constants given as follows. In light of \eqref{w-blowup} and \eqref{preshock-derivatives}, we identify $\kappa = \kappa_*$, $\bb = - \aa_1$, $\cc = \aa_2$, while the constant $\mm$ is taken to be sufficiently large, in terms of the large parameters $\kappa_0$ and $\eps^{-1}$ from Theorem~\ref{thm:blowup-profile}. Note however that \eqref{w-blowup} and \eqref{preshock-derivatives} only give the bounds \eqref{eq:u0:ass}--\eqref{eq:u0:ass:5} for $\theta$ in a $\eps$-dependent ball around $0$ (of radius $\eps^{4}$, recall that we have mapped $\xi_* \mapsto 0$), whereas in  \eqref{eq:u0:ass:quant} we require that these bounds hold for all $\theta \in \TT$. We note however that for $|\theta|$ which is at a fixed positive distance away from $0$, the bounds \eqref{eq:u0:ass}--\eqref{eq:u0:ass:5} follow once $\mm$ is chosen to be sufficiently large with respect to $\kappa_0$ and $\eps^{-1}$; this is because the bounds \eqref{thm-aw-good-bound} imply uniform $C^4$ regularity once a fixed distance from the pre-shock is chosen. Indeed, \eqref{thm-aw-good-bound}, \eqref{dx-eta-bound0}, \eqref{etax-lower-bad}, and \eqref{etax-best} show that for $|\theta|\geq \eps^{4}$, there exists a constant $C_\eps>0$ such that $|\p^\gamma_\theta w_0(\theta)|\leq C_\eps$ for $0\leq \gamma \leq 4$.

We also note that by \eqref{tau-kappa-xi-bounds} and \eqref{Taylor-coefficients} the coefficients in \eqref{eq:u0:ass:quant} satisfy the conditions
\begin{align*}
|\kappa - \kappa_0| \leq \eps^3 \,,
\qquad
{\tfrac 12} \leq \bb \leq {2} \,, 
\qquad
|\cc| \leq {\eps^{\frac 12}} \,,
\end{align*}
where we recall that $\kappa_0>1$ was chosen sufficiently large.
In order to simplify our argument we shall  frequently use the relations
\begin{align}
 |\cc|   \ll {\bb} \leq 2 \qquad \mbox{and} \qquad 4  \leq \kappa \ll \mm \,.
\label{eq:b:m:ass}
\end{align}
In particular, we shall use that $\mm$ sufficiently large with respect to $\kappa$: if $\mathsf{C}>0$ is a universal constant (independent of $\kappa, \bb, \cc, \mm$), then $\kappa \mathsf{C} \leq \mm^{\frac{1}{10}}$. Similarly, we shall use that $|\cc|$ is sufficiently small with respect to $\bb$, so that $\mathsf{C} \bb |\cc| \leq 1$.

The initial conditions for the second Riemann variable and the entropy function are given by 
\begin{align}
z_0(\theta ) \equiv 0\,,
\qquad \mbox{and} \qquad 
k_0(\theta) \equiv 0\,. 
\label{eq:z0:k0:ass}
\end{align}
Lastly, in view of Theorem~\ref{thm:blowup-profile} we identify $a_0(\theta) = a(\theta-\xi_*,T_*) \in C^{1,\frac 13}$ and
$\varpi_0(\theta) = \varpi(\theta-\xi_*,T_*) \in C^{1}$. In particular, due to \eqref{a-bound1} and \eqref{ax-bound1}, 
\begin{align}
\norm{a_0}_{W^{ 1,\infty}(\TT)} \leq  \tfrac{3}{2}  \kappa\,,
 \label{eq:a0:ass}
\end{align}
and due to \eqref{svort-thm-bounds}, we have that 
\begin{align}
{\tfrac{10}{\kappa }} \le \varpi_0(\theta) \le \tfrac{28}{\kappa }  \qquad \mbox{and} \qquad   \sabs{\varpi'_0(x)} \le {\mm}  \,,
 \label{eq:svort0:ass}
\end{align}
for all $\theta \in \TT$.

\begin{remark}[\bf The small parameter $\bar \eps$ and the large constant $C$]
\label{rem:parameters:1}
Throughout Sections~\ref{sec:development} and~\ref{sec:C2}, we shall denote by $C =C (\kappa,\bb, \cc,\mm) \geq 1$ a generic constant, which only depends on  the parameters $\kappa, \bb, \cc$, and $\mm$, which appear in \eqref{eq:u0:ass:quant}, and which may increase from line to line. We shall also denote by $\bar \eps =  \bar \eps(\kappa,\bb, \cc,\mm) \in (0,1]$ a sufficiently small constant, which only depends on  the parameters $\kappa, \bb, \cc$, and $\mm$. Note that the parameter $\bar \eps$ is not the same as the parameter $\eps$ in Section~\ref{sec:formation}.
\end{remark}

\subsection{Definitions}

\begin{definition}[\bf Jump, mean, left value, right value, domain]
\label{def:jump:etc}
Given a smooth curve $\sc \colon [0,T]\to \TT$, we shall denote  
\begin{align}
\DD_{T} = (\TT \times [0, T]) \setminus (\sc(t),t)_{t\in[0,T]}
\label{eq:DD:T:def}
\end{align}
the space-time domain which excludes a shock curve.
Given any function $f\colon \DD_T \to \RR$ we denote the left and right values of $f$ at $\sc$ as
\begin{align}
f_-(t) =  \lim_{\theta \to \sc(t)^{-}} f(\theta,t) 
\qquad \mbox{and} \qquad 
f_+(t) =  \lim_{\theta\to \sc(t)^{+}} f(\theta,t)\,.
\label{eq:left:right:states}
\end{align}
We denote the {\em jump of $f$ across $\sc$ by}
\begin{align}
\jump{f} = \jump{f(t)} = f_-(t) - f_+(t)
\,,
\end{align}
and the {\em mean of $f$ at $\sc$ by}
\begin{align}
\mean{f} = \mean{f(t)} = \tfrac 12 \left( f_-(t) + f_+(t) \right)
\,,
\end{align}
for all $t\in [0,T]$.
The dependence of $f_-, f_+$, $\jump{f}$, and $\DD_{T}$ on the curve $\sc$ is not displayed.
\end{definition}

Next, we define a  space ${\mathcal X}_{T}$ which will be used for the construction of unique solutions.  

\begin{definition}[\bf Functional space for shock emanating from $C^{1/3}$ pre-shock]
\label{def:XT}
Let $\mm>1$ be as in \eqref{eq:b:m:ass}. 
Given $T >0$ and a curve $\sc \colon [0,T]\to \TT$, define the norm
\begin{align}
 |\!|\!| (v,z,k,a)  |\!|\!| _{T }
 &= \sup_{(\theta,t) \in \DD_{T}} \max  \Bigl\{
 t^{-1} (50 \mm^{2})^{-1} \abs{ v   (\theta,t)} \,, \mm^{-3}  \left( \bb^3 t^3 + (\theta- \sc(t))^2 \right)^{\frac{1}{6}}    \abs{ \p_\theta v (\theta,t) } 
\,,   \notag\\
 &\qquad  \qquad \qquad \qquad
  \mm^{-1} t^{- \frac 32}  \abs{ z(\theta,t)}
\,, 
 \mm^{-1} t^{-\frac 12} \abs{ \p_\theta z (\theta,t)  } \,,
\mm^{-\frac 12} t^{-\frac 32} \abs{ k (\theta,t) } 
 \,,  \notag \\
  &\qquad  \qquad \qquad \qquad   
  \mm^{-\frac 12} t^{-\frac 12} \abs{\p_\theta k(\theta,t) } 
 \,,  (4\mm)^{-1} \abs{a (\theta,t) } 
 \,,  (4\mm)^{-1} \abs{\p_\theta a (\theta,t)} \Bigr\}
 \label{eq:the:norm}
\end{align}
where $\DD_T$ is as defined in \eqref{eq:DD:T:def}.
 For $T>0$  we also define
\begin{align}
{\mathcal X}_{T} 
= \Bigl\{ 
(w,z,k,a) \in C^1_{\theta,t}(  \DD_{T}) 
\colon  (w,z,k,a)|_{t=0} = (w_0,0,0,a_0)  \,,  |\!|\!| (w - \wb,z,k,a)  |\!|\!| _{T} \leq 1  \Bigr\} 
\,,
\label{eq:the:X:T}
\end{align}
where $\wb$ is the solution of the 1D Burgers equation in $\DD_T$ with datum $w_0$, which jumps across the shock curve $\sc$ (see Proposition~\ref{prop:Burgers} for its precise definition).
\end{definition}

In order to state the desired properties for $\sc$, in terms of the parameters $\kappa$ and $\bb$ appearing in \eqref{eq:u0:ass}, we define two time-dependent subsets of $\TT$. The first set, $\Sigma$, will be shown to contain the  location of the shock front for $w$ at time $t$, while the second set, $\Omega$, contains the labels of the two particle trajectories associated with the flow of $w$, which fall into the shock at time $t$.
\begin{definition}[\bf Regular shock curve]
\label{def:regular:shock:curve}
For every $t \in [0,  \kappa \mm^{-4} ]$, we define 
\begin{subequations}
\begin{align}
\Sigma(t) &= \bigl[\kappa t - \tfrac 12  {\mm^4} t^2, \kappa t + \tfrac 12  {\mm^4} t^2 \bigr]
\label{eq:xi:range}
\\
\Omega(t) &= \bigl[ - \tfrac{5}{4} (\bb t)^{\frac 32} , - \tfrac{3}{4} (\bb t)^{\frac 32}  \bigr] \cup \bigl[ \tfrac{3}{4}(\bb t)^{\frac 32} , \tfrac{5}{4} (\bb t)^{\frac 32} \bigr]
\label{eq:x:range:interest}
\end{align}
\end{subequations}
extended periodically on the circle $\TT$.
For a given $T \in (0, \kappa \mm^{-4})$, we say that  $t\mapsto\sc(t) \colon [0,T] \to \TT$ 
is a {\em regular shock curve} if it  $\sc$ satisfies 
\begin{align}
\sc(t) \in \Sigma(t) \,,
\qquad  
\abs{\dot{  \sc}(t) - \kappa} \leq  {\mm^4} t  \,,
\qquad 
\abs{\ddot{\sc}(t)} \leq 6  {\mm^4} \,,
\label{eq:sc:ass}
\end{align}
for all $ t\in (0,T] $.
\end{definition}

\subsection{The shock development problem in azimuthal symmetry}

We defined a solution to the development problem in Definition \ref{def:sol:azimuthal}.  The main result of this section is to establish the existence and the uniqueness of such solutions.
\begin{theorem}[\bf Azimuthal shock  development]
\label{thm:main:development}
Given pre-shock initial data $(w_0,z_0,k_0,a_0)$ and $\varpi_0$ satisfying conditions \eqref{eq:u0:ass:quant}--\eqref{eq:svort0:ass}, there exist:
\begin{enumerate}
\item \label{item:5.5.i} $\bar\eps = \bar\eps(\bb,\mm,\cc,\kappa) > 0 $ sufficiently small;  
\item \label{item:5.5.ii} a $C^2$ regular shock curve $\sc \colon [0,\bar \eps] \to \TT$, in the sense of Defintion~\ref{def:regular:shock:curve}; in particular, $\sc$ solves the ordinary differential equation \eqref{sdot1}, corresponding to Rankine-Hugoniot jump condition; 
\item \label{item:5.5.iii} a {\em unique} solution $(w,z,k,a) \in {\mathcal X}_{\bar \eps}$ to the system \eqref{eq:w:z:k:a}, in the sense of Definitions~\ref{def:sol:azimuthal} and~\ref{def:XT};
\item \label{item:5.5.iv} two $C^1$ smooth  curves $\sc_1, \sc_2 \colon [0,\bar \eps] \to \TT$, with $\sc_1(0) = \sc_2(0)= 0$ and $\sc_1(t) < \sc_2(t) < \sc(t)$ for $t \in (0,\bar \eps]$, such that $\sc_i$ is a characteristic curve for the $\lambda_i$ wave-speed, $i\in \{1,2\}$;
\end{enumerate}
such that the following hold:
\begin{enumerate}
  \setcounter{enumi}{4}
\item \label{item:5.5.v} letting $\DD_{\bar \eps}^k = \{ (\theta,t) \in \DD_{\bar \eps} \colon \sc_2(t) < \theta < \sc(t)\}$  we have that $k \equiv 0$ on $(\DD_{\bar \eps}^k)^\complement$ with $k(\theta,t) = \OO(( \theta-\sc_2(t))^{\frac 32})$ in $\DD_{\bar \eps}^k$, cf.~\eqref{eq:metal:1}, and $\p_\theta k(\sc_2(t),t) = 0$;
\item \label{item:5.5.vi} letting  $\DD_{\bar \eps}^z = \{ (\theta,t) \in \DD_{\bar \eps} \colon \sc_1(t) < \theta < \sc(t)\}$, we have that $z \equiv 0$ on $(\DD_{\bar \eps}^z)^\complement$ with $z(\theta,t) = \OO(( \theta-\sc_1(t))^{\frac 32})$ in $\DD_{\bar \eps}^z$, cf.~\eqref{eq:metal:3}, and $\p_\theta z(\sc_1(t),t) = 0$;
\item \label{item:5.5.vii} on $\sc(t)$, the function $w(\cdot,t)$ exhibits an $\OO(t^{\frac 12})$ jump, cf.~\eqref{eq:J:M:rough}, while the functions $z(\cdot,t)$ and $k(\cdot,t)$ exhibit $\OO(t^{\frac 32})$ jumps, cf.~\eqref{eq:zl:and:kl:on:shock:L:infinity}, and solve the system of algebraic equations \eqref{pjump77}-\eqref{pjump7};
\item \label{item:5.5.viii} the specific vorticity $\varpi$ (see its definition in \eqref{xland-svort:def}) solves \eqref{xland-svort} in $\DD_{\bar \eps}$, is uniformly bounded with $\OO(\kappa^{-1})$ upper and lower (see~\eqref{eq:varpi:boot}), and is continuous across the shock curve $\sc(t)$;
\item \label{item:5.5.ix} the function $a(\cdot,t)$ is continuous across $\sc(t)$, while  $\p_\theta a (\cdot,t)$  exhibits an $\OO(t^{\frac 12})$ jump.
\end{enumerate}
\end{theorem}

\subsection{A given  shock curve determines $w$, $z$, $k$, and $a$}
\label{sec:curve:determines:all}
The goal of this subsection is to show that given a regular shock curve $\{ \sc(t)\}_{t \in [0,\bar \eps]}$, as in Definition~\ref{def:regular:shock:curve}, we may compute a solution $(w,z,k,a)$ of the system \eqref{eq:w:z:k:a}--\eqref{eq:wave-speeds}  with initial datum as described in Section~\ref{sec:shock:IC}, and which exhibits {\em a jump discontinuity across the curve $\sc(t)$}. This statement is summarized in Proposition~\ref{thm:curve:determines:all} below. Note that at this stage we do not assume that $\sc$ satisfies the ODE which corresponds to the jump conditions in Section~\ref{sec:RH:Euler}; this will be discussed in Section~\ref{sec:shock:evo}.

With the above notation, the main result of this section is:

\begin{proposition}[\bf Computing $w,z,k$, and $a$, in terms of $\sc$]
\label{thm:curve:determines:all} 
Consider initial datum $(w_0,z_0,k_0,a_0)$ which satisfy conditions \eqref{eq:u0:ass:quant}, \eqref{eq:z0:k0:ass}, and \eqref{eq:a0:ass}. 
Let $T_0 >0 $ be given, and assume that  $\sc \colon [0,T_0] \to \TT$ is a given regular shock curve, as in \eqref{eq:sc:ass}.
Then, there exists $\bar \eps \in (0,T_0]$, which is sufficiently small with respect the parameters $\kappa, \bb, \cc,\mm$,  such that the following hold on $[0,\bar \eps]$:
\begin{enumerate}
\item \label{item:5.6.i} There exist functions $(w,z,k,a)$ which belong to the space ${\mathcal X}_{\bar \eps}$ defined in \eqref{eq:the:X:T}.
\item  \label{item:5.6.ii}  On the spacetime region $ \DD_{\bar \eps}$, defined in terms of $\sc$ in \eqref{eq:DD:T:def}, the functions $(w,z,k,a)$ solve  the azimuthal Euler equations \eqref{eq:w:z:k:a}--\eqref{eq:wave-speeds}. 
\item  \label{item:5.6.iii}  The function $w$ has a jump discontinuity on $(\sc(t),t)_{t\in (0,\bar \eps]}$ which satisfies \eqref{eq:J:M:rough}.
\item  \label{item:5.6.iv} There exist $C^1$ smooth curves $\sc_1, \sc_2 \colon [0,\bar \eps] \to \TT$ which are the $\lambda_1$ and $\lambda_2$ characteristics through the point shock. They satisfy $\sc_1(0) = \sc_2(0) = 0$, $\sc_1(t) < \sc_2(t) < \sc(t)$ for all $t\in (0,\bar \eps]$, and we have the bounds $|\dot \sc_1(t) - \frac 13 \kappa| = \OO(t^{\frac 13})$, and $|\dot \sc_2(t) - \frac 23 \kappa| = \OO(t^{\frac 13})$.
\item  \label{item:5.6.v} The function $z$ has a jump discontinuity on $(\sc(t),t)_{t\in (0,\bar \eps]}$ which satisfies \eqref{eq:zl:on:shock:L:infinity}. Moreover, for every $t\in [0,\bar \eps]$ we have that $z(\theta,t) = 0 $ for $\theta \in \TT \setminus [\sc_1(t) ,\sc(t)]$.
\item  \label{item:5.6.vi}  The function $k$ has a jump discontinuity on $(\sc(t),t)_{t\in (0,\bar \eps]}$ which satisfies \eqref{eq:kl:on:shock:L:infinity}. Moreover, for every $t\in [0,\bar \eps]$ we have that $k(\theta,t) = 0 $ for $\theta \in \TT \setminus [\sc_2(t) ,\sc(t)]$.
\item  \label{item:5.6.vii} We have that $(w_-,w_+,z_-,k_-)$ satisfy the system of algebraic equations \eqref{pjump77}-\eqref{pjump7}, arising from the Rankine--Hugoniot conditions.
\end{enumerate}
\end{proposition}

The proof of Proposition~\ref{thm:curve:determines:all} is the content of Sections~\ref{sec:Burgers:explicit}--\ref{sec:construction:iteration}, and is summarized in Section~\ref{sec:kucf}.
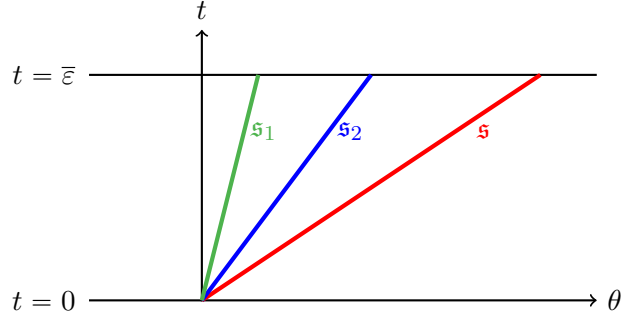
\begin{figure}[htb!]
\centering
\begin{tikzpicture}[scale=1.5]
    \draw [<->,thick] (0,2.4) node (yaxis) [above] {$t$}
        |- (3.5,0) node (xaxis) [right] {$\theta$};
    \draw[black,thick] (0,0)  -- (-1,0) ;
     \draw[black,thick] (-1,2)  -- (3.5,2) ;
    \draw[name path=A,red,ultra thick] (0,0) -- (3,2) ;
        \draw[red] (2.5,1.5) node { $\sc$}; 
    \draw[name path = B,blue,ultra thick] (0,0)  -- (1.5,2) ;
    \draw[blue] (1.32,1.5) node { $\sc_2$}; 
     \draw[green!40!gray, ultra thick] (0,0)  -- (.5,2) ;
      \draw[green!40!gray] (.55,1.5) node { $\sc_1$}; 
     \draw[black] (-1.4,0) node { $t=0$}; 
      \draw[black] (-1.4,2) node { $t=\bar \eps$}; 
\end{tikzpicture}

\vspace{-0.2cm}
\caption{\footnotesize The curves $\sc_1$, $\sc_2$, and $\sc$ discussed in Proposition~\ref{thm:curve:determines:all} all originate from the pre-shock}
\end{figure}

\subsection{Computing $w$ when $a=z=k=0$}
\label{sec:Burgers:explicit}
In light of \eqref{eq:the:norm} and \eqref{eq:the:X:T}, it is natural to treat $z$ and $k$ as a perturbation of $0$. As such, it convenient to first look at the evolution \eqref{xland-w} for $w$, in the case that $a=k=z=0$. In this case \eqref{xland-w} and the definition of $\lambda_3$ in \eqref{eq:wave-speeds} show that $w$ solves the 1d Burgers equation; to distinguish this solution from the true $w$, we denote it as $\wb$.

\begin{proposition}[\bf Burgers solution with a prescribed shock location]
\label{prop:Burgers}
Let   $w_0$ be  as described in \eqref{eq:u0:ass:quant}, and assume that  $\sc \colon[0,T_0] \to \TT$ satisfies \eqref{eq:sc:ass}. There exists $\bar \eps \in (0,T_0]$ and a function $\wb \colon \DD_{\bar \eps} \to \RR$ which solves
\begin{subequations}
\label{eq:basic:shit}
\begin{alignat}{2}
\p_t \wb + \wb \p_\theta \wb &= 0\,, \qquad && \mbox{in} \qquad \DD_{\bar \eps}\,, \\
\wb  &= w_0\,,  \qquad && \mbox{on} \qquad \TT \times \{0\}\,,
\end{alignat}
\end{subequations}
which is $C^2$ smooth in $\DD_{\bar \eps}$, and has a jump discontinuity across the curve $(\sc(t),t)_{t\in(0,\bar \eps]}$, with jump across $\sc$ and mean at $\sc$ bounded as
\begin{subequations}
\label{eq:the:burgers:jumps}
\begin{alignat}{2}
&\abs{\jump{\wb(t)}-2\bb^{\frac 32} t^{\frac 12}} \leq  t \,,
\qquad
&&\abs{\mean{\wb(t)} -\kappa} \leq \tfrac 13  {\mm^4} t 
\label{eq:basic:jump}
\,,\\
&\abs{\tfrac{d}{dt} \jump{\wb(t)}- \bb^{\frac 32} t^{-\frac 12} } \leq  {2 \mm^4} \,,
\qquad
&&\abs{\tfrac{d}{dt}\mean{\wb(t)} } \leq  {\mm^4}  \,,
\label{eq:dt:jump}\\
&\abs{\tfrac{d^2}{dt^2} \jump{\wb(t)} + \tfrac 12 \bb^{\frac 32} t^{-\frac 32} } \leq       2 \mm^4    t^{-1}   \,,
\qquad
&&\abs{\tfrac{d^2}{dt^2}\mean{\wb(t)}} \leq    \mm^4    t^{-1} 
\,.
\label{eq:dt:dt:jump}
\end{alignat}
\end{subequations}
\end{proposition}

In Proposition~\ref{prop:Burgers} we use the notation from Remark~\ref{def:jump:etc} and Definition~\ref{def:jump:etc}. Prior to the proof of Proposition~\ref{prop:Burgers}, it is convenient to establish an auxiliary result for the derivatives of $w_0$ (cf.~Lemma~\ref{lem:u0:u0'}), and a result (cf.~Lemma~\ref{lem:a=0:inversion}) which concerns the invertibility of the 
usual flow map for the Burgers equation: 
\begin{align}
\etab(x,t) = x + t w_0(x)\,, 
\label{eq:etab:def}
\end{align}
which is well-defined for every $x\in \TT$.\footnote{Here and throughout the remainder of the paper we shall denote the Eulerian variable by $\theta$, while for the corresponding Lagrangian label we use $x$.}  We first record a few estimates for $w_0$, which follow from \eqref{eq:u0:ass:quant}:
\begin{lemma}
\label{lem:u0:u0'}
There exists $\bar \eps \in (0,1]$ such that for every $t \in (0, \bar \eps ]$ we have 
\begin{subequations}
\label{eq:u0:interest:all}
\begin{alignat}{2}
& \sabs{w_0(x)} \leq \mm\,, \qquad 
&& x\in \TT\,, \label{eq:u0:interest} \\
&\sabs{w_0'(x)}\leq  \tfrac{2}{5} t^{-1}\,, \qquad 
&& \tfrac 45 (\bb t)^{\frac 32} \leq |x| \leq \pi \,, \label{eq:u0':interest}\\
&\sabs{w_0''(x)} \leq  \tfrac{1}{3} \bb^{- \frac 32} t^{-\frac 52}\,, \qquad 
&& \tfrac 45 (\bb t)^{\frac 32} \leq |x| \leq \pi \, \label{eq:u0'':interest} \\
&\sabs{w_0'''(x)} \leq  {2\mm (\bb t)^{-4}} \,, \qquad 
&& \tfrac 45 (\bb t)^{\frac 32} \leq |x| \leq \pi \, \label{eq:u0''':interest} \\
&\abs{\tfrac{w_0'(x)}{1+ t w_0'(x)}} \leq  \tfrac{2}{3}  t^{-1}\,, \qquad 
&& \tfrac 45 (\bb t)^{\frac 32} \leq |x| \leq \pi \, \label{eq:quotient:interest} 
\,.
\end{alignat}
\end{subequations}
\end{lemma}
\begin{proof}[Proof of Lemma~\ref{lem:u0:u0'}]
For simplicity, we only give the proof for  $x>0$.  
The bound \eqref{eq:u0:interest} follows directly from \eqref{eq:u0:ass:2} since \eqref{eq:u0:ass:2a} implies that $w_0$ is nonnegative. In order to prove \eqref{eq:u0':interest} we  use  assumption \eqref{eq:u0:ass:3}, which gives
\begin{align*}
 \sabs{w_0'(x)} t 
 &\leq \tfrac{1}{3}\bb \left(\tfrac{4}{5}\right)^{-\frac 23} (\bb t)^{-1} t + \tfrac{2}{3}  |\cc|\left(\tfrac{4}{5}\right)^{-\frac 13}  (\bb t)^{-\frac 12} t + \mm t  \leq \tfrac{1}{3}  \left(\tfrac{4}{5}\right)^{-\frac 23}   +  C t^{\frac 12} 
  \leq \tfrac{2}{5}  
\end{align*}
upon choosing $\bar \eps$ (and hence $t$) to be sufficiently small, in terms of $\kappa, \bb, \cc$, and $\mm$. The proof of \eqref{eq:u0'':interest} is similar to the one of \eqref{eq:u0':interest}, except that we appeal to assumption~\eqref{eq:u0:ass:4} and derive
\begin{align}
 \sabs{w_0''(x)} \bb^{\frac 32} t^{\frac 52}
 \leq \tfrac{2}{9} \left(\tfrac{4}{5}\right)^{-\frac 53} + C t^{\frac 12} 
 \leq \tfrac{1}{3}  
\end{align}
once $\bar \eps$ (and hence $t$) is small enough. The bound \eqref{eq:u0''':interest} immediately follows from~\eqref{eq:u0:ass:5} and \eqref{eq:b:m:ass}. Lastly, the estimate \eqref{eq:quotient:interest} is a direct consequence of \eqref{eq:u0':interest}.
\end{proof}

Second, we discuss the invertibility of  $\etab$:

\begin{lemma}[\bf Local inversion of the Burgers flow map]
\label{lem:a=0:inversion}
Let $w_0$ be as described in \eqref{eq:u0:ass:quant},  assume that  $\sc$ satisfies \eqref{eq:sc:ass} on $[0,T_0]$, and let $\etab$ be defined as in \eqref{eq:etab:def}. Then, there exists a sufficiently small $\bar \eps   \in (0,T_0]$, which only depends on $\kappa, \bb, \cc, \mm$,  such that for $t\in (0,\bar \eps]$ the following holds. 
 There exists a {\em largest} $\xbr = \xbr(t)>0$ and a {\em smallest} $\xbl = \xbl(t)<0$ such that
\begin{align}
 \sc(t) = \etab(\xbpm(t),t)  
 \label{eq:xpm:def}
 \end{align}
and moreover we have
 \begin{align}
\sabs{\xbpm(t) \mp (\bb t)^{\frac 32}} \leq  {\mm^4}   t^2 
\qquad \Rightarrow \qquad 
 \tfrac{4}{5}   (\bb t)^{\frac 32} < \abs{\xbpm(t)} < \tfrac{6}{5} (\bb  t)^{\frac 32}\,.
\label{eq:range:for:solutions}
\end{align}
We also define $\xbpm(0) = 0$. Note that $\xbpm(t) \in \Omega(t)$ for all $t\in [0,\bar \eps]$.
Moreover, defining the set of labels 
\begin{align*}
\Upsilon_{\mathsf B} (t) = \TT \setminus [\xbl(t),\xbr(t)]
\end{align*}
we have that the map $\etab(\cdot,t)  \colon \Upsilon_{\mathsf B} (t) \to \TT \setminus \{\sc(t)\}$ is a bijection
satisfying the bounds 
\begin{subequations}
\begin{align}
\abs{\p_x \etab(x,s) - 1} &\leq \tfrac{2}{5} \label{eq:px:eta:b} \\
\abs{ \etab(x,s) - \sc(s)} &\geq \tfrac 45 \bb^{\frac 32} t^{\frac 12} (t-s) \,.
\label{eq:no:trespassing}
\end{align}
\end{subequations}
for all $s\in [0,t)$ and $x \in \Upsilon_{\mathsf B} (t)$.
The  above estimate implies that the trajectory $\{ \eta(x,s) \}_{s\in[0,t]}$  can not intersect the shock curve prior to time $s=t$, for every $x\in \Upsilon_{\mathsf B} (t)$. Lastly, the inverse map $\etab^{-1} (\cdot,t) \colon \TT\setminus \{ \sc(t)\} \to \Upsilon_{\mathsf B}(t)$ satisfies the estimates
\begin{subequations}
\label{eq:tropic:thunder}
\begin{align}
\tfrac{5}{7} &\leq  \p_\theta \etab^{-1}(\theta,t)   \leq \tfrac{5}{3} \label{eq:tropic:1} \\
\tfrac 45 (\bb t)^{\frac 32} + \tfrac 12 \abs{\theta - \sc(t)} &\leq \abs{\etab^{-1}(\theta,t)} \leq  
\tfrac 65 (\bb t)^{\frac 32} + 2 \abs{\theta - \sc(t)}\label{eq:tropic:2}  
\end{align}
\end{subequations}
for all $(\theta,t) \in \DD_{\bar \eps}$. 
\end{lemma}
 
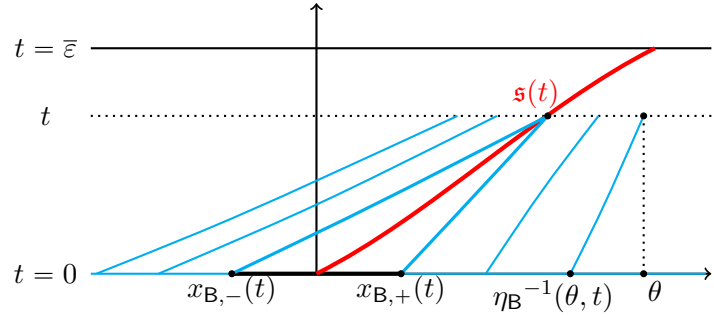
\begin{figure}[htb!]
\centering
\begin{tikzpicture}[scale=1.5]
\draw [<->,  thick] (0,2.4) node (yaxis) [above] {} |- (3.5,0) node (xaxis) [right] {};
\draw[black,ultra thick] (0.75,0)  -- (-0.75,0) ;
\draw[black,thick] (-2,2)  -- (3.5,2) ;
\draw[name path=A,red,ultra thick] (0,0) .. controls (1,0.5) and (2,1.5) .. (3,2);
\draw[red] (1.95,1.6) node { $\sc(t)$}; 
\draw[black] (-2.4,1.4) node { $t$}; 
\draw[black] (-2.4,0) node { $t=0$}; 
\draw[black] (-2.4,2) node { $t=\bar \eps$}; 
\draw[black,dotted,thick] (2.02,1.4)  -- (-2,1.4) ;
\draw[cyan,thick] (-0.75,0)  -- (-2,0) ;
\draw[cyan,thick] (0.75,0)  -- (3.45,0) ;
\draw[black,dotted,thick] (3.5,1.4)  -- (2.07,1.4) ;
\draw[name path=B,cyan, very thick] (-.75,0) .. controls (0.5,0.6) .. (2.05,1.4);
\draw[name path=C,cyan, very thick] (.75,0) .. controls (1.5,0.8) .. (2.05,1.4);
\draw[black] (-0.75,-0.15) node { $\xbl(t)$}; 
\draw[black] (0.75,-0.15) node { $\xbr(t)$}; 
\draw[cyan, thick] (1.5,0) .. controls (2,0.75) .. (2.5,1.4);
\draw[cyan, thick] (2.25,0) .. controls (2.55,0.6) .. (2.9,1.4);
\draw[cyan, thick] (-1.4,0) .. controls (0,0.6) .. (1.6,1.4);
\draw[cyan, thick] (-1.95,0) .. controls (-0.7,0.5) .. (1.25,1.4);
\filldraw[black] (-.75,0) circle (0.75pt);
\filldraw[black] (.75,0) circle (0.75pt);
\filldraw[black] (2.05,1.4) circle (0.75pt);
\draw[black,dotted,thick] (2.9,1.4)  -- (2.9,0) ;
\draw[black] (3,-0.15) node { $\theta$}; 
\draw[black] (2.1,-0.2) node { $\etab^{-1}(\theta,t)$}; 
\filldraw[black] (2.9,0) circle (0.75pt);
\filldraw[black] (2.25,0) circle (0.75pt);
\filldraw[black] (2.9,1.4) circle (0.75pt);
\end{tikzpicture}

\vspace{-0.2cm}
\caption{\footnotesize Several Lagrangian paths $\{\etab(x,s)\}_{s\in[0,t]}$ are represented by cyan paths. The extremal points $\xbpm(t)$ are the two labels which are colliding into the shock curve precisely at time $t$. All the labels in between them have collided with the shock curve at some time $s\in [0,t)$.}
\end{figure}

\begin{proof}[Proof of Lemma~\ref{lem:a=0:inversion}]

It is convenient to denote
\begin{align}
g_0(x) =  w_0(x) - \kappa + \bb x^{\frac 13}  - \cc x^{\frac 23}
\label{eq:g0:def}
\end{align}
so that in view of \eqref{eq:u0:ass:quant} we have that $|g_0(x) | \leq  \mm |x|$ and $|g_0'(x)| \leq  \mm$. 
For $t>0$ we let 
\begin{align}
\tau = (\bb t)^{\frac 12}, \qquad y = x^{\frac 13} \tau^{-1}, \qquad \zeta = (\sc(t) - \kappa t) \tau^{-3} \,. 
\label{eq:tau:z:zeta:def}
\end{align}
Note that the condition $\sc(t) \in \Sigma(t)$ in \eqref{eq:sc:ass} together with \eqref{eq:b:m:ass} imply that $|\zeta| \leq \bb^{-2} \mm^4 \tau \ll 1$, an in particular $|\zeta| \leq \frac{1}{10}$. 
With this notation, for any $t>0$ the equation \eqref{eq:xpm:def} is equivalent to
\begin{align*}
&\tau^3 \zeta + \kappa t = \tau^3 y^3 + t \left(\kappa - \bb  \tau y + \cc \tau^2 y^2 +  g_0( \tau^3 y^3)   \right)  
\,.
\end{align*}
After collecting terms, and dividing by $\tau^3$, we obtain that the above equality is equivalent to 
\begin{align}
0 =  - \zeta + y^3  -  y + \underbrace{\cc \bb^{-1} \tau y^2 + \tau^{-1} \bb^{-1} g_0(\tau^3 y^3)}_{=: G(y,\tau)}\,.
\label{eq:Lite}
\end{align}
In view of  the aforementioned properties of $g_0$, we have that for all $|y|\leq 10$ and all $0< \tau \leq \bar \eps$, with $\bar \eps$ sufficiently small in terms of $\kappa, \bb, \cc, \mm$, we have that 
\begin{subequations}
\label{eq:G:litter:box}
\begin{align}
\abs{G(y,\tau) - \cc \bb^{-1} \tau y^2} &\leq C \tau^2  \label{eq:G:litter:box:a}\\
\abs{\p_\tau G(y,\tau) - \cc \bb^{-1}   y^2} &\leq C \tau \label{eq:G:litter:box:b}\\
\abs{\p_y G(y,\tau)- 2 \cc \bb^{-1} \tau y } &\leq C \tau^2 \label{eq:G:litter:box:c}
\end{align}
\end{subequations}
where $C>0$ only depends on $\kappa, \bb, \cc$, and $\mm$.

Returning to \eqref{eq:Lite}, we next claim that for every fixed $\zeta \in [-\frac{1}{10} ,\frac{1}{10}]$ and any $\tau$ sufficiently small, there exists a  unique {\em most negative} root $y_{-} = y_-(\zeta,\tau)$ and a unique {\em most positive} root $y_+= y_+(\zeta,\tau)$ of the implicit equation
\begin{align}
y^3 - y + G(y,\tau) = \zeta \,. 
\label{eq:Lite:litter}
\end{align}
The key observation is that in view of \eqref{eq:G:litter:box:a}, when when $\tau=0$,  the equation in the above display becomes  $\zeta = y^3 - y$. 
For every $\zeta \in (- \frac{2}{3\sqrt{3}} , \frac{2}{3\sqrt{3}} ) \supset [-\frac{1}{10} ,\frac{1}{10}]$ we introduce two functions $\Root_+(\zeta)$ and $\Root_-(\zeta)$ which are the largest (positive) root and respectively the smallest (negative) root of the equation
\begin{align}
\zeta = \Root^3 - \Root \,.   \label{Z-cubic-eqn}
\end{align}
The power series of these functions is given by 
\begin{align}
\Root_\pm(\zeta) = \pm 1 +\tfrac 1 2 \zeta \mp \tfrac{3}{8} \zeta^2 + \tfrac{1}{2} \zeta^3 \mp \tfrac{105 }{128} \zeta^4 + 	\tfrac{3}{2} \zeta^5 +  \OO(|\zeta|^6)
\label{eq:z:pm:power:series}
\end{align}
and is valid for $\abs{\zeta} \ll 1$. In particular, we have 
\begin{align}
\Root_+(\zeta) + \Root_-(\zeta) =  \zeta + \zeta^3  + 3 \zeta^5 + \OO(|\zeta|^7)
\,.
\label{eq:z:pm:power:series:2}
\end{align}
For later purposes, it is also convenient to note here that 
\begin{align}
\abs{\Root_+(\zeta) + \Root_-(\zeta) -  \zeta} \leq \tfrac{6 }{5} \zeta^3
\qquad \mbox{and} \qquad
\abs{\Root_+(\zeta) - \Root_-(\zeta) -  2} \leq    \zeta^2
\label{eq:Taylor:magic} 
\end{align}
for all $|\zeta| \leq \tfrac 15$. 
With this notation, we have thus obtained the desired roots of \eqref{eq:Lite:litter} when $\tau =0$, namely
\begin{align*}
\Root_\pm^3(\zeta) - \Root_\pm(\zeta) + G(\Root_{\pm}(\zeta),0) = \zeta \,. 
\end{align*}
The proof is then completed by  an application of the implicit function theorem. This is possible since 
\begin{align*}
\p_y\left(y^3 - y + G(y,\tau)\right)|_{(y,\tau)= (\Root_{\pm},0)} = 3 \Root_\pm^2 - 1 \neq 0\,.
\end{align*}
In fact, for every $|\zeta| \leq \frac{1}{10}$, one may verify that $\frac 32 \leq 3 \Root_\pm(\zeta)^2 - 1 \leq \frac 52$, since $\Root_{\pm}$ are explicit functions. The implicit function theorem guarantees the existence of an $\bar \eps>0$, such that if $\tau \in (0, (\bb \bar \eps)^{\frac 12}]$ and $|\zeta| \leq \frac{1}{10}$, the equation \eqref{eq:Lite:litter} has a most negative root $y_-(\zeta,\tau)$ which is $\OO(\tau)$-close to $\Root_{-}(\zeta)$, and a most positive root $y_+(\zeta,\tau)$, which is $\OO(\tau)$-close to $\Root_+(\zeta)$. Upon unpacking the definitions in \eqref{eq:tau:z:zeta:def}, we have thus identified 
\begin{align}
\xbpm^{\frac 13} (t) = (\bb t)^{\frac 12} y_{\pm} \left(\frac{\sc(t) - \kappa t}{(\bb t)^{\frac 32}}, (\bb t)^{\frac 12}\right)
\,,
\label{eq:x:litter:box}
\end{align}
for all $t\in (0,\bar \eps]$, which solves \eqref{eq:xpm:def}.

Note however that $|\zeta| \leq \bb^{-2} \mm^4 \tau$, and that $\tau\leq (\bb \bar \eps)^{\frac 12}$ is taken to be small. In this $\tau$-dependent range for $\zeta$ we may obtain a sharper estimate than the $\abs{y_{\pm}(\zeta,\tau) - \Root_{\pm}(\zeta)} \leq C \tau$ claimed above. Indeed, since the bounds \eqref{eq:G:litter:box:b}--\eqref{eq:G:litter:box:c} are available, from the Taylor theorem with remainder applied to \eqref{eq:Lite:litter},  we may deduce that 
\begin{align*}
\abs{y_\pm(\zeta,\tau) - \Root_{\pm}(\zeta) + \tau \cc \bb^{-1}  \frac{\Root_{\pm}(\zeta)}{3 \Root_{\pm}^2(\zeta) - 1 }} \leq C \tau^2
\end{align*}
if $\bar \eps$ is  sufficiently small, for a constant $C = C (\kappa,\bb,\cc,\mm)>0$. Taking into account the power series expansion of $\Root_{\pm}$ in \eqref{eq:z:pm:power:series}, and $\bar \eps$ to be sufficiently small (hence $\tau$ sufficiently small), we deduce that 
\begin{align}
\abs{y_\pm(\zeta,\tau) \mp 1 - \tfrac 12 \zeta \pm \tfrac{\cc}{2\bb} \tau} \leq C \tau^2\,,
\qquad \mbox{for all} \qquad 
|\zeta| \leq \bb^{-2} \mm^4
\quad \mbox{and} \quad \tau \leq (\bb \bar \eps)^{\frac 12} \,.
\label{eq:z:litter:box} 
\end{align}
In particular, keeping in mind \eqref{eq:tau:z:zeta:def} and \eqref{eq:x:litter:box}, we deduce from \eqref{eq:z:litter:box} the   estimate
\begin{align}
\abs{\xbpm(t)^{\frac 13}  \mp (\bb t)^{\frac 12}  - \frac{\sc(t) - \kappa t}{2 \bb t}  \pm \frac{\cc t }{2} }
\leq C t^{\frac 32} \,, 
\label{eq:sharp:range:for:solutions}
\end{align}
for all $t\in (0,\bar \eps]$, where $C = C(\kappa,\bb,\cc, \mm)>0$ is a computable constant.
The bound  \eqref{eq:range:for:solutions} is an immediate consequence of  \eqref{eq:sharp:range:for:solutions}, the working assumptions \eqref{eq:b:m:ass} and  \eqref{eq:sc:ass}, upon taking $\bar \eps$ to be sufficiently small.  

The bound \eqref{eq:px:eta:b} is a direct consequence of \eqref{eq:u0':interest}, \eqref{eq:range:for:solutions}, and the fact that by \eqref{eq:etab:def} we have $\p_x \etab(x,s) - 1 = s w_0'(x)$. Therefore, the map $\etab(\cdot,t)$ is a strictly increasing function on the label $x \in \TT$, thus being injective from $\Upsilon_{\mathsf B}(t) \mapsto \TT \setminus \{ \sc(t)\}$. Surjectivity follows from the intermediate value theorem, and fact that by \eqref{eq:xpm:def} we have $\lim_{x \to \xbl(t)^-} \etab(x,t) = \sc(t) = \lim_{x\to \xbr(t)^+} \etab(x,t)$. 
In order to show that for every $x \in \Upsilon_{\mathsf B}(t)$ the trajectory $\etab(x,\cdot)$ does not meet the shock curve prior to time $t$, by the monotonicity property of $\etab$ in the $x$ variable, we only need to show that $\etab(\xbl(t),s) < \sc(s)$ and that $\etab(\xbr(t),s) > \sc(s)$. These two statements are established in the same way, so we only give the proof for the label $\xbl(t)$. By appealing to \eqref{eq:xpm:def}, the $\dot{\sc}$ assumption in \eqref{eq:sc:ass}, the $w_0$ assumption in \eqref{eq:u0:ass:quant}, and the previously established estimate \eqref{eq:sharp:range:for:solutions},  we have that 
\begin{align*}
\sc(s) - \etab(\xbl(t),s) 
&=  - \int_s^t \left( \dot{\sc}(\tau) - (\p_t \etab)(\xbl(t),\tau) \right) d\tau \notag\\
&=  \int_s^t \left(w_0(\xbl(t)) - \kappa  \right) d\tau - \int_s^t \left(\dot{\sc}(\tau) - \kappa \right) d\tau 
\notag\\
&\geq \left( - \bb \xbl^{\frac 13}(t) -2 |\cc| \bb t \right) (t-s) - \tfrac{1}{2} \mm^4 (t^{2} - s^{2})   
\notag\\
&\geq \tfrac 45 \bb^{\frac 32} t^{\frac 12} (t-s)
\end{align*}
for any $s\in [0,t)$, with $t \leq \bar \eps$ which is sufficiently small.

The proof is concluded once we establish \eqref{eq:tropic:thunder}. The bound \eqref{eq:tropic:1} is an immediate consequence of \eqref{eq:px:eta:b} and the inverse function theorem. For the proof of \eqref{eq:tropic:2}, let us first consider a point $\theta$ which is to the left of $\sc(t)$. Then, by the mean value theorem and \eqref{eq:xpm:def}, we have that 
\begin{align*}
\etab^{-1}(\theta,t) - \xbl(t) =  \etab^{-1}(\theta,t) - \etab^{-1}(\sc(t),t)) = (\theta-\sc(t)) (\p_\theta \etab^{-1})(\bar \theta,t)
\end{align*}
for some $\bar \theta \in (y,\sc(t))$. The above identity, combined with \eqref{eq:tropic:1} and the first inequality in \eqref{eq:range:for:solutions} implies \eqref{eq:tropic:2}, upon taking $\bar \eps$ sufficiently small. The proof in the case that $y$ is to the right of $\sc(t)$ is identical.
\end{proof}

Next, we discuss the solution $\wb$ to \eqref{eq:basic:shit}  and its properties.
\begin{proof}[Proof of Proposition~\ref{prop:Burgers}]
By Lemma~\ref{lem:a=0:inversion}, for all $(\theta,t) \in \DD_{\bar \eps}$ we may define
\begin{align}
\wb(\theta,t) = w_0(\etab^{-1}(\theta,t)) \label{eq:wb:def}\,.
\end{align}
By the of construction $\etab$ and the properties of $w_0$, the above defined $\wb$ is $C^2$ smooth in $\DD_{\bar \eps}$ and solves \eqref{eq:basic:shit} in this region. Indeed, differentiating the relation $\wb(\etab(x,t),t) = w_0(x)$ and using the definition of $\etab$ we  have the identities
\begin{subequations}
\label{windy-day3}
\begin{align}
 \p_\theta  \wb(\theta,t) &= \frac{ w_0'(\etab^{-1}(\theta,t))}{1 + t w_0'(\etab^{-1}(\theta,t))}  
 \label{eq:px:wb:def}\\
 \p_{\theta}^2 \wb(\theta,t) &= \frac{w_0''(\etab^{-1}(\theta,t))}{(1 + t w_0'(\etab^{-1}(\theta,t)))^3}
  \label{eq:pxx:wb:def}
\end{align}
\end{subequations}
for all $\theta \in \TT \setminus \{\sc(t)\}$. In particular, combining \eqref{eq:px:wb:def} with \eqref{eq:tropic:2} and \eqref{eq:u0:ass:quant}, gives that 
\begin{subequations}
\label{thegoodstuff}
\begin{align}
 \abs{\p_\theta  \wb(\theta,t)} &\leq \tfrac 45 \bb  (  (\bb t)^3 +  \abs{\theta - \sc(t)}^2 )^{-\frac 13}  \label{thegoodstuff1}\\
 \abs{\p_{\theta}^2  \wb(y,t)} &\leq 2 \bb  (  (\bb t)^3 +  \abs{\theta - \sc(t)}^2 )^{-\frac 56}   \label{thegoodstuff2}
\end{align}
\end{subequations}
for all $(\theta,t) \in \DD_{\bar \eps}$ such that $\abs{\theta - \sc(t)} \leq \bar \eps^{\frac 12}$, as soon as $\bar \eps$ is sufficiently small.

Next, we we discuss the mean and the jump of $\wb$ at the shock curve. We have that 
\begin{align*}
\jump{\wb(t)} 
&= w_0(\xbl(t)) -  w_0(\xbr(t)) \notag\\
&= \left( \xbr^{\frac 13}(t) - \xbl^{\frac 13}(t) \right) \left( \bb - \cc \xbr^{\frac 13}(t) - \cc \xbl^{\frac 13}(t)\right) + g_0(\xbl(t)) - g_0(\xbr(t))
\end{align*}
where we recall the notation from \eqref{eq:g0:def}.  Using \eqref{eq:sharp:range:for:solutions},   \eqref{eq:u0:ass}, and \eqref{eq:b:m:ass},   we deduce that
\begin{align*}
\abs{\jump{\wb(t)} - 2  \bb^{\frac 32} t^{\frac 12}} \leq 8 \bb |\cc| t \leq t
\end{align*}
upon choosing $\bar \eps$ to be sufficiently small with respect to $\kappa, \bb, \cc$, and $\mm$.
This proves the first bound in \eqref{eq:basic:jump}.
Similarly, 
\begin{align*}
\mean{\wb(t)} &= \tfrac 12 \left( w_0(\xbl(t)) +  w_0(\xbr(t)) \right) \notag\\
&= \kappa - \tfrac 12 \bb \left( \xbl^{\frac 13}(t) + \xbr^{\frac 13}(t)\right) + \tfrac 12 \cc \left( \xbl^{\frac 23}(t) + \xbr^{\frac 23}(t)\right) + \tfrac 12 \left( g_0(\xbl(t)) - g_0(\xbr(t)) \right) \,.
\end{align*}
From \eqref{eq:sharp:range:for:solutions}, \eqref{eq:range:for:solutions}, and \eqref{eq:u0:ass} we deduce that 
\begin{align*}
\abs{\mean{\wb(t)} - \kappa + \frac{\sc(t)-\kappa t}{2t} } \leq C t^{\frac 32} \,.
\end{align*}
The second inequality in \eqref{eq:basic:jump} now follows from \eqref{eq:sc:ass}.

Appealing to the definitions \eqref{eq:xpm:def}, \eqref{eq:etab:def}, and \eqref{eq:wb:def}, we arrive at 
\begin{align*}
\tfrac{d}{dt} \left( \wb(\sc(t)^{\pm},t) \right) 
=\tfrac{d}{dt} \left(w_0(\xbpm(t)) \right) 
= w_0'(\xbpm(t)) \tfrac{d}{dt} \xbpm(t)
= w_0'(\xbpm(t)) \frac{\dot{\sc}(t) - w_0(\xbpm(t))}{1 + t w_0'(\xbpm(t))}
\,.
\end{align*}
Therefore, using \eqref{eq:u0:ass}, \eqref{eq:quotient:interest}, the asymptotic description \eqref{eq:sharp:range:for:solutions} for $\xbpm(t)$,  and the assumption on $\dot \sc$ from \eqref{eq:sc:ass}, after a tedious computation we obtain  
\begin{align*}
\abs{\tfrac{d}{dt} \left( \wb(\sc(t)^{\pm},t) \right) \pm \tfrac 12 \bb^{\frac 32} t^{-\frac 12}}
&\leq  \abs{\bb \xbpm(t)^{\frac 13}  \frac{w_0'(\xbpm(t))}{1 + t w_0'(\xbpm(t))}
  \mp \tfrac 12 \bb^{\frac 32} t^{-\frac 12}} \notag\\
&\qquad + \frac{t |w_0'(\xbpm(t))| }{1 + t w_0'(\xbpm(t))} \left( \mm^4 + \frac{\kappa - w_0(\xbpm(t))+ \bb \xbpm(t)^{\frac 13}}{t} \right) \notag\\
&\leq  \tfrac 56 {\mm^4} +  \bb  |\cc| + C t^{\frac 12}  \leq  {\mm^4} \,.
\end{align*}
From the above estimate, it is clear that \eqref{eq:dt:jump} follows. Differentiating once more, we obtain 
\begin{align*}
\tfrac{d^2}{dt^2} \left( \wb(\sc(t)^{\pm},t) \right)    
&=  \ddot{\sc}(t) \frac{w_0'(\xbpm(t)) }{1 + t w_0'(\xbpm(t))}  +  w_0''(\xbpm(t)) \frac{(\dot{\sc}(t) - w_0(\xbpm(t)))^2}{(1 + t w_0'(\xbpm(t)))^3} \notag\\
&\qquad - \frac{2 (w_0'(\xbpm(t)))^2(\dot{\sc}(t) - w_0(\xbpm(t)))}{(1 + t w_0'(\xbpm(t)))^2} 
\end{align*}
and therefore, after an even more tedious computation, we arrive at 
\begin{align*}
\abs{\tfrac{d^2}{dt^2} \left( \wb(\sc(t)^{\pm},t) \right) \mp \tfrac 14 \bb^{\frac 32} t^{-\frac 32}}
&\leq\left( \tfrac 13  {\mm^4}  + 3    \bb |\cc| \right) t^{-1}  +C t^{-\frac 12}
\leq   { \mm^4}   t^{-1} \,.
\end{align*}
The claim \eqref{eq:dt:dt:jump} now follows, thereby completing the proof of the proposition.
\end{proof}

\subsubsection{Lagrangian trajectories for velocity fields that are close to $\wb$}

For future purposes, see Section~\ref{sec:transport}, at this stage it is convenient to consider velocities $\lambda_3 \colon \DD_{\bar \eps} \to \RR$ which are close to the $\wb$ we have constructed in Proposition~\ref{prop:Burgers}, in the sense that $\lambda_3 \in C^1_{\theta,t} (\DD_{\bar \eps})$, and we have the pointwise bounds
\begin{subequations}
\label{eq:lambda:3:abstract}
\begin{align}
\abs{\lambda_3(\theta,t) - \wb(\theta,t)} &\leq \Rsf_1 t + C t^{\frac 32}
\label{eq:lambda:3:abstract:a}\\
\abs{\p_\theta \lambda_3(\theta,t) - \p_\theta \wb(\theta,t)} &\leq \Rsf_2 ( (\bb t)^3 + (\theta-\sc(t))^2)^{-\frac 16} + C t^{\frac 12}
\label{eq:lambda:3:abstract:b}
\end{align}
\end{subequations}
for all $(\theta,t) \in \DD_{\bar \eps}$, for positive constants $\Rsf_1, \Rsf_2, C$ which only depend on $\kappa, \bb,\cc$, and $\mm$; see \eqref{eq:R:def} for the values of $\Rsf_1, \Rsf_2$ which are used in the proof, namely $\Rsf_1 = \Rsf_2 = \mm^3$. 

Note that in view of \eqref{eq:wb:def} and \eqref{thegoodstuff1}, assumptions \eqref{eq:lambda:3:abstract} imply that $\lambda_3$ is $C^1$ smooth on the complement of the shock curve. In particular, this means that for every label $x \in \TT \setminus \{0\}$, we are guaranteed the short time ($x$-dependent time) unique solvability of the ODE
\begin{align}
 \p_t \eta(x,t) = \lambda_3(\eta(x,t),t) \,, \qquad  \eta(x,0)  = x \,.
 \label{eq:dt:eta:def}
\end{align}
In view of the assumed regularity of $\lambda_3$, for a given label $x$ the path $\eta(x,t)$ can be continued on a maximal time interval $[0,T_x)$, where the stopping time $T_x$ is defined as 
\begin{align}
 T_x := 
\min \bigl\{ \bar \eps , \sup \{ t \in [0,\bar \eps] \colon |\eta(x,t) - \sc(t)| > 0 \} \bigr\}
\label{eq:Tx:def}
\,.
\end{align}
That is, if the trajectory $\eta(x,\cdot)$ intersects the shock curve prior to time $\bar \eps$, then we record this stopping time in $T_x$, and in this case we have $\eta(x,T_x) = \sc(T_x)$.  Note that since $\sc \in C^1$, and since $\lambda_3$ is  $C^1$ smooth on the complement of the shock curve, the stopping time $T_x$ is continuous in $x$.

Next, for every $t \in (0,\bar \eps]$, in analogy to \eqref{eq:xpm:def}, we wish to define in a unique way two {\em extremal} labels $x_\pm(t)$ with the property that 
\begin{align}
\sc(t) = \eta(x_{\pm}(t), t)\,. 
\label{eq:x:pm:def}
\end{align}
By \eqref{eq:Tx:def} we have that the above definition is equivalent to $T_{x_{\pm}(t)} = t$, which then motivates
\begin{align}
x_-(t) = \inf \{ x \in [-\pi,0) \colon T_x \geq t \}\,,
\quad 
x_+(t) =  \sup \{ x \in (0,\pi]  \colon T_x \geq t \}\,,
\quad 
\Upsilon(t) = \TT \setminus [x_-(t),x_+(t)] \,.
\label{eq:x:pm:Upsilon}
\end{align}
By the continuity of $T_x$ in $x$, the above $\inf/\sup$ are in fact $\min/\max$.
Moreover, for every $x \in \Upsilon(t)$, we know that $T_x \geq t$. One of our goals will be to show that $\eta(\cdot,t) \colon \Upsilon(t) \to \TT \setminus \{\sc(t)\}$ is a bijection, for every $t \in (0,\bar \eps]$. 

As mentioned above, $T_x \in (0,\bar \eps]$ if $x\neq 0$. Now for fixed $x$ and   $t\in [0,T_x)$, by Lemma~\ref{lem:a=0:inversion} we may define 
\begin{align}
q(t) = \etab^{-1}(\eta(x,t),t) \,,
\end{align}
and note that $q(t) \in \Upsilon_\bb(t)$ and that $q(0) = x$. Since $\etab^{-1}$ solves the transport equation with speed $\wb$, and $\eta$ solves \eqref{eq:dt:eta:def}, we have that 
\begin{align*}
\tfrac{d}{dt} q = (\p_t \etab^{-1}) \circ \eta + (\p_\theta \etab^{-1}) \circ \eta \p_t \eta 
=(\lambda_3 - \wb) \circ \eta (\p_\theta \etab^{-1}) \circ \eta \,.
\end{align*}
Thus, by also appealing to \eqref{eq:lambda:3:abstract:a} and \eqref{eq:tropic:1}, we have that 
\begin{align}
\abs{\etab^{-1}(\eta(x,t),t) - x} = \abs{q(t) - q(0)} \leq  \Rsf_1 t^2
\label{eq:AC:DC:ride:on}
\end{align}
whenever $t< T_x$, upon taking $\bar \eps$ to be sufficiently small. By \eqref{eq:x:pm:Upsilon}, we note that \eqref{eq:AC:DC:ride:on} in particular holds for all $t\in (0,\bar \eps]$, and all $x\in \Upsilon(t)$. Note that from \eqref{eq:xpm:def}, \eqref{eq:x:pm:def}, \eqref{eq:AC:DC:ride:on}, and continuity, we have that 
\begin{align*}
\abs{x_\pm(t) - \xbpm(t)} = \abs{x_\pm(t) - \etab^{-1}(\eta(x_\pm(t),t))} \leq \Rsf_1 t^2 
\end{align*}
for all $t\in (0,\bar \eps]$, and thus similarly to \eqref{eq:range:for:solutions} we have that 
 \begin{align}
\sabs{x_\pm(t) \mp (\bb t)^{\frac 32}} \leq  t^2 (  {\mm^4} + \Rsf_1)
\qquad \Rightarrow \qquad 
 \tfrac{4}{5}   (\bb t)^{\frac 32} < \abs{x_{\pm}(t)} < \tfrac{6}{5} (\bb  t)^{\frac 32}\,.
\label{eq:range:for:solutions:new}
\end{align}
upon taking $\bar \eps$ to be sufficiently small.

 If $T_x < \bar \eps$, and $t \in [0,T_x)$, the bound \eqref{eq:AC:DC:ride:on} and the identity \eqref{eq:px:wb:def} allow us to estimate 
\begin{align*}
\int_0^t \abs{\p_\theta \wb(\eta(x,s),s)}  ds 
&= \int_0^t \frac{ \abs{w_0'(q(s))}}{1 + s w_0'(q(s))} ds \notag\\
&\leq \int_0^t \frac{ \abs{w_0'(x)}}{1 + s w_0'(x)} ds   +  \Rsf_1 \int_0^t s^{2} \sup_{|\bar x - x| \leq  \Rsf_1 s^{2}} \frac{\abs{w_0''(\bar x)}}{(1+ s w_0'(\bar x))^2} ds \notag\\
&\leq \frac{|w_0'(x)|}{w_0'(x)} \log  \left( 1+ t w_0'(x)  \right)   + \tfrac{8}{5}\bb^{-\frac 32} \Rsf_1 \int_0^t s^{-\frac 12} ds \,.
\end{align*}
At this stage we recall that the values of $x$ that we are interested in satisfy $|x| \geq (\bb t)^{\frac 32} - t^2 (  {\mm^4} + \Rsf_1) \geq \frac{9}{10} (\bb t)^{\frac 32}$. 
We distinguish two cases: $\frac{9}{10} (\bb t)^{\frac 32} \leq |x| \leq  \bb^{\frac 32} t$, and $\bb^{\frac 32} t \leq |x| \leq \pi$.  Using assumption \eqref{eq:u0:ass:3}, in the first case we deduce that 
$$
0 > t w_0'(x)>  - \tfrac 13 \bb t x^{-\frac 23} ( 1 + 3 |\cc| t^{\frac 13} ) >  - \tfrac 13  (\tfrac{9}{10})^{-\frac 23} ( 1 + 3 |\cc| t^{\frac 13} )   > - \tfrac{7}{19}  \,.
$$
In the other case, we we use that $t \leq \bar \eps \ll 1$, and thus 
$$
t |w_0'(x)| \leq \tfrac{1}{3} \bb t^{\frac 13} + \tfrac{2 }{3} |\cc| \bb^{-\frac 12} t^{\frac 23} + \mm t \leq  \bar \eps^{\frac 13}  \,.
$$
From the above three inequalities, and the fact that $\sgn(r) \log (1 + r) \leq \log(\frac{19}{12})$ for all $r \in (-\frac{7}{19}, \bar \eps^{\frac 13})$, we deduce that 
\begin{align}
\int_0^t \abs{\p_\theta \wb(\eta(x,s),s)}  ds 
\leq \log(\tfrac{19}{12}) + \tfrac{16}{5}\bb^{-\frac 32}\Rsf_1 t^{\frac 12}  
\leq \tfrac{19}{40}  \,,
\label{eq:AC:DC:ride:on:and:on}
\end{align}
since $t\leq  \bar \eps \ll 1$. As before, we note in particular that \eqref{eq:AC:DC:ride:on:and:on}  holds for all $t\in (0,\bar \eps]$, and all $x\in \Upsilon(t)$. We note that using \eqref{eq:pxx:wb:def}, \eqref{eq:AC:DC:ride:on}, and \eqref{eq:u0'':interest}, in addition to \eqref{eq:AC:DC:ride:on:and:on} we have 
\begin{align}
\int_0^t \abs{\p_{\theta}^2 \wb(\eta(x,s),s)}  ds 
& \leq \int_0^t \frac{|w_0''(\etab^{-1}(\eta(x,s),s))|}{( 1 + t w_0'(\etab^{-1}(\eta(x,s),s)))^3} ds
\leq 3 (\bb t)^{-\frac 32} 
\label{eq:AC:DC:ride:on:and:on:on}
\end{align}
whenever $x \in \Upsilon(t)$.  Here we have used that $|\etab^{-1}(\eta(x,s),s)| \geq |x| - \Rsf_1 s^2 \geq \tfrac 45 (\bb t)^{\frac 32}$ for $s\leq t \leq \bar \eps$.

With \eqref{eq:AC:DC:ride:on:and:on} in hand, and appealing also to \eqref{eq:lambda:3:abstract:b}, for every $x\in \Upsilon(t)$ we may now have
\begin{align}
\p_x \eta(x,t) 
&=  \exp\left( \int_0^t (\p_\theta \wb)(\eta(x,s),s) ds\right) \exp\left( \int_0^t (\p_\theta \lambda_3 - \p_\theta \wb)(\eta(x,s),s) ds\right)  \,,
\label{eq:useless:piece:of:crap}
\end{align}
and thus
\begin{align}
\tfrac 12 \leq \exp(-\tfrac 12 - 4 \Rsf_2  \bb^{-\frac 12} t^{\frac 12}) \leq  \p_x \eta(x,t) \leq \exp(\tfrac 12 + 4 \Rsf_2   \bb^{-\frac 12}  t^{\frac 12}) \leq \tfrac 74
\label{eq:px:eta}
\end{align}
since $\bar \eps$ is sufficiently small with respect to $\kappa, \bb, \cc$, and $\mm$. This shows that the map $\eta(\cdot,t)$ is  strictly monotone (thus injective) on either side of the shock curve; combined with \eqref{eq:x:pm:def} and the intermediate function theorem (ensuring surjectivity), we obtain that $\eta(\cdot,t) \colon \Upsilon(t) \to \TT \setminus \{\sc(t)\}$ is a bijection, as claimed earlier. Moreover, \eqref{eq:px:eta} shows that for every $x\in \Upsilon(t)$, the curve $\eta(x,s)$ does not intersect the shock curve prior to time $t$; in fact, by the monotonicity of $\eta$ we have that $\abs{\sc(s) - \eta(x,s)} \geq \abs{\sc(s) - \eta(x_\pm(t),s)}$, and analogously to \eqref{eq:no:trespassing}, using \eqref{eq:range:for:solutions:new} we have that 
\begin{align}
\sc(s) - \eta(x_-(t),s) 
&=  - \int_s^t   \dot{\sc}(\tau) - \lambda_3(\eta(x_-(t),\tau),\tau)  d\tau \notag\\
&=  \int_s^t \left(w_0(\etab^{-1}(\eta(x_-(t),\tau),\tau)) - \kappa  \right) d\tau \notag\\
&\qquad + \int_s^t \left(\kappa - \dot{\sc}(\tau)   \right) + (\lambda_3 - \wb) (\eta(x_-(t),\tau),\tau)  d\tau 
\notag\\
&\geq (w_0(x_-(t)) - \kappa) (t-s) -\tfrac{1}{2} (  \mm^3 + 4 \Rsf_1) (t^{2} - s^{2})   
\notag\\
&\geq \tfrac 45 \bb^{\frac 32} t^{\frac 12} (t-s)
\label{eq:finally:useful}
\end{align}
for all $s\in [0,t)$, and all $x\in \Upsilon(t)$. This bound shows that $\Upsilon(s)  \supset \Upsilon(t)$ for $s<t$.

Recalling the $\etab(x,t)$ is defined by \eqref{eq:etab:def} for all $x\in \TT$, and in particular for $x\in \Upsilon(t)$, from \eqref{eq:range:for:solutions:new} and \eqref{eq:px:eta} we immediately deduce that
\begin{subequations} 
\begin{alignat}{2}
&\sabs{\eta(x,t) - \etab(x,t)} \leq \tfrac 32 \Rsf_1 t^2, \qquad &&\mbox{for all} \qquad  x \in \Upsilon(t)   
\,, \label{cb1:new}\\
&\sabs{\p_x \eta(x,t) - \p_x \etab(x,t)} \leq (16 \Rsf_1 \bb^{-\frac 32} + 8 \Rsf_2 \bb^{-\frac 12}) t^{\frac 12}, \qquad &&\mbox{for all} \qquad x \in \Upsilon(t) 
\,,\label{cb2:new} 
\end{alignat}
\end{subequations}
for all $t\in (0,\bar \eps]$. 
The bound \eqref{cb1:new} follows from \eqref{eq:AC:DC:ride:on}, the mean value theorem, and the fact that by \eqref{eq:u0':interest} we have that $|\p_x \etab(x,t) - 1| \leq \frac{9}{20}$ for all $x\in \Upsilon(t)$  (in analogy to \eqref{eq:px:eta:b}). In order to prove the bound \eqref{cb2:new}, we use
\begin{align*}
\p_t (\p_x \eta - \p_x \etab) 
&= (\p_\theta \wb) \circ \eta\; ( \p_x \eta - \p_x \etab)
+ (\p_\theta \lambda_3 - \p_\theta \wb)\circ \eta \; \p_x \eta \notag\\
&\qquad + \left( (\p_\theta \wb) \circ \eta - (\p_\theta \wb) \circ \etab \right) \p_x \etab \,,
\end{align*}
and the fact that $\p_x \eta(x,0) - \p_x\etab(x,0) = 0$. First, we note that due to \eqref{eq:u0'':interest}, \eqref{eq:px:eta:b},  \eqref{eq:px:wb:def}, \eqref{eq:AC:DC:ride:on}, and the mean value theorem, we have that 
\begin{align}
\sabs{\left( (\p_\theta \wb) \circ \eta - (\p_\theta \wb) \circ \etab \right) \p_x \etab}
&\leq 2 \left| \frac{ w_0'(\etab^{-1}(\eta(x,t),t))}{1 + t w_0'(\etab^{-1}(\eta(x,t),t))}  -  \frac{ w_0'(x)}{1 + t w_0'(x)}   \right| \notag\\
&\leq 4  \Rsf_1 \bb^{-\frac 32} t^{-\frac 12} \,. \label{thechicken}
\end{align}
Second, by the assumption \eqref{eq:lambda:3:abstract:b} and the bound \eqref{eq:px:eta} we know that 
\begin{align}
 \sabs{(\p_\theta \lambda_3 - \p_\theta \wb)\circ \eta \; \p_x \eta} \leq 2 \Rsf_2 (\bb t)^{-\frac 12} \,.
\end{align}
Combining the above two estimates with the evolution equation for $\p_x \eta - \p_x \etab$ and \eqref{eq:AC:DC:ride:on:and:on}, we obtain \eqref{cb2:new}.
 
The results in this section may be summarized as follows: 
\begin{lemma}
\label{lem:AC:DC}
Let  $\eta$ be defined by \eqref{eq:dt:eta:def}, with $\lambda_3$ satisfying \eqref{eq:lambda:3:abstract}. Then, by possibly further reducing the value of $\bar \eps$, solely in terms of $\kappa, \bb, \cc, \mm$, the following hold. With the definition of $\Upsilon(t)$ in \eqref{eq:x:pm:Upsilon}, we have that $\eta(\cdot,t) \colon \Upsilon(t) \to \TT\setminus \{\sc(t)\}$ is a bijection. For $x\in \Upsilon(t)$, the curve $\{\eta(x,s)\}_{s\in[0,t]}$ does not intersect the shock curve, and by \eqref{eq:px:eta}, \eqref{cb1:new}, \eqref{cb2:new}, we have the estimates
\begin{subequations}
\begin{align}
\tfrac 12 &\leq \p_x \eta(x,t) \leq \tfrac 74 \label{eq:cb:0} \\
\tfrac 13 \kappa &\leq \p_t \eta(x,t) \leq \tfrac 32 \mm \label{eq:cb:00} \\
\sabs{\eta(x,t) - \etab(x,t)}  &\leq \tfrac 32 \Rsf_1 t^2 \label{eq:cb:1}\\
\sabs{\p_x \eta(x,t) - \p_x \etab(x,t)}  &\leq (16 \Rsf_1 \bb^{-\frac 32} + 8 \Rsf_2 \bb^{-\frac 12}) t^{\frac 12} \label{eq:cb:2}
\end{align} 
\end{subequations}
The inverse map $\eta^{-1} \colon \DD_{\bar \eps} \to \TT \setminus \{0\}$ is continuous in space-and-time, with bounds
\begin{subequations}
\label{eq:steves:desires}
\begin{align}
\tfrac 47 &\leq \p_\theta \eta^{-1}(\theta,t)  \leq 2 \label{eq:steves:desires:a}\\
- 3 \mm &\leq \p_t \eta^{-1}(\theta,t)  \leq -\tfrac 14 \kappa \label{eq:steves:desires:b}
\end{align}
for all $(\theta,t) \in \DD_{\bar \eps}$. 
\end{subequations}
Lastly, from \eqref{eq:AC:DC:ride:on:and:on} and \eqref{eq:AC:DC:ride:on:and:on:on} we have that 
\begin{subequations}
\begin{align}
\int_0^t \abs{\p_\theta \wb(\eta(x,s),s)}  ds &\leq \tfrac{19}{40}
\label{eq:AC:DC:con}\\
\int_0^t \abs{\p_{\theta}^2 \wb(\eta(x,s),s)}  ds &\leq 3 (\bb t)^{-\frac 32}
\label{eq:AC:DC:defcon}
\end{align}
\end{subequations}
for all $x\in \Upsilon(t)$, and all $t\in   [0,\bar \eps]$.
\end{lemma}
\begin{proof}[Proof of Lemma~\ref{lem:AC:DC}]
The only estimates which were not established in the discussion above the lemma are \eqref{eq:cb:00} and \eqref{eq:steves:desires}. In order to prove \eqref{eq:cb:00}, we appeal to \eqref{eq:dt:eta:def}, \eqref{eq:lambda:3:abstract:a}, \eqref{eq:wb:def}, \eqref{eq:u0:ass:2}, \eqref{eq:u0:ass:2}, and take $\bar \eps$ to be sufficiently small:
\begin{align*}
\p_t \eta(x,t) = \lambda_3 (\eta(x,t)) =  \wb(\eta(x,t)) + \OO(t) = \underbrace{w_0(\etab^{-1}(\eta(x,t),t))} _{\in [\frac{\kappa}{2},\mm]}+ \OO(t) \in [\tfrac{\kappa}{3} , \tfrac{3\mm}{2}]
\,.
\end{align*}
The bound \eqref{eq:steves:desires:a} follows from \eqref{eq:cb:0} and the inverse function theorem. 
Lastly, in order to prove \eqref{eq:steves:desires:b}, we use that $\eta^{-1}$ solves the transport equation dual to the ODE \eqref{eq:dt:eta:def}, namely $\p_t \eta^{-1} + \lambda_3 \p_{y} \eta^{-1} = 0$. As such, from \eqref{eq:steves:desires:a}, \eqref{eq:lambda:3:abstract:a}, \eqref{eq:wb:def}, \eqref{eq:u0:ass:2}, \eqref{eq:u0:ass:2}, we obtain that 
\begin{align*}
\p_t \eta(\theta,t) = - \wb(\theta,t) \p_\theta \eta^{-1} + \OO(t)  = - \underbrace{w_0(\etab^{-1}(\theta,t))}_{\in [\frac{\kappa}{2},\mm]} \underbrace{\p_\theta \eta^{-1}}_{\in [\frac 47, 2]} + \OO(t)  \in [-3\mm, - \tfrac{\kappa}{4}]
\end{align*}
upon taking $\bar \eps$ to be sufficiently small.
\end{proof}

\subsubsection{Estimates for derivatives of $\wb$ along flows transversal to the shock}
In analogy to Lemma~\ref{lem:AC:DC}, we also have an estimate for the time integral of $\p_\theta \wb$ along any flow which is transversal to $\sc$. More precisely, we have:
\begin{lemma}
\label{lem:Steve:needs:this}
Fix $t\in (0,\bar\eps]$ and $ 0 \leq \theta <\sc(t) $. For some $t' \in [0,t)$, assume that we are given a differentiable curve $\gamma \colon [t',t] \to \DD_{\bar \eps}$ which does not intersect the shock curve $\sc$, such that $\gamma(t) = \theta$, and such that $\dot \gamma(s) \leq \mu \kappa$ for all $s\in [t',t]$, for some $\mu \in [0,1)$. Then, we have that 
\begin{subequations}
\label{eq:Steve:needs:this}
\begin{align}
 \int_{t'}^t \abs{\p_{\theta} \wb(\gamma(s),s)} ds &\leq \tfrac{13 \bb}{(1-\mu) \kappa^{{\frac{2}{3}} }} t^{\frac 13} 
 \label{eq:Steve:needs:this:1} \\
  \int_{t'}^t \abs{\p_{\theta}^2 \wb(\gamma(s),s)} ds 
 &\leq \tfrac{9 \bb}{(1-\mu)\kappa} \left( \tfrac 12 |\gamma(t') - \sc(t')| + \tfrac 45 (\bb t')^{\frac 32} \right)^{-\frac 23}  \label{eq:Steve:needs:this:2good} \\
&\leq \tfrac{11}{(1-\mu) \kappa} t'^{-1}
 \label{eq:Steve:needs:this:2}
\end{align}
\end{subequations}
where $C = C (\kappa,\bb, \mu) > 0$ is an explicitly computable constant. 
\end{lemma}
\begin{proof}[Proof of Lemma~\ref{lem:Steve:needs:this}]
As in the proof of Lemma~\ref{lem:AC:DC}, the goal is to understand the evolution of $x(s) := \etab^{-1}(\gamma(s),s)$. First, we note that since $\gamma$ lies on the left side of $\sc$, the point $x(s)$ is well-defined, and satisfies $x(s) \leq - \frac 45 (\bb s)^{\frac 32}$. Next, from the definition of $\etab$ and its inverse, we have that
\begin{align}
\dot{x}(s) 
= (\p_t \etab^{-1})(\gamma(s),s) + \dot\gamma (s) (\p_\theta \etab^{-1})(\gamma(s),s) 
&= \frac{\dot\gamma(s) - (\p_t \etab)(\etab^{-1}(\gamma(s),s),s)}{(\p_x \etab)(\etab^{-1}(\gamma(s),s),s)} 
\notag\\
&= \frac{\dot\gamma(s) - w_0(x(s))}{1 + s w_0'(x(s))} 
\,.
\label{eq:Miller:Lite:is:not:bad:1}
\end{align}
Due to the aforementioned lower bound on $|x(s)|$ and the estimate \eqref{eq:u0':interest}, the denominator of the fraction on the right side of \eqref{eq:Miller:Lite:is:not:bad:1} lies in the interval $[\frac 12,\frac 32 ]$. Furthermore, since $t\leq \bar \eps$ and $\bar \eps$ is sufficiently small, we have that $|x(t)| = |\etab^{-1}(y,t)|$ is sufficiently small to ensure via \eqref{eq:u0:ass} that $\abs{w_0(x(t)) - \kappa} \leq 2 \bb |x(t)|^{\frac 13} \leq \frac{1-\mu}{4} \kappa$.  Also, from \eqref{eq:Miller:Lite:is:not:bad:1} we may deduce that $\abs{\dot x(s)} \leq 4 \mm$ which implies $|x(s)| \leq |x(t)| + 4 \mm t$; therefore, since $t\leq \bar \eps$ is sufficiently small, we may show that $\abs{w_0(x(s)) - \kappa} \leq \frac{1-\mu}{2} \kappa$ for all $s\in [t',t]$. We then immediately obtain from \eqref{eq:Miller:Lite:is:not:bad:1} that 
\begin{align}
- (3-\mu)\kappa \leq \frac{- w_0(x(s))}{1 + s w_0'(x(s))} \leq \dot{x}(s) \leq \frac{\mu \kappa - w_0(x(s))}{1 + s w_0'(x(s))} \leq - \frac{(1-\mu)\kappa}{3}
\label{eq:Miller:Lite:is:not:bad:2}.
\end{align}
Then, using \eqref{eq:px:wb:def} and the fact that $x(s)$ is strictly negative, we obtain that 
\begin{align}
\int_{t'}^t \abs{\p_\theta \wb(\gamma(s),s)}  ds 
&= \int_{t'}^t \frac{ \abs{w_0'(x(s))}}{1 + s w_0'(x(s))} ds 
\leq \bb \int_{t'}^t (x(s))^{-\frac 23} ds \notag\\
&\leq - \tfrac{3 \bb}{(1-\mu)\kappa} \int_{t'}^t \dot{x}(s) (x(s))^{-\frac 23} ds \notag\\
&= \tfrac{9 \bb}{(1-\mu)\kappa} \left(x(t')^{\frac 13}  - x(t)^{\frac 13}\right) \notag\\
&\leq \tfrac{9 \bb}{(1-\mu)\kappa} |x(t)|^{\frac 13} = \tfrac{9 \bb}{(1-\theta)\kappa} |\etab^{-1}(\theta,t)|^{\frac 13}
\leq \tfrac{9 \bb}{(1-\mu)\kappa} (3 \kappa t)^{\frac 13}
\end{align}
In the last inequality we have used that since $0 \leq \theta < \sc(t)$ we have that $|\etab^{-1}(\theta,t)| \leq |\etab^{-1}(0,t)| \leq 3 \kappa t$
for all $t\leq \bar \eps$, which is sufficiently small. 

The proof of \eqref{eq:Steve:needs:this:2} is nearly identical, but instead of \eqref{eq:px:wb:def} we appeal to \eqref{eq:pxx:wb:def}, arriving at 
\begin{align}
\int_{t'}^t \abs{\p_{\theta}^2 \wb(\gamma(s),s)}  ds 
\leq \tfrac{9 \bb}{(1-\mu) \kappa} \left(x(t')^{-\frac 23}  - x(t)^{-\frac 23}\right)
\leq \tfrac{9 \bb}{(1-\mu) \kappa}  \left(x(t')\right)^{-\frac 23}   
\end{align}
In order to obtain  \eqref{eq:Steve:needs:this:2good}--\eqref{eq:Steve:needs:this:2}, we use the above bound and \eqref{eq:tropic:2}, which implies that  $|x(t')| = |\etab^{-1}(\gamma(t'),t')| \geq \frac 12 |\gamma(t') - \sc(t')| + \frac 45 (\bb t')^{\frac 32} \geq \frac 45 (\bb t')^{\frac 32}$. 
\end{proof}

\subsection{$z$ and $k$ on the shock curve}
\label{sec:z:k:shock}

For every $t\in (0,\bar \eps]$, let us {\em assume} that we are given a left speed $\vl = \vl (t)=w(\sc(t)^-,t)$ and a right speed $\vr = \vr(t) =w(\sc(t)^+,t)$ at the point $(\sc(t),t)$. Furthermore, let us assume that $\vl$ and $\vr$ behave similarly to the solution of the Burgers equation computed in Proposition~\ref{prop:Burgers}; by this we mean that the jump  and the mean at $(\sc(t),t)$, defined by
\begin{align}
\jump{w} = \jump{w}(t) = \vl(t)- \vr(t) \,, \qquad \mean{w} = \mean{w}(t) = \tfrac 12 \left(\vl (t) + \vr(t) \right)\,,
\label{eq:J:M:def}
\end{align}
satisfy the bounds
\begin{align}
 \sabs{\jump{w}(t) - 2 \bb^{\frac 32} t^{\frac 12}} \leq \Rsf_j t
\qquad \mbox{and} \qquad 
\abs{\mean{w}(t)-\kappa} \leq   \Rsf_m t
\,,
\label{eq:J:M:rough} 
\end{align}
for all $t \in (0,\bar \eps]$, for two constants $\Rsf_j,\Rsf_m >0$ which only depend on $\kappa, \bb, \cc$, and $\mm$. These bounds are consistent with \eqref{eq:basic:jump} and \eqref{1krod} (to be established below).

The variables $\vl$ and $\vr$ are the same as those in equations \eqref{pjump77}--\eqref{pjump7}. Our goal in this subsection is to solve the coupled system of equations \eqref{pjump77}--\eqref{pjump7}, for the jumps of $z$ and $k$ at the fixed point $(\sc(t),t)$, as a function the left speed $\vl$ and right speed $\vr$, at this point. Since $z$ and $k$ are equal to $0$ on the right side of the shock curve, we note that the jumps of $z$ and $k$ are equal to their values on the {\em left} of $(\sc(t),t)$; as such, we work with the unknowns  
\begin{align}
\zl = \zl(t) = \jump{z}(t)\,, \qquad \kl = \kl(t) = \jump{k}(t)\,.
\end{align}
In fact, because we expect $\kl$ to be close to $0$ (see \eqref{spcube}), and since \eqref{pjump77}--\eqref{pjump7} contain the variables $e^{-\kl}$ and $e^{\kl}$, which are thus close to $1$, it is more convenient to replace $\kl$ with the unknown
\begin{align}
\el = \el(t) = e^{\kl(t)}-1 \,. 
\end{align}
Then, with this notation  the equations \eqref{pjump77}--\eqref{pjump7} may be rewritten as the system 
\begin{subequations}
\label{eq:zminus:kminus}
\begin{align}
{\mathcal E}_{1}(\vl,\vr,\zl,\el) &=0  \label{zminus-on-shock} \\
{\mathcal E}_{2}(\vl,\vr,\zl,\el) &=0  \label{kminus-on-shock}  
\end{align} 
\end{subequations}
where
\begin{subequations}
\label{eq:zminus:kminus:def}
\begin{align}
&{\mathcal E}_{1}(\vl,\vr,\zl,\el) \notag\\
&= \left(   (\vl -\zl)^2 (\vl + \zl)^2 + \tfrac 18  (\vl - \zl)^4 -  \tfrac{9}{8}  (1+\el) \vr^4 \right)
\left((\vl - \zl)^2 - (1+\el) \vr^2 \right) \notag\\
&\qquad - \left(  (\vl - \zl)^2(\vl + \zl) -  (1+\el)  \vr^3 \right)^2
\label{zminus-on-shock:def}\\
&{\mathcal E}_{2}(\vl,\vr,\zl,\el) \notag\\
&=\el  (\vl - \zl)^4(3 \vr^2  (1+\el)  -  (\vl - \zl)^2  )  -  \left(  (\vl - \zl)^2 -  (1+\el) \vr^2\right)^3   \,.
\label{kminus-on-shock:def}
\end{align}
\end{subequations}
We view \eqref{eq:zminus:kminus} as a coupled system of equations for the unknowns $\zl$ and $\el$ (or alternatively, $\kl$), with $\vl$ and $\vr$ given. The correct root of \eqref{eq:zminus:kminus} is given by:
\begin{lemma}[\bf Existence and asymptotic formula for $\zl$ and $\kl$]
\label{lem:zminus-on-shock} 
Assume that $\vl$ and $\vr$ are such that their jump and mean at $(\sc(t),t)$ satisfy \eqref{eq:J:M:rough}. Then, the system of equations  \eqref{eq:zminus:kminus} has a smallest (in absolute value) root $(\zl, \el)$, such that $\zl$ and $\kl = \log(\el+1)$  satisfy the bounds
\begin{subequations}
\label{eq:zm:and:km:on:xi}
\begin{align} 
\abs{ \zl(  t) +  \tfrac{9 \jump{w}(t)^3 }{16 \mean{w}(t)^2}     } 
&\leq C_0 t^{\frac 52}  \,.   \label{zm-on-xi}
\\
\abs{ \kl( t) -  \tfrac{4  \jump{w}(t)^3}{\mean{w}(t)^3}      } 
&\leq  C_0 t^{\frac 52}  \,,   \label{km-on-xi}
\end{align} 
\end{subequations}
where $C_0 = C_0(\kappa,\bb,\cc,\mm)>0$ is an explicitly computable constant.
In particular, in view of \eqref{eq:J:M:rough} we have the estimates
\begin{subequations}
\label{eq:zl:and:kl:on:shock:L:infinity}
\begin{align}
\sabs{\zl( t) + \tfrac{9 \bb^{\frac 92}}{2 \kappa^2} t^{\frac 32} } \leq  C t^2
\qquad &\Rightarrow \qquad
\abs{\zl( t)} \leq \tfrac{5 \bb^{\frac 92}}{ \kappa^{2}} t^{\frac 32} 
\label{eq:zl:on:shock:L:infinity} \\
\sabs{\kl( t) - \tfrac{32 \bb^{\frac 92}}{\kappa^3} t^{\frac 32} } \leq C t^2
\qquad &\Rightarrow \qquad
\abs{\kl( t)} \leq \tfrac{ 40 \bb^{\frac 92}}{ \kappa^{3}} t^{\frac 32}
\label{eq:kl:on:shock:L:infinity} 
\end{align}
\end{subequations}
for all $t \in (0,\bar \eps]$, assuming that $\bar \eps$ is sufficiently small.
\end{lemma}

\begin{proof}[\bf Proof of Lemma~\ref{lem:zminus-on-shock}]
Throughout the proof,  we fix  $t \in (0,\bar \eps]$, and omit the $t$ dependence of the unknowns. In view of \eqref{eq:J:M:rough}, we view $\jump{w}$ as a small parameter, thus suitable for asymptotic expansions, and $\mean{w}$ as an $\OO(1)$ parameter. As such, in \eqref{eq:zminus:kminus:def} we replace
\begin{align*}
\vl= \mean{w} + \tfrac 12 \jump{w}\, , \qquad \vr = \mean{w} - \tfrac 12 \jump{w} \,. 
\end{align*}
Because we expect $|\zl| , |\el| \ll 1$, we first perform a Taylor series expansion of  \eqref{eq:zminus:kminus}, and identify only the linear terms with respect to $\zl$ and $\el$. This becomes
\begin{align*}
& \tfrac{1}{16} \jump{w}^4 (12 \mean{w}^2 - \jump{w}^2)
- \tfrac 18 \jump{w} (32 \mean{w}^4 + 8 \mean{w}^2 \jump{w}^2 + 6 \mean{w} \jump{w}^3 - \jump{w}^4)\zl \notag\\
&\quad 
- \tfrac{1}{64} \jump{w} (48\mean{w}^5 + 40 \mean{w}^3 \jump{w}^2 - 48 \mean{w}^2 \jump{w}^3 + 3 \mean{w} \jump{w}^4 + 4 \jump{w}^5 )  \el 
= \OO(|\zl|^2 + |\el|^2) \\
&- 8 \mean{w}^3 \jump{w}^3
+ 12 \jump{w}^2 (2 \mean{w}^3 + \mean{w}^2 \jump{w}) \zl 
\notag\\
&\quad
+ \tfrac{1}{32} (64 \mean{w}^6 + 240 \mean{w}^4 \jump{w}^2 - 512 \mean{w}^3 \jump{w}^3 + 60 \mean{w}^2 \jump{w}^4 + \jump{w}^6)\el
 = \OO(|\zl|^2 + |\el|^2) 
 \,.
\end{align*}
By dropping the higher order terms in $| \jump{w} | \ll 1$, 
this motivates our definition of the approximate solutions $\zl^{\rm app}$ and $\el^{\rm app}$ as the solutions of the linear system
\begin{align}
&
\left(\begin{matrix}
4 \jump{w} \mean{w}^4
&  \tfrac 34 \jump{w} \mean{w}^5 \\
24 \jump{w}^2 \mean{w}^3
& 2 \mean{w}^6
\end{matrix}\right)
\left(\begin{matrix}
\zl^{\rm app}  \\
\el^{\rm app}   
\end{matrix}\right)
 =
\left(\begin{matrix}
\tfrac 34 \jump{w}^4 \mean{w}^2  \\
8 \jump{w}^3 \mean{w}^3
\end{matrix}\right)
\,.
\label{eq:zl:el:matrix}
\end{align}
This system is uniquely solvable, and yields 
\begin{subequations}
\label{eq:zl:kl:app:def}
\begin{align}
 \zl^{\rm app} &= - \frac{9 \jump{w}^3 }{16 \mean{w}^2 } Q_1\left(\frac{\jump{w}}{\mean{w}}\right)\,,
 \qquad Q_1(x) = \frac{1}{1 - \frac 94 x^2}\,,
 \label{eq:zl:app:def}\\
 \el^{\rm app} &=   \frac{4 \jump{w}^3 }{\mean{w}^3} Q_2\left(\frac{\jump{w}}{\mean{w}}\right) \,,
 \qquad  Q_2(x) =  \frac{1 - \frac{9}{16}x^2}{1 - \frac{9}{4}x^2}   \,.
  \label{eq:el:app:def}
 \end{align} 
\end{subequations}
In order to apply the implicit function theorem, we at last introduce the variables
\begin{align}
Z = \frac{\zl - \zl^{\rm app}}{\jump{w}^{5}}
\qquad \mbox{and} \qquad  
E = \frac{\el - \el^{\rm app}}{\jump{w}^{5}}
\label{eq:normalized:E:and:Z}
\end{align}
and substitute in the system \eqref{eq:zminus:kminus:def} the ansatz 
$\zl = \zl^{\rm app} + Z \jump{w}^5$ and $\el = \el^{\rm app} + E \jump{w}^5$. After some algebraic manipulations, the system of equations \eqref{eq:zminus:kminus:def} is rewritten as system
\begin{subequations}
\label{eq:i:am:stupid}
\begin{align}
0&= {\mathcal F}_1 (\jump{w},\mean{w}, Z, E) \notag\\
&=  \jump{w}^{-6}  {\mathcal E}_1 (\mean{w} + \tfrac 12 \jump{w},\mean{w} - \tfrac 12 \jump{w},\zl^{\rm app} + Z \jump{w}^5,\el^{\rm app} + E \jump{w}^5)  \\
0&= {\mathcal F}_2 (\jump{w},\mean{w}, Z, E) \notag\\
&=  \jump{w}^{-5} {\mathcal E}_2 (\mean{w} + \tfrac 12 \jump{w},\mean{w} - \tfrac 12 \jump{w},\zl^{\rm app} + Z \jump{w}^5,\el^{\rm app} + E \jump{w}^5) 
\end{align}
\end{subequations}
for the unknowns $Z$ and $E$. 
Defining 
$$ P_w = (0 ,\mean{w}, -\tfrac{27}{64} \mean{w}^{-4}, -  15  \mean{w}^{-5}) \, , $$
we observe that 
\begin{align*}
&{\mathcal F}_1 (P_w)= 0 \,,
\qquad 
\partial_Z {\mathcal F}_1 (P_w)= - 4 \mean{w}^4\,,
\qquad 
\partial_Z {\mathcal F}_2 (P_w)=  0\,, \\
& {\mathcal F}_2 (P_w)=  0 \,,
\qquad
\partial_E {\mathcal F}_1 (P_w)= 0 \,,
\qquad 
\partial_E {\mathcal F}_2 (P_w) =  2\mean{w}^6.
\end{align*}
Thus, the Jacobian determinant associated to $({\mathcal F}_1,{\mathcal F}_2)(\cdot,\cdot,Z,E)$  evaluated at  $P_w$ equals to $-8 \mean{w}^{10} \neq 0$. Here we are using that by \eqref{eq:J:M:rough} we have that  $\abs{\mean{w}-\kappa} \leq \frac{\kappa}{2}$, and thus $\mean{w}\neq 0$. Thus, by the implicit function theorem, there  exists a $J_0 = J_0(\mean{w}) >0$, such that for all $\abs{\jump{w}} \leq J_0$, we have a unique solution $Z = Z (\jump{w},\mean{w})$ and $E = E(\jump{w},\mean{w})$ of \eqref{eq:i:am:stupid}, with $Z(0,\mean{w}) = -\tfrac{27}{64} \mean{w}^{-4}$ and $E(0,\mean{w}) = - 15 \mean{w}^{-6}$. To conclude, we note that since $J_0$ depends only on $\mean{w}$, it may be estimated solely in terms of $\kappa$; and since by \eqref{eq:J:M:rough} we have that $\abs{\jump{w}} \leq 3  \bb^{\frac 32} \bar \eps^{\frac 12}$ with $\bar \eps$ which is sufficiently small in terms of $\kappa$ and $\bb$, we deduce that  the condition $\abs{\jump{w}} \leq J_0$ is automatically guaranteed. 

As a consequence, from the above discussion  we deduce that for all  $t \leq \bar \eps$, we have
\begin{align} 
\abs{\zl  - \zl^{\rm app} } 
\leq C_0 \jump{w}^5\,,
\qquad\mbox{and}\qquad
\abs{\el  - \el^{\rm app} } 
\leq C_0 \jump{w}^5 
\,,\label{eq:orange:orangutan:3}
\end{align}
where $C_0>0$ is a constant which only depends on $\kappa$. 

The proof of the bounds \eqref{zm-on-xi}--\eqref{km-on-xi} are now essentially completed, upon combining \eqref{eq:J:M:rough}, \eqref{eq:zl:kl:app:def}, and \eqref{eq:orange:orangutan:3}. To see this, note that the rational function $Q_1$ appearing in the definition \eqref{eq:zl:app:def} satisfies $|Q_1(x) - 1| \leq 3 x^2$ for all $x \leq \frac{1}{10}$. Thus, we obtain that 
\begin{align*}
\abs{\zl^{\rm app} + \frac{9 \jump{w}^3}{16 \mean{w}^2}}  \leq C_0 \jump{w}^5
\end{align*}
since $\mean{w} \geq \frac{\kappa}{2}$ when $\bar \eps$ is sufficiently small. The bound \eqref{zm-on-xi} follows from the above estimate, \eqref{eq:J:M:rough}, and \eqref{eq:orange:orangutan:3}. Similarly, by using that the rational function $Q_2$ appearing in the definition \eqref{eq:zl:app:def} satisfies $|Q_2(x) - 1| \leq 2 x^2$ for all $x \leq \frac{1}{10}$, we obtain the bound
\begin{align*}
\abs{\el^{\rm app} -   \frac{4 \jump{w}^3}{\mean{w}^3}}  \leq C_0 \jump{w}^5\,,
\end{align*}
which may be combined with \eqref{eq:J:M:rough} and \eqref{eq:orange:orangutan:3}, 
to establish
\begin{align}
\abs{ \el  -  \frac{4\jump{w}^3}{\mean{w}^3}     } 
&\leq C_0  \jump{w}^{5}
\label{eq:orange:orangutan:333}
\end{align}
with $C_0>0$ a constant which depends only on $\kappa$ and $\bb$. The bound \eqref{km-on-xi} now follows because $\kl = \log(1+ \el)$, and $\abs{\log (1+ \el) - \el} \leq 2 \el^2$ for $\abs{\el} \leq \frac 12$; clearly, $\abs{\el} = \OO(t^{\frac 32}) \leq \frac 12$ in view of \eqref{eq:orange:orangutan:333}.

The bounds \eqref{eq:zl:on:shock:L:infinity}--\eqref{eq:kl:on:shock:L:infinity} follow from \eqref{zm-on-xi}--\eqref{km-on-xi}, \eqref{eq:J:M:rough}, and the fact that $t \leq \bar \eps$, which in turn may be made arbitrarily small with respect to $\kappa$ and $\bb$. 
\end{proof}

Let us further assume that $\vr$ and $\vl$ are differentiable with respect to $\xi$ and $t$ for all $(\xi,t) \in \Omega_{\bar \eps}$. By implicitly differentiating \eqref{zminus-on-shock}--\eqref{kminus-on-shock}, we may then deduce:

\begin{lemma}[\bf Lipschitz bounds for $\zl$ and $\kl$]
 \label{lem:zl:kl:Lip}
For $t\in (0,\bar \eps]$, assume that $\vl$ and $\vr$ are such that their jump and mean at $(\sc(t),t)$ satisfy \eqref{eq:J:M:rough}, and further assume that $\mean{w}$ and $\jump{w}$ are differentiable with respect to $t$. Then, the smallest roots of the  system of equations  \eqref{eq:zminus:kminus}  are such that $\zl$ and $\kl = \log(\el+1)$ satisfy
the pointwise estimates
\begin{subequations}
\label{eq:xi:t:derivatives:zl:kl}
\begin{align}
\abs{\tfrac{d}{dt} \zl (t)+ \tfrac{d}{dt} \left( \tfrac{9 \jump{w}(t)^3}{16 \mean{w}(t)^2}\right)} \leq C_0 t^2 \Bigl( |\tfrac{d}{dt} \jump{w}(t)| + |\tfrac{d}{dt}  \mean{w}(t)|\Bigr)
\label{eq:xi:t:derivatives:zl} 
\\
\abs{e^{\kl(t)} \tfrac{d}{dt}  \kl(t) - \tfrac{d}{dt}   \left( \tfrac{4 \jump{w}(t)^3}{\mean{w}(t)^3}\right)  } \leq C_0 t^2 \Bigl(|\tfrac{d}{dt} \jump{w}(t)| + |\tfrac{d}{dt}  \mean{w}(t)| \Bigr)
\label{eq:xi:t:derivatives:kl}
\end{align} 
\end{subequations}
where the constant $C_0>0$ only depends on $\kappa$, $\bb$, and $\mm$.  
\end{lemma}

\begin{proof}[Proof of Lemma~\ref{lem:zl:kl:Lip}]
 From the definition $\kl = \log(1+ \el)$
we obtain that $\frac{d}{dt} \kl = e^{-\kl} \frac{d}{dt} \el$, and thus, in order to prove the lemma it is sufficient to obtain derivative bounds for $\zl$ and $\el$.

Implicitly differentiating \eqref{eq:zminus:kminus} we arrive at  
\begin{align}
\frac{d}{dt} \left(\begin{matrix}
 \zl \\
 \el
\end{matrix}\right)
= - \left(\begin{matrix}
\p_{\zl} {\mathcal E}_1   &  \p_{\el} {\mathcal E}_1     \\
\p_{\zl} {\mathcal E}_2   &  \p_{\el} {\mathcal E}_2       
\end{matrix}\right)^{-1}
\left(\begin{matrix}
\p_{\vl} {\mathcal E}_1   &  \p_{\vr} {\mathcal E}_1 \\
\p_{\vl} {\mathcal E}_2   &  \p_{\vr} {\mathcal E}_2  
\end{matrix}\right)
\frac{d}{dt}
\left(\begin{matrix}
 \vl \\
  \vr
\end{matrix}\right)
\,,
\label{eq:xi:derivatives:zl:el}
\end{align}
pointwise for $t \in (0,\bar \eps]$, where we recall that the functions ${\mathcal E}_1$ and ${\mathcal E}_2$ are defined in \eqref{eq:zminus:kminus:def}.  
In order to evaluate these Jacobi matrices, we resort to the notation in \eqref{eq:J:M:def} and rewrite $\vl = \mean{w} + \frac 12 \jump{w}$ and $\vr = \mean{w} - \frac 12 \jump{w}$; furthermore, we write $\zl = \zl^{\rm app} + \OO(\jump{w}^5)$ and $\el = \el^{\rm app} + \OO(\jump{w}^6)$ as justified by \eqref{eq:normalized:E:and:Z}, with $\zl^{\rm app}$ defined by \eqref{eq:zl:app:def}, and $\el^{\rm app}$ given by \eqref{eq:el:app:def}. We emphasize that the implicit constants in the $\OO(\jump{w}^5)$ and $\OO(\jump{w}^6)$ symbols only depend on $\kappa$ and $\bb$, since the bounds on the solutions $Z$ and $E$ of \eqref{eq:i:am:stupid} only depend on $\kappa$ and $\bb$. After some   tedious computations, we arrive at 
\begin{align}
- \left(\begin{matrix}
\p_{\zl} {\mathcal E}_1   &  \p_{\el} {\mathcal E}_1     \\
\p_{\zl} {\mathcal E}_2   &  \p_{\el} {\mathcal E}_2       
\end{matrix}\right)^{-1}
\left(\begin{matrix}
\p_{\vl} {\mathcal E}_1   &  \p_{\vr} {\mathcal E}_1 \\
\p_{\vl} {\mathcal E}_2   &  \p_{\vr} {\mathcal E}_2  
\end{matrix}\right) =  
\left(\begin{matrix}
-\frac{27 \jump{w}^2}{16 \mean{w}^2} +  \frac{9 \jump{w}^3}{16 \mean{w}^3}  &  \frac{27 \jump{w}^2}{16 \mean{w}^2}  +  \frac{9 \jump{w}^3}{16 \mean{w}^3}  \\
\frac{12 \jump{w}^2}{\mean{w}^3} - \frac{6 \jump{w}^3}{\mean{w}^4}   &  -\frac{12 \jump{w}^2}{\mean{w}^3} - \frac{6 \jump{w}^3}{\mean{w}^4}
\end{matrix}\right) + \OO(\jump{w}^4) \,,
\label{eq:Matrix:Masturbation:1}
\end{align}
where the implicit constant only depends on $\kappa$ and $\bb$.
From \eqref{eq:J:M:rough}, \eqref{eq:xi:derivatives:zl:el}, \eqref{eq:Matrix:Masturbation:1}, and recalling that $\frac{d}{dt} \vl =\frac{d}{dt}  \mean{w} + \frac 12 \frac{d}{dt} \jump{w}$ and $\frac{d}{dt}  \vr = \frac{d}{dt}  \mean{w} - \frac 12 \frac{d}{dt}  \jump{w}$,  we deduce that there exists a constant $C_0> 0$, which only depends on $\kappa$ and $\bb$, such that 
\begin{align}
\abs{\tfrac{d}{dt} \zl + \tfrac{d}{dt} \left( \tfrac{9 \jump{w}^3}{16 \mean{w}^2}\right)} + \abs{\tfrac{d}{dt}  \el - \tfrac{d}{dt}  \left( \tfrac{4 \jump{w}^3}{\mean{w}^3}\right)  } \leq C_0 \jump{w}^4 \Bigl( \abs{\tfrac{d}{dt}  \jump{w}} + \abs{\tfrac{d}{dt}  \mean{w}}\Bigr)
\label{eq:xi:derivatives:zl:zl} 
\end{align} 
The bounds 
\eqref{eq:xi:t:derivatives:zl:kl} follow from \eqref{eq:xi:derivatives:zl:zl}, upon recalling that $\jump{w} = \OO(t^{\frac 12})$.
\end{proof}

A direct consequence of Lemmas~\ref{prop:Burgers},~\ref{lem:zminus-on-shock}, and~\ref{lem:zl:kl:Lip} is the following statement, which will be useful in the proof of Proposition~\ref{thm:curve:determines:all}.  
\begin{corollary}
\label{cor:jumps:abstract}
In addition to the assumption of Lemmas~\ref{prop:Burgers}, assume that $\jump{w}$ and $\mean{w}$ satisfy the bounds \eqref{eq:J:M:rough}. Let  $\zl(t)$ and $\kl(t)$ be as defined in Lemma~\ref{lem:zminus-on-shock}.
In addition, suppose that there exists $\Rsf = \Rsf(\kappa,\bb,\cc,\mm)>0$ such that for all $t \in (0,\bar \eps]$ we have
\begin{alignat}{2}
\abs{\tfrac{d}{dt} \jump{w}(t) -\tfrac{d}{dt}  \jump{\wb}(t)} \leq2 \Rsf , \qquad  && \abs{\tfrac{d}{dt} \mean{w}(t) - \tfrac{d}{dt} \mean{\wb}(t)} \leq \Rsf   
\label{eq:dt:jumps:on:shock:are:close}
\,.
\end{alignat}
Then, assuming that $\bar \eps$ is sufficiently small with respect to $\kappa, \bb, \cc$ and $\mm$, we have that 
\begin{alignat}{2}
&\sabs{\tfrac{d}{dt} \zl(t) + \tfrac{27 \bb^ {\frac{9}{2}} }{4 \kappa^2} t^{\frac 12} } \leq  C t \,, \qquad  && \sabs{\tfrac{d}{dt}\kl(t) - \tfrac{48 \bb^{\frac 92}}{\kappa^3} t^{\frac 12}} \leq C t \,,
\label{eq:dt:zl:kl:on:shock}
\end{alignat}
for all $t\in (0,\bar \eps]$, where $C   = C (\kappa,\bb, \cc,\mm)>0$ is a constant. 

In addition to \eqref{eq:dt:jumps:on:shock:are:close}, if we are also given that 
\begin{alignat}{2}
\abs{\tfrac{d^2}{dt^2} \jump{w}(t) -\tfrac{d^2}{dt^2}  \jump{\wb}(t)} \leq  2 \Rsf^* t^{-1} , \qquad  && \abs{\tfrac{d^2}{dt^2} \mean{w}(t) - \tfrac{d^2}{dt^2} \mean{\wb}(t)} \leq  \Rsf^* t^{-1} 
\label{eq:dt:dt:jumps:on:shock:are:close}
\,,
\end{alignat}
for a constant $\Rsf^*=\Rsf^*(\kappa,\bb,\cc,\mm) > 0$. Then, by possibly further reducing the value of $\bar \eps$ we also have the estimates 
\begin{alignat}{2}
&\sabs{\tfrac{d^2}{dt^2} \zl(t) + \tfrac{27 \bb^{\frac 92}}{8 \kappa^2} t^{-\frac 12} } \leq C   \,, \qquad  
&& \sabs{\tfrac{d^2}{dt^2} \kl(t) - \tfrac{24 \bb^{\frac 92}}{ \kappa^3} t^{-\frac 12} } \leq C  \,,
\label{eq:dt:dt:zl:kl:on:shock}
\end{alignat}
where $C   = C (\kappa,\bb, \cc,\mm)>0$ is a constant. 
\end{corollary}
\begin{proof}[Proof of Corollary~\ref{cor:jumps:abstract}]
Recall that by assumption the bound \eqref{eq:J:M:rough} holds, and thus by Lemma~\ref{lem:zminus-on-shock} we have the estimate \eqref{eq:zl:and:kl:on:shock:L:infinity}.
The assumption \eqref{eq:dt:jumps:on:shock:are:close} and the bound \eqref{eq:dt:jump} imply that 
\begin{align}
\sabs{\tfrac{d}{dt} \jump{w} - \bb^{\frac 32} t^{-\frac 12}} + \sabs{\tfrac{d}{dt} \mean{w}} \leq
3  {\mm^4} +3 \Rsf \,, 
\label{eq:spanish:sparkling:wine}
\end{align}
and thus the right sides of \eqref{eq:xi:t:derivatives:zl} and \eqref{eq:xi:t:derivatives:kl} are  $\OO(t^{\frac 32})$. For the bound on the time derivative of $\zl$, we appeal to \eqref{eq:xi:t:derivatives:zl}, which gives
\begin{align*}
\abs{\tfrac{d}{dt} \zl  + \tfrac{27 \jump{w}^2}{16 \mean{w}^2} \tfrac{d}{dt} \jump{w} - \tfrac{9 \jump{w}^3}{8\mean{w}^3} \tfrac{d}{dt} \mean{w}  } \leq C t^{\frac 32}\,.
\end{align*}
Incorporating into the above estimate the bounds \eqref{eq:spanish:sparkling:wine} and \eqref{eq:J:M:rough}, we arrive at the $\zl$ bound in~\eqref{eq:dt:zl:kl:on:shock}.
The time derivative of $\kl$ is bounded by appealing to \eqref{eq:xi:t:derivatives:kl},
which yields
\begin{align*}
\abs{e^{\kl} \tfrac{d}{dt} \kl  - \tfrac{12 \jump{w}^2}{\mean{w}^3} \tfrac{d}{dt} \jump{w} + \tfrac{12 \jump{w}^3}{\mean{w}^4} \tfrac{d}{dt} \mean{w}  } \leq C t^{\frac 32}\,.
\end{align*}
Using \eqref{eq:spanish:sparkling:wine}, \eqref{eq:kl:on:shock:L:infinity}, and \eqref{eq:J:M:rough}, the $\kl$ bound in~\eqref{eq:dt:zl:kl:on:shock} now follows.

In order to prove \eqref{eq:dt:dt:zl:kl:on:shock}, we first note that assumption \eqref{eq:dt:dt:jumps:on:shock:are:close} and the bound \eqref{eq:dt:dt:jump} imply that
\begin{align}
 \sabs{\tfrac{d^2}{dt^2} \jump{w} + \tfrac 12 \bb^{\frac 32} t^{-\frac 32}}  
 + 
 \sabs{\tfrac{d^2}{dt^2} \mean{w}}
\leq 3 \left(  {5 \mm^4} + \Rsf^*\right) t^{-1}  \,.
\label{eq:italian:sparkling:wine}
\end{align} 
Next, we implicitly differentiate \eqref{eq:zminus:kminus} a second time, to obtain
\begin{align}
&\frac{d^2}{dt^2} \left(
\begin{matrix}
 \zl \\
 \el
\end{matrix}\right) 
+
 \left(\begin{matrix}
\p_{\zl} {\mathcal E}_1   &  \p_{\el} {\mathcal E}_1     \\
\p_{\zl} {\mathcal E}_2   &  \p_{\el} {\mathcal E}_2       
\end{matrix}\right)^{-1} 
\left(\begin{matrix}
\p_{\vl} {\mathcal E}_1   &  \p_{\vr} {\mathcal E}_1 \\
\p_{\vl} {\mathcal E}_2   &  \p_{\vr} {\mathcal E}_2  
\end{matrix}\right)
\frac{d^2}{dt^2}
\left(\begin{matrix}
 \vl \\
  \vr
\end{matrix}\right) \notag\\ 
&\qquad  + \left(\begin{matrix}
\p_{\zl} {\mathcal E}_1   &  \p_{\el} {\mathcal E}_1     \\
\p_{\zl} {\mathcal E}_2   &  \p_{\el} {\mathcal E}_2       
\end{matrix}\right)^{-1} 
\left(\begin{matrix}
\p_{\vl\vl} {\mathcal E}_1 &   \p_{\vl\vr} {\mathcal E}_1  & \p_{\vr\vr} {\mathcal E}_1   \\
\p_{\vl\vl} {\mathcal E}_2 &   \p_{\vl\vr} {\mathcal E}_2  & \p_{\vr\vr} {\mathcal E}_2   \\
\end{matrix}\right)
\left(\begin{matrix}
(\frac{d}{dt} \vl)^2   \\
2 \frac{d}{dt} \vl \frac{d}{dt} \vr \\
(\frac{d}{dt} \vr)^2  
\end{matrix}\right)
\notag\\
&\qquad = - \left(\begin{matrix}
\p_{\zl} {\mathcal E}_1   &  \p_{\el} {\mathcal E}_1     \\
\p_{\zl} {\mathcal E}_2   &  \p_{\el} {\mathcal E}_2       
\end{matrix}\right)^{-1} 
\left(\begin{matrix}
\p_{\zl\zl} {\mathcal E}_1 &    \p_{\zl\kl} {\mathcal E}_1  & \p_{\kl\kl} {\mathcal E}_1   \\
\p_{\zl\zl} {\mathcal E}_2 &    \p_{\zl\kl} {\mathcal E}_2  & \p_{\kl\kl} {\mathcal E}_2
\end{matrix}\right)
\left(\begin{matrix}
(\frac{d}{dt} \zl)^2   \\
2 \frac{d}{dt} \zl \frac{d}{dt} \kl \\
(\frac{d}{dt} \kl)^2  
\end{matrix}\right)
\notag\\
&\qquad  - 2\left(\begin{matrix}
\p_{\zl} {\mathcal E}_1   &  \p_{\el} {\mathcal E}_1     \\
\p_{\zl} {\mathcal E}_2   &  \p_{\el} {\mathcal E}_2       
\end{matrix}\right)^{-1} 
\left(\begin{matrix}
\p_{\zl\vl} {\mathcal E}_1 
& \p_{\zl\vr} {\mathcal E}_1     
&\p_{\kl\vl} {\mathcal E}_1 
& \p_{\kl\vr} {\mathcal E}_1 
\\
\p_{\zl\vl} {\mathcal E}_2 
& \p_{\zl\vr} {\mathcal E}_2     
&\p_{\kl\vl} {\mathcal E}_2 
& \p_{\kl\vr} {\mathcal E}_2 
\end{matrix}\right)
\left(\begin{matrix}
\frac{d}{dt} \zl \frac{d}{dt} \vl \\
\frac{d}{dt} \zl \frac{d}{dt} \vr \\  
\frac{d}{dt} \kl \frac{d}{dt} \vl \\
\frac{d}{dt} \kl \frac{d}{dt} \vr
\end{matrix}\right)
\,.
\label{eq:want:to:vomit}
\end{align}
By appealing to \eqref{eq:Matrix:Masturbation:1}, 
\eqref{eq:xi:t:derivatives:zl:kl}, and \eqref{eq:J:M:rough}, similarly to \eqref{eq:xi:derivatives:zl:zl} we deduce that the right side of \eqref{eq:want:to:vomit} equals
\begin{align}
&\left(\begin{matrix}
 \frac{3\mean{w}}{16 \jump{w}}  \frac{d}{dt} \zl \frac{d}{dt} \kl  +  \frac{\mean{w}^2}{16 \jump{w}}(\frac{d}{dt} \kl)^2  \\
0  
 \end{matrix} \right)
+  \OO\left( \left(| \tfrac{d}{dt} \zl| + | \tfrac{d}{dt} \kl|\right)^2 \right)
\notag\\
&+
\left( \begin{matrix}  
- \frac{8}{\mean{w}} \tfrac{d}{dt} \mean{w}  +  \frac{2 \jump{w}^2 - 2 \mean{w}^2}{\mean{w}^2 \jump{w}}  \tfrac{d}{dt} \jump{w} 
& \frac{27 \jump{w}^2 - 12 \mean{w}^2}{32 \jump{w} \mean{w}}  \tfrac{d}{dt} \jump{w} + \frac{3}{8} \tfrac{d}{dt} \mean{w} \\
- \frac{24 \jump{w}}{\mean{w}^3} \tfrac{d}{dt} \jump{w} 
&- \frac{21 \jump{w}}{2 \mean{w}^2} \tfrac{d}{dt} \jump{w}  - \frac{12}{\mean{w}} \tfrac{d}{dt} \mean{w}   
\end{matrix}\right) 
\left(\begin{matrix}
\tfrac{d}{dt} \zl\\
\tfrac{d}{dt} \kl
\end{matrix}\right)
\notag\\
&+ \OO \left( \jump{w}^2 \left( \sabs{\tfrac{d}{dt}  \jump{w}} + \sabs{\tfrac{d}{dt}  \mean{w}}\right) \left(| \tfrac{d}{dt} \zl| + | \tfrac{d}{dt} \kl|\right)  \right)
\,.
\label{eq:want:to:vomit:again}
\end{align}
Similarly, one may verify that the sum of the last two terms on the left side of \eqref{eq:want:to:vomit} is given by
\begin{align}
&\left(\begin{matrix}
\frac{27 \jump{w}^2}{16 \mean{w}^2} \tfrac{d^2}{dt^2} \jump{w} -  \frac{9 \jump{w}^3}{8 \mean{w}^3}  \tfrac{d^2}{dt^2} \mean{w}  \\
- \frac{12 \jump{w}^2}{\mean{w}^3} \tfrac{d^2}{dt^2} \jump{w}  + \frac{12 \jump{w}^3}{\mean{w}^4} \tfrac{d^2}{dt^2} \mean{w}   
\end{matrix}\right) + 
\left(\begin{matrix}
 \frac{9 \jump{w}}{4 \mean{w}^2} ( \tfrac{d}{dt} \jump{w})^2 + \frac{27  \jump{w}^2}{2 \mean{w}^3} \tfrac{d}{dt} \jump{w} \tfrac{d}{dt} \mean{w} \\
- \frac{24 \jump{w}}{\mean{w}^3} ( \tfrac{d}{dt} \jump{w})^2 + \frac{72 \jump{w}^2}{ \mean{w}^4} \tfrac{d}{dt} \jump{w} \tfrac{d}{dt} \mean{w} 
\end{matrix}\right) \notag\\
&
+  \OO\left( \jump{w}^3 \left( \sabs{\tfrac{d}{dt}  \jump{w}} + \sabs{\tfrac{d}{dt}  \mean{w}}\right) \right)
\label{eq:want:to:vomit:again:and:again}
\end{align}
where the implicit constants only depend on $\kappa, \bb, \cc$, and $\mm$.

To conclude we use the bounds  \eqref{eq:J:M:rough},  \eqref{eq:spanish:sparkling:wine}, \eqref{eq:italian:sparkling:wine}, \eqref{eq:zl:and:kl:on:shock:L:infinity}, and \eqref{eq:dt:zl:kl:on:shock}
in the equality given by \eqref{eq:want:to:vomit}, \eqref{eq:want:to:vomit:again}, and \eqref{eq:want:to:vomit:again:and:again}, to arrive at 
\begin{align}
\abs{\tfrac{d^2}{dt^2} \zl + \tfrac{27\jump{w} }{16 \mean{w}^2} \left(2 (\tfrac{d}{dt} \jump{w})^2 + \jump{w}  \tfrac{d^2}{dt^2} \jump{w} \right)} 
&\leq  C  
\label{eq:want:to:vomit:zl:twice}
\end{align}
and by also appealing to $\frac{d^2}{dt^2} \kl = e^{-\kl} \frac{d^2}{dt^2} \el - (\frac{d}{dt}\kl)^2$ we obtain
\begin{align}
\abs{e^{\kl} \tfrac{d^2}{dt^2} \kl - \tfrac{12 \jump{w}}{\mean{w}^3}  \left( 2 (\tfrac{d}{dt} \jump{w})^2 + \jump{w} \tfrac{d^2}{dt^2} \jump{w} \right) } 
&\leq  C  t^{\frac 12} 
\,,
\label{eq:want:to:vomit:kl:twice}
\end{align}
where $C  = C  (\kappa,\bb,\cc,\mm)>0$. To conclude, we combine \eqref{eq:want:to:vomit:zl:twice}--\eqref{eq:want:to:vomit:kl:twice} with the precise estimates for $\jump{w}$ and its first two time derivatives, cf.~\eqref{eq:J:M:rough}, \eqref{eq:spanish:sparkling:wine}, and \eqref{eq:italian:sparkling:wine} and arrive at \eqref{eq:dt:dt:zl:kl:on:shock}.
\end{proof}

\subsection{Transport structure, spacetime regions, and characteristic families}
\label{sec:transport}
 
\subsubsection{A new form of the $w$ and $z$ equations}
We first observe that using \eqref{xland-k} and recalling that $ c= \tfrac{1}{2} (w-z)$, we can write the system \eqref{eq:w:z:k:a} as
\begin{subequations} 
\label{eq:wzka}
\begin{align}
\p_t w + \lambda_3 \p_\theta w & = - \tfrac{8}{3}  a w  + \tfrac{1}{4}  c( \p_t k + \lambda _3 \p_\theta k)   \,,  \label{xland-w2} \\
\p_t z + \lambda_1 \p_\theta z & = - \tfrac{8}{3}  a z -   \tfrac{1}{4}  c( \p_t k + \lambda _1 \p_\theta k)   \,,  \label{xland-z2} \\
\partial_t k   + \lambda_2 \p_\theta k & = 0  \,, \label{xland-k2} \\
\partial_t a   + \lambda_2  \p_\theta a & = - \tfrac43 a^2 + \tfrac{1}{3} (w+ z)^2 - \tfrac{1 }{6} (w- z )^2  \,, \label{xland-a2} 
\end{align}
\end{subequations} 
Our iteration scheme will be based on \eqref{eq:wzka}, and in particular on the estimates for $\p_\theta w$ that the specific form of the equations
 \eqref{xland-w2} and  \eqref{xland-z2} provide.     It will be convenient to introduce the vector of unknowns
\begin{align} 
U = (w,z,k,a) \, \label{U}
\end{align} 
 
 \subsubsection{Characteristic families, shock-intersection times, spacetime regions}
  Recalling the definition of the wave speeds \eqref{eq:wave-speeds}, we let $\eta$ denote the $3$-characteristic which satisfies
\begin{subequations} 
\begin{align} 
 \p_t \eta(x,t)  = \lambda_3(\eta(x,t),t)\,, \qquad \eta(x,0) =x \,,
 \label{Flow0}
\end{align} 
for $t\in (0,\bar \eps)$. 
We also define the $1$- and $2$-characteristics as 
\begin{alignat}{2}
\p_s \psi_t (\theta,s) &=  \lambda_1  ( \psi_t (\theta,s),s) \,, \qquad   \psi_t   (\theta,t)&=\theta  \,, \label{Flow1} \\
\p_s \phi_t  (\theta,s)  &=  \lambda_2  ( \phi_t  (\theta,s),s) \,,  \qquad     \phi_t   (\theta,t)&=\theta\,,  \label{Flow2}
\end{alignat}
\end{subequations} 
for $s\in (0,t)$. We note that $\eta$ has a prescribed initial datum at time $0$, while $\pt$ and $\pst$ have a prescribed terminal datum, at time $t$. Moreover, note that as opposed to $\eta$, the characteristics $\pt$ and $\pst$ may cross the shock curve $(\sc(t),t)_{t\in[0,\bar \eps]}$ in a continuous fashion; this will be shown to be possible because $\lambda_1$ and $\lambda_2$ have bounded {\em one-sided} derivatives on the shock.  

\begin{definition} 
\label{def:stopping:times}
For $(\theta,t) \in \TT \times [0,\bar \eps]$ consider the integral curves $ \psi_t  (\theta,s)$ and $\pt (\theta,s)$ defined by the ODEs \eqref{Flow1}--\eqref{Flow2}. If the curves $(\pst(\theta,s),s)_{s\in [0,t]}$ and $(\sc(s),s)_{s\in[0,t]}$, respectively $(\pt(\theta,s),s)_{s\in [0,t]}$ and $(\sc(s),s)_{s\in[0,t]}$, intersect then we define the {\it shock-intersection  times} $\st(\theta,t)$  and $\stt(\theta,t)$ as the  (largest)  time at which 
\begin{align} 
 \psi_t (\theta, \stt (\theta,t)) = \sc( \stt (\theta,t)) \,, \qquad \mbox{and} \qquad 
  \phi_t (\theta, \st (\theta,t)) = \sc( \st (\theta,t)) \,.  
  \label{stoppingtimes}
\end{align}  
If the curves $(\pst(\theta,s),s)_{s\in [0,t]}$ and $(\sc(s),s)_{s\in[0,t]}$, respectively $(\pt(\theta,s),s)_{s\in [0,t]}$ and $(\sc(s),s)_{s\in[0,t]}$, do not intersect, then we overload notation and define $\stt(\theta,t) = \bar \eps$, respectively $\st(\theta,t) = \bar \eps$.
\end{definition} 
Implicit in the above definition is the assumption that if the characteristics $\pst(\theta,\cdot)$ or $\pt(\theta,\cdot)$ intersect the shock curve, then they do so only once; we will indeed prove this holds, due to the {\em transversality} of these characteristics.

\begin{figure}[htb!]
\centering
\begin{tikzpicture}[scale=2.5]
\draw [<->,thick] (0,2.4) node (yaxis) [above] { } |- (3.5,0) node (xaxis) [right] { };
\draw[black,thick] (0,0)  -- (-2,0) ;
\draw[black,thick] (-2,2)  -- (3.5,2) ;
\draw[name path=A,red,ultra thick] (0,0) .. controls (1,0.5) and (2,1.5) .. (3,2);
\draw[red] (3,1.9) node { $\sc$}; 
\draw[name path = B,blue,ultra thick] (0,0) .. controls (0.5,0.85) and (1,1.25) .. (1.5,2) ;
\draw[blue] (1.3,1.9) node { $\sc_2$}; 
\draw[green!40!gray, ultra thick] (0,0) .. controls (.2,1) and (0.4,1.5) .. (.5,2) ;
\draw[green!40!gray] (.4,1.9) node { $\sc_1$}; 
\draw[black] (-2.4,0) node { $t=0$}; 
\draw[black] (-2.4,2) node { $t = \bar \eps$}; 

\draw[black,dotted] (1.8,0)  -- (1.8,1.7);
\draw[black] (-0.1,1.7) node { $t$}; 
\draw[black,dotted] (0,1.7)  -- (1.8,1.7);
\draw[black] (1.8,-0.1) node { $\theta$}; 
\filldraw[black] (1.8,1.7) circle (0.5pt);

\draw[blue] (0.65,0.38) .. controls (0.9,0.77) and (1.5,1.4) .. (1.8,1.7) ;
\draw[blue] (0.35,0) .. controls (0.44,0.12) and (0.52,0.24) .. (0.65,0.38)  ;
\filldraw[blue] (0.65,0.38) circle (0.5pt);
\draw[blue,dotted] (0.65,0.38)  -- (0.65,0);
\draw[blue,dotted] (0.65,0.38)  -- (0,0.38);
\draw[blue] (0.65,-0.1) node { $\sc(\st(\theta,t))$}; 
\draw[blue] (-0.2,0.38) node { $\st(\theta,t)$}; 
\draw[dotted,->,blue] (2,0.5) -- (0.98,0.79);
\draw[blue] (2.3,0.5) node { $\phi_t(\theta,s)$}; 

\draw[red] (-1.5,0) .. controls (0.5,0.9) and (1,1.25) .. (1.8,1.7) ;
\filldraw[red] (-1.5,0) circle (0.5pt) node[anchor=north] {$x$};
\draw[red] (-1.3,0.65) node { $\eta(x,s)$}; 
\draw[dotted,->,red] (-1.3,0.55) -- (-1,0.27);

\draw[green!40!gray] (1.5,1) .. controls (1.65,1.35) .. (1.8,1.7) ;
\draw[green!40!gray] (1.1,0) .. controls (1.24,0.33) and (1.37,0.66) .. (1.5,1) ;

\filldraw[green!40!gray] (1.5,1) circle (0.5pt);
\draw[green!40!gray,dotted] (1.5,1)  -- (1.5,0);
\draw[green!40!gray,dotted] (1.5,1)  -- (0,1);
\draw[green!40!gray] (1.4,-0.1) node { $\sc(\stt(\theta,t))$}; 
\draw[green!40!gray] (-0.2,1) node { $\stt(\theta,t)$}; 
\draw[dotted,->,green!40!gray] (2.45,1) -- (1.7,1.42);
\draw[green!40!gray] (2.5,0.9) node { $\psi_t(\theta,s)$}; 

\end{tikzpicture}
\vspace{-0.2cm}
\caption{\footnotesize Fix a spatial location $(\theta,t)$, just to the left of the given shock curve $\sc$, which is represented in red. The flow $\eta(x,s)$ defined in \eqref{Flow0}, and the label $x$ such that $\eta(x,t) = \theta$, are also represented in red. The flow $\pt(\theta,s)$ defined in \eqref{Flow2}, its associated shock-intersection time $\st(\theta,t)$ from \eqref{stoppingtimes}, and the curve $\sc_2$ from \eqref{stop-time12}, are represented in blue. The flow $\pst(\theta,s)$ defined in \eqref{Flow1}, its associated  shock-intersection time $\stt(\theta,t)$ from \eqref{stoppingtimes}, and the curve $\sc_2$ from \eqref{stop-time12}, are represented in green.}
\label{fig:many:characteristics:1}

\end{figure}
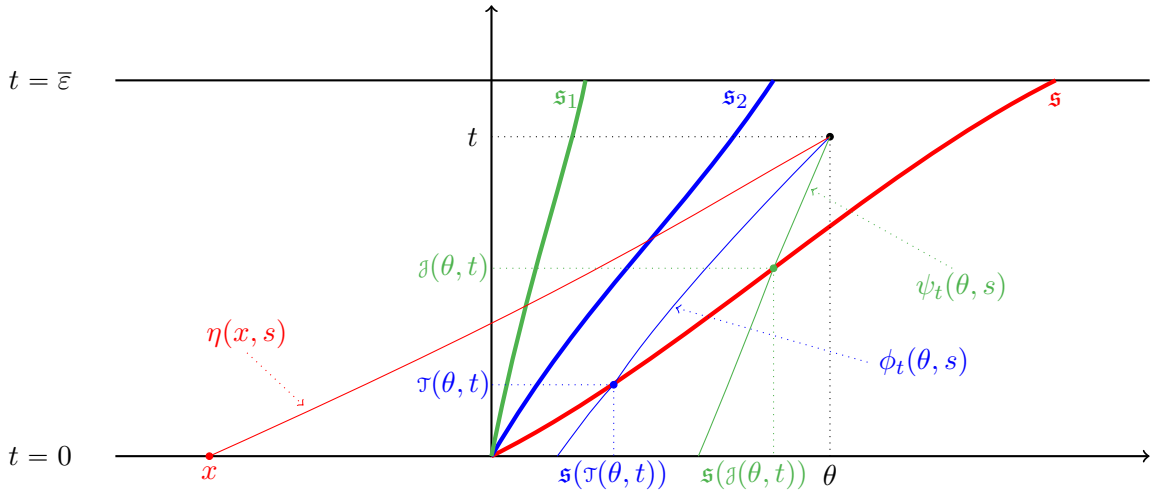

\begin{definition}\label{s1s2}
Define ${\mathring \theta}_1, {\mathring \theta}_2 \in \TT$ implicitly by the  equations $\stt({\mathring \theta}_1,\bar \eps) = 0$ and $\st({\mathring \theta}_2,\bar \eps) = 0$.
For all $t\in [0,\bar \eps]$ we define
\begin{align}
\sc_1(t) = \psi_{\bar\eps}({\mathring \theta}_1,t) \,, \qquad \mbox{and} \qquad 
\sc_2(t) = \phi_{\bar\eps}({\mathring \theta}_2,t) \,.\label{stop-time12}
\end{align}
In particular, $\sc_1(0) = \sc_2(0) = 0$, and $\stt(\sc_1(t),t) = \st(\sc_2(t), t ) = 0$.
The spacetime curves  $\sc_1(t)$,  $\sc_2(t)$,  and $\sc(t)$, divide the spacetime region $ \mathcal{D} _{\bar \eps}$ into four regions with
distinct behavior.
We also define the  sets
\begin{align*} 
\mathcal{D} ^{z} _{\bar \eps} & = \{ (\theta, s) \in \mathcal{D}_{\bar \eps} \colon \sc_1(s) < \theta < \sc(s)\,, s\in(0,\bar\eps]  \} \,, \\
\mathcal{D} ^{k} _{\bar \eps} & = \{ (\theta, s) \in \mathcal{D}_{\bar \eps} \colon \sc_2(s) < \theta < \sc(s)\,, s\in(0,\bar\eps]  \} \,.
\end{align*} 
\end{definition}
Implicit in the above definition is the assumption that the points ${\mathring \theta}_1$ and ${\mathring \theta}_2$ exist, and are uniquely defined; we will indeed prove that this holds, due to the {\em monotonicity} of $\pst(\theta,s)$ and $\pt(\theta,s)$ with respect $\theta$, and the the regularity of these curves with respect to $y$ and $s$.

\begin{definition} 
It is convenient to define  the vectors
\begin{align} 
U =(w,z,k,c,a) \qquad \mbox{and} \qquad  \uu(t) =  (w,z,k,c)(\sc(t)^-,t)  = (w_-,\zl,\kl,\cl)(t)  \,.  \label{uu}
\end{align} 
\end{definition} 

\begin{remark}[Notation for derivatives]
Throughout the remainder of manuscript we shall interchangeably use the following notations for the derivatives of various functions $f$ with respect to the Lagrangian label $x$ or the Eulerian variable $\theta$: $\p_x f \leftrightarrow f_x, \p_{x}^2 f \leftrightarrow f_{xx}, \p_\theta f \leftrightarrow f_\theta, \p_{\theta}^2 f \leftrightarrow f_{\theta\theta}$. Similarly, we shall sometimes denote time derivatives as $\p_t f \leftrightarrow f_t$. Derivatives for function restricted to the shock curve, shall be denoted as $\frac{d}{dt} ( f(\sc(t),t)) = \dot f|_{(\sc(t),t)}$; this notation for instance shall be used for the function $\uu$ defined in \eqref{uu}.
\end{remark}

\subsubsection{Identitities up to the first derivative for   $w$, $z$, $k$, and $a$}
 There are particularly useful forms of the equations for $w$, $z$, $k$, and $a$ and their first derivatives.  These identities will be used both for designing 
 a simple iteration scheme for the construction of unique solutions, and also for second derivative estimates in Section \ref{sec:C2}.
 
 \vspace{.05in}
 \noindent{\bf Identies for $w$.}
Equation \eqref{xland-w2} can then be written as
\begin{align} 
\tfrac{d}{dt} (w \circ \eta) =  \tfrac{1}{4}   c \circ \eta \tfrac{d}{dt} (k \circ \eta) - \tfrac{8}{3}  (a w) \circ \eta \,. \label{w0}
\end{align} 
Differentiating this equation,  we find that
\begin{align} 
\tfrac{d}{dt} (w_\theta \circ \eta \ \eta_x ) & = \tfrac{1}{4}   c \circ \eta \tfrac{d}{dt} (k_\theta \circ \eta \ \eta_x)
+ \tfrac{1}{4}   c_\theta \circ \eta \ \eta_x ( k_t + \lambda _3  k_\theta) \circ \eta - \tfrac{8}{3} \p_\theta (aw) \circ \eta \ \eta_x \notag \\
& = \tfrac{1}{4} \tfrac{d}{dt} (c \circ \eta \ k_\theta \circ \eta \ \eta_x) - \tfrac{1}{4}   ( c_t + \lambda _3 c_\theta) \circ \eta \ k_\theta \circ \eta \ \eta_x 
\notag\\
& \qquad\qquad
+ \tfrac{1}{4}   c_\theta \circ \eta \ \eta_x (\p_t k + \lambda _3 \p_\theta  k) \circ \eta - \tfrac{8}{3} \p_\theta (aw) \circ \eta \ \eta_x \notag \\
& = \tfrac{1}{4} \tfrac{d}{dt} \bigl( (c k_\theta) \circ \eta \ \eta_x\bigr) + \tfrac{1}{6}   \bigl(c k_\theta(z_\theta +c_\theta +4a) \bigr) \circ \eta  \ \eta_x 
 - \tfrac{8}{3} (aw)_\theta \circ \eta \ \eta_x \,. \label{prelim-dtwx}
\end{align}

To obtain the last equality, we have used that \eqref{xland-sigma} can be written as
\begin{align*} 
\p_t c + \lambda _3 \p_\theta  c = -  \tfrac{2}{3}  c \p_\theta  z - \tfrac{8}{3} c a \,,
\end{align*} 
and that
$\p_t k = - \lambda _2\p_\theta  k$ with the fact that $\lambda _3 - \lambda _2 =  \frac{2}{3}  c$.   Integrating \eqref{prelim-dtwx} in time,
we obtain that
\begin{align} 
w_\theta \circ \eta  =\tfrac{ w_0'}{ \eta_x} + \tfrac{1}{4} (  c k_y) \circ \eta 
+\tfrac{1}{\eta_x} \!\! \int_0^t \Bigl(  \tfrac{1}{6}   c k_\theta (z_\theta + c_y+4a) - \tfrac{8}{3} \p_\theta (aw) \Bigr)\! \circ \!\eta  \ \eta_xdt' \,.
\label{prelim-wx}
\end{align} 
We wish to emphasize that although \eqref{xland-w} appears to have derivative loss on the right side, the structure of 
\eqref{xland-w2} leads to the identity \eqref{prelim-wx} which shows that there is, in fact, no such loss incurred.

Notice that by expanding the time derivative in \eqref{w0} and using \eqref{xland-k2}, we find that
\begin{align*} 
\p_t w \circ \eta = - w_\theta \circ \eta \ \lambda _3 \circ \eta + \tfrac{1}{6} c^2 k_\theta \circ \eta - \tfrac{8}{3}  a w \circ \eta
\end{align*} 
It follows that
\begin{align} 
(\p_t w+ \dot\sc \p_\theta w) \circ \eta
&= (\dot\sc - \lambda_3 \circ \eta) w_\theta \circ \eta  + \tfrac{1}{6} c^2 k_\theta \circ \eta - \tfrac{8}{3}  a w \circ \eta \notag \\
& = \tfrac{ w_0'}{ \eta_x}  (\dot\sc - \lambda_3 \circ \eta)+  \tfrac{1}{4} (  c k_\theta) \circ \eta  (\dot\sc - \lambda_3 \circ \eta)
+ \tfrac{1}{6} c^2 k_\theta \circ \eta - \tfrac{8}{3}  a w \circ \eta \notag \\
&\qquad 
+\tfrac{(\dot\sc - \lambda_3 \circ \eta)}{\eta_x} \!\! \int_0^t \Bigl(  \tfrac{1}{6}   c k_\theta (z_\theta + c_\theta + 4a) - \tfrac{8}{3} \p_\theta (aw) \Bigr)\! \circ \!\eta  \ \eta_xdt' \,. \label{for-dt2-sc-bound}
\end{align}

\vspace{.05in}
 \noindent{\bf Identies for $z$ and $k$.}
 Equation \eqref{xland-z2} can then be written as
\begin{align} 
\tfrac{d}{ds} (z \circ \pst) =  - \tfrac{1}{4}   c \circ \pst \tfrac{d}{ds} (k \circ \pst) - \tfrac{8}{3}  (a z) \circ \pst \,. \label{z0}
\end{align} 
Differentiating \eqref{z0}, a  similar identity to \eqref{prelim-wx} holds for $\p_\theta  z$.   The analogous 
computation to \eqref{prelim-dtwx} shows that
\begin{align} 
\tfrac{d}{ds} (z_\theta \circ \psi_t \p_\theta  \psi_t ) 
& = -\tfrac{1}{4} \tfrac{d}{ds} \bigl( (c  \ k_\theta) \circ \psi_t \p_\theta  \psi_t\bigr) - \Bigl(  \tfrac{1}{12}   c k_\theta (w_\theta + z_\theta + 8a) + \tfrac{8}{3} \p_\theta (a z) \Bigr)\! \circ\! \psi_t  \ \p_\theta  \psi_t \,, \label{prelim-dtwz}
\end{align} 
and thus, upon integration in time from $\stt(\theta,t)$ to $t$,  we find that
\begin{subequations} 
\label{prelim-zx}
\begin{align} 
z_\theta (y,t) & =\Bigl( \left( z_\theta(\sc(\stt),\stt) + \tfrac{1}{4} (  c k_\theta) (\sc(\stt),\stt)\right)   \p_\theta \psi_t ( \sc(\stt),\stt) 
+ \mathcal{F} _{z_\theta}(U,\pst,\stt)\Bigr)  (y,t) \,,  \\
 \mathcal{F} _{z_\theta}  (U,\pst,\stt) &=  - \tfrac{1}{4}   (c k_\theta )(\theta,t) 
- \!\! \int_{\stt(\theta,t)}^t \Bigl(  \tfrac{1}{12}   c k_\theta (w_\theta + z_\theta + 8a) + \tfrac{8}{3} \p_\theta (a z) \Bigr)\! \circ\! \psi_t  \ \p_\theta  \psi_t dt' \,. \label{Fzx}
\end{align} 
\end{subequations} Again, the identity \eqref{prelim-zx} shows that no derivative loss occurs for $\p_\theta  z$ as well.   This formula is not yet in its final form.  We
shall view the given shock curve $(\sc(t),t)$ as a Cauchy surface for both $z$ and $k$.  As such, we shall write the first term on the
right in \eqref{prelim-zx} in terms of the differentiated {\it data on the shock curve},  which we now make precise.

The transport equation \eqref{xland-k2} allows us to write
$\frac{d}{ds} ( k \circ \phi_t) =0$, so that integration from $\st(\theta,t)$ to $t$ shows that for all $(\theta,t) \in \mathcal{D}^k_{\bar \eps}$,
\begin{align} 
k(\theta,t) = k(\sc(\st(\theta,t)),\st(\theta,t)) \,. \label{the-k-identity}
\end{align} 
Differentiation then gives
\begin{align} 
\tfrac{d}{ds} ( \p_\theta k \circ \phi_t \ \p_\theta  \phi_t) =0\,,  \label{kx0}
\end{align} 
and integration using \eqref{Flow2} and \eqref{stoppingtimes}  shows that
\begin{align} 
\p_\theta  k(\theta,t) = \p_\theta  k( \sc( \st(\theta,t) ), \st(\theta,t)) \  \p_\theta  \phi_t( \sc( \st(\theta,t) ), \st(\theta,t)) \,. \label{kx-interior}
\end{align} 

Letting $\dot \kl (t):= \frac{d}{dt} \kl(t) $ denote differentiation along the shock curve, from 
\eqref{xland-k2} we have the coupled system
\begin{subequations} 
\label{systemk-dt}
\begin{align} 
\dot \kl (t) & = \p_t k ( \sc(t),t) + \dot \sc(t)  \p_\theta  k ( \sc(t),t) \,, \\
0& =\p_t k ( \sc(t),t) + \lambda_2 ( \sc(t),t) \p_\theta  k ( \sc(t),t) \,.
\end{align} 
\end{subequations} 
We see that
\begin{align} 
\p_\theta  k( \sc(t),t)  =\frac{ \dot \kl (t) }{ \dot \sc(t) - \lambda_2 ( \sc(t),t)} \,, \label{kx-gen}
\end{align} 
and thus with \eqref{stoppingtimes}, 
\begin{align} 
\p_\theta  k( \sc(\st(\theta,t)),\st(\theta,t))) = \frac{ \dot \kl (\st(\theta,t))) }{ \dot \sc(\st(\theta,t))) -   \p_s \phi_t(\theta, \st(\theta,t) )}  \,. \label{kx-shock}
\end{align} 
Substitution of \eqref{kx-shock} into \eqref{kx-interior} shows that for all $(\theta,t) \in \mathcal{D} ^k_{\bar \eps}$, 
\begin{align} 
\p_\theta  k(\theta,t)  = \frac{ \dot \kl (\st(\theta,t))) }{ \dot \sc(\st(\theta,t))) -   \p_s \phi_t(\theta, \st(\theta,t) )}  \  \p_\theta  \phi_t( \sc( \st(\theta,t) ), \st(\theta,t)) \,. \label{kx-final}
\end{align} 

Once again, we let $\dot \zl (t) $ denote differentiation along the shock curve so that using
\eqref{xland-z},  we obtain the coupled system
\begin{subequations} 
\label{systemz-dt}
\begin{align} 
\dot \zl (t) & = \p_t z ( \sc(t),t) + \dot \sc(t)  \p_\theta  z ( \sc(t),t) \,, \\
(\tfrac{1}{6} c^2 \p_\theta  k - {\tfrac 83 az})(\sc(t),t)  & =\p_t z ( \sc(t),t) + \lambda_1 ( \sc(t),t) \p_\theta  z ( \sc(t),t) \,.
\end{align} 
\end{subequations} 
Thus,
\begin{align} 
\p_\theta  z( \sc(t),t)  = \frac{ \dot \zl (t)  - \frac{1}{6}(c^2 \p_\theta  k)(\sc(t),t) + {\frac 83 (a z)(\sc(t),t)}}{ \dot \sc(t) - \lambda_1 ( \sc(t),t)} \,, \label{zx-gen}
\end{align} 
and hence with \eqref{kx-gen},
\begin{align} 
\p_\theta  z( \sc(\stt),\stt)  = \
 \frac{ \dot \zl (\stt)  
-  \tfrac{1}{6}    \tfrac{ \cl^2 (\stt)\dot \kl (\stt)  }{ \dot \sc(\stt) - \lambda_2 ( \sc(\stt),\stt)}  + {\frac 83  a_-(\stt) z_-(\stt)} }
{ \dot \sc(\stt) - \p_s \psi_t( \theta,\stt)} \,, \label{zx-good}
\end{align} 
where $\stt=\stt(\theta,t)$.
We can now substitute \eqref{kx-shock} and  \eqref{zx-good} into \eqref{prelim-zx} to conclude that 
\begin{align} 
\p_\theta  z &  =
\left( \frac{ \dot \zl (\stt)  
-  \frac{1}{6}   \tfrac{\cl^2(\stt) \dot \kl (\stt) }{ \dot \sc(\stt) - \lambda_2 ( \sc(\stt),\stt)}  + {\frac 83  a_-(\stt) z_-(\stt)} }
{ \dot \sc(\stt) - \p_s \psi_t( \sc(\stt),\stt)} 
+ \frac{1}{4} \frac{\cl(\stt) \dot \kl (\stt) }{ \dot \sc(\stt) - \lambda_2 ( \sc(\stt),\stt)}\right)
 \p_\theta \psi_t ( \sc(\stt),\stt)  +  \mathcal{F} _{z_\theta} 
\,,
\label{zx-final}
\end{align}
for any $(\theta,t) \in \mathcal{D} ^z_{\bar\eps}$.
We define
\begin{align} 
\mathcal{H}_{z_\theta} (\uu,\dot\uu,\pst,\stt):=\left( \frac{ \dot \zl (\stt)  
-  \frac{1}{6}   \tfrac{\cl^2(\stt) \dot \kl (\stt) }{ \dot \sc(\stt) - \lambda_2 ( \sc(\stt),\stt)}  + {\frac 83  a_-(\stt) z_-(\stt)} }
{ \dot \sc(\stt) - \p_s \psi_t( \sc(\stt),\stt)} 
+ \frac{1}{4} \frac{\cl(\stt) \dot \kl (\stt) }{ \dot \sc(\stt) - \lambda_2 ( \sc(\stt),\stt)}\right)
 \p_\theta \psi_t ( \sc(\stt),\stt) \,,   \label{Hzx}
\end{align} 
so that \eqref{zx-final} is concisely written as 
\begin{align} 
\p_\theta  z &  =\mathcal{H}_{z_\theta} (\uu,\dot\uu,\pst,\stt) +   \mathcal{F} _{z_\theta}  (U,\pst,\stt) \,, \label{zx-Final}
\end{align}
with $\mathcal{F} _{z_\theta} $ and  $ \mathcal{H}_{z_\theta}$ given by \eqref{Fzx} and  \eqref{Hzx}, respectively.

\vspace{.05in}
 \noindent{\bf Identies for $a$.}
We next obtain identities for $\p_\theta  a$, first in $ \mathcal{D}_{\bar \eps}$.  We write \eqref{xland-a2} as
$\partial_t a   + \lambda_2  \p_\theta  a  = - \tfrac43 a^2 +  \tfrac{1 }{6} (w^2+ z^2 ) + wz $.  
We consider this equation along the characteristics $\pt$ and integrate from time $s\in[0,t]$ to $t$ to find that
\begin{align} 
a (\theta,t) = a(\pt(\theta,s),s)  + \int_s^t  \Bigl(- \tfrac43  a^2  + \tfrac{1}{6}w^2 + \tfrac{1}{6}z^2 + w  z  \Bigr) \circ \pt  dr \,. \label{a-lag}
\end{align} 
Differentiation shows that
\begin{align} 
\p_\theta  a(\theta,t) = \p_\theta  a(\pt(\theta,s),s)  \p_\theta  \pt (\theta,s) 
+ \int_s^t \p_\theta  \Bigl(- \tfrac43  a^2 + \tfrac{1}{6}w^2 + \tfrac{1}{6}z^2 + w z  \Bigr) \circ \pt \  \p_\theta  \pt  dr \,. \label{dxa-lag}
\end{align}

\subsection{Construction of solutions by an iteration scheme}
\label{sec:construction:iteration}

\subsubsection{Wave speeds, characteristics, and stopping times}
 For each $n \ge 1$, 
 the three wave speeds are given by 
\begin{align}
\lambda_1^{(n)}  =  \tfrac{1}{3} w^{(n)} + z^{(n)}   \,, \qquad 
\lambda_2^{(n)}  = \tfrac{2}{3} w^{(n)}+  \tfrac 23 z^{(n)} \,, \qquad
\lambda_3^{(n)}  =  w^{(n)}+  \tfrac{1}{3}  z^{(n)} \,. \label{eq:cc}
\end{align}

For $n\ge 1$, we define  $\psi_t ^{(n)}  $ and $ \phi_t^{(n)}  $ as flows solving
\begin{subequations} 
\begin{alignat}{2}
\p_s \psi_t^{(n)}  (\theta,s) &=  \lambda_1^{(n)}   ( \psi_t^{(n)}  (\theta,s),s) \,, \qquad   \psi_t ^{(n)}  (\theta,t)&=\theta  \,, \label{flow1} \\
\p_s \phi_t ^{(n)}  (\theta,s)  &=  \lambda_2^{(n)}   ( \phi_t^{(n)}  (\theta,s),s) \,,  \qquad     \phi_t ^{(n)}  (\theta,t)&=\theta\,.  \label{flow2}
\end{alignat}
\end{subequations} 
We next define $\eta^{(n)}$ to be the solution of
\begin{align} 
\p_s \eta ^{(n)} (x,s) =  \lambda_3^{(n)} ( \eta^{(n)} (x,s),s) \,, \qquad \eta^{(n)} (x,0) =x  \,. \label{flow3}
\end{align} 

Using the characteristics $\pt^{(n)} $ and $ \pst ^{(n)} $, we define
the shock-intersection times $\st ^{(n)}(\theta,t)$ and $\stt ^{(n)} (\theta,t)$ as in Definition~\ref{def:stopping:times}.   Similarly, the curves $\sc_1^{(n)}(t)$ and 
$\sc_2^{(n)}(t)$ and the spacetime regions $\mathcal{D} ^{z,(n)} _{\bar \eps}$ and $\mathcal{D} ^{k,(n)} _{\bar \eps}$ are defined just as in Definition~\ref{s1s2}. The rigorous justification of these definitions is provided in Lemma~\ref{lem:12flows}.
 
\subsubsection{Specification of the first iterates}
We begin by defining the first iterate $\eta^{(1)}$  associated to the $3$-characteristic and
 $w ^{(1)}  $ as follows.  First, we set
\begin{align} 
\eta ^{(1)}  (x,s) = \etab(x,s) = x + s w_0(x) \,, \label{eta1}
\end{align} 
and then define 
\begin{align} 
w ^{(1)}(\theta,t) = \wb(\theta,t) =  w_0(\ie^{(1)}(\theta,t)  )\,, \qquad z ^{(1)}  =0 \,, \qquad k ^{(1)}  =0 \,, \qquad  a ^{(1)}  =a_0 \,, \label{wzka1}
\end{align} 
where $ \ie^{(1)}   := (\eta^{(1)} )^{-1} = \etab^{-1}$. 
We also define $\pst^{(1)}$ and $\pt^{(1)}$ via \eqref{flow1}--\eqref{flow2} as the characteristic flows of the velocity fields $\frac 13 \wb$ and respectively $\frac 23 \wb$.

\subsubsection{The iteration scheme for $w ^{(n+1)} $}
We can now state the iteration scheme for all $n \ge 1$.    We set
$$
c ^{(n)} = \tfrac{1}{2} (w ^{(n)} + z ^{(n)} ) \,,
$$
and define $ w ^{(n+1)} $ as  the solution to
\begin{align}
\tfrac{d}{dt} (w ^{(n+1)}  \circ \eta ^{(n)} ) = -{\tfrac{8}{3}} ( a^{(n)}  w^{(n)}  ) \circ \eta ^{(n)} 
+  \tfrac{1}{4}  c^{(n)} \circ \eta ^{(n)}  \tfrac{d}{dt} ( k^{(n)} \circ \eta ^{(n)} )  \,,
\label{dtwn-alt}
\end{align} 
with initial condition $w ^{(n+1)}  \circ \eta ^{(n)} (x,0) = w_0(x)$. 
Integrating in time shows that
\begin{align} 
 w ^{(n+1)} ( \eta ^{(n)}(x,t),t ) = w_0(x) 
 &- {\tfrac{8}{3}}\int_0^t     ( a^{(n)}  w^{(n)}   )( \eta ^{(n)}(x,t')  ,t') dt'  \notag \\
 & + {\tfrac{1}{4}} \int_0^t
 c ^{(n)} ( \eta ^{(n)}(x,t')  ,t')  \tfrac{d}{dt'} \left( k ^{(n)} ( \eta ^{(n)}(x,t')  ,t')\right) dt' \,.    \label{wn-alt}
\end{align} 
It follows that for all  $(\theta,t) \in  \DD_{\bar \eps}$, $w ^{(n+1)}$ is the solution to
\begin{subequations} 
\begin{align}
\p_s w ^{(n+1)} + \lambda_3 ^{(n)}  \p_\theta  w ^{(n+1)}  &=   - \tfrac{8}{3}a^{(n)} w^{(n)} 
+ \tfrac{1}{4}  c ^{(n)} \bigl( \p_t k^{(n)} + \lambda _3^{(n)}  \p_\theta  k^{(n)} \bigr) \,,  \label{xland-w:2} \\
w^{(n+1)} (x,0)&=w_0(x) \,.
\end{align} 
\end{subequations} 

In terms of the restrictions of $w^{(n+1)}$ on the left and right sides of shock curve, i.e. $w^{(n+1)}_-(t) = \lim_{\theta \to \sc(t)^-} w^{(n+1)}(\theta,t)$ and respectively $w^{(n+1)}_+(t) =  \lim_{\theta \to \sc(t)^+} w^{(n+1)}(\theta,t)$, via Lemma~\ref{lem:zminus-on-shock} we define the functions $z^{(n+1)}_-(t)$ and $k^{(n+1)}_-(t) $ as the solutions of the system of equations~\eqref{eq:zminus:kminus:def}
\begin{align}
{\mathcal E}_1(w_-^{(n+1)},w_+^{(n+1)}, z^{(n+1)}_-, \el^{(n+1)}) = {\mathcal E}_2(w_-^{(n+1)},w_+^{(n+1)}, z^{(n+1)}_-, \el^{(n+1)}) = 0 
\label{eq:hate:iterations}
\end{align}
and $k^{(n+1)}_- = \log(1 + \el^{(n+1)})$. 

\subsubsection{The iteration scheme for $a ^{(n+1)} $} 
For all  $n\ge 1$ and $(\theta,t) \in  \DD_{\bar \eps}$, we define $a ^{(n+1)}$ to be the solution of the Cauchy problem
\begin{subequations} 
\begin{align} 
\partial_t a^{(n+1)}   + \lambda_2^{(n)}  \p_\theta   a^{(n+1)} & = - \tfrac43   (a^{(n)} )^2  + \tfrac{1}{3}(w^{(n)} )^2 
+ \tfrac{1}{3}(z^{(n)} )^2 + w ^{(n)} z ^{(n)}  \,, \label{xland-a:2}  \\
a^{(n+1)}(x,0)&=a_0(x) \,.
\end{align}
\end{subequations} 
In view of \eqref{a-lag}, this function is explicitly given by
\begin{align}
a^{(n+1)}(\theta,t) = a_0 (\pt^{(n)}(\theta,0)) + \int_0^t  \Bigl(- \tfrac43  (a^{(n)})^2  + \tfrac{1}{6} (w^{(n)})^2 + \tfrac{1}{6}(z^{(n)})^2 + w^{(n)}  z^{(n)}  \Bigr)( \pt^{(n)}(\theta,s),s)  ds \label{eq:xland-a:2}  \,.
\end{align}

\subsubsection{The iteration scheme for $z ^{(n+1)} $}
For all $n \ge 1$,  and for all $(\theta,t) \in 
\mathcal{D}^{z,(n)} _{\bar \eps} $ we define $ z ^{(n+1)} $ to be the solution of the ODE
\begin{subequations}
\begin{align} 
\tfrac{d}{ds} \bigl( z ^{(n+1)} \circ \pst ^{(n)} \bigr) & = - \tfrac{8}{3}  \bigl(a^{(n)} z^{(n)}\bigr) \circ \pst^{(n)}  - \tfrac{1}{4}  c ^{(n)} \circ \pst^{(n)} 
\tfrac{d}{ds}  \bigl( k ^{(n)} \circ \pst ^{(n)} \bigr)  \,, \label{z-alt} 
\end{align} 
for all $s\in (\stt^{(n)}(\theta,t),t]$, 
with Cauchy data defined on the shock curve by
\begin{align}
z^{(n+1)}(\pst^{(n)}(\theta,\stt^{(n)}(\theta,t)), \stt^{(n)}(\theta,t))
= z^{(n+1)}(\sc(\stt^{(n)}(\theta,t))^-,\stt^{(n)}(\theta,t)) 
= z^{(n+1)}_-( \stt^{(n)}(\theta,t))
\end{align}
\end{subequations}
where the function $z^{n+1}_-$ is defined on the shock curve $(\sc(t),t)_{t\in [0,\bar \eps]}$ as the correct  root of \eqref{eq:hate:iterations} given by Lemma~\ref{lem:zminus-on-shock}. 
In Eulerian variables, we note that the equation \eqref{z-alt} is merely 
\begin{align}
\p_t z^{(n+1)} + \lambda_1^{(n)} \p_\theta  z^{(n+1)} & = - \tfrac{8}{3}  a^{(n)} z^{(n)} - \tfrac{1}{4}  c ^{(n)} ( \p_t k^{(n)} + \lambda _1^{(n)}  \p_\theta  k^{(n)} ) \label{xland-z:2}  
\end{align}
for $(\theta,t) \in \mathcal{D}^{z,(n)} _{\bar \eps} $.
On the other hand, for $(\theta,t) \in (\mathcal{D}^{z,(n)} _{\bar \eps} )^\complement$, we simply define
\begin{align}
z^{(n+1)} (\theta,t) =0
\end{align}
which corresponds to the solution of \eqref{z-alt} with $k^{(n)} \equiv 0$, and Cauchy data at $t=0$ given by $z_0 \equiv 0$.

\subsubsection{The iteration scheme for $k ^{(n)} $}
Having defined $w^{(n+1)} $ and $z^{(n+1)} $, we solve for $\pt^{(n+1)} $ using \eqref{flow2}. In turn, this defines the curve $\sc_2^{(n+1)}$, the shock intersection times $\st^{(n+1)}(\theta,t)$, and the region $\DD_{\bar \eps}^{k,(n+1)}$.

For $n \ge 1$ and  $(\theta,t) \in \mathcal{D}^{k,(n+1)} _{\bar \eps} $, we define $k ^{(n+1)} $ to be the solution of
\begin{subequations}
\begin{align} 
\tfrac{d}{ds} \bigl(k ^{(n+1)} \circ \pt ^{(n+1)} \bigr) & = 0 \,,  \label{k-alt}
\end{align} 
for all $s\in (\st^{(n+1)}(\theta,t),t]$, 
with Cauchy data defined on the shock curve by
\begin{align}
&k^{(n+1)}(\pt^{(n+1)}(\theta,\st^{(n+1)}(\theta,t)), \st^{(n+1)}(\theta,t)) \notag\\
&\qquad = k^{(n+1)}(\sc(\st^{(n+1)}(\theta,t))^-,\st^{(n+1)}(\theta,t)) 
= k^{(n+1)}_-( \st^{(n+1)}(\theta,t))
\end{align}
\end{subequations}
where the function $k^{n+1}_- = \log(1 + \el^{(n+1)})$ is defined on the shock curve $(\sc(t),t)_{t\in [0,\bar \eps]}$ as the correct  root of \eqref{eq:hate:iterations} given by Lemma~\ref{lem:zminus-on-shock}. In Eulerian variables, we note that the equation \eqref{k-alt} is the same as
\begin{align}
\p_t k^{(n+1)} + \lambda_2^{(n+1)} \p_\theta  k^{(n+1)} & = 0  \label{xland-k:2}   
\end{align}
for all $(\theta,t) \in \mathcal{D}^{k,(n+1)} _{\bar \eps}$.
On the other hand, for $(\theta,t) \in (\mathcal{D}^{k,(n+1)} _{\bar \eps})^\complement$, we define
\begin{align}
k^{(n+1)} (\theta,t)& =0 \,,
\end{align}
which is the solution of \eqref{k-alt} with Cauchy data at time $t=0$ given by $k_0 \equiv 0$.

\subsubsection{Alternative forms of the iteration for $w ^{(n+1)} $, $ z ^{(n+1)} $, and $c ^{(n+1)} $}
Using that $\p_t k ^{(n)} = - \lambda _2 ^{(n)} \p_\theta  k ^{(n)}$, 
we can also write \eqref{xland-w:2} and \eqref{xland-z:2} as
\begin{subequations} 
\begin{align} 
\p_t w ^{(n+1)} +  \lambda_3 ^{(n)}  \p_\theta  w ^{(n+1)}  &=   - \tfrac{8}{3}a^{(n)} w^{(n)} 
+ \tfrac{1}{6}  (c ^{(n)})^2 \p_\theta  k^{(n)}  \,, \label{xland-w:3}  \\
\p_t z ^{(n+1)} + \lambda_1 ^{(n)}  \p_\theta  z ^{(n+1)}  &=   - \tfrac{8}{3}a^{(n)} z^{(n)} 
+ \tfrac{1}{6}  (c ^{(n)})^2 \p_\theta  k^{(n)} \,,  \label{xland-z:3} 
\end{align} 
\end{subequations} 
and therefore
\begin{align} 
\p_t {c^{(n+1)}} = - \tfrac{1}{2} \lambda_3^{(n)} \p_\theta w^{(n+1)} - \tfrac{1}{2} \lambda_1^{(n)} \p_\theta z^{(n+1)} - \tfrac{8}{3} a ^{(n)} c ^{(n)}  \,,
\label{cn0}
\end{align} 
which has the equivalent forms
\begin{subequations} 
\begin{align} 
\p_t c ^{(n+1)} + \lambda _2 ^{(n)} \p_\theta  c ^{(n+1)} + \tfrac{1}{2} c^{(n)}  \p_\theta  \lambda_2^{(n+1)}  = - \tfrac{8}{3}a^{(n)} c^{(n)} 
\,, \label{cn} \\
\p_t c ^{(n+1)} + \lambda _3 ^{(n)} \p_\theta  c ^{(n+1)} + \tfrac{2}{3} c^{(n)}  \p_\theta  z^{(n+1)}  = - \tfrac{8}{3}a^{(n)} c^{(n)} 
\,, \label{cn3} \\
\p_t c ^{(n+1)} + \lambda _1 ^{(n)} \p_\theta  c ^{(n+1)} + \tfrac{2}{3} c^{(n)}  \p_\theta  w^{(n+1)}  = - \tfrac{8}{3}a^{(n)} c^{(n)} 
\,. \label{cn1} 
\end{align} 
\end{subequations} 

Although it is not necessary to obtain any estimates, we record at this stage the evolution equation for the specific vorticity given according to \eqref{xland-svort:def} by $\varpi^{(n)} = 4 (w^{(n)} + z^{(n)} - \p_\theta  a^{(n)}) (c^{(n)})^{-2} e^{k^{(n)}}$. By combining \eqref{xland-a:2}, 
\eqref{xland-k:2}, \eqref{xland-w:3}, \eqref{xland-z:3}, and \eqref{cn}, we obtain
\begin{align}
\label{eq:varpi:n}
&\p_t \varpi^{(n+1)} + \lambda_2^{(n)} \p_\theta  \varpi^{(n+1)}
- \tfrac 83  \tfrac{c^{(n)}}{c^{(n+1)}} a^{(n)} \varpi^{(n+1)} 
- \tfrac 43 \left(\tfrac{c^{(n)}}{c^{(n+1)}}) \right)^2 \p_\theta  k^{(n)} e^{k^{(n+1)}}
\notag\\
&=  \left( \tfrac 83 a^{(n)} + \p_\theta  \lambda_2^{(n)}\right) \tfrac{c^{(n)}}{c^{(n+1)}} \left( \varpi^{(n+1)} - \varpi^{(n)}   \right) \notag\\
&\quad+ \left( \tfrac 83 a^{(n)} + \p_\theta  \lambda_2^{(n)}\right) \tfrac{c^{(n)}}{(c^{(n+1)})^2} \varpi^{(n)}  e^{k^{(n+1)}} \left( c^{(n+1)} e^{- k^{(n+1)} }-  c^{(n)}  e^{ -k^{(n)}} \right) \notag\\
&\quad + \left( \tfrac{c^{(n)}}{c^{(n+1)}} - \p_\theta  k^{(n+1)} \right) \varpi^{(n+1)} \p_\theta  \left( \lambda_2^{(n+1)}-\lambda_2^{(n)}\right) 
\notag\\
&\quad - \tfrac{16}{3} \tfrac{c^{(n)}}{(c^{(n+1)})^2} e^{k^{(n+1)}} \p_\theta \left( c^{(n+1)} -  c^{(n)}\right)
+ 4 \tfrac{1}{(c^{(n+1)})^2} e^{k^{(n+1)}} \p_\theta  \lambda_2^{(n)} \left(a^{(n+1)} - a^{(n)}\right)
\,.
\end{align}
At this stage we only remark that if $(w,z,k,a,\varpi)^{(n)}$ were to equal $(w,z,k,a,\varpi)^{(n+1)}$, then the right side of \eqref{eq:varpi:n} vanishes, as is natural.

\subsubsection{The iteration space}

We will prove stability under iteration $n \mapsto n+1$ of the following bound
\begin{align}
\label{eq:w:z:k:a:boot:*}
 |\!|\!| (w^{(n)} - w^{(1)},z^{(n)},k^{(n)},a^{(n)})  |\!|\!| _{\bar \eps} \leq 1 
\end{align}
where the norm $ |\!|\!| \cdot |\!|\!| _{\bar \eps}$ is as defined in \eqref{eq:the:norm}. For convenience of the reader, we recall that \eqref{eq:w:z:k:a:boot:*} means 
\begin{subequations}
\label{eq:w:z:k:a:boot}
\begin{align}
 & \sabs{ w^{(n)} (\theta,t) - w ^{(1)}   (\theta,t)} \leq R_1 t  \label{eq:w:boot} \\
& \sabs{ \p_\theta  w^{(n)} (\theta,t) - \p_\theta  w ^{(1)}   (\theta,t)} \leq  R_2  \left( \bb^3 t^3 + (\theta- \sc(t))^2 \right)^ {-\frac{1}{6}}    \label{eq:w:dx:boot} \\
& \sabs{ z^{(n)} (\theta,t)} \leq  R_3 t^{\frac 32}   \label{eq:z:boot} \\
& \sabs{ \p_\theta  z^{(n)} (\theta,t)  } \leq R_4 t^{\frac 12}  \label{eq:z:dx:boot}  \\
& \sabs{ k^{(n)} (\theta,t) } \leq R_5 t^{\frac 32}   \label{eq:k:boot}  \\
& \sabs{\p_\theta  k^{(n)} (\theta,t) } \leq R_6 t^{\frac 12}  \label{eq:k:dx:boot}   \\
& \sabs{a^{(n)} (\theta,t) } + \abs{\p_\theta  a^{(n)} (\theta,t)} \leq R_7 \label{eq:a:boot} 
\,,
\end{align}
\end{subequations}
for all $(\theta,t) \in \DD_{\bar \eps}$, where
\begin{align}
\label{eq:R:def}
{R_1 =  50 \mm^2}\,,
\quad 
R_2 = \mm^3\,,
\quad
{R_3} =  R_4  = \mm \,,
\quad
R_5  = R_6 = {\mm^{\frac 12}}\,,
\quad 
R_7  = 4\mm
\,.
\end{align}

\begin{lemma} \label{lem:dtwn}
Assume that $(w^{(n)},z^{(n)},k^{(n)},a^{(n)}) \in {\mathcal X}_{\bar \eps}$.  Then  for all $(\theta,t) \in \mathcal{D}_{\bar \eps}$,
\begin{align} 
\sabs{ \p_t w ^{(n)}(\theta,t) -  \p_t w ^{(1)}(\theta,t) } 
\leq {3\mm^4} \left( (\theta - \sc(t))^2  +  t^3\right)^ {-\frac{1}{6}}   
\,.   \label{eq:w:dt:boot}
 \end{align} 
\end{lemma} 
\begin{proof}[Proof of Lemma \ref{lem:dtwn}]
Using the identity \eqref{xland-w:3} and the fact that
$\p_t w ^{(1)}  + w^{(1)} \p_\theta  w ^{(1)}  =0$, we have that
\begin{align} 
 \p_t w ^{(n)} -  \p_t w ^{(1)} 
&=  -   \lambda_3 ^{(n-1)} (  \p_\theta w ^{(n)} -  \p_\theta  w^{(1)}) - ( \lambda_3^{(n-1)} - w^{(1)}  ) \p_\theta  w^{(1)}  
\notag\\
& -  \tfrac{8}{3}a^{(n-1)} w^{(n-1)} 
+ \tfrac{1}{6} ( c ^{(n-1)})^2 \p_\theta k^{(n-1)} 
 \,, \label{dtwn-pre1}
\end{align} 
Now from \eqref{thegoodstuff1} and \eqref{eq:w:z:k:a:boot},  we have that for $\bar \eps$ taken sufficiently small, 
$$\sabs{( \lambda_3^{(n-1)} - w^{(1)}  ) \p_\theta  w^{(1)}  } \le  R_1\,, \ \    \tfrac{8}{3} \sabs{ a^{(n-1)} w^{(n-1)} } \le 3 \mm R_7\,, \ \ \text{and} \ \ 
 \tfrac{1}{6} \sabs{ ( c ^{(n-1)})^2 \p_\theta k^{(n-1)}} \les t^ {\frac{1}{2}} \,. $$
Then from   \eqref{eq:w:dx:boot} and with $\bar \eps$ taken even small, we have that
\begin{align*} 
\sabs{ \p_t w ^{(n)} -  \p_t w ^{(1)} } & \le 2 \mm  R_2  \left( \bb^3 t^3+ ( \theta - \sc(t))^2 \right)^ {-\frac{1}{6}} + 3\mm R_7 + R_1  \\
&\le 3 \mm^4  \left(\bb^3 t^3+ (\theta - \sc(t))^2 \right)^ {-\frac{1}{6}}\,,
\end{align*} 
where we have used \eqref{eq:b:m:ass}, and that $t \le \bar \eps$.  Hence, 
we obtain the bound  \eqref{eq:w:dt:boot}.
\end{proof} 

\subsubsection{The behavior of $w ^{(n)}$, $ z ^{(n)} $, and $k ^{(n)} $ on the shock curve}

\begin{lemma}\label{lem:jumpn}   Assume that $(w^{(n-1)},z^{(n-1)},k^{(n-1)},a^{(n-1)}) \in {\mathcal X}_{\bar \eps}$ and that $w ^{(n)}  \in {\mathcal X}_{\bar \eps}$. Then
 for all $t \in (0,\bar \eps]$ we have
 \begin{subequations}
\begin{alignat}{2}
&\sabs{ \jump{ w ^{(n)}(t) }  - \jump{ w ^{(1)}(t)  }}   \le  2 R_1 t  \,, \qquad && 
\sabs{ \mean{ w ^{(n)}(t) } - \mean{ w ^{(1)}(t) }   }   \le R_1 t  \,,     \label{1krod}  \\
& \abs{\tfrac{d}{dt} \jump{w^{(n)}} (t) -\tfrac{d}{dt}  \jump{w ^{(1)}  }(t)} \leq 2 R_1, \qquad  
&& \abs{\tfrac{d}{dt} \mean{w^{(n)} }(t) - \tfrac{d}{dt} \mean{w ^{(1)}  }(t)} \leq  R_1 \label{2krod} \,,
\end{alignat}
\end{subequations}
where $R_1$ is as defined in \eqref{eq:R:def}. In particular, in view of \eqref{eq:basic:jump} and \eqref{eq:dt:jump}, we have that 
\begin{subequations}
\label{eq:jump:mean:sharp}
\begin{alignat}{2}
&\sabs{\jump{w^{(n)}}(t) - 2 \bb^{\frac 32} t^{\frac 12}} \leq 3 \mm^3 t  \,, \qquad && 
\abs{\mean{w^{(n)}}(t)-\kappa} \leq  \tfrac 12 \mm^4 t \,,
\label{eq:J:M:sharp:bound} 
 \\
& \abs{\tfrac{d}{dt} \jump{w^{(n)}} (t) - \bb^{\frac 32} t^{-\frac 12}} \leq 3 \mm^4 , \qquad  
&& \abs{\tfrac{d}{dt} \mean{w^{(n)}}(t) } \leq  2 \mm^4  
\label{eq:J:M:sharp:bound:dt}  \,,
\end{alignat}
\end{subequations}
for all $t \in (0,\bar \eps]$.
\end{lemma} 
\begin{proof}[Proof of Lemma \ref{lem:jumpn}]
By assumption, $w^{(n)}$ satisfies the bound \eqref{eq:w:boot}, and so the inequalities in~\eqref{1krod}   follow. In order to prove \eqref{2krod}, we shall use that  $\sabs{ \jump{ w ^{(n-1)}(t) }  } \le \sabs{ \jump{ w ^{(1)}(t) }  }+\sabs{ \jump{ w ^{(n-1)}(t)- w ^{(1)}(t) }  } $, and hence 
by \eqref{eq:basic:jump} and  \eqref{eq:w:boot},  
\begin{align} 
\sabs{ \jump{ w ^{(n-1)}(t) } }   \le {\tfrac{21}{10}}  \bb^ {\frac{3}{2}} t^ {\frac{1}{2}}  + 2 R_1 t  \le {\tfrac{11}{5}}  \bb^ {\frac{3}{2}} t^ {\frac{1}{2}} 
\label{dork7}
\end{align} 
where we have taken $\bar \eps$ sufficiently small for the last inequality.
Next,  we have that from \eqref{dtwn-pre1},
\begin{align*} 
&\tfrac{d}{dt} \jump{w}^{(n)} -\tfrac{d}{dt}  \jump{w ^{(1)}  }  \\
&  
  =  \jump{ \p_t w ^{(n)} - \p_t w ^{(1)}  }+ \dot \sc  \jump{ \p_\theta w^{(n)} - \p_\theta w^{(1)} }\\
& 
 =
\bigl(\dot\sc(t) -   w ^{(1)} \bigr) \jump{\p_\theta w^{(n)} -   \p_\theta w^{(1)} } +\bigl(\lambda_3^{(n-1)} - w ^{(1)}   \bigr) \jump{\p_\theta w^{(n)} -  \p_\theta  w^{(1)} } 
 -  \jump{ \lambda_3 ^{(n-1)} }( \p_\theta  w^{(n)} -    \p_\theta w^{(1)})  \\
&\quad - ( \lambda_3^{(n-1)} - w^{(1)}  ) \jump{ \p_\theta  w^{(1)}   } 
- \jump{ \lambda_3^{(n-1)} - w^{(1)} } \p_\theta  w^{(1)}  
 -  \tfrac{8}{3} a^{(n-1)} \jump{ w^{(n-1)} }
+ \tfrac{1}{6}  \jump { ( c ^{(n-1)})^2 \p_\theta k^{(n-1)}  }
\,.
\end{align*} 
By \eqref{eq:u0:ass}, \eqref{eq:sc:ass}, and   \eqref{eq:range:for:solutions}, we see that   $w ^{(1)}  = \wb$ evaluated on the shock curve,
 $\sabs{ \dot \sc - w ^{(1)}  } =\OO(t) $.
Thus, using  the bounds \eqref{eq:w:z:k:a:boot} and \eqref{dork7} shows that 
\begin{align*} 
\abs{ \tfrac{d}{dt} \jump{w}^{(n)} -\tfrac{d}{dt}  \jump{w ^{(1)}  }} \le \sabs{ \jump{ \lambda_3^{(n-1)} - w^{(1)} } \p_\theta  w^{(1)} } + C t^ {\frac{1}{2}} 
\le 2 R_1 \,,
\end{align*} 
for $\bar \eps $ taken sufficiently small.
This proves the first bound in \eqref{2krod}, while the second  follows similarly.
\end{proof} 
Having established Lemma \ref{lem:jumpn}, the conditions of Lemmas~\ref{lem:zminus-on-shock}, Lemma~\ref{lem:zl:kl:Lip}, and
Corollary \ref{cor:jumps:abstract} are satisfied, which together yield
\begin{lemma}[\bf $\zl ^{(n)} $ and $\kl^{(n)} $ on the shock curve]
\label{lem:zknp1-shock} 
Let $w ^{(n)} $ be as in Lemma~\ref{lem:jumpn}. Applying Lemma~\ref{lem:zminus-on-shock}, 
on the shock curve we define $\zl^{(n)}$ and $\kl^{(n)}$ as the solutions of \eqref{eq:hate:iterations} with $n$ replacing $n+1$. In particular, $\zl^{(n)}$ and $\kl^{(n)}$ are 
explicit functions of $ \jump{ w ^{(n)} }$ and $\mean{w ^{(n)} }$ and satisfy
the following bounds:
\begin{subequations}
\label{zknp1:shock}
\begin{align} 
\abs{ \zl ^{(n)} (  t) +  \tfrac{9 \jump{w^{(n)} }(t)^3 }{16 \mean{w^{(n)} }(t)^2}     } 
&\leq C_0 t^{\frac 52}  \,.   
\\
\abs{ \kl^{(n)} ( t) -  \tfrac{4  \jump{w^{(n)} }(t)^3}{\mean{w^{(n)} }(t)^3}      } 
&\leq  C_0 t^{\frac 52}  \,,  
\end{align} 
\end{subequations}
where $C_0 = C_0(\kappa,\bb,\cc,\mm)>0$ is an explicitly computable constant.
Moreover, 
\begin{subequations}
\label{dork2}
\begin{alignat}{2}
&\sabs{\zl ^{(n)} ( t)} \leq 5 \bb^{\frac 92} \kappa^{-2} t^{\frac 32} \,, \qquad  && \sabs{\kl ^{(n)} ( t)} \leq 40 \bb^{\frac 92} \kappa^{-3} t^{\frac 32} \,,
\label{zknp1good}  \\
&\sabs{\tfrac{d}{dt}  \zl^{(n)} (t) } \leq 8 \bb^{\frac 92} \kappa^{-2}   t^{\frac 12}\,, \qquad  && \sabs{\tfrac{d}{dt} \kl^{(n)} (t)} \leq 50 \bb^{\frac 92} \kappa^{-3}   t^{\frac 12}  \,,   
\label{dotzknp1good}
\end{alignat}
\end{subequations}
for all $t \in (0,\bar \eps]$, assuming that $\bar \eps$ is sufficiently small.
\end{lemma}

\subsubsection{Existence, uniqueness, and invertibility of characteristics}

The following lemma follows from  \eqref{eq:dt:eta:def}--\eqref{eq:range:for:solutions:new} and Lemma \ref{lem:AC:DC}.
\begin{lemma}[Bijection set of labels]\label{lem:Upsilon-n}
Assume that $(w^{(n)},z^{(n)},k^{(n)},a^{(n)}) \in {\mathcal X}_{\bar \eps}$. Then, 
for each $t \in (0, \bar\eps]$, there exists a {\em largest} $ x^{(n)} _+(t)>0$ and a {\em smallest} $x^{(n)} _-= x_-(t)<0$ such that
\begin{align}
 \sc(t) = \eta^{(n)} (x^{(n)} _{\pm}(t),t)  
 \label{xnpm}
 \end{align}
where
 \begin{align}
- \tfrac{6}{5} (\bb t)^{\frac 32} < x^{(n)} _-(t) < -  \tfrac{4}{5}  (\bb t)^{\frac 32}
\qquad \mbox{and} \qquad 
 \tfrac{4}{5}   (\bb t)^{\frac 32} < x^{(n)} _+(t) < \tfrac{6}{5} (\bb  t)^{\frac 32}\,.
\label{xn-bounds}
\end{align}
Furthermore,  there exists a set of labels
$$
\Upsilon^{(n)} (t) = \TT \setminus[x^{(n)} _-(t),x^{(n)}_+(t)] \,,
$$
such that 
$\eta^{(n)} (\cdot,t)  \colon\Upsilon^{(n)} \to \TT \setminus \{\sc(t)\}$ is a bijection, and
the inverse map $\ie^{(n)}  \colon \DD_{\bar \eps} \to \TT \setminus \{0\}$ is continuous in spacetime.
\end{lemma}

\begin{lemma}[\bf Bounds for $3$-characteristics]\label{lem:characteristic-bounds}
Assume that $(w^{(n)},z^{(n)},k^{(n)},a^{(n)}) \in {\mathcal X}_{\bar \eps}$. Then, we have   
\begin{subequations} 
\begin{alignat}{2}
&\tfrac 12 \leq \p_x \eta^{(n)} (x,t) \leq \tfrac 74  , \qquad &&\mbox{for all} \qquad  x \in \Upsilon ^{(n)}   \,, \label{cb0}\\
&\sabs{\eta^{(n)}(x,t) - \eta^{(1)}(x,t)} \leq \tfrac 32 \Rsf_1 t^2 , \qquad &&\mbox{for all} \qquad  x \in \Upsilon ^{(n)}   \,, \label{cb1}\\
&\sabs{\p_x \eta^{(n)}(x,t) - \p_x \eta^{(1)}(x,t)} \leq (16 \Rsf_1 \bb^{-\frac 32} + 8 \Rsf_2) t^{\frac 12}, \qquad &&\mbox{for all} \qquad x \in \Upsilon ^{(n)}  \,,\label{cb2} 
\end{alignat}
\end{subequations}
and
\begin{align} 
\int_0^t \abs{\p_\theta  w ^{(1)}  (\eta^{(n)} (x,s),s)}  ds &\leq \tfrac{19}{40} \,.
 \label{dxw1-integral} 
\end{align} 
\end{lemma}
\begin{proof}[Proof of Lemma \ref{lem:characteristic-bounds}]
From Lemma \ref{lem:Upsilon-n}, all of the conditions of Lemma \ref{lem:AC:DC} hold, so the stated inequalities are thus obtained.
\end{proof}

\begin{lemma}\label{lem:12flows}
For $n \ge 1$, assume that $w^{(n)}$ and $z^{(n)}$ satisfy the bounds \eqref{eq:w:boot}--\eqref{eq:z:dx:boot}. 

Then, for every $(\theta,t) \in \mathcal{D} _{\bar \eps}$ there exists a unique  Lipschitz smooth integral curve $\pst^{(n)} (\theta,\cdot) \colon [0,t] \to \DD_{\bar \eps}$ satisfying \eqref{flow1}. There exists a unique point $\mathring{\theta}_1^{(n)} \in \TT$ such that $\psi_{\bar \eps}^{(n)}(\mathring{\theta}_1^{(n)},0) = 0$, which allows us to define as in Definition~\ref{s1s2} the curve $\sc_1^{(n)}$ and the space-time region $\DD_{\bar \eps}^{z,(n)}$. For every $(\theta,t) \in \mathcal{D}^{z,(n)}_{\bar \eps}$,   there exists a unique shock-intersection time $ 0 < \stt^{(n)} (\theta,t) < t$ satisfying \eqref{stoppingtimes}.   Moreover, for $(\theta,t) \in (\overline{\mathcal{D} _{\bar \eps}^{z,(n)}})^\complement$, the characteristic curve $(\pst^{(n)}(\theta,s),s)_{s\in[0,t]}$ does not intersect the shock curve $(\sc(s),s)_{s\in[0,t]}$. 

Similarly, for every $(\theta,t) \in \mathcal{D} _{\bar \eps}$ there exists a unique Lipschitz smooth integral curve $\pt^{(n)} (\theta,\cdot) \colon [0,t] \to \DD_{\bar \eps}$ satisfying \eqref{flow2}. There exists a unique point $\mathring{y}_2^{(n)} \in \TT$ such that $\phi_{\bar \eps}^{(n)}(\mathring{\theta}_2^{(n)},0) = 0$, which allows us to define as in Definition~\ref{s1s2} the curve $\sc_2^{(n)}$ and the space-time region $\DD_{\bar \eps}^{k,(n)}$. For every $(\theta,t) \in \mathcal{D}^{k,(n)}_{\bar \eps}$,   there exists a unique shock-intersection time $ 0 < \st^{(n)} (\theta,t) < t$ satisfying \eqref{stoppingtimes}.   Moreover, for $(\theta,t) \in (\overline{\mathcal{D} _{\bar \eps}^{k,(n)}})^\complement$, the characteristic curve $(\pt^{(n)}(\theta,s),s)_{s\in[0,t]}$ does not intersect the shock curve $(\sc(s),s)_{s\in[0,t]}$. 

Lastly, we have the estimates
\begin{subequations} 
\label{c1c2diffkappa}
\begin{align} 
\pst^{(n)} (\theta,s)  =  \tfrac{1}{3}  \kappa s +   (\theta-  \tfrac{1}{3}  k t) + \OO(t^ {\frac{4}{3}} ) =  \tfrac{1}{3}  \kappa s +   (\theta-  \sc^{(n)} _1(t)) + \OO(t^ {\frac{4}{3}} )\,,
\ \ (\theta,t) \in \mathcal{D}^{z,(n)}_{\bar \eps}\,,  \\
\pt^{(n)} (\theta,s)  =  \tfrac{2}{3}  \kappa s +   (\theta-  \tfrac{2}{3}  k t) + \OO(t^ {\frac{4}{3}} ) =  \tfrac{2}{3}  \kappa s +   (\theta-  \sc^{(n)} _2(t)) + \OO(t^ {\frac{4}{3}} )\,,  
\ \ (\theta,t) \in \mathcal{D}^{z,(n)}_{\bar \eps}\,,  
\end{align} 
\end{subequations} 
and
\begin{align} 
\sup_{s \in [0,t]}  \sabs{\p_\theta  \pt^{(n)} (\theta,s) -1} \le C t^{\frac{1}{3}}  \,, \qquad
\sup_{s \in [0,t]}  \sabs{\p_\theta  \pst^{(n)} (\theta,s) -1} \le C t^{\frac{1}{3}} \,, \qquad (\theta,t) \in \mathcal{D}_{\bar \eps}\,, \label{phixn-bound} 
\end{align} 
where the constant $C>0$ only depends on $\kappa, \bb$, and $\mm$.
\end{lemma} 

\begin{proof}[Proof of Lemma \ref{lem:12flows}] We prove the lemma for the $1$-characteristics $\pst^{(n)}$, the proof for the $2$-characteristics $\pt^{(n)}$ being exactly the same. 

We begin with the existence and uniqueness of $1$-characteristics passing through any point $(\theta,t) \in  \mathcal{D}_{\bar \eps}$.   Using the definition \eqref{flow1}, we see that
\begin{subequations}
\label{blob-a-tron3}
\begin{align}
\p_s \pst^{(n)}  (\theta,s) &= \lambda_1 ^{(n)} ( \pst^{(n)}  (\theta,s),s) \notag\\
&=  \tfrac{1}{3} w ^{(1)}( \pst^{(n)}  (\theta,s),s) +     \Bigl(\tfrac 13 (w^{(n)}  - w^{(1)}) + z^{(n)} \Bigr)  ( \pst^{(n)}  (\theta,s),s) \,, \\
\psi_t ^{(n)}  (\theta,t)&=\theta  \,, 
\end{align}
\end{subequations} 
where we recall cf.~\eqref{wzka1} that $w^{(1)} = \wb$ and $z^{(1)} = 0$.
The bounds \eqref{thegoodstuff1}, \eqref{eq:w:dx:boot}, and \eqref{eq:z:dx:boot} show that $\lambda_1^{(n)}$ is Lipschitz continuous in $\DD_{\bar \eps}$; moreover, as long as $\pst^{(n)}(\theta,s) \in \DD_{\bar \eps}$, we have the explicit estimate
\begin{align} 
&\sabs{\p_\theta  \lambda_1^{(n)} ( \pst^{(n)} (\theta, s) , s)}\notag\\
&\quad \leq 
\tfrac{4}{15} \left(  (\bb  s)^3 +  \sabs{\pst^{(n)} (\theta, s)  - \sc( s)}^2\right)^{-\frac 13} 
+ \tfrac 13 R_2 \left(  (\bb  s)^3 +  \sabs{\pst^{(n)} (\theta, s)  - \sc( s)}^2\right)^{-\frac 16}  + R_4 s^{\frac 12}  
\notag\\
&\quad \leq \tfrac{1}{3} \bb \left(  (\bb  s)^3 +  \sabs{\pst^{(n)} (\theta, s)  - \sc( s)}^2\right)^{-\frac 13} + 2 \mm^3 \,.
 \label{blob-a-tron4}
 \end{align} 
Hence, by the Cauchy-Lipschitz theorem,  for each  such $(\theta,t) \in \DD_{\bar \eps}$, there is a unique local in time solution time to \eqref{blob-a-tron3}. Using \eqref{blob-a-tron4} and the bound $|\lambda_1^{(n)}| \leq \tfrac 12 \mm$, this solution $\pst^{(n)}(\theta,s)$ may be {\em maximally extended} as a Lipschitz function of $s$ on the time interval $[s_*,t]$, where $\pst(\theta,s_*) \in \partial \DD_{\bar \eps}$. In our case, this means that either  $[s_*,t] = [0, t]$ (if $(\pst^{(n)}(s),s)$ does not intersect the shock curve $(\sc(s),s)$ for $s\in (0,t]$), or  $[s_*,t] = [\stt^{(n)}(\theta,t),t]$, where we have denoted by $\stt^{(n)} (\theta,t) \in [0,t)$ the {\em largest} value of $s$ at which $\pst^{(n)} (\theta,s) = \sc (s)$. Of course, if $t< \bar \eps$ the solution $\pst(\theta,s)$ may also be similarly maximally extended to times $s$ past $t$, up to the time $s^*$ at which $\pst(\theta,s^*)$ reaches $\partial \DD_{\bar \eps}$.

In order to complete the existence and uniqueness part claimed in Lemma~\ref{lem:12flows}, we need to show that if $\stt^{(n)}(\theta,t) \in (0,t)$, then the integral curve may be uniquely continued as a Lipschitz function of $s$ also on the time interval $[0,\stt^{(n)}(\theta,t)]$. We note that in this case the limit $\lim_{s\to \stt^{(n)}(\theta,t)^+} \pst^{(n)}(\theta,s)$ is well-defined, and so to ensure continuity we let $\pst^{(n)}(\theta,\stt^{(n)}(\theta,t))$ equal this limit. The desired claim follows once we prove the following two statements: first, that the shock surface $(\sc(s),s)_{s\in [0,t]}$ is a non-characteristic surface for the ODE \eqref{blob-a-tron3}, so that $\pst^{(n)}(\theta,\stt^{(n)}(\theta,t)) = \sc(\stt^{(n)}(\theta,t)$ may serve as Cauchy data for the {\em transversal characteristic} $\pst(\theta,s)$ with $s< \stt^{(n)}(\theta,t)$; second, that the curve $\pst(\theta,s)$ does not intersect the shock curve for $s\in [0, \stt^{(n)}(\theta,t))$, thereby ensuring the uniqueness/well-definedness of $\stt^{(n)}(\theta,t)$ implicitly assumed in Definition~\ref{def:stopping:times}.

The {\em transversality} of $\pst^{(n)}$ and the shock surface is established as follows. We first  carefully estimate $\lambda_1^{(n)}$ in the vicinity of the shock curve. By \eqref{eq:wb:def}, \eqref{eq:tropic:2}, \eqref{eq:w:boot}, and \eqref{eq:z:boot}, for any $\tilde \theta$ such that $|\tilde \theta - \sc(s)| \leq \kappa t$ we have that 
\begin{align}
\sabs{\lambda_1^{(n)}(\tilde \theta, s) - \tfrac 13 \kappa} \leq \tfrac 13 \sabs{w_0(\etab^{-1}(\tilde \theta,s)) - \kappa} + \tfrac 13 R_1 s + R_3 s^{\frac 32} \leq  \tfrac 32 \bb (\kappa t)^{\frac 13} \,,
\label{blob-a-tron9}
\end{align}
since $\bar \eps$, and hence $s\leq t$, are sufficiently small. Note that if $|\tilde \theta - \sc(s)| \leq \kappa s$, then in the upper bound \eqref{blob-a-tron9} we may replace $t^{\frac 13}$ by $s^{\frac 13}$. Next, we note that
 the vector normal to the shock curve is given by $\bigl(  -1, \dot{\sc}(s) \bigr) $ while the tangent vector to the characteristic curve is given by
$\bigl( \p_s \pst^{(n)}  (\theta,s) , 1 \bigr) = \bigl(\lambda_1^{(n)}(\pst^{(n)}(\theta,s),s) , 1 \bigr)$. 
Computing the dot-product,  and appealing to \eqref{eq:sc:ass} and \eqref{blob-a-tron9}, we obtain that 
\begin{align} 
\bigl(  -1, \dot\sc(s) \bigr)  \cdot \bigl( \p_s \pst^{(n)}  (s) , 1 \bigr) =  \dot\sc(s) - \lambda_1^{(n)}(\pst^{(n)}(\theta,s),s)   =  \tfrac{2}{3}  \kappa +  \OO( s^ {\frac{1}{3}} )
\ge \tfrac{1}{2} \kappa  \,,
\label{blob-a-tron99}
\end{align} 
since $\bar\eps$ is small enough, and $s = \stt^{(n)}(\theta,t)$. Therefore, the characteristic curve $\pst^{(n)}$ intersects the shock curve transversally, and the crossing angle is bounded from below uniformly for on $[0,\bar \eps]$. As mentioned above, this means that we can use the values of the flows $\pst^{(n)}$  on the shock curve as Cauchy data, and continue the solutions in a Lipschitz fashion for $s < \stt^{(n)}(\theta,t)$. The fact that the angle measured in \eqref{blob-a-tron99} has a sign, and the smoothness of $\sc$, also ensures the uniqueness of the {\it shock-intersection time} $\stt(\theta,t) \in (0, t)$, so that it is a well-defined object. This concludes the proof of existence, uniqueness, and Lipschitz regularity for the characteristic curves $\pst^{(n)}(\theta,\cdot) \colon [0,t] \to \TT$. 

Next, we turn to the proof of the bound \eqref{phixn-bound}. Differentiating \eqref{blob-a-tron3} shows that 
\begin{align} 
 \p_\theta  \pst^{(n)} (\theta,s) 
 & = e^{\int_{s}^t (\p_\theta  \lambda_1^{(n)})(\pst^{(n)} (\theta,s') ,s') ds'} \notag\\
 & = e^{{\frac{1}{3}} \int_s^t  \p_\theta  w^{(1)} (  \pst^{(n)}(\theta,s') ,s')ds'}  
 e^{ \int_s^t \left( \frac{1}{3} (\p_\theta  w^{(n)}- \p_\theta  w^{(1)}) + \p_\theta  z^{(n)} \right)(  \pst^{(n)}(\theta,s') ,s') ds'} 
\,.
\label{blob-a-tron999}
\end{align} 
For $s' \in [s,t]$ such that $|\pst^{(n)}(\theta,s') - \sc(s')| \geq \kappa t$, from \eqref{blob-a-tron4} we deduce that $|\p_\theta  \lambda_1^{(n)}(\pst^{(n)}(\theta,s'),s')| \leq \frac 13 \bb (\kappa t)^{-\frac 23} + 2\mm^3 \leq \frac 23 \bb (\kappa t)^{-\frac 23}$, and thus the contribution from such $s'$ to the integral on the right side of \eqref{blob-a-tron999} is bounded from above by $\exp( 2 \bb \kappa^{-\frac 23} t^{\frac 13})$. On the other hand, $s' \in [s,t]$ such that $|\pst^{(n)}(\theta,s') - \sc(s')| \leq  \kappa t$, we may appeal to \eqref{blob-a-tron9}, so that $\p_s \pst^{(n)}(\theta,s') \leq \frac 12 s'$; this allows us to apply Lemma~\ref{lem:Steve:needs:this} with $\gamma = \pst^{(n)}(\theta,\cdot)$ and  $\mu = \frac 12$, for these intervals of $s'$, and together with the bounds \eqref{eq:w:z:k:a:boot} we deduce that the contribution from such $s'$ to the integral on the right side of \eqref{blob-a-tron999} is bounded from above by $\exp(30 \bb \kappa^{-\frac 23} t^{\frac 13})$. Combining these estimates we deduce that 
for all $s\in [0,t]$ and $t\in (0,\bar \eps]$, 
\begin{align} 
\sabs{ \p_\theta  \pst^{(n)} (\theta,s) - 1 } \le 40 \bb \kappa^{-\frac 23} t^{\frac 13} \,,
\label{blob-a-tron9999}
\end{align} 
when $\bar\eps$ is sufficiently small. This proves \eqref{phixn-bound} for the flow $\pst^{(n)}$, which implies that $\pst^{(n)}$ is continuous on $\TT \times [0,t]$, and is uniformly Lipschitz continuous both with respect to $\theta$ and with respect to $s$. 

The bound \eqref{blob-a-tron9999} does not just provide regularity with respect to $\theta$ of the flow $\pst^{(n)}(\theta,s)$, but it also shows that it is a monotone increasing function of $\theta$. This allows us to show the existence and uniqueness of a point $\mathring{\theta}_1^{(n)} \in \TT$ such that $\psi_{\bar \eps}^{(n)}(\mathring{\theta}_1^{(n)},0) = 0$. Existence follows by the intermediate function theorem, applied to $\psi_{\bar \eps}^{(n)}(\theta,0) \colon \TT \to \TT$: indeed, from \eqref{blob-a-tron9} (applied with $t = \bar \eps$) and \eqref{eq:sc:ass}, we see that for $ \sc(\bar \eps) < \theta$ we have $\psi^{(n)}_{\bar \eps}(\theta,0) \geq \frac 14 \kappa \bar \eps > 0$; on the other hand, for $\theta < \sc(\bar \eps) - \frac 34 \kappa \bar \eps$, we have $\psi^{(n)}_{\bar \eps}(\theta,0) \leq - \frac 18 \kappa \bar \eps < 0$. The uniqueness of $\mathring{\theta}_1^{(n)}$ follows by the monotonicity in $\theta$ guaranteed by  \eqref{blob-a-tron9999}. Note that the above argument gives the rough bound $\sc(\bar \eps) - \frac 34 \kappa \bar \eps \leq \mathring{\theta}_1^{(n)} \leq \sc(\bar \eps)$.

Thus, as in Definition~\ref{s1s2} the curve $\sc_1^{(n)}$ and the space-time region $\DD_{\bar \eps}^{z,(n)}$ are now well-defined. The fact that for $(\theta,t) \in (\overline{\mathcal{D} _{\bar \eps}^{z,(n)}})^\complement$  the  curve $(\pst^{(n)}(\theta,s),s)_{s\in[0,t]}$ does not intersect the shock curve $(\sc(s),s)_{s\in[0,t]}$, and the fact that for $(\theta,t) \in \mathcal{D} _{\bar \eps}^{z,(n)}$ intersection does indeed occur at a unique time $\stt^{(n)}(\theta,t)$, now follows from the monotonicity of $\pst^{(n)}(\theta,s)$ with respect to $\theta$, the definition of $\sc_1^{(n)}$, the transversality \eqref{blob-a-tron99}, and its consequences discussed earlier.

In order to conclude the proof, it remains to establish \eqref{c1c2diffkappa}.
From the aforementioned rough bound on $\mathring{\theta}_1^{(n)}$, appealing to the definition $\sc_1(s) = \psi_{\bar \eps}^{(n)}(\mathring{\theta}_1^{(n)},s)$, the bound \eqref{eq:sc:ass}, integrating \eqref{blob-a-tron9} with $\tilde \theta = \sc_1^{(n)}(s)$, and using that $\sc_1^{(n)}(0) = 0 = \sc(0)$, we see that 
\begin{align} 
\sabs{\sc(s) - \sc_1^{(n)}(s) - \tfrac 23 \kappa s} 
&\leq \sabs{\sc_1^{(n)}(s) - \tfrac 13 \kappa s} + \sabs{\sc(s) - \kappa s}  \notag\\
&\leq \tfrac 32 \bb \kappa^{-\frac 13} s^{\frac 43} +  \mm^4 s^2  \leq 2 \bb \kappa^{-\frac 13} s^{\frac 43} \qquad \mbox{for all}  \qquad  s \in [0, \bar\eps] \,.
\label{blob-a-tron6}
\end{align}  
More generally, for any $(\theta,t) \in \DD_{\bar \eps}^{z,(n)}$, we may integrate \eqref{blob-a-tron9} with $\tilde \theta = \pst^{(n)}(\theta,s)$ and deduce that 
\begin{align}
\pst^{(n)}(\theta,s) 
= \theta - \int_s^t \lambda_1^{(n)}( \pst^{(n)}(\theta,s')) ds'
= \theta - \frac 13 \kappa (t-s) + \OO(t^{\frac 43})\,,
\label{blob-a-tron7}
\end{align}
which proves the first equality in \eqref{c1c2diffkappa}. The second equality follows by combining \eqref{blob-a-tron7} with \eqref{blob-a-tron6}, which in turn shows via \eqref{eq:sc:ass} that $\sc_1^{(n)}(t) = \frac 13 \kappa t + \OO(t^{\frac 43})$. 

The arguments for the $2$-characteristic $\pt^{(n)}(\theta,s)$ are identical, except that $\frac 13 \kappa t$ must be replaced with $\frac 23 \kappa t$ because $\lambda_2^{(n)}$ contains $\frac 23 w^{(1)}$ instead of $\frac 13 w^{(1)}$. We omit these redundant details.
\end{proof}

\subsubsection{Stability of the iteration space}
\begin{proposition}[${\mathcal X}_{\bar \eps}$ is stable under iteration]\label{lem:stable} 
Let $\bar \eps$ be taken sufficiently small  with respect to $\kappa, \bb, \cc$, and $\mm$. 
For all $n\ge 1$, the map 
$$ (w^{(n)} , z^{(n)}, k^{(n)}, a^{(n)} ) \mapsto 
(w^{(n+1)}, z^{(n+1)}, k^{(n+1)}, a^{(n+1)} )$$
maps ${\mathcal X}_{\bar \eps} \to {\mathcal X}_{\bar \eps}$. 
In particular, the iterates $(w^{(n+1)}, z^{(n+1)}, k^{(n+1)}, a^{(n+1)} ) $ satisfy the bounds \eqref{eq:w:z:k:a:boot}.
\end{proposition} 
\begin{proof}[Proof of Proposition \ref{lem:stable}]
In the course of the proof, we will repeatedly let  $\bar \eps$, and hence $t$, to be  sufficiently small with respect to
$\kappa, \bb, \cc, \mm$.

\noindent
{\bf Estimates for $w ^{(n+1)} $.}
By  Lemma \ref{lem:Upsilon-n}, for any $(\theta,t)  \in \mathcal{D} _{\bar \eps}$, there exists a label $x \in \Upsilon^{(n)} (t)$ such that 
$\eta ^{(n)} (x,t)=\theta$.

By the triangle inequality,
\begin{align*} 
\sabs{\bigl(w ^{(n+1)} -  w ^{(1)}\bigr)  \circ \eta ^{(n)} } & \le
\sabs{w ^{(n+1)} \circ \eta ^{(n)}  - w_0}  + \sabs{w ^{(1)}  \circ \eta ^{(1)}  - w ^{(1)}  \circ \eta ^{(n)}   }     \\
& 
=\sabs{w ^{(n+1)} \circ \eta ^{(n)}  - w_0}  + \sabs{   w_0 \circ \etab^{-1} \circ \eta^{(n)}  - w_0 }   \,.
\end{align*} 
By the fundamental theorem of calculus,
\begin{align*} 
w_0 \circ \etab^{-1} \circ \eta^{(n)}  - w_0 
&= \int_0^t \frac{d}{dt}  \bigl(w_0 \circ \etab^{-1} \circ \eta^{(n)} \bigr) ds\notag\\
&= \int_0^t w_0'  \circ \etab^{-1} \circ \eta^{(n)} \left( (\p_t \etab^{-1}) \circ \eta^{(n)}  + (\p_\theta  \etab^{-1}) \circ \eta \p_t \eta^{(n)} \right) ds  \\
&= \int_0^t w_0'  \circ \etab^{-1} \circ \eta^{(n)}(\lambda_3^{(n)}  - w ^{(1)}  ) \circ \eta^{(n)}  \   \bigl(\eta_x(\etab ^{-1} \circ  \eta^{(n)} )\bigr) ^{-1}   ds \,.
\end{align*} 
The bounds \eqref{eq:u0':interest},  \eqref{cb0} and \eqref{eq:w:z:k:a:boot} show that
\begin{align} 
\sabs{   w_0 \circ \etab^{-1} \circ \eta^{(n)}  - w_0 }  \le \tfrac{1}{2}  R_1 t \,. 
\label{bloob}
\end{align}

Next,  using the identity \eqref{wn-alt}, we have that
\begin{align*} 
\sabs{w ^{(n+1)} \circ \eta ^{(n)}  - w_0}
& \le
 {\tfrac{8}{3}}  \int_0^t\sabs{   \bigl( a^{(n)}  w^{(n)}  \bigr) \circ \eta ^{(n)} } ds  
 + {\tfrac{1}{4}} \int_0^t \sabs{
 c ^{(n)} \circ \eta^{(n)}  \tfrac{d}{ds} \Bigl( k ^{(n)}  \circ \eta^{(n)} \Bigr) } ds .
\end{align*} 
The bounds \eqref{eq:w:z:k:a:boot} 
with $\bar \eps$ taken sufficiently small, 
\begin{align*} 
\sabs{w ^{(n+1)} \circ \eta ^{(n)}  - w_0}& \le 3 \mm R_7 t \,.
\end{align*} 
Together with the bound \eqref{bloob} and the fact that
$\eta ^{(n)} ( x,t)$ is a diffeomorphism for each label $x \in \Upsilon ^{(n)} $, we  have that
for all $(\theta,t)  \in \mathcal{D} _{\bar \eps}$,
\begin{align*} 
\sabs{w ^{(n+1)} (\theta,t)  -  w ^{(1)} (\theta,t)   } & \le  \tfrac{3}{4} R_1 t    \,,
\end{align*}
as long as $12 \mm R_7  \le  R_1$. This inequality holds due to the choices in \eqref{eq:R:def}.

Let us now show that  the estimate \eqref{eq:w:dx:boot} holds.
Following the procedure we used to obtain the identity \eqref{prelim-wx}, we
differentiating \eqref{dtwn-alt}, use \eqref{cn}, and obtain that
\begin{align} 
\tfrac{d}{dt}  \bigl( w_\theta ^{(n+1)}  \circ \eta ^{(n)} \ \eta_x ^{(n)}  \bigr) &=  
 \tfrac{1}{4}    \tfrac{d}{dt} \bigl( (c ^{(n)}  k_\theta ^{(n)} ) \circ \eta ^{(n)} \ \eta_x ^{(n)}   \bigr) 
 +\mathcal{F}^{(n)} _{w_\theta} \circ \eta ^{(n)} \ \eta_x ^{(n)}  \,, \label{wx}
\end{align} 
where 
\begin{align} 
\mathcal{F}^{(n)} _{w_\theta} & =  
  k_\theta ^{(n)} \Bigl(
   \tfrac{1}{6}   c ^{(n)} c_\theta ^{(n)} + \tfrac{1}{6}   c ^{(n-2)} z_\theta ^{(n)} +  \tfrac{2}{3}  a ^{(n-2)} c ^{(n-2)} 
   + \tfrac{1}{4}   ( \lambda_3^{(n-2)} -\lambda_3^{(n)} ) c_\theta ^{(n)}  \Bigr) \circ \eta ^{(n)} \ \eta_x ^{(n)} \notag \\
& \qquad \qquad   - \tfrac83 \p_\theta (a ^{(n)} w ^{(n)} )
   \,.  \label{Fwx-alt}
\end{align} 
An equivalent form of \eqref{wx} is given by
\begin{align} 
\p_t w_\theta ^{(n+1)} + \lambda_3^{(n)} w_{\theta\theta} ^{(n+1)} + \p_\theta \lambda_3 ^{(n)} w_\theta ^{(n+1)}   =  
 \tfrac{1}{4}  \left(  \tfrac{d}{dt} \bigl( (c ^{(n)}  k_\theta ^{(n)} ) \circ \eta ^{(n)} \ \eta_x ^{(n)}   \bigr) \bigl(\eta_x ^{(n)}\bigr) ^{-1}   \right) \circ \ie ^{(n)}
 +\mathcal{F}^{(n)} _{w_\theta}  \,, \label{wx2}
\end{align} 
Therefore, 
\begin{align*} 
&\tfrac{d}{dt}\left(  (w_\theta ^{(n+1)} - w_\theta ^{(1)}  ) \circ \eta ^{(n)} \right)  
+ w_\theta^{(n)}   \circ \eta ^{(n)} (w_\theta ^{(n+1)}  - w_\theta ^{(1)}  ) \circ \eta^{(n)} 
 \\
&\qquad \qquad =  -   w_\theta^{(1)}   \circ \eta ^{(n)} 
(w_\theta ^{(n)}  - w_\theta ^{(1)}  ) \circ \eta^{(n)} 
-(w^{(n)} - w^{(1)}  ) \circ \eta ^{(n)}  w_{\theta\theta} ^{(1)}   \circ \eta^{(n)} \\
& \qquad\qquad \ \ 
+ \tfrac{1}{4}   \tfrac{d}{dt} \Bigl( (c ^{(n)}  k_\theta ^{(n)} ) \circ \eta ^{(n)} \ \eta_x ^{(n)}   \Bigr) \bigl(\eta_x ^{(n)}\bigr) ^{-1}  
 +\mathcal{F}^{(n)} _{w_\theta} \circ \eta ^{(n)}  \,.
\end{align*} 
For $0 \le s \le t$, let us define the integrating factor $ \mathcal{I}_{s,t} = e^{ \int_s^t w_y^{(n)} (\eta^{(n)} (x,r),r)dr }$.  Then, we have 
that 
\begin{align} 
& (w_\theta ^{(n+1)} - w_\theta ^{(1)}  ) \circ \eta ^{(n)} \notag
 \\
&\ =  \overbrace{\int_0^t -\mathcal{I}_{t,s}   \Bigl(   w_\theta^{(1)}   \circ \eta ^{(n)} 
(w_\theta ^{(n)}  - w_\theta ^{(1)}  ) \circ \eta^{(n)} \Bigr)  ds }^{\Isf_1}  
+\overbrace{\int_0^t  -\mathcal{I}_{t,s}   \Bigl( 
(w^{(n)} - w^{(1)}  ) \circ \eta ^{(n)}  w_{\theta\theta} ^{(1)}   \circ \eta^{(n)}  \Bigr)  ds}^{\Isf_2} \notag\\
& \ \ \
+ \underbrace{\tfrac{1}{4}  \int_0^t \mathcal{I}_{t,s}   \left(  \tfrac{d}{dt} \Bigl( (c ^{(n)}  k_\theta ^{(n)} ) \circ \eta ^{(n)} \ \eta_x ^{(n)}   \Bigr) 
\bigl(\eta_x ^{(n)}\bigr) ^{-1}  \right) ds}_{\Isf_3}
 +  \underbrace{\int_0^t \mathcal{I}_{t,s}  \mathcal{F}^{(n)} _{w_\theta} \circ \eta ^{(n)}    ds}_{\Isf_4} \,. \label{wy-diff}
\end{align} 
From \eqref{eq:w:dx:boot}, $\sabs{w_\theta ^{(n)} - w_\theta ^{(1)}  } \le R_2 t^ {-\frac{1}{2}} $ and thanks to  \eqref{dxw1-integral}, we have
that for $\bar \eps$ small enough, 
\begin{align*} 
\sabs{\mathcal{I}_{s,t} }= e^{ \int_s^t \sabs{ (w_y^{(n)}-w_\theta ^{(1)} )  (\eta^{(n)} (x,r),r)}dr }e^{ \int_s^t \sabs{ w_\theta ^{(1)} (\eta^{(n)} (x,r),r)} dr }
\le \tfrac{17}{10}  \,.
\end{align*} 
Let us now estimate each integral $\Isf_1$, $\Isf_2$, $\Isf_3$,  and $\Isf_4$  on the right side of \eqref{wy-diff}.   First, we have that  
\begin{align} 
\abs{\Isf_1} \le \tfrac{17}{10}  R_2 \int_0^t \abs{   w_\theta^{(1)}   \circ \eta ^{(n)} }
\Bigl( \bb ^3 s^3 + \bigl( \eta^{(n)}(x,s) - \sc(s)\bigr)^2 \Bigr)^ {-\frac{1}{6}}  ds \,. \label{Int0}
\end{align} 
The Burgers characteristic satisfies $\p_t (\sc(t) - \etab(x,t) ) =\dot\sc(s) - w_0(x)$. Integration from $s$ to $t$ for $0\le s \le t$ together
with the inequality  \eqref{eq:sc:ass}, the fact that $\abs{x} \ge \tfrac{3}{4}  (\bb t)^ {\frac{3}{2}} $, and taking
$\bar \eps$ sufficiently small, shows that
\begin{align*} 
\sc(s) - \etab(x,s)
&\geq \sc(t) - \etab(x,t)  + (t-s) (\kappa -  w_0(x) - C t) \notag\\
&\ge \sc(t) - \etab(x,t)    +( \tfrac{3}{4})^{\frac{1}{3}}     \bb^ {\frac{3}{2}}  t^ {\frac{1}{2}} (t-s ) - C  (t-s)t\,.
\end{align*} 
Using that  that $\theta = \eta^{(n)} (x,t)$,  \eqref{eq:u0:ass} and  \eqref{cb1}, and taking
$\bar \eps$ even smaller if necessary, we see that
\begin{align*} 
\sc(s) - \eta^{(n)} (x,s)   \ge\sc(t) - \theta  + {\tfrac{\sqrt{3} }{2}}   \bb^ {\frac{3}{2}}  t^ {\frac{1}{2}} (t-s ) \,,
\end{align*} 
and hence
\begin{align*} 
\bb^3 s^3+ \bigl(\eta^{(n)} (x,s) - \sc(s) \bigr)^2 \ge  \bigl(\theta - \sc(t)\bigr)^2  +  \tfrac{3}{4}      \bb^3  t (t-s )^2 + \bb^3 s^3 \,.
\end{align*} 
The function $ \tfrac{3}{4}      \bb^3  t (t-s )^2 + \bb^3 s^3$ has a minimum at $s=  \tfrac{t}{2}  $ and takes the value there of $ \tfrac{5}{16}\bb^3 t^3$, so that
\begin{align} 
\bb^3 s^3+ \bigl(\eta^{(n)} (x,s) - \sc(s) \bigr)^2 \ge \tfrac{5}{16}\Bigl(  \bigl(\theta - \sc(t)\bigr)^2 + \bb^3 t^3 \Bigr)   \,. 
\label{blobster}
\end{align}
Thus,  with \eqref{blobster}, the integral $\Isf_1$ in \eqref{Int0} is bounded as
\begin{align} 
\abs{\Isf_1}
 &\le \tfrac{17}{10}(\tfrac{5}{16})^ {-\frac{1}{6}} R_2
\Bigl(  \bigl(\theta - \sc(t)\bigr)^2 + \bb^3 t^3 \Bigr)^ {-\frac{1}{6}}   \int_0^t \abs{   w_\theta^{(1)}   \circ \eta ^{(n)} } ds 
 \notag \\
&
  \le \tfrac{19}{40}  \tfrac{17}{10}(\tfrac{5}{16})^ {-\frac{1}{6}} R_2
\Bigl(  \bigl(\theta - \sc(t)\bigr)^2 + \bb^3 t^3 \Bigr)^{-\frac{1}{6}} 
  \,, \label{Int1}
\end{align} 
the last inequality following from \eqref{dxw1-integral}.   
It is important to note that $\tfrac{19}{40}  \tfrac{17}{10}(\tfrac{5}{16})^ {-\frac{1}{6}} < {\frac{99}{100}}  $.

For the integral $\Isf_2$ in \eqref{wy-diff}, the estimate \eqref{thegoodstuff2} shows that
\begin{align*} 
\abs{\Isf_2} \le 2 \tfrac{17}{10} R_1 \bb \int_0^t  s    \Bigl(  (\bb s)^3 +  \abs{y - \sc(s)}^2 \Bigr)^{-\frac 56} ds \,.
\end{align*} 
Using \eqref{blobster} and that $ (\tfrac{5}{16})^ {-\frac{5}{6}} \le 3$, we then have that
\begin{align*} 
\abs{\Isf_2} \le 3 \tfrac{17}{10} R_1 \bb  \Bigl(  (\bb t)^3 +  \abs{\theta - \sc(t)}^2 \Bigr)^{-\frac 56} t^2 
\le 6 R_1 \bb  \Bigl(  (\bb t)^3 +  \abs{\theta - \sc(t)}^2 \Bigr)^{-\frac 16}
\,.
\end{align*} 
Thus, $w^{(n+1)}$ satisfies \eqref{eq:w:dx:boot} as soon as we choose $1200 R_1 \bb \le R_2$. In view of \eqref{eq:b:m:ass}, this inequality is ensured by the choice of $R_2$ and $R_1$ given in \eqref{eq:R:def}.

To bound $\Isf_3$, we integrate-by-parts and find that
\begin{align*} 
\Isf_3=
\tfrac{1}{4} \mathcal{I}_{t,s}  \Bigl( (c ^{(n)}  k_\theta ^{(n)} ) \circ \eta ^{(n)}  \Bigr) 
-
\tfrac{1}{4}  \int_0^t    (c ^{(n)}  k_\theta ^{(n)} ) \circ \eta ^{(n)} \ \eta_x ^{(n)}
 \tfrac{d}{dt} \Bigl(  \mathcal{I}_{t,s}  \bigl(\eta_x ^{(n)}\bigr) ^{-1}  \Bigr)  ds \,.
\end{align*} 
Since $ \p_t \mathcal{I} _{0,t} = \mathcal{I} _{0,t} w_\theta^{(n)} \circ \eta^{(n)} $ and $\p_t  \bigl(\eta_x ^{(n)}\bigr) ^{-1} 
=  -\bigl(\eta_x ^{(n)}\bigr) ^{-1}  \p_\theta  \lambda_3^{(n)} \circ \eta^{(n)} $, using the bounds \eqref{eq:w:z:k:a:boot}, we obtain
\begin{align*} 
\abs{\Isf_3} \le C t^ {\frac{1}{2}}  \,.
\end{align*} 
Finally, using the definition of  $\mathcal{F}^{(n)} _{w_\theta}$ in   \eqref{Fwx-alt} and the bounds \eqref{eq:w:z:k:a:boot}, we also find that
\begin{align*} 
\abs{\Isf_4} \le C t^ {\frac{1}{2}}  \,.
\end{align*} 
By combining the bounds for $\Isf_1$, $\Isf_2$, $\Isf_3$, and $\Isf_4$, we taking $\bar \eps$ sufficiently small so we have shown that
\begin{align*} 
\abs{ w_\theta ^{(n+1)}(\theta,t)  - w_\theta ^{(1)}  (\theta,t)  } \le \tfrac{999}{1000}R_2 \,,
\end{align*} 
for all $(\theta,t) \in \DD_{\bar \eps}$, thus establishing that \eqref{eq:w:dx:boot} holds.

\vspace{.1in}
\noindent
{\bf  Estimates for $z ^{(n+1)} $.}     Let $(\theta,t)  \in \mathcal{D} ^{z, {(n+1)} }_{\bar \eps}$.  
We integrate \eqref{xland-z:3} from $\stt^{(n)} (\theta,t) $ to $t$ and obtain
\begin{align} 
z^{(n+1)} (\theta,t)  &=  \zl ^{(n+1)} (\sc(\stt^{(n)}(\theta,t) ) )- \int_{\stt^{(n)} (\theta,t) }^t \bigl( \tfrac{8}{3}a^{(n)} z^{(n)} 
- \tfrac{1}{6}  (c ^{(n)})^2 \p_\theta  k^{(n)}  \bigr) \circ \pst ^{(n)}  ds'\,,  \label{znp1}
\end{align} 
Having shown that $w ^{(n+1)}  \in {\mathcal X}_{\bar \eps}$ (continuity will be established below), then $ w ^{(n+1)} $ satisfies the criteria of Lemma \ref{lem:jumpn}
and thus we can appeal to Lemma \ref{lem:zknp1-shock} for the bound of $ \zl ^{(n+1)} (\sc(\stt^{(n)}(\theta,t) ) )$.
It follows from \eqref{eq:w:z:k:a:boot}  and \eqref{zknp1good} that
\begin{align*} 
\sabs{z^{(n+1)} (\theta,t) } \le ( 5 \bb^{\frac 92} \kappa^{-2} + \tfrac{1}{8} \kappa^2 R_6) t^ {\frac{3}{2}} \,,
\end{align*} 
which shows that \eqref{eq:z:boot} holds for $ z ^{(n+1)}$ if   $ 5 \bb^{\frac 92} \kappa^{-2} + \tfrac{1}{8} \kappa^2 R_6 \le R_3$.
Using \eqref{eq:b:m:ass}, this inequality holds due to the definition of $R_3$ and $R_6$ in \eqref{eq:R:def}.

Next,  integrating \eqref{z-alt} from $\stt ^{(n)} $ to $t$ and using the definitions of $\mathcal{F} _{z_\theta} $ and  $ \mathcal{H}_{z_\theta}$ given by \eqref{Fzx} and  \eqref{Hzx}, respectively, for all $(\theta,t) \in \mathcal{D} ^{z, {(n+1)} }_{\bar \eps}$,
\begin{align} 
\p_\theta  z ^{(n+1)}  & = 
\mathcal{H}_{z_\theta} (\uu^{(n)} ,\dot\uu^{(n)} ,\pst^{(n)} ,\stt^{(n)} )
 +   \mathcal{F} _{z_\theta}  (U^{(n)} ,\pst^{(n)} ,\stt^{(n)} )  \,.
\end{align} 
It follows from \eqref{eq:Steve:needs:this:1},  \eqref{eq:w:z:k:a:boot}, \eqref{dotzknp1good},  and \eqref{phixn-bound} that for $t$ sufficiently small, 
\begin{align} 
\sabs{\p_\theta  z ^{(n+1)} (\theta,t) } \le  2\kappa ^{-3}(8 \bb^{\frac 92}  +50 \bb^{\frac 92} ) t^ {\frac{1}{2}} +  \tfrac{\kappa}{2}   R_6 t^ {\frac{1}{2}}
\le R_4 t^{\frac{1}{2}}   \,,
\end{align} 
which proves that \eqref{eq:z:dx:boot} holds for $\p_\theta  z ^{(n+1)} $ whenever  $ 116\kappa ^{-3} \bb^{\frac 92} +  \tfrac{\kappa}{2}   R_6  \le R_4$. Using  \eqref{eq:b:m:ass}, this inequality holds by defining $R_4$ and $R_6$ as in \eqref{eq:R:def}.

\vspace{.1 in}
\noindent
{\bf Estimates for $k ^{(n+1)} $.}  We have shown that $w^{(n+1)} $ and $z^{(n+1)} $ satisfy the bounds \eqref{eq:w:z:k:a:boot}, and we will
prove below that both functions are continuous on $ \mathcal{D} _{\bar \eps}$ and hence are in the set ${\mathcal X}_{\bar \eps}$.  For 
each $(\theta,t)  \in \mathcal{D}^{(n+1)}$, we then have existence
of unique characteristics $\pt ^{(n+1)}(\theta,s)  $  and shock-intersection times $\st^{(n+1)} (\theta,t) $ satisfying the properties   in Lemma
\ref{lem:12flows}.

Let $(\theta,t)  \in \mathcal{D} ^{k, {(n+1)} }_{\bar \eps}$.  
We integrate \eqref{k-alt} from $\st^{(n+1)} (\theta,t) $ to $t$ and obtain that
\begin{align} 
k^{(n+1)} (\theta,t)  &=  \kl ^{(n+1)} ( \sc ( \st^{(n+1)} (\theta,t)  ))   \,. \label{kn-eq}
\end{align} 
Again, appealing to Lemma \ref{lem:zknp1-shock}, 
the bound \eqref{zknp1good} then gives
\begin{align} 
\abs{k^{(n+1)} (\theta,t) } \le 40 \bb^{\frac 92}  \kappa^{-3} t^{\frac{3}{2}} \,, \label{knp1}
\end{align} 
which shows that \eqref{eq:k:boot} holds for $k^{(n+1)}$ if  $40 \bb^{\frac 92}  \kappa^{-3} \le R_5$. The condition \eqref{eq:b:m:ass} justifies the definition of $R_5$ in~\eqref{eq:R:def}.

In the same way that we obtained \eqref{kx-interior} and  \eqref{kx-shock}, we also  have that
\begin{align} 
\p_\theta  k^{(n+1)} (\theta,t)   
& = \tfrac{ \dot \kk^{(n+1)}  (\st^{(n+1)} (\theta,t) )) }{ \dot \sc(\st^{(n+1)} (\theta,t) )) 
-   \p_s \phi_t^{(n+1)} (\theta, \st^{(n+1)} (\theta,t)  )} \  \p_\theta  \phi_t^{(n+1)} ( \sc( \st^{(n+1)} (\theta,t)  ), \st^{(n+1)} (\theta,t) )  \,, \label{kxn-shock}
\end{align} 
and thus
from  \eqref{dotzknp1good},  and \eqref{phixn-bound} that for $t$ sufficiently small, 
\begin{align} 
\sabs{\p_\theta  k ^{(n+1)} (\theta,t) } \le 200 \bb^{\frac 92} \kappa^{-4}    t^ {\frac{1}{2}}  \,,
\end{align} 
which shows that \eqref{eq:k:dx:boot} holds for $\p_\theta k^{(n+1)}$ if $200 \bb^{\frac 92} \kappa^{-4}   \le R_6$. The condition \eqref{eq:b:m:ass}justifies the definition of $R_6$ in \eqref{eq:R:def}.

\vspace{.1 in}
\noindent
{\bf Estimates for $a ^{(n+1)} $.}  We consider any point $(\theta,t)  \in \mathcal{D}^{k,(n)} _{\bar \eps}$.    By Lemma \ref{lem:12flows}, the 
characteristic curve $\pt ^{(n)} (\theta,s) $ exists for all $s \in [0,t]$.  
From \eqref{xland-a:2}, we have that
\begin{align} 
\tfrac{d}{dt} \bigl( a^{(n+1)} \circ \pt ^{(n)} \bigr) 
& = \Bigl(- \tfrac43   (a^{(n)} )^2  + \tfrac{1}{6}(w^{(n)} )^2 + \tfrac{1}{6}(z^{(n)} )^2 + w ^{(n)} z ^{(n)} \Bigr) \circ \pt^{(n)} \,,
\label{annp1}
\end{align} 
and hence 
\begin{align} 
a^{(n+1)} (\theta,t)  = a_0(\pt ^{(n)} (\theta,0)) + \int_0^t  \Bigl(- \tfrac43   (a^{(n)} )^2  + \tfrac{1}{6}(w^{(n)} )^2 + \tfrac{1}{6}(z^{(n)} )^2 + w ^{(n)} z ^{(n)} \Bigr) \circ \pt^{(n)} ds \,. \label{anplus1}
\end{align} 
Using \eqref{eq:u0:ass:2},  \eqref{eq:a0:ass}, and \eqref{eq:w:z:k:a:boot}, we  find that
\begin{align} 
\sabs{a^{(n+1)} (\theta,t)  } \le \mm  + C t  \le 2 \mm \,. \label{plop}
\end{align}

Differentiating \eqref{annp1} gives
\begin{align*} 
\tfrac{d}{dt} \bigl( \p_\theta  a^{(n+1)} \circ \pt ^{(n)}  \ \p_\theta  \pt ^{(n)} \bigr) 
& = \p_\theta  \Bigl(- \tfrac43   (a^{(n)} )^2  + \tfrac{1}{6}(w^{(n)} )^2 + \tfrac{1}{6}(z^{(n)} )^2 + w ^{(n)} z ^{(n)} \Bigr) \circ \pt^{(n)} \  \p_\theta  \pt ^{(n)}  \,,
\end{align*} 
and so
\begin{align*} 
\p_\theta  a^{(n+1)}(\theta,t)  
&= a_0'(\pt ^{(n)} (\theta,0)) \p_\theta  \pt ^{(n)} (y,0) \notag\\
&\qquad + \int_0^t \p_\theta  \Bigl(- \tfrac43   (a^{(n)} )^2  + \tfrac{1}{6}(w^{(n)} )^2 + \tfrac{1}{6}(z^{(n)} )^2 + w ^{(n)} z ^{(n)} \Bigr) \circ \pt^{(n)} \  \p_\theta  \pt ^{(n)} ds \,.
\end{align*} 
Employing the bounds  \eqref{eq:u0:ass:2}, \eqref{eq:Steve:needs:this:1}, \eqref{eq:w:z:k:a:boot}, and \eqref{phixn-bound}, we find that
\begin{align*} 
\sabs{\p_\theta  a^{(n+1)} (\theta,t)  } \le \mm  + C t^ {\frac{1}{3}}   \le 2 \mm \,,
\end{align*} 
which together with \eqref{plop} shows that \eqref{eq:a:boot} holds for $a^{(n+1)}$ given that $R_7$ is defined by \eqref{eq:R:def}.

\vspace{.1 in}
\noindent
{\bf Continuity of $w ^{(n+1)} $, $z ^{(n+1)} $, $k ^{(n+1)} $, and $a ^{(n+1)} $.}
Composing \eqref{wn-alt} with $\ie ^{(n)} $, we see that
\begin{align*} 
 w ^{(n+1)} (\theta,t)  &= w_0(\ie ^{(n)}(\theta,t) ) - {\tfrac{8}{3}}\int_0^t    \left( a^{(n)}  w^{(n)}  \right)( \eta ^{(n)}(\ie ^{(n)}(\theta,t) ,t')  ,t') dt'  \notag \\
 & \qquad  
 + {\tfrac{1}{4}} \int_0^t
 c ^{(n)} ( \eta ^{(n)}(\ie ^{(n)}(\theta,t) ,t')  ,t')  \tfrac{d}{dt'} \left( k ^{(n)} ( \eta ^{(n)}(\ie ^{(n)}(\theta,t) ,t')  ,t')\right) dt' \,.   
\end{align*} 
By Lemma \ref{lem:Upsilon-n}, $\ie ^{(n)} $ is continuous on $\mathcal{D}_{\bar \eps}$, and hence by the definition of the set ${\mathcal X}_{\bar \eps}$ given in
\eqref{eq:w:z:k:a:boot:*}, we see that $w ^{(n+1)} $ is then continuous on $\mathcal{D}_{\bar \eps}$.

Continuity of the shock-intersection time $\stt(\theta,t)$ follows from the continuity of $\pst$ on  $\mathcal{D} _{\bar \eps}$ and the continuity of
$\sc(t)$.   From \eqref{dotzknp1good}, we see that $\zl ^{(n+1)} (t)$ is continuous.  Therefore, the identity \eqref{dotzknp1good} together with
the definition of ${\mathcal X}_{\bar \eps}$ shows that $z ^{(n+1)} $ is continuous on $\mathcal{D}_{\bar \eps}$.  Continuity of $k^{(n+1)}$ follows in the same way
from the identity \eqref{knp1}.   The identity \eqref{anplus1} together with the \eqref{eq:w:z:k:a:boot:*} and the continuity of $a_0$ shows that
 $a ^{(n+1)} $ is  also continuous on $\mathcal{D}_{\bar \eps}$.
\end{proof}

\subsubsection{Contractivity of the iteration map}

We set 
\begin{align*} 
&\delta w^{(n)}  := w ^{(n)} - w ^{(n-1)}  \,, \ \  \delta z^{(n)}  := z ^{(n)} - z ^{(n-1)}  \,, \ \ \delta k^{(n)}  := k ^{(n)} - k ^{(n-1)}  \,, \\
&\delta c^{(n)}  := c ^{(n)} - c ^{(n-1)}  \,, \ \  \delta\lambda_i  := \lambda_i^{(n)} - \lambda_i^{(n-1)} \,,
\end{align*} 
for $i \in \{1,2,3\}$.

 \begin{proposition}[The iteration is contractive]\label{prop:contraction} The map 
$$ (w^{(n)} , z^{(n)}, k^{(n)}, a^{(n)} ) \mapsto 
(w^{(n+1)}, z^{(n+1)}, k^{(n+1)}, a^{(n+1)} ) : {\mathcal X}_{\bar \eps} \to {\mathcal X}_{\bar \eps}$$  satisfies the contractive estimate
\begin{align} 
&\max_{s \in [0,t]} \Bigl( \snorm{  \delta  w^{(n+1)} ( \cdot , s)}_{L^ \infty }
+\snorm{  \delta  z^{(n+1)} ( \cdot , s)}_{L^ \infty }
+\snorm{  \delta  k^{(n+1)} ( \cdot , s)}_{L^ \infty }
+\snorm{  \delta  a^{(n+1)} ( \cdot , s)}_{L^ \infty }
\Bigr) \notag \\
& \qquad \qquad
\le \tfrac{3}{4}   
\max_{s \in [0,t]} \Bigl( \snorm{  \delta  w^{(n)} ( \cdot , s)}_{L^ \infty }
+\snorm{  \delta  z^{(n)} ( \cdot , s)}_{L^ \infty }
+\snorm{  \delta  k^{(n)} ( \cdot , s)}_{L^ \infty }
+\snorm{  \delta  a^{(n)} ( \cdot , s)}_{L^ \infty } \,.  \label{n-contraction}
\end{align} 
\end{proposition} 
\begin{proof}[Proof of Proposition \ref{prop:contraction}]
From \eqref{xland-w:2}, we see that for any $(\theta,t)  \in \mathcal{D} _{\bar \eps}$,
\begin{align*} 
&\p_t \delta w ^{(n+1)}  + \lambda_3^{(n)} \p_\theta  \delta w^{(n+1)} +  \delta\lambda_3^{(n)} \p_\theta  w ^{(n)}  \notag\\
&\quad = \tfrac{1}{4}  c^{(n)} \left( \p_t \delta k ^{(n)} + \lambda _3 ^{(n)} \p_\theta  \delta k ^{(n)}\right)  + \tfrac{1}{4}   \delta\lambda_3^{(n)} c^{(n)} \p_\theta  k ^{(n-1)}  \\
& \quad\quad
+ \tfrac{1}{4} \delta  c^{(n)} \left( \lambda _3 ^{(n)} -\lambda_2^{(n-1)} \right) \p_\theta   k ^{(n-1)}
- \tfrac{8}{3} a ^{(n)} \delta w^{(n)} - \tfrac{8}{3} \delta a ^{(n)} \delta w^{(n-1)}  \,,
\end{align*} 
and thus for all $x \in \Upsilon^{(n)} (t)$, 
\begin{align*} 
&\p_t \left( \delta w ^{(n+1)}  \circ \eta ^{(n)} \right)  \notag\\
&= \tfrac{1}{4}  c^{(n)} \circ \eta^{(n)}  \p_t \left( \delta k ^{(n)} \circ \eta ^{(n)} \right)  
+  ( \delta w^{(n)} +  \tfrac{1}{3}  \delta z ^{(n)} ) \left(  \tfrac{1}{4} c^{(n)} \p_\theta  k ^{(n-1)}  -\p_\theta  w ^{(n)}\right)    \circ \eta ^{(n)} \\
& \quad
+ \tfrac{1}{8}( \delta  w^{(n)} + \delta z ^{(n)} ) \left( \lambda _3 ^{(n)} -\lambda_2^{(n-1)} \right) \p_\theta   k ^{(n-1)} \circ \eta ^{(n)} 
- \tfrac{8}{3}  \delta w^{(n)} a ^{(n)} \circ \eta ^{(n)}  - \tfrac{8}{3} \delta a ^{(n)}  w^{(n-1)} \circ \eta ^{(n)}  .
\end{align*} 
Using \eqref{cn0} and integrating by parts in time, 
 \begin{align*} 
 &\tfrac{1}{4} \int_0^t  c^{(n)} \circ \eta^{(n)}  \p_s \left( \delta k ^{(n)} \circ \eta ^{(n)} \right) ds 
 = \tfrac{1}{4} c^{(n)} \delta k ^{(n)} \circ \eta ^{(n)} \\
 & \qquad\qquad
+ \tfrac{1}{8} \int_0^t \left( \lambda_3^{(n-1)} w_\theta ^{(n)} + \lambda_1^{(n-1)} z_\theta ^{(n)} + \tfrac{64}{3} a ^{(n-1)} c ^{(n-1)} \right)
 \delta k ^{(n)}  \circ \eta ^{(n)} ds \,,
 \end{align*} 
 and thus, we have that
 \begin{align} 
 \delta w ^{(n+1)}  \circ \eta ^{(n)}  &= -\int_0^t w_\theta ^{(1)}  \delta w^{(n)}  \circ \eta ^{(n)} ds
 -\int_0^t (w_\theta ^{(n)} - w_\theta^{(1)}  ) \delta w^{(n)}  \circ \eta ^{(n)} ds
 +  \tfrac{1}{4}  \delta k ^{(n)} c^{(n)} \circ \eta ^{(n)} \notag\\
&\qquad
 + \tfrac{1}{8} \int_0^t  \delta k ^{(n)}  \left( \lambda_3^{(n-1)} w_\theta ^{(n)} + \lambda_1^{(n-1)} z_\theta ^{(n)} + \tfrac{64}{3} a ^{(n-1)} c ^{(n-1)} \right)
 \circ \eta ^{(n)} ds \notag  \\
 &\qquad
 + \tfrac{1}{8}\int_0^t 
 \delta  w^{(n)} \left( \left( 2 w ^{(n)} -  \tfrac{2}{3}  z ^{(n)}   -\lambda_2^{(n-1)} \right) -\tfrac{64}{3} a ^{(n)}  \right)\p_\theta   k ^{(n-1)} \circ \eta ^{(n)}
 \notag
 \\
 & \qquad
 +   \tfrac{1}{24}\int_0^t \delta z ^{(n)}  \left(   (4 w^{(n)} + z^{(n)} - 3\lambda _3 ^{(n-1)} ) \p_\theta  k ^{(n-1)}  -8 \p_\theta  w ^{(n)}\right)    \circ \eta ^{(n)}  ds
 \notag
 \\
 & \qquad
 - \tfrac{8}{3}\int_0^t \delta a ^{(n)}  w^{(n-1)} \circ \eta ^{(n)}ds \,. \label{deltawnp1}
  \end{align} 
  Appealing to \eqref{eq:AC:DC:con} and \eqref{eq:w:z:k:a:boot}, we find that
\begin{align} 
 \max_{s \in [0,t]}  \snorm{  \delta  w^{(n+1)} ( \cdot , s)}_{L^ \infty }
& \le ( \tfrac{1}{2} +C  t^ {\frac{1}{2}}) \max_{s \in [0,t]}   \snorm{  \delta  w^{(n)} ( \cdot , s)}_{L^ \infty }
 + C  \max_{s \in [0,t]} \snorm{  \delta  k^{(n)} ( \cdot , s)}_{L^ \infty } \notag  \\
& \qquad
  +( \tfrac{1}{6}   + C t^{\frac{1}{2}})  \max_{s \in [0,t]} \snorm{  \delta  z^{(n)} ( \cdot , s)}_{L^ \infty }  
   + C t \max_{s \in [0,t]} \snorm{  \delta  a^{(n)} ( \cdot , s)}_{L^ \infty }   \,. \label{wn-contract}
\end{align}

Using the evolution of $ z ^{(n)} $ given by \eqref{z-alt},  in the same way that we obtained \eqref{deltawnp1}, we find that for any $(\theta,t)  \in \mathcal{D} ^{z,{(n)} }_{\bar\eps}$, 
 \begin{align*} 
 \delta z ^{(n+1)} (\theta,t) & = \delta \zl^{(n+1)} ( \sc(\stt(\theta,t))) +  \tfrac{1}{4}  (\delta \kl ^{(n)} c^{(n)} ) ( \sc(\stt(\theta,t))) 
 - \tfrac{1}{4}  \delta k ^{(n)} c^{(n)} (\theta,t)  \\
& + \tfrac{1}{4}   \int_{\stt(\theta,t)}^t  \delta k ^{(n)} \left( ( \lambda_1^{(n)} -\lambda_1^{(n-1)}) c_\theta^{(n)}  +  \tfrac{2}{3} c ^{(n-1)} w_\theta^{(n)} 
 -  \tfrac{8}{3} a ^{(n-1)} c ^{(n-1)}    \right) \circ \pst ^{(n)} (\theta,s)  ds  \\
 &+ \int_{\stt(\theta,t)}^t \delta w ^{(n)} \left( -  \tfrac{1}{3}  z_\theta ^{(n)} - \tfrac{1}{12} c^{(n)} k_\theta ^{(n-1)} +  \tfrac{1}{3}  
 c^{(n-1)} k_\theta ^{(n-1)}  \right) \circ \pst ^{(n)} (\theta,s)  ds  \\
  &+ \int_{\stt(\theta,t)}^t \delta z ^{(n)} \left(   z_\theta ^{(n)} - \tfrac{1}{4} c^{(n)} k_\theta ^{(n)} +  \tfrac{1}{3}  
 c^{(n-1)} k_\theta ^{(n-1)} -  \tfrac{8}{3}  a ^{(n)}   \right) \circ \pst ^{(n)} (\theta,s)  ds  \\
 &- \tfrac{8}{3}\int_{\stt(\theta,t)}^t \delta a ^{(n)} z ^{(n-1)}  \circ \pst ^{(n)} (\theta,s)  ds  \,.
 \end{align*} 
 Using this identity together with  \eqref{eq:Steve:needs:this:1}, \eqref{eq:w:z:k:a:boot}, and \eqref{zknp1:shock} shows that
  \begin{align} 
 \max_{s \in [0,t]}  \snorm{  \delta  z^{(n+1)} ( \cdot , s)}_{L^ \infty }
& \le C t  \max_{s \in [0,t]}   \snorm{  \delta  w^{(n+1)} ( \cdot , s)}_{L^ \infty }
+C t^{\frac{3}{2}}   \max_{s \in [0,t]}   \snorm{  \delta  w^{(n)} ( \cdot , s)}_{L^ \infty }
 \notag  \\
&   + C  \max_{s \in [0,t]} \snorm{  \delta  k^{(n)} ( \cdot , s)}_{L^ \infty }
  + C t  \max_{s \in [0,t]} \snorm{  \delta  z^{(n)} ( \cdot , s)}_{L^ \infty }  
   + C t^{\frac{5}{2}}   \max_{s \in [0,t]} \snorm{  \delta  a^{(n)} ( \cdot , s)}_{L^ \infty }  .  \label{zn-contract}
\end{align} 
Next, the identity \eqref{kn-eq} together with the bound \eqref{zknp1:shock} provides us with the estimates
\begin{align} 
\max_{s \in [0,t]} \snorm{  \delta  k^{(n+1)} ( \cdot , s)}_{L^ \infty } 
&\leq C t \max_{s \in [0,t]}   \snorm{  \delta  w^{(n+1)} ( \cdot , s)}_{L^ \infty } \,, 
\notag\\
 \max_{s \in [0,t]} \snorm{  \delta  k^{(n)} ( \cdot , s)}_{L^ \infty } 
 &\leq C t \max_{s \in [0,t]}   \snorm{  \delta  w^{(n)} ( \cdot , s)}_{L^ \infty } \,.
  \label{kn-contract}
\end{align} 
Finally, using \eqref{xland-a:2}, we find that for $(\theta,t)  \in \mathcal{D}^{k,{(n)} }_{\bar \eps}$,
\begin{align*} 
\delta a ^{(n+1)} (\theta,t)  & =  
 \int_0^t \delta w ^{(n)} \left(  \tfrac{1}{3}  \delta w ^{(n)} + z ^{(n)} - \tfrac{2}{3}  a_\theta ^{(n)} \right) \circ \pt ^{(n)} ds  \\
& \qquad 
+ \int_0^t \delta z ^{(n)} \left(  \tfrac{1}{3}  \delta z ^{(n)} + w ^{(n-1)} - \tfrac{2}{3}  a_\theta ^{(n)} \right) \circ \pt ^{(n)} ds  
- \int_0^t \delta a ^{(n)}  \delta a ^{(n)} \circ \pt ^{(n)} ds   \,,
\end{align*} 
and therefore
\begin{align} 
& \max_{s \in [0,t]}  \snorm{  \delta  a^{(n+1)} ( \cdot , s)}_{L^ \infty }\notag\\
& \le   
C t \max_{s \in [0,t]}  \Bigl(    \snorm{  \delta  w^{(n)} ( \cdot , s)}_{L^ \infty } 
+   \snorm{  \delta  z^{(n)} ( \cdot , s)}_{L^ \infty } +   \snorm{  \delta  k^{(n)} ( \cdot , s)}_{L^ \infty } 
   +   \snorm{  \delta  a^{(n)} ( \cdot , s)}_{L^ \infty }  \Bigr) \,.   \label{an-contract}
\end{align} 
Summing the inequalities  \eqref{wn-contract}--\eqref{an-contract} yields
\begin{align*} 
&\max_{s \in [0,t]} \Bigl( \snorm{  \delta  w^{(n+1)} ( \cdot , s)}_{L^ \infty }
+\snorm{  \delta  z^{(n+1)} ( \cdot , s)}_{L^ \infty }
+\snorm{  \delta  k^{(n+1)} ( \cdot , s)}_{L^ \infty }
+\snorm{  \delta  a^{(n+1)} ( \cdot , s)}_{L^ \infty }
\Bigr) \\
&  \qquad
\le \tfrac{1}{2} \max_{s \in [0,t]} \snorm{  \delta  w^{(n)} ( \cdot , s)}_{L^ \infty } + \tfrac{1}{6} \max_{s \in [0,t]} \snorm{  \delta  z^{(n)} ( \cdot , s)}_{L^ \infty } + C t \max_{s \in [0,t]} \snorm{  \delta  w^{(n+1)} ( \cdot , s)}_{L^ \infty }  \\
&  \qquad\quad
 + C t^{\frac{1}{2}}  \max_{s \in [0,t]} \bigl(  \snorm{  \delta  w^{(n)} ( \cdot , s)}_{L^ \infty } + \snorm{  \delta  z^{(n)} ( \cdot , s)}_{L^ \infty } \bigr)
 + C t  \max_{s \in [0,t]} \bigl(  \snorm{  \delta  k^{(n)} ( \cdot , s)}_{L^ \infty } + \snorm{  \delta  a^{(n)} ( \cdot , s)}_{L^ \infty } \bigr) \,.
\end{align*} 
Choosing $ \bar \eps $ sufficiently small, we obtain the bound \eqref{n-contraction}.
\end{proof}

\subsubsection{Convergence of the iteration scheme}
\label{sec:convergence:of:scheme}
We define $\yy=\theta -\sc(t)$ and
$$
{\mathsf w} (\yy, t) = w(\theta,t) \,, \ \ {\mathsf z} (\yy,t) = z(\theta,t)  \,, \ \ {\mathsf k} (\yy,t) = k(\theta,t)  \,, \ \ {\mathsf a} (\yy,t) = a(\theta,t)  \,.
$$ 
The space-time gradient is denoted as $\nabla _{\yy,t}$, and it is convenient to introduce
$$
\Dsf_{\bar\eps} = ( \mathbb{T}  \setminus \{0\}) \times (0,\bar\eps) 
\,.
$$

The contractive estimate \eqref{n-contraction} shows that $( {\mathsf w}  ^{(n)} , {\mathsf z}  ^{(n)} , {\mathsf k}  ^{(n)} , {\mathsf a}  ^{(n)} ) 
\to ({\mathsf w} ,{\mathsf z} ,{\mathsf k} ,{\mathsf a} )$ uniformly in 
$\Dsf_{\bar \eps}$, and in particular we have that
\begin{align} 
 \lim_{n\to \infty }   \snorm{ {\mathsf w}  - {\mathsf w}  ^{(n)} }_{ L^ \infty (\Dsf_{\bar\eps})} = 0  \,.  \label{C0-convergence}
\end{align}
 
  Let us now describe the bounds on derivatives. 
According to  \eqref{eq:w:dx:boot} and \eqref{eq:w:dt:boot}, for all $\yy\neq0$ and $t\in[0,\bar\eps]$,  we have that
\begin{align*} 
\norm{  \left( t^3 + \yy^2 \right)^ {\frac{1}{6}} \nabla_{\yy,t} ({\mathsf w}  ^{(n)} -\wb)  }_{ \Dsf_{\bar\eps} } \le 
C  \,.
\end{align*} 
By the Banach–Alaoglu theorem, there exists a  limiting function $f$ and a  subsequence such that
$$
 \left( t^3 + \yy^2 \right)^ {\frac{1}{6}}  \nabla _{\yy,t} {\mathsf w}  ^{(n')} \rightharpoonup
  \left( t^3 + \yy^2 \right)^ {\frac{1}{6}} f  \,,
$$
the convergence in $L^ \infty (  \Dsf_{\bar\eps})$ weak-*.   Let us show that $f = \nabla _{\yy,t} {\mathsf w} $, the weak derivative of the uniform limit ${\mathsf w} $, 
and that the convergence holds for any subsequence.   For  test functions
$\varphi \in W^{1,1}_0 (\Dsf_{\bar\eps})$, 
\begin{align*} 
 \lim_{n\to \infty } \int_{ \Dsf_{\bar\eps}} \bigl( {\mathsf w}  - {\mathsf w}  ^{(n)}  \bigr)  \p_\yy\Bigl( \left( t^3 + \yy^2 \right)^ {\frac{1}{6}} \varphi \Bigr)\, d\yy dt 
 &= 
 \lim_{n\to \infty }   \tfrac{1}{3}  \int_{ \Dsf_{\bar\eps}} \bigl( {\mathsf w}  - {\mathsf w}  ^{(n)}  \bigr)  \yy  \left( t^3 + \yy^2 \right)^ {-\frac{5}{6}} \varphi \, d\yy dt  \\
 & \qquad \lim_{n\to \infty } \int_{ \Dsf_{\bar\eps}} \bigl( {\mathsf w}  - {\mathsf w}  ^{(n)}  \bigr)  \Bigl( \left( t^3 + \yy^2 \right)^ {\frac{1}{6}} \p_\yy\varphi \Bigr)\, d\yy dt \,.
\end{align*} 
It follows that
\begin{align*} 
&\abs{ \lim_{n\to \infty } \int_{ \Dsf_{\bar\eps}} \bigl( {\mathsf w}  - {\mathsf w}  ^{(n)}  \bigr)  \p_\yy\Bigl( \left( t^3 + \yy^2 \right)^ {\frac{1}{6}} \varphi \Bigr)\, d\yy dt } \\
& \qquad
\le \lim_{n\to \infty }   \snorm{ {\mathsf w}  - {\mathsf w}  ^{(n)} }_{ L^ \infty (\Dsf_{\bar\eps})}
\Bigl( \tfrac{1}{3}  \snorm{ \varphi }_{ L^ \infty (\Dsf_{\bar\eps})}  \int_{\Dsf_{\bar\eps}} \yy^ {-\frac{2}{3}} d\yy dt
+2 \int_{\Dsf_{\bar\eps}}\p_\yy \varphi\, d\yy dt
\Bigr) = 0 
\end{align*} 
by \eqref{C0-convergence}.   Similarly, if we replace $\p_\yy$ with $\p_t$, then the integral $\tfrac{1}{3}  \int_{\Dsf_{\bar\eps}} \yy^ {-\frac{2}{3}} d\yy dt$
is replaced with $ \tfrac{1}{2} \int_{\Dsf_{\bar\eps}}t^ {-\frac{1}{2}} d\yy dt$ - $\tfrac{1}{3}  \int_{\Dsf_{\bar\eps}} \dot\sc(t) \yy^ {-\frac{2}{3}} d\yy dt$, and the same conclusion holds, since again both integrals are bounded (using \eqref{eq:sc:ass}).  This shows that\footnote{In fact,  
$ \left( t^3 + \yy^2 \right)^ {\frac{1}{6}} {\mathsf w}  ^{(n)}\rightharpoonup  \left( t^3 + \yy^2 \right)^ {\frac{1}{6}}{\mathsf w}   $ in  $W^{1, \infty }(  \Dsf_{\bar\eps})$ 
weak-*.} 
$$
 \left( t^3 + \yy^2 \right)^ {\frac{1}{6}}  \nabla _{\yy,t} {\mathsf w}  ^{(n)} \rightharpoonup
  \left( t^3 + \yy^2 \right)^ {\frac{1}{6}} \nabla _{\yy,t} {\mathsf w}   \qquad  \text{ in  $L^ \infty (  \Dsf_{\bar\eps})$ weak-*},
$$
and hence we have by lower semi-continuity that ${\mathsf w} $ satisfies  \eqref{eq:w:boot}, \eqref{eq:w:dx:boot}, and \eqref{eq:w:dt:boot}.
The weak convergence for $(\p_\yy {\mathsf z} ^{(n)} ,\p_\yy {\mathsf k}^{(n)}  , \p_\yy {\mathsf a}^{(n)}  )\rightharpoonup (\p_\yy {\mathsf z} ,\p_\yy {\mathsf k} , \p_\yy {\mathsf a} )$ in $L^ \infty (  \Dsf_{\bar\eps})$ weak-* is standard.  We conclude that
\begin{align} 
(w,z,k,a) \in {\mathcal X}_{\bar \eps} \label{for-shock-stuff} \,.
\end{align} 
Let $\varphi \in C^ \infty_0 ( \Dsf_{\bar\eps}) $.
Integration of \eqref{xland-w:3} shows that
\begin{align*} 
&   \int_{ \Dsf_{\bar\eps}}\Bigl(
\p_t {\mathsf w}  ^{(n+1)} +  \bigl({\mathsf w}  ^{(n)}-\wb \bigr)  \p_\yy {\mathsf w}  ^{(n+1)} \notag\\
&\qquad \qquad +  \bigl( \tfrac{1}{3}{\mathsf z}  ^{(n)} + \wb -\dot\sc(t)\bigr)  \p_\yy {\mathsf w}  ^{(n+1)}  + \tfrac{8}{3}{\mathsf a} ^{(n)} {\mathsf w} ^{(n)} 
- \tfrac{1}{6}  (c ^{(n)})^2 \p_\yy {\mathsf k} ^{(n)} \Bigr)  \varphi  \, d\yy dt \\
& = \overbrace{ \int_{ \Dsf_{\bar\eps}}\Big( \p_t {\mathsf w} ^{(n+1)} + \bigl({\mathsf w}  ^{(n)}-\wb \bigr)  \p_\yy {\mathsf w}  ^{(n+1)}\Bigl) \varphi d\yy dt}^{ \mathcal{I}_1 ({\mathsf w}  ^{(n)}  )  }\\
&\qquad \qquad  
+\underbrace{\int_{ \Dsf_{\bar\eps}} \bigl( \tfrac{1}{3}{\mathsf z}  ^{(n)} + \wb -\dot\sc(t)\bigr)  \p_\yy {\mathsf w}  ^{(n+1)}  + \tfrac{8}{3}{\mathsf a} ^{(n)} {\mathsf w} ^{(n)} 
- \tfrac{1}{6}  (c ^{(n)})^2 \p_\yy {\mathsf k} ^{(n)} \Bigr)  \varphi  \, d\yy dt}_{ \mathcal{I} _2 ({\mathsf w}  ^{(n)} , {\mathsf z} ^{(n)} )} \,.
\end{align*} 
Its clear that $ \mathcal{I}_2( {\mathsf w}  ^{(n)} , {\mathsf z}  ^{(n)} , {\mathsf k}  ^{(n)} , {\mathsf a} ^{(n)} ) \to \mathcal{I} _2( {\mathsf w} ,{\mathsf z} ,{\mathsf k} , {\mathsf a} )$.  Let us show that $ \mathcal{I}_1( {\mathsf w}  ^{(n)} ) \to \mathcal{I} _1( {\mathsf w} )$.   
We have that
\begin{align*} 
\abs{ \mathcal{I} _1( {\mathsf w}  ) -  \mathcal{I}_1( {\mathsf w}  ^{(n)}  ) } 
& \le 
\abs{ \int_{ \Dsf_{\bar\eps}}  \left( t^3 + \yy^2 \right)^ {\frac{1}{6}} \bigl( \p_t {\mathsf w}  - \p_t {\mathsf w}  ^{(n+1)} \bigr) \  \varphi  \,  \left( t^3 + \yy^2 \right)^ {-\frac{1}{6}} \, d\yy dt } \\
& \qquad   \qquad
+
\abs{ \int_{ \Dsf_{\bar\eps}}  \left( t^3 + \yy^2 \right)^ {\frac{1}{6}} \bigl( \p_\yy {\mathsf w}  - \p_\yy {\mathsf w}  ^{(n+1)} \bigr) \  \varphi  \, ({\mathsf w} -\wb)  \left( t^3 + \yy^2 \right)^ {-\frac{1}{6}} \, d\yy dt } \\
& \qquad   \qquad
+ \snorm{ {\mathsf w}  - {\mathsf w}  ^{(n)} }_{ L^ \infty (\Dsf_{\bar\eps})}  \abs{ \int_{ \Dsf_{\bar\eps}}  \left( t^3 + \yy^2 \right)^ {\frac{1}{6}}\p_\theta  {\mathsf w}  ^{(n+1)} \varphi
 \left( t^3 + \yy^2 \right)^ {-\frac{1}{6}} \, d\yy dt }  \,.
\end{align*} 
Since $\left( t^3 + \yy^2 \right)^ {-\frac{1}{6}}  \in L^1 ( \Dsf_{\bar\eps})$, we see that the first two summands converges to $0$ by weak-* convergence in
 $L^ \infty (  \Dsf_{\bar\eps})$, while the second term converges to $0$ by the strong convergence \eqref{C0-convergence}.   It follows
 that ${\mathsf w} $ satisfies
 $$
 \int_{ \Dsf_{\bar\eps}} 
\Bigl( \p_t {\mathsf w}  +  \lambda_3   \p_\yy {\mathsf w}  + \tfrac{8}{3}{\mathsf a}  {\mathsf w}  - \tfrac{1}{6}  c^2 \p_\yy  {\mathsf k} \Bigr)  \varphi  
\, d\yy dt  =0 \,, $$
and together with the standard weak convergence argument for the other variables, we have that $({\mathsf w} ,{\mathsf z} ,{\mathsf k} ,{\mathsf a} )$ are solutions to
\eqref{eq:wzka} in $ \mathcal{D} _{\bar \eps}$.

Thanks to the uniform convergence $( { w}  ^{(n)} , { z}  ^{(n)} , { k}  ^{(n)} , {a}  ^{(n)} ) 
\to ({ w} ,{ z} , {k} ,{ a} )$  in $ \mathcal{D} _{\bar \eps}$,  it follows that the time derivatives
$\p_s (\eta ^{(n)} , \pt ^{(n)} , \pst^{(n)} ) \to \p_s (\eta  , \pt , \pst ) $ uniformly, and that 
$\eta (\cdot,t)  \colon\Upsilon(t)\to \TT \setminus \{\sc(t)\}$ is a bijection, and
the inverse map $\ie \colon \DD_{\bar \eps} \to \TT \setminus \{0\}$ is continuous in spacetime, where
the set of labels $\Upsilon^{(n)}  (t) \to \Upsilon (t)$ in the sense that
$\Upsilon (t) = \TT \setminus[x_-(t),x_+(t)] $
and $x^{(n)} _-(t) \to x_-(t)$ and       $x^{(n)}_+(t) \to x_+(t)$ uniformly.

Moreover, the uniform convergence $( { w}  ^{(n)} , { z}  ^{(n)} , { k}  ^{(n)} , {a}  ^{(n)} ) 
\to ({ w} ,{ z} , {k} ,{ a} )$  in $ \mathcal{D} _{\bar \eps}$, combined with the definitions \eqref{eq:hate:iterations} and the  continuity of $ \mathcal{E} _1 $ and $ \mathcal{E} _2$, implies that ${\mathcal E}_{1}(\vl,\vr,\zl,\el) = {\mathcal E}_{2}(\vl,\vr,\zl,\el) =0 $. Thus, the  equations relating $\zl$ and $\kl$ to $w_-$ and $w_+$ hold on the given shock curve.

\subsection{Proof of Proposition~\ref{thm:curve:determines:all}}
\label{sec:kucf}
The analysis given in Sections~\ref{sec:Burgers:explicit}--\ref{sec:construction:iteration} completes the proof of Propoisition~\ref{thm:curve:determines:all}, here we just summarize our findings. Given a regular shock curve $\sc$ satisfying \eqref{eq:sc:ass}, we have shown that there exists $\bar \eps>0$ sufficiently small (solely in terms of $\kappa,\bb,\cc,\mm$) such that the iteration described in Section~\ref{sec:construction:iteration} produces a limit point $(w,z,k,a) \in {\mathcal X}_{\bar \eps}$ (see \eqref{for-shock-stuff}), which solves the azimuthal form of the Euler equations~\eqref{eq:wzka} in $\DD_{\bar \eps}$; this proves items~\ref{item:5.6.i}, and \ref{item:5.6.ii}. From the last paragraph of the above section, we have that $(w_-,w_+,z_-,k_-)$ satisfy the system of algebraic equations \eqref{pjump77}-\eqref{pjump7}, arising from the Rankine--Hugoniot conditions, and by passing $n\to \infty$ in \eqref{eq:jump:mean:sharp} and \eqref{zknp1:shock}, we have that $\jump{w}$, $\jump{z}$, and $\jump{k}$ satisfy the bounds claimed in \eqref{eq:J:M:rough} and respectively~\eqref{eq:zl:and:kl:on:shock:L:infinity}; this proves items~\ref{item:5.6.iii}, \ref{item:5.6.v}, \ref{item:5.6.vi}, and~\ref{item:5.6.vii}. The stated bounds on $\sc_1$ and $\sc_2$, which are uniform limits of $\sc_1^{(n)}$ and $\sc_2^{(n)}$, follow by passing $n\to \infty$ in Lemma~\ref{lem:12flows},  proving item~\ref{item:5.6.iv}.

\subsection{Evolution of the shock curve}
\label{sec:shock:evo}

Proposition~\ref{thm:curve:determines:all} shows that given a shock curve $(\sc(t),t)_{t\in [0,\bar \eps]}$ which satisfies assumptions~\eqref{eq:sc:ass}, we may compute a  solution $(w,z,k,a)$ of  the  azimuthal form of the Euler equations \eqref{eq:w:z:k:a}--\eqref{eq:wave-speeds} on the spacetime region $ \DD_{\bar \eps} = (\TT \times [0,\bar \eps]) \setminus (\sc(t),t)_{t\in [0,\bar \eps]}$; moreover, this solution exhibits a jump discontinuity from the $(w_+,0,0)$ state on the right of the shock curve to the state $(w_-,z_-,k_-)$ on the left of the shock curve, and this jump is consistent with the system of algebraic equations \eqref{pjump77}--\eqref{pjump7} arising from the Rankine--Hugoniot conditions. Throughout this section we shall implicitly use that we have a map 
\begin{align}
(\sc, w_0, a_0) \xrightarrow{\mbox{Proposition~\ref{thm:curve:determines:all}}} (w,z,k,a) \,.
\label{eq:shock:to:solution:map}
\end{align}
Since at this stage of the proof uniqueness has not yet been established (this is achieved in Section~\ref{sec:uniqueness} below), in the map \eqref{eq:shock:to:solution:map} we select any one of the solutions guaranteed by Proposition~\ref{thm:curve:determines:all}. 

We note that throughout the proof of Proposition~\ref{thm:curve:determines:all}, the shock curve itself is {\em fixed}, and does not solve an evolutionary equation. The goal of this section is to provide an iteration scheme whose fixed point $\sc$ is a $C^2$ smooth curve which solves the equation \eqref{sdot1} (recall that in view of Lemma~\ref{lem:zminus-on-shock} the jump conditions \eqref{sdot1} and \eqref{sdot2} are equivalent), which we recall is 
\begin{align}
\dot \sc(t) = \FFF_{\sc}(t) \,, \qquad \sc(0) = 0 \,,
\label{eq:shock:ODE}
\end{align}
where 
\begin{align}
\FFF_{\sc}(t)
&= \frac 23  \frac{ (\vl(t) - \zl(t))^2 (\vl(t)+\zl(t)) - \vr(t)^3 }{ (\vl(t)-\zl(t))^2 - \vr(t)^2} 
\label{eq:F:def}
\end{align}
and we have implicitly used the notation \eqref{eq:left:right:states} to denote the limits from the left (indicated by a $-$ index) and the limit from the right (indicated by a $+$ index) at the shock point $(\sc(t),t)$ for the functions $(w,z,k)$. We emphasize however that the $(\vl,\vr,\zl,\kl)$ appearing in \eqref{eq:F:def} do not just depend on $\sc$ because they are one sided limits of their respective functions $(w,z,k)$ on the curve $(\sc(t),t)$; they also depend on $\sc$ because the functions $(w,z,k)$ themselves arise from the mapping \eqref{eq:shock:to:solution:map} given by Proposition~\ref{thm:curve:determines:all}; this mapping is {\em implicit} and {\em nonlinear}. Moreover, we note that due to Lemma~\ref{lem:zminus-on-shock} the $\zl$ and $\kl$ appearing in \eqref{eq:F:def} are themselves smooth functions of $\vl$ and $\vr$, so that $\FFF_\sc$ is truly a function that depends solely on $\vl$ and $\vr$, or alternatively, $\jump{w}$ and $\mean{w}$.

\subsubsection{Properties of $\FFF_{\sc}$}
Before giving the iteration scheme used to construct a solution to \eqref{eq:shock:ODE}, we establish a few useful properties of the function $\FFF_\sc$ defined in \eqref{eq:F:def}. 

\begin{lemma}
\label{lem:F:xi:t:approx}
Assume that $\sc$ satisfies \eqref{eq:sc:ass}, let $(w,z,k)$ be defined via \eqref{eq:shock:to:solution:map}, and $\FFF_{\sc}$ be given by~\eqref{eq:F:def}.
We then have that  
\begin{subequations}
\begin{align}
\abs{\FFF_{\sc}(t) - \kappa } 
&\leq  \tfrac 12  {\mm^4} t  \,,
\label{eq:F:xi:t:approx:1}\\
\abs{\tfrac{d}{dt} \FFF_{\sc}(t) } 
&\leq 5  {\mm^4} \,,
\label{eq:dot:F:xi:t:approx}
\end{align}
\end{subequations}
for all $t\in (0,\bar \eps]$.
\end{lemma}
\begin{proof}[Proof of Lemma~\ref{lem:F:xi:t:approx}]
First we note that the function $w$ satisfies \eqref{eq:J:M:rough} with 
\begin{align}
\Rsf_j = 1 + 2 R_1 =  1 + 100 \mm^2 \leq  \mm^3
\qquad \mbox{and} \qquad 
\Rsf_m = \tfrac 13  {\mm^4}  + R_1 \leq \tfrac 13  {\mm^4} +  \mm^3  
\,.
\label{eq:Rj:Rm:explicit}
\end{align} 
This holds in view of \eqref{eq:basic:jump}, \eqref{eq:w:boot}, and the definition of $R_1$ in \eqref{eq:R:def}.

Due to \eqref{zm-on-xi}, in order to approximate the function $\FFF_{\sc}$ it is natural to insert $- \frac{9}{16} \jump{w}^3 \mean{w}^{-2}$  in \eqref{eq:F:def}, instead of $\zl$. Using the identities $\vl = \mean{w} + \frac 12 \jump{w}$ and $\vr = \mean{w} - \frac 12 \jump{w}$, this gives us the leading order terms in $\FFF_{\sc}$ defined by
\begin{align}
\FFF_{\sc}^{\rm app}
&= \frac 23  \frac{ (\vl + \frac{9 \jump{w}^3}{16\mean{w}^{2}}  )^2 (\vl - \frac{9 \jump{w}^3}{16\mean{w}^{2}}) - \vr^3 }{ (\vl + \frac{9 \jump{w}^3}{16\mean{w}^{2}})^2 - \vr^2} \,.
\label{eq:F:app:def}
\end{align}
Furthermore, since the formula in \eqref{eq:F:app:def} is explicit, using \eqref{eq:J:M:rough} we obtain that 
\begin{align}
\abs{\FFF_{\sc}^{\rm app} - \mean{w} + \frac{7 \jump{w}^2}{24 \mean{w}}}
&\leq C t^{\frac 32}
\label{eq:orange:orangutan:44}
\end{align}
since $t \leq \bar \eps$, and $\bar \eps$ is sufficiently small; here $C = C(\kappa,\bb,\cc,\mm)>0$. The error we make in the approximation \eqref{eq:F:app:def} may be bounded using the intermediate value theorem and the bounds \eqref{eq:J:M:rough}, \eqref{zm-on-xi}, \eqref{eq:zl:on:shock:L:infinity} as
\begin{align}
\abs{\FFF_{\sc}  - \FFF_{\sc}^{\rm app}} 
&\leq \frac 23 \abs{\zl + \frac{9\jump{w}^3}{16\mean{w}^{2}}  }  \abs{1 + \frac{  \jump{w} ( \mean{w} - \frac 12 \jump{w})}{(\jump{w} - z_*)^2} - \frac{ \mean{w}^2 - \frac 14 \jump{w}^2}{(2 \mean{w} - z_*)^2}}
\notag\\
&\leq C t^{\frac 52} \left( 1 + \frac{3 \bb^{\frac 32} \kappa t^{\frac 12}}{(\bb^{\frac 32} t^{\frac 12} - 5 \kappa^{-2} \bb^{\frac 92} t^{\frac 32})^2} +\frac{\kappa^2}{4 (\kappa - 5 \kappa^{-2} \bb^{\frac 92} t^{\frac 32})^2}  \right)
\notag\\
&\leq C t^{2} 
\label{eq:orange:orangutan:4}
\end{align}
since $t\leq \bar \eps$ is sufficiently small; here $z_*$ lies in between $\zl$ and $- \frac{9}{16} \jump{w}^3 \mean{w}^{-2}$, and $C = C(\kappa,\bb,\cc,\mm)>0$.
Combining \eqref{eq:orange:orangutan:44}--\eqref{eq:orange:orangutan:4} and \eqref{eq:J:M:rough} --- with $\Rsf_j$ and $\Rsf_m$ as determined by \eqref{eq:Rj:Rm:explicit}, we arrive at 
\[
\abs{\FFF_{\sc}(t) - \kappa } 
\leq \left( \tfrac 13  {\mm^4} +  \mm^3   + 2  \bb^3 \kappa^{-1} \right) t
\leq  \tfrac 12  {\mm^4} \,,
\]
thereby proving \eqref{eq:F:xi:t:approx:1}. In this last inequality we have also appealed to \eqref{eq:b:m:ass}.

In order to prove \eqref{eq:dot:F:xi:t:approx}, we first differentiate \eqref{eq:F:def} with respect to $t$, to arrive at
\begin{align}
 \frac 32 \frac{d}{dt} \FFF_{\sc} 
&=   \left( 1 + \frac{  \jump{w} ( \mean{w} - \frac 12 \jump{w})}{(\jump{w} - \zl)^2} - \frac{ \mean{w}^2 - \frac 14 \jump{w}^2}{(2 \mean{w} - \zl)^2}\right)\frac{d}{dt} \zl \notag\\
&+  \frac{\jump{w}^3 - 2 \jump{w}^2 \zl + \jump{w} \zl^2 + \zl (2 \mean{w} - \zl)^2}{2 (\jump{w} - \zl)^2 (2 \mean{w} - \zl)}  \frac{d}{dt} \jump{w} \notag\\
&+ \left( 1 + \frac{( \mean{w} - 2 \jump{w} ) ( \mean{w}  +\frac 12 \jump{w} -   \zl) (\jump{w} + \zl)}{2 (\jump{w} - \zl) (\mean{w} - \frac 12 \zl)^2} \right) \frac{d}{dt} \mean{w} 
\label{eq:orange:orangutan:444}
\end{align}
By combining \eqref{eq:orange:orangutan:444} with the bounds the derivative bounds \eqref{eq:spanish:sparkling:wine} (which holds due to \eqref{2krod} with constant $\Rsf = R_1 \leq \mm^3$ as defined in \eqref{eq:R:def}), \eqref{eq:dt:zl:kl:on:shock}, and the amplitude estimates  \eqref{eq:J:M:rough} (with \eqref{eq:Rj:Rm:explicit}) and \eqref{eq:zl:on:shock:L:infinity}, we arrive at 
\begin{align}
\frac 32 \abs{\frac{d}{dt} \FFF_{\sc}} 
&\leq \kappa \bb^{-\frac 32} t^{-\frac 12}  \abs{\frac{d}{dt} \zl} + 2 \bb^{\frac 32} \kappa^{-1} t^{\frac 12} \abs{\frac{d}{dt} \jump{w}} + \left( \frac 32 + 2 \bb^3 \kappa^{-2} t \right) \abs{\frac{d}{dt} \mean{w}}
\notag\\
&\leq \kappa \bb^{-\frac 32} t^{-\frac 12}  \left( 8 \bb^{\frac 92} \kappa^{-2} t^{\frac 12} \right) + 2 \bb^{\frac 32} \kappa^{-1} t^{\frac 12} \left( 2 \bb^{\frac 32} t^{-\frac 12} \right)  + 2 \left( 3   {\mm^4}+ 3 \mm^3 \right) 
\notag\\
&\leq 7  {\mm^4}
\,.
\label{eq:refer:to:this:shit:later}
\end{align}
In the second inequality above we have used that $t \leq \bar \eps$ is sufficiently small with respect to $\kappa, \bb, \cc$, and $\mm$, while in the third inequality we have used \eqref{eq:b:m:ass}. This concludes the proof of \eqref{eq:dot:F:xi:t:approx}.
\end{proof}

\subsubsection{The shock curve iteration} 
 
In view of \eqref{eq:shock:ODE} and \eqref{eq:F:xi:t:approx:1}--\eqref{eq:dot:F:xi:t:approx} we   note that the inequalities \eqref{eq:sc:ass} are stable (since $\frac 12 < 1$ and $5<6$). Upon integrating in time, the condition $|\dot \sc (t) - \kappa| \leq  {\mm^4} t$ present in \eqref{eq:sc:ass}, automatically implies $\sc(t) \in \Sigma(t)$. 

Next, we define a sequence of curves $\sc^{(i)}$ for $i\geq 0$, as follows. For $i=0$, we let $\sc^{(0)}(t) = \kappa t$. This curve trivially satisfies the conditions in \eqref{eq:sc:ass}.
Next, given a curve $\sc^{(i)}$ defined on $[0,\bar \eps]$ which satisfies \eqref{eq:sc:ass}, we first compute via \eqref{eq:shock:to:solution:map} a tuple $(w,k,z,a)^{(i)}$ associated to $\sc^{(i)}$:
\begin{align}
(\sc^{(i)}, w_0, a_0) \xrightarrow{\mbox{Proposition~\ref{thm:curve:determines:all}}} (w^{(i)},z^{(i)},k^{(i)},a^{(i)})  \,.
\label{eq:shock:to:solution:map:i}
\end{align}
Then, according to \eqref{eq:F:def}, from $(\vl^{(i)}, \vr^{(i)},\zl^{(i)})$, which are one-sided restrictions on $\sc^{(i)}$, we may uniquely define a velocity field $\FFF_{\sc^{(i)}}(t)$, which may be in turn integrated to define
\begin{align}
\sc^{(i+1)}(t) = \int_0^t \FFF_{\sc^{(i)}}(s) ds 
\label{eq:shock:curve:iterate}
\end{align}
for all $t\in [0,\bar \eps]$. Since $\sc^{(i)}$ satisfies \eqref{eq:sc:ass}, by Lemma~\ref{lem:F:xi:t:approx}, we have that $\FFF_{\sc^{(i)}}$ satisfies the bounds in \eqref{eq:F:xi:t:approx:1}--\eqref{eq:dot:F:xi:t:approx}. Using \eqref{eq:shock:curve:iterate} and Lemma~\ref{lem:F:xi:t:approx}, we in turn deduce that $\sc^{(i+1)}$ satisfies \eqref{eq:sc:ass},  on the same time interval $\bar \eps$. Thus, under the above described iteration $\sc^{(i)} \mapsto \sc^{(i+1)}$, the set of inequalities \eqref{eq:sc:ass} is stable.

The sequence of curves $\{\sc^{(i)}\}_{i\geq 0}$ is uniformly bounded in $W^{2,\infty}(0,\bar \eps)$, in light of the bounds \eqref{eq:sc:ass}, and for $i\geq 0$ it satisfies \eqref{eq:shock:curve:iterate}. From the Arzela-Ascoli theorem, we may thus deduce that there exists at least one sub-sequential uniform limit $\sc$, of the family $\{\sc^{(i)}\}_{i\geq 0}$, which inherits the bounds \eqref{eq:sc:ass}. However, in order to show that this limit point $\sc$ solves \eqref{eq:shock:ODE}, we would need to show that $\FFF_{\sc^{(i)}} \to \FFF_{\sc}$ when $\sc^{(i)} \to \sc$. This continuity of $\FFF_{\sc}$ with respect to $\sc$ is addressed in the next section, where we in fact show that the sequence $\{\sc^{(i)}\}_{i\geq 0}$ is in fact Cauchy in $W^{1,\infty}(0,\bar \eps)$.

\subsubsection{Contraction mapping and convergence of the shock curve iteration} 
By \eqref{eq:F:def}, in order to compare $\FFF_{\sc^{(i+1)}}$ and  $\FFF_{\sc^{(i)}}$, it is obviously sufficient and necessary to compare the tuples $(\vl^{(i+1)}, \vr^{(i+1)},\zl^{(i+1)})$ and $(\vl^{(i+1)}, \vr^{(i+1)},\zl^{(i+1)})$. Note however that these tuples represent restrictions of the functions $(w^{(i+1)},z^{(i+1)})$ and $(w^{(i)},z^{(i)})$, which are themselves {\em defined on different domains}; thus in order to compare $(w^{(i+1)},z^{(i+1)})$ and $(w^{(i)},z^{(i)})$, we need to re-map them of a fixed domain, by shifting $\yy = \theta - \sc^{(i+1)}(t)$, respectively $\yy = \theta - \sc^{(i)}(t)$. 

As such, for every $i \geq 0$, and for $(\yy,t) \in (\TT \setminus \{0\}) \times [0,\bar \eps]$, we define
\begin{align}
(\wsf^{(i)},\zsf^{(i)},\ksf^{(i)}, \asf^{(i)})\left(\yy,t\right) =  (w^{(i)},z^{(i)},k^{(i)},a^{(i)}) \left(\yy + \sc^{(i)}(t) ,t\right)
\,,
\label{eq:i:unknowns}
\end{align}
where $s^{(0)} (t) = \kappa t$, and for $i\geq 1$ the curve $s^{(i)}$ is defined recursively via \eqref{eq:shock:curve:iterate} .
Since Proposition~\ref{thm:curve:determines:all} and the bound \eqref{for-shock-stuff} guarantee that $(w^{(i)},z^{(i)},k^{(i)},a^{(i)}) \in {\mathcal X}_{\bar \eps}$ are well-defined and differentiable on the spacetime domain $\TT \times [0,\bar \eps] \setminus \{(\sc^{(i)}(t),t)\}_{t\in[0,\bar \eps]}$, the new unknowns $(\wsf^{(i)},\zsf^{(i)},\ksf^{(i)}, \asf^{(i)})$ are all well-defined and differentiable on the $i$-independent domain $(\TT \setminus \{0\}) \times [0,\bar \eps]$ with bounds inherited from the space ${\mathcal X}_{\bar \eps}$ defined in \eqref{eq:w:z:k:a:boot:*}, allowing us to compare them to each other. Note that due to the shift \eqref{eq:i:unknowns}, we have   
\begin{align*}
\vl^{(i)}(t) = \lim_{\yy \to 0^-} \wsf^{(i)}(\yy,t) \,,
\   
\vr^{(i)}(t) = \lim_{\yy \to 0^+} \wsf^{(i)}(\yy,t) \,,
\  
\zl^{(i)}(t) = \lim_{\yy \to 0^-} \zsf^{(i)}(\yy,t)\,,
\  
\kl^{(i)}(t) = \lim_{\yy\to 0^-} \ksf^{(i)}(\yy,t)\,,
\end{align*}
the system of equations \eqref{eq:zminus:kminus} (which encode the jump conditions) are satisfied for every $i\geq  0, t\in [0,\bar \eps]$, and $\FFF_{\sc^{(i)}}$ may be expressed in terms of the above variables.
Moreover, by \eqref{eq:wzka} we have that for each $i \geq 0$ the unknowns in \eqref{eq:i:unknowns} solve the system of equations 
\begin{subequations} 
\label{eq:wzka:i}
\begin{align}
\left(\p_t  + (\lambda_3^{(i)} -\dot\sc^{(i)}) \p_{\yy}\right) \wsf^{(i)} & = - \tfrac{8}{3}  \asf^{(i)} \wsf^{(i)}  + \tfrac{1}{4}  \csf^{(i)} \left( \p_t + (\lambda_3^{(i)} -\dot\sc^{(i)} ) \p_{\yy} \right)\ksf^{(i)}   \,,  \label{eq:wzka:i:w} \\
\left(\p_t + (\lambda_1^{(i)} -\dot\sc^{(i)}) \p_{\yy}\right) \zsf^{(i)} & = - \tfrac{8}{3}  \asf^{(i)} \zsf^{(i)} -   \tfrac{1}{4}  \csf^{(i)} \left( \p_t + (\lambda_1^{(i)} -\dot\sc^{(i)}) \p_{\yy} \right)\ksf^{(i)}   \,,  \label{eq:wzka:i:z} \\
\left(\p_t + (\lambda_2^{(i)} -\dot\sc^{(i)}) \p_{\yy}\right) \ksf^{(i)} & = 0  \,, \label{eq:wzka:i:k} \\
\left(\p_t + (\lambda_2^{(i)} -\dot\sc^{(i)}) \p_{\yy}\right)  \asf^{(i)} & = - \tfrac43 (\asf^{(i)})^2 + \tfrac{1}{3} (\wsf^{(i)}+\zsf^{(i)})^2 - \tfrac{1}{6} (\wsf^{(i)}-\zsf^{(i)})^2  \,, \label{eq:wzka:i:a}
\end{align}
\end{subequations} 
in the interior of $(\TT \setminus \{0\}) \times [0,\bar \eps]$, where we have denoted $\csf^{(i)}= \tfrac{1}{2} (\wsf^{(i)}-\zsf^{(i)})$, and have use the usual notation for the three wave speeds at level $i$.

Since we have seen earlier that for all $i\geq 0$ the curves $\sc^{(i)}$ satisfy \eqref{eq:sc:ass}, by the proof of Lemma~\ref{lem:F:xi:t:approx} (see the first line of estimate \eqref{eq:refer:to:this:shit:later}) and the mean value theorem, for all $i \geq 0$ we have that 
\begin{align}
\abs{  \FFF_{\sc^{(i+1)}} - \FFF_{\sc^{(i)}} } 
&\leq \tfrac 23\kappa \bb^{-\frac 32} t^{-\frac 12}  \sabs{\zsf_-^{(i+1)} - \zsf_-^{(i)}} + \tfrac 43 \bb^{\frac 32} \kappa^{-1} t^{\frac 12} \sabs{\jump{\wsf^{(i+1)}} - \jump{\wsf^{(i)}}} \notag\\
&\qquad + \left( 1 + 3 \bb^3 \kappa^{-2} t \right)\sabs{\mean{\wsf^{(i+1)}} - \mean{\wsf^{(i)}}}
\,,
\label{eq:FFF:increment}
\end{align}
holds uniformly for $t\in (0,\bar \eps]$. Thus, it remains to estimate the right side of \eqref{eq:FFF:increment}.

For this purpose, we fix an $i\geq 0$, and denote
\begin{align}
\label{eq:drifter:increments}
 (\delta \wsf, \delta \zsf, \delta \ksf, \delta \asf,\delta\csf,\delta \dot \sc)
 =
 (  \wsf^{(i+1)},   \zsf^{(i+1)},   \ksf^{(i+1)},   \asf^{(i+1)}, \csf^{(i+1)},  \dot \sc^{(i+1)})
 -
 (  \wsf^{(i)},   \zsf^{(i)},   \ksf^{(i)},   \asf^{(i)}, \csf^{(i)},  \dot \sc^{(i)})\,.
\end{align}
We note that $ (\delta \wsf, \delta \zsf, \delta \ksf, \delta \asf,\delta\csf)(x,0) = 0$.
We subtract from \eqref{eq:wzka:i} at level $i+1$, the equations \eqref{eq:wzka:i} at level  $i$, in order to estimate the increments defined above, via the maximum principle, to obtain
\begin{itemize}
\item From \eqref{eq:wzka:i:k} we have that 
\begin{align*}
 \left(\p_t + (\lambda_2^{(i+1)} -\dot\sc^{(i+1)}) \p_{\yy}\right) \delta \ksf 
 & = - \p_{\yy} \ksf^{(i)} \left( \tfrac 23 \delta w + \tfrac 23 \delta z - \delta \dot \sc \right)
 \,.
\end{align*} 
Since Proposition~\ref{thm:curve:determines:all} guarantees that $k^{(i)} \in {\mathcal X}_{\bar \eps}$, the function $\ksf^{(i)}$ satisfies the bound \eqref{eq:k:dx:boot}, and so similarly to \eqref{kn-contract} we may obtain 
\begin{align}
 \sup_{[0,t]} \norm{\delta \ksf}_{L^\infty} \leq \mm^3  t^{\frac 32} \left(  \sup_{[0,t]} \norm{\delta \wsf}_{L^\infty} +  \sup_{[0,t]} \norm{\delta \zsf}_{L^\infty} +  \sup_{[0,t]} |\delta \dot \sc| \right)
 \label{eq:drifter:delta:k}
\end{align}
where the $L^\infty$ norms are taken over the domain $\TT \setminus \{0\}$.

\item Similarly, from \eqref{eq:wzka:i:a} we have 
\begin{align*}
 \left(\p_t + (\lambda_2^{(i+1)} -\dot\sc^{(i+1)}) \p_{\yy}\right) \delta \asf 
 & = - \p_{\yy} \asf^{(i)} \left( \tfrac 23 \delta w + \tfrac 23 \delta z - \delta \dot \sc \right) - \tfrac 43  (\asf^{(i+1)} + \asf^{(i)} )\delta \asf \notag\\
 &\quad + \tfrac 13 (\wsf^{(i+1)} + \wsf^{(i)} + \zsf^{(i+1)} + \zsf^{(i)} ) (\delta \wsf + \delta \zsf) \notag\\
 &\quad - \tfrac 16 (\wsf^{(i+1)} + \wsf^{(i)} - \zsf^{(i+1)} - \zsf^{(i)} ) (\delta \wsf - \delta \zsf)
 \,.
\end{align*}
Using that $(w^{(i)},z^{(i)},k^{(i)},a^{(i)}) \in {\mathcal X}_{\bar \eps}$, and since $|w^{(1)}(\theta,t) | = |\wb(\theta,t) | \leq \mm$, similarly to \eqref{an-contract} we obtain 
\begin{align*}
 \sup_{[0,t]} \norm{\delta \asf}_{L^\infty} 
 &\leq \mm^3  t  \left(  \sup_{[0,t]} \norm{\delta \wsf}_{L^\infty} +  \sup_{[0,t]} \norm{\delta \zsf}_{L^\infty} +  \sup_{[0,t]} |\delta \dot \sc| \right) + 3 \mm^3 t  \sup_{[0,t]} \norm{\delta \asf}_{L^\infty} \notag\\
 &\qquad +\left(\mm t + \mm^3 t^2 + \mm^3 t^{\frac 52} \right)  \left(  \sup_{[0,t]} \norm{\delta \wsf}_{L^\infty} +  \sup_{[0,t]} \norm{\delta \zsf}_{L^\infty}  \right)
\end{align*}
and thus, taking into account \eqref{eq:R:def},
\begin{align}
 \sup_{[0,t]} \norm{\delta \asf}_{L^\infty} \leq 4 \mm^3 t  \left(  \sup_{[0,t]} \norm{\delta \wsf}_{L^\infty} +  \sup_{[0,t]} \norm{\delta \zsf}_{L^\infty} +  \sup_{[0,t]} |\delta \dot \sc| \right)  \label{eq:drifter:delta:a}
\end{align}
since $  t \leq  \bar \eps \ll 1$.

\item Next, we  turn to \eqref{eq:wzka:i:w}, which  gives
\begin{align*}
 \left(\p_t  + (\lambda_3^{(i+1)} -\dot\sc^{(i+1)}) \p_{\yy}\right) \delta \wsf
 &=  - \p_{\yy} \wsf^{(i)} \left(  \delta \wsf + \tfrac 13 \delta z - \delta \dot \sc \right) 
 - \tfrac 83 a^{(i+1)} \delta \wsf - \tfrac 83 \wsf^{(i)} \delta \asf
 \\
&  + \tfrac{1}{4}  \csf^{(i+1)}   \left( \p_t + (\lambda_3^{(i+1)} -\dot\sc^{(i+1)} ) \p_{\yy} \right)\delta \ksf 
\\ 
&  
- \tfrac{1}{4}  \csf^{(i+1)}   \left(\delta \wsf  + \tfrac 13 \delta \zsf - \delta \dot\sc^{(i)}   \right)\p_{\yy} \ksf^{(i)}  
+ \tfrac{1}{4}  \delta \csf  \left(  \lambda_3^{(i)}  - \lambda_2^{(i)} \right)\p_{\yy} \ksf^{(i)}
\end{align*}
Recalling  that $c^{(i+1)}$ solves $(\p_t + \lambda_3^{(i+1)} \p_\theta ) c^{(i+1)} = - \frac  83 a^{(i+1)} c^{(i+1)} - \frac 23 c^{(i+1)} \p_\theta  z^{(i+1)}$, see e.g.~\eqref{cn3},
we obtain from the above that 
\begin{align*}
\left(\p_t  + (\lambda_3^{(i+1)} -\dot\sc^{(i+1)}) \p_{\yy}\right) \left( \delta \wsf - \tfrac 14 \csf^{(i+1)} \delta \ksf\right)
 &  =  - \p_{\yy} \wsf^{(i)} \left(  \delta \wsf + \tfrac 13 \delta \zsf - \delta \dot \sc \right) 
 - \tfrac 83 a^{(i+1)} \delta \wsf - \tfrac 83 \wsf^{(i)} \delta \asf
 \\
&    - \tfrac{1}{4} \delta \ksf   \left(\tfrac  83 \asf^{(i+1)} \csf^{(i+1)} + \tfrac 23 \csf^{(i+1)} \p_{\yy} \zsf^{(i+1)}   \right)  
\\ 
&   
- \tfrac{1}{4}  \csf^{(i+1)}   \left(\delta \wsf  + \tfrac 13 \delta \zsf - \delta \dot\sc^{(i)}   \right)\p_{\yy} \ksf^{(i)}  
+ \tfrac{1}{6}   \csf^{(i)} \p_{\yy} \ksf^{(i)}\delta \csf 
\,.
\end{align*}
Following \eqref{deltawnp1}, the above equation is composed with the flow of $\lambda_3^{(i+1)} - \dot\sc^{(i+1)}$, which of course is just $\eta^{(i+1)} - \sc^{(i+1)}$, and then integrated in time. Note that $\p_{\yy} \wsf^{(i)} \circ (\eta^{(i+1)} - \sc^{(i+1)}) = (\p_\theta  w^{(i)})\circ \eta^{(i+1)}$ and \eqref{eq:AC:DC:con} holds. Thus, using that $(w^{(i)},z^{(i)},k^{(i)},a^{(i)}) \in {\mathcal X}_{\bar \eps}$ and $(w^{(i+1)},z^{(i+1)},k^{(i+1)},a^{(i+1)}) \in {\mathcal X}_{\bar \eps}$ similarly to \eqref{wn-contract} we may deduce that 
\begin{align*}
 \sup_{[0,t]} \norm{\delta \wsf}_{L^\infty} 
 &\leq   \mm \left( 1 + 4 \mm^3 t \right)  \sup_{[0,t]} \norm{\delta \ksf}_{L^\infty}   
 + \left( \tfrac{19}{40} + 4 \mm^3 t \right)  \sup_{[0,t]} \norm{\delta \wsf}_{L^\infty} + \left(\tfrac 16 + 2 \mm^4  t^{\frac 32}\right) \sup_{[0,t]} \norm{\delta \zsf}_{L^\infty} \notag\\
 &\qquad + 3 \mm t \sup_{[0,t]} \norm{\delta \asf}_{L^\infty} + \left( \tfrac{19}{40} + \mm^4 t^{\frac 32}  \right) \sup_{[0,t]} |\delta \dot \sc| \,.
\end{align*}
Upon taking $\bar \eps$ to be sufficiently small with respect to $\mm$, taking into account \eqref{eq:R:def} we deduce
\begin{align}
 \sup_{[0,t]} \norm{\delta \wsf}_{L^\infty} 
\leq   \mm^3 \sup_{[0,t]} \norm{\delta \ksf}_{L^\infty}   
 +  \tfrac 12   \sup_{[0,t]} \norm{\delta \zsf}_{L^\infty} 
 + 2 \mm^3 t \sup_{[0,t]} \norm{\delta \asf}_{L^\infty} + (\tfrac{19}{21} + 8 \mm^3 t) \sup_{[0,t]} |\delta \dot \sc|  
 \label{eq:drifter:delta:w}
\end{align}

\item Lastly, from \eqref{eq:wzka:i:z} and \eqref{cn1} we similarly deduce
\begin{align*}
 \left(\p_t  + (\lambda_1^{(i+1)} -\dot\sc^{(i+1)}) \p_{\yy}\right) \left( \delta \zsf + \tfrac 14 \csf^{(i+1)} \delta \ksf \right)
 &=  - \p_{\yy} \zsf^{(i)} \left(  \tfrac 13 \delta \wsf +  \delta \zsf - \delta \dot \sc \right) 
 - \tfrac 83 a^{(i+1)} \delta \zsf - \tfrac 83 \zsf^{(i)} \delta \asf
 \\
&  
+ \tfrac{1}{4}   \delta \ksf  \left( \tfrac  83 \asf^{(i+1)} \csf^{(i+1)} - \tfrac 23 \csf^{(i+1)} \p_{\yy} \wsf^{(i+1)}  \right)  
\\ 
&  
+\tfrac{1}{4}  \csf^{(i+1)}   \left(\tfrac 13 \delta \wsf  + \delta \zsf - \delta \dot\sc^{(i)}   \right)\p_{\yy} \ksf^{(i)}  
+\tfrac{1}{6}   \csf^{(i)} \p_{\yy} \ksf^{(i)} \delta \csf
\end{align*}
and then similarly to \eqref{eq:drifter:delta:w} we have
\begin{align}
 \sup_{[0,t]} \norm{\delta \zsf}_{L^\infty} 
 &\leq  \mm^3  \sup_{[0,t]} \norm{\delta \ksf}_{L^\infty}   
 +  \mm^6 t^{\frac 32}   \sup_{[0,t]} \norm{\delta \wsf}_{L^\infty}  
 + 3 \mm^3 t^{\frac 52}  \sup_{[0,t]} \norm{\delta \asf}_{L^\infty} +  \mm^6 t^{\frac 32}  \sup_{[0,t]} |\delta \dot \sc| \,.
  \label{eq:drifter:delta:z}
\end{align}

\end{itemize}
Combining the estimates \eqref{eq:drifter:delta:k}-\eqref{eq:drifter:delta:z}, and defining 
\begin{align}
 N_i(t):=\sup_{[0,t]} \norm{\delta \wsf}_{L^\infty}+  t^{-\frac 34} \sup_{[0,t]} \norm{\delta \zsf}_{L^\infty}  +  t^{-1} \sup_{[0,t]} \norm{\delta \ksf}_{L^\infty} + t^{-\frac 12}  \sup_{[0,t]} \norm{\delta \asf}_{L^\infty}
 \,,
   \label{eq:drifter:Ni:def}
\end{align}
where we recall the notation in \eqref{eq:drifter:increments}, we arrive at
\begin{align*}
N_i(t) \leq 3 (1+ \mm^3) t^{\frac 14} N_i(t) +   (\tfrac{19}{21} + 6 \mm^3 t^{\frac 12} ) \sup_{[0,t]} |\delta \dot \sc|
\end{align*}
and thus upon taking $t \leq \bar \eps$ to be sufficiently small in terms of $\mm$, we deduce
\begin{align}
N_i(t) \leq \tfrac{20}{21}    \sup_{[0,t]} |\delta \dot \sc| = \tfrac{20}{21}    \sup_{[0,t]} \sabs{\dot \sc^{(i+1)} - \dot \sc^{(i)}}
\,.
\label{eq:FFF:increment:2}
\end{align}

Recalling the definitions~\eqref{eq:shock:curve:iterate} and \eqref{eq:drifter:Ni:def}, from the bounds \eqref{eq:FFF:increment} and \eqref{eq:FFF:increment:2} we deduce that 
\begin{align}
\sup_{[0,t]} \sabs{\dot \sc^{(i+2)} - \dot \sc^{(i+1)}} 
&= \sup_{[0,t]} \abs{  \FFF_{\sc^{(i+1)}} - \FFF_{\sc^{(i)}} } \notag\\
&\leq \tfrac 23\kappa \bb^{-\frac 32} t^{-\frac 12}  \sup_{[0,t]} \norm{\zsf^{(i+1)} - \zsf^{(i)}}_{L^\infty} \notag\\
&\qquad \qquad + \left( 1 +3 \bb^{\frac 32} \kappa^{-1} t^{\frac 12} + 3 \bb^3 \kappa^{-2} t \right)\sup_{[0,t]} \norm{ \wsf^{(i+1)} - \wsf^{(i)}} \notag\\
&\leq \tfrac 23\kappa \bb^{-\frac 32} t^{\frac 14} N_i(t) + \left( 1 + \tfrac 83 \bb^{\frac 32} \kappa^{-1} t^{\frac 12} + 3 \bb^3 \kappa^{-2} t \right) N_i(t)
\notag\\
&\leq    ( 1 + t^{\frac 15}  ) N_i(t)
\notag\\
&\leq \tfrac{20}{21} ( 1 + t^{\frac 15}  )   \sup_{[0,t]} \sabs{\dot \sc^{(i+1)} - \dot \sc^{(i)}}
\notag\\
&\leq \tfrac{41}{42} \sup_{[0,t]} \sabs{\dot \sc^{(i+1)} - \dot \sc^{(i)}}
\label{eq:FFF:increment:3}
\end{align}
upon taking $\bar \eps$, and hence $t$, sufficiently small with respect to $\kappa,\bb,\cc$, and $\mm$. Note that $\frac{41}{42} < 1$, and so we have a contraction. Since $s^{(0)} = \kappa t$, and all the sequence of iterates satisfy \eqref{eq:sc:ass}, we deduce that 
\begin{align}
\sup_{[0,t]} \sabs{\dot \sc^{(i+1)} - \dot \sc^{(i)}} \leq \left( \tfrac{41}{42} \right)^i  \sup_{[0,t]} \sabs{\dot \sc^{(1)} - \kappa } \leq \left( \tfrac{41}{42} \right)^i  \mm^4 t \,.
\label{eq:FFF:increment:4}
\end{align}

The bounds \eqref{eq:FFF:increment:3}--\eqref{eq:FFF:increment:4} have as consequence the fact that the sequence of shock curve iterates $\{\sc^{(i)}\}_{i\geq 0}$ defined in \eqref{eq:shock:curve:iterate} is Cauchy in $W^{1,\infty}(0,\bar \eps)$, and thus has a unique limit point
\begin{align}
\sc = \lim_{i \to \infty} \sc^{(i)}
 \qquad \mbox{in} \qquad W^{1,\infty}(0,\bar \eps) \,,
\end{align}
which inherits the bound \eqref{eq:sc:ass}. The bound \eqref{eq:FFF:increment:3} moreover shows that  $\FFF_{\sc^{(i)}} \to \FFF_{\sc}$ as $i \to \infty$ in $C^0(0,\bar \eps)$, and by \eqref{eq:shock:curve:iterate} we obtain that $\sc$ solves shock evolution equation \eqref{eq:shock:ODE}, as desired.

Lastly, in view of \eqref{eq:shock:to:solution:map}, associated to this limit point $\sc$, which satisfies the bound \eqref{eq:sc:ass}, Proposition~\ref{thm:curve:determines:all} determines a unique solution $(w,z,k,a) \in {\mathcal X}_{\bar \eps}$ of the azimuthal  form of the Euler equations \eqref{eq:w:z:k:a}--\eqref{eq:wave-speeds} on either side of the shock curve, which also satisfies the Rankine-Hugoniot jump conditions \eqref{pjump77}--\eqref{pjump7}, and the shock speed $\dot \sc$ is given by \eqref{sdot1}, as desired.

\subsection{Uniqueness of solutions}
\label{sec:uniqueness} 
The uniqueness of solutions holds in the following sense. Consider $w_0$ which satisfies \eqref{eq:u0:ass:quant}, and $a_0$ which satisfies \eqref{eq:a0:ass}. For $i \in \{1,2\}$, assume that $\sc^{(i)}$ is a $C^2$ smooth shock curve defined on $[0,T]$ for some $T>0$, which satisfies \eqref{eq:sc:ass} on $[0,T]$. Assume that $(w,z,k,a)^{(i)}$ are $C^1_{x,t}$ smooth solutions of the azimuthal  form of the Euler equations \eqref{eq:w:z:k:a}--\eqref{eq:wave-speeds} on the spacetime domain $\DD_{T}$, i.e., on either side of the shock curve $\sc$, with initial datum $(w_0,0,0,a_0)$. Moreover, assume that the restrictions of $(w,z,k)^{(i)}$  satisfy the Rankine-Hugoniot jump conditions \eqref{pjump77}--\eqref{pjump7}, and that the shock speed $\dot \sc$ is given  by \eqref{sdot1}. Lastly, assume that $(w,z,k,a)^{(i)} \in {\mathcal X}_{\bar \eps}$, as defined in \eqref{eq:w:z:k:a:boot}--\eqref{eq:w:z:k:a:boot:*}. Then, if $\bar \eps \leq T$ is sufficiently small (in terms of the constants $\kappa, \bb, \cc, \mm$), we have that $\sc^{(1)} \equiv \sc^{(2)}$ on ${0,\bar \eps}$, and $(w,z,k,a)^{(1)}\equiv (w,z,k,a)^{(2)}$ on $\DD_{\bar \eps}$. 

The proof of this statement is a direct consequence of the contraction mapping established in Section~\ref{sec:shock:evo},  and of the fact that $z^{(i)}(\cdot,t) \equiv 0$ on $\TT\setminus[\sc_1^{(i)}(t),\sc^{(i)}(t)]$, and $k^{(i)}(\cdot,t) \equiv 0$ on $\TT\setminus[\sc_2^{(i)}(t),\sc^{(i)}(t)]$. More precisely, for $i\in \{1,2\}$ use the definition \eqref{eq:i:unknowns} to remap the two sets of solutions to the same space-time domain, and then use \eqref{eq:drifter:increments} (with $i=1$) to denote their difference. As in \eqref{eq:drifter:Ni:def}, define
\begin{align*}
N(t): = \sup_{[0,t]} \norm{\delta \wsf}_{L^\infty}+  t^{-\frac 34} \sup_{[0,t]} \norm{\delta \zsf}_{L^\infty}  +  t^{-1} \sup_{[0,t]} \norm{\delta \ksf}_{L^\infty} + t^{-\frac 12}  \sup_{[0,t]} \norm{\delta \asf}_{L^\infty}
\,.
\end{align*}
Then, as in \eqref{eq:FFF:increment:2} and \eqref{eq:FFF:increment:3}, we may show that the bounds
\begin{align*}
N(t) &\leq  \tfrac{20}{21}    \sup_{[0,t]} |\delta \dot \sc|
\end{align*}
and 
\begin{align*}
 \sup_{[0,t]} |\delta \dot \sc|
\leq (1 + t^{\frac 15}) N(t)
\end{align*}
hold for all $t\in [0,\bar \eps]$, whenever $\bar \eps$ is chosen to be sufficiently small with respect to the aforementioned parameters. This shows that $N(t) = 0 = \delta \dot \sc(t)$ for all $t\in [0,\bar \eps]$. Since $\sc^{(i)}(0)=0$, it follows that $\delta \sc \equiv 0$, and thus also that $N \equiv 0$, thereby concluding the uniqueness proof.

\subsection{Proof of Theorem~\ref{thm:main:development}}
\label{sec:proof:dev:C1}
The proof of Theorem~\ref{thm:main:development} is a direct consequence of Proposition~\ref{thm:curve:determines:all}, of the contraction mapping established in Section~\ref{sec:shock:evo}, and of the uniqueness in Section~\ref{sec:uniqueness}, as described next.

The parameter $\bar \eps>0$ in item~\ref{item:5.5.i} is chosen to be possibly smaller than what is required in Proposition~\ref{thm:curve:determines:all}, as required by the estimates in Sections~\ref{sec:shock:evo} and~\ref{sec:uniqueness}. The {\em existence} of the regular shock curve $\sc$ and of the solution $(w,z,k,a) \in {\mathcal X}_{\bar \eps}$ to the azimuthal form of the Euler equations \eqref{eq:w:z:k:a}, follows from the contraction mapping in Section~\ref{sec:shock:evo}. Note that in view of \eqref{eq:shock:ODE}, the shock curve $\sc$ obeys the correct ODE, while the desired properties for $(w,z,k,a)$ follow from Proposition~\ref{thm:curve:determines:all} applied to this limiting shock curve. The {\em uniqueness} of the solution $(\sc, w,k,z,a)$ such that $\sc$ satisfies \eqref{eq:sc:ass}  and $(w,z,k,a) \in {\mathcal X}_{\bar \eps}$, is established in section~\ref{sec:uniqueness}. Taking into account Proposition~\ref{thm:curve:determines:all}, we have thus established items~\ref{item:5.5.i}, \ref{item:5.5.ii}, \ref{item:5.5.iii}, \ref{item:5.5.iv}, \ref{item:5.5.vii}, and along with the support properties for $k$ and $z$ claimed in items~\ref{item:5.5.v} and~\ref{item:5.5.vi}. 

In order to complete the proof of the theorem, it remains to establish the following: the precise bounds for $k$ near $\sc_2$ (as claimed in item~\ref{item:5.5.v}), the precise bounds for $z$ near $\sc_1$ (as claimed in item~\ref{item:5.5.vi}), the specific vorticity bounds (and its continuity across $\sc$) claimed in item~\ref{item:5.5.viii}, and the continuity of $a$, respectively the jump for $\p_\theta a$ across $\sc$, as claimed in item~\ref{item:5.5.ix}. These properties of the solution are established in  Subsections~\ref{sec:make:up:pig:1} and~\ref{sec:make:up:pig:2}, below.

\subsubsection{Improved bounds for $z$ and $k$ near $\sc_1$ respectively $\sc_2$}
\label{sec:make:up:pig:1}
The information $(w,z,k,a) \in {\mathcal X}_{\bar \eps}$ does not directly provide estimates for $z(\theta,t) $ and $k(\theta,t) $ which vanish as $\theta \to \sc_1(t)^+$, respectively $\theta \to \sc_2(t)^+$. Such bounds may however be easily obtained, as follows.

From \eqref{xland-k}, the definitions of the stopping time $\st$ and of the flow $\pt$, and the estimate \eqref{eq:kl:on:shock:L:infinity}, we obtain   
\begin{align}
\sabs{k(\theta,t) } = \sabs{k_-(\sc(\st(\theta,t)),\st(\theta,t))} \leq 40 \bb^{\frac 92} \kappa^{-3} \st(\theta,t)^{\frac 32}
\label{eq:metal:1}
\end{align}
for all $(\theta,t)  \in \mathcal{D} ^k_{\bar\eps}$. Similarly, from \eqref{kx-final}, \eqref{eq:dt:zl:kl:on:shock}, and \eqref{phixn-bound} (with $n\to \infty$) we deduce that 
\begin{align}
\sabs{\p_\theta  k(\theta,t) } 
&\leq \tfrac{4}{\kappa} \sabs{\tfrac{d}{dt} \kl (\st(\theta,t)) }  
\leq 200 \bb^{\frac 92} \kappa^{-4}  \st(\theta,t)^{\frac 12}
\label{eq:metal:2}
\end{align}
for all $(\theta,t)  \in \mathcal{D} ^k_{\bar\eps}$. Since $\st(\theta,t) \approx \frac{3}{\kappa} (\theta  - \sc_2(t))$, see e.g.~\eqref{eq:bohemian:1} below, the above two estimates give a precise order of vanishing for $k$ and $k_y$ as $y\to \sc_2(t)^+$.

Next, let us consider the behavior of $z$ near $\sc_2(t)$. For $(\theta,t)  \in \mathcal{D} ^k_{\bar\eps}$, from \eqref{xland-z}  we obtain
 \begin{align} 
 z(\theta,t)  = z( \sc(\stt(\theta,t)),\stt(\theta,t)) e^{-\frac 83 \int_{\stt(\theta,t)}^t a \circ \pst ds} +\tfrac{1}{6}  \int_{\stt(\theta,t)}^t   (c^2 k_\theta ) \circ \pst e^{-\frac 83 \int_{s}^t a \circ \pst ds'} ds
 \,.
 \label{eq:metal:2a}
 \end{align} 
Using  \eqref{eq:zl:on:shock:L:infinity},  \eqref{eq:w:boot},    \eqref{eq:a:boot}, and  \eqref{eq:metal:2}, we deduce that 
\begin{align} 
 \sabs{z( y,t)  } 
 \le 
5 \bb^{\frac 92} \kappa ^{-2} \stt(\theta,t)^{\frac{3}{2}}  + 40 \mm^2 \bb^{\frac 92} \kappa^{-4} \int_{\stt(\theta,t)}^t \st(\pst(\theta,s),s)^{\frac 12} ds \,.  \notag
\end{align} 
In order to estimate the integral term in the above estimate, we use \eqref{c1c2diffkappa} to bound  $ \tfrac{5}{2} \kappa ^{-1} (\theta - \sc_2(t))\le \st(\theta,t) \le \tfrac{7}{2} \kappa ^{-1} (\theta - \sc_2(t))$ for all $\sc_2(t) < \theta  < \sc(t)$,  for $\bar \eps$ sufficiently small. As such, it is natural to define $\gamma(s) = \pst(\theta,s)- \sc_2(s)$, and note that due to  \eqref{blob-a-tron9}, we have  
$\dot \gamma(s) = \lambda_1(\pst(\theta,s),s)-  \dot{\sc}_2(s)  \in [- \frac{\kappa}{2}, - \frac{\kappa}{4}]$. 
Hence, 
\begin{align} 
\int_{\stt(\theta,t)}^t \st(\pst(\theta,s),s)^{\frac 12} ds 
&\le 2 \kappa ^{-\frac 12} 
\int_{\stt(\theta,t)}^t \gamma(s)^{\frac 12} ds 
\leq
- 8 \kappa^{-\frac 32} \int_{\stt(\theta,t)}^t \dot \gamma(s) (\gamma(s))^{\frac 12} ds \notag \\
&= 6 \kappa^{-\frac 32} \left(\gamma(\stt(\theta,t))^{\frac 32} -  \gamma(t)^{\frac 32}  \right) 
\le 6 \kappa^{-\frac 32}  (\sc(\stt(\theta,t)) - \sc_2(\stt(\theta,t)))^ {\frac{3}{2}}  \le 2   \stt(\theta,t)^ {\frac{3}{2}}  \,.
 \notag
\end{align} 
Combining the above two inequalities we arrive at 
\begin{align} 
\sabs{ z(\theta,t) } \leq 12 \bb^{\frac 92} \kappa ^{-2} \stt(\theta,t)^{\frac{3}{2}}
\label{eq:metal:3} 
\end{align} 
for all $(y ,t) \in \mathcal{D} ^k_{\bar \eps}$. For $(\theta,t)  \in \mathcal{D} ^z_{\bar\eps}\setminus \overline{\mathcal{D} ^k_{\bar \eps}}$, the same bound as in \eqref{eq:metal:3} holds. Indeed, for $s \in [\stt(\theta,t),t]$ such that $\pst(\theta,s) \not \in  \mathcal{D} ^k_{\bar \eps}$, we have that $k_\theta (\pst(\theta,s),s) = 0$, so that the integrand in the second term in \eqref{eq:metal:2a} vanishes for such $s$. On the other hand, for $s \in [\stt(\theta,t),t]$ such that $\pst(\theta,s) \in  \mathcal{D} ^k_{\bar \eps}$ we again appeal to \eqref{eq:metal:2}, and to the fact that $\stt( \pst(\theta,s),s) = \stt(\theta,t)$. Estimate \eqref{eq:metal:3}  and the bound $ \tfrac{5}{2} \kappa ^{-1} (\theta  - \sc_1(t))\le \stt(\theta,t) \le \tfrac{7}{2} \kappa ^{-1} (\theta  - \sc_2(t))$, which holds for $\sc_1(t) < \theta  < \sc_2(t)$ and $\bar \eps$ sufficiently small, gives the rate of vanishing of $z(\theta,t) $ as $\theta \to \sc_1(t)^+$. Moreover, since $z(\sc_1(t),t)=0$ by using the definition of the derivative as the limit of finite differences, from \eqref{eq:metal:3} we immediately deduce also that 
\begin{align}
(\p_\theta  z) (\sc_1(t),t) = 0 \,.
\label{eq:metal:4}
\end{align}

\subsubsection{Bounds for the specific vorticity, the radial velocity, and its derivative}
\label{sec:make:up:pig:2}
The continuity of the radial velocity $a$ on $\TT \times [0,\eps]$ is a consequence of the construction: the continuous initial data $a_0$ (see~\eqref{eq:a0:ass}) is propagated smoothly along the characteristic flow of $\lambda_2$ (which is continuous, in fact Lipschitz continuous in space and time) in the domain $(\DD_{\bar \eps}^k)^\complement$, and in particular a limiting value for $a$ from the right side of the shock curve is obtained; these values of $a$ on the shock curve then serve as Cauchy data for the region $ \DD_{\bar \eps}^k$, using that  the flow of $\lambda_2$ is transversal to the shock curve. In detail, from \eqref{eq:xland-a:2}, the the Lipschitz regularity of $\pt^{(n)}(\theta ,\cdot)$ with respect to both $\theta $ and $t$ (see Lemma~\ref{lem:12flows} and its proof, the boundedness of $\p_t \pt^{(n)}$ follows in the same way as \eqref{phixn-bound}, since $\p_t \pt^{(n)}$ solves the same equation as $\p_\theta  \pt^{(n)}$ except with datum $0$ instead of $1$ at $(\theta,t) $), the continuity of $a_0$, and the bounds \eqref{eq:w:z:k:a:boot}, inductively imply that $a^{(n)}$ is continuous on $\TT \times [0,\eps]$, and thus so is its uniform limit $a$.  In particular, $\jump{a(\cdot,t)} = 0$.

Concerning the specific vorticity, we note that from the uniform bound \eqref{eq:w:z:k:a:boot:*} and the lower bound on $w_0$ in \eqref{eq:u0:ass:2a}, we have that the sequence of specific vorticities $\{ \varpi^{(n)} \}_{n\geq 1}$, where $\varpi^{(n)}= 4 (w^{(n)} + z^{(n)} - \p_\theta  a^{(n)}) (c^{(n)})^{-2} e^{k^{(n)}}$, is uniformly bounded in $L^\infty(\DD_{\bar \eps})$, by $300 \mm \kappa^{-2}$. Thus the weak-* limiting   vorticity $\varpi$ also lies in $L^\infty(\DD_{\bar \eps})$, and inherits this global bound. By repeating the argument in Section~\ref{sec:convergence:of:scheme}, since the right side of \eqref{eq:varpi:n} vanishes as $n\to \infty$ (when integrated against smooth test functions),  we obtain that $\varpi$   is a $L^\infty_{x,t}$ weak solution of \eqref{xland-svort} in $\DD_{\bar \eps}$. Since $(w,z) \in {\mathcal X}_{\bar \eps}$, we have that $\lambda_2$ is  Lipschitz, giving uniqueness of weak solutions to \eqref{xland-svort}, and thus $\varpi$ can be computed classically by integrating along the characteristics of $\lambda_2$ (see \eqref{eq:svort:formula} below). 

In order to obtain a sharper estimate for the limiting specific vorticity $\varpi$ we recall that from \eqref{eq:svort0:ass} that 
\begin{align}
10 \kappa^{-1} \leq \varpi_0(\theta ) \leq 28 \kappa^{-1}
\label{eq:svort:time:0}
\end{align}
for all $\theta \in \TT$. Integrating the evolution \eqref{xland-svort} along the characteristics $\pt(\theta,s)$, for $s\in [0,t]$, we obtain that 
\begin{align}
\varpi(\theta,t)  
&= \varpi_0(\pt(\theta,0)) e^{\frac 83 \int_0^t a(\pt(\theta,s),s) ds}  \notag\\
&+ \begin{cases}
\tfrac 43 \int_{\st(\theta,t)}^t e^{k(\pt(\theta,s),s)} (\p_\theta  k)(\pt(\theta,s),s) e^{\frac 83 \int_{s}^t a(\pt(\theta,s'),s') ds'} ds
\,, &  \mbox{for } (\theta,t) \in \DD_{\bar \eps}^k \\
0\,, &  \mbox{for } (\theta,t)  \in (\DD_{\bar \eps}^k)^\complement
\end{cases}
\,.
\label{eq:svort:formula}
\end{align}
Then, for all $(\theta,t) \in \DD_{\bar \eps}$, using the bounds \eqref{eq:a:boot}, \eqref{eq:k:boot}, and \eqref{eq:k:dx:boot}, we deduce that 
\begin{align}
\abs{\varpi(\theta,t) - \varpi_0(\pt(\theta,0))} \leq 3 R_7 t \abs{\varpi_0(\pt(\theta,0))} e^{3 R_7 t} + R_6 ( t^{\frac 32} -\st(\theta,t)^{\frac 32}) e^{3 R_7 t + R_5 t^{\frac 32}} \leq C t
\,.
\end{align}
Since $t \leq \bar \eps \ll 1$, it follows from the above estimate and \eqref{eq:svort:time:0} that
\begin{align}
9 \kappa^{-1} \leq \varpi(\theta,t)  \leq 30 \kappa^{-1}
\,,
\label{eq:varpi:boot}
\end{align}
for all $(\theta,t) \in \DD_{\bar \eps}$. 

The continuity of the specific vorticity across the shock curve $\sc$ follows from \eqref{eq:svort:formula}, the  continuity of $\varpi_0$ (see~\eqref{eq:svort0:ass}), the continuity of $a$  established earlier, the Lipschitz continuity of $\pt(\theta,\cdot)$ in both space and time (which holds in light of the argument in Lemma~\ref{lem:12flows} and the uniform convergence $\lambda_1^{(n)} \to \lambda_1$),   the transversality of the flow $\pt(\theta,\cdot)$ to the shock curve, the bounds \eqref{eq:w:z:k:a:boot}, and the fact that by definition $\st(\theta,t) \to t$ as $y\to\sc(t)^-$. 
 
It only remains to consider the behavior of $\p_\theta a$ near the shock curve, claimed in item~\ref{item:5.5.ix}. From \eqref{xland-svort:def} we have that 
$\p_\theta a = w + z - \tfrac 14 \varpi c^2 e^k$
and thus, using the continuity of $\varpi$ across the shock curve, for every $t\in (0,\bar \eps]$ we deduce that 
\begin{align*}
\jump{\p_\theta a}  
&= \jump{w}  + \jump{z}  - \tfrac 14 \varpi|_{(\sc(t),t)} \jump{w}  \mean{c}   \mean{e^k} + \tfrac 14 \varpi|_{(\sc(t),t)} \jump{z}  \mean{c}   \mean{e^k}   - \tfrac 14 \varpi|_{(\sc(t),t)} \mean{c^2} \jump{e^k} 
\notag\\
&=  \underbrace{\jump{w}\left( 1  - \tfrac 14 \varpi|_{(\sc(t),t)}  \mean{c}   \mean{e^k}\right)}_{=:J_{a,1} } + \underbrace{\jump{z}  + \tfrac 14 \varpi|_{(\sc(t),t)} \jump{z}  \mean{c}   \mean{e^k}   - \tfrac 14 \varpi|_{(\sc(t),t)} \mean{c^2} \jump{e^k}}_{=: J_{a,2} } \,.
\end{align*}
Using the fact that $(w,z,k,a) \in {\mathcal X}_{\bar \eps}$, the precise information on $\wb$ provided by Proposition~\ref{prop:Burgers}, that the specific vorticity satisfies \eqref{eq:varpi:boot}, and that 
the jumps in $z$ and $k$ (hence also the jump in $e^k$) satisfy~\eqref{eq:zl:and:kl:on:shock:L:infinity}, we obtain
\begin{align*}
\sabs{J_{a,2}(t)} \leq C t^{\frac 32} 
\end{align*}
and that 
\begin{align*}
& J_{a,2}(t) = \left(2 \bb^{\frac 32} t^{\frac 12} + \OO(t)\right)\left( 1 - \tfrac 18 \varpi|_{(\sc(t),t)}  \kappa + \OO(t) \right) \notag\\
\xrightarrow{\eqref{eq:varpi:boot}}
\qquad &
- 3\bb^{\frac 32} t^{\frac 12} \leq - \tfrac{11}{4} \bb^{\frac 32} t^{\frac 12} - C t \leq J_{a,2}(t) \leq - \tfrac 14 \bb^{\frac 32} t^{\frac 12} + C t \leq - \tfrac 15 \bb^{\frac 32} t^{\frac 12} 
\end{align*}
for all $t \in (0,\bar \eps]$.  By combining the above three displays we arrive at 
\begin{align*}
- 4\bb^{\frac 32} t^{\frac 12} \leq \jump{\p_\theta a}(t) \leq  - \tfrac 16 \bb^{\frac 32} t^{\frac 12}
\end{align*}
since $\bar \eps$, and hence $t$, is sufficiently small. The above estimate concludes the proof of Theorem~\ref{thm:main:development}.


\section{A precise description of the higher order singularities}
\label{sec:C2}
The goal of this section is to establish:
\begin{theorem}[Shocks, cusps, and weak discontinuities] \label{thm:C2} 
Let $\bar \eps>0$, $\sc \in C^2$, $\sc_1, \sc_2 \in C^1$, $(w,z,k,a) \in {\mathcal X}_{\bar \eps}$ be as in Theorem~\ref{thm:main:development}.
For $t\in(0,\bar\eps]$, we have the following upper bounds on higher order derivatives:
\begin{subequations}
\label{eq:thm-bounds-final} 
\begin{align} 
\sabs{w_{\theta\theta} (\theta,t) }& \les
\begin{cases}
t^{-\frac{5}{3}} ,  &\mbox{if }  \theta \leq \sc_2(t) \mbox{ or } \theta\geq \sc(t) + \frac{\kappa t}{3} \\
 t^{-\frac{5}{2}}  + \st(\theta,t)^{-\frac{1}{2}}  & \mbox{if }  \sc_2(t) < \theta < \sc(t) \\
 t^{-\frac{5}{2}}  & \mbox{if }   \sc(t) < \theta < \sc(t) + \frac{\kappa t}{3}
\end{cases} \,, \label{thm-wyy-bound-final} \\
\sabs{z_{\theta\theta} (\theta,t) }& \les
\begin{cases}
\st(\theta,t)^{-{\frac{1}{2}} } ,  &\mbox{if }    \sc_2(t) < \theta < \sc(t)\\
\stt(\theta,t)^{-{\frac{1}{2}} } & \mbox{if }     \sc_1(t) < \theta \leq \sc_2(t)
\end{cases} \,, \label{thm-zyy-bound-final} \\
\sabs{k_{\theta\theta} (\theta,t) } & \les 
\st(\theta,t) ^{-{\frac{1}{2}} }  \qquad \mbox{if }   \sc_2(t) < \theta < \sc(t)\,,
 \label{thm-kyy-bound1} \\
 \sabs{a_{\theta\theta} (\theta,t) }& \les
\begin{cases}
t^{-\frac{2}{3}} ,  &\mbox{if } y \leq \sc_2(t) \mbox{ or } \theta \geq \sc(t) + \frac{\kappa t}{3} \\
 t^{-1}  + t\st(\theta,t)^{-\frac{1}{2}}  & \mbox{if }  \sc_2(t) < \theta < \sc(t) \\
 t^{-1}  & \mbox{if }    \sc(t) < \theta < \sc(t) + \frac{\kappa t}{3}
\end{cases} \,, \label{thm-ayy-bound-final} \\
\sabs{\varpi_{\theta} (\theta,t) } &\les 1 + {\bf 1}_{(\theta,t)  \in \mathcal{D} ^k_{\bar \eps}}  \left( t - \st(\theta,t)\right) \st(\theta,t)^{-\frac 12} \,,\label{thm-varpiyy-bound-final} 
\end{align} 
\end{subequations}
where the implicit constants in $\les$  only depend on $\mm$,  cf.~\eqref{eq:second:order:boot}--\eqref{eq:N4:N7}, and \eqref{eq:dy:varpi:fin}.
In particular, for every $t>0$, the first and second derivatives of $(w,z,k,a)$ are bounded on both $\sc(t)^-$ and $\sc(t)^+$. 

Moreover,  $\sc_1(t)$ and $\sc_2(t)$ are $C^{1}$ smooth curves of weak characteristic discontinuities in the  following precise sense:
\begin{enumerate} 
\item The spacetime curve $\sc_2(t)$ is
a {\em weak contact discontinuity} with the property that second derivatives of $(w,z,k,a)$ blow up on $\sc_2^+(t)$; in particular,  for generic constants $c$ and $ C$, 
\begin{align}
c (\theta - \sc_2(t))^{-\frac{1}{2}}  \le  w_{\theta\theta} (\theta,t)  ,  - z_{\theta\theta} (\theta,t) , k_{\theta\theta} (\theta,t)  ,- t^{-1} a_{\theta\theta} (\theta,t)    \le C (\theta - \sc_2(t))^{-\frac{1}{2}}  \, 
\label{eq:s2:dyy:bounds}
\end{align} 
for $ \sc_2(t)<\theta$ and $\theta-\sc_2(t) \ll t$. The sum $w_{\theta\theta} +z_{\theta\theta} $ remains bounded on $\sc_2(t)$ and
\begin{align} 
\sabs{w_{\theta\theta} (\theta,t) +z_{\theta\theta} (\theta,t) } \les t^ {-\frac{1}{2}}  \,,
\label{eq:s2:dyy:magic}
\end{align} 
for $ \sc_2(t)<\theta<\sc_2(t) + \frac{\kappa t}{6}$. Lastly, the functions $(w_\theta,z_\theta,k_\theta,a_\theta)$ form   $C^ {\frac{1}{2}} $-cusps along $\sc_2(t)^+$.
\item The spacetime curve $\sc_1(t)$ is a weak discontinuity such that only $z_{\theta\theta} $ blows up on $\sc_1(t)^+$, 
\begin{align}
c (\theta - \sc_1(t))^{-\frac{1}{2}}  \le - z_{\theta\theta} (\theta,t)   \le C (\theta - \sc_1(t))^{-\frac{1}{2} }   \,,
\label{eq:s1:dyy:bounds}
\end{align} 
for $ \sc_1(t)<\theta$ with $\theta-\sc_1(t) \ll t$, 
while second derivatives of $(w,k,z)$ remain bounded in terms of inverse powers of $t$. The function $z_{\theta}$ forms a  $C^ {\frac{1}{2}} $-cusp along $\sc_1(t)^+$.
\end{enumerate} 
\end{theorem}

\begin{figure}[htp]
    \centering
    \includegraphics[width=.9 \textwidth]{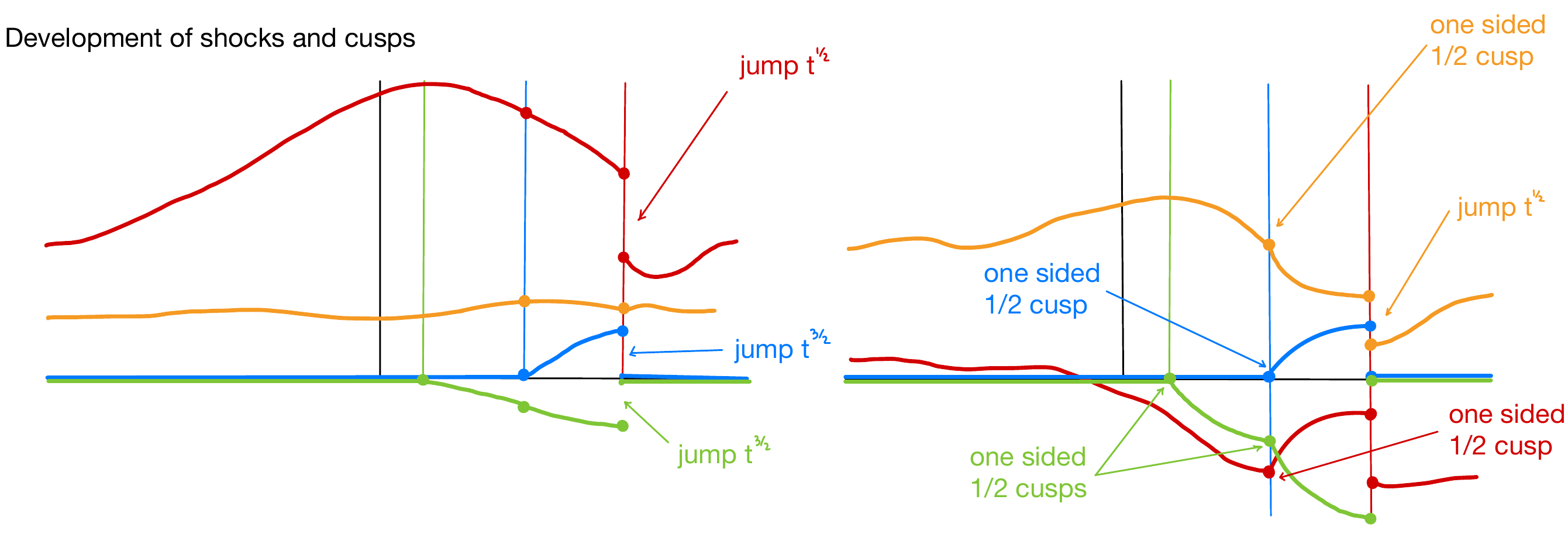}
    \caption{\footnotesize Schematic of the tuple $(w,z,k,a)$ at $t\in (0,\bar \eps]$. On the left, we have sketched $w$ in red, $z$ in green, $k$ in blue, and $a$ in orange. On the right, we have sketched the derivatives $w_\theta$ in red, $z_\theta$ in green, $k_\theta$ in blue, and $a_\theta$ in orange.}
    \label{fig:galaxy}
\end{figure}

The proof of Theorem~\ref{thm:C2} is the subject of the remainder of this section: in Section~\ref{sec:dyy:bootstraps} we give the bootstrap assumptions which yield~\eqref{eq:thm-bounds-final}, Sections~\ref{sec:dyy:flows}--\ref{sec:zyy} are dedicated to closing these bootstraps, while Sections~\ref{sec:dyy:lower:bound} and~\ref{sec:dy:Holder} are dedicated to the analysis of the weak singularities emerging on $\sc_1$ and $\sc_2$. The summary of the proof is given in Section~\ref{sec:proof:thm:C2}.

We note that the bounds for the second order derivatives of $(w,z,k,a)$ claimed in Theorem~\ref{thm:C2} greatly differ according to the location of the space-time point $(\theta,t) $ where they are evaluated: while {\em far away} from $\sc_1, \sc_2, \sc$ all information concerning $w$ and $a$ is propagated smoothly from the initial datum, for $(\theta,t) $ near the space-time curves $\sc_1,\sc_2, \sc$, obtaining upper bounds and {\em matching lower bounds} for second derivatives is a delicate matter, which requires a region-by-region analysis. Accordingly, we shall consider three separate cases: 
\begin{itemize}
\item $(\theta,t)  \in \mathcal{D} ^k_{\bar \eps}$, the region between $\sc_2$ and $\sc$. Here, for all $t>0$ the second derivatives of $(w,z,k,a)$ are bounded as $\theta \to \sc(t)^-$, but they all blow up as $\theta \to \sc_2(t)^+$, due to the presence of the entropy.
\item $(\theta,t)  \in \mathcal{D} ^z_{\bar \eps} \setminus  \overline{\mathcal{D} ^k_{\bar \eps} } $, the region between $\sc_1$ and $\sc_2$. In this region $k \equiv 0$, and this implies that the second derivatives of $w$ remain bounded as  $\theta \to \sc_2(t)^-$; nonetheless, the second derivative of $a$ still develops a singularity here, highlighting the two-dimensional nature of Euler in azimuthal symmetry model. On the other hand, approaching $\sc_1(t)$ from the right side, only the second derivative of $z$ develops a singularity. 
\item $(\theta,t)  \in \mathcal{D}_{\bar \eps} \setminus \overline{\mathcal{D} ^z_{\bar \eps}}$, the region which is either to the left of $\sc_1$ or the the right of $\sc$. In this region we have that $z \equiv 0$ and $k\equiv 0$, and thus the analysis reduces  to the study of $w$ and $a$ alone. We show that for all $t>0$, these quantities have bounded second derivatives, uniformly in this region, essentially because they are determined solely in terms of the initial data. 
\end{itemize}
\begin{remark}
\label{rem:FU:1}
Naturally, the further away $(\theta,t) $ are from $\sc_1(t)$ (to the left) or $\sc(t)$ (to the right), the further away we are from any singular behavior, and so the bounds for $\p_\theta ^2 w$ and $\p_\theta ^2 a$ become better. As such, for simplicity of the presentation we only give proofs of estimates for second derivatives at points $(\theta,t)  \in \mathcal{D}_{\bar \eps} \setminus \overline{\mathcal{D} ^z_{\bar \eps}}$ which are {\em close} to $\sc_1$ or $\sc$: either $\sc_1(t) - \bar \eps^{\frac 12} \leq \theta \leq \sc_1(t)$, or $\sc(t) < \theta < \sc(t) + \bar \eps^{\frac 12}$. In particular, the closeness considered is $t$-independent, and thus on the complement of this region it is not hard to establish bounds for $\p_\theta ^2 (\theta,t) $ and $\p_\theta ^2 a(\theta,t) $ which are uniform in time for $t\in [0,\bar \eps]$; these bounds only depend on $\bar \eps$, which is a fixed parameter. 
\end{remark}
\begin{remark} 
\label{rem:FU:3}
By the uniform convergence of our iteration scheme and \eqref{c1c2diffkappa}, we have that
\begin{subequations} 
\label{c1c2diffkappa2}
\begin{align} 
\pst(\theta,s)  =  \tfrac{1}{3}  \kappa s +   (\theta -  \tfrac{1}{3}  k t) + \OO(t^ {\frac{4}{3}} ) =  \tfrac{1}{3}  \kappa s +   (\theta -  \sc_1(t)) + \OO(t^ {\frac{4}{3}} )\,, \label{psi-line}
\ \ (\theta,t)  \in \mathcal{D}^{z}_{\bar \eps}\,,  \\
\pt (\theta,s)   =  \tfrac{2}{3}  \kappa s +   (\theta -  \tfrac{2}{3}  k t) + \OO(t^ {\frac{4}{3}} ) =  \tfrac{2}{3}  \kappa s +   (\theta -  \sc_2(t)) + \OO(t^ {\frac{4}{3}} )\,,  \label{phi-line}
\ \ (\theta,t)  \in \mathcal{D}^{z}_{\bar \eps}\,.
\end{align} 
\end{subequations} 
\end{remark} 

\begin{remark}[\bf Bounds on wave speeds $1$ and $2$]
\label{rem:FU:2}
Recall that $\pt$ and $\pst$ are the flows of the wave speeds $\lambda_2$ and $\lambda_1$, which are the identity at time $t$. Throughout this section we shall use the following fact: for all $t\in [0,\bar \eps]$, and all $y \in [\sc_1(t) - \bar \eps^{\frac 12}, \sc(t) + \bar \eps^{\frac 12}]$, we have
\begin{subequations}
\label{eq:FU:2}
\begin{align}
 \sabs{\p_s \pt(\theta,s) - \tfrac{2 \kappa}{3}} &=  \sabs{\lambda_2(\pt(\theta,s),s) - \tfrac{2 \kappa}{3}} \leq 4 \bb |\pt(\theta,s) - \sc(s)|^{\frac 13} + 4 \bb^{\frac 32} s^{\frac 12} \label{eq:ps:pt:bnd} \\
  \sabs{\p_s \pst(\theta,s) - \tfrac{\kappa}{3}} &= \sabs{\lambda_1(\pst(\theta,s),s) - \tfrac{\kappa}{3}} \leq 4 \bb |\pst(\theta,s) - \sc(s)|^{\frac 13} + 4 \bb^{\frac 32} s^{\frac 12} \label{eq:ps:pst:bnd}
\end{align}
\end{subequations}
for all $s\in [0,t]$, where $C = C(\kappa,\bb,\cc,\mm)>0$ is a constant. The proofs of \eqref{eq:ps:pt:bnd} and \eqref{eq:ps:pst:bnd} are identical, and rely on the fact that $z(\cdot,s) = \OO(s^{\frac 32})$, and that for $\bar y \in \{\pt(\theta,s),\pst(\theta,s)\}$ we have 
\begin{align*}
|\kappa - w(\bar \theta,s)| 
&\leq |\kappa - \wb(\bar \theta,s)| + R_1 s \notag\\
&\leq |\kappa - w_0(\etab^{-1}(\bar \theta,s))| + R_1 s \notag\\
&\leq 2 \bb |\etab^{-1}(\bar \theta,s)|^{\frac 13} + R_1 s \notag\\
&\leq 3 \bb^{\frac 32} s^{\frac 12} + 4 \bb |\bar \theta - \sc(s)|^{\frac 13} + R_1 s
\end{align*}
The aforementioned restriction on $\theta$ not being too far to the left of $\sc_1(t)$ or too far to the right of $\sc(t)$ was used in the third inequality above, because in light of \eqref{eq:u0:ass} this allows us to bound $|w_0(x) - \kappa| \leq 2 \bb |x|^{\frac 13}$, since $x = \etab^{-1}(\bar \theta,s)$ satisfies $|x| \leq \bar \eps^{\frac 14} \ll 1$. 
Note that a direct consequence of \eqref{eq:ps:pt:bnd}--\eqref{eq:ps:pst:bnd} and \eqref{eq:sc:ass}, we have that
\begin{align}
\label{eq:distance:curves}
\sabs{\sc(t) - \sc_2(t) - \tfrac{\kappa t}{3}} \leq  C t^{\frac 43}
\qquad\mbox{and} \qquad
\sabs{\sc_2(t) - \sc_1(t) - \tfrac{\kappa t}{3}} \leq  C t^{\frac 43}
\end{align}
holds uniformly for all $t\in [0,\bar \eps]$, for a suitable constant $C = C (\kappa,\bb,\cc,\mm)>0$.
\end{remark}

\subsection{Second derivative bootstraps}
\label{sec:dyy:bootstraps}
The core of the proof of Theorem~\ref{thm:C2} is to obtain suitable second derivative estimates for the unknowns $(w,z,k,a)$, and on the first derivative of $\varpi$, consistent with \eqref{eq:thm-bounds-final}. We achieve this by postulating a number of {\em bootstrap bounds} --- see \eqref{eq:second:order:boot}, \eqref{eq:second:order:boot2}, \eqref{eq:second:order:boot3} below --- and then show that these same bounds hold with a constant which is better by a factor of $2$. Note that the $\varpi_{\theta} $ and $a_{\theta\theta} $ estimates are  direct consequences of these bootstrap bounds, see Lemmas~\ref{lem:varpi:y} and \ref{lem:ayy}, they are not part of the bootstraps themselves.
Rigorously, the  bounds \eqref{eq:second:order:boot}, \eqref{eq:second:order:boot2}, and \eqref{eq:second:order:boot3}  need to be established iteratively for the sequence of approximations $(w^{(n)},z^{(n)},k^{(n)})$ which were considered in Section~\ref{sec:construction:iteration}; then, these estimates  hold for the unique limiting solution $(w,z,k)$ by passing $n\to \infty$.  When $n=1$ the bounds \eqref{eq:second:order:boot}, \eqref{eq:second:order:boot2}, and \eqref{eq:second:order:boot3} are trivially seen to hold in view of the definition given in \eqref{wzka1}. Then, assuming the bootstraps bounds hold for $(w^{(n)},z^{(n)},k^{(n)})$, the analysis in Sections~\ref{sec:dyy:flows}--\ref{sec:zyy} below,  shows that they hold for the next iterate $(w^{(n+1)},z^{(n+1)},k^{(n+1)})$ defined in Section~\ref{sec:construction:iteration}, and that they in fact hold with a better constant. In the proof in this section, instead of carrying around the super-indices $\cdot^{(n)}$ and $\cdot^{(n+1)}$ (as was done in Section~\ref{sec:construction:iteration}), we write the proof as if we had already passed $n \to \infty$, and work directly with the limiting solution. This abuse of notation is justified as described above in this paragraph.

\subsubsection{Bootstraps for the cone   $\mathcal{D} ^k_{\bar \eps}$}
For all $(\theta,t)  \in \DD^k_{\bar \eps}$, we suppose that
\begin{subequations}
\label{eq:second:order:boot}
\begin{align}
 \abs{ \p_{\theta}^2 w (\theta,t)  - \p_{\theta}^2 \wb  (\theta,t) } &\leq  M_1\bigl(\st(\theta,t)^{-\frac 12} + t^{-2}\bigr) \label{eq:w:dxx:boot} \\
\abs{ \p_{\theta}^2 z (\theta,t)   } &\leq M_2 \st(\theta,t)^{-\frac 12}  \label{eq:z:dxx:boot}  \\
\abs{\p_{\theta}^2 k (\theta,t)  } &\leq M_3 \st(\theta,t)^{-\frac 12}  
\label{eq:k:dxx:boot}  
\,,
\end{align}
\end{subequations}
where 
\begin{align} 
\label{eq:M1:M4}
M_1 = 10 \mm^4 \,, \qquad
M_2 = 10 \mm^3 \,,\qquad
M_3 = 2 \mm^2  \,.
\end{align}

\subsubsection{Bootstraps for the cone  $\mathcal{D} ^z_{\bar \eps} \setminus  \overline{\mathcal{D} ^k_{\bar \eps} } $}
\begin{subequations}
\label{eq:second:order:boot2}
For all $(\theta,t)  \in \mathcal{D} ^z_{\bar \eps} \setminus  \overline{\mathcal{D} ^k_{\bar \eps} } $, 
\begin{align}
 \abs{ \p_{\theta}^2 w (\theta,t)  - \p_{\theta}^2 \wb  (\theta,t) } 
 &\leq  N_1 t^{-\frac{2}{3}} \label{eq:w:dxx:boot2} \\
\abs{ \p_{\theta}^2 z (\theta,t)   } &\leq N_2 \stt(\theta,t)^{-\frac 12}  \label{eq:z:dxx:boot2} 
\,
\end{align}
\end{subequations}
where 
\begin{align}
\label{eq:N1:N3}
 N_1 = 5 \mm^4\,, \qquad 
 N_2 = 8 \mm^3\,.
\end{align}

\subsubsection{Bootstraps for  $\mathcal{D}_{\bar \eps} \setminus \mathcal{D} ^z_{\bar \eps}$}
For all $(\theta,t)  \in \mathcal{D}_{\bar \eps} \setminus \overline{\mathcal{D} ^z_{\bar \eps}}$. 
\begin{subequations} 
\label{eq:second:order:boot3}
\begin{align}
 \abs{ \p_{\theta}^2 w (\theta,t)  - \p_{\theta}^2 \wb  (\theta,t) }&
\le  \begin{cases}
N_4 t^{-\frac{2}{3}} ,  &\mbox{if } \theta \le \sc_1(t) \mbox{ or } \theta \geq \sc(t) + \frac{\kappa t}{3}\\
N_5 t^{-2}, & \mbox{if } \sc(t) < \theta <  \sc(t) + \frac{\kappa t}{3} \,, 
\end{cases}  \label{eq:w:dxx:boot3} 
\end{align}
\end{subequations}
where 
\begin{align}
\label{eq:N4:N7}
N_4 = 5 \mm^4\,, \qquad 
N_5 = 10 \mm^4\,.
\end{align}

\subsubsection{Bounds for $\varpi_{\theta} $ and $a_{\theta\theta} $}
We first show that the bootstrap for the second derivative of $k$ implies a good estimate for the derivative for the specific vorticity. 
\begin{lemma}\label{lem:varpi:y}  
Assume that $(w,z,k,a) \in {\mathcal X}_{\bar \eps}$ is such that \eqref{eq:k:dxx:boot} holds. 
Then, for all $(\theta,t)  \in \mathcal{D} _{\bar\eps}$,  we have
\begin{align} 
\label{eq:dy:varpi:fin}
 \sabs{\varpi_{\theta}(\theta,t) }& \leq   2\mm +  {\bf 1}_{(\theta,t)  \in \mathcal{D} ^k_{\bar \eps}} 4 \mm^2  \left( t - \st(\theta,t)\right) \st(\theta,t)^{-\frac 12} \,.  
\end{align} 
\end{lemma} 
\begin{proof}[Proof of Lemma~\ref{lem:varpi:y} ]
We differentiate the equation for the specific vorticity \eqref{xland-svort} with respect to $\theta$ and obtain
\begin{align*} 
(\p_t + \lambda_2 \p_\theta ) \varpi_{\theta}  + \left(\p_\theta  \lambda_2 - \tfrac 83 a\right) \varpi_{\theta}  = \tfrac 83 a_\theta \varpi + \tfrac 43 e^k \left(k_y^2 + k_{\theta\theta} \right) \,.
\end{align*} 
For any fixed $(\theta,t)  \in \mathcal{D} _{\bar \eps}$, we compose the above identity with $\pt(\theta,s)$ and arrive at 
\begin{align*} 
\tfrac{d}{ds}  \bigl( \varpi_{\theta}   \circ \pt   \bigr) 
+  \bigl(\p_\theta  \lambda_2\circ \pt - \tfrac{8}{3} a \circ \pt\bigr) (\varpi_{\theta}  \circ \pt)  = 
\bigl( \tfrac{8}{3} a_\theta \varpi
  +  \tfrac{4}{3}  e^k (k_\theta^2+k_{\theta\theta} )\bigr) \circ \pt  \,.
\end{align*}
Denoting the integrating factor associated to the above equation by
\begin{align}
I_{\varpi_{\theta} } = I_{\varpi_{\theta} }(\theta,t;s) 
&= - \int_s^t  \bigl(\p_\theta  \lambda_2(\pt(\theta,r),r) - \tfrac{8}{3} a (\pt(\theta,r),r) \bigr) dr \notag\\
&= - \tfrac 23 \int_s^t  \bigl(\p_\theta  w(\pt(\theta,r),r) + \p_\theta  z(\pt(\theta,r),r)  - 4 a (\pt(\theta,r),r) \bigr) dr \,,
\label{eq:int:factor:dy:varpi}
\end{align}
and using that $\pt(\theta,t) = \theta$, we then obtain
\begin{align} 
\varpi_{\theta} (\theta,t)  =   \varpi_0'(\pt(\theta,0)) e^{I_{\varpi_{\theta} }(\theta,t;0)}   + \int_0^t \bigl( \tfrac{8}{3} a _\theta \varpi
  +  \tfrac{4}{3}  e^k (k_\theta^2+k_{\theta\theta} )\bigr)(\pt(\theta,s),s) e^{I_{\varpi_{\theta} }(\theta,t;s)} ds \,.
  \label{eq:dy:varpi}
\end{align}

First, we estimate the integrating factor in \eqref{eq:int:factor:dy:varpi}, for a fixed $(\theta,t) $ in the region of interest, as described in Remark~\ref{rem:FU:1}. Using \eqref{eq:ps:pt:bnd} and \eqref{eq:sc:ass}, we have that the curve $\pt(\theta,s)$ is transversal to the shock curve $\sc$, in the sense that $\p_s \pt(\theta,s) \leq  \frac{2}{3} \kappa+ \OO(\bar \eps^{\frac 13}) \leq \frac{3}{4} \kappa < \dot \sc$. Hence, we may apply Lemma~\ref{lem:Steve:needs:this} with $\gamma(s) = \pt(\theta,s)$, separately on the intervals $[t',t] \mapsto [\st(\theta,t),t]$ and $[t',t] \mapsto [s,\st(\theta,t)]$, with the second case being of course empty if $\st(\theta,t) \leq s$. In this way, from estimate~\eqref{eq:Steve:needs:this:1}, \eqref{eq:w:dx:boot}, \eqref{eq:z:dx:boot}, and the triangle inequality, we deduce that 
\begin{align*}
\abs{I_{\varpi_{\theta} }(\theta,t;s) } \leq  40 \bb \kappa^{-\frac 23}  t^{\frac 13} + 2R_2 \bb^{-\frac 12} t^{\frac 12} + R_4 t^{\frac 32}\leq  50 \bb \kappa^{-\frac 23} t^{\frac 13} \,.
\end{align*}
As such,
\begin{align}
\abs{e^{I_{\varpi_{\theta} }(\theta,t;s)} -1} \leq 60 \bb \kappa^{-\frac 23} t^{\frac 13}
\label{eq:int:factor:dy:varpi:1}
\end{align}
uniformly for $s\in [0,t]$, since $t \leq \bar \eps \ll 1$.

Second, we appeal to  the bounds \eqref{eq:k:boot}, \eqref{eq:k:dx:boot},  \eqref{eq:a:boot}, and \eqref{eq:varpi:boot}, to deduce that
\begin{align}
\int_0^t \sabs{ \tfrac{8}{3} a_\theta  \varpi   +  \tfrac{4}{3}  e^k (k_\theta^2) }(\pt(\theta,s),s) ds 
\leq 12 \kappa^{-1} R_7 t +  R_6^2 t^2  \leq C t 
\label{eq:int:factor:dy:varpi:2}
\end{align}
for a suitable $C= C(\kappa,\bb,\cc,\mm)>0$.

Third, we use \eqref{eq:k:boot},  \eqref{eq:int:factor:dy:varpi:1}, the bound \eqref{eq:k:dxx:boot}, and the fact that $k\equiv 0$ on $[\mathcal{D} ^k_{\bar \eps}]^\complement$ to deduce that for all $(\theta,t)  \in \mathcal{D} ^k_{\bar \eps}$, we have
\begin{align}
\abs{ \int_{\st(\theta,t)}^t  \bigl(e^k k_{\theta\theta} \bigr)(\pt(\theta,s),s) e^{I_{\varpi_{\theta} }(\theta,t;s)} ds}
\leq 
 4 \mm^2 \left( t - \st(\theta,t)\right) \st(\theta,t)^{-\frac 12} 
 \label{eq:useful:later}
\end{align}
for a suitable $C= C(\kappa,\bb,\cc,\mm)>0$. Here we have implicitly used that $\st((\pt(\theta,s),s)) = \st(\theta,t)$.

Finally, by appealing to the $\varpi_0'$ estimate in \eqref{eq:svort0:ass}, we deduce from \eqref{eq:dy:varpi}, \eqref{eq:int:factor:dy:varpi:2}, and \eqref{eq:useful:later} that
\begin{align*}
\sabs{\varpi_{\theta} (\theta,t) }
&\leq \mm ( 1 + 60 \bb \kappa^{-\frac 23} t^{\frac 13}) + Ct +  {\bf 1}_{(\theta,t)  \in \mathcal{D} ^k_{\bar \eps}} 4 \mm^2  \left( t - \st(\theta,t)\right) \st(\theta,t)^{-\frac 12} \notag\\
&\leq 2\mm +  {\bf 1}_{(\theta,t)  \in \mathcal{D} ^k_{\bar \eps}} 4 \mm^2  \left( t - \st(\theta,t)\right) \st(\theta,t)^{-\frac 12} 
\,,
\end{align*}
which completes the proof of \eqref{eq:dy:varpi:fin}.
\end{proof}

The previously established estimate for the derivative of the specific vorticity,  \eqref{eq:dy:varpi:fin},  immediately implies a bound for the second derivative of the radial velocity $a$:
\begin{lemma}\label{lem:ayy}  
Assume that $(w,z,k,a) \in {\mathcal X}_{\bar \eps}$ is such that \eqref{eq:second:order:boot}, \eqref{eq:second:order:boot2}, and \eqref{eq:second:order:boot3} hold. Then, for all $(\theta,t)  \in \mathcal{D} _{\bar\eps}$ we have
\begin{align} 
\label{eq:ayy:improve:bootstrap}
 \sabs{a_{\theta\theta} (\theta,t) }& \leq
\begin{cases}
N_3  t^{-\frac{2}{3}} \,,  &\mbox{if }  \theta \leq \sc_2(t) \mbox{ or } \theta \geq \sc(t) + \frac{\kappa t}{3}\\
M_5 ( t^{-1}  + t\st(\theta,t)^{-\frac{1}{2}} ) \,,  & \mbox{if }  \sc_2(t) < \theta \leq \sc(t)\\
N_7 t^{-1} \,, & \mbox{if }   \sc(t) < \theta < \sc(t) + \frac{\kappa t}{3}
\end{cases}    
\end{align} 
where the constants $N_3$, $M_5$, and $N_7$ are defined as as 
\begin{align}
\label{eq:N3:M5:N7}
N_3 = \mm^3\,, \qquad M_5 = \mm^4\, , \qquad N_7 = \mm^3\,.
\end{align}
\end{lemma} 
\begin{proof}[Proof of Lemma \ref{lem:ayy}]
The proof directly follows from the bounds on the derivative of the specific vorticity contained in the bootstrap estimates \eqref{eq:dy:varpi:fin}, \eqref{eq:dy:varpi:fin}, and \eqref{eq:dy:varpi:fin}. 
We rewrite the definition \eqref{xland-svort:def} as $a_{\theta} = w+z - \tfrac{1}{4}   c^2 e^{-k} \varpi$, and upon differentiating we see that
\begin{align*} 
a_{\theta\theta} 
&= w_\theta +z_\theta - \tfrac{1}{4}   c (w_\theta -z_\theta) e^{-k } \varpi + \tfrac{1}{4}   c^2 e^{-k} k_\theta \varpi - \tfrac{1}{4}    c^2 e^{-k} \varpi_{\theta}  \notag\\
&= - \tfrac{1}{4}    c^2 e^{-k} \varpi_{\theta}  + w_\theta \left( 1 -\tfrac{1}{4}   c  e^{-k } \varpi \right) +z_\theta \left(1+ \tfrac{1}{4}   c  e^{-k } \varpi \right) + \tfrac{1}{4}   c^2 e^{-k} k_\theta \varpi  \,.
\end{align*} 
By using that $(w,z,k,a) \in {\mathcal X}_{\bar \eps}$ and the bound \eqref{eq:varpi:boot}, it follows that for all $(\theta,t) $ in the region of interest, we have
\begin{align}
\abs{a_{\theta\theta}  +  \tfrac{1}{4}    c^2 e^{-k} \varpi_{\theta}  - \p_\theta  \wb \left( 1 -\tfrac{1}{4}   c  e^{-k } \varpi \right)}
&\leq  C t^{-\frac 12}  
\label{eq:proto:1}
\end{align}
for a suitable $C= C(\kappa,\bb,\cc,\mm)>0$. For $\p_\theta  \wb$ estimates we refer to \eqref{thegoodstuff1}, $\varpi$ is bounded via \eqref{eq:varpi:boot}, while for bounds on $\p_\theta  \varpi$ we refer to \eqref{eq:dy:varpi:fin}, \eqref{eq:dy:varpi:fin}, and \eqref{eq:dy:varpi:fin}. We deduce
\begin{align}
\sabs{a_{\theta\theta} (\theta,t) } 
& \leq  \mm^3 + {\bf 1}_{(\theta,t)  \in \mathcal{D} ^k_{\bar \eps}}  \mm^4  \left( t - \st(\theta,t)\right) \st(\theta,t)^{-\frac 12} 
+ \mm^2 \left( (\bb t)^3 + |\theta - \sc(t)|^2\right)^{-\frac 13} 
\label{eq:proto:1a}
\end{align}
The  bound \eqref{eq:proto:1a} now directly implies \eqref{eq:ayy:improve:bootstrap}, as follows.

For $\theta \leq \sc_2(t)$ or $\theta\geq \sc(t) + \frac{\kappa t}{3}$, we have that $|\theta-\sc(t)| \geq \frac{\kappa t}{3} - C t^{\frac 43}$, and also $(\theta,t)  \not \in \DD_{\bar \eps}^k$. As such, the first bound stated in \eqref{eq:ayy:improve:bootstrap} follows from \eqref{eq:proto:1a} as soon as $ N_3  \geq 2 \mm^2 (\kappa/4)^{-\frac 23}$. This condition motivates the choice  of $N_3   =\mm^3$ in \eqref{eq:N3:M5:N7}. Similarly, the third bound in \eqref{eq:ayy:improve:bootstrap} follows from \eqref{eq:proto:1a} as soon as $ N_7 \geq 2 \mm^2 \bb^{-1}$; this condition holds since $N_7 = \mm^3$ as in \eqref{eq:N3:M5:N7}. Lastly, we consider the case that $\sc_2(t) <\theta < \sc(t)$, case in which \eqref{eq:proto:1a} implies
\begin{align}
\abs{\p_{\theta}^2  a(\theta,t) } 
\leq  \mm^3  +  \mm^4 (t - \st(\theta,t)) \st(\theta,t)^{-\frac 12} + \mm^2 (\bb t)^{-1}
\leq \mm^4 (t^{-1} + t \st(\theta,t)^{-\frac 12})
\,.
\label{eq:p:yy:a:2}
\end{align}
The bound \eqref{eq:p:yy:a:2} then clearly implies the second bound in \eqref{eq:ayy:improve:bootstrap} as soon as $ M_5 \geq \mm^4$; a condition which holds in view of the definition of $M_5$ in \eqref{eq:N3:M5:N7}. 
\end{proof}

\subsection{Second derivatives of the three wave speeds}
\label{sec:dyy:flows}

\subsubsection{Improved estimates for derivatives of $\eta -\etab$}
\begin{lemma}\label{lem:dxeta-dff} 
Given $(\theta,t)  \in \mathcal{D} _{\bar \eps}$,  define the label $x  \in \Upsilon(t)$ by $x= \eta ^{-1} (\theta,t) $.  Then
\begin{subequations} 
\begin{align} 
\abs{\p_x \eta(x,t) -\p_x \etab(x,t)} &\le 
\begin{cases}
50 \mm  t^ {\frac{4}{3}}  ,  &\mbox{if } \theta \not \in (\sc_2(t), \sc(t) + \frac{\kappa t}{3}) \\
10 \mm  t , & \mbox{if }  \theta \in (\sc_2(t), \sc(t) + \frac{\kappa t}{3})   \,, 
\end{cases}\,, \label{dxeta-diff} \\
\abs{\p^2_x \eta(x,t) -\p^2_x \etab(x,t)} &\le 
\begin{cases}
10  \mm  t^{\frac 13} ,  &\mbox{if } \theta \not \in (\sc_2(t), \sc(t) + \frac{\kappa t}{3}) \\
20 \mm  t^{-\frac{1}{2}},  & \mbox{if } \theta \in (\sc_2(t), \sc(t) + \frac{\kappa t}{3}) 
\end{cases}\,.
  \label{dxxeta-diff}
\end{align} 
\end{subequations} 
\end{lemma} 

\begin{proof}[Proof of Lemma \ref{lem:dxeta-dff}] 

We first record a few bounds for the derivatives of the Burgers flow map $\etab$. Using \eqref{eq:u0'':interest}--\eqref{eq:u0''':interest}, we have that   for  all $\sc_2(t) < \theta < \sc(t) + \frac{\kappa t}{3}$ and with $x = \eta^{-1}(\theta,t) $ 
\begin{align} 
\sabs{\p_x^2 \etab(x,t)} \le \sabs{ t w_0''(x)} \leq  \tfrac{1}{3} \bb^{- \frac 32} t^{-\frac 32}\,, \qquad
\sabs{\p_x^3 \etab(x,t)} \le \sabs{ t w_0'''(x)} \leq  {2\mm \bb^{-4}} t^{-3}\,. \label{first-fun-day}
\end{align} 
The above estimates hold since $|x| \geq \frac 45 (\bb t)^{\frac 32}$ 
For the case that $\theta \leq \sc_2(t)$ or $\theta\geq \sc(t) + \tfrac{\kappa t}{3}$,  similarly to \eqref{eq:finally:useful} we may show that 
\begin{align}
|\eta(x,s) - \sc(s)| \geq |\eta(x,t) - \sc(t)| + \tfrac 45 \bb^{\frac 32} t^{\frac 12} (t-s) \geq \tfrac{\kappa t}{4}  
\label{eq:finally:useful:2}
\end{align} 
and so $|x| = |\eta(x,0) - \sc(0)| \geq \frac{\kappa t}{4}$.
It follows from \eqref{eq:u0:ass:quant}  that for  labels $x$ such that $|x| \geq \frac{\kappa t}{5}$ 
\begin{align} 
\sabs{ w_0'(x)} \leq   \bb \kappa^{-\frac 23}  t^{-\frac 23}\,,  \ 
\sabs{\p_x^2 \etab(x,t)} \le \sabs{ t w_0''(x)} \leq  4 \bb \kappa^{-\frac 53}  t^{-\frac 23}\,,   \ 
\sabs{\p_x^3 \etab(x,t)} \le \sabs{ t w_0'''(x)} \leq {80 \mm \kappa^{-\frac 83}}  t^{-\frac 53}\,, \label{first-fun-dayL}
\end{align} 
upon taking $\bar \eps$ small enough.

In order to prove \eqref{dxeta-diff}, we appeal to the identities 
\begin{align}
\etab_x(x,t) =1 + \int_0^t \p_\theta \wb \circ \etab\etab_x ds
\,, \qquad
\eta_x(x,t) = 1+ \int_0^t (w_\theta +  \tfrac{1}{3}  z_\theta) \circ \eta \eta_x ds
\,.
\label{eq:red:panda:1}
\end{align} 
In anticipation of subtracting the two identities above, we first derive a useful identity for $\p_\theta  w \circ \eta  \eta_x$. To do so, we return to \eqref{prelim-dtwx}, which we rewrite as 
\begin{align} 
&\tfrac{d}{dt} \left( (w_\theta - \tfrac 14 c k_\theta) \circ \eta \ \eta_x \right) +  \left( ( \tfrac{8}{3}  a - \tfrac{1}{12} c k_\theta ) \circ \eta \right) \left( (w_\theta - \tfrac 14 c k_\theta) \circ \eta \ \eta_x \right)  
\notag\\
&\qquad =  \bigl( \tfrac{1}{48}   c k_\theta( c k_\theta + 4 z_\theta )  - \tfrac{8}{3} w  a_\theta   \bigr) \circ \eta  \ \eta_x  \,. 
\label{eq:red:panda:2} 
\end{align} 
At this stage it is convenient to introduce the $w$-good-unknown $q^w$ via
\begin{align} 
q^w(\theta,t)  & = w_\theta(\theta,t)  - \tfrac{1}{4}   c(\theta,t)  k_\theta (\theta,t)  \,, \label{qw-def}
\end{align} 
the integrating factor in \eqref{eq:red:panda:2} as 
\begin{align}
{\mathcal I}(x,s,t) =  \int_s^t \tfrac 83 a(\eta(x,s'),s') - \tfrac{1}{12} (ck_\theta)(\eta(x,s'),s') ds' \,,
\label{eq:Ia:def}
\end{align}
and the forcing term in \eqref{eq:red:panda:2} by
\begin{align}
 Q^w =  \tfrac{1}{48}   c k_\theta (c k_\theta + 4 z_\theta) - \tfrac{8}{3} w a_\theta  \,.
 \label{eq:Qw:def}
\end{align}
With this notation, integrating \eqref{eq:red:panda:2} and using that $k_0 = 0$, we arrive at
\begin{align}
q^w (\eta(x,t),t) \eta_x(x,t)
= w_0'(x) e^{-{\mathcal I}(x,0,t)} + \int_0^t Q^w(\eta(x,s),s) \eta_x(x,s) e^{-{\mathcal I}(x,s,t)} ds 
\label{eq:red:panda:3} 
\end{align}
Upon recalling the fact that  $\p_\theta  \wb \circ \etab \ {\etab}_x = w_0'$, from \eqref{eq:red:panda:1}, \eqref{eq:red:panda:3}, and the definition 
\begin{align}  \label{windy-day-garbage}
Q_1= \tfrac{1}{4} c k_\theta +  \tfrac{1}{3}  z_\theta \,, 
\end{align} 
we obtain
\begin{align}
\partial_t (\eta_x - \etab_x)
&= w_\theta \circ \eta \ \eta_x- \p_\theta  \wb \circ \etab \ {\etab}_x \notag\\
&= w_0'(x) \left(  e^{-{\mathcal I}(x,0,t)} - 1 \right) + Q_1 \circ \eta \ \eta_x 
+ \int_0^t Q^w(\eta(x,s),s) \eta_x(x,s) e^{-{\mathcal I}(x,s,t)} ds  
\label{nopork1}
\end{align}
which is the main identity relating the derivatives of $\eta$ and $\etab$. 

We will frequently use that the integrating factor ${\mathcal I}$ defined in \eqref{eq:Ia:def} satisfies
\begin{align}
\sabs{{\mathcal I}(\cdot,s,t)} &\leq \tfrac 83 R_7   (t-s) +\tfrac{1}{18} \mm R_6 (t^{\frac 32} - s^{\frac 32})
 \leq 12  \mm  (t-s) 
\,,
\label{eq:want:to:vomit:2aa} 
\end{align}
a bound which is a direct consequence of \eqref{eq:w:z:k:a:boot} and \eqref{eq:R:def}.

In order to prove \eqref{dxeta-diff}, we integrate \eqref{nopork1} on the interval $[0,t]$, use that $\eta_x(x,0) =1 = \etab_x(x,0)$, and  the fact that $(w,z,k,a) \in {\mathcal X}_{\bar \eps}$ (expressed through the bounds~\eqref{eq:w:z:k:a:boot}), and obtain  that 
\begin{align} 
\sabs{{\eta}_x(x,t) - {\etab}_x(x,t)} &\leq  |w_0'(x)|  \int_0^t \left|  e^{-{\mathcal I}(x,0,s)} - 1 \right| ds + C t^{\frac 32}
 \leq 8 \mm  t^2  |w_0'(x)| + C t^{2} 
\,. \label{windy-day1a}
\end{align} 
In the case that  $\theta = \eta(x,t) \not \in (\sc_2(t), \sc(t) + \frac{\kappa t}{3})$, since $|x| = |\eta^{-1}(\theta,t) | \geq \frac 45 (\bb t)^{\frac 32}$, from \eqref{eq:u0':interest} and \eqref{eq:R:def}, we obtain the second bound in \eqref{dxeta-diff}.  On the other hand, for $\theta = \eta(x,t) \in (\sc_2(t), \sc(t) + \frac{\kappa t}{3})$, from \eqref{eq:finally:useful:2} we have $|x|  \geq \frac{\kappa t}{4}$ and so from 
\eqref{first-fun-dayL}, \eqref{windy-day1a},  and the working assumption~\eqref{eq:b:m:ass}, we obtain that the first bound in  \eqref{dxeta-diff} holds.

We next estimate $\eta_{xx}-\etab_{xx}$.  Notice that by differentiating the identity \eqref{nopork1}, factors of $\eta_{xx}$ appear in both the  integral term, which at first leads to non-optimal bounds. Instead, we twice differentiate the equations $\p_s \eta = \lambda_3\circ \eta $ and $\p_s \etab = \wb \circ \etab$, to find that
\begin{align} 
\p_s( \eta_{xx} -\etab_{xx})
&= w_{\theta\theta}  \circ \eta\, \eta_x^2 - \wb_{\theta\theta} \circ \etab\, \etab_x^2 + w_\theta \circ \eta \, \eta_{xx} - \wb_{\theta} \circ \etab\, \etab_{xx} \notag \\
& = \underbrace{\bigl(\wb_{\theta\theta} \circ \etab\, \circ ( \etab ^{-1}\circ \eta) - \wb_{\theta\theta} \circ \etab\bigr) \eta_x^2}_{ \mathcal{K} _1}
+ \underbrace{(w_{\theta\theta} -\wb_{\theta\theta}) \circ \eta \ \eta_x^2}_{ \mathcal{K} _2} \notag \\
&\quad +  \underbrace{\wb_{\theta\theta} \circ \etab \ (\eta_x^2 - \etab_x^2)}_{ \mathcal{K} _3} 
+ \underbrace{(w_\theta \circ \eta - \wb_{\theta} \circ \etab) \etab_{xx}}_{ \mathcal{K} _4} 
+ w_\theta \circ \eta \ (\eta_{xx} - \etab_{xx}) \,.
\label{first-good-day2}
\end{align} 
We shall first  provide bounds for the terms $ \mathcal{K} _1$, $\mathcal{K}_2$, $\mathcal{K}_3$, and $\mathcal{K}_4$ on the right side of \eqref{first-good-day2} in
the regions $y$ far from $\sc(t)$ and $y$ close to $\sc(t)$, and then
apply the Gr\"onwall inequality to estimate $\eta_{xx} - \etab_{xx}$ in these two regions. 
To sharpen the bounds in the region close to $\sc(t)$, we then return to \eqref{nopork1} and differentiate it in $x$.

\vspace{.05 in} 
\noindent
{\bf The case $\theta\leq \sc_2(t)$ or $\theta\geq \sc(t) + \frac{\kappa t}{3}$.} 
We recall that $x = \eta ^{-1} (\theta,t) $ and define  the label  $\bar x = \etab ^{-1} (\theta,t) $. As earlier, from \eqref{eq:finally:useful:2} and \eqref{eq:AC:DC:ride:on} we have $|x| , |\bar x| \geq \frac{\kappa t}{5}$.
Using the mean value theorem, and estimates \eqref{eq:u0':interest},  \eqref{eq:pxx:wb:def},   \eqref{eq:tropic:1},  \eqref{eq:AC:DC:ride:on}, and  \eqref{first-fun-dayL}, we obtain that
\begin{align*} 
\sabs{\eta_ x^{-2}\mathcal{K} _1(x,s) } \le  2000 R_1 \mm \kappa^{-\frac{8}{3}}s^{-\frac{2}{3}}    +  C s^ {-\frac{1}{3}} \,,
\end{align*} 
so that using \eqref{eq:cb:0}, \eqref{eq:R:def}, and \eqref{eq:b:m:ass}
\begin{align} 
\sabs{\mathcal{K} _1(x,s)} \le {\mm^4} s^{-\frac{2}{3}} 
 \,.
\label{glassy-day0L}
\end{align} 
Then, using \eqref{eq:cb:0} and \eqref{eq:w:dxx:boot2} and \eqref{eq:w:dxx:boot3}, we have that
\begin{align} 
\sabs{\mathcal{K} _2(x,s)} \le 4 {(N_1+N_4)}  s^{-\frac{2}{3}}  \,. \label{glassy-day000L}
\end{align} 
In the above estimate we have implicitly used the fact that $\eta(x,s) \not \in \DD_{\bar \eps}^k$, which is a consequence of the assumption on $\theta$ being sufficiently far from $\sc(t)$ and of the bound \eqref{eq:finally:useful:2}.
Next, by \eqref{windy-day3}, the $s$-independent lower bound on $x$ provided by \eqref{eq:finally:useful:2}, and the $w_0$ estimates \eqref{eq:u0:ass:quant} and \eqref{eq:u0':interest}, we have
\begin{subequations}
\label{onemoretime}
\begin{align} 
 \abs{\p_{\theta}  \wb( \etab(x,s),s)} &\leq \tfrac 53 |w_0'(x)|  \le 2\bb (\kappa t)^{-\frac{2}{3}}     \,,
  \\
 \abs{\p^2_{\theta}  \wb(\etab(x,s),s)}&\leq (\tfrac 53)^3 |w_0''(x)| \le 16 \bb (\kappa t)^{-\frac{5}{3}}  \,,
\end{align} 
\end{subequations}
for all $s\in [0,t]$, 
and so
by \eqref{eq:cb:0} and  \eqref{dxeta-diff}
\begin{align} 
\sabs{\mathcal{K}_3(x,s)} \le C s^{\frac 43} t^{-\frac{5}{3}} \leq C s^{-\frac 13}
 \,, \label{the-star-chickenL}
\end{align} 
for a suitable $C = C(\kappa,\bb,\mm)>0$. 
Lastly, in order to bound $\mathcal{K}_4$, we write
\begin{align*} 
(w_\theta \circ \eta - \wb_{\theta} \circ \etab) = (w_\theta \circ \eta - \wb_{\theta} \circ \etab) \eta_x \eta_x ^{-1} \,,
\end{align*} 
and
\begin{align*} 
(w_\theta \circ \eta - \wb_{\theta} \circ \etab) \eta_x = \bigl( w_\theta \circ \eta \ \eta_x- \p_\theta  \wb \circ \etab \ {\etab}_x \bigr) - \p_\theta  \wb \circ \etab \bigl( \eta_x- \etab_x\bigr) \,.
\end{align*} 
Using the second equality in \eqref{nopork1}, similarly to \eqref{windy-day1a} but with $t$ replaced by $s$, we have that
$$\sabs{w_\theta \circ \eta \, \eta_x - \wb_{\theta} \circ \etab \, \etab_x} \le 4 R_7 s |w_0'(x)| + C s^{\frac 12} 
\leq  20 \bb \kappa^{-\frac 23} \mm s t^{-\frac 23} \leq C s^ {\frac{1}{3}} 
\,,
$$
where in the second inequality we have also appealed to \eqref{first-fun-dayL}.
Hence, by combining the above three displays with \eqref{eq:cb:0},  \eqref{dxeta-diff}, \eqref{first-fun-dayL}, and \eqref{onemoretime},  we have that
\begin{align} 
\sabs{\mathcal{K} _4(x,s) } \le  C t^{-\frac 23} \left(s^ {\frac{1}{3}}   + t^{-\frac 23} s^{\frac 43} \right) \leq C s^{-\frac 13}
\,,  \label{glassy-day1L}
\end{align} 
for a suitable $C = C(\kappa,\bb,\mm)>0$. 

Finally, using the bounds \eqref{glassy-day0L}--\eqref{glassy-day1L}, and the estimates \eqref{eq:AC:DC:con} and \eqref{eq:w:boot} we apply Gr\"onwall to 
\eqref{first-good-day2} and find that
\begin{align} 
\sabs{\eta_{xx}(x,t) -\etab_{xx}(x,t)} \le  e^{\frac{1}{2} + 2 R_2 \bb^{-\frac 12} t^{\frac 12}} 16 (  \mm^4 + {N_1 + N_4}) t ^{{\frac{1}{3}} } \leq 30 (  \mm^4 + {N_1 + N_4}) t^{\frac 13}   \,, \label{temp-junk2}
\end{align} 
in the case that $y \not \in (\sc_2(t),\sc(t) + \frac{\kappa t}{3})$.

\vspace{.05 in} 
\noindent
{\bf The case $\sc_2(t) < \theta < \sc(t) + \frac{\kappa t}{3}$.} 
We shall first use \eqref{first-good-day2} to provide a (non optimal) bound for the difference $ \eta_{xx}-\etab_{xx}$. Once we have such a bound, we will then return to the differentiated
form of \eqref{nopork1} to obtain the optimal bound.

Recall the definitions of the labels  $\bar x = \etab ^{-1} (\eta(x,t),t)$ and  $x = \eta ^{-1} (\theta,t) $.
At this stage it is convenient to introduce $s = \nu^\sharp(x,t) \in [0,t)$, the {\em largest time} at which {\em either} $\eta(x,s)=\sc_2(s) = \sc(s) - \frac{\kappa s}{3} + \OO(s^{\frac 43})$ {\em or} $\eta(x,s) = \sc(s) + \frac{\kappa s}{3}$. This time $\nu(x,t)$ exists in view of the intermediate function theorem since, $|\eta(x,0) - \sc(0)| =|x| \geq \frac 45 (\bb t)^{\frac 32} > 0$, and is unique since as in \eqref{eq:finally:useful} and in Lemma~\eqref{lem:12flows}, we have that the flow $\eta$ is transversal to both $\sc_2$ and to $\sc$. In fact, we recall from  \eqref{eq:finally:useful} that
\begin{align}
|\eta(x,s) - \sc(s)| \geq |y - \sc(t)| + \tfrac 45 \bb^{\frac 32} t^{\frac 12} (t-s)  \label{eq:finally:useful:3}
\end{align} 
and therefore, by also taking into account \eqref{eq:distance:curves}, we have that 
\begin{align}
\nu^\sharp(x,t) \geq \bb^{\frac 32} \kappa^{-1} t^{\frac 32}
\label{eta-stoptime:1}
\end{align}
uniformly for all $x = \eta^{-1}(\theta,t) $, and $\theta\in(\sc_2(t) , \sc(t) + \frac{\kappa t}{3})$.

Next, we return to bounding the terms on the right side of \eqref{first-good-day2}. 
Then by \eqref{eq:px:eta:b}, \eqref{eq:pxx:wb:def}, the mean value theorem,  \eqref{eq:AC:DC:ride:on}, and using   \eqref{eq:u0':interest},  \eqref{eq:u0'':interest},  \eqref{first-fun-day} we obtain
\begin{align}
\sabs{ \mathcal{K} _1(x,s) }
&\leq 4 \sabs{\etab^{-1}(\eta(x,s),s) - x} \abs{\frac{(1 + s w_0'(\tilde x)) w_0'''(\tilde x) - 3 s (w_0''(\tilde x))^2}{(1 + s w_0'(\tilde x))^4}} \notag\\
&\leq 4 R_1 s^2 (\tfrac 53)^4 \left(4 \mm \bb^{-4} t^{-4} + 16 \bb^{-3} s t^{-5} \right) \notag\\
&\leq \mm^4 s^2 t^{-4} \,.
\label{glassy-day0}
\end{align}
Here we have use that $\tilde x$ lies in between $x$ and $\etab^{-1}(\eta(x,s),s)$, and thus satisfies $\abs{\tilde x} \geq \frac 45 (\bb t)^{\frac 32}$.
 Next, by \eqref{eq:cb:0}, \eqref{eq:w:dxx:boot}, \eqref{eq:w:dxx:boot2}, and \eqref{eq:w:dxx:boot3},
\begin{align} 
\sabs{\mathcal{K} _2(x,s)} 
\le 4 ( M_1 + {N_5}) \left(s^{-2} + {\bf 1}_{\sc_2(s) <\eta(x,s) < \sc(s)} \st(\eta(x,s),s)^{-\frac{1}{2}} \right)
 \,.  \label{glassy-day000}
 \end{align}
Next, using \eqref{eq:pxx:wb:def} and the fact that $|x| \geq \frac 45 (\bb t)^{\frac 32}$, combined with the estimates \eqref{eq:u0':interest} and \eqref{eq:u0'':interest} we obtain that $|\wb_{\theta\theta} (\etab(x,s),s)| \leq 3 \bb^{-\frac 32} t^{-\frac 52}$. Hence, by also appealing to  \eqref{eq:cb:0} and  \eqref{dxeta-diff}, we deduce   
\begin{align} 
\sabs{\mathcal{K}_3(x,s) } \le C s  t^{-\frac 52}
 \,. \label{the-star-chicken}
\end{align} 
Finally, by \eqref{thechicken},  \eqref{eq:w:dx:boot}, \eqref{eq:R:def}, and  \eqref{first-fun-day}, 
\begin{align} 
\sabs{ \mathcal{K} _4(x,s) } 
&\le \left( 4 R_1 \bb^{-\frac 32} s^{-\frac 12} + R_2 (\bb s)^{-\frac 12}\right) s |w_0''(x)|  \notag\\
&\le \left( 4 R_1 \bb^{-\frac 32} s^{-\frac 12} + R_2 (\bb s)^{-\frac 12}\right) s \tfrac 13 \bb^{-\frac 32}t^{-\frac 52} \leq \mm^4 s^{\frac 12} t^{-\frac 52} 
\,.  \label{glassy-day1}
\end{align} 
Summing up the estimates \eqref{glassy-day0}--\eqref{glassy-day1}, we obtain
\begin{align}
&\sabs{ \mathcal{K} _1(x,s) } + \sabs{ \mathcal{K} _2(x,s) } + \sabs{ \mathcal{K} _3(x,s) } + \sabs{ \mathcal{K} _4(x,s) } \notag\\
&\qquad \leq   \left( 4 (M_1 + {N_5}) +2 \mm^4 \right)s^{-2} +  4 M_1 {\bf 1}_{\sc_2(s) <\eta(x,s)\leq \sc(s)} \st(\eta(x,s),s)^{-\frac{1}{2}} 
\,.
\label{eq:EASTER:1:15:AM}
\end{align}

Let $\tilde \nu (t) = \bb^{\frac 32} \kappa^{-1} t^{\frac 32}$ be the lower bound in \eqref{eta-stoptime:1}. 
With \eqref{eq:EASTER:1:15:AM} in hand we apply  the Gr\"onwall inequality to \eqref{first-good-day2} on the time interval $[\tilde \nu(t),t]$, which in view of \eqref{eta-stoptime:1} is slightly larger than $[\nu^\sharp(x,t) , t]$. The point here is that due to \eqref{eta-stoptime:1} we know that either $\eta(x, \tilde \nu) < \sc_2(\tilde \nu)$, or $\eta(x, \tilde \nu) > \sc(\tilde \nu) +  \frac{\kappa \tilde \nu}{3}$, and thus \eqref{temp-junk2} holds at the time $\tilde \nu$.
We thus deduce that 
\begin{align}
&\sabs{\eta_{xx}(x,t) -\etab_{xx}(x,t)} \notag\\
&\leq  30 (  \mm^4 + {N_1 + N_4}) \tilde \nu^{\frac 13} + (4 (M_1 + {N_5}) +2 \mm^4 ) \int_{\tilde \nu}^{t} s^{-2} ds + 4 M_1 \int_{\tilde \nu}^t {\bf 1}_{\sc_2(s) <\eta(x,s)\leq \sc(s)}  \st(\eta(x,s),s)^{-\frac 12} ds \notag\\
&\leq C t^{\frac 12} + (4 (M_1 + {N_5}) +2 \mm^4 ) \kappa \bb^{-\frac 32} t^{-\frac 32} + 4 M_1 \int_{\tilde \nu}^t {\bf 1}_{\sc_2(s) <\eta(x,s)\leq \sc(s)} \st(\eta(x,s),s)^{-\frac 12} ds \,.
\label{eq:EASTER:1:20:AM}
\end{align}
Note that if $\theta > \sc(t)$, then $\{\eta(x,s)\}_{s\in[0,t]}$ does not intersect $\DD_{\bar \eps}^k$, and so the integral term in the above is vacuous. We thus are left to consider the case $\theta \in ( \sc_2(t)  , \sc(t))$.

In order to bound the integral term on the right side of \eqref{eq:EASTER:1:20:AM}, for every $x \in [\eta^{-1}(\sc_2(t),t), \eta^{-1}(\sc(t),t))$ we define the intersection time $s = \nu_2(x)$ at which the $3$-characteristic $\eta(x,s)$ intersects the curve $\sc_2(s)$. Just as we showed that $\pt(\theta,s)$ is transverse
to the shock curve in the proof of Lemma~\ref{lem:12flows}, by the same argument,  for all labels $ x \in \Upsilon(t)$, the curve $\eta(x, s)$ is transverse to the 
characteristic curve $(\sc_2(t),t)$, and so there exists 
an $\sc_2(t)$-intersection time $ \nu_2(x)$ such that 
\begin{align} 
\eta(x,  \nu_2(x)) = \sc_2( \nu_2(x))\,. \label{eta-stoptime}
\end{align} 
Note that for these values of $x$, we have that $\nu_2(x) = \nu^\sharp(x,t)$, as was previously defined above \eqref{eq:finally:useful:3}.
When $x \not \in [\eta^{-1}(\sc_2(t),t), \eta^{-1}(\sc(t),t))$ we overload notation, and define $\nu_2(x) = \bar \eps$, to signify that $\eta(x,\cdot)$ does not intersect $\sc_2$.

For future purposes, for every $x \in [\eta^{-1}(\sc_1(t),t), \eta^{-1}(\sc(t),t))$ we define the intersection time $s = \nu_1(x)$ at which the $3$-characteristic $\eta(x,s)$ intersects the curve $\sc_1(s)$, i.e. 
\begin{align} 
\eta(x,  \nu_1(x)) = \sc_1( \nu_1(x))\,. \label{eta-stoptime1}
\end{align} 
The existence and uniqueness of $\nu_1(x)$ is again justified by the transversality of the $3$-characteristic and the $1$-characteristic. Again, for $x \not \in [\eta^{-1}(\sc_1(t),t), \eta^{-1}(\sc(t),t))$, we set $\nu_1(x) = \bar \eps$.

\begin{figure}[htb!]
\centering
\begin{tikzpicture}[scale=2.25]
\draw [<->,thick] (0,2.4) node (yaxis) [above] { } |- (3.5,0) node (xaxis) [right] { };
\draw[black,thick] (0,0)  -- (-2,0) ;
\draw[black,thick] (-2,2)  -- (3.5,2) ;
\draw[name path=A,red,ultra thick] (0,0) .. controls (1,0.5) and (2,1.5) .. (3,2);
\draw[red] (3,1.9) node { $\sc$}; 
\draw[name path = B,blue,ultra thick] (0,0) .. controls (0.5,0.85) and (1,1.25) .. (1.5,2) ;
\draw[blue] (1.3,1.9) node { $\sc_2$}; 
\draw[green!40!gray, ultra thick] (0,0) .. controls (.2,1) and (0.4,1.5) .. (.5,2) ;
\draw[green!40!gray] (.4,1.9) node { $\sc_1$}; 
\draw[black] (-2.4,0) node { $t=0$}; 
\draw[black] (-2.4,2) node { $t = \bar \eps$}; 

\draw[black,dotted] (1.8,0)  -- (1.8,1.7);
\draw[black] (-0.1,1.7) node { $t$}; 
\draw[black,dotted] (0,1.7)  -- (1.8,1.7);
\draw[black] (1.8,-0.1) node { $\theta$}; 
\filldraw[black] (1.8,1.7) circle (0.5pt);

\draw[red] (-1.5,0) .. controls (0.5,0.9) and (1,1.25) .. (1.8,1.7) ;
\filldraw[red] (-1.5,0) circle (0.5pt) node[anchor=north] {$x$};
\draw[red] (-1.3,0.65) node { $\eta(x,s)$}; 
\draw[dotted,->,red] (-1.3,0.55) -- (-1,0.27);

\draw[red,dotted,thick] (0.15,0.8)  -- (0,0.8);
\draw[red,dotted,thick] (0.8,1.16)  -- (0,1.16);
\filldraw[red] (0.85,1.16) circle (0.5pt);
\filldraw[red] (0.18,0.8) circle (0.5pt);
\draw[red] (-0.25,0.8) node { $\nu_1(x)$}; 
\draw[red] (-0.25,1.16) node { $\nu_2(x)$}; 
\end{tikzpicture}
\vspace{-0.2cm}
\caption{\footnotesize Fix a point $(\theta,t) $ which lies in between $\sc_1$ and $\sc_2$, and let $x$ be the label such that $\eta(x,t) = \theta$. The intersection time of $\eta(x,s)$ with $\sc_2$ is denoted by $\nu_2(x)$, while the intersection time of $\eta(x,s)$ with $\sc_1$ is denoted by $\nu_1(x)$.}
\label{fig:many:characteristics:2}
\end{figure}
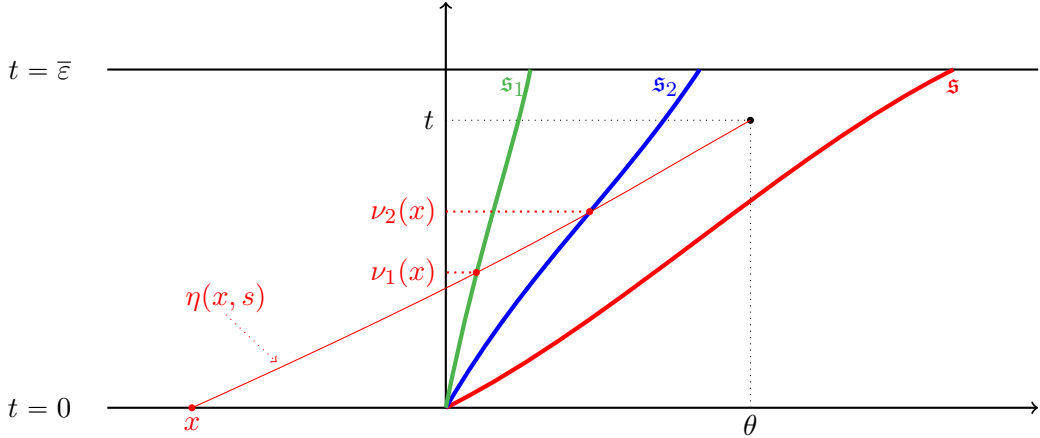

With this notation, we return to the integral term in \eqref{eq:EASTER:1:20:AM}, and
recall that $2 \kappa^{-1}(\theta- \sc_2(s)) \leq \st(\theta,s)  \leq 4 \kappa^{-1}(\theta- \sc_2(s))$. This justifies defining the curve $\gamma(s) = \eta(x,s) - \sc_2(s)$. Note that in view of Remark~\ref{rem:FU:3} and~\ref{rem:FU:2}, we have that 
 $\dot \gamma(s) = \lambda_3(\eta(x,s),s) - \dot{\sc}_2(s) \geq \frac 14 \kappa$.  Hence, 
\begin{align} 
\int_{\tilde \nu}^t {\bf 1}_{\sc_2(s) <\eta(x,s)\leq \sc(s)} \st(\eta(x,s),s)^{-\frac 12} ds
&\leq 
\int_{\nu_2(x)}^t \st(\eta(x,s))^{-\frac 12} ds \notag\\
&\le \kappa^{\frac 12} 
\int_{\nu_2(x)}^t (\eta(x,s) - \sc_2(s))^{-\frac 12} ds \notag\\
&\leq
4 \kappa^{-\frac 12} \int_{\nu_2(x)}^t \dot \gamma(s) (\gamma(s))^{-\frac 12} ds \notag\\
&\leq 8 \kappa^{-\frac 12}  \gamma(t)^{\frac 12} 
= 8 \kappa^{-\frac 12} (\theta - \sc_2(t))^{\frac 12}  \leq 8 t^{\frac 12}\,. \label{st-integral}
\end{align} 
In the last inequality above we have used that $|\theta -\sc_2(t) | \leq \sc(t) - \sc_2(t)  \leq \frac{\kappa t}{2}$. 
From \eqref{eq:EASTER:1:20:AM} and \eqref{st-integral}, we deduce the non-sharp upper bound
\begin{align} 
\sabs{\eta_{xx}(x,t) -\etab_{xx}(x,t)} \le  4 ( M_1 + {N_5} + \mm^4) \kappa \bb^{-\frac 32} t^{-\frac 32} 
\,,
\label{temp-junk1}
\end{align} 
for  $x = \eta^{-1}(\theta,t) $, when $ y \in(\sc_2(t) , \sc(t) + \frac{\kappa t}{3})$. 

Note that \eqref{temp-junk1} is weaker than the bound claimed in the second line of \eqref{dxxeta-diff}. This rough bound \eqref{dxxeta-bound} may now be used to  establish an optimal bound for $\eta_{xx} - \etab_{xx}$ as follows.
Estimate \eqref{temp-junk1} is combined with \eqref{first-fun-day} and \eqref{first-fun-dayL},  together with the bound \eqref{temp-junk2}, to show that for $\bar\eps $ taken sufficiently small we have
\begin{align} 
\sabs{\eta_{xx}(x,t)}  \le
\begin{cases}
12 \bb \kappa^{-\frac 53}  t^{-\frac 23}  ,  &\mbox{if } \theta \not \in(\sc_2(t) , \sc(t) + \frac{\kappa t}{3})\\
8 ( M_1 + {N_5} +  \mm^4) \kappa \bb^{-\frac 32} t^{-\frac 32}  & \mbox{if }   \theta  \in(\sc_2(t) , \sc(t) + \frac{\kappa t}{3})
\end{cases} \,,
 \label{dxxeta-bound}
\end{align} 
for $t\in[0,\bar\eps]$ where $x = \eta^{-1}(\theta,t)  \in \Upsilon(t)$. Moreover, by the definition of the time $\nu^\sharp(x,t)$ appearing in \eqref{eta-stoptime:1}, upon letting $(\theta,t)  \mapsto (\eta(x,s),s)$ in \eqref{dxxeta-bound}, we obtain that 
\begin{align} 
\sabs{\eta_{xx}(x,s)}  \le
\begin{cases}
12 \bb \kappa^{-\frac 53} s^{-\frac 23}  ,  &\mbox{if } s\leq \nu^\sharp(x,t)\\
8 ( M_1 + {N_5} +   \mm^4) \kappa \bb^{-\frac 32} s^{-\frac 32}  & \mbox{if }   \nu^\sharp(x,t) < s\leq t
\end{cases} \,,
 \label{dxxeta-bound-actually-useful}
\end{align} 
where we have overloaded notation and have defined $\nu^\sharp(x,t):=t$ whenever $ \eta(x,s) <  \sc_2(s)$ or $\eta(x,s) > \sc(s) + \frac{\kappa s}{3}$ for all $s\in [0,t]$.

Next, differentiating \eqref{nopork1}, we arrive at 
\begin{align}
\partial_s (\eta_{xx} - \etab_{xx}) 
&= w_0''  \left(  e^{-{\mathcal I}(\cdot,0,s)} - 1 \right) - w_0'  e^{-{\mathcal I}(\cdot,0,s)} \p_x {\mathcal I}(\cdot,0,s)  
\notag\\
&\quad + \p_\theta  Q_1 \circ \eta \ (\eta_x)^2 + Q_1 \circ \eta \ \eta_{xx}
\notag\\
&\quad + \int_0^s \left( \p_\theta  Q^w \circ \eta \ \eta_x^2 + Q^w\circ \eta \  \eta_{xx}  - Q^w \circ \eta \  \eta_x \p_x {\mathcal I}(\cdot,s',s) \right)e^{-{\mathcal I}(\cdot,s',s)} ds'
\label{eq:barf:1}
\end{align}
where we recall that ${\mathcal I}$, $Q^w$ and $Q^1$ are defined as in \eqref{eq:Ia:def}, \eqref{eq:Qw:def}, and \eqref{windy-day-garbage}.
From \eqref{eq:AC:DC:con}, \eqref{eq:w:z:k:a:boot}, \eqref{eq:k:dxx:boot}, and \eqref{st-integral} we deduce
\begin{subequations}
\label{eq:want:to:vomit:2}
\begin{align}
\sabs{\p_x {\mathcal I}(\cdot,s',s)} &\leq 24 \mm s + {\bf 1}_{s>\nu_2(x)} (\mm^{\frac 12} + \mm M_3)  s^{\frac 12}   \label{eq:want:to:vomit:2aaa}\\
\sabs{Q_1(\cdot,s)} &\leq \mm^2 s^{\frac 12}\label{eq:want:to:vomit:2a} \\
\sabs{Q^w(\cdot,s)} &\leq 12 \mm  \,.\label{eq:want:to:vomit:2b}
\end{align}
\end{subequations}
Moreover, differentiating  \eqref{eq:Qw:def} and \eqref{windy-day-garbage}, using  \eqref{eq:w:z:k:a:boot} we also obtain
\begin{subequations}
\label{eq:want:to:vomit:3}
\begin{align}
\sabs{\p_\theta  Q_1(\cdot,s)} &\leq  (\mm s)^{\frac 12} \sabs{w_\theta(\cdot,s)} + \tfrac{\mm}{4} \sabs{k_{\theta\theta} (\cdot,s)}+ \tfrac 13 \sabs{z_{\theta\theta} (\cdot,s)} + C s \label{eq:want:to:vomit:3a}\\
\sabs{\p_\theta  Q^w(\cdot,s)} &\leq \mm^3 s^{\frac 12} \sabs{k_{\theta\theta} (\cdot,s)} + \mm^2 s^{\frac 12}   \sabs{z_{\theta\theta} (\cdot,s)}  + 3 \mm \sabs{a_{\theta\theta} (\cdot,s)}+ 12 \mm  \sabs{w_{\theta}(\cdot,s)}   + C  \label{eq:want:to:vomit:3b}\,.
\end{align}
\end{subequations}
These bounds are used to estimate the three lines on the right side of \eqref{eq:barf:1} as follows.
Using \eqref{eq:want:to:vomit:2aa} and \eqref{eq:want:to:vomit:2aaa}, we obtain
\begin{align}
\mbox{first line on RHS of } \eqref{eq:barf:1} 
\leq 24 \mm s |w_0''(x)|  + 2 (\mm^{\frac 12} + \mm M_3) s^{\frac 12} |w_0'(x)| \,.
\label{eq:barf:2}
\end{align}
Next, using \eqref{eq:want:to:vomit:2a} and \eqref{eq:want:to:vomit:3a}, combined with \eqref{eq:cb:0},   \eqref{eq:z:dxx:boot}, \eqref{eq:k:dxx:boot},   \eqref{eq:z:dxx:boot2}, and \eqref{dxxeta-bound-actually-useful}, we estimate
\begin{align}
&\mbox{second line on RHS of } \eqref{eq:barf:1} \notag\\
&\quad \leq  \mm^2 s^{\frac 12} \left( 12 \bb \kappa^{-\frac 53} s^{-\frac 23} {\bf 1}_{s\leq \nu^\sharp(x,t)} + 8 (M_1 + {N_5} +\mm^4) \kappa \bb^{-\frac 32} s^{-\frac 32} {\bf 1}_{s> \nu^\sharp(x,t)} \right) + 4 (\mm s)^{\frac 12} \sabs{w_\theta(\cdot,s)} \notag\\
&\quad \quad +  ( \mm M_2 +2  M_3) \st(\eta(x,s),s)^{-\frac 12} {\bf 1}_{s > \nu_2(x)} + 2 N_2 \stt(\eta(x,s),s)^{-\frac 12} {\bf 1}_{\nu_1(x) < s < \nu_2(x)} + C s   \,.
\label{eq:barf:3}
\end{align}
The estimate for the third line of \eqref{eq:barf:1} is more delicate, and proceeds in several steps. By using  \eqref{dxxeta-bound-actually-useful}, \eqref{eq:want:to:vomit:2aaa}, \eqref{eq:want:to:vomit:2b}, and \eqref{eq:want:to:vomit:3b},
combined with \eqref{eq:cb:0}, \eqref{eq:AC:DC:con}, \eqref{eq:w:dx:boot}, \eqref{eq:second:order:boot}, \eqref{eq:second:order:boot2}, and \eqref{eq:second:order:boot3},   we have
\begin{align}
&\mbox{third line on RHS of } \eqref{eq:barf:1} \notag\\
&\quad \leq C 
+ C \int_0^{\min\{ s, \nu^\sharp(x,t)\}} (s')^{-\frac 23}  ds'  + {\bf 1}_{s> \nu^\sharp(x,t)} 96 (M_1 + {N_5}+ \mm^4) \kappa \bb^{-\frac 32} \mm  \int_{\nu^\sharp(x,t)}^s (s')^{-\frac 32}  ds' \notag\\
&\quad \quad + 4 (\mm^3 M_3 + \mm^2 M_2 + s^{\frac 12} M_5 ) {\bf 1}_{s>\nu_2(x)  }s^{\frac 12} \int_{\nu_2(x)}^{s}    \st(\eta(x,s'),s')^{-\frac 12} ds' \notag\\
&\quad \quad + 4 \mm^2 N_2 {\bf 1}_{s > \nu_1(x)  } s^{\frac 12} \int_{\nu_1(x)}^{\min\{s, \nu_2(x)\}}    \stt(\eta(x,s'),s')^{-\frac 12} ds' \notag\\
&\quad \quad + 12 \mm \left( (M_5 +N_7) {\bf 1}_{s>\nu_2(x)}  \int_{\nu_2(x)}^{s} (s')^{-1} ds' +  N_3 \int_{\nu_1(x)}^{\min\{s, \nu^\sharp(x,t)\}}  (s')^{-\frac 23} ds' \right) 
\,.
\label{eq:barf:4a}
\end{align}
Next, by using \eqref{st-integral}, and the fact that in view of the relations $\stt(\theta,s)  \approx \kappa^{-1}(\theta - \sc_1(s))$ and $\lambda_3(\eta(x,s),s) - \dot{\sc_1}(s) \geq \frac{1}{2} \kappa$ the same argument used to prove \eqref{st-integral} also establishes
\begin{align} 
\int_{\nu_1(x)}^t \stt(\eta(x,s))^{-\frac 12} ds \leq C t^{\frac 12}\,, \label{eq:stt:eta:int}
\end{align} 
and so from \eqref{eq:barf:4a}, \eqref{eta-stoptime:1}, \eqref{st-integral}, and \eqref{eq:stt:eta:int} we obtain that
\begin{align}
\mbox{third line on RHS of } \eqref{eq:barf:1} 
&\leq C  + {\bf 1}_{s> \nu^\sharp(x,t)} 200 (M_1 + {N_5}+ \mm^4) \kappa \bb^{-\frac 32} \mm   (\nu^\sharp(x,t))^{-\frac 12}   \notag\\
&\leq C  + {\bf 1}_{s> \bb^{\frac 32} \kappa^{-1} t^{\frac 32}} 200 (M_1 + {N_5}+ \mm^4) \kappa^{\frac 32} \bb^{-\frac 94} \mm  t^{-\frac 34}
\,.
\label{eq:barf:4}
\end{align}
Finally, using the bounds \eqref{eq:barf:2}, \eqref{eq:barf:3} (which needs to be combined with \eqref{eq:AC:DC:con}, \eqref{eq:w:dx:boot},  \eqref{eta-stoptime:1}, \eqref{st-integral}, \eqref{eq:stt:eta:int}), and \eqref{eq:barf:4}, we integrate \eqref{eq:barf:1} on $[0,t]$, use  \eqref{first-fun-day}, and arrive at
\begin{align}
(\eta_{xx} - \etab_{xx})(x,t)
&\leq 12 \mm t^2 |w_0''(x)| + 2 (\mm^{\frac 12} + \mm M_3) t^{\frac 32} |w_0'(x)| + C \log \frac{t}{\nu^\sharp(x,t)} + C  t^{\frac 12}   
\notag\\
&\qquad \qquad + 200 (M_1 + {N_5}+ \mm^4) \kappa^{\frac 32} \bb^{-\frac 94} \mm  t^{\frac 14}
\notag\\
&\leq 12 \mm t^2 |w_0''(x)| +  C \log t
\notag\\
&\leq  5 \mm \bb^{-\frac 32} t^{-\frac 12}
\label{eq:barf:final} 
\end{align}
for $x = \eta^{-1}(\theta,t) $ with $\theta \in (\sc_2(t) , \sc(t) + \frac{\kappa t}{3})$. This concludes the proof of the second inequality in \eqref{dxxeta-diff}.

\vspace{.05 in} 
\noindent
{\bf The case $\theta\leq \sc_2(t)$ or $\theta\geq \sc(t) + \frac{\kappa t}{3}$ revisited.} 
In order to prove the Lemma, we note that the constant claimed in the first inequality in \eqref{dxxeta-diff} is different than the one previously established in \eqref{temp-junk2}; this issue plays an important role proof of Lemma~\ref{lem:wyy}.

For this purpose we combine \eqref{eq:barf:1} with the bounds  \eqref{eq:barf:2}, \eqref{eq:barf:3}, \eqref{eq:barf:4} (the first line of this inequality is used here), and use the fact that for $y$ as above we have that $\nu^\sharp(x,t) ,\nu_2(x) \geq t > s$, to arrive at 
\begin{align*}
\sabs{\p_s (\eta_{xx} - \etab_{xx})} 
&\leq  24 \mm s |w_0''(x)|  + 2 (\mm^{\frac 12} + \mm M_3) s^{\frac 12} |w_0'(x)| \notag\\
&\qquad + 
  24 \bb \kappa^{-\frac 53}\mm^2  s^{-\frac 16} + 4  \mm^{\frac 12} s^{-\frac 12}  +    2 N_2 \stt(\eta(x,s),s)^{-\frac 12} {\bf 1}_{\nu_1(x) < s < \nu_2(x)}   + C 
\,.
\end{align*}
Integrating the above estimate on $[0,t]$ and appealing to \eqref{first-fun-dayL} and \eqref{eq:stt:eta:int} we obtain
\begin{align*}
\sabs{(\eta_{xx} - \etab_{xx})(x,t)}
 &\leq 12 \mm t^2 |w_0''(x)|  + 2 (\mm^{\frac 12} + \mm M_3) t^{\frac 32} |w_0'(x)| + C t^{\frac 12}
 \notag\\
 &\leq 48 \mm \bb \kappa^{-\frac 53} t^{\frac 13} + C t^{\frac 12}\,.
\end{align*}
Taking into account \eqref{eq:b:m:ass} and the fact that $t$ is sufficiently small with respect to $\kappa, \bb, \mm$, the above estimate proves the first inequality in \eqref{dxxeta-diff}.
\end{proof}

\subsubsection{Derivatives of the $1$- and $2$-characteristics}

\begin{lemma}\label{lem:dy12} For any $(\theta,t)  \in \mathcal{D}_{\bar \eps}$, 
\begin{align} 
\sup_{s \in [0,t]}  \abs{\p_\theta  \pt (\theta,s)  -1} & \le {60} \bb \kappa^{-\frac{2}{3}}  t^{\frac{1}{3}}  
\,,  
\qquad 
\sup_{s \in [0,t]}  \abs{\p_\theta  \pst (\theta,s)  -1}  \le {30} \bb \kappa^{-\frac{2}{3}}  t^{\frac{1}{3}} 
 \,.  \label{dxphipsi-bound} 
\end{align} 
\end{lemma} 
\begin{proof}[Proof of Lemma \ref{lem:dy12}] 
For any $(\theta,t)  \in \mathcal{D}_{\bar \eps}$ and $s \in [0,t]$,
 $\p_s \p_\theta  \pt =  \p_\theta \lambda_2 \circ \pt \ \p_\theta  \pt$, and since $\p_\theta  \pt(\theta,t)=1$, we see that
 $$\p_\theta  \pt(\theta,s) = e^{-\int_s^t \p_\theta \lambda_2 \circ \pt dr} 
 = e^{-{\frac{2}{3}} \int_s^t \p_\theta \wb\circ \pt dr}  e^{-{\frac{2}{3}} \int_s^t \p_\theta (w-\wb + z) \circ \pt dr} \,.
 $$
Similarly, for $s \in [0,t]$, $\p_s \p_\theta  \pst =  \p_\theta \lambda_1 \circ \pst \ \p_\theta  \pst$, and since $\p_\theta  \pst(\theta,t)=1$, so that
 $$\p_\theta  \pst(\theta,s) = e^{-\int_s^t \p_\theta \lambda_1 \circ \pst dr} 
 = e^{-{\frac{1}{3}} \int_s^t \p_\theta \wb\circ \pst dr}  e^{- \int_s^t \bigl({\frac{1}{3}} \p_\theta (w-\wb) + \p_\theta  z\bigr) \circ \pst dr}  \,.
 $$
By combining the above two identities with the bounds \eqref{eq:b:m:ass}, \eqref{eq:w:z:k:a:boot}, and \eqref{eq:Steve:needs:this:1} (with $\mu = \frac 34$ for $\pt$ and $\mu = \frac 12$ for $\pst$), and using that $\bar \eps$ is sufficiently small, the bound
\eqref{dxphipsi-bound} follows.
\end{proof} 

We next derive second derivative identities and bounds for these characteristics.   As we noted above, the bounds differ, depending on the spacetime region.  In
order to state these bounds, we first define the $2$-characteristic $\sc_1(t)$-intersection time.
Just as we showed that $\pt(\theta,s)$ is transverse
to the shock curve in the proof of Lemma \ref{lem:12flows}, by the same argument, the curve $\pt(\theta,s)$ is transverse to the characteristic curve $(\sc_1(t),t)$, and there exists 
an $\sc_1(t)$-intersection time $\st_1(\theta,t) $ such that 
$$\pt(\theta, \st_1(\theta,t) ) = \sc_1 ( \st_1(\theta,t) )\,.$$

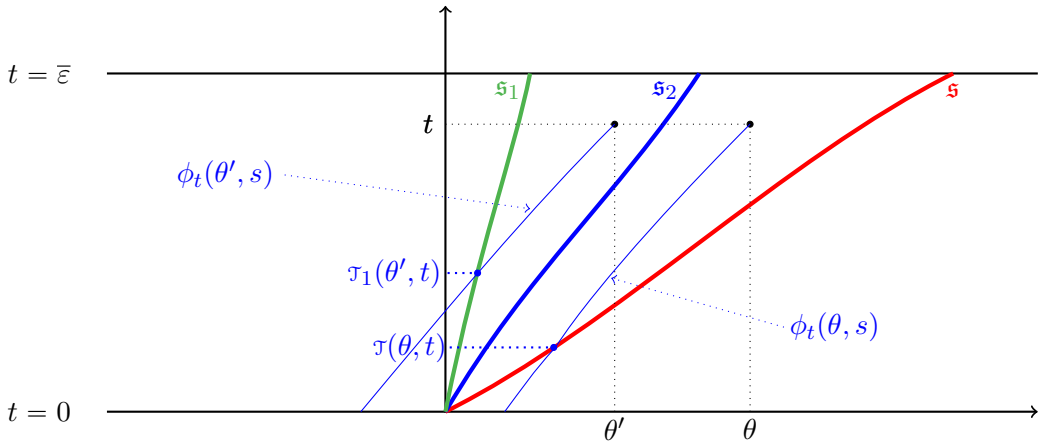
\begin{figure}[htb!]
\centering
\begin{tikzpicture}[scale=2.25]
\draw [<->,thick] (0,2.4) node (yaxis) [above] { } |- (3.5,0) node (xaxis) [right] { };
\draw[black,thick] (0,0)  -- (-2,0) ;
\draw[black,thick] (-2,2)  -- (3.5,2) ;
\draw[name path=A,red,ultra thick] (0,0) .. controls (1,0.5) and (2,1.5) .. (3,2);
\draw[red] (3,1.9) node { $\sc$}; 
\draw[name path = B,blue,ultra thick] (0,0) .. controls (0.5,0.85) and (1,1.25) .. (1.5,2) ;
\draw[blue] (1.3,1.9) node { $\sc_2$}; 
\draw[green!40!gray, ultra thick] (0,0) .. controls (.2,1) and (0.4,1.5) .. (.5,2) ;
\draw[green!40!gray] (.37,1.9) node { $\sc_1$}; 
\draw[black] (-2.4,0) node { $t=0$}; 
\draw[black] (-2.4,2) node { $t = \bar \eps$}; 

\draw[black,dotted] (1,0)  -- (1,1.7);
\draw[black] (-0.1,1.7) node { $t$}; 
\draw[black,dotted] (0,1.7)  -- (1.8,1.7);
\draw[black] (1,-0.1) node { $\theta'$}; 
\filldraw[black] (1,1.7) circle (0.5pt);

\draw[black,dotted] (1.8,0)  -- (1.8,1.7);
\draw[black] (-0.1,1.7) node { $t$}; 
\draw[black] (1.8,-0.1) node { $\theta$}; 
\filldraw[black] (1.8,1.7) circle (0.5pt);

\draw[blue] (0.65,0.38) .. controls (0.9,0.77) and (1.5,1.4) .. (1.8,1.7) ;
\draw[blue] (0.35,0) .. controls (0.44,0.12) and (0.52,0.24) .. (0.65,0.38) ;
\draw[blue] (-0.5,0) .. controls (0,0.6) and (0.5,1.2) .. (1,1.7) ;

\filldraw[blue] (0.64,0.38) circle (0.5pt);
\draw[blue,dotted,thick] (0.65,0.38)  -- (0,0.38);
\draw[blue] (-0.21,0.38) node { $\st(\theta,t)$}; 
\draw[dotted,->,blue] (2,0.5) -- (0.98,0.79);
\draw[blue] (2.3,0.5) node { $\phi_t(\theta,s) $};

\filldraw[blue] (0.19,0.82) circle (0.5pt);
\draw[blue,dotted,thick] (0.2,0.82)  -- (0,0.82);
\draw[blue] (-0.3,0.82) node { $\st_1(\theta',t)$}; 
\draw[dotted,->,blue] (-0.95,1.4) -- (0.5,1.2);
\draw[blue] (-1.3,1.4) node { $\phi_t(\theta',s)$}; 

\end{tikzpicture}

\vspace{-0.2cm}
\caption{\footnotesize Fix points: $(\theta',t)$ which lies in between $\sc_1$ and $\sc_2$, and $(\theta,t) $ which lies in between $\sc_1$ and $\sc$. The intersection time of $\pt(\theta',s)$ with $\sc_1$ is denoted by $\st_1(\theta',t)$, while the intersection time of $\pt(\theta,s)$ with $\sc$ is denoted as usual by $\st(\theta,t)$.}
\label{fig:many:characteristics:4a}
\end{figure}

\begin{lemma}\label{lem:dyy12} Let $(\theta,t)  \in \mathcal{D}_{\bar \eps}$.  
Then, for all $(\theta,t)  \in \mathcal{D}^k_{\bar \eps}$ we have 
\begin{align} 
\sup_{s \in [\st(\theta,t),t]} s \abs{\p_\theta ^2 \pt (\theta,s)  } & \le 3\mm^2 \kappa ^{-3}\,,  
\qquad
 \sup_{s \in [\stt(\theta,t),t]} s \abs{\p_\theta ^2 \pst (\theta,s) }  \le \mm^ {\frac{1}{2}}  \kappa^{-{\frac{3}{2}} }  \,, 
 \label{dxxphipsi-bound-all} 
\end{align} 
while for all $(\theta,t)  \in\mathcal{D}^z_{\bar\eps} \setminus \overline{ \mathcal{D} ^k_{\bar \eps}} $ it holds that 
\begin{align}
\sup_{s\in [\st_1(\theta,t) ,t]} s^{\frac 23} \abs{\p_\theta ^2 \pt (\theta,s)  } & \le 4\bb \mm^2 \kappa^{-3}   \,, 
\qquad
\sup_{s\in [\stt(\theta,t),t]}  \abs{\p_\theta ^2 \pst (\theta,s)  }  \le  \mm^ {\frac{1}{2}}  \kappa^{-{\frac{3}{2}} }  \stt(\theta,t)^{-1}     \,.  \label{dxxphi-bound2} 
\end{align} 
Lastly, for $(\theta,t)  \in\mathcal{D}_{\bar\eps} \setminus \overline{ \mathcal{D}^z_{\bar \eps}} $ we have
\begin{alignat} {2}
\sup_{s \in {[0,t]}}  \abs{\p_\theta ^2 \pt (\theta,s)  }& \le 3\mm^2 \kappa ^{-3} {t^{-1}}  \,,  \ \ 
&& \sup_{s \in {[0,t]}}  \abs{\p_\theta ^2 \pst (\theta,s)  } \le \mm^ {\frac{1}{2}}  \kappa^{-{\frac{3}{2}} } {t^{-1}}  \,, 
\ \ \ \sc(t) \le \theta \le \pi\,,
 \label{dxxphipsi-boundR}  \\
 \sup_{s \in {[0,t]}} s^{\frac 23} \abs{\p_\theta ^2 \pt (\theta,s)  }& \le 3\mm^2 \kappa ^{-3}   \,,  \ \ &&\sup_{s \in {[0,t]}}  \abs{\p_\theta ^2 \pst (\theta,s)  } \le \mm^ {\frac{1}{2}}  \kappa^{-{\frac{3}{2}} } {t^{-{\frac{2}{3}} }}  \,, 
\ \ \ -\pi \le  \theta \le \sc_1(t) \,.
 \label{dxxphipsi-boundL} 
\end{alignat} 
\end{lemma} 
\begin{proof}[Proof of Lemma \ref{lem:dyy12}] 
 It is convenient to introduce the (temporary) variables
$C= c \circ \pt$,  $B  = \p_s\pt = \lambda _2 \circ \pt$ and $A  = a \circ \pt$
so that using the chain-rule, the equation for $c$ given by \eqref{xland-sigma} can be written as
$$
\p_s C+ \tfrac{1}{2} C (\p_\theta  \pt) ^{-1} \p_\theta  B = -\tfrac{8}{3} A C \,.
$$
It follows that
$$
(\p_\theta  \pt)^ {\frac{1}{2}} \p_sC +  \tfrac{1}{2} C(\p_\theta  \pt) ^{- {\frac{1}{2}} } \p_\theta B =  -\tfrac{8}{3}(\p_\theta  \pt)^ {\frac{1}{2}} A C \,,
$$
and hence
$$
 \p_s \left( (\p_\theta  \pt)^ {\frac{1}{2}}C\right) +\tfrac{8}{3} A( \p_\theta  \pt)^ {\frac{1}{2}}C =0 \,.
$$
For $(\theta,t)  \in \mathcal{D} _{\bar\eps}^k$,  and letting  $s \in [\st(\theta,t) ,t)$, we integrate this equation from $s$ to $t$ and find that
\begin{align} 
\p_\theta  \pt(\theta,s) = e^{{\frac{16}{3}} \int_s^t (a \circ \phi_t)(y,s') ds'} \frac{c^2(\theta,t) }{c^2(\pt(\theta,s),s)} \,.  \label{dyphi}
\end{align} 
Differentiating \eqref{dyphi}, we find that
\begin{align} 
\p_\theta ^2 \pt(\theta,s) & =2 e^{{\frac{16}{3}}  \int_s^t a \circ \pt ds'}  \frac{c(\theta,t) }{c^3 ( \pt(\theta,s),s)} \Bigl( \tfrac{8}{3} c(\theta,t)  c ( \pt(\theta,s),s)\int_s^t (a_\theta \circ \pt \ \p_\theta  \pt)ds'  
\notag \\
& \qquad\qquad\qquad\qquad\qquad
+ c ( \pt(\theta,s),s) c_\theta(\theta,t)  -   c(\theta,t)  c_\theta( \pt(\theta,s),s)   \p_\theta  \pt(\theta,s) \Bigr) \,.
\label{dyyphi}
\end{align} 
In essence, the two worst terms in the above identity are $c_\theta(\theta,t) $ and $c_\theta( \pt(\theta,s),s)$, so that in view of \eqref{thegoodstuff} and \eqref{eq:w:z:k:a:boot} the bounds will be determined by how close $y$ is to $\sc(t)$, respectively $\pt(\theta,s)$ to $\sc(s)$.

A similar argument can be used to obtain a formula for $\p_\theta  \psi_t$.  To do so, we make the observation (see also~\eqref{cn1}) that  \eqref{xland-sigma} can be written using 
$\lambda_1$ as the transport velocity in the special form
\begin{align*} 
\p_t c +  \lambda _1 \p_\theta  c +  2 c \p_\theta  \lambda _1
= 2 \ss \p_\theta  z - \tfrac{8}{3}  a \ss \,.
\end{align*} 
We again introduce temporary variables
$C= c \circ \pst $ and $B  = \lambda _1 \circ \pst = \p_s \pst$,
so that
$$
\p_sC+ 2C (\p_\theta  \pst) ^{-1} \p_\theta  B  =   \left(2 \p_\theta z - \tfrac{8}{3}  a  \right)\circ \pst \,  C \,.
$$
Then, 
$$
 \p_s \left( (\p_\theta  \pst)^ 2 C  \right) -  \left(2 \p_\theta z - \tfrac{8}{3}  a  \right) \circ \pst \  (\p_\theta  \pst)^ 2 C =0 \,,
$$
and for any $(\theta,t)  \in \mathcal{D}^z_{\bar \eps}$ and  $s \in [\st(\theta,t) ,t)$, we integrate this equation from $s$ to $t$ and find that
\begin{align} 
\p_\theta  \pst(\theta,s)  = e^{ \int_s^t  \left( \frac{4}{3}a(\pst(\theta,s'),s')-  z_\theta(\pst(\theta,s'),s') \right)  ds'} \frac{c^{\frac{1}{2}} (\theta,t) }{c^ {\frac{1}{2}}(\pst(\theta,s),s) } \,.
\label{dypsi}
\end{align} 
Differentiating \eqref{dypsi} once more yields 
\begin{align} 
\p_\theta ^2 \pst(\theta,s) & =  \tfrac{1}{2} e^{\int_s^t  ( \frac{4}{3}a - z_\theta) \circ \pst ds'}  \frac{c^{\frac{1}{2}}(\theta,t)  }{c^{\frac{1}{2}}(\pst(\theta,s),s) }  \notag\\
& \times \Bigl( \int_s^t  ( \tfrac{8}{3}a_\theta -2z_{\theta\theta} ) \circ \pst \, \p_\theta  \pst ds' 
+  \frac{ \p_\theta  c (\theta,t)  }{ c(\theta,t)  } -  \frac{ \p_\theta  c (\pst(\theta,s),s) }{ c(\pst(\theta,s),s) }
 \p_\theta  \pst(\theta,s) \Bigr)\,.
 \label{dyypsi}
\end{align}
As before,  the  worst terms in the above identity are $c_\theta(\theta,t) $ and $c_\theta( \pt(\theta,s),s)$, but in order to justify this heuristic we need to estimate  the time integral of $z_{\theta\theta}  \circ \pst$. 

For $(\theta,t)  \in \mathcal{D}^k_{\bar \eps}$, we shall need a good bound for  $\int_{\stt(\theta,t)}^t z_{\theta\theta} ( \pst(\theta,s),s)ds$, and to this end, 
we employ an argument which is
very similar to the one we used to obtain \eqref{st-integral}.  
  Let us define $\gamma(s) = \pst(\theta,s)- \sc_2(s)$.  Since 
$\lambda_1(\pst(\theta,s),s)-  \dot{\sc}_2(s)   \leq - \frac{3}{10} \kappa$, we obtain
 $\dot\gamma(s) \leq -  \frac{3}{10} \kappa$. Moreover, using \eqref{c1c2diffkappa2} we have that for $\bar \eps$ sufficiently small,
 $\st(\theta,t) \ge \tfrac{5}{2} \kappa ^{-1} (\theta - \sc_2(t))$ for all $\sc_2(t) \le \theta \le \sc(t)$.
 Hence, 
\begin{align} 
\int_{\stt(\theta,t)}^t \st(\pst(\theta,s),s)^{-\frac 12} ds 
&\le  \tfrac{3}{5} \kappa ^{\frac 12} 
\int_{\stt(\theta,t)}^t ( \pst(\theta,s)-\sc_2(s)  )^{-\frac 12} ds \notag\\
&\leq -2 \kappa^{-\frac 12} \int_{\stt(\theta,t)}^t \dot \gamma(s) (\gamma(s))^{-\frac 12} ds \notag \\
&= 4 \kappa^{-\frac 12} \left(\gamma(\stt(\theta,t))^{\frac 12} -  \gamma(t)^{\frac 12}  \right)\notag \\
&\le 4 \kappa^{-\frac 12}  (\sc(\stt(\theta,t)) - \sc_2(\stt(\theta,t)))^ {\frac{1}{2}}  \le \tfrac 52  \stt(\theta,t)^ {\frac{1}{2}}  \,
 \label{stt-integral}
\end{align} 
From \eqref{stt-integral} and the bootstrap assumption \eqref{eq:z:dxx:boot}, we get
\begin{align} 
\int_{\stt(\theta,t)}^t \abs{ z_{\theta\theta}  \circ \pst }ds  \leq 3 M_2  \stt(\theta,t)^ {\frac{1}{2}} \,.
\label{zyy-integral}
\end{align} 

First consider $(\theta,t)  \in \mathcal{D}^k_{\bar \eps}$.
Combining \eqref{dyyphi} and \eqref{dyypsi}, with the bounds \eqref{thegoodstuff},   \eqref{eq:w:z:k:a:boot}, \eqref{dxphipsi-bound}, \eqref{zyy-integral},  and taking $\bar \eps$ sufficiently small, we see that  
$\sabs{\p_\theta ^2 \pt (\theta,s)  }  \le 3\mm^2 \kappa ^{-3}s^{-1} $ and $\sabs{\p_\theta ^2 \pst (\theta,s) }  \le \mm^ {\frac{1}{2}}  \kappa^{-{\frac{3}{2}} } s^{-1}$, 
which are the bounds  stated in \eqref{dxxphipsi-bound-all}.   
Here we use that $\st(\theta,t)$ and $\stt(\theta,t)$ are the shock intersection times for trajectories $\pt(\theta,s)$ and $\pst(\theta,s)$.

We next consider the case   $(\theta,t)  \in \mathcal{D}^z_{\bar\eps} \setminus \overline{ \mathcal{D} ^k_{\bar \eps}}$.
From \eqref{dyyphi}, by using \eqref{eq:w:z:k:a:boot}, \eqref{thegoodstuff1}, and \eqref{dxphipsi-bound},  we obtain
\begin{align*} 
\abs{\p_\theta ^2 \pt (\theta,s)  } & \le 4\bb \mm^2 \kappa^{-3}    s^{-{\frac{2}{3}} } \,,
\end{align*} 
for all $s\in [\st_1(\theta,t) ,t]$, which establishes the first bound in \eqref{dxxphi-bound2}.
Using the bootstrap assumption  \eqref{eq:z:dxx:boot2} and the bound \eqref{zyy-integral} for $s$ such that $\pst(\theta,s) \in \DD_{\bar \eps}^k$, respectively \eqref{eq:z:dxx:boot2} and the fact that $\stt(\pst(\theta,s),s) = \stt(\theta,t)$ for $\pst(\theta,s) \in \mathcal{D}^z_{\bar\eps} \setminus \overline{ \mathcal{D} ^k_{\bar \eps}}$, we obtain
$$
\int_{\stt(\theta,t)}^t  \sabs{ z_{\theta\theta} (\pst(\theta,s),s)} ds \le t N_2 \stt(\theta,t) ^{{-\frac 12}} + 3 M_2 \stt(\theta,t)^{\frac 12}\,. 
$$
Therefore, the identity \eqref{dyypsi} together with \eqref{thegoodstuff},   \eqref{eq:w:z:k:a:boot}, \eqref{dxphipsi-bound}, \eqref{dyyphi},  \eqref{dyypsi}, and the above estimate, show that
\begin{align*} 
 \sup_{s\in [\stt(\theta,t),t]} \abs{\p_\theta ^2 \pst (\theta,s)  } & \le  \mm^ {\frac{1}{2}}  \kappa^{-{\frac{3}{2}} }     \stt(\theta,t)^{-1 } \,,
\end{align*} 
for all $(\theta,t)  \in\mathcal{D}^z_{\bar\eps} \setminus \overline{ \mathcal{D} ^k_{\bar \eps}}$, which establishes the second bound in \eqref{dxxphi-bound2}. Note that this bound is only sharp when $s$ is very close to $\stt(\theta,t)$.

For the case that $(\theta,t)  \in \mathcal{D}_{\bar \eps}$ such that $\theta > \sc(t)$, we have that  $z=0$, and so the identities \eqref{dyyphi} and \eqref{dyypsi} show that
second derivatives of these characteristics are largest at
points $(\theta,t) $ which are very close to $\sc(t)$. Using that $|\pt(\theta,s) - \sc(s)| , |\pst(\theta,s) - \sc(s)| \gtrsim \kappa t$ for $s\in [0,t/2]$, using \eqref{thegoodstuff} and \eqref{eq:w:z:k:a:boot} it follows from \eqref{dyyphi} and respectively \eqref{dyypsi}  that 
\begin{align*} 
\sup_{s \in [0,t]}  \abs{\p_\theta ^2 \pt (\theta,s)  }& \le 3\mm^2 \kappa ^{-3} {t^{-1}}  \,, 
\qquad
\sup_{s \in [0,t]}  \abs{\p_\theta ^2 \pst (\theta,s)  } \le \mm^ {\frac{1}{2}}  \kappa^{-{\frac{3}{2}} } {t^{-1}}  \,, 
\end{align*} 
which establishes  \eqref{dxxphipsi-boundR} for $\sc(t) \le \theta \le \pi$.

For the case that $(\theta,t)  \in \mathcal{D} _{\bar \eps}$ such that $-\pi \le \theta \le \sc_1(t)$ we again have  that $z=0$. Using \eqref{thegoodstuff},   \eqref{eq:w:z:k:a:boot}, \eqref{dxphipsi-bound},
it similarly follows from \eqref{dyyphi} and \eqref{dyypsi} that
\begin{align} 
\sup_{s \in [0,t]} {s^{\frac 23}} \abs{\p_\theta ^2 \pt (\theta,s)  }& \le 3\mm^2 \kappa ^{-3}   \,,  
\qquad 
\sup_{s \in [0,t]}  \abs{\p_\theta ^2 \pst (\theta,s)  } \le \mm^ {\frac{1}{2}}  \kappa^{-{\frac{3}{2}} } {t^{-{\frac{2}{3}} }}\,, 
\end{align} 
which is the stated bound  \eqref{dxxphipsi-boundL}.   This improved growth rate of second derivatives makes use of the fact that for $-\pi \le \theta \le \sc_1(t)$, one the one hand we have $|\pst(\theta,s) - \sc(s)| \geq |\pst(\theta,s) - \sc_2(s)| \approx |\theta-\sc_2(t)| \gtrsim \kappa t$ for all $s\in [0,t]$, 
while on the other hand
$|\pt(\theta,s) - \sc(s)| \geq |\sc_1(s) - \sc(s)| \approx   \kappa s$ for all $s\in [0,t]$.
\end{proof}

\subsection{Second derivatives for $w$ along the shock curve}

\begin{lemma}
\label{lem:second:w:deriv:on:shock}
Assume that the shock curve $\sc$ satisfies \eqref{eq:sc:ass}, that $(w,z,k,a) \in {\mathcal X}_{\bar \eps}$ (as defined in \eqref{eq:w:z:k:a:boot}--\eqref{eq:R:def}), and that the second derivative bootstraps \eqref{eq:second:order:boot}--\eqref{eq:second:order:boot3} hold. Then we have that 
\begin{align}
\abs{ \tfrac{d^2}{dt^2} w(\sc(t)^{\pm},t) - \tfrac{d^2}{dt^2} \wb(\sc(t)^{\pm},t) }
&\leq (4 \bb^3  M_1  +  \mm^5 )   t^{-1}
\label{eq:whispers:0}
\end{align}
where   $M_1 = M_1(\kappa,\bb,\cc,\mm) >0$ is the constant from \eqref{eq:w:dxx:boot}. In particular, the bound \eqref{eq:dt:dt:jumps:on:shock:are:close} holds with the constant $\Rsf^* = 4 \bb^3 M_1 +  \mm^5  $, which in turn implies \eqref{eq:dt:dt:zl:kl:on:shock}.
\end{lemma}
\begin{proof}[Proof of Lemma~\ref{lem:second:w:deriv:on:shock}]
First, we note that from \eqref{eq:w:z:k:a}, Lemma~\ref{lem:u0:u0'}, and the fact that $(w,z,k,a) \in {\mathcal X}_{\bar \eps}$ cf.~\eqref{eq:w:z:k:a:boot}--\eqref{eq:R:def}, we have that 
\begin{subequations}
\label{eq:double:trouble}
\begin{align}
\abs{\p_\theta  w(\theta,t) }  &\leq \abs{\p_\theta  \wb(\theta,t) } + R_2 (\bb t)^{-\frac 12} \leq \tfrac{9}{11} t^{-1} +  R_2 (\bb t)^{-\frac 12} \leq t^{-1}
\label{eq:double:trouble:a}\\
\abs{\p_t w(\theta,t) }  &\leq (\mm + R_1 t + \tfrac 13 R_3 t^{\frac 32}) t^{-1} + \tfrac{8R_7}{3} (\mm + R_1 t) + \tfrac{R_6 }{24} t^{\frac 12} (\mm + R_1 t + R_3 t^{\frac 32})^2  \leq 2 \mm t^{-1}  \\
\abs{\p_t z(\theta,t) }  &\leq (\tfrac 13 (\mm + R_1 t) + R_3 t^{\frac 32}) R_4 t^{\frac 12} + \tfrac 83 R_7 R_3 t^{\frac 32}+ \tfrac{1}{24} (\mm + R_1 t + R_3 t^{\frac 32})^2 R_6 t^{\frac 12} \leq \mm^3 t^{\frac 12}\\
\abs{\p_t a(\theta,t) }  &\leq \tfrac 12 (\mm + R_1 t + R_3 t^{\frac 32}) R_7 + \tfrac 12 (\mm + R_1 t + R_3 t^{\frac 32})^2 \leq \mm^3
\end{align}
\end{subequations}
for all $(\theta,t) \in \DD_{\bar \eps}$, and in particular as $\theta \to \sc(t)^{\pm}$.

From the chain rule, we obtain that 
\begin{align}
\tfrac{d^2}{dt^2} w(\sc(t)^{\pm},t)
&= \ddot{\sc}(t) (w_\theta )(\sc(t)^{\pm},t) + (\dot \sc(t))^2 (w_{\theta\theta}  )(\sc(t)^{\pm},t)
+ 2 \dot{\sc}(t)   (w_{t\theta} )(\sc(t)^{\pm},t) + (w_{tt} )(\sc(t)^{\pm},t)
\,.
\label{eq:whispers:1}
\end{align}
From the  evolution equations \eqref{eq:w:z:k:a} and the definition of the wave speeds in \eqref{eq:wave-speeds} we have the identities
\begin{subequations}
\label{eq:yelling:and:screaming}
\begin{align}
w_{t\theta} &= - (w + \tfrac 13 z) w_{\theta\theta}   - (w_\theta  + \tfrac 13 z_\theta ) w_\theta  - \tfrac 83  (aw)_\theta + \tfrac{1}{12} (w-z) (w_\theta  - z_\theta ) k_\theta  + \tfrac{1}{24} (w-z)^2 k_{\theta\theta}   \\
w_{tt}  &= - (w + \tfrac 13 z) w_{t\theta}  - (w_t  + \tfrac 13 z_t ) w_\theta  - \tfrac 83 \p_t (aw) + \tfrac{1}{12} (w-z) (w_t  - z_t ) k_\theta  + \tfrac{1}{24} (w-z)^2 k_{t\theta} \notag\\
&=  (w + \tfrac 13 z) \left(  (w + \tfrac 13 z) w_{\theta\theta}   + (w_\theta  + \tfrac 13 z_\theta ) w_\theta  + \tfrac 83   (aw)_\theta -\tfrac{1}{12} (w-z) (w_\theta  - z_\theta ) k_\theta  - \tfrac{1}{24} (w-z)^2 k_{\theta\theta}  \right) \notag\\
&\quad +  \left((w + \tfrac 13 z) w_\theta + \tfrac 83 a w + \tfrac 13(\tfrac 13 w + z) z_\theta + \tfrac 89 a z - \tfrac{1}{18} (w-z)^2 k_{\theta}  \right) w_\theta - \tfrac 83 \p_t (aw) \notag\\
&\quad + \tfrac{1}{12} (w-z) (w_t  - z_t ) k_\theta  - \tfrac{1}{36} (w-z)^2 \left((w+z) k_{\theta\theta}  + (w_\theta + z_\theta) k_\theta \right)
\end{align}
\end{subequations}
pointwise for $(\theta,t) \in \DD_{\bar \eps}$. We shall in fact use \eqref{eq:yelling:and:screaming} only for $\theta \to \sc(t)^\pm$, so that the relevant bounds on second derivatives of $w$ are given by \eqref{eq:w:dxx:boot}, the second branch in \eqref{eq:w:dxx:boot3}, 
and from the estimate
$
|\p_{\theta}^2  \wb(\theta,t) | \leq \tfrac{11}{4} \bb^{-\frac 32} t^{-\frac 52}
$, 
which follows from Lemma~\ref{lem:u0:u0'} and \eqref{eq:pxx:wb:def}; together, these bounds and the fact $\st(\sc(t)^-,t) = t$, imply that
\begin{align*}
|\p_{\theta}^2  w(\sc(t)^{\pm} ,t)| \leq \tfrac{11}{4} \bb^{-\frac 32} t^{-\frac 52} +C t^{-2}   \leq 3 \bb^{-\frac 32} t^{-\frac 52}\,.
\end{align*}
Similarly, for the second derivative of $k$ we appeal to \eqref{eq:k:dxx:boot}, 
which gives
\begin{align*}
|\p_{\theta}^2  k(\sc(t)^{-} ,t)| \leq M_3 t^{-\frac 12} \,.
\end{align*}
From the above two estimates, the bounds \eqref{eq:yelling:and:screaming}, the fact that $(w,z,k,a) \in {\mathcal X}_{\bar \eps}$ cf.~\eqref{eq:w:z:k:a:boot}--\eqref{eq:R:def}, we deduce that at $(\sc(t)^\pm,t)$:
\begin{subequations}
\label{eq:whispers:2}
\begin{align}
\abs{w_{t\theta} + w w_{\theta\theta}  + (w_\theta)^2 } 
&\leq \tfrac 13 |z w_{\theta\theta} | + \tfrac 83 |a w_\theta| + C t^{-\frac 12} \notag\\
&\leq ( R_3 \bb^{-\frac 32} + 3 R_7 ) t^{-1} + C t^{-\frac 12}    \notag\\
&\leq  \tfrac 12 \mm^3  t^{-1} \\
\abs{w_{tt} - w^2 w_{\theta\theta}  - 2 w  (w_\theta)^2  } 
&\leq |w| \abs{w_{t\theta} + w w_{\theta\theta}  + (w_\theta)^2 } + \tfrac 13 |z w_{t\theta}| + \tfrac 83 |a w w_\theta| + \tfrac 83 |a \p_t w| + C t^{-\frac 12} \notag\\
&\leq \tfrac 12 \mm^4 t^{-1} + 3 R_3  \mm \bb^{-\frac 32} t^{-1}  + 10  \mm R_7   t^{-1}  + C t^{-\frac 12} \notag\\
&\leq  \mm^4  t^{-1}
\end{align}
\end{subequations}
upon taking $\bar \eps$, and hence $t$, to be sufficiently small, and using \eqref{eq:b:m:ass}.
Combining the $\sc$ bounds in \eqref{eq:sc:ass} with \eqref{eq:whispers:1} and \eqref{eq:whispers:2}, we thus deduce that 
\begin{align}
\abs{\tfrac{d^2}{dt^2} w(\sc(t)^{\pm},t) - (\dot \sc - w(\sc(t)^{\pm},t))^2 w_{\theta\theta} (\sc(t)^{\pm},t) + 2 (\dot \sc - w(\sc(t)^{\pm},t)) (w_\theta(\sc(t)^{\pm},t))^2} 
\leq \tfrac 12 \mm^5  t^{-1}
\,.
\label{eq:whispers:3}
\end{align}

In a similar fashion, we may show from \eqref{eq:basic:shit} that $\p_{t\theta} \wb = - \wb \p_{\theta}^2  \wb - (\p_\theta  \wb)^2$ and that $\p_{tt} \wb = \wb^2 \p_{\theta}^2  \wb + 2 \wb (\p_\theta  \wb)^2 $, and thus, as in \eqref{eq:whispers:1}, we have that 
\begin{align}
\tfrac{d^2}{dt^2} \wb(\sc(t)^{\pm},t) - (\dot \sc - \wb(\sc(t)^{\pm},t))^2 {\wb}_{\theta\theta}(\sc(t)^{\pm},t) + 2 (\dot \sc - \wb(\sc(t)^{\pm},t)) ({\wb}_\theta (\sc(t)^{\pm},t))^2 = 0
\,.
\label{eq:whispers:4} 
\end{align}
That is, for the Burgers solution we have \eqref{eq:whispers:3} without the $\OO(t^{-1})$ error term.
In order to prove \eqref{eq:whispers:0} it remains to subtract \eqref{eq:whispers:3} and \eqref{eq:whispers:4}. We obtain that 
\begin{align}
&\tfrac{d^2}{dt^2} \left( w(\sc(t)^{\pm},t) -  \wb(\sc(t)^{\pm},t)\right)\notag\\
&= \tfrac 12 \left( (\dot \sc(t) - w(\sc(t)^{\pm},t))^2  +(\dot \sc(t) - \wb(\sc(t)^{\pm},t))^2 \right) \p_{\theta}^2  (w-\wb) (\sc(t)^{\pm},t))\notag\\
&\qquad  +(w-\wb)(\sc(t)^{\pm},t)) \left( \dot{\sc}(t) - \tfrac 12 (w+\wb)(\sc(t)^{\pm},t)) \right) \p_{\theta}^2 (w+\wb) (\sc(t)^{\pm},t)) \notag\\
&\qquad - 2 \left( \dot{\sc}(t) - \tfrac 12 (w+\wb)(\sc(t)^{\pm},t)) \right) \p_\theta  (w-\wb)(\sc(t)^{\pm},t))  \p_\theta (w+\wb)(\sc(t)^{\pm},t)) \notag\\
&\qquad - (w-\wb)(\sc(t)^{\pm},t)) \left( (w_\theta (\sc(t)^{\pm},t)) )^2 + ({\wb}_\theta (\sc(t)^{\pm},t)) )^2 \right) + \OO(t^{-1})
\label{eq:whispers:5}
\end{align}
where the $\OO(t^{-1})$ term is bounded by the right side of \eqref{eq:whispers:3}. The estimate \eqref{eq:whispers:5} is now combined with the working assumption \eqref{eq:b:m:ass}, the $\dot \sc(t) - \kappa$ bound in \eqref{eq:sc:ass}, the $\wb$ estimates established in the proof of Proposition~\ref{prop:Burgers}, the estimates \eqref{eq:w:boot}--\eqref{eq:w:dx:boot}, and the bootstrap assumption \eqref{eq:w:dxx:boot}, to arrive at 
\begin{align}
&\abs{\tfrac{d^2}{dt^2} \left( w(\sc(t)^{\pm},t) -  \wb(\sc(t)^{\pm},t)\right)} \notag\\
&\leq    (\bb^{\frac 32} t^{\frac 12} +  (2 \mm^4 + R_1) t)^2  (2  M_1   t^{-2})  
 +R_1 t (\bb^{\frac 32} t^{\frac 12} +  (2 \mm^4 + R_1) t) (6 \bb^{-\frac 32} t^{-\frac 52})\notag\\
&\qquad  + 2 (\bb^{\frac 32} t^{\frac 12} +  (2 \mm^4 + R_1) t)  R_2 (\bb t)^{-\frac 12} (2 t^{-1}) 
 + R_1 t  \left( 2 t^{-2} \right) +  \mm^5 \bb^{-\frac 32} t^{-1}\notag\\
&\leq    ( 4 \bb^3 M_1 + 9 \mm^3 + \tfrac 12 \mm^5  ) t^{-1}
 \,.
 \label{eq:whispers:6}
\end{align}
This completes the proof of the lemma, upon appealing to  \eqref{eq:b:m:ass}.
\end{proof}

\subsection{Improving the bootstrap bounds for  $k_{\theta\theta} $}
 
\begin{lemma}\label{lem:kyy}  
For all $(\theta,t)  \in \mathcal{D}^k _{\bar\eps}$ we have that 
\begin{align} 
\sabs{\p_\theta ^2 k(\theta,t) } \le 
 \mm^2 \st(\theta,t) ^{-{\frac{1}{2}} } \,.
 \label{kyy-bound1}
\end{align} 
This justifies the choice of the constant $M_3$ in \eqref{eq:M1:M4} and improves the bootstrap assumption \eqref{eq:k:dxx:boot}.
\end{lemma} 
\begin{proof}[Proof of Lemma \ref{lem:kyy}]
Differentiating \eqref{kx0}, we have that
\begin{align} 
\tfrac{d}{ds} \left( \p_\theta ^2k \circ \phi_t \ (\p_\theta  \phi_t)^2 + \p_\theta  k \circ \pt \ \p_\theta ^2 \pt \right) =0\,,  \label{kxx0}
\end{align} 
and integrating in time from $\st(\theta,t)$ to $t$, we have that for each $(\theta,t)  \in \mathcal{D}^k_{\bar\eps}$,
\begin{align} 
\p_\theta ^2 k(\theta,t)  = \p_\theta ^2 k( \sc(\st), \st) \bigl( \p_\theta  \pt(\theta, \st )   \bigr)^2 + \p_\theta  k ( \sc(\st), \st )  \p_\theta ^2 \pt(\theta , \st ) \,, \ \ 
\st = \st(\theta,t) \,.  \label{kyy-goodboy}
\end{align} 
It follows from \eqref{kx-shock} that 
\begin{align} 
\p_\theta ^2 k(\theta,t)  = \p_\theta ^2 k( \sc(\st), \st) \bigl( \p_\theta  \pt(\theta, \st )   \bigr)^2 +\tfrac{ \dot \kl (\st)) }{ \dot \sc(\st)) -   \p_s \phi_t(y, \st )}  \p_\theta ^2 \pt( y, \st ) \,, 
\label{kyy-fck-all}
\end{align} 
where $\st=\st(\theta,t)$.
Next, by differentiating the system \eqref{systemk-dt}, a lengthy computation reveals that
\begin{align} 
\p_\theta ^2 k(\sc(t),t)
&= \tfrac{\ddot \kl(t)}{(\dot\sc(t) - \lambda _2(\sc(t),t))^2} \notag\\
&\quad - \Bigl( \ddot\sc(t) - \bigl(\p_t\lambda_2(\sc(t),t)  +(2\dot\sc(t) -\lambda_2(\sc(t),t))\p_\theta  \lambda_2(\sc(t),t) \bigr) \Bigr) \tfrac{\dot \kl}{(\dot\sc(t) - \lambda _2(\sc(t),t))^3}  \,. \label{kyy-on-shock}
\end{align} 
Substitution of \eqref{kyy-on-shock} into \eqref{kyy-fck-all} shows that
for all $(\theta,t)  \in \mathcal{D} ^k_{\bar\eps}$,
\begin{align} 
\p_\theta ^2 k(\theta,t)  &= \left. \left( \tfrac{\ddot \kl}{(\dot\sc - \lambda _2)^2} - \Bigl( \ddot\sc - \bigl(\p_t\lambda_2  +(2\dot\sc -\lambda_2)\p_\theta  \lambda_2 \bigr) \Bigr) \tfrac{\dot \kl}{(\dot\sc - \lambda _2)^3} \right)\right|_{(\sc(\st(\theta,t)),\st(\theta,t))}
\bigl( \p_\theta  \pt( y, \st(\theta,t) )   \bigr)^2  \notag \\
& \qquad\qquad
+\tfrac{ \dot \kl (\st(\theta,t))) }{ \dot \sc(\st(\theta,t))) -   \p_s \phi_t(y, \st(\theta,t) )}  \p_\theta ^2 \pt( y, \st(\theta,t) ) \,.
\label{kyy-shock}
\end{align} 
Given the bounds  \eqref{eq:w:z:k:a:boot} together with \eqref{eq:b:m:ass}, 
\eqref{eq:sc:ass}, \eqref{thegoodstuff},  \eqref{eq:dt:zl:kl:on:shock}, \eqref{eq:dt:dt:zl:kl:on:shock}, \eqref{eq:FU:2}, \eqref{eq:distance:curves}, \eqref{eq:double:trouble}, \eqref{dxphipsi-bound}, and \eqref{dxxphipsi-bound-all}  we find that
\begin{align*}
\sabs{\p_\theta ^2 k(\theta,t) }
&\leq (1 + C t^{\frac 13})^2 \left( \tfrac{16}{\kappa^2} \sabs{\ddot \kl (\st(\theta,t))} + \tfrac{64}{\kappa^3} \left(6\mm^4 + \tfrac{\kappa}{3} \st(\theta,t)^{-1} \right) \sabs{\dot \kl (\st(\theta,t))} \right) 
\notag\\
&\qquad + \tfrac{9 \mm^2}{\kappa^4} \sabs{\dot \kl (\st(\theta,t))} \st(\theta,t)^{-1}+ C 
\notag\\
&
\leq 50 \bb^{\frac 92} \kappa^{-5} (1 + 10 \mm^2 \kappa^{-2}) \st(\theta,t)^{-\frac 12}   
\notag\\
&
\leq  \mm^2  \st(\theta,t)^{-\frac 12} 
\end{align*}
for all $(\theta,t)  \in \mathcal{D}^k_{\bar \eps}$. See the  details in the proof of~\eqref{eq:double:trouble:7} below for a sharper bound than the one given above.
The estimate \eqref{kyy-bound1} thus holds, concluding the proof.  
\end{proof}

\subsection{Improving the bootstrap bounds for  $w_{\theta\theta} $}
\label{sec:improve:wyy}
\begin{lemma}\label{lem:wyy}  
For all $(\theta,t)  \in \mathcal{D} _{\bar\eps}$,  we have that  
\begin{align} 
\sabs{w_{\theta\theta} (\theta,t) -\wb_{\theta\theta}(\theta,t) }& \le 
\begin{cases}
\tfrac 12 \min\{N_1, N_4\} t^{-\frac 23},  &\theta \leq  \sc_2(t) \mbox{ or } \theta \geq \sc(t) + \frac{\kappa t}{3}\\
\tfrac 12 M_1 ( t^{-2}  +  \st(\theta,t)^{-\frac{1}{2}} ) & \mbox{if }   \sc_2(t) < \theta < \sc(t)  \\
\tfrac 12 N_5 t^{-2} & \mbox{if }   \sc(t) \le \theta < \sc(t) + \frac{\kappa t}{3}
\end{cases} \,, \label{wyy-diff-bound-final}
\end{align} 
where $M_1$ is as defined as in \eqref{eq:M1:M4}, $N_1$ is given by \eqref{eq:N1:N3}, while $N_4$ and $N_5$ are defined in \eqref{eq:N4:N7}.
In particular, we have improved the bootstrap bounds \eqref{eq:w:dxx:boot}, \eqref{eq:w:dxx:boot2}, and \eqref{eq:w:dxx:boot3}.  Moreover, we have 
\begin{align} 
\sabs{w_{\theta\theta} (\theta,t) }& \le 
\begin{cases}
15\bb \kappa^{-\frac{5}{3}}t^{-\frac{5}{3}} ,  &\mbox{if } \theta \leq  \sc_2(t) \mbox{ or } \theta \geq \sc(t) + \frac{\kappa t}{3}\\
 3\bb^{-\frac{3}{2}}  t^{-\frac{5}{2}}  + 5 \mm^4 \st(\theta,t)^{-\frac{1}{2}}  & \mbox{if }   \sc_2(t) < \theta < \sc(t)  \\
 3\bb^{-\frac{3}{2}}  t^{-\frac{5}{2}}  & \mbox{if }   \sc(t) \le \theta < \sc(t) + \frac{\kappa t}{3}
\end{cases} \,. \label{wyy-bound-final}
\end{align} 
\end{lemma} 
\begin{proof}[Proof of Lemma \ref{lem:wyy}] 
Throughout this proof we will take $\bar \eps$, and hence $t$, to be sufficiently small with respect to $\kappa,  \bb, \cc$ and $\mm$. For any $(\theta,t)  \in \mathcal{D} _{\bar\eps}$, we define $x \in \Upsilon(t)$ by $x = \eta^{-1} (\theta,t) $.  

Recall that the good unknown $q^w$ is defined in \eqref{qw-def}, and it satisfies \eqref{eq:red:panda:3}. 
Differentiating \eqref{eq:red:panda:3} with respect to the label $x$, we obtain the identity 
\begin{align}
\p_\theta  q^w(\eta(x,t),t) \, \eta_x^2(x,t) 
&= - q^w(\eta(x,t)) \eta_{xx} (x,t) \notag\\
&\qquad + \p_x \left( w_0'(x) e^{-{\mathcal I}(x,0,t)} + \int_0^t Q^w(\eta(x,s),s) \eta_x(x,s) e^{-{\mathcal I}(x,s,t)} ds \right) \,.
\end{align}
Taking into account the definition of $q^w$ in \eqref{qw-def} and the identity 
$\p_{\theta}^2  \wb \circ \etab \, \etab_x^2 + \p_\theta  \wb \circ \etab \, \etab_{xx} = w_0''(x)$, we thus obtain that 
\begin{align}
&\p_\theta  q^w \circ \eta \, \eta_x^2   - \p_{\theta}^2  \wb \circ \etab \, \etab_x^2\notag\\
&\qquad= w_0'(x) \etab_x^{-1} \left(  \etab_{xx} -  \eta_{xx} \right) +  \left( \p_\theta  \wb \circ \etab   - \p_\theta  w \circ \eta   + \tfrac 14 (c k_\theta) \circ \eta \right)  \eta_{xx}   \notag\\
&\qquad\qquad + w_0''  \left(  e^{-{\mathcal I}(\cdot,0,s)} - 1 \right) - w_0'  e^{-{\mathcal I}(\cdot,0,s)} \p_x {\mathcal I}(\cdot,0,s) \notag\\
&\qquad\qquad +  \int_0^s \left( \p_\theta  Q^w \circ \eta \ \eta_x^2 + Q^w\circ \eta \  \eta_{xx}  - Q^w \circ \eta \  \eta_x \p_x {\mathcal I}(\cdot,s',s) \right)e^{-{\mathcal I}(\cdot,s',s)} ds' \,.
\label{eq:out:of:beer:0}
\end{align}
The key observation is that second line in \eqref{eq:out:of:beer:0} is precisely the first line in \eqref{eq:barf:1}, while the third line in \eqref{eq:out:of:beer:0} is precisely the third line in \eqref{eq:barf:1}; we will use this fact to avoid redundant bounds.

\vspace{.05 in}
\noindent
{\bf Bounds in the region $ \sc_2(t) < \theta < \sc(t) + \frac{\kappa t}{3}$.}
By taking into account \eqref{eq:cb:0}, \eqref{dxxeta-diff},  \eqref{thechicken}, \eqref{first-fun-day}, \eqref{eq:w:z:k:a:boot} and \eqref{eq:R:def}, we obtain that 
\begin{align}
&\mbox{the first line on RHS of } \eqref{eq:out:of:beer:0} \notag\\
&\qquad \leq 40 \mm t^{-\frac 32}  + \left( 8 R_1 \bb^{-\frac 32} t^{-\frac 12}   + R_2 (\bb t)^{-\frac 12} + \tfrac 14 \mm R_6 t^{\frac 12} \right) \left(\tfrac 13 \bb^{-\frac 32} t^{-\frac 32} +  20 \mm t^{-\frac 12} \right) 
\notag\\
&\qquad \leq  \mm^3 (\bb t)^{-2} 
\,,
\label{eq:rock:you:1}
\end{align}
since $t$ is sufficiently small.
Next, since second line in \eqref{eq:out:of:beer:0} equals the first line in \eqref{eq:barf:1}, from \eqref{eq:barf:2},  \eqref{first-fun-day}, and the fact that $t > \nu^\sharp(x,t)$, we obtain
\begin{align}
\mbox{the second line on RHS of } \eqref{eq:out:of:beer:0}
\leq 24 \mm t |w_0''(x)|  + 2 (\mm^{\frac 12} + \mm M_3) t^{\frac 12} |w_0'(x)| 
\leq 10 \mm (\bb t)^{-\frac 32}\,.
\label{eq:rock:you:2}
\end{align}
Similarly, since third line in \eqref{eq:out:of:beer:0} equals the third line in \eqref{eq:barf:1}, from \eqref{eq:barf:4},  \eqref{first-fun-day}, and the fact that $t > \nu^\sharp(x,t)$, we obtain
\begin{align}
\mbox{the third line on RHS of } \eqref{eq:out:of:beer:0}
\leq C t^{-\frac 34} \,.
\label{eq:rock:you:3}
\end{align} 
By adding \eqref{eq:rock:you:1}, \eqref{eq:rock:you:2}, and \eqref{eq:rock:you:3}, since $t$ is sufficiently small we deduce that 
\begin{align}
\sabs{\p_\theta  q^w \circ \eta \, \eta_x^2   - \p_{\theta}^2  \wb \circ \etab \, \etab_x^2} 
\leq 2 \mm^3 (\bb t)^{-2}\,.
\label{eq:out:of:beer:2}
\end{align}

Next, by recalling the definition  of $q^w$  in \eqref{qw-def}, and appealing to \eqref{eq:w:z:k:a:boot}, \eqref{eq:R:def}, and \eqref{kyy-bound1}, we deduce   
\begin{align}
\sabs{w_{\theta\theta}  \circ \eta (\eta_x)^2 - {\wb}_{\theta\theta} \circ \etab ({\etab}_x)^2 }
\leq  3 \mm^3 (\bb t)^{-2} + \mm^3 \st(\theta,t)^{-\frac 12} \,.
\label{eq:out:of:beer:22}
\end{align}
With \eqref{eq:out:of:beer:22} in hand, we use the notation introduced in \eqref{first-good-day2} to rewrite
\begin{align}
w_{\theta\theta} (\theta,t)  - \wb_{\theta\theta}(\theta,t) 
&=  \eta_x^{-2} \left( w_{\theta\theta}  \circ \eta (\eta_x)^2 - {\wb}_{\theta\theta} \circ \etab ({\etab}_x)^2 \right) - \eta_x^{-2} {\mathcal K}_3 - \eta_x^{-2} {\mathcal K}_1 
\,,
\label{eq:out:of:beer:3}
\end{align}
and thus we may combine \eqref{eq:cb:0}, \eqref{glassy-day0}, and \eqref{the-star-chicken}, 
to arrive at
\begin{align}
\sabs{w_{\theta\theta} (\theta,t)  - \wb_{\theta\theta}(\theta,t) } \leq 5 \mm^4 t^{-2}   + 2 \mm^3 \st(\theta,t)^{-\frac 12}
\end{align}
since $\mm$ is large compared to $\bb$.
The above estimate proves the second and third bounds in \eqref{wyy-diff-bound-final} once we ensure that 
 $\tfrac 12 M_1 \geq 5 \mm^4  $  and  $\tfrac 12 N_5 \geq 5 \mm^4$. These conditions hold in view of the definitions \eqref{eq:M1:M4} and \eqref{eq:N4:N7}.

\vspace{.05 in}
\noindent
{\bf Bounds in the region $\theta \leq \sc_2(t)$ or $\theta \ge \sc(t) + \frac{\kappa t}{3}$.}
In order to estimate the first line on the right side of \eqref{eq:out:of:beer:0}, we rewrite
\begin{align}
w_\theta \circ \eta - \wb_{\theta} \circ \eta = \eta_x^{-1} \left( w_\theta \circ \eta \, \eta_x - \wb_{\theta}\circ \etab \, \etab_x \right) -  \eta_x^{-1} \wb_{\theta} \circ \etab \left(\eta_x - \etab_x  \right)
\end{align}
so that from the second equality in \eqref{nopork1}, \eqref{eq:cb:0}, \eqref{eq:want:to:vomit:2aa}, \eqref{eq:want:to:vomit:2a}, \eqref{eq:want:to:vomit:2b}, and \eqref{dxeta-diff}, we have
\begin{align}
\sabs{(w_\theta \circ \eta - \wb_{\theta} \circ \eta)(x,t)} 
&\leq  40 \mm t |w_0'(x)| + \mm^2 t^{\frac 12} + C t + 200 \mm |w_0'(x)| t^{\frac 43} \notag\\
&\leq  50 \mm t |w_0'(x)| + 2 \mm^2 t^{\frac 12} \,.
\end{align}
Thus, analogously to \eqref{eq:rock:you:1}, using \eqref{eq:cb:0},  \eqref{dxxeta-diff},  \eqref{first-fun-dayL}, and the fact that $k(\theta,t)  = 0$, we have
\begin{align}
\mbox{the first line on RHS of } \eqref{eq:out:of:beer:0} 
& \leq   20  |w_0'(x)| \mm t^{ \frac 13}  + \left(  50 \mm t |w_0'(x)| + 2 \mm^2 t^{\frac 12} \right) \left(4 \bb \kappa^{-\frac 53} t^{-\frac 23} +  10 \mm t^{ \frac 13} \right) 
\notag\\
& \leq  C t^{- \frac 13}  
\,.
\label{eq:rock:you:1a}
\end{align}
Next, similarly to \eqref{eq:rock:you:2} we have that 
\begin{align}
\mbox{the second line on RHS of } \eqref{eq:out:of:beer:0}
&\leq 24 \mm t |w_0''(x)|  + 2 (\mm^{\frac 12} + \mm M_3) t^{\frac 12} |w_0'(x)| \notag\\
&\leq  96 \mm \bb \kappa^{-\frac 53} t^{-\frac 23} + C t^{-\frac 16}\,.
\label{eq:rock:you:2a}
\end{align}
As in \eqref{eq:rock:you:3}, but this time using that $\nu^\sharp(x,t)  \geq t$, we obtain from the first line in  \eqref{eq:barf:4} that 
\begin{align}
\mbox{the third line on RHS of } \eqref{eq:out:of:beer:0}
\leq C  \,.
\label{eq:rock:you:3a}
\end{align} 
By adding \eqref{eq:rock:you:1a}, \eqref{eq:rock:you:2a}, and \eqref{eq:rock:you:3a}, using that $k(\eta(x,s),s) = 0$, since $t$ is sufficiently small we deduce   
 \begin{align}
\sabs{w_{\theta\theta}  \circ \eta (\eta_x)^2 - {\wb}_{\theta\theta} \circ \etab ({\etab}_x)^2 }
\leq  20 \mm t^{-\frac 23}  \,.
\label{eq:out:of:beer:22a}
\end{align}
Here we have also used \eqref{eq:b:m:ass}.
Finally, using the decomposition \eqref{eq:out:of:beer:3}, and appealing to the bounds \eqref{glassy-day0L} and \eqref{the-star-chickenL}
we deduce that 
\begin{align}
\sabs{w_{\theta\theta} (\theta,t)  - \wb_{\theta\theta}(\theta,t) } \leq  2\mm^4 t^{-\frac 23}\,.
\end{align}
The above estimate proves the first bound  in \eqref{wyy-diff-bound-final} once we ensure that 
 $\tfrac 12 \min\{ N_1, N_4\} \geq 2\mm^4 $. This condition  holds in view of the definitions \eqref{eq:N1:N3} and \eqref{eq:N4:N7}.
 
In order to complete the proof of the lemma, we note that \eqref{wyy-bound-final} follows from \eqref{wyy-diff-bound-final}, the triangle inequality, and \eqref{thegoodstuff2}\,.
\end{proof}

\begin{lemma} \label{lem:dyqw}
Recall  the definition of $q^w$ in \eqref{qw-def}. 
For all $(\theta,t)  \in \mathcal{D} _{\bar \eps}^k$ such that  $\sc_2(t) < \theta < \sc_2(t) + \tfrac{\kappa t}{6}$, we have that
\begin{align} 
\sabs{ q^w_\theta(\theta,t) } &\le 3 \bb (\kappa t)^{-\frac 53}  \,. \label{vladneeds2}
\end{align}
\end{lemma} 
\begin{proof}[Proof of Lemma \ref{lem:dyqw}] 
Combining \eqref{eq:out:of:beer:2} with \eqref{eq:u0:ass:quant}, \eqref{eq:pxx:wb:def}, \eqref{eq:cb:0},   \eqref{eq:finally:useful:3} (with $s=0$), and \eqref{eq:distance:curves} we deduce that $|\eta^{-1}(\theta,t) |\geq \frac{\kappa t}{7}$ and thus
\begin{align}
 \sabs{ q^w_\theta(\theta,t) }
 &\leq 2 \sabs{\wb_{\theta\theta}(\etab(\eta^{-1}(\theta,t) ,t),t)} + 2 \mm^3 (\bb t)^{-2} \notag\\
 &\leq 10 \sabs{w_0''(\eta^{-1}(\theta,t) )} +  2 \mm^3 (\bb t)^{-2} \notag\\
 &\leq 3 \bb (\kappa t)^{-\frac 53}\,.
\end{align}
The bound \eqref{vladneeds2} is thus proven.
\end{proof} 

\subsection{Improving the bootstrap bounds for  $z_{\theta\theta} $}
\label{sec:zyy}	
Just as we defined the function $q^w(\theta,t) $ in \eqref{qw-def}, we introduce the  function 
\begin{align} 
q^z(\theta,t)  & = z_\theta(\theta,t)  + \tfrac{1}{4}   c(\theta,t)  k_\theta (\theta,t)  \,. \label{qz-def} 
\end{align} 
Using this unknown, we rewrite the equation \eqref{prelim-dtwz}  as 
\begin{align} 
\tfrac{d}{ds} (q^z \circ \psi_t \p_\theta  \psi_t ) 
& = -Q^z \circ\! \psi_t  \ \p_\theta  \psi_t  \,, \label{qz-eqn}
\end{align} 
where
\begin{align} 
Q^z =   c k_\theta (\tfrac{1}{12} w_\theta + \tfrac{1}{12} z_\theta +  \tfrac{2}{3}  a) + \tfrac{8}{3} \p_\theta (a z)
 \,. \label{Qz-def}
\end{align} 
Differentiating \eqref{qz-eqn}, we have that
\begin{align} 
\tfrac{d}{ds} (q^z_{\theta} \circ \psi_t (\p_\theta  \psi_t)^2 + q^z \circ \psi_t \p_\theta ^2 \psi_t  ) 
& = 
- \p_\theta Q^z\! \circ\! \psi_t  \ (\p_\theta  \psi_t)^2
-Q^z\! \circ\! \psi_t  \ \p_\theta ^2 \psi_t \,,   \label{dyyqz-eqn} 
\end{align} 
which may be integrated on $[\stt(\theta,t),t]$ to obtain that 
\begin{align} 
q^z_{\theta} (\theta,t)   = \left. \bigl(q^z_{\theta}(\p_\theta  \psi_t)^2 + q^z \p_\theta ^2 \psi_t  \bigr) \right|_{ (\sc(\stt),\stt)} 
- \int_{\stt(\theta,t)}^t \Bigl( \p_\theta Q^z\! \circ\! \psi_t  \ (\p_\theta  \psi_t)^2+Q^z\! \circ\! \psi_t  \ \p_\theta ^2 \psi_t \bigr) ds \,. \label{dyyqz-identity}
\end{align} 
for all $(\theta,t)  \in \mathcal{D}^z_{\bar \eps}$. Here we have used that $\pst(\theta,\stt(\theta,t)) = \sc(\stt(\theta,t))$ and the fact that $\pst(\theta,t) = \theta$, which implies $\p_\theta  \pst(\theta,t) = 1$ and $\p_{\theta}^2  \pst(\theta,t) = 0$. In order to estimate the right side of \eqref{dyyqz-identity}, we first establish:
\begin{lemma} \label{lem:for-qz-stuff}
For $(\theta,t)  \in \mathcal{D} ^z _{\bar \eps}$, 
\begin{align} 
\int_{\stt(\theta,t)}^t \Bigl| \p_\theta Q^z\! \circ\! \psi_t  \ (\p_\theta  \psi_t)^2+Q^z\! \circ\! \psi_t  \ \p_\theta ^2 \psi_t \bigr| ds 
\le 3 \mm^3     \stt(\theta,t)^ {-\frac{1}{2}}  \,.
\label{lower-the-bridge1}
\end{align} 
\end{lemma} 
\begin{proof}[Proof of Lemma \ref{lem:for-qz-stuff}]
We decompose $\p_\theta  Q^z =  \mathcal{Q}_1 + \mathcal{Q} _2$, where
 \begin{align*} 
 \mathcal{Q}_1 & = \overbrace{\tfrac{1}{12}   c k_\theta w_{\theta\theta}  }^{ \mathcal{Q} _{1a}}+ 
 \overbrace{\tfrac{1}{24}   k_\theta w_\theta w_\theta }^{ \mathcal{Q} _{1b}} + \overbrace{  c k_{\theta\theta}  ( \tfrac{1}{12}  w_\theta + \tfrac{1}{12} z_\theta +\tfrac{2}{3}  a)}^{ \mathcal{Q} _{1c}} \,, \\
 \mathcal{Q}_2 & = c k_\theta( \tfrac{1}{12}  z_{\theta\theta}  +  \tfrac{2}{3}   a_\theta ) + \tfrac{8}{3} (az)_{\theta\theta} 
 + \tfrac{1}{24} k_\theta z_\theta w_\theta  \,.
 \end{align*} 
For $(\theta,t) \in \DD_{\bar \eps}^z \setminus \overline{\DD_{\bar \eps}^k}$, is convenient to introduce a time $\stt_1(\theta,t) $, which is defined as the time at which the curve $\pst(\theta,\cdot)$ intersects the curve $\sc_2$; recall that $\stt(\theta,t)$ is the time at which $\pst(\theta,\cdot)$ intersects the shock curve $\sc$. From \eqref{eq:distance:curves}, \eqref{dxphipsi-bound}, and the definitions of $\stt$ and $\stt_1$, we note that 
\begin{align}
\stt_1(\theta,t)  = 2 \stt(\theta,t) + \OO(\stt(\theta,t)^{\frac 43}) \,. 
\label{eq:sanity:2}
\end{align}
When $(\theta,t) \in \DD_{\bar \eps}^k$, we abuse notation and write $\stt_1(\theta,t)  = t$, emphasizing that  $\pst(\theta,\cdot)$ does not intersect $\sc_2$.
By definition, note that for $s \in (\stt_1(\theta,t) ,t]$, all the terms in $ \mathcal{Q}_1 \circ \pst$ and $ \mathcal{Q}_2\circ \pst$ vanish.

\begin{figure}[htb!]
\centering
\begin{tikzpicture}[scale=2.25]
\draw [<->,thick] (0,2.4) node (yaxis) [above] { } |- (3.5,0) node (xaxis) [right] { };
\draw[black,thick] (0,0)  -- (-2,0) ;
\draw[black,thick] (-2,2)  -- (3.5,2) ;
\draw[name path=A,red,ultra thick] (0,0) .. controls (1,0.5) and (2,1.5) .. (3,2);
\draw[red] (3,1.9) node { $\sc$}; 
\draw[name path = B,blue,ultra thick] (0,0) .. controls (0.5,0.85) and (1,1.25) .. (1.5,2) ;
\draw[blue] (1.3,1.9) node { $\sc_2$}; 
\draw[green!40!gray, ultra thick] (0,0) .. controls (.2,1) and (0.4,1.5) .. (.5,2) ;
\draw[green!40!gray] (.4,1.9) node { $\sc_1$}; 
\draw[black] (-2.4,0) node { $t=0$}; 
\draw[black] (-2.4,2) node { $t = \bar \eps$}; 

\draw[black,dotted] (1,0)  -- (1,1.7);
\draw[black] (-0.1,1.7) node { $t$}; 
\draw[black,dotted] (0,1.7)  -- (1,1.7);
\draw[black] (1,-0.1) node { $\theta$}; 
\filldraw[black] (1,1.7) circle (0.5pt);

\draw[green!40!gray] (.72,1) .. controls (.85,1.35) .. (1,1.7) ;
\draw[green!40!gray] (.35,0) .. controls (.46,0.33) and (.57,0.66) .. (.72,1) ;
\filldraw[green!40!gray] (.72,1) circle (0.5pt);
\filldraw[green!40!gray] (.43,0.24) circle (0.5pt);
\draw[green!40!gray,dotted,thick] (.44,.24)  -- (0,.24);
\draw[green!40!gray,dotted,thick] (.71,1)  -- (0,1);
\draw[green!40!gray] (-.3,1) node {$\stt_1(\theta,t) $}; 
\draw[green!40!gray] (-.3,0.24) node {$\stt(\theta,t)$}; 
\draw[dotted,->,green!40!gray] (2,0.5) -- (0.63,0.69);
\draw[green!40!gray] (2.3,0.5) node { $\pst(\theta,s)$}; 

\end{tikzpicture}

\vspace{-0.2cm}
\caption{\footnotesize Fix a point $(\theta,t) $ which lies in between $\sc_1$ and $\sc_2$. The intersection time of $\pst(\theta,s)$ with $\sc_2$ is denoted by $\stt_1(\theta,t) $, while the intersection time with $\sc$ is denoted as usual by $\stt(\theta,t)$.}
\label{fig:many:characteristics:4}

\end{figure}
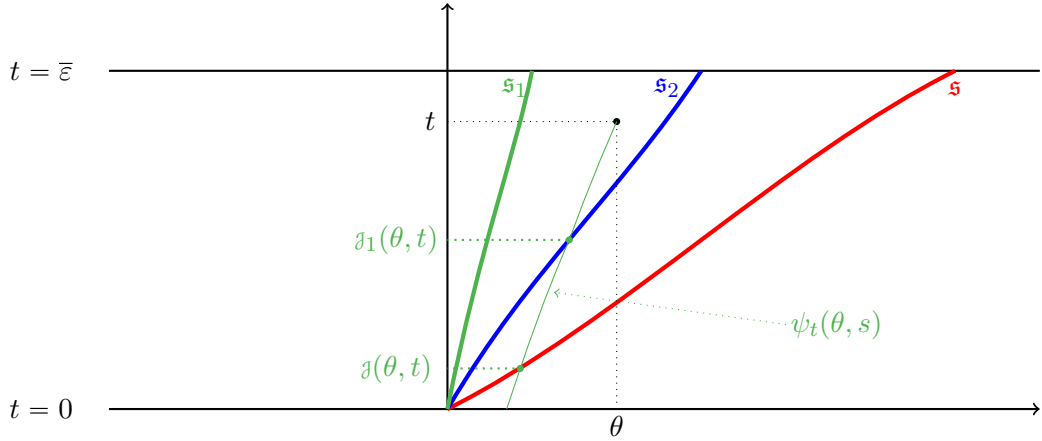

Let us thus consider first the case $(\theta,t) \in \DD_{\bar \eps}^k$. From \eqref{eq:metal:2} we have that
\begin{align} 
 \sabs{k_\theta( \pst(\theta,s),s)  } \le 
200 \bb^{\frac 92} \kappa ^{-4} \st(\pst(\theta,s),s)^ {\frac{1}{2}}  \,.  \label{good-kyflow-bound}
\end{align}  
for all $s\in [\stt(\theta,t),t]$. 
Thus, using \eqref{good-kyflow-bound} together with \eqref{eq:Steve:needs:this:2},  \eqref{eq:w:dxx:boot}, and the fact that $\st(\pst(\theta,s),s) \leq \stt(\pst(\theta,s),s) = \stt(\theta,t)$, we have that
\begin{align} 
\int_{\stt(\theta,t)}^t \sabs{\mathcal{Q} _{1a} \circ \pst} (\p_\theta  \psi_t)^2 ds
& \le  \tfrac{1}{4} (1+ C t^ {\frac{1}{3}} )     \mm\left( \int_{\stt(\theta,t)}^t \sabs{k_\theta (w_{\theta\theta} -\wb_{\theta\theta}) \circ \pst}  ds +   \int_{\stt(\theta,t)}^t \sabs{k_\theta \wb_{\theta\theta} \circ \pst}  ds\right)
\notag  \\
& \le     60 \mm    \bb^{\frac 92} \kappa^{-4}\left(  \int_{\stt(\theta,t)}^t  \bigl( M_1 + \stt(\theta,t)^ {\frac{1}{2}} s^{-2}\bigr) ds + 20 \kappa ^{-1} \stt(\theta,t) ^{-\frac 12} \right)
\notag \\
&  \le  40  \mm  \stt(\theta,t)^{-\frac{1}{2}} 
 \,.\label{fof3}
\end{align}
In the last inequality we have taken $t$ to be sufficiently small, and have used \eqref{eq:b:m:ass}.

Next, using  \eqref{eq:double:trouble:a}, \eqref{eq:w:dx:boot}, \eqref{dxphipsi-bound},  \eqref{eq:Steve:needs:this:1},  and \eqref{good-kyflow-bound},  we have that
\begin{align} 
\int_{\stt(\theta,t)}^t \sabs{\mathcal{Q} _{1b} \circ \pst} (\p_\theta  \psi_t)^2 ds 
& \le C  \stt(\theta,t)^ {-\frac{1}{2}}  \int_{\stt(\theta,t)}^t ( \sabs{ (w_\theta -\wb_{\theta}) \circ \pst} + \sabs{\wb_{\theta} \circ \pst} )ds \notag\\
&\le C t^ {\frac{1}{3}} \stt(\theta,t)^ {-\frac{1}{2}}  
 \,\label{fof4}
\end{align}
and with \eqref{eq:w:z:k:a:boot}, \eqref{eq:k:dxx:boot}, and  \eqref{stt-integral}, 
\begin{align} 
\int_{\stt(\theta,t)}^t \sabs{\mathcal{Q} _{1c} \circ \pst} (\p_\theta  \psi_t)^2 ds 
& \le 2 \mm \left(\tfrac{1}{12} \stt(\theta,t)^{-1} + \tfrac{1}{12} R_4 t^{\frac 12} + \tfrac 23 R_7 \right)\int_{\stt(\theta,t)}^t  \sabs{k_{\theta\theta}  \circ \pst} ds \notag \\
& \le \tfrac 13 \mm M_3 \stt(\theta,t)^ {-1} \int_{\stt(\theta,t)}^t   \st(\pst(\theta,s),s)^ {-\frac{1}{2}} ds \notag\\
& \le  2\mm^3 \stt(\theta,t)^{-\frac{1}{2}}  \,.
\label{fof5}
\end{align}
In the last inequality we have  taken into account the definition of $M_3$ in \eqref{eq:M1:M4},

Note that if $\theta \in (\sc_1(t), \sc_2(t))$ then the integrals in \eqref{fof3}, \eqref{fof4}, and \eqref{fof5} range  from $\stt(\theta,t)$ up to $\stt_1(\theta,t)  < t$, but this has no effect on the  bounds established  in \eqref{fof3}, \eqref{fof4}, and \eqref{fof5}.

Returning to our decomposition of $\p_\theta  \mathcal{Q}^z$ as $\mathcal{Q}_1 + \mathcal{Q}_2$, we note that  by the same bounds and arguments as above, and by appealing also to \eqref{eq:metal:3}, we also have that
\begin{align} 
\int_{\stt(\theta,t)}^t \sabs{\mathcal{Q} _{2} \circ \pst} (\p_\theta  \psi_t)^2 ds 
&\leq 2 \int_{\stt(\theta,t)}^t \left(3 R_7 \sabs{z_{\theta\theta}  \circ \pst} + C \stt(\theta,t)^{\frac 32} \sabs{a_{\theta\theta}  \circ \pst} + C \right) ds   \notag\\
&  \le C  \stt(\theta,t)^ {\frac{1}{2}}  + C t \stt(\theta,t)^ {-\frac{1}{2}} + C \stt(\theta,t)^{\frac 32} \log \tfrac{t}{\stt(\theta,t)} + C t 
\notag\\
&\leq C t \stt(\theta,t)^ {-\frac{1}{2}}
 \,.\label{fof6}
\end{align}
We note that for the bounds \eqref{fof3}--\eqref{fof6}, we have taken $\bar \eps$ sufficiently small.

Lastly, from \eqref{dxxphipsi-bound-all} we have that for $(\theta,t) \in \DD_{\bar \eps}^k$ we have $\abs{\p_\theta ^2 \pst (\theta,s) }  \le C s^{-1}$ so with the definition of $Q^z$ in \eqref{Qz-def} and the bounds \eqref{eq:Steve:needs:this:1}, \eqref{eq:w:z:k:a:boot},  and \eqref{good-kyflow-bound}, 
\begin{align} 
\int_{\stt(\theta,t)}^t \sabs{Q^z\! \circ\! \psi_t  \ \p_\theta ^2 \psi_t } ds 
&  \le  C t^ {\frac{1}{3}} \stt(\theta,t)^ {-\frac{1}{2}}
 \,.\label{fof7}
\end{align}
On the other hand, for $(\theta,t) \in \DD_{\bar \eps}^z \setminus \overline \DD_{\bar \eps}^k$, 
we have that $\abs{\p_\theta ^2 \pst (\theta,s) }  \le C \stt(\theta,t)^{-1}$ for $s\in [\stt(\theta,t),\stt_1(\theta,t) ]$ and hence using \eqref{eq:sanity:2} 
\begin{align} 
\int_{\stt(\theta,t)}^t \sabs{Q^z\! \circ\! \psi_t  \ \p_\theta ^2 \psi_t } ds 
&\leq \int_{\stt(\theta,t)}^{\stt_1(\theta,t) } \sabs{Q^z\! \circ\! \psi_t  \ \p_\theta ^2 \psi_t } ds  + \int_{\stt_1(\theta,t) }^t \sabs{Q^z\! \circ\! \psi_t  \ \p_\theta ^2 \psi_t } ds \notag\\
&  \le  C  \stt(\theta,t)^ {-\frac{1}{6}} + C t^ {\frac{1}{3}} \stt(\theta,t)^ {-\frac{1}{2}} \notag\\
&\leq C t^ {\frac{1}{3}} \stt(\theta,t)^ {-\frac{1}{2}}
 \,.\label{fof8}
\end{align}
Combining the bounds \eqref{fof3}--\eqref{fof7}, and taking $\bar \eps$ sufficiently small, we obtain the inequality \eqref{lower-the-bridge1}.
\end{proof}

\begin{lemma}\label{lem:zyy}  
For all $(\theta,t)  \in \mathcal{D}^z _{\bar\eps}$ we have the bounds
\begin{align} 
\sabs{z_{\theta\theta} (\theta,t) }& \le 
\begin{cases}
\tfrac 12 M_2 \st(\theta,t)^{-{\frac{1}{2}} } ,  &\mbox{if }  (\theta,t)  \in \mathcal{D} ^k_{\bar\eps}\\
\tfrac 12 N_2 \stt(\theta,t)^{-{\frac{1}{2}} } & \mbox{if }   (\theta,t)  \in \mathcal{D} ^z_{\bar\eps}\setminus\overline{\mathcal{D}^k_{\bar\eps}} 
\end{cases} \,, \label{zyy-bound-final}
\end{align} 
where $M_2$ and $N_2$ are defined in \eqref{eq:M1:M4}, respectively in \eqref{eq:N1:N3}.
Thus, the bootstrap assumptions \eqref{eq:z:dxx:boot} and \eqref{eq:z:dxx:boot2} are improved.
Moreover, the quantity $q^z$ defined in \eqref{qz-def} satisfies the bound
\begin{align}
\label{eq:lossy:qy:dz:bound} 
\sabs{q^z_\theta(\theta,t) } \leq 4 \mm^3 \stt(\theta,t)^{-\frac 12}
\end{align}
for all $\sc_2(t) < \theta  < \sc_2(t) + \tfrac{\kappa t}{6}$.
\end{lemma}

\begin{proof}[Proof of Lemma \ref{lem:zyy}]
 Using \eqref{dyyqz-identity} and  the definition of $q^z$ in \eqref{qz-def}, we see that for all $(\theta,t)  \in \mathcal{D}^z_{\bar\eps}$ and with $\stt = \stt(\theta,t)$, we have 
\begin{align} 
z_{\theta\theta} (\theta,t) & = \overbrace{\left. \bigl(  z_{\theta\theta} (\p_\theta  \pst)^2+ z_{\theta} \circ \psi_t \p_\theta ^2 \psi_t  \bigr)\right|_{(\sc(\stt),\stt)}}^{ \mathcal{H}_1}
+\overbrace{ \left. \tfrac{1}{4} \bigl(  c k_{\theta\theta}  (\p_\theta  \pst)^2 + c_\theta k_{\theta} (\p_\theta  \pst)^2 + c k_\theta \p_\theta ^2\pst \bigr)\right|_{(\sc(\stt),\stt)}}^{ \mathcal{H} _2} \notag \\
& \qquad - \underbrace{\tfrac{1}{4}   (c k_{\theta\theta}  + c_\theta k_\theta)(\theta,t) }_{\mathcal{H} _3}
- \int_{\stt(\theta,t)}^t \Bigl( \p_\theta Q^z\! \circ\! \psi_t  \ (\p_\theta  \psi_t)^2+Q^z\! \circ\! \psi_t  \ \p_\theta ^2 \psi_t \bigr) ds
\,.   \label{zyy-identity}
\end{align}

In order to get a good bound for the term $\mathcal{H} _1$ in \eqref{zyy-identity} on the shock curve, 
it remains for us to express $z_{\theta\theta} (\sc(\stt(\theta,t)), \stt(\theta,t))$ in terms of derivatives of functions along the shock curve.
Differentiating the system \eqref{systemz-dt}, taking into account the identity $\frac{d}{dt} (f (\sc(t),t)) = ((\p_t + \dot \sc \p_\theta ) f)(\sc(t),t)$, and the formulas
\begin{align*} 
\p_t (c^2 k_\theta) + \lambda_2 \p_\theta  (c^2 k_\theta)  &=   - 2 (  \p_\theta  \lambda_2 + \tfrac 83 a) (c^2 k_\theta) \,, \\
\p_{t\theta} z&= - \lambda_1 z_{\theta\theta}   - \p_\theta  \lambda_1 z_\theta - \tfrac 83 (a z)_\theta + \tfrac 16 \p_\theta  (c^2 k_\theta) 
\end{align*} 
which are direct consequences of \eqref{xland-z}, \eqref{xland-k}, and \eqref{xland-sigma}, after a straightforward but lengthy computation we arrive at 
\begin{align}
\ddot{z_-} 
&= (\p_t + \dot \sc \p_\theta )^2 z - (\p_t + \dot \sc\p_\theta ) (\p_t + \lambda_1 \p_\theta ) z - (\p_t + \dot \sc \p_\theta ) ( \tfrac 16 c^2 k_\theta - \tfrac 83 az ) \notag\\
&= (\dot \sc - \lambda_1) z_{t\theta} + \dot \sc (\dot \sc -\lambda_1) z_{\theta\theta}  +  z_\theta (\ddot{\sc}  - \p_t \lambda_1 - \dot \sc \p_\theta  \lambda_1) +  (\p_t + \dot \sc \p_\theta ) ( \tfrac 16 c^2 k_\theta - \tfrac 83 az )\notag \\
&= (\dot \sc -\lambda_1)^2 z_{\theta\theta}  - (\dot \sc - \lambda_1) \left( \p_\theta  \lambda_1 z_\theta + \tfrac 83 (a z)_\theta - \tfrac 16 \p_\theta  (c^2 k_\theta)\right) \notag \\
&\quad  +    z_\theta \left(\ddot{\sc} + \tfrac 13 \lambda_3 w_\theta + \lambda_1 z_\theta - \tfrac 29 c^2 k_\theta + \tfrac 89 a w + \tfrac 83 a z - \dot \sc  (\tfrac{1}{3} w_\theta + z_\theta) \right) 
\notag\\
&\quad - \tfrac 13  (\p_\theta  \lambda_2 + \tfrac 83 a) ( c^2 k_\theta ) + \tfrac 16 (\dot \sc -\lambda_2) \p_\theta  (c^2 k_\theta) \notag \\ 
&\quad -  \tfrac 83 a \left( (\dot \sc - \lambda_1) z_\theta - \tfrac 83 a z + \tfrac 16 c^2 k_\theta \right) -  \tfrac 83 z \left( (\dot \sc - \lambda_2) a_\theta - \tfrac 43 a^2 + \tfrac 16 (w^2 + z^2) + w z \right)\notag  \\ 
&= (\dot \sc -\lambda_1)^2 z_{\theta\theta}  
+ \tfrac 13 \left(  \dot \sc - \tfrac 12 w - \tfrac 56 z\right) c^2 k_{\theta\theta}  \notag \\
&\quad + \tfrac 13 \left(   (  \dot \sc -\tfrac 56  w -  \tfrac 12 z ) c  k_\theta   -   (2 \dot \sc -  \tfrac 43 w - \tfrac 43 z  ) z_\theta  \right) w_\theta  + {\mathcal R}_{z_{\theta\theta} }
\label{eq:z:yy:proto}
\end{align}
where we have denoted the remainder term ${\mathcal R}_{z_{\theta\theta} }$ by 
\begin{align}
 {\mathcal R}_{z_{\theta\theta} }
 &:=   z_\theta \left(\ddot{\sc} + 2 (\lambda_1- \dot \sc) z_\theta  - \tfrac 16 (2 \dot \sc  + \tfrac 13 w - 3 z)   c   k_\theta  - \tfrac 83 a (2 \dot \sc - 3 \lambda_1)    \right) 
\notag \\
&\quad -  \tfrac 83 a \left(\tfrac 12 c^2 k_\theta  - \tfrac 83 a z    \right) -  \tfrac 83 z \left( (2 \dot \sc - \lambda_2 - \lambda_1 ) a_\theta - \tfrac 43 a^2 + \tfrac 16 (w^2 + z^2) + w z \right)
\label{eq:R:z:yy}
\end{align}
At this stage we note that the reason we call the term ${\mathcal R}_{z_{\theta\theta} }$ a remainder term is as follows; from \eqref{eq:sc:ass}, \eqref{eq:w:z:k:a:boot}, and the properties of $\wb$, we may directly show that 
\begin{align}
 \sabs{{\mathcal R}_{z_{\theta\theta} }(\sc(t),t)} \leq C t^{\frac 12}
 \label{eq:R:z:yy:bnd}
\end{align}
for a suitable constant $C = C(\kappa,\bb,\cc,\mm)>0$. In comparison, the remaining terms in \eqref{eq:z:yy:proto} will be shown to be $\OO(t^{-\frac 12})$, so that ${\mathcal R}_{z_{\theta\theta} }$ is negligible.

The identities \eqref{eq:z:yy:proto} and \eqref{eq:R:z:yy} are valid at any point $(\sc(t),t)$ on the shock curve, so in particular at $(\sc(\stt),\stt)$.
Hence, we see that
\begin{align} 
 z_{\theta\theta}  (\sc( \stt),\stt)
& =\frac{\ddot\zl - \tfrac 13 ( \dot \sc - \tfrac 12 w - \tfrac 56 z) (c^2 k_{\theta\theta} )  - \tfrac 13 \left(  ( \dot \sc - \tfrac 56 w - \tfrac 12 z ) c k_\theta -  (2 \dot \sc -  \tfrac 43 w - \tfrac 43 z  )  z_\theta  \right) w_\theta }{ (\dot \sc - \lambda_1)^2}  \Bigr|_{(\sc( \stt),\stt)}
\notag
\\ &  
\qquad - \frac{ {\mathcal R}_{z_{\theta\theta} }}{(\dot \sc - \lambda_1)^2} \Bigr|_{(\sc( \stt),\stt)}  \,.  \label{zyy-shock}
\end{align}

By combining \eqref{zyy-shock}
with \eqref{kx-shock} (in which we replace $\st$ with $\stt$), \eqref{zx-good}, \eqref{kyy-on-shock} (with $t$ replaced by $\stt$), and the estimates \eqref{eq:zl:and:kl:on:shock:L:infinity}, \eqref{eq:dt:zl:kl:on:shock}, \eqref{eq:dt:dt:zl:kl:on:shock}, \eqref{eq:the:burgers:jumps}, \eqref{eq:w:z:k:a:boot}, \eqref{dxphipsi-bound},   \eqref{kyy-bound1},  and taking $\bar \eps$ sufficiently small, we find that
 \begin{align} 
\sabs{ z_{\theta\theta}  (\sc( \stt),\stt)\p_\theta \pst(\sc( \stt),\stt)^2} 
&\leq 
3 \kappa^{-2}  ( 4 \bb^{\frac 92} \kappa^{-2}  +   \kappa^3  \mm^2 +  (\kappa^2 R_6 + \kappa R_4)  ) \stt^{-\frac 12} + C \stt^{\frac 12} 
 \le 6 \kappa \mm^2 \stt^{-\frac 12} \,.
\label{zyy-on-shock}
\end{align} 
On the other hand, from \eqref{eq:w:z:k:a:boot} and \eqref{dxxphipsi-bound-all}, 
\begin{align} 
 \abs{z_\theta(\sc( \stt),\stt)\p_\theta ^2 \pst(\sc( \stt),\stt) } & \le  R_4 \stt^{\frac 12} \mm^ {\frac{1}{2}}  \kappa^{-\frac{3}{2}}   \stt^ {-1}   \leq \mm^2 \stt^{-\frac 12}  \,.
 \label{zy-on-shock-tmp} 
\end{align} 
Combining \eqref{zyy-on-shock} and \eqref{zy-on-shock-tmp},  we have thus bounded the first term $\mathcal{H}_1$ on the right side of \eqref{zyy-identity} as
\begin{align} 
\sabs{\mathcal{H}_1} \le
7 \kappa \mm^2 \stt(\theta,t) ^{-{\frac{1}{2}} } \,.
\label{goodcat00}
\end{align} 

Next, we turn our attention to the second term, ${\mathcal H}_2$, in \eqref{zyy-identity}. Using  \eqref{eq:w:z:k:a:boot}, \eqref{dxphipsi-bound}, \eqref{dxxphipsi-bound-all}, \eqref{dxxphi-bound2}, \eqref{kyy-bound1}, and the fact that $\st( \sc(\stt(\theta,t)), \stt(\theta,t)) = \stt(\theta,t)$, we similarly obtain that 
\begin{align} 
\sabs{\mathcal{H}_2} 
&\le \kappa \mm^2 \stt(\theta,t)^{- \frac 12} + R_6 \stt(\theta,t)^{-\frac 12} + \kappa^{-\frac 12} R_6 \mm^{\frac 12} \stt(\theta,t)^{-\frac 12} + C \notag\\
&\leq 2 \kappa \mm^2 \stt(\theta,t)^{- \frac 12}\,.
\label{coolcat1} 
\end{align}

Since the integral term in \eqref{zyy-identity} was previously estimated in Lemma~\ref{lem:for-qz-stuff}, it thus remains to bound the term ${\mathcal H}_3$ on the right side of \eqref{zyy-identity}. Note that if $(\theta,t) \in \DD_{\bar \eps}^z\setminus \overline \DD_{\bar \eps}^k$, then $k$ vanishes, and so ${\mathcal H}_3 = 0$. In the case that $(\theta,t) \in \DD_{\bar \eps}^k$, by appealing to \eqref{eq:w:z:k:a:boot}, the bound $\frac{\kappa}{5} \leq c(\theta,t)  \leq \mm$, and \eqref{kyy-bound1}, we obtain
\begin{align} 
\sabs{\mathcal{H}_3} 
&\le \tfrac 14 \mm^3 \st(\theta,t)^{-\frac 12} + \tfrac 18 ( t^{-1}  + R_4 t^{\frac 12}) R_6 t^{-\frac 12}   \notag\\
&\leq \tfrac 14 \mm^3 \st(\theta,t)^{-\frac 12} + \mm t^{-\frac 12} \notag\\
&\leq  \mm^3 \st(\theta,t)^{-\frac 12}  \,.
\label{coolcat1a} 
\end{align}
In the last inequality of \eqref{coolcat1a} we have used that $\st(\theta,t) \leq t$. 

By combining the identity \eqref{zyy-identity} with the bounds \eqref{goodcat00}, \eqref{coolcat1},   \eqref{coolcat1a}, \eqref{lower-the-bridge1},  we have that 
\begin{align}
\sabs{z_{\theta\theta} (\theta,t) }
&\leq 9 \kappa \mm^2 \stt(\theta,t)^{-\frac 12} + 3 \mm^3 \stt(\theta,t)^{-\frac 12}
+
\begin{cases}
 \mm^3 \st(\theta,t)^{-\frac 12} \, , \qquad &\mbox{for} \qquad \theta \in \DD_{\bar \eps}^k \\
0\, , \qquad &\mbox{for} \qquad \theta \in \DD_{\bar \eps}^z\setminus \overline \DD_{\bar \eps}^k 
\end{cases} 
\,.
\end{align}
Taking into account that for $(\theta,t) \in \DD_{\bar \eps}^k$ by \eqref{c1c2diffkappa2} we have that $\st(\theta,t) \leq \stt(\theta,t)$, and we have $9 \kappa \leq \mm$, the above bound completes the proof of \eqref{zyy-bound-final}, once we ensure that $\frac 12 M_2 \geq 5 \mm^3 $ and $\frac 12 N_2 \geq 4 \mm^3$. This justifies the choices of $M_2$ and $N_2$ are defined in \eqref{eq:M1:M4}, respectively in \eqref{eq:N1:N3}.

In order to complete the proof of the Lemma, we need to establish the bound \eqref{eq:lossy:qy:dz:bound}, which is useful later in the proof. For this purpose, note that in view of \eqref{dyyqz-identity}, \eqref{zyy-identity}, the fact that $q^z_\theta(\theta,t)  = z_{\theta\theta} (\theta,t)  + {\mathcal H}_3$, and of the bounds  bounds \eqref{goodcat00}, \eqref{coolcat1},   \eqref{lower-the-bridge1},  we have that
\begin{align}
\sabs{q^z_\theta(\theta,t) } \leq 9 \kappa \mm^2 \stt(\theta,t)^{-\frac 12} +  3 \mm^3 \stt(\theta,t)^{-\frac 12}
\end{align}
which thus concludes the proof of \eqref{eq:lossy:qy:dz:bound}, and of the lemma. 
\end{proof}

\subsection{Lower bounds for second derivatives}
\label{sec:dyy:lower:bound}
In this section we prove that various second derivatives of the solution blow up as we approach the curves $\sc_1$ and $\sc_2$ from the right side. Throughout this section we fix $t \in (0,\bar \eps]$ and shall make reference to the following asymptotic descriptions:
\begin{subequations}
\label{eq:bohemian}
\begin{align}
\lim_{\theta\to \sc_2(t)^+} \frac{\theta -\sc_2(t)}{\st(\theta,t)} &= \frac{\kappa}{3}
\label{eq:bohemian:1} \\
\lim_{\theta\to \sc_2(t)} \stt(\theta,t) &\geq  \frac{t}{3}  
\label{eq:bohemian:2}\\
\lim_{\theta\to \sc_1(t)^+} \frac{\theta -\sc_1(t)}{\stt(\theta,t)} &= \frac{\kappa}{3} 
\label{eq:bohemian:4} \,.
\end{align} 
\end{subequations}
Here we have implicitly used that $\pt(\sc_2(t),s) = \sc_2(s)$, and $\pst(\sc_1(t),s) = \sc_1(s)$.
The bounds are a consequence of \eqref{c1c2diffkappa2}, \eqref{dxphipsi-bound}, and the definitions of $\sc_1$, $\sc_2$, $\pt$, and $\pst$. For example, in order to prove \eqref{eq:bohemian:1}, note that by the mean value theorem we have
\begin{align*}
\sc(\st(\theta,t)) - \sc_2(\st(\theta,t)) 
= \pt(\theta,\st(\theta,t))- \pt(\sc_2(t),\st(\theta,t))
= (\theta - \sc_2(t))  
\underbrace{\p_\theta \pt(\bar y, \st(\theta,t))}_{= 1 + \OO(\st^{\frac 13})}
\end{align*}
while by \eqref{eq:distance:curves} we have
\begin{align*}
\sc(\st(\theta,t)) - \sc_2(\st(\theta,t))  = \tfrac{\kappa}{3} \st(\theta,t) + \OO(\st(\theta,t)^{\frac 43})\,.
\end{align*}
The proof of \eqref{eq:bohemian:4} is similar. Lastly, in order to prove \eqref{eq:bohemian:2}, we use that one the hand
\begin{align*}
\sc(\stt(\theta,t)) - \sc_2(\stt(\theta,t)) = \tfrac{\kappa}{3} \stt(\theta,t) +  \OO(\stt(\theta,t)^{\frac 43}) \,,
\end{align*}
while on the other hand
\begin{align*}
 \sc(\stt(\theta,t)) - \sc_2(\stt(\theta,t)) 
 &= \pst(\sc_2(t),\stt(\theta,t)) - \pt(\sc_2(t),\stt(\theta,t)) \notag\\ 
 &= \int_{\stt(\theta,t)}^{t} \underbrace{\left( \p_s \pt -\p_s \pst\right)}_{= \frac{\kappa}{3} + \OO(r^{\frac 13})} ds
 = (t -\stt(\theta,t))  \left( \tfrac{\kappa}{3} + \OO(\bar \eps^{\frac 13}) \right)
 \,.
\end{align*}
By combining the above two estimates, it follows that $\stt(\sc_2(t),t) \geq t ( \frac 12 - \OO(\bar \eps^{\frac 13}) ) \geq \frac  t 3$,   proving \eqref{eq:bohemian:2}.

\subsubsection{Singularities on $\sc_2$, from the right side}
Note that the second derivative upper bounds established in \eqref{eq:second:order:boot}  blow up as $\theta \to \sc_2(t)^+$; the purpose of this subsection is to obtain lower bounds which are within a constant factor of these upper bounds, and thus also diverge as $\theta\to \sc_2(t)^+$.  
 
In this proof we shall frequently use the following facts. First, that $\frac{\kappa}{5} \leq c(\theta,t)  \leq \mm$ for all $(\theta,t) \in \DD_{\bar \eps}$. This follows from the identity $c(\theta,t)  = \tfrac 12 \wb(\theta,t)  + \tfrac 12 (w-\wb -z)$, which in view of \eqref{eq:wb:def}, and  \eqref{eq:w:z:k:a:boot} implies $c(\theta,t)  = \tfrac 12 w_0(\etab^{-1}(\theta,t) ) + \OO(t)$; the desired bound now follows from \eqref{eq:u0:ass:2} and \eqref{eq:u0:ass:2a}. Second, we note that a slightly sharper bound is required for $\p_\theta  \wb$ on the shock curve (when compared to \eqref{thegoodstuff1}). From \eqref{eq:sharp:range:for:solutions} we note that $\etab^{-1}(\sc(t)^-,t) = - (\bb t)^{\frac 32} + \OO(t^2)$. By appealing to \eqref{eq:u0:ass:3} we then obtain that $w_0'(\etab^{-1}(\sc(t)^-,t) ) = - \frac 13 t^{-1} + \OO(t^{-\frac 12})$ as $t \to 0$. We then conclude from \eqref{eq:px:wb:def} that 
\begin{align}
 \p_\theta \wb(\sc(t),t ) = \frac{ - \frac 13 t^{-1} + \OO(t^{-\frac 12})}{1 + t (-\frac 13 t^{-1} + \OO(t^{-\frac 12})) } = - \tfrac{1}{2t} + \OO(t^{-\frac 12})
 \label{eq:double:trouble:6}
\end{align}
as for $0 < t\leq \bar \eps$.

{\bf Lower bound for $|k_{\theta\theta} |$ on $\sc_2^+$}.  The desired lower bound turns out to be a consequence of \eqref{kyy-shock}. 

We first consider the second line of \eqref{kyy-shock}. 
Let $\theta> \sc_2(t)$ with $ \theta - \sc_2(t) \leq \frac{\kappa t}{6}$. Note  in this range of $\theta$, due to \eqref{eq:bohemian:1} and the fact that $t \leq \bar \eps$, we have $\st(\theta,t) \leq \frac{t}{3} \leq \bar \eps$. We claim that for a constant $C = C(\kappa,\mm,\bb,\cc)>0$ we have
\begin{align}
\p_\theta ^2 \pt(\theta,\st(\theta,t)) \geq   \tfrac{\kappa^2}{100 \mm^3} \st(\theta,t)^{-1} - C t^{-\frac 23} \geq   \tfrac{ \kappa^2}{100 \mm^3} \st(\theta,t)^{-1} - C \st(\theta,t)^{-\frac 23} 
\geq 0
\label{eq:kinda:magic}
\end{align}
for $\sc_2(t) < \theta < \sc_2(t) + \frac{\kappa t}{6}$, once $\bar \eps$ is sufficiently small. In order to prove \eqref{eq:kinda:magic}, we consider the formula \eqref{dyyphi} with $s=\st(\theta,t)$. We note that the largest term in \eqref{dyyphi}, the one  containing $c(\theta,t)  c_\theta(\sc(\st),\st) \p_\theta  \pt(\theta,\st)$, is positive. Indeed, from the  bounds $ \frac{\kappa}{5} \leq c(\theta,t)  \leq \mm $, \eqref{eq:double:trouble:6}, the bound \eqref{dxphipsi-bound},  \eqref{eq:w:dx:boot}, and \eqref{eq:z:dx:boot}, we obtain that 
\begin{align*}
&c(\theta,t)  c_\theta(\sc(\st(\theta,t))^+,\st(\theta,t)) \p_\theta  \pt(\theta,\st(\theta,t)) \notag \\
&\qquad = \tfrac 12  c(\theta,t)   \p_\theta  \pt(\theta,\st) \left( \p_\theta  \wb(\sc(\st)^+,\st) + \p_\theta  (w - \wb -z)(\sc(\st)^+,\st) \right) 
\notag\\
&\qquad = \tfrac 12  c(\theta,t)   (1 + \OO(t^{\frac 13}))  \left( - \tfrac{1}{2} \st^{-1}+ \OO(\st^{-\frac 12})  \right) 
\notag\\
&\qquad \leq - \tfrac{\kappa}{40} \st(\theta,t)^{-1}
\end{align*}
since $\st(\theta,t) \leq \frac t3 \ll 1$. The remaining terms in \eqref{dyyphi} may be estimated from above by 
\begin{align*}
2 e^{16 \mm t} \left(\tfrac{80 \mm^2}{\kappa^2} R_7 t + \tfrac{25 \mm }{\kappa^{2}} \left( \tfrac{4\bb}{5} (\tfrac{\kappa t}{6})^{-\frac 23} + R_2 (\tfrac{\kappa t}{6})^{-\frac 12} + R_4 t^{\frac 12} \right)\right)
\leq C t^{-\frac 23}
\end{align*}
for a constant $C = C(\kappa,\mm,\bb,\cc)>0$. The above two estimates then imply 
\begin{align*}
\p_\theta ^2 \pt(\theta,\st(\theta,t)) \geq 2 e^{-16 \mm t} \tfrac{\kappa}{5\mm^3} \tfrac{\kappa}{40} \st(\theta,t)^{-1} - C t^{-\frac 23} \,,
\end{align*}
and \eqref{eq:kinda:magic} follows.

Next, we return to the second line of \eqref{kyy-shock}, from \eqref{eq:dt:zl:kl:on:shock} we have
\begin{align}
 \dot \kl (\st(\theta,t)) = \tfrac{48 \bb^{\frac 92}}{\kappa^3} \st(\theta,t)^{\frac 12} + \OO(\st(\theta,t)) \geq 0
 \label{eq:kinda:magic:1}
\end{align}
since $\st(\theta,t) \leq t$ is small. Moreover, from \eqref{eq:ps:pt:bnd} and \eqref{eq:sc:ass} we have  $\frac{\kappa}{4}\leq (\dot \sc - \p_s \pt)(\theta,\st(\theta,t))  \leq \frac{\kappa}{2} $. As a consequence, from \eqref{eq:kinda:magic} and \eqref{eq:kinda:magic:1}, we obtain
\begin{align}
\mbox{second line of }\eqref{kyy-shock} 
&\geq \tfrac{2}{\kappa} \left( \tfrac{48 \bb^{\frac 92}}{\kappa^3} \st(\theta,t)^{\frac 12} - C \st(\theta,t) \right) \left( \tfrac{ \kappa^2}{100 \mm^3} \st(\theta,t)^{-1} - C \st(\theta,t)^{-\frac 23}  \right) \notag\\
&\geq \tfrac{48 \bb^{\frac 92}}{50 \kappa^2 \mm^2} \st(\theta,t)^{-\frac 12} - C \st(\theta,t)^{-\frac 16}
\notag\\
&\geq \tfrac{\bb^{\frac 92}}{2 \kappa^2 \mm^2} \st(\theta,t)^{-\frac 12}
 \qquad \mbox{as} \qquad \theta \to \sc_2(t)^+\,.
 \label{eq:kinda:magic:2}
\end{align}

Next, we consider the terms on the first line of \eqref{kyy-shock}. 
From the definition of $\lambda_2$ in \eqref{eq:wave-speeds} and the evolution equation \eqref{xland-w}, we obtain that 
\begin{align*}
  \left( \p_t\lambda_2  +(2\dot\sc -\lambda_2)\p_\theta  \lambda_2 \right) - \ddot \sc
 &= \tfrac 23 \p_\theta  w (\dot \sc - \lambda_2 + \dot \sc - w - \tfrac 13 z) \notag\\
 &\qquad 
 + \tfrac 23 (2\dot \sc - \lambda_2) \p_\theta  z + \tfrac 12 \p_t z + \tfrac 23 \left( \tfrac{1}{24} (w-z)^2 \p_\theta  k - \tfrac 82 a w\right) - \ddot \sc
 \,. 
\end{align*}
Taking into account the bound \eqref{eq:sc:ass}, Proposition~\ref{prop:Burgers}, and the fact that $(w,z,k,a) \in {\mathcal X}_{\bar \eps}$ (in particular, that \eqref{eq:double:trouble} holds), we obtain
\begin{align}
\label{eq:double:trouble:3}
 & \left( \p_t\lambda_2  +(2\dot\sc -\lambda_2)\p_\theta  \lambda_2 \right)(\sc(\st(\theta,t)),\st(\theta,t)) - \ddot \sc(\st(\theta,t)) \notag\\
 &= \tfrac 23 \p_\theta  \wb(\sc(\st(\theta,t),\st(\theta,t)) \left(2 \dot \sc (\st(\theta,t))- \tfrac 53 \wb (\sc(\st(\theta,t)),\st(\theta,t))  \right) 
 + \OO(\st(\theta,t)^{-\frac 12})
\end{align}
as $\theta \to \sc_2(t)^+$, or equivalently, as $\st(\theta,t) \to 0^+$. Next, from \eqref{eq:dt:zl:kl:on:shock} and \eqref{eq:dt:dt:zl:kl:on:shock} we note that 
\begin{align}
\ddot \kl (\st(\theta,t))) = \tfrac{1}{2 \st(\theta,t)} \dot \kl (\st(\theta,t))) + \OO(1)
 \label{eq:double:trouble:4}
\end{align}
as $\st(\theta,t) \to 0^+$. By combining \eqref{eq:double:trouble:3},  \eqref{eq:double:trouble:4},  the bound $\frac{\kappa}{2} \geq \dot \sc (\st(\theta,t)) - \lambda_2 (\sc(\st(\theta,t)),\st(\theta,t)) \geq \frac{\kappa}{4}$, and \eqref{eq:dt:zl:kl:on:shock}, we deduce
\begin{align}
& \mbox{first line of }\eqref{kyy-shock} \notag\\
&\quad = \left(\tfrac{ \p_\theta  \pt( \theta, \st(\theta,t) )  }{\dot \sc (\st(\theta,t)) - \lambda_2 (\sc(\st(\theta,t)),\st(\theta,t))}\right)^2
\dot \kl (\st(\theta,t))  \left(\tfrac{1}{2 \st(\theta,t)} + \tfrac 23\p_\theta  \wb(\sc(\st(\theta,t),\st(\theta,t)) \right) + \OO(1)
 \label{eq:double:trouble:5}
\end{align}
as $\st(\theta,t) \to 0^+$.
At this stage we appeal to \eqref{eq:double:trouble:6} with $t$ replaced by $\st = \st(\theta,t) \to 0^+$, which is the relevant regime for $\theta\to \sc_2(t)^+$.
From \eqref{eq:double:trouble:5}, \eqref{eq:double:trouble:6}, \eqref{dxphipsi-bound}, and  \eqref{eq:dt:zl:kl:on:shock}
we finally conclude that 
\begin{align}
\mbox{first line of }\eqref{kyy-shock} 
&= \left(\tfrac{ \p_\theta  \pt( \theta, \st(\theta,t) )  }{\dot \sc (\st(\theta,t)) - \lambda_2 (\sc(\st(\theta,t)),\st(\theta,t))}\right)^2
\dot \kl (\st(\theta,t))) \tfrac{1}{6 \st(\theta,t)}   + \OO(1)
\notag\\
&\geq \tfrac{3}{\kappa^2} \tfrac{48 \bb^{\frac 92} \st(\theta,t)^{\frac 12}}{\kappa^3}  \tfrac{1}{6 \st(\theta,t)}   - C \notag\\
&\geq \tfrac{24 \bb^{\frac 92} }{\kappa^4} \st(\theta,t)^{-\frac 12}
 \label{eq:double:trouble:7}
\end{align}
as $\st(\theta,t) \to 0^+$.

Lastly, by combining \eqref{eq:kinda:magic:2} with \eqref{eq:double:trouble:7}, we obtain that 
\begin{align}
\lim_{\theta \to \sc_2(t)^+} \p_\theta ^2 k(\theta,t)  \st(\theta,t)^{\frac 12}\geq \tfrac{24 \bb^{\frac 92} }{\kappa^4}
\label{eq:double:trouble:final}
\,.
\end{align}
In view of \eqref{eq:bohemian:1}, the above estimate and \eqref{eq:k:dxx:boot} thus precisely determines the blowup rate of $\p_\theta ^2 k(\theta,t) $ as $\theta\to \sc_2(t)^+$: this rate lies within two constants of $(\theta- \sc_2(t))^{- \frac 12}$.

{\bf Lower bound for $|z_{\theta\theta} |$ on $\sc_2^+$}.
Next, we show that the upper bound \eqref{eq:z:dxx:boot} also has a corresponding lower bound which blows up as $\theta \to \sc_2(t)^+$. We start by recalling the function $q^z$ defined in \eqref{qz-eqn}, and the formula for its derivative in \eqref{dyyqz-identity}. As above, we let $\sc_2(t) < \theta < \sc_2(t) + \frac{\kappa t}{6}$ and denote $\stt = \stt(\theta,t)$.
From estimate \eqref{eq:lossy:qy:dz:bound}, and by appealing to \eqref{eq:bohemian:1} which yields  $\stt(\theta,t) \geq \frac 14 t$  in the range of $\theta$ considered here, we arrive at 
\begin{align*}
\abs{z_{\theta\theta}  + \tfrac 14 c k_{\theta\theta}  + \tfrac 14 c_\theta k_\theta}(\theta,t)  = \abs{\p_\theta  q^z (\theta,t) } \leq C \stt(\theta,t)^{-\frac 12} \leq C t^{-\frac 12} \,,
\end{align*}
for all $\theta>\sc_2(t)$ which is close to $\sc_2(t)$. Furthermore, since~\eqref{eq:w:z:k:a:boot} and \eqref{eq:double:trouble:a} imply that
$\abs{c_\theta k_\theta}(\theta,t)  \leq \frac 12 (t^{-1} +  R_4 t^{\frac 12}) R_6 t^{\frac 12} \leq C t^{-\frac 12}$, 
the above estimate implies 
\begin{align}
\label{eq:stormy:monday:1}
\abs{z_{\theta\theta}  + \tfrac 14 c k_{\theta\theta} }(\theta,t)   \leq C t^{-\frac 12} \,,
\end{align}
for a suitable constant $ C = C(\kappa,\bb,\cc,\mm)>0$. 

Lastly, since $ \frac{\kappa}{5} \leq c(\theta,t)  \leq \mm $, we see that the blowup rate for $k_{\theta\theta} $ as $\theta\to \sc_2^+(t)$, given by \eqref{eq:double:trouble:final}, is immediately transferred to $z_{\theta\theta} $, and we have
\begin{align}
\lim_{\theta \to \sc_2(t)^+} \p_\theta ^2 z(\theta,t)  \st(\theta,t)^{\frac 12}
&\le - \tfrac 14 \lim_{\theta \to \sc_2(t)^+} c(\theta,t)  \p_\theta ^2 k(\theta,t)  \st(\theta,t)^{\frac 12}
 \leq -\tfrac{ \bb^{\frac 92} }{  \kappa^3}
\label{eq:stormy:monday:2}
\,.
\end{align}
Here we have used the fact that $\lim_{\theta\to \sc(t)^+} \st(\theta,t) t^{-\frac 12} = 0$.
The estimate \eqref{eq:stormy:monday:2}, and the upper bound \eqref{eq:z:dxx:boot}, show that $\p_{\theta}^2 z(\theta,t)  \to - \infty$ as $\theta\to \sc_2(t)^+$, at a rate which is proportional to $- (\theta-\sc_2(t))^{-\frac 12}$. 

{\bf Lower bound for $|w_{\theta\theta} |$ on $\sc_2^+$}.
The argument is nearly identical to the one for the second derivative of $z$. We recall that the variable $q^w$ defined in \eqref{qw-def} satisfies the derivative bound \eqref{vladneeds2}. By appealing to the fact that $(w,z,k,a) \in {\mathcal X}_{\bar \eps}$, the estimate \eqref{thegoodstuff2} for the second derivative of the Burgers solution, and to \eqref{vladneeds2}, we arrive at
\begin{align}
\abs{w_{\theta\theta}  - \tfrac 14 c k_{\theta\theta} }(\theta,t)  
&\leq \tfrac 14 \abs{ c_\theta k_{\theta}}(\theta,t)  +  \abs{ q^w_\theta(\theta,t) } \notag\\
&\leq \tfrac 18 (t^{-1} + R_4 t^{\frac 12}) R_6 t^{\frac 12} +  3 \bb (\kappa t)^{-\frac 53}\notag\\
&\leq C t^{-\frac 53} 
\label{eq:stormy:monday:3}
\end{align}
for all $\theta \in ( \sc_2(t), \sc_2(t) + \tfrac{\kappa t}{6})$, for a suitable constant $ C = C(\kappa,\bb,\cc,\mm)>0$. This estimate is the parallel bound to \eqref{eq:stormy:monday:1} for the second derivative of $z$. It implies, in a similar fashion to \eqref{eq:stormy:monday:3}, that 
\begin{align}
\lim_{\theta \to \sc_2(t)^+} \p_\theta ^2 w(\theta,t)  \st(\theta,t)^{\frac 12}
&\geq \tfrac{ \bb^{\frac 92} }{  \kappa^3}
\label{eq:stormy:monday:4}
\,.
\end{align}
The estimate \eqref{eq:stormy:monday:4}, and the upper bound \eqref{eq:w:dxx:boot}, show that $\p_{\theta}^2 w(\theta,t)  \to + \infty$ as $\theta\to \sc_2(t)^+$, at a rate which is proportional to $(\theta-\sc_2(t))^{-\frac 12}$. 
 
{\bf Lower bound for $|a_{\theta\theta} |$ on $\sc_2^+$}.
As before, consider $\theta \in (\sc_2(t), \sc_2(t) + \frac{\kappa t}{6})$. By combining \eqref{thegoodstuff1}, \eqref{eq:k:boot}, \eqref{eq:varpi:boot}, and \eqref{eq:proto:1}, we arrive at the  bound
\begin{align}
 \abs{a_{\theta\theta}  +  \tfrac{1}{4}    c^2 e^{-k} \varpi_{\theta}  }
&\leq C t^{-\frac 23} +  C t^{-\frac 12}  
\leq C t^{-\frac 23}
\,.
\label{eq:proto:2}
\end{align}
The desired lower bound on $a_{\theta\theta} $ is thus inherited from $\varpi_{\theta} $, which we recall is given by \eqref{eq:dy:varpi}. The principal contribution is due to the term containing the time integral of $k_{\theta\theta} $.  Indeed, using the same argument used to prove \eqref{eq:dy:varpi:fin}, we have that 
\begin{align}
\abs{ \varpi_{\theta} (\theta,t)  - \tfrac 43 \int_{\st(\theta,t)}^t  \bigl(e^k k_{\theta\theta} \bigr)(\pt(\theta,s),s) e^{I_{\varpi_{\theta} }(\theta,t;s)} ds} \leq  C \,.
\label{eq:proto:3}
\end{align}
The analysis reduces to establishing a lower bound which is commensurate with the upper bound \eqref{eq:useful:later}. The main idea here is as follows. From \eqref{eq:kinda:magic:2} and \eqref{eq:double:trouble:7}, as in \eqref{eq:double:trouble:final}
we have that $\p_{\theta}^2  k(\theta,t)  \geq 24 \bb^{\frac 92} \kappa^{-4} \st(\theta,t)^{-\frac 12}$, for all $\theta$ sufficiently close to $\sc_2(t)$, i.e. $\sc_2(t)< \theta < \sc_2(t) + \frac{\kappa t}{6}$. Therefore, if the point $(\theta,t) $ is replaced by the point $(\pt(\theta,s),s)$, which in view of Remark~\ref{c1c2diffkappa2} and estimate \eqref{eq:distance:curves} is such that $\pt(\theta,s)$ is sufficiently close to $\sc_2(s)$, we have that 
\begin{align*}
 \p_{\theta}^2  k(\pt(\theta,s),s) \geq 24 \bb^{\frac 92} \kappa^{-4} \st(\pt(\theta,s),s)^{-\frac 12} = 24 \bb^{\frac 92} \kappa^{-4} \st(\theta,t)^{-\frac 12}
\end{align*}
uniformly for all $s \in [\st(\theta,t),t]$. In particular, $\p_\theta ^2 k \circ \pt > 0$, and so by combining \eqref{eq:proto:2}--\eqref{eq:proto:3}, with \eqref{eq:k:dx:boot}, \eqref{eq:int:factor:dy:varpi:1}, and with the estimate $\frac{\kappa}{5} \leq c \leq \mm$, we arrive at
\begin{align}
 a_{\theta\theta} (\theta,t)  
 &\leq - \tfrac 14 c^2 e^{-k} \varpi_{\theta}  + C t^{-\frac 23} \notag\\
 &\leq - \tfrac{\kappa^2}{75} \int_{\st(\theta,t)}^t  k_{\theta\theta} (\pt(\theta,s),s) ds + C t^{-\frac 23} 
\notag\\
&\leq - \tfrac{\bb^{\frac 92}}{4 \kappa^2} (t-\st(\theta,t)) \st(\theta,t)^{-\frac 12} \,.
\end{align}
The above estimate implies 
\begin{align}
\lim_{\theta \to \sc_2(t)^+} \st(\theta,t)^{\frac 12} a_{\theta\theta} (\theta,t) 
\leq - \tfrac{\bb^{\frac 92}}{4 \kappa^2} t \,,
\label{eq:oxigen:2}
\end{align}
which may be combined with the upper bound \eqref{eq:ayy:improve:bootstrap} show that $a_{\theta\theta} (\theta,t)  \to -\infty$ as $\theta\to \sc_2(t)^+$, at a rate which is proportional to $t (\theta-\sc_2(t))^{-\frac 12}$.

\subsubsection{Singularities on $\sc_1$, from the right side}
Passing to the limit $\theta \to \sc_1(t)^+$ in the estimates \eqref{eq:second:order:boot2}, we obtain that 
\begin{align*}
\lim_{\theta \to \sc_1(t)^+} | w_{\theta\theta}  (\theta,t) | \leq C t^{-\frac 53}
\,, \qquad \mbox{and} \qquad 
\lim_{\theta \to \sc_1(t)^+} | a_{\theta\theta}  (\theta,t) | \leq C t^{-\frac 23} 
\,,
\end{align*}
for a suitable constant $C = C(\kappa,\bb,\cc,\mm)>0$,
which shows that these quantities do not blow up as $\theta$ approaches $\sc_1$ from the right side. 
The only quantity that does indeed blow up is the second derivative of $z$.

Here we establish a lower bound for $|\p_{\theta}^2 z(\theta,t) |$ which is commensurate with \eqref{eq:z:dxx:boot2} as $\theta \to \sc_1(t)^+$; more precisely we claim that 
\begin{align}
\lim_{\theta \to \sc_1(t)^+} \stt(\theta,t)^{\frac 12} z_{\theta\theta}  (\theta,t)  \leq - \tfrac 14 \bb^{\frac 92} \kappa^{-4} \,,
\label{eq:sanity}
\end{align}
which shows the precise rate of divergence of $\p_\theta ^2 z$ towards $-\infty$ as $\theta$ approaches $\sc_1$ from the right side.
The proof of \eqref{eq:sanity} is quite involved, and will be broken up into several parts, which correspond to estimating the various terms in \eqref{dyyqz-identity}. We rewrite this identity as
\begin{align}
z_{\theta\theta} (\theta,t)  = {\mathcal B}_1 + {\mathcal B}_2 + {\mathcal B}_3 \,,
\label{eq:sanity:3}
\end{align}
where we define
\begin{align}
{\mathcal B}_1 
&:= q^z_{\theta}(\sc(\stt(\theta,t)),\stt(\theta,t)) (\p_\theta  \psi_t)^2(\theta,\stt(\theta,t)) + q^z(\sc(\stt(\theta,t)),\stt(\theta,t)) \p_\theta ^2 \psi_t (\theta,\stt(\theta,t)) = {\mathcal B}_{11} + {\mathcal B}_{12}
\label{eq:the:HUN} 
\\
{\mathcal B}_2
&:= - \int_{\stt_1(\theta,t) }^t \Bigl( \p_\theta Q^z\! \circ\! \psi_t  \ (\p_\theta  \psi_t)^2+Q^z\! \circ\! \psi_t  \ \p_\theta ^2 \psi_t \bigr) ds 
\label{eq:the:TURK}
\\
{\mathcal B}_3
&:= - \int_{\stt(\theta,t)}^{\stt_1(\theta,t) } \Bigl( \p_\theta Q^z\! \circ\! \psi_t  \ (\p_\theta  \psi_t)^2+Q^z\! \circ\! \psi_t  \ \p_\theta ^2 \psi_t \bigr) ds 
\label{eq:the:SLAV}
\end{align}
and $\stt_{1}$ is the time at which $\pst(\theta,\cdot)$ intersects the curve $\sc_2$; as given by \eqref{eq:sanity:2}, see also Figure~\ref{fig:many:characteristics:4}.
Since $\theta \to \sc_1(t)^+$ is equivalent in view of \eqref{eq:bohemian:4}  to $\stt(\theta,t) \to 0$, our goal is to extract the leading order term in ${\mathcal B}_1$ with respect to $\stt\ll 1$, and then to obtain sharp estimates for ${\mathcal B}_2$ and ${\mathcal B}_3$ with respect to $\stt$. In this direction we claim:
\begin{lemma}
\label{lem:the:HUN}
Fix $t\in (0,\bar \eps]$ and $\sc_1(t) < \theta < \sc_1(t) + \frac{\kappa t}{6}$. Then we have that 
\begin{align}
-  \tfrac{86}{16} \bb^{\frac 92} \kappa^{-4} \stt(\theta,t)^{-\frac 12} \leq  {\mathcal B}_1 \leq -  \tfrac{85}{16} \bb^{\frac 92} \kappa^{-4} \stt(\theta,t)^{-\frac 12} \,,
\label{eq:the:HUN:*}
\end{align}
where the term ${\mathcal B}$ is as defined in \eqref{eq:the:HUN}.
\end{lemma}
 
\begin{lemma} \label{lem:insanity}
Fix $t\in (0,\bar \eps]$ and $\sc_1(t) < \theta < \sc_1(t) + \frac{\kappa t}{6}$. Then we have that 
\begin{subequations}
\label{eq:insanity}
\begin{align} 
{\mathcal B}_2 &\le    t^{\frac 14}  \bb ^\frac{9}{2} \kappa^{-4} \stt(\theta,t)^{-\frac{1}{2}}   \,, \label{vinsanity} \\
{\mathcal B}_3 &\le  \tfrac 92 \bb^\frac{9}{2} \kappa^{-4} \stt(\theta,t)^{-\frac{1}{2}}  + C \stt(\theta,t)^{-\frac 16} \,, \label{insanity}
\end{align} 
\end{subequations}
where the terms ${\mathcal B}_2$ and ${\mathcal B}_3$ defined in \eqref{eq:the:TURK} and respectively in \eqref{eq:the:SLAV}. Note that the sum of the estimates in \eqref{eq:insanity} gives an improvement over \eqref{lower-the-bridge1}, in the sense that the constant is sharper. 
\end{lemma} 

\vspace{0.05in}
\noindent {\bf Proof of \eqref{eq:sanity}}.
We note that the bound \eqref{eq:sanity} follows from \eqref{eq:sanity:3}, \eqref{eq:the:HUN:*}, \eqref{vinsanity},   \eqref{insanity}, and the inequality 
$$\tfrac 92 + C  t^{\frac 14} + C \stt^{\frac 13} - \tfrac{85}{16} \leq  - \tfrac 14 \,,$$ 
for $\bar \eps$, and hence $t$ and $\stt$, sufficiently small. 
Thus, in order to complete the proof of the \eqref{eq:sanity}, it only remains to prove Lemmas~\ref{lem:the:HUN} and~\ref{lem:insanity}. These proofs occupy the remainder of this subsection.

\begin{proof}[Proof of Lemma~\ref{lem:the:HUN}]
We recall that $q^z$ is defined in \eqref{qz-def} as $z_\theta + \frac 14 c k_\theta$. The easiest term is the sound speed. From \eqref{eq:u0:ass}, \eqref{eq:range:for:solutions}, \eqref{eq:w:boot}, and \eqref{eq:z:boot} we  note that
\begin{align}
c(\sc(\stt),\stt) 
&=  \tfrac 12 \wb(\sc(\stt),\stt) + \tfrac 12 ( w- \wb + z)(\sc(\stt),\stt) \notag\\
&= \tfrac 12 w_0(x_{{\mathsf B},-}(\stt)) + \OO(\stt) \notag\\
&= \tfrac{\kappa}{2} - \tfrac 12 \bb^{\frac 32} \stt^{\frac 12} + \OO(\stt)
\,,
\label{eq:HUN:1}
\end{align}
as $\stt\to 0$.
The next term we consider is the $y$ derivative of $k$, restricted to the shock curve. This term is given by \eqref{kx-gen}, with $t$ replaced by $\stt$. The denominator of this fraction is given by 
$ 
\dot \sc(\stt) -   \lambda_2(\sc(\stt),\stt) = \dot \sc(\stt)  - \tfrac 43 c(\sc(\stt),\stt) - \tfrac 43 z(\sc(\stt),\stt) = \tfrac 13 \kappa + \OO(\stt^{\frac 12})\,,
$ by appealing to \eqref{eq:z:boot} and \eqref{eq:HUN:1}.
By combining the above estimate with the identity \eqref{eq:dt:zl:kl:on:shock}, we arrive at 
\begin{align}
k_\theta(\sc(\stt),\stt) = \frac{ \frac{48 \bb^{\frac 92}}{\kappa^3} \stt^{\frac 12} + \OO(\stt) }{\frac 13 \kappa + \OO(\stt^{\frac 12})}   =  144 \bb^{\frac 92} \kappa^{-4} \stt^{\frac 12} + \OO(\stt)
\,, \label{ky:LUV1} 
\end{align}
as $\stt\to 0$. The last ingredient needed to compute $q^z$ on the shock curve is to obtain a leading order term for the derivative of $z$. For this term we appeal to identity \eqref{zx-gen} with $t$ replaced by $\stt$. As above, we may show that $\dot \sc(\stt) -   \lambda_1(\sc(\stt),\stt) = \tfrac 23 \kappa + \OO(\stt^{\frac 12})$, and we may appeal to the estimate \eqref{eq:dt:zl:kl:on:shock} and the already established \eqref{eq:HUN:1} and \eqref{ky:LUV1}, to deduce 
\begin{align}
z_\theta(\sc(\stt),\stt) = \frac{\Bigl( - \frac{27 \bb^{\frac 92}}{4 \kappa^2} \stt^{\frac 12} + \OO(\stt)\Bigr)  - \frac{1}{6}\left( (\frac \kappa 2 + \OO(\stt^{\frac 12}))^2 \frac{144 \bb^{\frac 92}}{\kappa^{4}} \stt^{\frac 12} + \OO(\stt) \right) }{\frac 23 \kappa + \OO(\stt^{\frac 12})}  = - \tfrac{153}{8} \bb^{\frac 92} \kappa^{-3} \stt^{\frac 12} + \OO(\stt)
 \label{zy:LUV1} 
 \,.
\end{align}
We then combine the definition of $q^z$ in \eqref{qz-def} with \eqref{eq:HUN:1}--\eqref{zy:LUV1} and arrive at 
 \begin{align}
q^z(\sc(\stt),\stt) 
&=  - \tfrac{9}{8} \bb^{\frac 92} \kappa^{-3} \stt^{\frac 12} + \OO(\stt)
\,,
\label{eq:HUN:2}
\end{align}
as $\stt \to 0$.  In order to have a complete asymptotic description of the second term on the right side of \eqref{eq:the:HUN}, we need to determine $\p_\theta ^2 \pst(\theta,\stt)$. For this purpose, we use \eqref{dyypsi} with $s$ replaced by $\stt = \stt(\theta,t)$, and we recall that we are interested in the region $\sc_1(t) < \theta < \sc_1(t) + \tfrac{\kappa t}{6}$. By using \eqref{thegoodstuff1}, \eqref{eq:w:z:k:a:boot}, \eqref{dxphipsi-bound}, \eqref{zyy-integral}, \eqref{eq:double:trouble:6}, \eqref{eq:HUN:1}
\begin{align}
\p_\theta ^2 \pst(\theta,\stt) 
& =  \tfrac{1}{2} e^{\int_{\stt}^t  ( \frac{4}{3}a - z_\theta) \circ \pst ds'}  \tfrac{c^{\frac{1}{2}}(\theta,t)  }{c^{\frac{1}{2}}(\sc(\stt),\stt) }  \Bigl( \int_{\stt}^t  ( \tfrac{8}{3}a_\theta -2z_{\theta\theta} ) \circ \pst \, \p_\theta  \pst ds' 
+  \tfrac{ \p_\theta  c (\theta,t)  }{ c(\theta,t)  } -  \tfrac{ \p_\theta  c (\sc(\stt),\stt) }{ c(\sc(\stt),\stt) }
 \p_\theta  \pst(\theta,\stt) \Bigr) 
 \notag\\
 & =  \tfrac{1}{2} e^{\OO(t-\stt)}  \tfrac{(\frac{\kappa}{2} + \OO(t^{\frac 13}))^{\frac 12}}{(\tfrac{\kappa}{2} + \OO(\stt^{\frac 12}))^{\frac 12}}  \Bigl( \OO(t-\stt) + \OO(\stt^{\frac 12})  
-  \tfrac{   \OO(t^{-\frac 23}) }{\frac{\kappa}{2} + \OO(t^{\frac 13})} 
-  \tfrac{ - \tfrac 14 \stt^{-1} + \OO(\stt^{-\frac 12})}{\tfrac{\kappa}{2} + \OO(\stt^{\frac 12})}
(1 + \OO(\stt^{\frac 13})) \Bigr)
\notag\\
&=   \tfrac{1}{4 \kappa} \stt^{-1} + \OO( t^{-\frac 23} )  + \OO(\stt^{-\frac 12})  
\label{eq:HUN:3}
\end{align}
for $\stt < t \ll 1$. From \eqref{eq:HUN:2} and \eqref{eq:HUN:3}, and using that $\stt \leq t$,  we finally obtain that the second term in \eqref{eq:the:HUN} is given by 
\begin{align}
{\mathcal B}_{12} 
= - \left( \tfrac{9}{8} \bb^{\frac 92} \kappa^{-3} \stt^{\frac 12} - \OO(\stt)\right) \left( \tfrac{1}{4 \kappa} \stt^{-1} + \OO(\stt^{-\frac 23} )  \right)
= - \tfrac{9}{32} \bb^{\frac 92} \kappa^{-4} \stt^{-\frac 12} + \OO(\stt^{-\frac 16})\,.
\label{eq:the:HUN:2}
\end{align}

It remains to consider the first term on the right side of \eqref{eq:the:HUN}. We recall that $q^z_\theta = z_{\theta\theta}  + \frac 14 c_\theta k_\theta + \frac 14 c k_{\theta\theta} $. Thus, in view of \eqref{eq:HUN:1} and \eqref{ky:LUV1}, we need to estimate separately three terms on the shock curve: $c_\theta, k_{\theta\theta} $, and $z_{\theta\theta} $. First, similarly to \eqref{eq:HUN:1}, we have from \eqref{eq:double:trouble:6} and \eqref{eq:w:z:k:a:boot} that 
\begin{align}
c_\theta (\sc(\stt),\stt) 
=  \tfrac 12 (\p_\theta  \wb)(\sc(\stt),\stt) + \tfrac 12 \p_\theta  ( w- \wb + z)(\sc(\stt),\stt)
= - \tfrac 14  \stt^{-1} + \OO(\stt^{-\frac 12}) 
\,,
\label{eq:HUN:4}
\end{align}
as $\stt\to 0$. Next, we turn to $\p_{\theta}^2  k$, which is given by \eqref{kyy-on-shock}. By appealing to \eqref{eq:double:trouble:3}, \eqref{eq:double:trouble:6}, 
\eqref{eq:dt:zl:kl:on:shock},   \eqref{eq:dt:dt:zl:kl:on:shock}, and  \eqref{eq:w:z:k:a:boot}, we obtain
\begin{align}
k_{\theta\theta} (\sc(\stt),\stt)
&= 
\tfrac{\ddot \kl(\stt)}{(\dot\sc(\stt) - \lambda _2(\sc(\stt),\stt))^2} \notag\\
&\qquad + \Bigl(  \bigl(\p_t\lambda_2(\sc(\stt),\stt)  +(2\dot\sc(\stt) -\lambda_2(\sc(\stt),\stt))\p_\theta  \lambda_2(\sc(\stt),\stt) \bigr) - \ddot\sc(\stt) \Bigr) 
\tfrac{\dot \kl}{(\dot\sc(\stt) - \lambda _2(\sc(\stt),\stt))^3}
\notag\\
&= \tfrac{\tfrac{24 \bb^{\frac 92}}{ \kappa^3} \stt^{-\frac 12} + \OO(1)}{(\frac{\kappa}{3} + \OO(\stt^{\frac 12}))^2} 
+ \Bigl( -  \tfrac{\kappa}{9}  \stt^{-1}    + \OO(\stt^{-\frac 12}) \Bigr)
\tfrac{\tfrac{48 \bb^{\frac 92}}{\kappa^3} \stt^{\frac 12} + \OO(\stt)}{(\frac{\kappa}{3} + \OO(\stt^{\frac 12}))^3} 
\notag\\
&= 72 \bb^{\frac 92} \kappa^{-5} \stt^{-\frac 12} + \OO(1)
\label{ky:LUV2} 
\end{align}
as $\stt \to 0$. Lastly, we turn to $\p_{\theta}^2  z$, which is given by the expression \eqref{zyy-shock}.
By using \eqref{eq:sc:ass}, \eqref{eq:w:z:k:a:boot}, and \eqref{eq:R:z:yy},  we first rewrite
\begin{align} 
 z_{\theta\theta}  (\sc( \stt),\stt)
& =\frac{\ddot\zl - \tfrac 13 ( \kappa - \tfrac 12 w ) (c^2 k_{\theta\theta} )  - \tfrac 13 \left(  (\kappa - \tfrac 56 w ) c k_\theta -  (2 \kappa -\tfrac 43 w) z_\theta  \right) w_\theta }{ (\dot \sc - \lambda_1)^2}  \Bigr|_{(\sc( \stt),\stt)}
+ \OO(\stt^{\frac 12})  \,.  
\end{align} 
Then, by appealing to \eqref{eq:dt:dt:zl:kl:on:shock}, \eqref{eq:w:z:k:a:boot}, \eqref{eq:double:trouble:6}, \eqref{eq:HUN:1}, \eqref{ky:LUV1}, \eqref{zy:LUV1}, \eqref{eq:HUN:4}, and \eqref{ky:LUV2}, from the above formula we obtain
\begin{align}
z_{\theta\theta} (\sc(\stt),\stt)
& =\frac{\ddot \zl - \tfrac 13 ( \kappa - \tfrac 12 \wb ) (c^2 k_{\theta\theta} )  - \tfrac 13 \left(  (\kappa - \tfrac 56 \wb ) c k_\theta -  (2 \kappa  -\tfrac 43 \wb) z_\theta  \right) \p_\theta \wb }{(\frac{2\kappa}{3} + \OO(\stt^{\frac 12}))^2}     \Bigr|_{(\sc( \stt),\stt)}
+ \OO(1)  \notag\\
& =\frac{- \frac{27 \bb^{\frac 92}}{8 \kappa^2} \stt^{-\frac 12} - \tfrac 13 \frac{\kappa}{2} \frac{\kappa^2}{4}  \frac{72 \bb^{\frac 92}}{ \kappa^{5}} \stt^{-\frac 12}  - \tfrac 13 \left( \frac{\kappa}{6} \frac{\kappa}{2} \frac{144 \bb^{\frac 92}}{\kappa^{4}} \stt^{\frac 12} +  \frac{2\kappa}{3}  \frac{153\bb^{\frac 92}}{8 \kappa^{3}}   \stt^{\frac 12} \right) (-\frac{1}{2\stt}) }{(\frac{2\kappa}{3} + \OO(\stt^{\frac 12}))^2}   
+ \OO(1)  \notag\\
&= -  \tfrac{81}{16}  \bb^{\frac 92} \kappa^{-4} \stt^{-\frac 12} + \OO(1)
\,,   \label{zy:LUV2}
\end{align} 
as $\stt \to 0$. 
Using the definition of $q^z_y$, upon combining \eqref{eq:HUN:1}, \eqref{ky:LUV1}, \eqref{eq:HUN:4}, \eqref{ky:LUV2}, and \eqref{zy:LUV2} we obtain
 \begin{align}
q^z_\theta(\sc(\stt),\stt) 
&= - \tfrac{81}{16} \bb^{\frac 92} \kappa^{-4} \stt^{-\frac 12} + \OO(1)  + \tfrac 14 \left( - \tfrac 14  \stt^{-1} + \OO(\stt^{-\frac 12})  \right) \left(  144 \bb^{\frac 92} \kappa^{-4} \stt^{\frac 12} + \OO(\stt) \right) \notag\\
&\qquad + \tfrac 14 \left( \tfrac{\kappa}{2} + \OO(\stt^{\frac 12})\right) \left( 72 \bb^{\frac 92} \kappa^{-5} \stt^{-\frac 12} + \OO(1)\right)
\notag\\
&= - \tfrac{81}{16} \bb^{\frac 92} \kappa^{-4} \stt^{-\frac 12} + \OO(1)
\,.
\label{eq:HUN:5}
\end{align}
Lastly, by combining \eqref{dxphipsi-bound} with \eqref{eq:HUN:5}, and using that $\stt \leq t$, we obtain that 
the first term in \eqref{eq:the:HUN} is given by 
\begin{align}
{\mathcal B}_{11} 
&=  \left( - \tfrac{81}{16} \bb^{\frac 92} \kappa^{-4} \stt^{-\frac 12} + \OO(1) \right) \left(1 + \OO(t^{\frac 13})\right)^2  =  - \left(\tfrac{81}{16} \bb^{\frac 92} \kappa^{-4} + \OO(t^{\frac 13})\right) \stt^{-\frac 12}  + \OO(1)\,.
\label{eq:the:HUN:1}
\end{align}
Adding the bounds \eqref{eq:the:HUN:2} and \eqref{eq:the:HUN:1} completes the proof of the lemma.
\end{proof}

\begin{proof}[Proof of Lemma \ref{lem:insanity}]
Recall from \eqref{Qz-def} that $Q^z =   c k_\theta (\tfrac{1}{12} w_\theta + \tfrac{1}{12} z_\theta +  \tfrac{2}{3}  a) + \tfrac{8}{3} \p_\theta (a z)$. 
As in the proof of Lemma \ref{lem:for-qz-stuff}, we write $\p_\theta  Q^z =  \mathcal{Q}_1 + \mathcal{Q} _2$, where
 \begin{align*} 
 \mathcal{Q} _1 & = \overbrace{\tfrac{1}{12}   c k_\theta w_{\theta\theta}  }^{ \mathcal{Q} _{1a}}+ 
 \overbrace{\tfrac{1}{24}   k_\theta w_\theta w_\theta}^{ \mathcal{Q} _{1b}} + \overbrace{  c k_{\theta\theta}  ( \tfrac{1}{12}  w_\theta + \tfrac{1}{12} z_\theta +\tfrac{2}{3}  a)}^{ \mathcal{Q} _{1c}} \,, \\
 \mathcal{Q} _2 & = c k_\theta( \tfrac{1}{12}  z_{\theta\theta}  +  \tfrac{2}{3}   a_\theta) + \tfrac{8}{3} (az)_{\theta\theta} 
 + \tfrac{1}{24} k_\theta z_\theta w_\theta  \,. 
 \end{align*}

We first give the proof of the more difficult bound, \eqref{insanity}. Several times in this proof we require a bound on $\int_{\stt}^{\stt_1} |k_{\theta\theta}  w_\theta| \circ \pst$. In order to obtain a suitable estimate, we recall the bound of $\stt_1$ in \eqref{eq:sanity:2}, and introduce  the time which lies half way in between $\stt$ and $\stt_1$,  namely $\stt_2 = \stt + \frac 12 (\stt_1 - \stt) = \frac 32 \stt + \OO(\stt^{\frac 43})$. The reason is as follows. For $s\in [\stt, \stt_2]$, from Remark~\ref{rem:FU:3} we may deduce that $\st(\pst(\theta,s),s) \geq \frac 15 \stt(\theta,t)$; this lower bound is useful when combined with \eqref{eq:k:dxx:boot}, \eqref{dxphipsi-bound}, \eqref{eq:w:dx:boot}, and \eqref{eq:Steve:needs:this:1} for $\gamma(s) = \pst(\theta,s)$:
\begin{align}
\int_{\stt}^{\stt_2} |k_{\theta\theta}  w_\theta| \circ \pst \, (\p_\theta  \pst)^2 ds 
\les \stt^{-\frac 12} \int_{\stt}^{\stt_2} (| \p_\theta  \wb \circ \pst|  + s^{-\frac 12} )ds 
\les \stt^{-\frac 16}
\,.  \label{eq:HUN_0}
\end{align}
On the other hand, for the contribution coming from $s\in [\stt_2, \stt_1]$, the trick is to use that $|\pst(\theta,s) - \sc(s)|\geq \frac{\kappa s}{8}$. Then, we may appeal to  the bound \eqref{stt-integral}, to \eqref{eq:w:dx:boot}, and to the estimate \eqref{thegoodstuff1}, which in this region gives that $|\p_\theta  \wb(\pst(\theta,s),s)| \leq \frac 45 \bb |\pst(\theta,s) - \sc(s)|^{-\frac 23} \les s^{-\frac 23} \les \stt^{-\frac 23}$, concluding in 
\begin{align}
\int_{\stt_2}^{\stt_1} |k_{\theta\theta}  w_\theta| \circ \pst \, (\p_\theta  \pst)^2 ds 
\les \stt^{-\frac 23} \int_{\stt_2}^{\stt_1} \st(\pst(\theta,s),s)^{-\frac 12} ds 
\les \stt^{-\frac 16}
\,.  
\end{align}
Combining the above two bounds, and the fact that $\frac{\kappa}{5} \leq c \leq \mm$, we conclude that
\begin{align}
\int_{\stt}^{\stt_1} |c k_{\theta\theta}  w_\theta| \circ \pst \, (\p_\theta  \pst)^2 ds 
\les \stt^{-\frac 16}
\,. \label{eq:HUN*}
\end{align}
The remaining contribution to  ${\mathcal Q}_{1c}$ is bounded as
\begin{align} 
\int_{\stt}^{\stt_1} \abs{ c  k_{\theta\theta}  (\tfrac{1}{12} z_\theta + \tfrac 23 a)} \circ \pst (\p_\theta  \psi_t)^2ds \les 
 \stt^{\frac{1}{2}}\,. \label{Q1c-less-insane}
\end{align} 

Next, let us   estimate  $-\int_{\stt(\theta,t)}^{\stt_1(\theta,t) }\mathcal{Q} _{1b}\! \circ\! \psi_t  \ (\p_\theta  \psi_t)^2ds$.   From \eqref{eq:sanity:2}, we see that $\stt_1(\theta,t)  = 2 \stt(\theta,t) 
+ \OO(t^ {\frac{4}{3}} )$.  Hence,  using the bounds \eqref{eq:Steve:needs:this:1}, \eqref{eq:w:dx:boot}, \eqref{good-kyflow-bound}, and  \eqref{eq:double:trouble:6}, 
we have that
\begin{align} 
\int_{\stt}^{\stt_1} \bigl(   k_\theta (w_\theta)^2\bigr)  \circ \pst (\p_\theta  \psi_t)^2ds \les \stt^{-\frac{1}{2}} \int_{\stt}^{\stt_1} \left( \sabs{w_\theta -\wb_{\theta}} \circ \pst + \sabs{\wb_{\theta}} \circ \pst \right)ds \les
 \stt ^ {-\frac{1}{6}}\,. \label{Q1b-insane}
\end{align} 

Next, in order to bound the contribution from $\mathcal{Q}_{1a}$,  we define
\begin{align*} 
 \mathcal{A} & =-\tfrac{1}{12}\bigl( c k_\theta w_{\theta\theta}  + cw_\theta k_{\theta\theta} \bigr) \, \\
 \mathcal{G} &=  \tfrac{2}{3}  c w_\theta k_{\theta\theta}  +   \tfrac{1}{3}  c c_\theta (k_\theta)^2 
+ \tfrac{1}{6}   c^2 k_{\theta\theta}  k_\theta -  \tfrac{8}{3}  (aw)_\theta k_\theta -  k_\theta(w_\theta)^2 + k_\theta w_\theta z_\theta\,.
\end{align*} 
A straightforward computation shows that the product $k_\theta w_y$ solves the equation 
\begin{align} 
\p_s (k_\theta w_\theta) + \lambda _1 \p_\theta (k_\theta w_\theta) + 2 (k_\theta w_\theta)\p_\theta  \lambda _1
&=  \tfrac{48}{3}  \mathcal{A} +  \mathcal{G} \,. \label{kywy-eqn}
\end{align} 
We now obtain an explicit solution to \eqref{kywy-eqn}. In order to solve \eqref{kywy-eqn}, we set 
$$
\chi = (k_\theta w_\theta) \circ \pst \,, \qquad  \mathcal{F} =  (\tfrac{48}{3}  \mathcal{A} +  \mathcal{G}) \circ \pst \,,
$$
and by employing the chain-rule, we write \eqref{kywy-eqn} as
\begin{align*} 
\p_s \chi + 2 \chi (\p_\theta \pst) ^{-1} \p_s( \p_\theta  \pst) =  \mathcal{F} \,.
\end{align*} 
It follows that
\begin{align*} 
\tfrac{d}{ds}  \left( (\p_\theta  \pst)^2  \chi\right) = (\p_\theta  \pst)^2 \mathcal{F}  \,,
\end{align*} 
and integration from $\stt$ to $\stt_1$ yields the identity
\begin{align} 
&( k_\theta w_\theta) (\sc_2(\stt_1),\stt_1)  (\p_\theta  \pst(\theta, \stt_1))^2  - ( k_\theta w_\theta) (\sc(\stt),\stt)w_\theta(\sc(\stt),\stt)  (\p_\theta  \pst(\theta, \stt))^2\notag\\
&= \int_\stt^{\stt_1}  (\tfrac{48}{3}  \mathcal{A} +  \mathcal{G}) \circ \pst \, (\p_\theta  \pst)^2 ds \notag\\
&= \int_\stt^{\stt_1}  \left( \mathcal{G} - \tfrac 43 c w_\theta k_{\theta\theta} \right) \circ \pst \, (\p_\theta  \pst)^2 ds
- \tfrac{48}{3}  \int_\stt^{\stt_1}  \left( \tfrac{1}{12} c k_\theta w_{\theta\theta}  \right)\circ \pst \, (\p_\theta  \pst)^2 ds
\,.
\label{eq:night:time:madness}
\end{align} 
First, we note that since $\st(\sc_2(\stt_1),\stt_1) = 0$, the estimate \eqref{eq:metal:2} implies that $k_\theta(\sc_2(\stt_1),\stt_1) = 0$, and so the first term on the left side of \eqref{eq:night:time:madness} vanishes. 
The first term on the right side  of \eqref{eq:night:time:madness} is estimated using  \eqref{thegoodstuff1}, \eqref{eq:w:z:k:a:boot}, \eqref{eq:k:dxx:boot}, \eqref{dxphipsi-bound}, \eqref{stt-integral},  \eqref{good-kyflow-bound},   \eqref{eq:HUN*}, and \eqref{Q1b-insane} as 
\begin{align} 
 \int_\stt^{\stt_1} \left| \mathcal{G} - \tfrac 43 c w_\theta k_{\theta\theta} \right|  \circ \pst \, (\p_\theta  \pst)^2  ds \les \stt^ {-\frac{1}{6}} \,. \label{pre-insane1}
\end{align}  
Moreover, the estimates \eqref{eq:w:dx:boot},  \eqref{dxphipsi-bound},  \eqref{eq:double:trouble:6}, and \eqref{ky:LUV1}  show that
\begin{align} 
-   ( k_\theta w_\theta) (\sc(\stt),\stt) (\p_\theta  \pst(\theta, \stt))^2
& = \left( \tfrac{144\bb^\frac{9}{2} }{\kappa^{4}}   \stt^{\frac{1}{2}} + \OO(\stt) \right) \left(\tfrac 12 \stt^{-1} + \OO(\stt^{-\frac 12}) \right) \left(1 + \OO(\stt^{\frac 13}) \right)^2 \notag\\
& = \tfrac{72\bb^\frac{9}{2} }{\kappa^{4}}   \stt^{-\frac{1}{2}} + \OO(\stt^{-\frac 16})\,.
\label{pre-insane1a}
\end{align} 
By using \eqref{eq:night:time:madness}, the observation $k_\theta(\sc_2(\stt_1),\stt_1) = 0$, and the bounds \eqref{pre-insane1} and \eqref{pre-insane1a} we obtain that
\begin{align} 
-  \int_\stt^{\stt_1}  \left( \tfrac{1}{12} c k_\theta w_{\theta\theta}  \right)\circ \pst \, (\p_\theta  \pst)^2 ds 
=  \tfrac{9 \bb ^\frac{9}{2}}{2  \kappa^{4}}   \stt(\theta,t)^{-\frac{1}{2}} + \OO(\stt^{-\frac 16})  \,. \label{pre-insane2}
\end{align} 
Combining \eqref{eq:HUN*}, \eqref{Q1c-less-insane}, \eqref{Q1b-insane}, \eqref{pre-insane1}, and \eqref{pre-insane2}, we have proven that for $\bar\eps$ small enough,
\begin{align} 
-\int_{\stt(\theta,t)}^{\stt_1(\theta,t) }  \mathcal{Q} _1 \! \circ\! \psi_t  \ (\p_\theta  \psi_t)^2 ds 
\le  \tfrac{9 \bb ^\frac{9}{2}}{2  \kappa^{4}}   \stt(\theta,t)^{-\frac{1}{2}} + C \stt(\theta,t)^{-\frac 16}  \,. \label{Q1-insanity}
\end{align} 

In addition to the bounds \eqref{thegoodstuff1}, \eqref{eq:w:z:k:a:boot}, \eqref{eq:k:dxx:boot}, \eqref{dxphipsi-bound}, \eqref{stt-integral},  \eqref{good-kyflow-bound}, by
also appealing to \eqref{eq:z:dxx:boot} and \eqref{zyy-integral}, we deduce that
\begin{align} 
-\int_{\stt(\theta,t)}^{\stt_1(\theta,t) }  \mathcal{Q} _2 \! \circ\! \psi_t  \ (\p_\theta  \psi_t)^2  ds  \les  \stt(\theta,t)^ {\frac{1}{2}}  \,. \label{Q2-insanity}
\end{align} 
Moreover, by using the identity \eqref{dyypsi} for $\p_\theta ^2\pst$, we see that the integrand $Q^z\! \circ\! \psi_t  \ \p_\theta ^2 \psi_t $ is estimated in the identical fashion as
the term $\mathcal{Q}_{1b}$ in \eqref{Q1b-insane}, and hence we have that
\begin{align} 
-\int_{\stt(\theta,t)}^{\stt_1(\theta,t) }  Q^z\! \circ\! \psi_t  \ \p_\theta ^2 \psi_t    ds \les  \stt ^ {-\frac{1}{6}}  \,. \label{Qz-insanity}
\end{align} 
Together, the bounds \eqref{Q1-insanity}, \eqref{Q2-insanity}, and \eqref{Qz-insanity} establish the desired inequality \eqref{insanity}, for $\stt(\theta,t) \ll 1$.

The proof of the lemma is completed once we establish \eqref{vinsanity}. These estimates are however simpler because by the definition of the time $\stt_1(\theta,t) $, for all $\sc_1(t) < \theta < \sc_1(t) + \frac{\kappa t}{6}$, and for all $s \in (\stt_1(\theta,t) , t)$, we have that $(\pst(\theta,s),s) \in \mathcal{D} ^z_{\bar\eps}\setminus\overline{\mathcal{D}^k_{\bar\eps}}$, and $k \equiv0$ in this region. In particular, this means that in this region we have that $Q^z = \frac 83 \p_\theta  (az)$, and $\p_\theta  Q^z  = {\mathcal Q}_2 = \frac 83 \p_{\theta}^2  (a z)$; there are no dangerous $k$ terms. As such the bounds we seek directly follow from \eqref{fof6} and \eqref{fof7}:
\begin{align}
 {\mathcal B}_2 \leq \int_{\stt_1(\theta,t) }^t \sabs{\p_\theta   {Q}^z \circ \pst} (\p_\theta  \psi_t)^2 ds + \int_{\stt_1(\theta,t) }^t \sabs{Q^z\! \circ\! \psi_t  \ \p_\theta ^2 \psi_t } ds 
&  \le  C t^ {\frac{1}{3}} \stt(\theta,t)^ {-\frac{1}{2}}
\end{align}
for a suitable constant $C$. The bound \eqref{vinsanity} follows since $t\leq \bar \eps \ll 1$. This completes the proof of Lemma~\ref{lem:insanity}.
\end{proof}

\subsection{Precise H\"older estimates for derivatives}
\label{sec:dy:Holder}

Here we combine the upper bounds established in Section~\ref{sec:dyy:bootstraps}, with the lower bounds proven in Section~\ref{sec:dyy:lower:bound}, to precisely characterize the behavior of $(w_\theta, z_\theta, k_\theta, a_\theta)$ as $\theta\to \sc_1(t)^+$ and $\theta\to \sc_2(t)^+$.

We first consider the behavior of these derivatives on  $\sc_1(t)$. Note that on the left side of $\sc_1(t)$, by \eqref{eq:second:order:boot3} and \eqref{wyy-bound-final} we have that the order second derivatives of $w$ and $a$ are finite for every $t \in (0,\bar \eps]$, but that the bounds are not uniform in $t$ as $t\to 0^+$ (as should be expected, since $w_0,a_0 \not \in C^2$). On this left side of $\sc_1(t)$, we moreover have that $k \equiv z \equiv 0$. Similarly, on the right side of $\sc_1(t)$, the second derivative of $w$ is bounded due to \eqref{wyy-bound-final}, the second derivative of $a$  is bounded in light of~\eqref{eq:ayy:improve:bootstrap}, these bounds not being uniform as $ t\to 0^+$, while $k \equiv 0$. It remains to consider the behavior of $z_\theta(\theta,t) $ as $\theta \to \sc_1(t)^+$. From \eqref{eq:metal:4} we know that $z_\theta(\sc_1(t),t) = 0$, so that using \eqref{eq:z:dxx:boot2} and \eqref{eq:bohemian:4}
\begin{align}
\sup_{0< h < \frac{\kappa t}{6}} \frac{\sabs{z_\theta(\sc_1(t) + h,t) - z_\theta(\sc_1(t),t)}}{h^\alpha} 
&\leq  \sup_{0< h < \frac{\kappa t}{6}} h^{1-\alpha} \int_0^1 \sabs{z_{\theta\theta} (\sc_1(t) + \lambda h,t)} d\lambda \notag\\
&\leq N_2 \sup_{0< h < \frac{\kappa t}{6}} h^{1-\alpha} \int_0^1 \stt(\sc_1(t) + \lambda h,t)^{-\frac 12} d\lambda \notag\\
&\leq 2 N_2  \kappa^{-\frac12}   \sup_{0< h < \frac{\kappa t}{6}} h^{1-\alpha} \int_0^1 \sabs{ \lambda h}^{-\frac 12} d\lambda \notag\\
&= 32 \mm^3  \kappa^{-\frac12}  \sup_{0< h < \frac{\kappa t}{6}} h^{\frac 12 -\alpha} 
\,.
\label{eq:zyy:s1:upper}
\end{align}
The right side of \eqref{eq:zyy:s1:upper} is finite whenever $\alpha \leq \frac 12$. Thus, from \eqref{eq:z:dxx:boot2}, \eqref{eq:bohemian:2}, and \eqref{eq:zyy:s1:upper}, we deduce that $z \in C^{1,\frac 12}$ in  $\mathcal{D} ^z_{\bar \eps} \setminus  \overline{\mathcal{D} ^k_{\bar \eps} }$. The remarkable fact is that due to \eqref{eq:sanity}, this upper bound is sharp: for any $\alpha> \frac 12$, $z \not \in C^{1,\alpha}$ near $\sc_1$. Indeed, by \eqref{eq:sanity}, we have that for $h$ sufficiently small but positive, 
\begin{align}
\frac{ z_\theta(\sc_1(t) + h,t) - z_\theta(\sc_1(t),t)}{h^\alpha} 
&= h^{1-\alpha} \int_0^1 z_{\theta\theta} (\sc_1(t) + \lambda h,t) d\lambda \notag\\
&\leq - \tfrac 18 \bb^{\frac 92} \kappa^{-4} h^{1-\alpha} \int_0^1 \stt(\sc_1(t) + \lambda h,t)^{-\frac 12} d\lambda \notag\\
&\leq - \tfrac{1}{16} \bb^{\frac 92} \kappa^{- \frac 92} h^{1-\alpha} \int_0^1 (\lambda h)^{-\frac 12} d\lambda \notag\\
&\leq - \tfrac{1}{8} \bb^{\frac 92} \kappa^{- \frac 92} h^{\frac 12-\alpha} 
\,.
\label{eq:zyy:s1:lower}
\end{align}
For $\alpha> \frac 12$, the right side of \eqref{eq:zyy:s1:lower} converges to $-\infty$ as $h \to 0^+$, proving that $z_\theta \not \in C^\alpha$ in the vicinity of $\sc_1$.

Next, we consider the behavior of derivatives  on  $\sc_2(t)$. On the left side of $\sc_2(t)$ we have that $k\equiv 0$, while the second derivatives of $w$, $z$, and $a$ are bounded in terms of inverse powers of $t$ in view of \eqref{eq:z:dxx:boot2}, \eqref{eq:bohemian:2}, \eqref{eq:ayy:improve:bootstrap}, and \eqref{wyy-bound-final}. On the right side of $\sc_2$, the situation is different. Similarly to \eqref{eq:zyy:s1:upper}, we may use \eqref{wyy-bound-final}, \eqref{eq:z:dxx:boot}, \eqref{eq:ayy:improve:bootstrap}, and \eqref{eq:bohemian:1} to show that $w_\theta, z_\theta, k_\theta, a_\theta \in C^{\alpha}$ near $\sc_2$, for any $\alpha \leq \frac 12$. Indeed, the only difference to \eqref{eq:zyy:s1:upper} is that $ \stt(\sc_1(t) + \lambda h,t)^{-\frac 12} $ is replaced by $ \st(\sc_1(t) + \lambda h,t)^{-\frac 12} \leq 2 \kappa^{-\frac 12} (\lambda h)^{-\frac 12}$. Moreover, for any $\alpha > \frac 12$, similarly to \eqref{eq:zyy:s1:lower}, we may use \eqref{eq:double:trouble:final}, \eqref{eq:stormy:monday:2}, \eqref{eq:stormy:monday:4}, and \eqref{eq:oxigen:2} to prove 
\begin{subequations}
\label{eq:dyy:s2:lower}
\begin{align}
\frac{k_\theta(\sc_2(t) + h,t) - k_\theta(\sc_1(t),t)}{h^\alpha} 
& \geq 12 \bb^{\frac 92} \kappa^{- \frac 92} h^{\frac 12-\alpha}   \\
\frac{z_\theta(\sc_2(t) + h,t) - z_\theta(\sc_1(t),t)}{h^\alpha} 
& \leq - \tfrac 12 \bb^{\frac 92} \kappa^{- \frac 72} h^{\frac 12-\alpha}   \\
\frac{w_\theta(\sc_2(t) + h,t) - w_\theta(\sc_1(t),t)}{h^\alpha} 
& \geq \tfrac 12 \bb^{\frac 92} \kappa^{- \frac 72} h^{\frac 12-\alpha}   \\
\frac{a_\theta(\sc_2(t) + h,t) - a_\theta(\sc_1(t),t)}{h^\alpha} 
& \leq - \tfrac 18 \bb^{\frac 92} \kappa^{- \frac 52} t h^{\frac 12-\alpha}  
\,,
\end{align}
\end{subequations}
for $h >0$ sufficiently small. The estimates in \eqref{eq:dyy:s2:lower} show that $k_\theta ,z_\theta, w_\theta ,a_\theta \not \in C^{\alpha}$ for any $\alpha > \frac 12$.

\subsection{Proof of Theorem~\ref{thm:C2}}
\label{sec:proof:thm:C2}
The bounds in \eqref{eq:thm-bounds-final} are merely a restatement of the bootstrap bounds stated in \eqref{sec:dyy:bootstraps} for $(w_{\theta\theta} ,z_{\theta\theta} ,k_{\theta\theta} )$. The bounds for $a_{\theta\theta} $ and $\varpi_{\theta} $ follow as shown in Lemmas~\ref{lem:varpi:y} and~\ref{lem:ayy}. These bootstrap estimates were closed (i.e., improved by a factor of $2$) by the analysis in Sections~\ref{sec:dyy:flows}--\ref{sec:zyy}. As discussed in the first paragraph of Section~\ref{sec:dyy:bootstraps}, this analysis should formally be carried out at the level of the approximating sequence $(w^{(n)},z^{(n)},k^{(n)})$, but we have not chosen to do so for simplicity of the presentation. One remark is  in order at this point: when dealing with the approximating sequence $(w^{(n)},z^{(n)},k^{(n)},a^{(n)})$ the identities \eqref{dyyphi} and \eqref{dyypsi} for the second derivatives of $\pt$ and $\pst$ are not available; this is because the structure of the equation for the sound speed at $c^{(n+1)}$, given in \eqref{cn}--\eqref{cn1}, lacks a necessary $n \to n+1$ symmetry; in this case, estimates for $\p_{\theta}^2  \pt^{(n)}$ and $\p_{\theta}^2  \pst^{(n)}$ are obtained simply by differentiating \eqref{flow1} and \eqref{flow2} twice with respect to $y$ and appealing to the bootstrap bounds for $\p_{\theta}^2  w^{(n)}$ and $\p_{\theta}^2  z^{(n)}$; the resulting bounds are however exactly the same as the ones given in Lemma~\ref{lem:dyy12}.

The bounds in \eqref{eq:s2:dyy:bounds} follow from \eqref{eq:thm-bounds-final} on the one hand, and \eqref{eq:bohemian}, \eqref{eq:double:trouble:final}, \eqref{eq:stormy:monday:2}, \eqref{eq:stormy:monday:4}, \eqref{eq:oxigen:2}, on the other hand. The estimate \eqref{eq:s2:dyy:magic} follows by adding the bounds in \eqref{eq:stormy:monday:1} and \eqref{eq:stormy:monday:3}, observing that the terms $\frac 14 c k_{\theta\theta} $ cancel. The characterization of the singularity formed by $(w_\theta,z_\theta,k_\theta,a_\theta)$ as $\theta\to \sc_2(t)^+$ as being precisely a $C^{\frac 12}$ cusp is given by Section~\ref{sec:dy:Holder}, estimate \eqref{eq:dyy:s2:lower}. 
The estimate \eqref{eq:s1:dyy:bounds} is implied by bounds \eqref{eq:thm-bounds-final} and \eqref{eq:sanity}. The characterization of the singularity formed by $z_\theta$ as $\theta \to \sc_1(t)^+$ as being precisely a $C^{\frac 12}$ cusp is given by Section~\ref{sec:dy:Holder}, estimate \eqref{eq:zyy:s1:lower}. This concludes the proof of  Theorem~\ref{thm:C2}.


\section{Shock development for 2D Euler}
\label{sec:2D:Euler}

In view of the transformations $(u_\theta,u_r,\sigma,S) \mapsto (b,c,k,a) \mapsto (w,z,k,a)$ described in \eqref{scale0} and \eqref{eq:riemann}, the  results obtained in Sections~\ref{sec:formation}--\ref{sec:C2} for the azimuthal variables $(w,z,k,a,\varpi)$   imply the following results for the usual hydrodynamic variables $(u,\rho, E, p)$. First, from Theorem~\ref{thm:blowup-profile} we deduce: 

\begin{theorem}[\bf  Shock formation for 2D Euler with azimuthal symmetry]
\label{thm:main:formationeuler}
There exists $\kappa_0>1$ sufficiently large, and $\eps>0$ sufficiently small, such that the following holds. Consider initial data at time $-\eps$ given by
\begin{align*}
\left(u_r,u_\theta,\sigma,S\right)(r,\theta,-\eps) :=\left(r a(\theta,-\eps),\tfrac{r}{2} w(\theta,-\eps) ,\tfrac{r}{2}w(\theta,-\eps),0\right) \,,
\end{align*}
with $(w,a)(\cdot,-\eps)$ satisfying conditions \eqref{w0-to-W0}--\eqref{eq:A_bootstrap:IC}. 
In particular, the initial data is smooth and has azimuthal symmetry. Then, there exists $T_* > -\eps$ (explicitly computable), and a unique   solution $(u,\sigma,S) \in C^0([-\eps,T_*);C^4(\RR^2\setminus\{0\}))$ of the Euler equations \eqref{eq:Euler2}, which has the azimuthal symmetry~\eqref{scale0}. The associated density is  $\rho = \frac{1}{4} \sigma^2 e^{{-{\kcal}} } = \frac{1}{16} r^2 w^2$, and the total energy is $E=  \tfrac{1}{2} \rho \abs{u}^2 + \frac 12  \rho^2 e^{\kcal} = \tfrac{1}{32} r^4 w^2 (a^2 + \frac{5}{16} w^2)$.
Moreover, at time blowup time $T_*$ we have $S(\theta,T_*) = 0$,  and  there exists a unique angle $\xi_* \in \TT$ (explicitly computable) such that an {\em azimuthal pre-shock} forms on the half-infinite ray $\{ (r,\xi_*,T_*)\}_{r\in \RR_+}$. The azimuthal pre-shock is described by the fact that for $|\theta - \xi_*| \ll_\eps 1$ we have
\begin{align*}
u_\theta(r,\theta,T_*) &=  \tfrac{1}{2} r \left( \kappa_* + {\mathsf a}_1 (\theta - \xi_*)^{\frac 13} + {\mathsf a}_2 (\theta - \xi_*)^{\frac 23} + {\mathsf a}_3 (\theta - \xi_*)+ \OO((\theta - \xi_*)^{\frac 43})\right) \\
u_r(r,\theta,T_*) &=  r \left( {\mathsf a}_0' + {\mathsf a}_1' (\theta - \xi_*)  + {\mathsf a}_2' (\theta - \xi_*)^{\frac 43} + \OO((\theta - \xi_*)^{\frac 53})\right) \\
\rho(r,\theta,T_*) &=  \tfrac{1}{16} r^2 \left( \kappa_*^2 + 2 {\mathsf a}_1 \kappa_* (\theta - \xi_*)^{\frac 13}  + ({\mathsf a}_1^2 + 2 {\mathsf a}_2 \kappa_*) (\theta - \xi_*)^{\frac 23} + \OO(\theta - \xi_*)  \right) \\
E(r,\theta,T_*) &=  \tfrac{1}{32} r^4 \left(  \kappa_*^2  ({\mathsf a}_0'^2+ \tfrac{5 \kappa_*^2}{16} ) + {\mathsf a}_1 \kappa_* (2 {\mathsf a}_0'^2 + \tfrac{5\kappa_*^2}{4}) (\theta - \xi_*)^{\frac 13}  +  \OO((\theta - \xi_*)^{\frac 23})\right) 
\end{align*}
where $\kappa_*, {\mathsf a}_1,{\mathsf a}_2,{\mathsf a}_3, {\mathsf a}_0',{\mathsf a}_1', {\mathsf a}_2'$ are suitable constants which may be computed in terms of the data. Moreover, in view of \eqref{preshock-derivatives} we have that these asymptotic descriptions are valid (to leading order), for the first three derivatives of the solution, and for $|\theta - \xi_*| \ll_\eps 1$. For angles $\theta$ which are at any fixed distance away from $\xi_*$, the functions $(u,\rho,E)(r,\theta,T_*)$ are $C^4$ smooth. Lastly, the specific vorticity and its derivatives remain uniformly bounded up to $T_*$. 
\end{theorem}

The above result, which establishes the formation of the pre-shock and gives its detailed description, is nothing but a rewriting of Theorem~\ref{thm:blowup-profile} in terms of the usual fluid variables. This is possible in view of the mapping $(u_r , u_\theta, \sigma , S) = (r a, \frac 12 r w, \frac 12 r w, 0)$, valid on $[-\eps,T_*]$, and the above mentioned formulas for the density and energy. The series expansion for the radial velocity $r a(\theta,T_*)$ is not explicitly stated in Theorem~\ref{thm:blowup-profile}, but it immediately follows from the fact that $a$ has regularity precisely $C^{1,1/3}$ and no better, and from the bounds on $a \circ \eta$ obtained in Section~\ref{sec:formation}. 

For the development part of our result, for simplicity of notation it is convenient to re-label the pre-shock location $(r,\xi_*,T_*) \mapsto (r,0,0)$. Moreover, the fields at which we arrive at the end of the formation part, namely $(u,\sigma,S)(\cdot,T_*)$, are re-labeled as $(u_0,\sigma_0,S_0)$. Then, from Theorems~\ref{thm:main:development} and~\ref{thm:C2} we obtain:
\begin{theorem}[\bf  Shock  development for 2D Euler with azimuthal symmetry]
\label{thm:main:developmenteuler}
Given pre-shock initial data 
\begin{align*}
(u_r,u_\theta,\sigma,S)|_{t=0} :=(r a_0,\tfrac{r}{2}(w_0+ z_0) ,\tfrac{r}{2}(w_0-z_0),k_0),
\end{align*}
 with $(w_0,z_0,a_0,k_0)$ satisfying conditions \eqref{eq:u0:ass:quant}--\eqref{eq:svort0:ass}, there exist:
\begin{enumerate}
\item  $\bar\eps   > 0 $ sufficiently small; 
\item  a shock surface $\mathcal{S}:=\{ (r,\theta,t) \in \RR^2 \times [0,\bar \eps] \colon \theta= \sc(t)\}$ with  $\sc \in C^2([0,\bar \eps])$;
\item fields $(u, \rho,E)$  with  $\rho = \frac{1}{4} \sigma^2 e^{{-{\kcal}} }$ and 
$E=  \frac{1}{2} \rho \abs{u}^2 + \frac 12 \rho^2 e^{\kcal} $, such that the $(u,\rho, E, {\mathcal S})$ is a {\em regular shock solution of the compressible Euler equations}~\eqref{euler-weak} on the time interval $[0,\bar \eps]$, in the sense of Definition~\ref{shockdef}; 
\item two $C^1$ smooth  functions $\sc_1, \sc_2 \colon [0,\bar \eps] \to \TT$, with $\sc_1(0) = \sc_2(0)= 0$ and $\sc_1(t) < \sc_2(t) < \sc(t)$ for $t \in (0,\bar \eps]$, such that  $\mathcal{S}_i:= \{ (r,\theta,t) \in \RR^2 \times [0,\bar \eps] \colon  \theta= \sc_i(t)\}$ is a characteristic surface for the $\lambda_i$ wave-speed, where $\lambda_1 =  u_\theta - \frac 12 \sigma$ and $\lambda_2 =   u_\theta$;
\end{enumerate}
such that for any $t\in (0,\bar \eps]$ all fields are twice differentiable at points $(r,\theta)$ with $\theta \not \in \{\sc_1(t), \sc_2(t),\sc(t)\}$, and the following hold:
\begin{enumerate}
  \setcounter{enumi}{4}
\item  letting $\DD_{\bar \eps}^{(2)} = \{ (r, \theta,t) \in \mathbb{R}^2 \times(0,\bar\eps]   \colon \sc_2(t) < \theta < \sc(t)\}$  we have that 
\begin{itemize}
\item  $S\in C^{1,1/2}(\DD_{\bar \eps}^{(2)})$,  $S \equiv 0$ on $(\DD_{\bar \eps}^{(2)})^\complement$,  and
$\frac{1}{C} \leq ( \theta-\sc_2(t))^{\frac 12} \partial_\theta^2S(r,\theta,t) \leq  C $ as $\theta\to \sc_2(t)^+$,
\item  $p, u_\theta \in C^{2}(\DD_{\bar \eps}^{(2)})$, 
 $|\partial_\theta^2 u_\theta(r,\theta,t)|  \leq C r t^{-\frac{1}{2}}$ and 
 $|\partial_\theta^2 p (r,\theta,t)| \leq C r^4 t^{-2}$ as $\theta\to \sc_2(t)^+$, 
\item  $u_r \in C^{1,1/2}(\DD_{\bar \eps}^{(2)})$ and
 $-r t C \leq ( \theta-\sc_2(t))^{\frac 12} \partial_\theta^2u_r(r,\theta,t) \leq - \frac{1}{C} r t $ as $\theta\to \sc_2(t)^+$,
\item $\rho \in C^{1,1/2}(\DD_{\bar \eps}^{(2)})$ and
 $-r^2 C\leq ( \theta-\sc_2(t))^{\frac 12} \partial_\theta^2\rho(r,\theta,t) \leq - \frac{1}{C} r^2$ as $\theta\to \sc_2(t)^+$,
\end{itemize}
for a suitable constant $C>0$;

\item  letting  $\DD_{\bar \eps}^{(1)} = \{ (r, \theta,t) \in  \mathbb{R}^2 \times(0,\bar\eps]  \colon \sc_1(t) < \theta< \sc_2(t)\}$, we have
\begin{itemize}
\item
 $S(r,\theta,t) = 0$ on $\DD_{\bar \eps}^{(1)}$, 
 \item
 $u_\theta \in C^{1,1/2}(\DD_{\bar \eps}^{(1)})$ and
 $ \frac{1}{C} r \leq ( \theta-\sc_1(t))^{\frac 12} \partial_\theta^2u_\theta(r,\theta,t) \leq  C r $ as $\theta\to \sc_1(t)^+$,
   \item   $u_r \in C^{2}(\DD_{\bar \eps}^{(1)})$ and
 $|\partial_\theta^2 u_r(r,\theta,t)|  \leq C r t^{-1}$  as $\theta\to \sc_1(t)^+$,
  \item $\rho \in C^{1,1/2}(\DD_{\bar \eps}^{(1)})$ and
 $-r^2 C\leq ( \theta-\sc_1(t))^{\frac 12} \partial_\theta^2\rho(r,\theta,t) \leq - \frac{1}{C} r^2$ as $\theta\to \sc_1(t)^+$,
\end{itemize}
for a suitable constant $C>0$;
\item  on $\mathcal S$, the functions $u_\theta(r, \cdot,t)$ and $\partial_\theta u_r(r, \cdot, t)$ exhibit $\OO(r t^{\frac 12})$ jumps, the density $\rho(r,\cdot,t)$ exhibits an $\OO(r^2 t^{\frac 12})$ jump,  the entropy $S(r,\cdot,t)$  exhibits an $\OO(t^{\frac 32})$ jump, the total energy $E(r,\cdot,t)$ exhibits an $\OO(r^4 t^{\frac 12})$ jump (cf.~\eqref{eq:J:M:rough} and ~\eqref{eq:zl:and:kl:on:shock:L:infinity}), while  $u_r(r,\cdot,t)$ does not jump.
\end{enumerate}
Moreover, this solution is unique in the class of entropy producing regular shock solutions (cf.~Definition~\ref{shockdef})  with azimuthal symmetry, such that the corresponding azimuthal variables $(w,z,k,a)$ belong to the space ${\mathcal X}_{\bar \eps}$ (cf.~Definition~\ref{def:XT}).
\end{theorem}

The above theorem directly follows from our previous two Theorems \ref{thm:main:development} and \ref{thm:C2}, by taking into account the relation between the fluid variables and the azimuthal variables in \eqref{scale0},  and in turn to the Riemann variables in \eqref{eq:riemann}.   The bounds on second derivatives are all a consequence of Theorem \ref{thm:C2}. In the region $\DD_{\bar \eps}^{(2)}$, the bounds for the entropy $S$  and radial velocity $u_r$  follow from \eqref{eq:s2:dyy:bounds}. Since $u_\theta= \tfrac{r}{2} (w+z)$, the bound for the second derivative of $u_\theta$ in the region $\DD_{\bar \eps}^{(2)}$, which does not blow up as $\theta \to \sc_2(t)^+$ in positive time,  follows from \eqref{eq:s2:dyy:magic}.  Since $\rho = \tfrac{r^2}{4} c^2 e^{-k}$,  the claimed bound for the second derivative of the density follows from  \eqref{eq:the:X:T},  \eqref{eq:s2:dyy:bounds}, \eqref{vladneeds2}, and \eqref{eq:lossy:qy:dz:bound}  since  we may write
\begin{align*}
 \tfrac{16}{r^2}\partial_\theta^2 \rho 
 &=  c e^{-k} \left(2 c_{\theta\theta} - c k_{\theta\theta} \right) + (\mbox{terms which are bounded as } \theta \to \sc_2(t)^+ \mbox{ in terms of  powers of } t^{-1}) \notag\\
  &=  c e^{-k} \left( q^w_{\theta} - q^z_\theta - \tfrac 12  c k_{\theta\theta} \right) + (\mbox{terms which are bounded as } \theta \to \sc_2(t)^+ \mbox{ in terms of   powers of } t^{-1}) \notag\\
   &=   - \tfrac 12  c^2 e^{-k} k_{\theta\theta}  + (\mbox{terms which are bounded as } \theta \to \sc_2(t)^+ \mbox{ in terms of   powers of } t^{-1})\,.
\end{align*}
and so the singularity of $k_{\theta\theta}$ on $\sc_2$ carries over to $\rho$. Lastly, the claimed estimate for the second derivative of pressure, which does not blow up as $\theta \to \sc_2(t)^+$ in positive time, follows from the identity $p = \frac{1}{32} r^4 c^4 e^{-k}$ and a similar computation as above
\begin{align}
 \tfrac{32}{r^4}\partial_\theta^2 p 
 &=  c^3 e^{-k} \left(4 c_{\theta\theta} - c k_{\theta\theta} \right) + (\mbox{terms which are bounded as } \theta \to \sc_2(t)^+ \mbox{ in terms of  powers of } t^{-1}) \notag\\
  &= 2 c^3 e^{-k} \left(  q^w_{\theta} - q^z_\theta   \right) + (\mbox{terms which are bounded as } \theta \to \sc_2(t)^+ \mbox{ in terms of   powers of } t^{-1}) \notag\\
   &=  (\mbox{terms which are bounded as } \theta \to \sc_2(t)^+ \mbox{ in terms of   powers of } t^{-1})\,.
   \label{eq:done}
\end{align}
The dependence of the bound on $t^{-1}$ follows from  \eqref{eq:the:X:T}, \eqref{vladneeds2}, and \eqref{eq:lossy:qy:dz:bound}. 

In the region $\DD_{\bar \eps}^{(1)}$,  we have that $w_{\theta \theta}$ is bounded in terms of inverse powers of $t$ and $z_{\theta\theta}$ satisfies \eqref{eq:s1:dyy:bounds}, which gives the bounds on $u_\theta$ and $\rho$.  The bound for the radial velocity appears in \eqref{thm-wyy-bound-final}.

The size of the jumps along the shock curve, and the uniqueness statement, follow directly from Theorem~\ref{thm:main:development}. To avoid redundancy we omit further details.

\subsection*{Acknowledgments} 
T.B. \ was supported by the NSF grant DMS-1900149 and a Simons Foundation Mathematical and Physical Sciences Collaborative Grant.
T.D.\ was supported by NSF grant DMS-1703997.   S.S.\ was supported by  NSF grant DMS-2007606 and the Department of Energy Advanced Simulation and Computing (ASC) Program.  V.V.\ was supported by the NSF CAREER grant DMS-1911413.



\end{document}